\documentclass[a4paper, twoside, 10 pt]{memoir}
\usepackage{amssymb}
\usepackage{latexsym} 
\usepackage{theorem} 
\usepackage{amsmath, amscd, bbm}  
\usepackage[dvips]{graphicx} 
\usepackage[all]{xy} 
\usepackage{color}
\usepackage{psfrag} 
\usepackage[applemac]{inputenc}
\usepackage[T1]{fontenc}
\usepackage{lmodern}
\usepackage{textcomp}
\usepackage{booktabs}
\usepackage{remreset}
\usepackage{pdfsync}
\usepackage{titleref}
\usepackage{url}
\input{Z.Layout.tex}
\input{Z.Definitions.tex}

\newcommand{\version}{}
\graphicspath{{figures-eps/}}
\usepackage{makeidx}
\makeindex

\begin{document}
\frontmatter
\chapterstyle{front}
\pagestyle{front}
%
\thispagestyle{empty}
{
%
%
\vspace*{1cm}

\usefont{T1}{cmss}{bx}{n}  \fontsize{29.86pt}{0pt} \selectfont
\begin{center}
Notes  \\[6mm] on the Sigma invariants
\end{center}
\usefont{T1}{cmss}{bx}{n}  \fontsize{20.74pt}{0pt} \selectfont
\vspace{1cm}
\begin{center}
Version 2
\end{center}

\vspace*{4cm}

\usefont{T1}{cmss}{m}{n} \fontsize{12pt}{0pt} \selectfont 
\begin{center}
edited by Ralph Strebel, \par
\vspace{0.3cm}
professor emeritus of Mathematics\par
\vspace{0.3cm}
at the University of Fribourg
\end{center}
\vspace{4 cm}
\usefont{T1}{cmss}{m}{sl} \fontsize{12pt}{0pt} \selectfont 

\hfill February 2013
}
\newpage
\thispagestyle{empty}
 \hspace*{0cm}
\includegraphics[width=4.0cm]{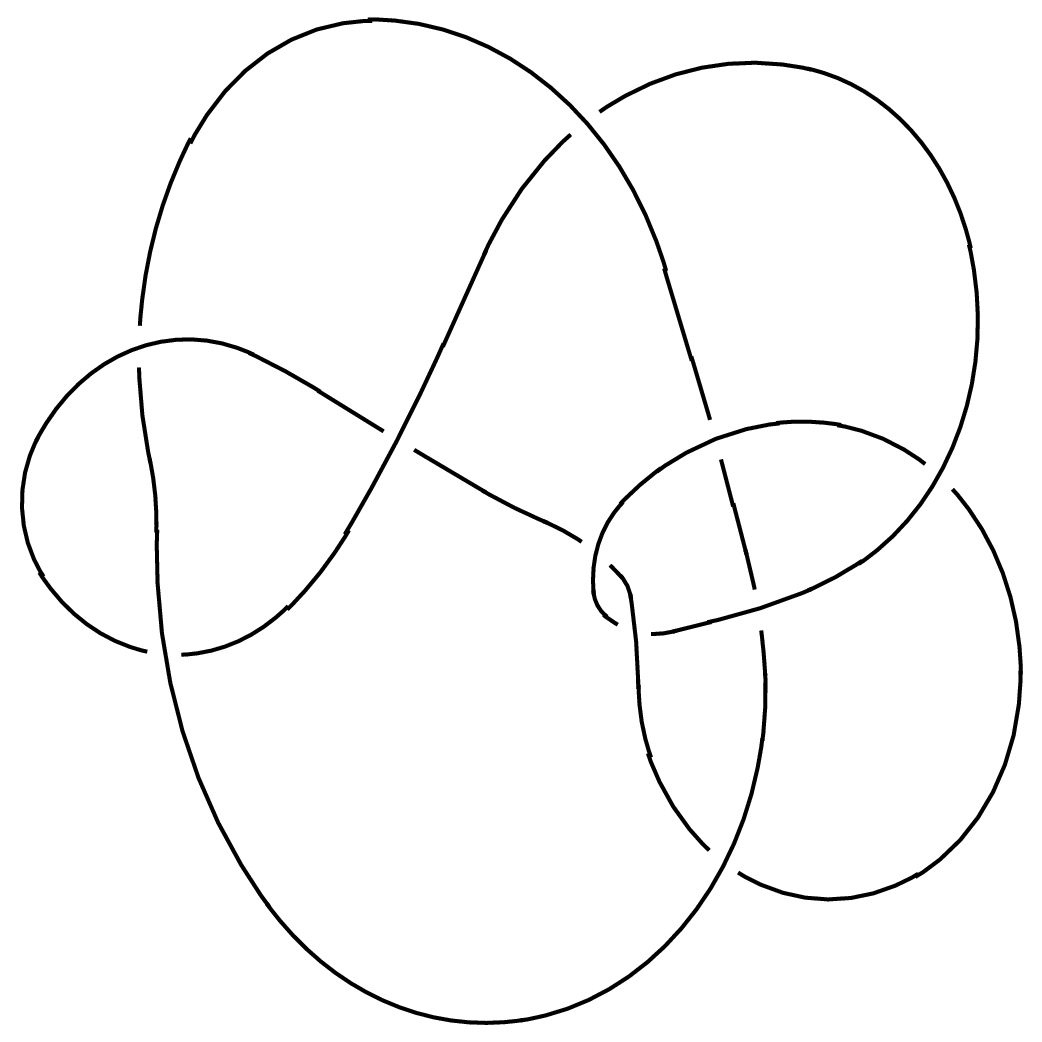}\par
\vspace*{-1.4cm} \hspace*{4.9cm}
\begin{minipage}[c]{7cm}{
\begin{align*}
r =\;&a^2b^{-1}a^{-1}baba^{-2}bab^{-2}aba^{-2}bab\;\cdot\\
&a^{-1}b^{-1}a^2b^{-1}a^{-1}b^{-1}aba^{-2}baba^{-2}b\;\cdot\\
&ab^{-1}a^{-1}b^{-1}a^2b^{-1}a^{-1}b^2a^{-1}b^{-1}a^2b^{-1}\cdot \\
&a^{-1}b^{-1}aba^{-2}
baba^{-1}b^{-1}a^2b^{-1}a^{-1}b^{-1}
\end{align*}
}
\end{minipage}

\vspace*{-1.0cm}
\begin{figure}[htb]
\psfrag{1}{\hspace*{-1.5mm} \tiny $1$}
\psfrag{2}{  \hspace*{-2.3mm}\tiny $2$}
\psfrag{3}{\hspace*{-0.7mm} \tiny $3$}
\psfrag{4}{\hspace*{-0.2mm}\tiny $4$}
\psfrag{5}{\hspace*{-2.5mm} \tiny $5$}
\psfrag{6}{  \hspace*{-2.1mm}   \tiny $6$}
\psfrag{7}{\hspace*{-2.2mm} \tiny $7$}
\psfrag{8}{\hspace*{-1.2mm} \tiny $8$}
\psfrag{9}{\hspace*{-1.7mm} \tiny $9$}
\psfrag{10}{\hspace*{-1.9mm} \tiny $10$}
\psfrag{11}{\hspace*{-1.8mm} \tiny $11$}
\psfrag{12}{  \hspace*{-1.3mm}\tiny $12$}
\psfrag{13}{\hspace*{-1.5mm} \tiny $13$}
\psfrag{14}{\hspace*{-0.2mm}\tiny $14$}
\psfrag{15}{\hspace*{-2.3mm} \tiny $15$}
\psfrag{16}{  \hspace*{-3.2mm}   \tiny $16$}
\psfrag{17}{\hspace*{-1.5mm} \tiny $17$}
\psfrag{18}{\hspace*{-2.0mm} \tiny $18$}
\psfrag{19}{\hspace*{-2.7mm} \tiny $19$}
\psfrag{20}{\hspace*{-1.7mm} \tiny $20$}
\psfrag{21}{\hspace*{-1.7mm} \tiny $21$}
\psfrag{22}{  \hspace*{-1.5mm}\tiny $22$}
\psfrag{23}{\hspace*{-1.4mm} \tiny $23$}
\psfrag{24}{\hspace*{-0.5mm}\tiny $24$}
\psfrag{25}{\hspace*{-1.6mm} \tiny $25$}
\psfrag{26}{  \hspace*{-2.3mm}   \tiny $26$}
\psfrag{27}{\hspace*{-1.3mm} \tiny $27$}
\psfrag{28}{\hspace*{-0.1mm} \tiny $28$}
\psfrag{29}{\hspace*{-1.7mm} \tiny $29$}
\psfrag{30}{\hspace*{-1.7mm} \tiny $30$}
\psfrag{31}{\hspace*{-0.6mm} \tiny $31$}
\psfrag{32}{  \hspace*{-1.8mm}\tiny $32$}
\psfrag{33}{\hspace*{-1.8mm} \tiny $33$}
\psfrag{34}{\hspace*{-0.9mm}\tiny $34$}
\psfrag{35}{\hspace*{-1.4mm} \tiny $35$}
\psfrag{36}{  \hspace*{-2.5mm}   \tiny $36$}
\psfrag{37}{\hspace*{-1.8mm} \tiny $37$}
\psfrag{38}{\hspace*{-1.7mm} \tiny $38$}
\psfrag{39}{\hspace*{-1.3mm} \tiny $39$}
\psfrag{40}{\hspace*{-1.7mm} \tiny $40$}
\psfrag{41}{\hspace*{-0.7mm} \tiny $41$}
\psfrag{42}{\hspace*{-1.6mm} \tiny $42$}
\psfrag{43}{\hspace*{-1.8mm} \tiny $43$}
\psfrag{44}{\hspace*{-0.6mm}\tiny $44$}
\psfrag{45}{\hspace*{-1.5mm} \tiny $45$}
\psfrag{46}{  \hspace*{-1.7mm}   \tiny $46$}
\psfrag{47}{\hspace*{-1.7mm} \tiny $47$}
\psfrag{48}{\hspace*{-1.2mm} \tiny $48$}
\psfrag{49}{\hspace*{-1.4mm} \tiny $49$}
\psfrag{50}{\hspace*{-2.6mm} \tiny $50$}
\psfrag{51}{\hspace*{-1.6mm} \tiny $51$}
\psfrag{52}{  \hspace*{-1.7mm}\tiny $52$}
\psfrag{53}{\hspace*{-2mm} \tiny $53$}
\psfrag{54}{\hspace*{-0.6mm}\tiny $54$}
\psfrag{55}{\hspace*{-1.7mm} \tiny $55$}
\psfrag{56}{  \hspace*{-2.8mm}   \tiny $56$}
\psfrag{57}{\hspace*{-2mm} \tiny $57$}
\psfrag{58}{\hspace*{-1.4mm} \tiny $58$}
\psfrag{59}{\hspace*{-1.4mm} \tiny $59$}
\psfrag{60}{\hspace*{-2mm} \tiny $60$}
\psfrag{61}{\hspace*{-2.0mm} \tiny $61$}
\psfrag{62}{  \hspace*{-0.8mm}\tiny $62$}
\psfrag{63}{\hspace*{-1.2mm} \tiny $63$}
\psfrag{64}{\hspace*{-0.4mm}\tiny $64$}
\psfrag{65}{\hspace*{-1.0mm} \tiny $65$}
\psfrag{66}{  \hspace*{-2.8mm}   \tiny $66$}
\psfrag{67}{\hspace*{-2.8mm} \tiny $67$}
\psfrag{68}{\hspace*{-1.9mm} \tiny $68$}
\psfrag{aa}{\hspace*{-1.0mm} \scriptsize $a$}
\psfrag{bb}{\hspace*{-1.0mm} \scriptsize $b$}

\vspace*{5mm}
\begin{center}
\includegraphics[width=11cm]{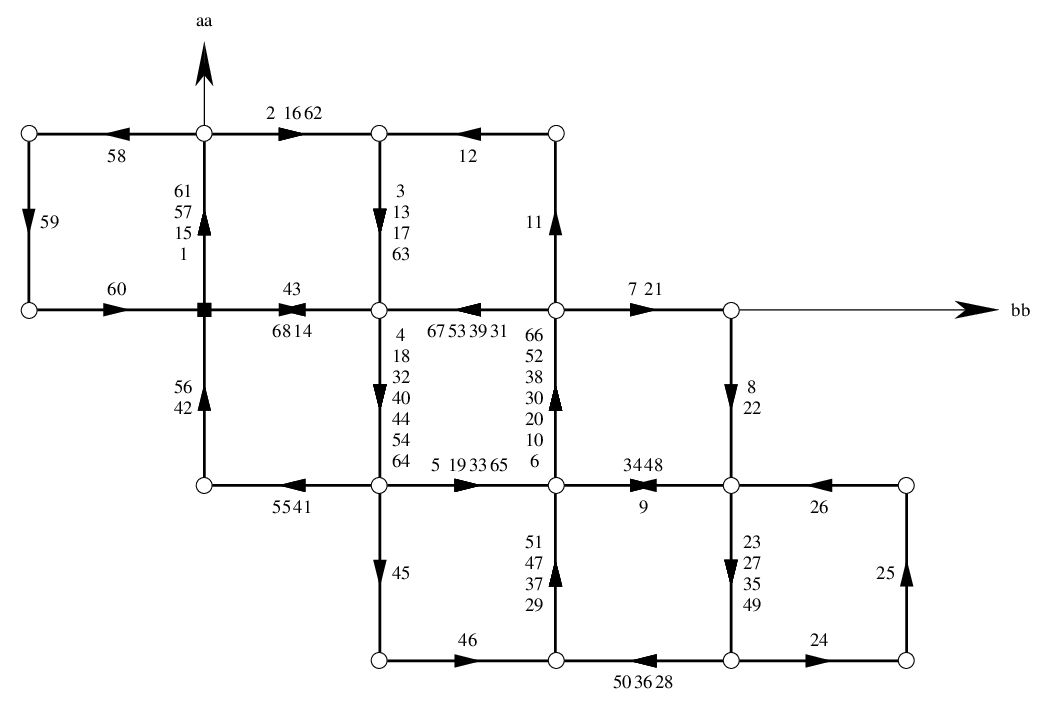}
\end{center}
\vspace*{-0.5cm}\hspace*{0.0cm}
\psfrag{aa}{\hspace*{-2.0mm} \scriptsize $\chi(a)$}
\psfrag{bb}{\hspace*{-2.0mm} \scriptsize $\chi(b)$}
\psfrag{ori}{\hspace*{-2.0mm} \scriptsize $(0,0)$}
\includegraphics[height=4cm]{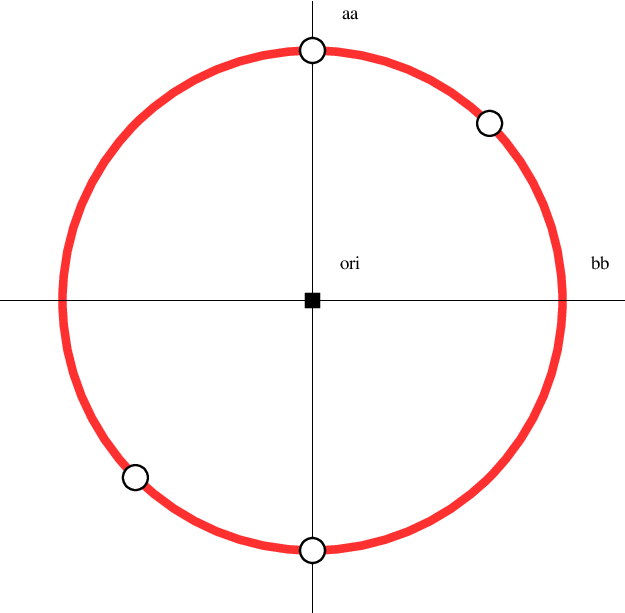} 
\end{figure}
\bigskip

\begin{center}
\sffamily
Diagram of the link L10n43, defining relator of its group,\par
path tracing out the relator and invariant of the group
\end{center}
%
\chapter{Preface}
\label{chap:Notice-reader}
%
\markboth{Preface}{Preface}

Towards the end of the 1980s
Robert Bieri and I started to write up a monograph on the Sigma invariants.
The first such invariant had been introduced in 1980 in a paper 
that dealt mainly with soluble groups 
and culminated in a characterization of the \emph{finitely related} metabelian groups 
among the \emph{finitely generated} metabelian groups (\cite{BiSt80}).
Several years later, Robert Bieri, Walter Neumann and R. Strebel invented new Sigma invariants,
defined for all finitely generated groups, 
and showed that they had implications far beyond the realm of soluble groups.
The theory put forward in  \cite{BNS} enables one, in particular, 
to characterize those normal subgroups $N$ of a finitely generated group $G$
that are finitely generated (as groups) and contain the derived group $G'$ of $G$.
At about the same time, 
Ken Brown, Gilbert Levitt, Ga{\"e}l  Meigniez and Jean-Claude Sikorav 
found alternate definitions for one of the Sigma invariants studied in \cite{BNS} 
(see \cite{Bro87b}, \cite{Lev87}, \cite{Mei87} and \cite{Sik87}).
And a bit later, 
Robert Bieri and his student Burkhardt Renz succeeded in defining higher dimensional Sigma-invariants 
and in establishing analogues of some of the main theorems of \cite{BNS}
(see \cite{BiRe88} and \cite{Ren89}).
Among their results, is a characterization of the normal subgroups $N$ of a finitely presented group $G$
that are finitely presentable and contain $G'$.
\smallskip

\textbf{1.}\quad By the end of the 1980s,
time seemed ripe for a comprehensive account of the various kinds of Sigma invariants 
developed in the previous decade and of their mutual relations.
Robert Bieri and I started work on such a memoir in 1988 and completed a first version in 1992.
This version consisted of 4 chapters, written up in detail, 
and 3 appendices, the first one dealing with metabelian groups, 
the original motivation for the creation of the theory,
the second one sketching the higher dimensional theory,
the third one being a collection of notes.
The first two of these appendices indicated
that this version was only half-finished, and so we tried to complete it in the following years. 
But we did not reach this goal.
\smallskip

\textbf{2.}\quad 
After my retirement in 2007, 
I made a new attempt at completing the monograph.
My plan was to write a memoir with roughly the same scope,
but to incorporate new developments that had taken place  in the intervening years,
and to present applications to other parts of Group Theory and attractive examples 
as early as feasible.
These tasks turned out to be more time consuming than expected,
and so I have decided to publish the monograph in several installments of which this is the second one.
\smallskip

\textbf{3.}\quad 
In presenting these \emph{Notes on the Sigma Invariants}
it is a pleasure to acknowledge the help I received from many colleagues.
My foremost thanks go to Robert Bieri.
It is with him that I tried to fathom the structure of finitely presented metabelian groups
in the second half of the 1970s, a problem that led us in 1980 to introduce the first Sigma invariant.
This introduction was followed in the 1980s by several joint papers dealing 
with various aspects of the invariant and also generalizations thereof,
and culminated in 1987 with the introduction of the invariant $\Sigma_{G'}(G)$,
this time in collaboration with Walter Neumann.
My debt to Robert is also substantial in another respect:
the first chapter of these \emph{Notes} is in content and form close to  Chapter I of the monograph \cite{BiSt92}
which, in turn, is an outgrowth of a lecture Robert gave in the second half of the 1980s.

I owe also many thanks to other colleagues. 
Their suggestions have served me well; I list their names in alphabetical order:
Roger Alperin,
Gilbert Baumslag,
Markus Brodmann,
Ken Brown,
Gerhard Burde,
Ian Chiswell,
Laura Ciobanu,
Yves de Cornulier,
Reinhard Diestel,
Nathan Dunfield,
Martin J. Dunwoody, 
Ross Geoghegan,
Slava Grigorchuk,
John Groves,
Pierre de la Harpe,
Jim Howie,
Manfred Karbe,
Desi Kochloukova,
J{\"u}rg Lehnert,
Charles Livingston,
John Meier,
Chuck Miller III,
Gaël Meigniez,
Peter M. Neumann,
Derek J. S. Robinson
and
Hanspeter Scherbel.
To all of them, I express my gratitude.
\smallskip

\textbf{4.}\quad 
Much of this text has been written in a small village in the Swiss Alps.
I maintained contact with the world outside mostly through mails,
but had the good luck of being invited to two stimulating meetings:
the first one, the \emph{Finitely Presented Solvable Groups Conference},
taking place in March 2011 at the City College of New York
and being organized by Gilbert Baumslag, Stuart Margolis, Gretchen Ostheimer, Vladimir Shpilrain, 
and Sean Cleary.
The second one, a workshop held at the Erwin Schr{\"o}dinger International Institute in Vienna,
organized by Goulnara Arzhantseva and Mark Sapir in December 2011  
and entitled \emph{Infinite Monster Groups}.
I thank all the organizers for giving me these opportunities.

Almost all books contain errors, and these \emph{Notes} will be no exception. 
I thus encourage readers to communicate to me misprints and errors,
but also comments, criticism, questions and suggestions.
Their contributions will be well-come.
\bigskip

\hfill
Feldis,  February  2013
\vspace{3cm}

{\small
\begin{center}
D\'epartement de Math\'ematiques, Chemin du Mus\'ee 23, \\
 Universit\'e de Fribourg, 1700 Fribourg (Switzerland)\\
 ralph.strebel@unifr.ch
 \end{center}
 }

\setcounter{tocdepth}{2}
\markboth{Contents}{Contents}
\tableofcontents
%
%
\chapter{Introduction}
\label{chap:Intro-Part-1}
%
{
\renewcommand{\theequation}{\arabic{equation}}
\renewcommand{\thethmNN}{\arabic{thmNN}}
\markboth{Introduction to Part 1}{Introduction to Part 1}
The origin of the Sigma invariants lies in an answer to a question G. Baumslag had raised in 1973.
In the early 1970s, G. Baumslag and V. N. Remeslennikov had discovered independently 
that there are many more finitely generated metabelian groups with a finite presentation
that one might have suspected.
They proved, in particular, that \emph{every finitely generated metabelian group embeds into a a finitely presented metabelian group} (\cite{Bau73}, \cite{Rem73a}).
G. Baumslag therefore to ask (\cite[p.\;70]{Bau74}):
\begin{problemNN}
\label{problem:G.Baumslag}
Is there any way of discerning finitely presented metabelian groups
from the other finitely generated metabelian groups?
\end{problemNN}
\index{Baumslag, G.}
\index{Remeslennikov, V. N.}

Schur's multiplicator $H_2(-,\Z)$ is not sufficiently discriminating to settle this problem.
This fact called for a search of additional  conditions 
that a finitely presentable metabelian or, more generally, soluble group must satisfy.
One such condition was published by R. Bieri and R. Strebel  in 1978 (cf.\;\cite[Thm.\;A]{BiSt78}):
\setcounter{thmNN}{1}
\begin{thmNN}
\label{thm:BiSt78-Intro}
Assume $G$ is a soluble group  and $N$ is a normal subgroup with infinite cyclic quotient.
Choose an element $t \in G$ that generates a complement of $N$.
If $G$ can be finitely presented, 
then $G$ is an ascending HNN-extension over a \emph{finitely generated base group} $B$ contained in $N$.
\index{Bieri, R.}
\index{Strebel, R.}

More precisely, there exists a sign $\varepsilon$ 
so that $u = t^\varepsilon$ and $B$ satisfy the conditions
\begin{equation}
\label{eq:Ascending-HNN-Intro}
B  \subseteq u B u^{-1} \quad\text{and}\quad \bigcup\nolimits_{j \geq 0} u^j B u^{-j} = N.
\end{equation}
\end{thmNN}

Suppose now that $G$ is a finitely presented soluble group
whose abelianization $G_{\ab} = G/G'$ has torsion-free rank $r_0(G_{\ab})$ greater than 1.
Then $G$ contains infinitely many normal subgroups $N$
to which Theorem \ref{thm:BiSt78-Intro} can be applied.
To express the impact of all these applications, 
one needs a space made up of all possible couples $(N, \varepsilon)$
and a subset of this space that records the couples 
for which requirement \eqref{eq:Ascending-HNN-Intro} is satisfied.
Such a space is the rational sphere $S_\Q(G)$ associated to $G$ 
and the subset recording the answers is a primitive version of an Sigma invariant.
In the refined version,
the rational sphere $S_\Q(G)$ is replaced  by its ambient real sphere $S(G)$
and conditions \eqref{eq:Ascending-HNN-Intro}
are supplemented by a new type of requirement
which makes sense for every point of the sphere $S(G)$. 
This new condition was discovered around 1985 by Bieri, Neumann and Strebel;
the answer to Problem \ref{problem:G.Baumslag},
as given in \cite{BiSt80}, uses a weaker substitute. 
\index{Neumann, W. D.}

\subsection{Sigma invariants and their ambient sphere}
\label{ssec:Intro-invariants}

I next describe some invariants that will play a rôle in this monograph, 
beginning with their ambient sphere.

\subsubsection*{The sphere $S(G)$}
Common to all invariants denoted by Sigma is the fact
that they associate to a group $G$ a subset of a sphere $S(G)$.
Typically, $G$ is an infinite, finitely generated group;
for certain invariants the group must satisfy more stringent conditions.
The sphere $S(G)$ is derived from the real vector space $\Hom(G, \R)$; 
this vector space consists of all homomorphisms  $\chi \colon G \to \R$ of $G$ 
into the additive group of the field $\R$.
In case $G$ is the fundamental group of a path connected topological space $X$,
$\Hom(G, \R)$  is canonically isomorphic to 
$\Hom( \Cohom_{1}(X, \Z), \R)$ and  to $\Cohom^1(X, \R)$.

Each non-zero homomorphism $\chi \colon G \to \R$ gives rise to a submonoid of $G$,
namely
\begin{equation}
\label{eq:Definition-G-sub-chi}
G_\chi = \{ g \in G \mid  \chi(g) \geq 0 \} = \chi^{-1} ( [0, \infty)).
\end{equation}
This submonoid  does not change if $\chi$ is replaced by a positive multiple;
the monoids are therefore parametrized by the open rays in $\Hom(G, \R)$
emanating from the origin.
These rays are the points of the space $S(G)$, 
called \emph{character sphere} of $G$.
If $G$ is finitely generated, 
the real vector space $\Hom(G, \R)$ is finite dimensional 
and carries a canonical topology induced by its norms.
The space $S(G)$ equipped with the quotient topology is then homeomorphic 
to the unit sphere in a Euclidean vector space of the appropriate dimension. 

\subsubsection*{Homotopical and homological invariants}
There are two kinds of  Sigma invariants.
Those of the first sort depend only on a group $G$ and are often called \emph{homotopical},
partly because the definition of some of them is
in terms of homotopical properties of a space associated to $G$,
partly because they depend only on $G$, 
similar to the homotopy groups $\pi_{n} (X, x_{0})$ which depend only on a pointed space $(X,x_{0})$.
The invariants of the other type depend on a couple $(G, A)$ consisting of a finitely generated group $G$
and a $\Z{G}$-module $A$ or, more generally, a $G$-operator group.
These invariants are called \emph{homological} for analogous reasons.

\subsubsection*{The invariant $\Sigma^0(G;A)$}
\label{Intro_Invariant-Sigma0}
The ancestor of all later invariants was  introduced in \cite{BiSt80} and denoted there by $\Sigma_A(G)$.
It depends on a finitely generated \emph{abelian} group $G$ and a finitely generated $\Z{G}$-module $A$;
it is thus of homological type.
Its definition can be stated easily and shows how the submonoids enter into play;
so I shall give it here in spite of the fact that it plays no r{\^{o}}le 
in this version of the monograph. 
I shall, however, use the notation proposed by R. Bieri and B. Renz in \cite{BiRe88}.
\begin{equation}
\label{eq:Notice-Definition-Sigma-sub-A}
\Sigma^0(G; A)= \{ [\chi| \in S(G) \mid  A \text{ finitely generated over the monoid ring } \Z{G_\chi\}}.
\end{equation}

In terms of this invariant, the main results of \cite{BiSt80}, Theorems A and B, can now be expressed as follows:
\begin{thmNN}
\label{thmNN:Characterization-fp-metabelian-group}
Let $H$ be a finitely generated group, $G= H_{\ab}$ its abelianization
and $A = H'/H''$ the abelianization of its derived group $H'$ viewed as a \emph{left} $\Z{G}$-module via conjugation.
Then $A$ is a finitely generated $\Z{G}$-module and the following statements hold:
\begin{enumerate}[(i)]
\item if $H$ is finitely related and soluble, then $\Sigma^0(G; A) \cup -\Sigma^0(G; A) = S(G)$;
\item if $H$ is metabelian and $\Sigma^0(G; A) \cup -\Sigma^0(G; A) = S(G)$ then $H$ is finitely related.
\end{enumerate}
Here $-\Sigma$ denotes the image of $\Sigma$ under the antipodal map  $[\chi] \mapsto [-\chi]$ of  $S(G)$.
\end{thmNN}
\index{Bieri, R.}
\index{Strebel, R.}

\subsubsection*{Generalizing the invariant $\Sigma^0(G; A)$}
In \cite{BiSt80} the invariant $\Sigma^0(G;A)$ is only defined for finitely generated \emph{modules} $A$ 
over the group ring of a finitely generated \emph{abelian} group $G$.
Definition  \eqref{eq:Notice-Definition-Sigma-sub-A},
however,
admits of an obvious generalization to finitely generated modules over an arbitrary finitely generated group.
This extension appeared in \cite{Str81b}.
More radical generalizations were then proposed in \cite{BNS}.
\index{Bieri, R.}
\index{Strebel, R.}
\index{Neumann, W. D.}

These later generalizations are far from evident and so I would like to say a word on their motivation;
to do so,
I give some details of the proof of assertion (i) in Theorem \ref{thmNN:Characterization-fp-metabelian-group}.
This assertion claims
that the invariant $\Sigma^0(H_{ab};H'_{ab})$ of a finitely related soluble group $H$ must
contain at least one point from every pair $\{[\chi], [-\chi] \}$ of antipodal points.
The proof consists, roughly speaking, in representing the derived group  $H'$  of $H$
as a free product with amalgamation $M_{-}\star_{M_{0}} M_{+}$,
the representation depending on the point $[\chi] \in S(G)$.
Since $H'$ is soluble, this representation allows one to infer, after possible readjustment,
that $H'$ is either $M_{+}$ or $M_{-}$.
An analysis of the generating systems of $M_{+}$ and $M_{-}$ then leads to the conclusion
that  $A$ is finitely generated either over $\Z{G_{\chi}}$ or over $\Z{G_{-\chi}}$.

The above outline explains the purpose of $\Sigma^0(G;A)$:
it records local properties of the $\Z{G}$-module $A$.
In the proof of claim (ii)
these local data are then used to find a finite set of relators of the metabelian group $H$.
In this search, 
two properties of the sphere $S(G)$ and the invariant are exploited:
the sphere is compact  and $\Sigma^0(G;A)$ is an open subset of the sphere 
that comes equipped with a canonical open covering.

The method used in proving claim (ii) of Theorem \ref{thmNN:Characterization-fp-metabelian-group}
can be employed to obtain a result dealing with the module $A$.
It states that $A$ is finitely generated as an abelian group 
if, and only if, $\Sigma^0(G;A)$  coincides with $S(G)$.
Actually a relative version of this result is also valid:
instead of the entire sphere one considers great subspheres defined by
\begin{equation}
\label{eq:Definition-S(G,G1)}
S(G, G_{1})  = \left\{ [\chi] \in S(G) \mid \chi(G_{1}) = \{0 \} \right\}.
\end{equation}
Here $G_{1}$ is a subgroup of $G$. 
The relative version now reads:
\begin{equation}
\label{eq:Finite-generation-over-G1}
S(G, G_{1}) \subseteq \Sigma^0(G;A) \Longleftrightarrow 
A \text{ is finitely generated over the subring } \Z{G_{1}}.
\end{equation}
The right hand side of this equation admits of a reformulation in simpler terms:
let $H_{1}$ be the preimage of $G_{1}$ under the canonical projection $\pi \colon H \epi G = H_{\ab}$.
Then the statement on the right  holds precisely  if the group $H_{1}$ is finitely generated.

\subsubsection*{The invariant $\Sigma(G)$}
The characterization just stated is generalized in \cite{BNS} 
from finitely generated metabelian to arbitrary finitely generated groups.
To do so,
new invariants are introduced, two in the homotopical and two in the homological vein. 
Here I shall restrict attention to the most important among them;
it  is called $\Sigma(G)$ or, more explicitly $\Sigma_{G'}(G)$, in \cite{BNS}
and defined like this:
\begin{equation}
\label{eq:Notice-Definition-Sigma-sub-Gprime}
\Sigma_{G'}(G) = 
 \{ [\chi| \in S(G) \mid   G' \text{ is \emph{fg}  over a \emph{fg} submonoid of } G_\chi\}.
\end{equation}
The insistence on $G'$ of being finitely generated, not over the entire,
but over a finitely generated submonoid of $G_{\chi}$,
renders many of the proofs in \cite{BNS} technical and lengthy and may seem contrived.
Notwithstanding this quibble,
the invariant yields the characterization mentioned at the beginning of this section:

\begin{thmNN}
\label{thmNN:Characterization-fg-normal-subgroup}
Let $N$ be a normal subgroup of the finitely generated group $G$  with $G/N$ abelian.
Then $N$ is finitely generated if, and, only if, $S(G,N) \subseteq \Sigma(G)$.

In particular, $G'$ is finitely generated if, and only it, $\Sigma(G) = S(G)$.
\end{thmNN}

The question now arises whether there is a simpler, alternate definition of $\Sigma(G)$.
The paper \cite{BNS} contains such a definition (stated in part (ii) of Proposition 3.4),
but at the time of writing the paper 
none of the three authors seems to have noticed the geometric content of this new definition.
Only later did Robert Bieri and, independently, Ga{\"e}l Meigniez 
(see \cite[Section 6]{Mei87} or \cite[Theorem 3.19]{Mei90})   
detect 
that this new definition can be restated by saying 
that a certain subgraph of the Cayley graph of $G$ is connected.
\index{Bieri, R.}
\index{Meigniez, G.}

\subsubsection*{The invariant $\Sigma^1(G)$}
The invariant $\Sigma(G)$  
when expressed in terms of Cayley graphs 
has been denoted by $\Sigma^1(G)$  since the late 1980s; 
the new definition can be stated as follows.

Let $G$ be a finitely generated group,
$\XX$ a finite generating set of $G$
and let $\Gamma(G, \XX)$ be the Cayley graph of $G$ with respect to the generating set $\XX$.
For each non-zero homomorphism $\chi \colon G \to \R$,
let $\Gamma(G, \XX)_{\chi}$ denote the subgraph induced by the submonoid $G_{\chi}$.
Set
\begin{equation}
\label{eq;Notice-Definition-Sigma1}
\Sigma^1(G) = 
\{ [\chi| \in S(G) \mid  \text{the subgraph }  \Gamma(G, \XX)_{\chi}  \text{ is connected}\}.
\end{equation}
Then $\Sigma^1(G)$ does not depend on the choice of the finite generating set $\XX$ 
and it coincides with the invariant $\Sigma(G) =\Sigma_{G'}(G) $ of \cite{BNS}.

The fact that $\Sigma^1(G)$ does not depend on $\XX$ reminds one of a similar fact 
underlying the definition,
in terms of resolutions  $\mathbf{F} \epi \Z$,
of the homology groups $\Cohom_{*}(G,\Z)$ of the $\Z{G}$-module $\Z$;
this fact can be exploited just as one does with resolutions:
often a good choice of a generating set $\XX$ renders the computation of $\Sigma^1(G)$ easy.

So far, the definition of  $[\chi] \in \Sigma^1(G)$ 
in terms of the connectedness of the subgraph $\Gamma(G,\XX)_{\chi}$  has been presented 
as a requirement that is simpler and more familiar than the condition 
stated in definition \eqref{eq:Notice-Definition-Sigma-sub-Gprime} of $\Sigma(G)$.
But this is only one aspect of the new definition.
More important is the discovery
that several fundamental results about $\Sigma^1(G)$ ---
established  in \cite{BNS}  for $\Sigma(G)$ ---
can be proved for $\Sigma^1(G)$ by geometric arguments.
These proofs have been found by Robert Bieri in the late 1980s.
In addition, 
the definition of $\Sigma^1(G)$ in terms of connectedness of subgraphs of the Cayley graph 
indicates an avenue to analogous invariants $\Sigma^m(G)$ in higher dimensions:
loosely speaking,
the Cayley graph gets replaced by an $m$-dimensional space 
and connectivity by $(m-1)$-connectivity.
A main goal of these generalizations
--- carried out by Burkhardt Renz in his thesis \cite{Ren88} under the supervision of Robert Bieri ---
is an analogue of Theorem \ref{thmNN:Characterization-fg-normal-subgroup} for higher dimensions.
In it,
the finite generation of the normal subgroup $N$ is replaced by the property of being of type $F_{m}$.
(A group $N$ is of type $F_{m}$ 
if there is an Eilenberg-MacLane space for $N$ with finite $m$-skeleton.) 
\index{Bieri, R.}
\index{Renz, B.}

%
\subsection{Contents of the actual version}
\label{ssec:Contents-Part-1}
%
This version comprises chapters \ref{ch:Sigma-1-Cayley-graph}, 
\ref{ch:Sigma-1-Cayley-graph-complements} and \ref{ch:Alternative-definitions-overview} 
and a glimpse of two chapters that will be included in later versions.
\subsubsection*{Chapter \ref{ch:Sigma-1-Cayley-graph}}
\label{sssec:Contents-Chapter-A-Intro}
The chapter begin with a few remarks about the sphere, associated to a finitely generated group, 
its geometrical structure and about Cayley graphs.
The invariant $\Sigma^1$ itself is introduced in Section \ref{sec:Introducing-invariant}.
In the following three sections, fundamental results about $\Sigma^1$ are established:
first the so-called $\Sigma^1$-criterion and the openness of $\Sigma^1$ in its ambient sphere.
Theorem \ref{thmNN:Characterization-fg-normal-subgroup} is proved in the next section,
while the final section \ref{sec:Sigma1-and-fp-groups}
deals with a far reaching generalization of claim (i) in Theorem \ref{thmNN:Characterization-fp-metabelian-group}.
The proofs in Chapter \ref{ch:Sigma-1-Cayley-graph} use intensively  the fact 
that $\Sigma^1$ is defined in terms of the connectivity of certain graphs and have a geometric ring.
\subsubsection*{Chapter \ref{ch:Sigma-1-Cayley-graph-complements}}
\label{sssec:Contents-Chapter-B-Intro}
In Chapter \ref{ch:Sigma-1-Cayley-graph}, the emphasis is on basic properties of $\Sigma^1$;
some examples, computations and applications are also given,
but their goal is mainly to illustrate definitions and results.
Chapter \ref{ch:Sigma-1-Cayley-graph-complements} supplements the first chapter 
by additional examples and detailed computations.
It starts out with some words on the morphisms of spheres induced by homomorphisms $\varphi \colon G \to G_1$.
These morphisms provide useful aids in computing the invariant of a group.
As an illustration, 
the invariant of graph groups (alias right angled Artin groups) is determined
(see Theorem \ref{thm:Sigma1-right-angled-Artin-group}). 

In Section \ref{sec:Sigma1-criterion-revisited}  
two algorithms are described that produce lower bounds for $\Sigma^1(G)$.
The input for both algorithms is a set of relators satisfied by the chosen set of generators;
typically, neither bound will coincide with $\Sigma^1(G)$.
The algorithms are illustrated by examples of groups of piecewise linear homeomorphisms of the real line 
and by groups of links.
 The third section explores another theme:
the sphere $S(G)$ of a finitely generated group $G$ contains \emph{points of rank 1};
they are represented by homomorphisms $\chi \colon G \to \R$ with infinite cyclic image.
A point $[\chi]$ of rank 1 lies in $\Sigma^1(G)$ if, and only if, 
$G$ is an \emph{ascending HNN-extension with finitely generated base group} $B \subseteq N = \ker \chi$;
see the statement of Theorem  \ref{thm:BiSt78-Intro} for more details.
The mentioned result is actually only one out of several consequences of a structure theorem 
for \emph{finitely presented groups} $G$ admitting a rank 1 character.
In Section \ref{sec:Sigma1-via-HNN-extensions}, 
this structure theorem is established and four of its consequences are discussed.

In the last section,
an algorithm, due to Ken Brown \cite{Bro87b}, is derived;
it allows one to compute the invariant of a group
that is given by a presentation with a single defining relation.
\index{Brown, K. S.}
%
\subsubsection*{Chapter \ref{ch:Alternative-definitions-overview}}
\label{sssec:Contents-Chapters-C-Intro}
One peculiarity of the invariant $\Sigma(G) = \Sigma_{G'}(G)$ was already noticeable in the account \cite{BNS}:
$\Sigma(G)$ has alternate definitions which, at first sight, seem unrelated to it. 
Each such definition opens up new horizons;
for this reason they are precious.

In Chapters \ref{ch:Sigma-1-Cayley-graph}
and \ref{ch:Sigma-1-Cayley-graph-complements}
the Cayley graph definition is at center stage.
Chapter \ref{ch:Alternative-definitions-overview}
redresses the balance by giving voice to alternate definitions:
they range from the main definition used in \cite{BNS},
the description in terms of actions on $\R$-trees detected by Ken Brown \cite{Bro87b},
to definitions involving closed one-forms (cf. \cite{Lev87}).
Also established in Chapter \ref{ch:Alternative-definitions-overview}
is the equality of $\Sigma^1$ with the invariant of homological type $\Sigma^1(-; \Z)$ 
introduced by R. Bieri and B. Renz in \cite{BiRe88}.
\index{Bieri, R.}
\index{Brown, K. S.}
\index{Levitt, G.}
\index{Renz, B.}
%
\subsubsection*{Chapter D (preview)}
\label{sssec:Contents-Chapters-D-Intro}
In this chapter, three alternate definitions,
 introduced in Chapter \ref{ch:Alternative-definitions-overview},
will be studied in greater detail.
The invariant  $\Sigma^0(G;A)$ is the topic of the first two sections.
In Section D1, we discuss results 
that are available for arbitrary finitely generated groups $G$ and finitely generated $G$-modules $A$.
If $G$ is polycyclic, these results can be sharpened.
In the second section, the group $G$ is required to be (finitely generated) \emph{abelian};
this is the original set-up of the invariant, introduced and exploited in \cite{BiSt80}.
In this special set-up, a new tool is available: 
the invariant $\Sigma^0(G;A)$ depends only on the commutative ring $R = \Z{G}/\Ann_{\Z{G}}(A)$;
moreover,
characters $\chi \colon G \to \R$ representing points outside of $\Sigma^0(G;A)$ are characterized  by the fact
that they can be extended to valuations $v \colon R \to \R \cup \{\infty \}$ (see \cite{BiGr84}).
\index{Bieri, R.}
\index{Renz, B.}

In Section D3 we shall turn to the invariant $\Sigma_A(G)$ propounded in \cite{BNS}
and establish results that have no direct counter-parts in the theory of $\Sigma^1$
and which allow one to deduce very useful properties of $\Sigma^1$.
In the final section D4, 
we shall give an algebraic characterization of transitive measured $G$-trees.

%
\subsubsection*{Chapter E (preview)}
\label{sssec:Contents-Chapter-E-Intro}
%
The invariant $\Sigma^1$ is defined for the class of all finitely generated groups,
hence for $2^{\aleph_{0}}$ isomorphism types of groups.
It encapsulates non-trivial information for each of these groups,
provided, of course,
one can compute the invariant with sufficient precision.
Over the years, various classes of groups have been detected 
which are of interest to a wider segment of group theorists
and for which $\Sigma^1$ can be computed.

One such class consists of graphs groups; 
its invariant is determined in section \ref{sec:Computing-Sigma1-via-change-groups}.
Another one is formed by one-relator groups; their invariant can be worked out by an algorithm
(see \ref{sec:Invariant-one-relator-group}).
In Chapter \ref{ch:Alternative-definitions-overview},
the invariants of two more classes are calculated, 
those of a sequence of groups introduced by C. H. Houghton  in \cite{Hou78} 
and those of wreath products.

The listed classes are by now means the only ones 
for which the computation of $\Sigma^1$ is possible and yields valuable insights.
Chapter E will aim at presenting the most interesting of them.
\index{Houghton, C. H.}

%
%
}

\mainmatter
\chapterstyle{main}
\pagestyle{main}
\setsecnumdepth{subsubsection}
%
%
\chapter{The invariant $\Sigma^1$}
\label{ch:Sigma-1-Cayley-graph}
%
%
In this chapter, 
the first $\Sigma$-invariant is introduced and some of its properties are established.
This invariant will be defined in terms of connectivity properties of subgraphs of Cayley graphs 
and denoted by $\Sigma^1$.
It has alternate definitions; some of them will be discussed in Chapter \ref{ch:Alternative-definitions-overview}.
%
%
\section{Setting the stage}
\label{sec:Preliminaries}
%
This section sets the stage for the invariant $\Sigma^1(G)$ of a finitely generated group $G$.
The invariant is a subset of a sphere $S(G)$ and consists of those points $[\chi]$ 
for which an associated subgraph $\Gamma_{\chi}$ of the Cayley graph $\Gamma(G, \XX)$ is connected.

Accordingly, 
the section starts out with the definition of the sphere $S(G)$, 
lists some of its properties and then turns to graphs.
The invariant $\Sigma^1(G)$ will enter scene only in section
\ref{sec:Introducing-invariant}.
%
\subsection{Characters and the character sphere}
\label{ssec:Character-sphere}
%
Let $G$ be a finitely generated group.
A homomorphism $\chi \colon  G \to \R$ into the additive group of the field of real numbers 
will be called a  \emph{character of} $G$.
\index{Notation!chi@$\chi$}%
\index{Character of a group!definition}
The set $\Hom(G,\R)$ of all characters of $G$ has the structure of  a real vector space;
\index{Notation!Hom(G,G)@$\Hom(G,\R)$}%
as $G$ is finitely generated it is finite dimensional.
Its dimension is the subject of
\begin{lem}
\label{lem:Dimension-S(G)}
The dimension of $\Hom(G,\R)$ is equal 
to the torsion-free rank of the abelianization  $G_{\ab} = G/[G,G]$ of $G$.
\end{lem}

\begin{proof}
Since the additive group of $\R$ is a torsion-free abelian group,
every character $\chi \colon G \to \R$ factors  over the canonical projection 
\[
\can \colon G  \epi G_{\ab} = G/[G,G] \epi \overline{G} = G_{\ab}/T(G_{\ab}).
\]
Here $T(G_{\ab})$ denotes the \emph{torsion-subgroup}  of the  abelianization $G_{\ab}$ of $G$; 
it is finite.
This projection induces an isomorphism  of vector spaces
\[
\can^* \colon \Hom(\overline{G}, \R) \iso \Hom(G,\R).
\]
The quotient group $\overline{G}$, being a finitely generated torsion-free abelian group, 
is free abelian of rank $n$, say.
Since a homomorphism  $\bar{\chi} \colon \overline{G} \to \R$ 
is  determined by  its images on a  basis of $\overline{G}$
and as these images can be prescribed arbitrarily, 
the dimension $\dim_{\R}\Hom(G, \R)$ of the real vector space  $\Hom(G,\R)$
equals the rank of the free abelian group $\overline{G}$.
This rank, in turn, coincides with the torsion-free rank $r_{0}(G_{\ab})$ of $G_{\ab}$.
(For the justifications of the claims about abelian groups made in the above 
see, e.\;g., \cite[Section 4.2]{Rob96} or \cite[Section 5.1]{LeRo04}.) 
\index{Notation!rank@$r_{0}$}
\end{proof}

%
\subsubsection{Equivalent characters} 
\label{sssec:Equivalence-characters}
%
Two characters $\chi_1$, $\chi_2$ of $G$ will be called \emph{equivalent}
 if there is a positive real number $r$ with $\chi_1 = r  \cdot \chi_2$. 
The equivalence class $[\chi]$ of a non-zero  character $\chi$ 
is  the ray emanating from $0$ and passing through $\chi$. 
\index{Character of a group!equivalence relation}
\index{Notation![chi]@$[\chi]$}%
The set of all equivalence classes
\begin{equation}
\label{eq:Equivalent-characters}
S(G) = \left\{[\chi] \mid \chi \in \Hom(G,\R) \smallsetminus \{ 0 \} \;\right\}
\end{equation}
together with the structure inherited from $\Hom(G,\R)$
will be called  the \emph{character sphere of the group} $G$.
\index{Character sphere!definition}
\begin{remark}
\label{Remark:Empty-sphere}
If the abelianized group $G_{\ab}$ is finite ,
the vector space $\Hom(G, \R)$ is reduced to the zero homomorphism 
and the character sphere $S(G)$ is empty. 
The methods and results of this monograph are of no interests for such groups.
\end{remark}
\smallskip

In the remainder of section \ref{ssec:Character-sphere}  
some of its properties of the sphere $S(G)$ are discussed.
Further features will be studied in section \ref{ssec:Morphisms}.
 %
\subsubsection{Rank of a point} 
\label{ssec:Rank-point}
%
We begin with a stratification of the sphere.
Each character $\chi \colon G \to \R$ has a \emph{rank} 
given by the $\Z$-rank of the additive group $\chi(G) \subset \R$.
Equivalent characters have the same rank,
so one can speak of the rank of a point $[\chi] \in S(G)$; it will be denoted by  $\rk[\chi]$. 
\index{Character of a group!rank}
\index{Character sphere!rank}
\index{Character of a group!rational}
The points of rank $1$  are also called \emph{rational}. 
A point is rational if and only 
if it can be  represented by a character $\chi \colon G \to \R$  with $\chi(G) = \Z$. 
The set of all rational points of $S(G)$ is the sphere of rational characters,  denoted $S_{\Q}(G)$;
it is dense in $S(G)$ with respect to the topology defined below
(see Lemma \ref{lem:Density-rank-1-points}).
%
\subsubsection{Geometric structure}
\label{sssec:Geometric-structure-sphere}
%
The sphere $S(G)$ inherits from the  $\R$-linear structure of $\Hom(G,\R)$  
the following geometric features:
\begin{itemize}
\item 
a \emph{topology:} 
since $V = \Hom(G,\R)$ is a finite dimensional real vector space,
all norms $\| \cdot \|$ on it induce the same topology. 
The sphere $S(G)$, being a quotient of the open subset  $V \smallsetminus \{0\}$, 
inherits the quotient topology.
Equipped with this topology,   it is is homeomorphic to the unit sphere $\s^{n-1}$ in the Euclidian space of dimension $ n = r_0(G_{\ab})$.
 \index{Character sphere!topology}%
\item 
a \emph{family of subspheres:}
every subset  $\SS$ of $V =\Hom(G,\R)$ 
gives rise to a great subsphere  $\left\{[\chi] \mid \chi(\SS) = \{ 0\}  \; \right \}$ of $S(G)$. 
Particularly  important are the subspaces consisting of a characters 
that vanish on a subgroup $H \leq G$. 
 The corresponding subsphere will then be denoted by
\begin{equation}
\label{def:S-of-G-and-H}
S(G,H) = \left\{ [\chi] \mid \chi(H) = \{ 0 \} \right\}.
\end{equation}
\index{Character sphere!subspheres}%
\index{Notation!Subsphere-G-H@$S(G,H)$}%
\item 
a  \emph{family of sub-hemispheres:}
every open (respectively closed) half space ${\HH}$ of $\Hom(G,\R)$ 
gives rise to an open (respectively closed) hemisphere 
\[
\left\{[\chi] \mid \chi \in {\HH} \setminus 0  \right\}.
\] 
\index{Character sphere!sub-hemispheres}%
\end{itemize}
%
%
\subsubsection{Introduction of coordinates} 
\label{sssec:Coordinates-sphere}
%
\index{Character sphere!coordinates}%
The sphere $S(G)$ and its topology have been defined without the use of coordinates.
Occasionally, however, it  is convenient to work with a Euclidean model of $S(G)$.
In such a situation, coordinates will be introduced as follows.

Let $\E^n$ denote the  real vector space $\R^n$ with standard basis  $(e_1, \ldots, e_n)$ 
and equipped with the usual inner product $\langle - ,  - \rangle$, 
and let $\Z^n$ denote the standard lattice in $\E^n$. 
Consider now a finitely generated group $G$ the abelianization of which has torsion-free rank $n$.
Then the groups $ \overline{G} = G_{\ab} / T(G_{\ab})$  and $\Z^n$ are isomorphic
and every isomorphism  $ \iota \colon \overline{G}  \iso\Z^n$ 
induces an epimorphism $\vartheta =    \iota \circ \can \colon G\epi \Z^n$.
This epimorphism and the scalar product on $\E^n$ give rise to an isomorphism 
$\vartheta^* \colon  \E^n \iso  \Hom(G,\R)$
which sends the vector $v \in \E^n$ to the character  $g \mapsto \langle v, \vartheta(g) \rangle$.
Its restriction to the unit sphere of $\E^n$ induces then a homeomorphism of spheres
\begin{equation}
\label{eq:Definition-sigma-theta}
\sigma(\vartheta) \colon \s^{n-1} \iso S(G).
\end{equation}
\index{Notation!Chart-sigma-theta@$\sigma(\vartheta)$}%
It sends the unit vector $u \in \s^{n-1}$ to the point  $[\chi_u] \in S(G)$
where $\chi_u \colon G \to \R$ denotes the character
\begin{equation}
\label{eq:Definition-chi-u}
\chi_u(g) = \langle u, \vartheta(g) \rangle.
\end{equation}
%
%
\subsection{On graphs and Cayley graphs}
\label{ssec:Graphs}
%
The definition of $\Sigma^1(G)$ 
uses the Cayley graph $\Gamma(G,\XX)$ of the group $G$ 
with respect to a finite generating system $\XX$ of $G$.
For this reason, we continue with a few remarks on graphs and Cayley graphs.
%
\subsubsection{Terminology and notation used for graphs}
\label{ssec:Terminology-graphs}
%
Most of the graphs occurring in this monograph will be Cayley graphs of  groups.
The definition of such a graph uses \emph{oriented} edges;
but it does not furnish inverses of edges provided by the definition. 
On the other hand,
the geometric definition of a path presupposes
that an edge can be traversed in both directions of an oriented edge.
If the inverses of the edges are not furnished by the construction of the graph, 
they will therefore have to be added later on.
These facts have influenced the type of graph that will be favoured in the sequel:
it is a graph equipped with positively oriented edges $e^+$, but lacking inverses $e^-$ of the given edges:
\begin{definition}
\label{definition:Oriented-graph}
\index{Definition of!oriented graph}%
An  \emph{oriented graph} $\Gamma$ 
is given by a non-empty vertex set  $V = \ver(\Gamma)$,  
an edge set $E = \edg(\Gamma)$, 
and two maps $\iota \colon  E \to V$,   $\tau \colon  E \to V$  
assigning to an edge $e$ its origin $\iota(e)$ and its terminus $\tau(e)$, 
respectively. 
\end{definition}
\index{Notation!Vertexset@$\ver(\Gamma)$}
\index{Notation!Edgeset@$\edg(\Gamma)$}

Given an oriented graph $(V, E, \iota, \tau)$,
one constructs a new graph having the same vertex set,
but a larger edge set.
Specifically,
one introduces for each edge $e \in E$ a new element $e^{-1}$
such that the assignment $e \mapsto e^{-1}$ is injective
and the sets  $E^- = \{e^{-1} \mid e \in E\}$  and $E$ are disjoint. 
Then one extends the inversion  $E \to E^-$, $e \mapsto e^{-1}$ 
and the  endpoint maps $\iota$, $\tau$ to the union $E^\pm = E \cup E^-$ 
by  setting
\[
(e^{-1})^{-1} = e, \quad \iota(e^{-1}) = \tau(e) \text{  and  }\tau(e^{-1}) = \iota(e).
\]
Figure
\ref{fig:Various-type-graphs}
illustrates the concepts of a combinatorial graph, an oriented graph 
and a graph resulting from an oriented graph by adding inverse edges.
\begin{figure}[htbp]
\begin{center}
\includegraphics[width=3.6cm]{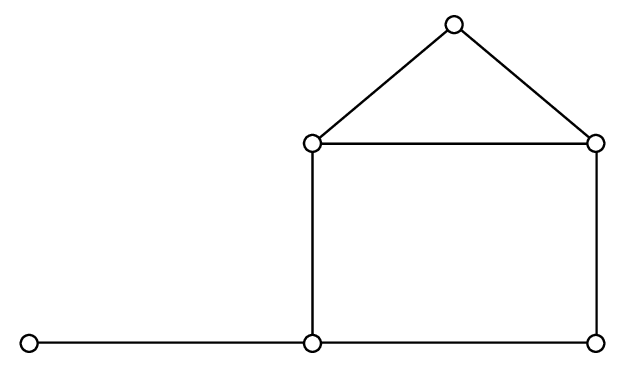}
\hspace*{0.3cm}
\includegraphics[width=3.6cm]{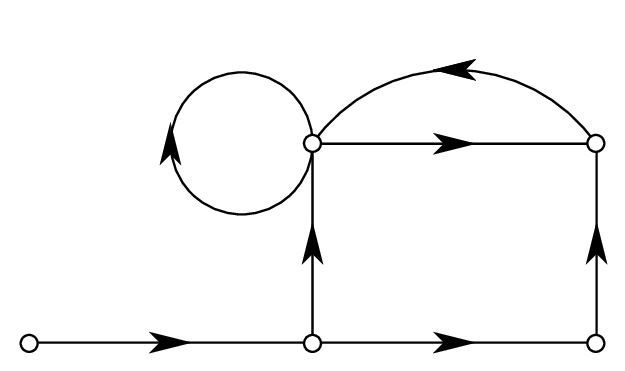}
\hspace*{0.3cm}
\includegraphics[width=3.6cm]{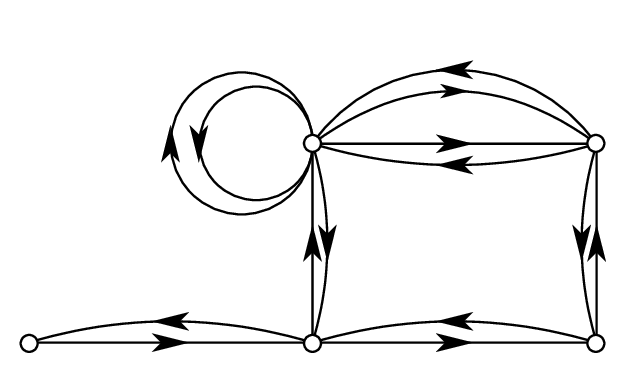}
\caption{A combinatorial, an oriented  and a completed graph}
\label{fig:Various-type-graphs}
\end{center}
\end{figure}

An edge path of length $m\geq 1$ in $\Gamma$ 
is a sequence of edges and  inverse edges $p = e_1e_2 \ldots e_m$ 
such that $\tau(e_i) = \iota(e_{i+1})$ for all $i = 1, \ldots,  m-1$. 
The set of all edge paths of $\Gamma$ will be denoted by  $P(\Gamma)$;
\index{Notation!P-Gamma@$P(\Gamma)$}%
the set $E^\pm$ can be thought of as a subset of $P(\Gamma)$.
One extends the maps $\iota$, $\tau$ and the inversion on $E^\pm$ to $P(\Gamma)$ 
by putting  $\iota(p) = \iota(e_1)$, $\tau(p) = \tau(e_m)$ and 
$p^{-1} =  e_m^{-1} e_{m-1}^{-1} \ldots e_1^{-1}$. 
It is also convenient to introduce for each vertex $v \in V$ 
a unique empty path $\varnothing_v \in P(\Gamma)$ 
with $\iota(\varnothing_v) = \tau(\varnothing_v) = v$ 
and  $\varnothing_v^{-1} = \varnothing_v$. 
If one identifies $v \in V$ with $\varnothing_v \in  P(\Gamma)$ 
one can consider the vertex set $V$ as a subset of  $P(\Gamma)$.

\begin{definition}
\label{definition:Group-acting-graph}
A group $G$ is said to \emph{act on an oriented graph} $\Gamma = (V, E, \iota, \tau)$ 
if $G$ acts on the sets $V$ and $E$ in such  a way
that the maps $\iota$, $\tau$ are equivariant. 
By putting $g.(e^{-1}) = (g.e)^{-1}$,  $g.(e_1e_2 \ldots e_m) = (g.e_1)(g.e_2) \ldots (g.e_m)$
and $g.\varnothing_v =  \varnothing_{g.v}$ 
the action of $G$ on $V \cup E$ can be extended to a  $G$-action on $P(\Gamma)$ 
so that $\iota(g.p) = g.\iota(p)$,  $\tau(g.p) = g.\tau(p)$ and $(g.p)^{-1} = g.p^{-1}$ for all 
$g \in G$ and $ p \in P(\Gamma)$.

If $a$ is a vertex or an edge of the graph, 
the stabilizer  $\{g \in G\mid g.a = a\}$ of $a$  will be denoted by $G_a$ 
If all stabilizers are trivial, the group  $G$ is  said to act \emph{freely}.
\end{definition}
%
\subsubsection{Terminology and notation used for Cayley graphs}
\label{sssec:Terminology-Cayley-graph}
Let $G$ be a group. 
The construction of Cayley graph of $G$ involves a generating set $\XX$ of $G$.
If $G/N$ is a quotient of $G$,
the canonical image  of  $\XX$ on $G/N$ generates, of course, the quotient group,
but the map $\XX \to G/N$ that sends $x $ to $xN$ may no longer be injective.
In the sequel,
we  therefore use generating subsets only in special situations
and work, in general, with maps $\eta \colon \XX \to G$ the image of which generates $G$;
such maps will be called \emph{generating systems} of $G$.
\index{Notation!eta@$\eta \colon \XX \to G$}%
To ease notation,
the symbol $\eta$ will often be suppressed.

\begin{definition}
\label{definition-Cayley-graph}
\index{Graph!Cayley graph}%
\index{Cayley graph!definition}%
\index{Notation!Gamma-G-XX@$\Gamma(G, \XX)$}%
Let $G$ be a group and  $\eta \colon \XX \to G$ a generating family of $G$. 
The \emph{Cayley graph} $\Gamma =  \Gamma(G, \XX)$ of $G$ with respect to $\XX$
is the oriented  graph 
\[
(G, G \times \XX, \iota, \tau)
\] 
with vertices the elements of $G$, 
(positively) oriented edges the pairs $e =  (g,x) \in G \times \XX$, 
and origin and terminus given by $\iota(e) = g$ and $\tau(e) = g\cdot \eta(x)$. 
\footnote{In the sequel, the product $g \cdot \eta(x)$ will often be written $g\cdot x$ or $gx$.}

The Cayley graph is equipped with a canonical $G$-action induced by left multiplication 
on the vertex set $G$ and on the first factor of the edge set $G \times \XX$.
\end{definition}

The  inverse edges of the Cayley graph form a set  $ G \times \XX^{-1}$;
here  $\XX^{-1}$ denotes a set that is disjoint from $\XX$ 
and related to $\XX$ with a bijection $\inv \colon x \mapsto x^{-1}$.
Origin, terminus and inverse of edges and the inverse edges are then given by the same expressions:
if $e = (g,y) \in G \times \XX^\pm$ then $\iota(e) = g$,  $\tau(f) = gy$ and  $e^{-1} = (gy, y^{-1})$.
(Of course, $(x^{-1})^{-1}$ is to be interpreted as $x$.)

An edge path $p = e_1e_2 \cdots e_m \in P(\Gamma)$ is uniquely determined 
by its origin $\iota(p) \in G$ and the word $y_1y_2 \ldots  y_m \in W(\XX^\pm)$ with $e_i = (g_i, y_i)$. 
Thus the set of all  paths  $P(\Gamma)$ can be identified with the set of pairs 
$p = (g,w)  \in G \times W(\XX^\pm)$; 
we shall use this identification  whenever convenient. 
Using this notation,
the origin, terminus and the inverse of $p$ are  then given by 
$\iota(p) = g$,  $\tau(p) = g\eta(w)$ and $p^{-1} =  (\tau(p),\, w^{-1})$. 
Finally, 
the action of the group $G$ on the set of edge paths $P(\Gamma)$ 
is given by the left $G$-action on the first factor of $G \times W(\XX^\pm)$.

\begin{examples}
\label{examples:Cayley-graphs}
\index{Cayley graph!examples|(}%
The Cayley graphs used in the monograph are those of infinite groups;
as only a small patch of such a graph can be illustrated by a figure,
the first two examples will be Cayley graphs of small finite groups.
\smallskip

a) The \emph{dihedral group}  $D_4$ of order 8 admits a  presentation with generators $a$, $t$ 
and defining relations $a^4 = 1$, $t^2 = 1$, $ta= a^3t$.
The Cayley graph $\Gamma(D_4,\{a,t\})$ has 8  vertices and 16 (positively) oriented edges;
see figure \ref{fig:Cayley-graphs-finite-groups}.
\begin{figure}[htbp]
\psfrag{1}{\small 1}
\psfrag{a}{\small $a$}
\psfrag{a2}{\hspace*{1mm}\small $a^2$}
\psfrag{a3}{\hspace*{-1mm}\small $a^3$}
\psfrag{t}{\small $t$}
\psfrag{at}{\hspace*{-2mm}\small $at$}
\psfrag{a2t}{\hspace*{-2mm}\small $a^2t$}
\psfrag{a3t}{\small $a^3t$}
\psfrag{ea}{\hspace*{-3mm}\small  $(1,a)$}
\psfrag{et}{\hspace*{-1mm}\small  $(1,t)$}
\psfrag{i}{\small $i$}
\psfrag{m1}{\hspace*{-1mm}\small $-1$}
\psfrag{mi}{\hspace*{-1.5mm}\small  $-i$}
\psfrag{j}{\hspace*{-1.5mm}  \small  $j$}
\psfrag{mj}{\hspace*{-1mm}\small  $-j$}
\psfrag{k}{\hspace*{-0.5mm}\small $k$}
\psfrag{mk}{\hspace*{0mm}\small  $-k$}
\psfrag{ei}{\hspace*{-3mm}\small  $(1,i)$}
\psfrag{ej}{\hspace*{-4mm}\small  $(1,j)$}
\begin{center}
\includegraphics[width=4.5cm]{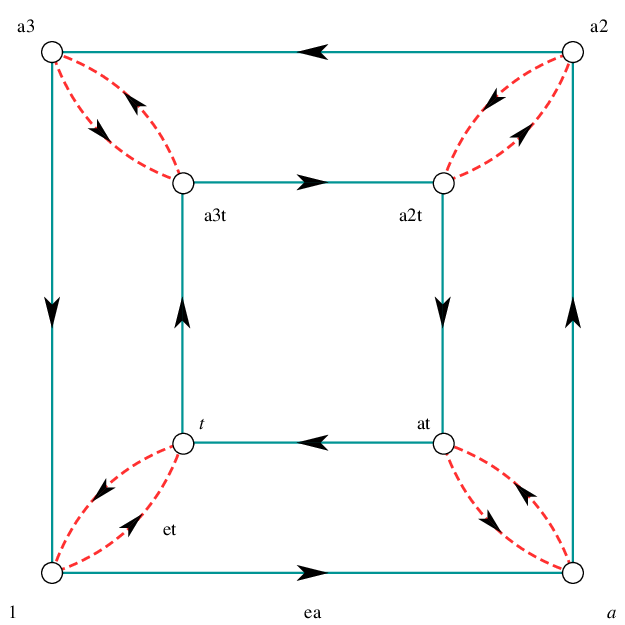}
\hspace*{1.5cm}
\includegraphics[width=4.5cm]{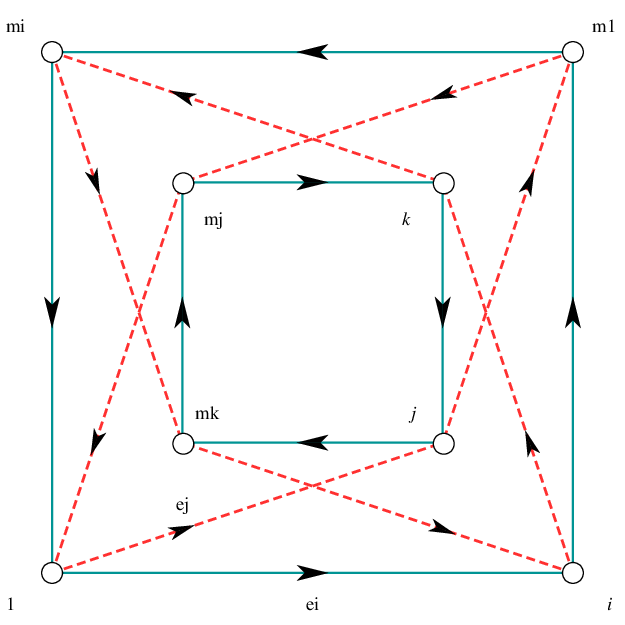}
\caption{Cayley graphs of the dihedral group and the quaternion group}
\label{fig:Cayley-graphs-finite-groups}
\end{center}
\end{figure}

b) The \emph{quaternion group} $Q_8$ of order 8 is the subgroup of the multiplicative group
of the field of quaternions $\H$. 
It is generated by the quaternions $i$ and $j$, both elements of order 4. 
Its Cayley graph depicted on the right of figure \ref{fig:Cayley-graphs-finite-groups}.

c) We now turn to Cayley graphs of infinite groups. The first one embeds into the plane,
but the second one will no longer have this exceptional property.
Let $G$ be the \emph{free abelian group of rank} 2 and  $\XX = \{a,b\}$ a basis of $G$. 
If one identifies $G$ with the standard lattice $\Z^2$ of the Euclidian plan $\R^2$, 
the Cayley graph of $G$ can be visualized by a Euclidean grid;
see figure \ref{fig:Cayley-graph-free-abelian-rank-2}.
\begin{figure}[htb]
\psfrag{1}{\hspace*{-1mm}\small$1$}
\psfrag{a}{\hspace*{1.5mm}\small $a$}
\psfrag{am1}{\hspace*{-5.8mm}\small $a^{-1}$}
\psfrag{a2}{\hspace*{-0.5mm}  \small $a^2$}
\psfrag{b}{\hspace*{-1mm}\small $b$}
\psfrag{bm1}{\hspace*{-2.3mm}\small $b^{-1}$}
\psfrag{b2}{\hspace*{-1.0mm}\small $b^2$}
\psfrag{ab}{\hspace*{0.5mm}\small $ab$}
\psfrag{ar1}{\hspace*{-2.5mm}\small $(1,a)$}
\psfrag{ar2}{\hspace*{-0mm}\small $(1,b)$}
\vspace*{3mm}
\begin{center}
\includegraphics[width = 5cm]{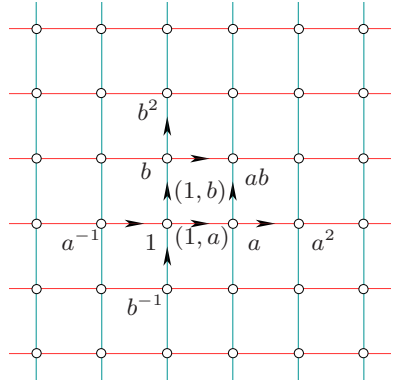}
\end{center}
\caption{Cayley graph of a free abelian group of rank 2}
\label{fig:Cayley-graph-free-abelian-rank-2}
\end{figure}
\smallskip

d) The Cayley graph of our last example is harder to visualize;
in this respect,
it  is  typical for the Cayley graphs of the groups that are at center stage in this monograph.
Let $\Z[\tfrac{1}{2}]$ denote the additive group of  all dyadic rational numbers 
acted on the left by an infinite cyclic group $C = \gp(s)$, 
in such a way that the generator $s$  acts by multiplication by 2, 
and let $G$ to the semi-direct product  $\Z[\tfrac{1}{2}]  \rtimes C$.
(This group is often referred to as Baumslag-Solitar group $BS(1,2)$.) 
The product of two elements 
$(x,s^m)$, $(x',s^{m'})$  of $\Z[\frac{1}{2}] \rtimes C $ is then given by the formula
\begin{equation}
\label{eq:Multiplication-metabelian-example}
(x,t^m) \cdot (x',s^{m'}) = ( x + 2^{m} \cdot x', s^{m + m'}).
\end{equation}

If one identifies $C$ with the additive group of the ring $\Z$ by means of  the map $t^m \mapsto m$, 
the elements of $G$ correspond to points  in the plane; 
these points are dense on each horizontal line $\R \times \{m\}$ with $m \in \Z$. 

The group $G$ is generated by the elements $a = (1, s^0)$ and $u = (0, s)$.
The Cayley graph $\Gamma$ of $G$ with respect to this generating set does not embed in the plane,
but it is still useful to describe it in geometric terms.
An edge $(g=(x,m),a)$ can be visualized by a segment of the horizontal line $\R \times \{m\} $
with origin $g = (x,m)$ and terminus $ga = (x + 2^m,m)$;
note that its length is not constant, but increases with increasing second coordinate.
An edge $(g=(x,m),u)$ corresponds to a segment of the vertical line $\{x\} \times \R$,
starting at  $g = (x,m)$ and  ending in $(x, m+1)$.

Our next aim is to convey an idea of the shape of the Cayley graph $\Gamma = \Gamma(G, \{a, u \})$.
Consider the subset $M = \Z \times \{s^m \mid m \in \N \}$.
The corresponding full subgraph has the form indicated by figure 
\ref{fig:Cayley-graph-metabelian-group-1}.
The Cayley graph is a union of such subgraphs. 
First of all, one has an infinite  collection of subgraphs arising from the given one by translating it to the right 
by  $x \in [0,1) \cap \Z[\tfrac{1}{2}]$ units. 
The union of these subgraphs is a graph $\Gamma_0$.
This graph $\Gamma_0$ is contained in an isomorphic subgraph $\Gamma_{-1}$ 
which arises from $\Gamma_0$ by rescaling the horizontal axis by $\tfrac{1}{2}$ 
and translating the rescaled graph by a unit towards the bottom. 
Similarly one defines a subgraph $\Gamma_{-2}$ 
having the same relationship to $\Gamma_{-1}$, 
as $\Gamma_{-1}$ has to $\Gamma_0$. And so on and so forth.
\begin{figure}[htb]
\psfrag{Z}{\hspace*{-1mm}\small $m$}
\psfrag{Z2}{\hspace*{0mm}\small $x$}
\psfrag{1}{\hspace*{-3mm}\small $1$}
\psfrag{a}{\hspace*{-1mm}\small $a$}
\psfrag{am1}{\hspace*{-2.5mm}\small $a^{-1}$}
\psfrag{a2}{\hspace*{-2mm}\small $a^2$}
\psfrag{t}{\hspace*{0mm}\small $u$}
\psfrag{t2}{\hspace*{0mm}\small $u^2$}
\psfrag{t3}{\hspace*{0mm}\small $u^3$}
\psfrag{ar1}{\hspace*{-2.5mm}\small $(1,a)$}
\psfrag{ar2}{\hspace*{-1.5mm}\small $(1,u)$}
\begin{center}
\includegraphics[width=11cm]{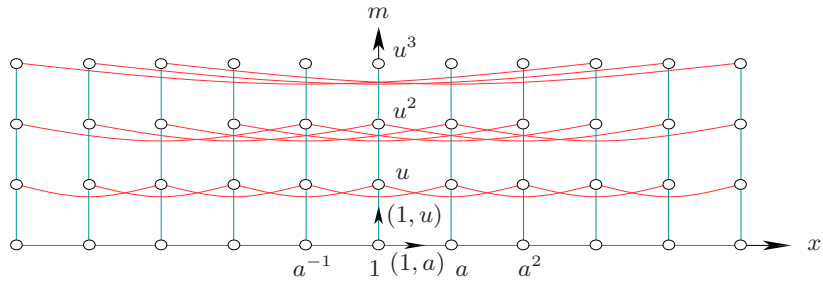}
\caption{A patch of the Cayley graph of the group defined by \eqref{eq:Multiplication-metabelian-example}}
\label{fig:Cayley-graph-metabelian-group-1}
\end{center}
\end{figure}
\end{examples}
\index{Cayley graph!examples|)}%
%
 
%
\section{Enter the invariant $\Sigma^1$}
\label{sec:Introducing-invariant}
%
In this section,
the invariant is defined and computed in some simple cases.
%
\subsection{Definition}
\label{ssec:Definition-Sigma1}
%
As in the previous section,
$G$ denotes a finitely generated group. 
Let $\eta \colon \XX \to G$ be a finite system of generators of $G$
and $\Gamma = \Gamma(G,\XX)$ the Cayley graph of $G$ with respect to $\XX$.
For every character $\chi \colon  G \to \R$ 
we define a subset
\begin{equation}
\label{def:G-sub-chi}
\index{Notation!G-sub-chi@$G_\chi$}
G_\chi = \{g \in G \mid \chi(g) \geq 0\}.
\end{equation}
This set is actually a submonoid of $G$;
we view it as a subset of the vertex set  of the Cayley graph $\Gamma$ 
and define $\Gamma_\chi = \Gamma(G, \XX)_\chi$ to be the subgraph of $\Gamma$ 
generated by the subset $G_\chi$.
Thus the vertex set of $\Gamma_\chi$ is the submonoid $G_\chi$, 
the set of oriented edges consists of all the  edges $(g,y)  \in G \times \XX$ 
with both $g \in G_\chi$ and $gy \in G_\chi$,
and the incidence functions  $\iota_*$, $\tau_*$ 
are induced by the functions $\iota$ and $\tau$ of $\Gamma(G,\XX)$.
If $\chi_1$ and $\chi_2$ are equivalent characters
the subsets $G_{\chi_1}$ and $ G_{\chi_2}$ coincide;
so   $G_\chi$ and $\Gamma_\chi$ depend only on  the point $[\chi] \in S(G)$.
\index{Notation!Gamma-sub-chi@$\Gamma_\chi$}

Since $\eta \colon \XX \to G$ is a generating system,
the graph $\Gamma(G,\XX)$ is connected.
The subgraph $\Gamma_\chi$, however, need not have this property.
This prompts us to introduce the subset 
\begin{equation}
\label{eqf:Definition-Sigma-one-of-G}
\index{Definition of!invariant Sigma1@invariant $\Sigma^1$}
\index{Invariant Sigma1@Invariant $\Sigma^1$!definition}
\Sigma^1(G) = \{[\chi] \mid \Gamma_\chi(G,\XX) \text{ is connected}\};
\end{equation}
it is often referred to as \emph{(homotopical) geometric invariant of the group} $G$.

\begin{remarks} 
\label{remarks:Definition-Sigma-1}
a) The definition of $\Sigma^1(G)$ involves the choice of a generating system $\XX$,
but the set $\Sigma^1(G)$ does not depend on this choice 
(see Theorem \ref{thm:Sigma1-well-defined} below); in addition. it is invariant under isomorphisms
(see section \ref{sssec:Isomorphisms-groups}).
These facts justify the epithet \emph{invariant}.
The adjective \emph{geometric} refers to the aspect that the mapping $G \mapsto \Sigma^1(G)$ 
associates to a (finitely generated) group $G$ a geometric object, namely a subset of the sphere $S(G)$. 
In section \ref{sec:Sigma1-criterion}, 
it will be proved that $\Sigma^1(G)$ is an open subset of $S(G)$.

b) The superscript 1 in the symbol $\Sigma^1$ is meant to indicate
that $\Sigma^1$ is the first member in a sequence of invariants $\Sigma^2$, $\Sigma^3$, \ldots; 
they are all subsets of the sphere $S(G)$,
but are defined in terms of higher connectivity of higher dimensional spaces associated to groups.
\end{remarks}
%
\subsubsection{Examples illustrating the definition}
\label{sssec:Sigma1-first-examples}
%
%
In the following examples 
the global features of the Cayley graphs are so clear
that the invariants can be determined by inspection.

1)
Let $G$ be the free abelian group of rank 2 and 
$\XX = \{a,b\}$ a basis of $G$. 
If we identify $G$ with the standard lattice $\Z^2 \subset \R^2$ 
the Cayley graph with respect to the generating set $\XX$ can be visualized as a Euclidean grid.
Every character $\chi \colon G \to \R$ extends to a unique $\R$-linear map  $\chi_{\R} \colon  \R^2 \to \R$; 
if $\chi \neq 0$,
 the subset $G_\chi$ is the intersection of the lattice $\Z^2$ 
 with the half plane $\{x \in  \R^2 \mid \chi_{\R}(x) \geq 0\}$. 
Figure \ref{fig:Half-of-Cayley-graph-free-abelian-rank-2}
indicates that $\Gamma\chi$ is connected;
a rigorous justification of this suggestion is easy and will be given later 
(see example \ref{examples:Groups-with-non-trivial-centre}a).
So  $\Sigma^1(G) = S(G)$.
\index{Computation of Sigma1@Computation of $\Sigma^1$ for!abelian groups}
\begin{figure}[htb]
\psfrag{pos}{\small $\chi > 0$}
\psfrag{neg}{\small $\chi < 0$}
\psfrag{nul}{\hspace*{-3mm}\small $\chi = 0$}
\begin{center}
\includegraphics[width = 7cm]{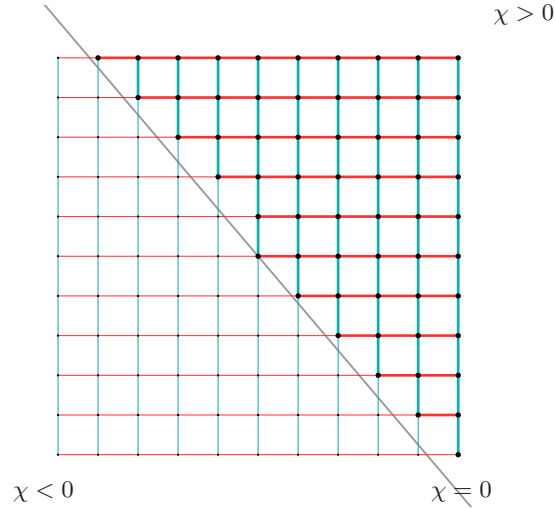}
\end{center}
\caption{Cutting  in half the Cayley graph of a free abelian group of rank 2}
\label{fig:Half-of-Cayley-graph-free-abelian-rank-2}
\end{figure}

2) The second example is the Baumslag-Solitar group $BS(1,2)$
whose Cayley graph has been discussed in Example  \ref{examples:Cayley-graphs}d).
\index{Soluble groups!examples of Baumslag-Solitar}%
The group $G$ is the semi-direct product of the abelian group  $\Z[\tfrac{1}{2}]$  by an infinite cyclic group 
$C = \gp(s)$, 
the generator $s$  acting by multiplication by 2. 
The product of two elements of $G$ is given by
\begin{equation}
\label{eq:Multiplication-metabelian-example-bis}
(x,s^m) \cdot (x',s^{m'}) = ( x + 2^{m} \cdot x', s^{m + m'}).
\end{equation}
The group $G$ is generated by the elements $a = (1, s^0)$ and $u = (0, s)$;
they satisfy relation $uau^{-1}= a^2$.
If one writes the relation in the form $ a^{-1}\cdot uau^{-1} = a$,
one sees that every character $\chi \colon G \to \R$ maps $a$ to 0.
So $S(G)$ is made up of only two points, 
the point $[\chi]$  with $\chi$ given by $\chi(u) = 1$ and its antipode $[-\chi]$. 

The subgraph $\Gamma_\chi$ is a union of subgraphs isomorphic to the graph 
depicted in figure \ref{fig:Cayley-graph-metabelian-group-1}.
Each of these subgraphs is connected ---
actually they are the connected components of $\Gamma_\chi$ ---,
but $\Gamma_\chi$ itself  is not connected.

The subgraph $\Gamma_{-\chi}$, however, is connected. 
To see this,
note that for every number $x \in \Z[\tfrac{1}{2}]$ there exists a natural number $n$ 
such that $x$ is an integral multiple of $2^{-n}$.
The point $(x, -n)$ can therefore be connected to the origin 
by a path made up of a sequence of horizontal edges of length $2^{-n}$
followed by a sequence of $n$ vertical edges.
Since every vertex $(x,m)$  with $m \leq 0$ 
can be connected to a vertex of the form $(x, -n)$  by a sequence of vertical edges,
it follows that $\Gamma_{-\chi}$ is connected.
Altogether we have shown that 
\index{Computation of Sigma1@Computation of $\Sigma^1$ for!group BS(1,2)@group $\BS(1,2)$}
\begin{equation}
\label{eq:Sigma-1-metabelian-example}
S(G) = \{[\chi], [-\chi]\} \quad \text{and} \quad \Sigma^1(G) = \{[-\chi]\}. 
\end{equation}

3)  We conclude with a class of groups whose invariants are empty.
Let  $F_n$ be a free group of rank $n \geq 2$ or, more generally, a free product $G_{1} \star G_{2}$ 
with non-trivial finitely generated factors $G_{1}$ and $G_{2}$.
If the abelianized group $G_{\ab} \isoinv (G_{1})_{\ab} \oplus  (G_{2})_{\ab}$ is finite,
$S(G)$ and hence $\Sigma^1(G)$ are empty.
If not, let  $\chi$ be non-zero character of $G$;
exchanging $G_{1}$ and $G_{2}$ if need be, 
we may assume  that $G_{1}$ contains an element $g_{1}$ with $\chi(g_{1}) > 0$.
Choose $g_{2} \in G_{2}\smallsetminus \{1\}$ with $\chi(g_2) \geq 0$
and set  $g =g_{1}^{-1} g_{2} g_{1}$.
Then $g$ is an element of $G_\chi$
whose unique normal form of $g$ is $(g_{1}^{-1}, g_{2}, g_{1})$.  
It follows that $1$ and $g$ cannot be connected inside $\Gamma(G, \XX)_{\chi}$
if $\eta \colon \XX \to G$ is a suitably chosen finite generating system.

Indeed,
let $\XX$ be the union $\XX_{1} \cup \XX_{2}$ of a finite generating systems $\XX_{1}$ of $G_{1}$ 
and $\XX_{2}$ of $G_{2}$.
Consider a path $p = (1, w)$ from 1 to $g$. The word $w$ is a product of subwords $w_{1}\cdots w_{k}$
with each $w_{j}$ either a word in $\XX_{1}^\pm$ or in $\XX_{2}^\pm$.
If one of these subwords represents $1 \in G$ the path $p$ contains a loop;
omitting it leads to a path $p'$ containing a subset of the vertices lying along the path $p$.
We may thus assume that none of the subwords $w_{j}$ represents $1 \in G$.
But if so,
the uniqueness of the normal form implies that $\eta(w_{1}) = g_{1}^{-1}$;
since $\chi(g_{1}^{-1})$ is negative 
the path $p$ does therefore not run inside the subgraph $\Gamma_{\chi}(G; \XX_{1}\cup \XX_{2})$.
\index{Computation of Sigma1@Computation of $\Sigma^1$ for!free groups}
\index{Computation of Sigma1@Computation of $\Sigma^1$ for!free products}
%
\subsection{Invariance} 
\label{ssec:Invariance-Definition-Sigma1}
%
The definition of $\Sigma^1(G)$,
as given by formula \eqref{eqf:Definition-Sigma-one-of-G},
involves the generating system $\eta \colon \XX \to G$. 
We shall prove that the set $\Sigma^1(G)$ does not depend on the choice of $\eta$
and so deserves to be called an invariant of the group $G$.

In the proof of this claim and in many other verifications,
an extension of a character $\chi$ from  $G$ 
to the set of all paths $ P(\Gamma)$ on the Cayley graph $\Gamma = \Gamma(G, \XX)$ is useful.
Given  $\chi \colon G \to \R$ and a path $p = e_1e_2 \cdots e_m = (g, y_1y_2 \cdots y_m)$,
set
\index{Notation!v-sub-chi@$v_\chi$|(}
\begin{align}
v_\chi(p) 
& = 
\min\{
\chi(\iota(e_1)), \chi(\tau(e_1)), \ldots, \chi (\tau(e_m))\}
\notag\\
& = 
\min\{\chi(g),\; \chi(gy_1),  \; \chi(gy_1y_2), \ldots, \chi(gy_1 \cdots y_m)\}
\label{eq:Definition-v-sub-chi-1}\\
& = 
\chi(g) + \min\{0, \chi(y_1), \chi(y_1y_2), \ldots, \chi(y_1 \cdots y_m)\}. \notag
\label{eq:Definition-v-sub-chi-2}
\end{align}
One  verifies readily
that $v_\chi$  satisfies the following identities:
\begin{equation}
\label{eq:Properties-valuation-path}
\left.
\begin{aligned}
v_\chi(p^{-1}) 
& =  v_\chi(p)\\
v_\chi(g.p) 
& =  
\chi(g) + v_\chi(p)\\
v_\chi(pq) 
& =
\min\{v_\chi(p),\; v_\chi(q)\}
\end{aligned}
\hspace*{0.2cm}
\begin{aligned}
\text{for } &p \in  P(\Gamma),\\
\text{for } &g \in  G \text{ and } p \in P(\Gamma),\\
\text{for } &p,q \in P(\Gamma) \text{ with } \tau(p) = \iota(q).
 \end{aligned}
 \right\}
\end{equation}
\index{Notation!v-sub-chi@$v_\chi$|)}

\begin{remark} 
\label{remark:Valuation-for-words}
The set $W(\XX^\pm)$ can be regarded as 
the set of edge path $p \in P(\Gamma)$ with $\iota(p) = 1$; 
so equation
\eqref{eq:Definition-v-sub-chi-1}
defines a map $v_\chi:W(\XX^\pm) \to \R$. 
If the set $W(\XX^\pm)$ is identified with a set  of paths starting at 1,
 formula \eqref{eq:Properties-valuation-path}
has to be adapted:
\begin{equation}
\label{eq:Properties-valuation-word}
\left.
\begin{aligned}
v_\chi(w^{-1}) 
& =  v_\chi(w) -  \chi(w)\\
v_\chi(ww') 
& =
\min\{v_\chi(w),\chi(w) + v_\chi(w')\}
\end{aligned}
\hspace*{0.2cm}
\begin{aligned}
\text{for } &w \in  W(\XX^\pm),\\
\text{for } &w,w' \in W(\XX^\pm).
 \end{aligned}
 \right\}
\end{equation}
\end{remark}
\begin{thm} 
\label{thm:Sigma1-well-defined}
\index{Invariant Sigma1@Invariant $\Sigma^1$!independence}
The subset $\Sigma^1(G) \subseteq S(G)$ does not depend on  $\eta \colon \XX \to G$.
\end{thm}

\begin{proof}
The passage from one finite generating system to another such system 
can be carried out in finitely many steps, 
each of which consists in adjoining or deleting a redundant generator. 
So it suffices to compare the Cayley graphs 
$\Gamma = \Gamma(G,\XX)$  and $\Gamma' = \Gamma(G,\XX \cup \{z\})$. 
For every non-zero character $\chi \colon G \to \R$,
the graphs $\Gamma_\chi$ and  $\Gamma'_\chi$ have  the same vertices
and $\Gamma_\chi$  is a subgraph of $\Gamma'_\chi$.
Hence $\Gamma'_\chi$ is connected if $\Gamma_\chi$ is so. 

Conversely,
assume $\Gamma'_\chi$ is connected. 
Let $y_1y_2 \cdots y_m$ be a word in  $\XX^\pm$ 
which represents the redundant generator $z$.
Next choose  an element $t \in \XX^\pm$ with $\chi(t) > 0$ and a natural number $k$ 
such that $\chi(t^k) = k\chi(t) \geq -v_\chi(y_1y_2 \cdots y_m)$.
Given  vertices $g$ and $h$ of  $\Gamma_\chi$,
we can find a path from $g$ to $h$ that stays inside $\Gamma_\chi$ like this.
Set $g' =t^{-k}gt^k$ and $h' =t^{-k}ht^k$. 
Then $\chi(g') = \chi(g)$ and so $g'$ is a vertex of $\Gamma_{\chi}$; similarly, $h' \in \Gamma_{\chi}$.
Because $\Gamma'_\chi$ is connected,
there exist  a path $ p' = (g', w')$ inside $\Gamma'_\chi$ 
that leads  from $g'$ to $h'$.
The word $w'$ may contain  the letter $z$ or its inverse $z^{-1}$;
by replacing each occurrence of $z$ or $z^{-1}$ by the word  $y_1y_2 \cdots y_m$, respectively its inverse,
one obtains an $\XX^\pm$-word $w$;
it gives rise to a path $p = (g',w)$ from $g'$ to $h'$ in the Cayley graph $\Gamma(G,\XX)$.
This path need not be contained in  the the subgraph $\Gamma_\chi$,
but it can only leave it by  an \emph{a priori} known amount,
for the construction of $p$ guarantees that
\[
v_\chi(p) \geq v_\chi(y_1y_2 \cdots y_m) \geq -\chi(t^k).
\]
Let  $\alpha_{u}$ denote the automorphism of the Cayley graph $\Gamma(G,\XX)$
that is induced by left multiplication  by the element $u = t^k$.
The translated path $\alpha_{u}(p)$ starts in $u \cdot g' =  t^k \cdot (t^{-k} g t^k) = g\cdot t^k$,
ends in $h \cdot t^k$ and runs inside $\Gamma_\chi$,
and so the path $(g, t^k w t^{-k})$ leads from $g$ to $h$ and stays inside $\Gamma\chi$.
\end{proof}

%
\subsection{Consequences of the invariance}
\label{ssec:Consequences-invariance-Sigma1}
%
Theorem \ref{thm:Sigma1-well-defined} 
is a basic result in the theory of the invariant $\Sigma^1$;
it is also helpful in calculations
in that it permits one to adapt the generating system to the situation at hand.
The proofs of the following three results illustrate this aspect.
\footnote{Further results making use of Theorem  \ref{thm:Sigma1-well-defined}  
will be given in section \ref{ssec:Properties-Sigma1-change-groups}.}
%
%
\subsubsection{Groups with non-trivial centre or with finitely generated derived group } 
\label{sssec:Group-non-trivial-centre-or-fg-derived-group}
%
\begin{prp}
\label{prp:Sigma1-centre}
\index{Computation of Sigma1@Computation of $\Sigma^1$ for!groups with centre}
Let $G$ be a finitely generated group with centre $\zeta(G)$.
Then $\Sigma^1(G)$ contains every point $[\chi]$ which does not vanish on $\zeta(G)$.
In symbols,
\begin{equation}
\label{eq:Sigma1-centre}
 S(G, \zeta(G))^c   \subseteq \Sigma^1(G).
\end{equation}
\end{prp}
\begin{proof}
Let $\chi$ be a character with $\chi(\zeta(G)) \neq \{0\}$. 
Pick $z \in \zeta(G) \smallsetminus \ker(\chi)$
and choose a generating set of $G$ that  includes $z$.
Given any pair of elements $g$ and $h$ in $G_\chi$ 
there exists a path $p = (g, w)$ in $\Gamma(G, \XX)$ from $g$ to $h$.
If $v_\chi(p) \geq 0$, the path runs inside $\Gamma_\chi$;
otherwise there exists a natural number $k$ with $\chi (z^k) \geq |v_\chi(p)|$.
The path $(g, z^k \cdot w \cdot z^{-k})$  then stays inside $\Gamma_\chi$ 
and it leads from $g$ to $g \cdot z^{k}wz^{-k} = g \cdot w = h$.
It follows that  $\Gamma_\chi$ is connected 
and so $[\chi] \in \Sigma^1(G)$ by Theorem \ref{thm:Sigma1-well-defined}.
\end{proof}

\begin{examples}
\label{examples:Groups-with-non-trivial-centre}
Here are two illustrations of Proposition \ref{prp:Sigma1-centre}.

a) Let $G$ be a (finitely generated) \emph{abelian} group.
Then $\zeta(G) = G$ and so the set $S(G, \zeta(G))$ is empty. 
Proposition \ref{prp:Sigma1-centre} therefore implies that $\Sigma^1(G) = S(G)$.
\index{Computation of Sigma1@Computation of $\Sigma^1$ for!abelian groups}

b) Let $p$, $q$ be positive integers and $G$ the group
with generators $x_1$, $x_2$ and single defining relation $x_1^p= x_2^q$.
The abelianization of $G$ has $\Z$-rank $1$ and so the sphere $S(G)$ consists of two points.
These points are represented by the character $\chi_0 \colon G \to \R$  
with $\chi_0(x_1) =  q$, $\chi_0(x_2) = p$ and its antipode.
The element $z =x_1^p \in G$ commutes with the generator $x_1$;
as it coincides with $x_2^q$, it commutes also with $x_2$
and so it belongs to the centre of the group $G$.
Since neither $\chi_0$ nor $-\chi_0$ map $z$ to 0, 
Proposition \ref{prp:Sigma1-centre}  allows one to conclude
that $\Sigma^1(G) = \{[\chi_0], [-\chi_0] \} =S(G)$.
\index{Computation of Sigma1@Computation of $\Sigma^1$ for!group of torus knot}
\end{examples}

\begin{prp}
\label{prp:Fg-derived-group}
\index{Computation of Sigma1@Computation of $\Sigma^1$ for!groups with fg derived group}
If the derived group  of $G$ is finitely generated $\Sigma^1(G) = S(G)$.
\end{prp}
\begin{proof}
Choose a finite generating system $\eta \colon \XX \to G$ that includes a generating set $\YY$ of $G'$.
Given $g \in G_{\chi}$, represent $g$  by an $\XX$-word $w$. 
Reorder the letters in $w$ so 
that those with positive $\chi$-values precede the other letters, obtaining a word $w'$. 
Then the  exponent sums of the word $w_1 = w \cdot (w')^{-1}$ with respect to the set of generators $\XX$ all vanish;
so $w_1$ represents a element $h$  in the derived group $G'$ of $G$.
Represent $h$ by a $\YY^\pm$-word, say $u$.
Then the path $p = (1, w'\cdot u)$ runs inside $\Gamma_{\chi}$ and links the unit element with $g$.
\end{proof}
%
\subsubsection{Invariant of a direct product of groups} 
\label{sssec:Invariant-direct-product}
%
Given finitely generated groups $G_1$ and $G_2$, 
let $G$ denote their direct product $G_1 \times G_2$ 
and $\pi_1$, $\pi_2$ the canonical projections of $G_1 \times G_2$ onto $G_i$. 
If $\chi_1 \colon G_1 \to \R$ is a non-zero character, 
the composition $\pi_1 \circ \chi_1$ is non-zero; 
so $\pi_1$ induces an embedding  $\Hom(\pi_1, \id) \colon \Hom(G_1, \R) \mono \Hom(G, \R)$ of vector spaces
and hence an embedding $\pi_1^* \colon S(G_1) \mono S(G)$ of character spheres.
Similarly, $\pi_2$ induces an embedding $\pi_2^* \colon S(G_2) \mono S(G)$.
These embeddings allow one to express  the \emph{complement}  $\Sigma^1(G)^c $ 
of $\Sigma^1(G) $ by a simple formula:
\begin{prp}
\label{prp:Sigma1-direct-product}
\index{Computation of Sigma1@Computation of $\Sigma^1$ for!direct products}
The complement of the invariant of a direct product $G_1 \times G_2$ of finitely generated groups $G_1$, $G_2$ is given by the the formula
\begin{equation}
\label{eq:Sigma-1-direct-product}
\Sigma^1(G_1 \times G_2)^c = \pi^*_1(\Sigma^1(G_1))^c \cup \pi^*_2(\Sigma^1(G_2))^c .
\end{equation}
\end{prp}

\begin{proof}
Let $\XX_1$  be a finite generating subset  of $G_1$ which includes the neutral element $1_{G_1}$,
and let $\XX_2 \subset G_2$ have the analogous properties.
Then the set
\[
\XX = \XX_1  \times \{1_{G_2}\}  \cup \{1_{G_1} \} \times \XX_2
\]
is finite and it generates the direct product $G = G_1 \times G_2$.
Theorem \ref{thm:Sigma1-well-defined} allows one to base our computation on this generating set.

For each  vertex $g =(g_1,g_2)$  of the Cayley graph $\Gamma(G, \XX)$
there exists an $\XX_1^\pm$-word $y_{11}y_{12} \cdots y_{1h}$
representing the element $g_1$
and an  $\XX_2^\pm$-word $y_{21}y_{22} \cdots y_{2k}$ representing  $g_2$.
These words give rise  to the path
\begin{equation}
\label{eq:Path-form-12}
p_{12} = \left(1_G, (y_{11},1)  (y_{12},1) \cdots (y_{1h},1) \cdot (1,y_{21})  (1, y_{22}) \cdots (1, y_{2k})\right)
\end{equation}
but also to the path
\begin{equation}
\label{eq:Path-form-21}
p_{21} = (1_G, (1,y_{21})  (1, y_{22}) \cdots (1, y_{2k}) \cdot (y_{11},1)  (y_{12},1) \cdots (y_{1h},1) ).
\end{equation}
Both paths connect the origin with the vertex $(g_1,g_2)$

Consider now  a non-zero character  $\chi \colon G \to \R$
and let $g= (g_1,g_2)$ be a vertex of the subgraph $G_{\chi}$,  
\ie, an element of $G$ with  $\chi (g) \geq 0$.
Assume first $\chi$ vanishes on $G_2$ 
and denote the restriction of $\chi$ to $G_1$ by $\chi_1$. 
Then a  path of the form $p_{12}$ runs inside the subgraph $\Gamma_\chi(G,\XX)$ 
if, and only if, the path
$p_1 = (1, y_{11}  y_{12} \cdots y_{1h}) $  stays inside  $\Gamma_{\chi_1} (G_1, \XX_1)$.
This shows that $\pi_1^*$ maps $\Sigma^1(G_1)$ bijectively onto $\Sigma^1(G) \cap S(G, G_2)$.
Similarly, $\pi_2^*$ sends $\Sigma^1(G_2)$ bijectively onto $\Sigma^1(G) \cap S(G, G_1)$.

Assume now that $\chi$ vanishes neither on $G_1 \times \{1\}$ nor on $\{1\}  \times G_2$ . 
If $\chi(g_1,g_2) \geq 0$,
then $\chi(g_1, 1)\geq  0 $ or $\chi(1,g_2)\geq  0$.
If $\chi(g_1, 1)\geq  0$,
pick a path $p_{12}$ from $1_G$ to $g$ having the form  \eqref{eq:Path-form-12}.
This path may leave the subgraph $\Gamma_\chi$,
but its origin $1_G$, its mid point $(g_1,1)$ and its terminus $g$ are in $\Gamma_\chi$.
We are going to modify the first half of the path, and later on the second one.

Let $u$ be an element of $\{1\} \times \XX_2^{\pm}$ with positive $\chi$-value.
For every $\ell \in \N$ the path
\[
q_\ell  = (1_G,u^\ell \cdot (y_{11},1)  (y_{12},1) \cdots (y_{1h},1) \cdot u^{-\ell})
\]
leads from the origin to the mid point $(g_1,1)$ 
and it runs inside the subgraph $\Gamma_\chi$ if  $\ell$ is  large enough.
The second part of the path $p_{12}$ can be modified similarly.
The result of both transformations is a path  inside $\Gamma_\chi$ from $1_G$ to $g=(g_1,g_2)$.

If, on the other hand,  $\chi(1, g_2)\geq  0$,
one picks a path $p_{21}$ from $1_G$ to $g$ having the form  \eqref{eq:Path-form-21}
and modifies it in like manner.
\end{proof}

\begin{examples}
\label{examples:Direct-product}
Proposition \ref{prp:Sigma1-direct-product} allows one to work out the invariant 
of more complicated groups than those considered hitherto. 
Here are some specimens.

a) Let $G_1$ be the Baumslag-Solitar group discussed in example 2 of  \ref{sssec:Sigma1-first-examples};
it has the presentation
$\pi_1 \colon \langle a_1, u_1\mid u_{1}a_{1}u_{1}^{-1} = a_{1}^2 \rangle \iso G_1$.
Let $G_2$ be an isomorphic copy of $G_1$,
generated by  elements $a_2$, $u_2$, and set $G = G_1 \times G_2$.
By the discussion carried out in  example 2 of  \ref{sssec:Sigma1-first-examples}
the invariant $\Sigma^1(G_1)^c$ consists of the point $[\chi_1]$  with $\chi_1(u_1) = 1$ (and $\chi_1(a_1) = 0$);
in view of Proposition \ref{prp:Sigma1-direct-product}
the complement $\Sigma^1(G)^c$ consists therefore of two points,
the point $[\chi_1 \circ \pi_1]$ and the analogously defined point $[\chi_2 \circ \pi_2]$.
\index{Computation of Sigma1@Computation of $\Sigma^1$ for!direct products|(}

b) Next let  $G = C \times F_2$ be the direct product 
of an infinite cyclic group $C$ generated by $t$ 
and a free group $F_2$ of rank 2 with basis $\{x_1,x_2\}$;
the sphere of $G$ is then two-dimensional.
Since the element $t$ lies in the centre of $G$,
Proposition \ref{prp:Sigma1-centre} implies that the complement $\Sigma^1(C \times F_2)^c$
is contained in the subsphere $S(G, C)$;
this subsphere is a great circle that can be thought of as the equator of a 2-sphere.
Proposition \ref{prp:Sigma1-direct-product} and example 3 in  \ref{sssec:Sigma1-first-examples}
allow one to infer a stronger conclusion:
\emph{$\Sigma^1(G)^c$ coincides with the equator $S(G, C)$}.

c) Consider, finally, the direct product $G = F_2 \times F_2$ of two free groups of rank 2. 
Then $G_{\ab}$ is free abelian of rank 4  and so $S(G)$ is a 3-sphere.
In view of Proposition \ref{prp:Sigma1-direct-product}
and example 3 in  \ref{sssec:Sigma1-first-examples},
the complement $\Sigma^1(G)^c$ is  the union of the subspheres $S(G,G_2)$ and $S(G,G_1)$.
These subspheres are both great circles.
\end{examples}
\index{Computation of Sigma1@Computation of $\Sigma^1$ for!direct products|)}
%
%
\subsection{The subspaces $\Gamma(G, \XX)_\chi^I$}
\label{ssec:Subspaces:Gamma-chi}
%
We conclude this section with a comment on the definition of $\Sigma^1$.
Let $\Gamma(G,\XX)$ be the Cayley graph of a  finitely generated group $G$,
where $\eta \colon \XX \to G$ is  a finite generating family.
Each non-zero character $\chi \colon G \to \R$ gives rise to a collection  $I \mapsto G^I_\chi$ 
of subgraphs of $\Gamma$,
one for each  non-empty  interval $I \subseteq \R$.

The collection is defined as follows:
given a non-empty interval $I$,
set
\begin{equation}
\label{eq:G-superI-subchi}
G^I_\chi = \{g \in G \mid \chi(g) \in I\},
\end{equation}
and then define $\Gamma^I_\chi$ to be the full subgraph of $\Gamma(G,\XX)$ 
generated  by this subset. 
\index{Notation!G-sub-chi-super_I@$G_\chi^I$}
Special cases of this construction are
$\Gamma_\chi =  \Gamma_\chi^{[0,\infty)}$ 
and 
$\Gamma_{-\chi} = \Gamma_\chi^{(-\infty,  0]}$.

For every $g \in G$ the translated graph $g.\Gamma^I_\chi $ coincides with $\Gamma_\chi^{\chi(g)+I}$; 
hence $\Gamma^I_\chi$ and $\Gamma_\chi^{r+I}$ 
are isomorphic graphs whenever $r \in \im \chi$. 
It is not clear  to what extent this conclusion remains valid for other real numbers $r$; 
it does not  hold in general, 
for $G_\chi^{r+I}$ may be empty if $I$ is rather small. 
For a ray one has, however,
\begin{lem}
\label{lem:Connected-rays}
If $\Gamma_{a_*} = \Gamma_\chi^{[a_*,\infty)}$ is connected  for some $a_* \in \R$ 
then $\Gamma_a = \Gamma_\chi^{[a,\infty)}$ connected for every $a \in \R$.
\end{lem}

\begin{proof} 
Assume first $a < a_*$.
Let $t \in \XX^\pm$ be an element with positive $\chi$-value and choose $\ell$ so large 
that  $ a + \chi(t^\ell) \geq a_*$.
For every  vertex $g\in \Gamma_a$ the path $p_g = (g,t^k)$ 
stays inside $\Gamma_a$ and connects $g$ with the vertex $gt^\ell$  that lies in $\Gamma_{a_*}$. 
This implies that $\Gamma_a$ is connected.
If, on the other hand, $a_*< a$, choose an element $h\in G$ with $a  + \chi(h) <a_*$.
Then the subgraph $h(\Gamma_{a})$ is connected by the previous argument;
as is isomorphic to $\Gamma_{a}$, this subgraph is connected, too. 
\end{proof}

The above lemma will find various applications in the sequel.
A first one is
\begin{crl}
\label{crl:Infinitely-many-components}
If $\Gamma(G, \XX)_\chi$ is not connected,
it has infinitely many components.
\end{crl}
\begin{proof}
If $\Gamma_\chi$ has finitely many components, 
there exist a  negative number $a$ 
such that these components can be connected inside $\Gamma_\chi^{[a, \infty)}$.
This implies that $\Gamma_\chi^{[a, \infty)}$ is connnected,
whence so is $\Gamma_\chi$ by lemma \ref{lem:Connected-rays}.
\end{proof}
%
 
%
%
\section{The $\Sigma^1$-criterion}
\label{sec:Sigma1-criterion}
%
%
Let $G$ be a finitely generated group with finite generating system $\eta \colon \XX \to G$.
If $G$ admits a non-zero character --- and this is the only case we are genuinely interested in ---
the Cayley graph $\Gamma(G, \XX)$ is an infinite geometric object 
whose global form is, typically, hard to visualize.
If a good grasp of  $\Gamma(G, \XX)_\chi$ were a prerequisite for the determination of  $\Sigma^1(G)$,
one could not hope to work out the invariant, except in some simple cases.
The $\Sigma^1$-criterion now shows
that  a small number of  local conditions on the Cayley graph $\Gamma(G, \XX)$
imply that a subgraph $\Gamma(G, \XX)_\chi$ is connected.
In many examples of interest,
this criterion enables one to determine $\Sigma^1(G)$ completely; 
in less favourable instances,  
it permits one to construct at least a non-empty, open subset of the invariant.
\footnote{The proof that a subgraph $\Gamma(G, \XX)_\chi$ is \emph{not} connected
requires often other techniques; 
they will be discussed in later chapters.}
\smallskip

This section divides into three parts:
in the first,  Theorem \ref{thm:Sigma1-criterion},
the main result of the section,  is stated and  proved;
it  provides a necessary and sufficient condition for a point $[\chi]$ to belong to $\Sigma^1(G)$.
Then two applications of the main result  will be given;
one of them says that $\Sigma^1(G)$ is an open subset of the sphere $S(G)$.
In the third part,
a variant of the $\Sigma^1$-criterion is presented. 
%
\subsection{The main result}
\label{ssec:Stating-criterion}
%
It reads as follows:
\begin{thm} 
\label{thm:Sigma1-criterion}
\index{Sigma-criterion@$\Sigma^1$-criterion!geometric form}
Let $G$ be a  group and $\eta \colon \XX \to G$ a finite generating system of $G$.
For every non-zero character $\chi \colon G \to \R$
and  for every choice of $t \in \YY = \XX \cup \XX^{-1} $  with $\chi(t) > 0$,
the following conditions are equivalent:
\begin{enumerate}[(i)]  
\item
 the subgraph $\Gamma(G,\XX)_\chi$ of the Cayley graph $\Gamma(G,\XX)$ is connected;
 \item
 for every $y \in \YY  $  there exists a path $p_y$ from $t$ to $y\cdot t$ in $\Gamma(G,\XX)$
 that satisfies the inequality  $v_\chi(p_y) > v_\chi((1,y))$.
 \end{enumerate}
\end{thm}
\begin{figure}[htb]
\psfrag{aa}{\hspace*{-0.0mm}\small  $t$}
\psfrag{bb}{\hspace*{-0.3mm}\small  $t$}
\psfrag{ori}{\hspace*{0.3mm}\small  0}
\psfrag{py}{\hspace*{0.5mm}\small  $p_y$}
\psfrag{v}{\hspace*{-0.5mm}\small  $v_{\chi}(1,y)$}
\psfrag{x}{\hspace*{-0.6mm}\small  $\chi$}
\psfrag{y}{\hspace*{1mm}\small  $y$}

\psfrag{chi}{\hspace*{0.5mm}\small $\chi$}
\begin{center}
\includegraphics[width=3.1cm]{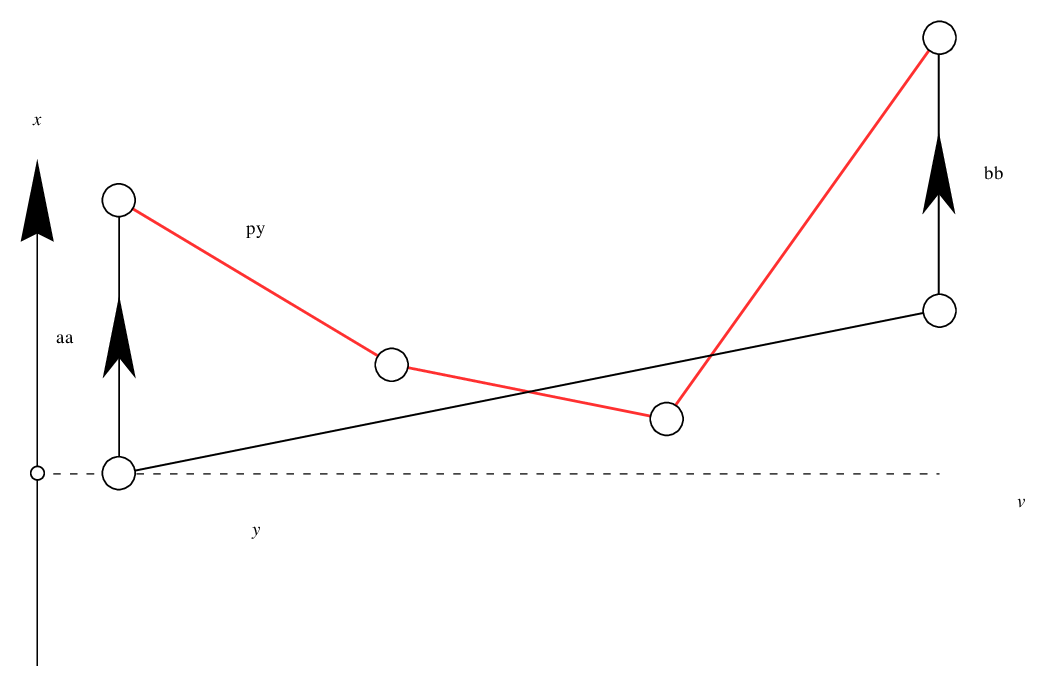}
\hspace{9mm}
\includegraphics[width=3.1cm]{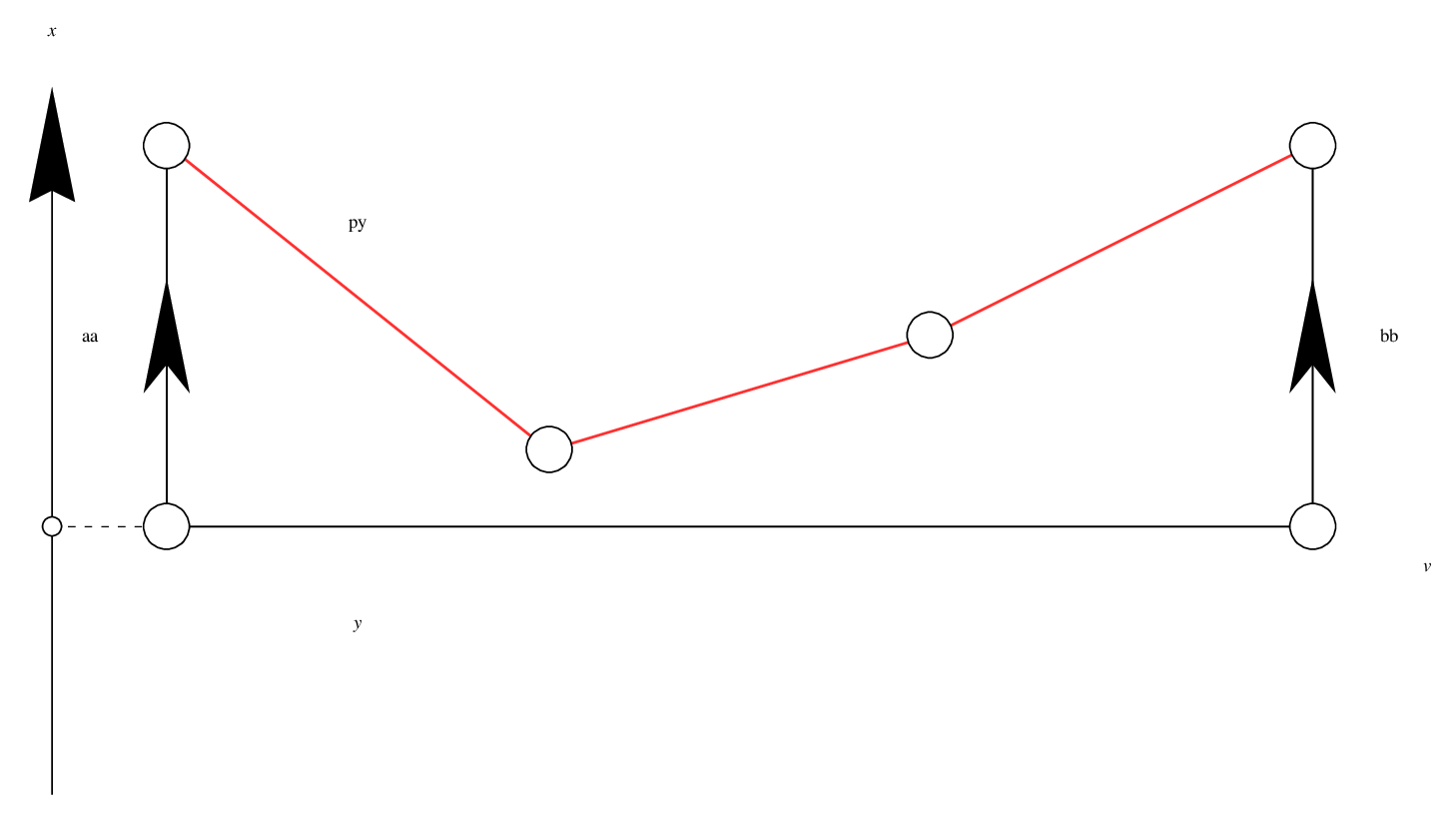}
\hspace{9mm}
\includegraphics[width=3.1cm]{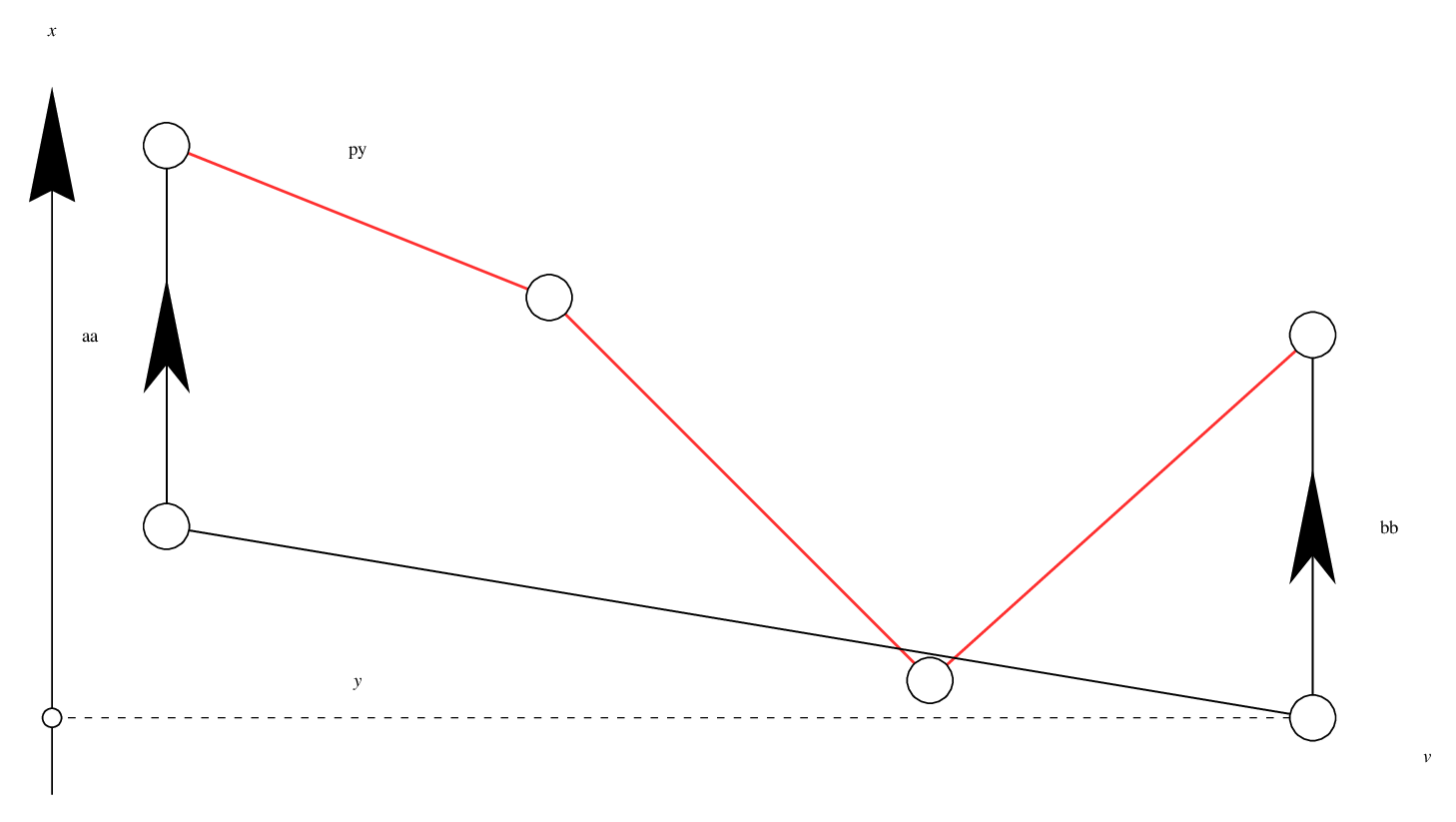}
\end{center}
\caption{Condition $v_\chi(p_y) > v_\chi((1,y))$ for $\chi(y)$ positive, zero or negative}
\label{fig:Condition}
\end{figure}

\begin{remarks}
\label{remark:Sigma-1-criterion}
a)  Figure  \ref{fig:Condition} illustrates the condition $v_\chi(p_y) > v_\chi((1,y))$ 
which the edge paths $p_y$ are required to satisfy. 
It reveals that the existence of the edge path $p_{y^{-1}}$ can be deduced from that of $p_y$.

b) Statement (ii)  in Theorem \ref{thm:Sigma1-criterion}
will be be referred to as the \emph{geometric version of the $\Sigma^1$-criterion}.
There exists more algebraic versions of the criterion;
they will be discussed in section \ref{sec:Sigma1-criterion-revisited}.
\index{Invariant Sigma1@Invariant $\Sigma^1$!geometric criterion} 

c) The geometric version of the $\Sigma^1$-criterion is convenient in proofs of general results,
examples being those of the main result or of Theorem
\ref{thm:Openness-Sigma-1} below.
The algebraic versions are better suited to calculations of the invariant in examples. 

d) The $\Sigma^1$-criterion is due to Robert Bieri (cf. \cite[Thm.\;3.1]{BiSt92}).
It is the analogue of Proposition 2.1 given in \cite{BNS}.
\index{Bieri, R.}
\index{Neumann, W. D.}
\index{Strebel, R.}
\end{remarks}
%

\begin{proof}
Assume first that statement (i) is valid.
Given $y \in \YY$,
define $r_y$ to be the number $\min\{\chi(t), \chi(y\cdot t)\}$.  
Since $\Gamma_\chi$ is  connected,
the graph $\Gamma_\chi^{[r_y,\infty)}$ is connected by lemma \ref{lem:Connected-rays}. 
So there is an edge path $p_y = (t, w_y)$ from $t$ to $y\cdot t$ with
\[
v_\chi(p_y) \geq r_y  =  \min\{\chi(t),\chi(y\cdot t)\} =   v_\chi((1,y)) + \chi(t) >   v_\chi((1,y)).
\]
By varying $y$ we thus obtain a collection of paths $p_y$ 
with $v_\chi(p_y) > v_\chi((1,y))$ and so statement (ii) holds.
\smallskip

 Conversely,
assume (ii) holds and define a number $d$ by setting
\[
d = \min \left\{v_\chi(p_y) - v_\chi((1,y)) \} \mid y \in \YY\right\};
\]
it is positive.
Now let $g \in G_\chi$ be a vertex of the subgraph $\Gamma_\chi$;
we aim at constructing a path from 1 to $g$ that does not leave this subgraph.
Since the Cayley graph $\Gamma$  is connected 
there exists a path $p$ from $1$ to $g$. 
This path $p$ will now be subjected to the transformation $T\colon  P(\Gamma) \to P(\Gamma)$
which replaces each edge $(h,y)$ of $p$  by the path $(h, t w_y t^{-1})$ 
and then eliminates the subpaths of the form $(h',t^{-1}t)$ (see Figure \ref{fig:Transformation-tau}).
This transformation fixes the endpoints of  $p$ and satisfies the inequality 
\begin{equation}
\label{eq:Improvement-caused-by-tau}
v_\chi(T(p)) \geq \min\{v_\chi(p) + d, 0\}.
\end{equation}

\begin{figure}[htb]
\psfrag{g0}{\hspace*{0.9mm}\footnotesize  1}
\psfrag{Y1}{\hspace*{-2mm}\footnotesize  $y^{-1}$}
\psfrag{g1}{\hspace*{0mm}\footnotesize  $g_1$}
\psfrag{Y2}{\hspace*{-1.0mm}\footnotesize $t^{-1}$}
\psfrag{g2}{\hspace*{0mm}\footnotesize  $g_2$}
\psfrag{Y3}{\hspace*{0.0mm}\footnotesize $y'$}
\psfrag{g3}{\hspace*{0.5mm}\footnotesize  $g_3$}
\psfrag{Y4}{\hspace*{-0.5mm}\footnotesize $y$}
\psfrag{g4}{\hspace*{0.5mm}\footnotesize  $g_4$}
\psfrag{Y5}{\hspace*{1mm}\footnotesize  $y$}
\psfrag{g5}{\hspace*{0.8mm}\footnotesize  $g$}
\psfrag{chi}{\hspace*{-0.0mm}\footnotesize $\chi$}
\psfrag{tt}{\hspace*{-0.0mm}\small $t$}
\begin{center}
\includegraphics[width=12.5cm]{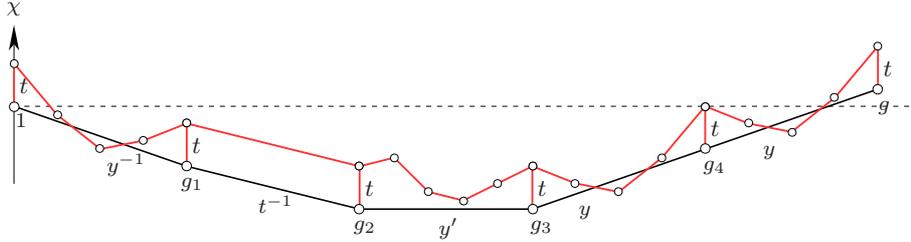}
\caption{The path transformation $T$}
\label{fig:Transformation-tau}
\end{center}
\end{figure}
By iterating the described transformation,
one obtains a sequence $p$, $T(p)$, $T^2(p)$, \ldots of paths from 1 to $g$ 
that leave the subgraph $\Gamma_\chi$ by smaller and smaller amounts.  
If $\ell \geq |v_\chi(p)|/d$, the path $T^\ell(p)$ runs therefore inside $\Gamma_\chi$ from 1 to $g$.
This shows that $\Gamma_\chi$ is connected.
\end{proof}
%

%
\subsection[Corollaries]{Corollaries of main result}
\label{ssec:Sigma-criterion-corollaries}
%
We continue with two important consequences of Theorem \ref{thm:Sigma1-criterion}.
%
\subsubsection{Openness of $ \Sigma^1(G)$}
\label{sssec:Openness-Sigma-1}
%
The first consequence could be called a corollary of Theorem \ref{thm:Sigma1-criterion}.
In view of its importance we prefer to give it the title of a theorem.
\begin{thm} 
\label{thm:Openness-Sigma-1}
\index{Invariant Sigma1@Invariant $\Sigma^1$!openness} 
For every finitely generated group $G$,
the geometric invariant $\Sigma^1(G)$ is an open subset of the sphere $S(G)$.
\end{thm}

\begin{proof} 
Suppose  $\chi$ is a  non-zero character of $G$ that represents a point of $\Sigma^1(G)$. 
Let $\XX$ be a finite system of generators of $G$ 
and choose an element $t \in \YY = \XX \cup \XX^{-1}$ with $\chi(t) > 0$.
Implication (i) $\Rightarrow$ (ii)  of Theorem  \ref{thm:Sigma1-criterion}
then produces, for every $y \in \YY$, a path   $p_y \in P(\Gamma(G, \XX))$ 
from $t$ to $yt$ that satisfies the inequality $v_\chi(p_y) > v_\chi((1,y))$. 
As this is a \emph{finite} system of inequalities it remains valid 
if $\chi$ is replaced by a character $\psi$ sufficiently close to $\chi$. 
Implication  (ii) $\Rightarrow$ (i)  of Theorem  \ref{thm:Sigma1-criterion} 
then applies to the character $\psi$  and shows that $[\psi]$ belongs to $\Sigma^1(G)$.
So the open neighbourhood
\begin{equation}
\label{eq:Defining-neighbourhood-chi}
\OO([\chi]) = \left\{[\psi] \mid \psi(t) > 0 \text{ and } v_\psi(p_y) > v_\psi((1,y))  
\text{ for }  y \in \YY \right\}
\end{equation}
of $[\chi]$ is contained in $\Sigma^1(G)$.
\end{proof}

The above theorem and the fact that closed subsets of the sphere $S(G)$ are compact,
are key ingredients of some finiteness results 
that are amongst the main applications of the theory of  $\Sigma^1$. 
One such result is Theorem \ref{thm:Characterizing-fg-N}.
%
\subsubsection{Characterization of rank 1 points in $\Sigma^1$} 
\label{sssec:Characterization-rank-1-points-Sigma-1}
%
A second consequence of Theorem \ref{thm:Sigma1-criterion} concerns rank 1 points in  $\Sigma^1$
and describes them in terms of a widely known algebraic property of groups:
\begin{prp}
\label{prp:Ascending-HNN-extension}
 If $G$ is a finitely generated group,
 $\chi\colon  G \epi \Z \mono \R$ a rank 1 character and $t \in G$ an element with $\chi(t) = 1$,
the following conditions are equivalent:
\begin{enumerate} [(i)]
\item $[\chi] \in \Sigma^1(G)$,
\item $N = \ker \chi$ contains a finitely generated group  $H$ with the properties
\begin{equation}
\label{eq:Ascending-HNN-extension}
 t^{-1}H t \subseteq H \text{ and }  \bigcup\nolimits_{\ell \in \N} t^{\ell }H t^{-\ell} = N.
\end{equation}
\end{enumerate}
\end{prp}
\index{Invariant Sigma1 and@Invariant $\Sigma^1$ and!ascending HNN-extensions} 
\index{Invariant Sigma1@Invariant $\Sigma^1$!and rank 1 points}

\begin{proof} (i) $\Rightarrow$ (ii).
Since  $t$ generates a complement of  $N = \ker \chi$, 
the group $G$ admits a finite generating set of the form $\XX = \{t\}  \cup  \{a_1, a_2, \ldots, a_m\}$ 
where each generator $a_j$ lies in $N$.
As the subgraph $\Gamma(G,\XX)_\chi$ is connected,
implication (i) $\Rightarrow$ (ii) of Theorem  \ref{thm:Sigma1-criterion} 
then provides one, for $j = 1$, 2, \ldots, $m$,  with a  path $p_j$ 
that connects  $t$ with $a_jt$  and satisfies the inequality $v_{\chi}(p_j) > v_\chi((1,a_j))=0$;
as $\chi$ takes values in $\Z$,  
this inequality amounts to say that $v_\chi(p_j) = 1$.
Let $w_j$ be the word with $p_j = (t, w_j)$.
This word  will then satisfy the relation $t^{-1} a_jt = w_j$ and the equality $v_\chi(w_j) = 0$.
This equality can be restated by saying 
that $w_j$ is freely equal to a product of conjugates 
$a_{i,\ell} = t^\ell a_i t^{-\ell}$ of  the $a_i$ with \emph{non-negative} exponents $\ell$.

Let $\mu$ denote the largest index $\ell$ for which some conjugate $a_{i,\ell}$ occurs in one of the words $w_j$
and set
\[
\HH = \{ a_{j,\ell} \mid 1 \leq j \leq m \text{ and } 0 \leq \ell \leq \mu \}.
\]
Then  $\HH$ generates a finitely generated subgroup $H$ of $N$,
and $H$ satisfies the inclusion  $t^{-1}H t \subseteq H$.
Indeed,  
if $\ell>0$ and $a_{j,\ell} \in \HH$ then $t^{-1} a_{j,\ell}t \in \HH$ by the definition of the generators $a_{j,\ell}$ 
and of $\HH$;
if $\ell = 0$, 
the relation $t^{-1} a_j t = w_j$ and the definition of $\HH$ guarantee 
that $t^{-1} a_{j,0} t \in H$.
The  statement  $\bigcup\nolimits_{\ell \in \N} t^\ell H t^{-\ell} = N$, finally, holds
since  $N$ is generated by the conjugates $a_{j,\ell}$ with $1 \leq j \leq m$ and $\ell \in \Z$.
\smallskip

(ii) $\Rightarrow$ (i).
Let $H$ be  a finitely generated subgroup of $N $ that enjoys properties 
\eqref{eq:Ascending-HNN-extension}.
Choose a finite set of generators $\BB = \{b_1, b_2, \ldots, b_f\}$ of $H$ 
and express each conjugate $t^{-1}b_j t$ as a word $w_j$ in $\BB^\pm$;
such expressions exist by the assumption 
that $t^{-1}H t\subseteq H$.
In view of the second  assumption in \eqref{eq:Ascending-HNN-extension}
the set $\XX = \{t\} \cup \BB$ generates $G$.
The $\Sigma^1$-criterion is then satisfied for $\chi$.
Indeed, 
the relation $t^{-1}b_j t = w_j$ implies that the path $p_j = (t, w_j)$ leads from $t$ to $b_jt$,
and the fact that $w_j$ is a word in $\BB^\pm$ and $\chi(\BB) = \{0\}$ implies
that 
\[v_\chi(p_j)= v_\chi(w_j)  + \chi(t) = 1 > 0 =v_\chi((1,b_j).
\]
Moreover, 
the path $p_t = (t, t)$ connects $t$ with $t \cdot t$ and satisfies $v_\chi(p_t) > v_\chi((1,t))$.
The claim now follows from Theorem \ref{thm:Sigma1-criterion}.
\end{proof}

\begin{remarks}
\label{remarks:HNN-extension}
a) Suppose $H$ is a group that contains two isomorphic subgroups $S$ and $T$.
Each isomorphism $\mu \colon S \iso T$ then leads to group,
called  \emph{HNN-extension with base group $H$, associated subgroups $S$, $T$ and stable letter $y$},
and defined by the presentation
\begin{equation}
\label{eq:Definition-HNN-extension} 
G = \langle H, y \mid y\cdot s\cdot y^{-1} = \mu(s) \text{ for every  } s \in S \rangle.
\end{equation}
This construction was introduced and put to good use in the paper \cite{HNN49}
by G. Higman, B.\;H. Neumann and H. Neumann.  
Their analysis revealed, in particular, 
that the canonical homomorphism $\kappa \colon H \to G$ is injective.
(See \cite[Sect.\;IV.2]{LS77} for more information on HNN-extensions.)

b) On comparing formulae 
\eqref{eq:Ascending-HNN-extension} and \eqref{eq:Definition-HNN-extension},
one sees that statement (ii) says
that $G$ is a HNN-extension with finitely generated base group $H$,
associated subgroups $S=t^{-1}Ht$ and $T = H$ and stable letter $t$.
The normal subgroup $N$ is then the union of the ascending chain 
$H \subseteq \act{t}{-1.5}{H}  \subseteq \act{t^{2}}{-1.5}{H} \subseteq \cdots$
whence such an HNN-extension is often called \emph{ascending}.
\index{Ascending HNN-extension!definition}
\end{remarks}
\begin{example}
\label{example:Locally-infinite-cyclic-by-infinite-cyclic}
Given two relatively prime, non-zero integers   $p$ and $q$, 
consider the  additive group $A = \Z[1/(p\cdot q)]$  of the  subring of $\Q$
generated by the rational number $1/(p\cdot q)$.
Multiplication by the fraction $p/q$ induces an automorphism $\mu = \mu_{p/q} $ on $A$
and so one can form the semi-direct product $G = G_{p,q}$ of $A$ and an infinite cyclic group $C$ 
generated by an element $s$ that acts on $A$ by $\mu$.
The multiplication of  $G_{p,q}$ is then given by
\[
(x,s^m) \cdot (x',s^{m'}) = ( x + (p/q)^{m} \cdot x', s^{m + m'}).
\]
The  group $G_{p,q}$ is metabelian and
a short calculation, based on Euclid's extended algorithm,  
shows that the elements $a=(1,s^0)$ and $t = (0,s)$ generate it.

If $p=q \in \{1,-1\}$, the group $G = G_{p,q}$ is   free-Abelian of rank 2;
in all other cases, the torsion-free rank of $G_{\ab}$ equals 1 and so  $S(G)$ consists of two points;
they are represented by the character $\chi$ that sends $t$ to $1$  and the subgroup $A$ to $\{0\}$
and its antipode $-\chi$.
If $G$ is free-abelian, the invariant is the entire circle $S(G)$ 
(see example \ref{examples:Groups-with-non-trivial-centre}a).
In the other cases,
$\Sigma^1(G)$ can be determined with the help of Proposition \ref{prp:Ascending-HNN-extension}.
Note first
that every finitely generated subgroup $H $ of $A$ is infinite cyclic, 
for such a group is contained in a  group of the form $\Z \cdot (1/p\cdot q)^k$.
So multiplication by the  fraction $(p/q)^{-1} = q/p$ maps $H$ into $H$ if, 
and only if,  $q/p$ is an integer.
It follows that $[\chi|$ lies in $\Sigma^1(G)$  if and only if  $|p| = 1$
 and that its antipode lies in $\Sigma^1(G)$ exactly if $|q| = 1$.
\end{example} 
\index{Computation of Sigma1@Computation of $\Sigma^1$ for!metabelian groups}
\index{Groups!of Baumslag-Solitar}
%
%
\subsection{An alternate definition of $\Sigma^1$} 
\label{ssec:Alternate-definition-Sigma-1}
%
In Section \ref{sec:Introducing-invariant}, the subset $\Sigma^1$ 
has been defined in terms of the connectivity of subgraphs $\Gamma_\chi$ 
of the Cayley graph $\Gamma$. 
In this section, 
we investigate an alternate condition, 
involving the path components of a (sufficently large) stripe of the Cayley graph,
and show that it leads to the same invariant.
The new definition will turn out to be very useful
in the analysis, to be undertaking in section \ref{sssec:Construction-graph-Delta},
of a class of bi-partite graphs $\Delta_b(\chi)$.
\subsubsection{Comparing the components of $\Gamma(G,\XX)_{\chi}$} 
\label{sssec:Comparing-components}
%
Let $G$ be a finitely generated group, $\chi \colon G \to \R$ a non-zero character of $G$ and $\Gamma_\chi$
the corresponding half of the Cayley graph of $G$ (with respect to a finite generating system).
If $\Gamma_\chi$ is not connected, it has infinitely many components 
(see Corollary \ref{crl:Infinitely-many-components}).
The next lemma implies that these components are isomorphic to each other.
\begin{lem} 
\label{lem:Transitive-action-on components}
Let $\Gamma(G,\XX)$ be the Cayley graph of a group $G$ 
with finite generating system $\eta \colon \XX \to G$ and let $\chi$ be a non-zero character of $G$.
If  $a$ and $b$ are real numbers satisfying the conditions $a  \leq b$ and  $ b \in \im \chi$,
the following assertions hold:
\begin{enumerate}[(i)]
\item each component $\CC$ of the subgraph $\Gamma_\chi^{[a, \infty]}$ contains a vertex $g$ 
with $\chi(g) = b$;
\item the kernel of $\chi$ acts transitively on the components of $\Gamma_\chi^{[a,\infty)}$.
\item the kernel of $\chi$ acts transitively on the components of $\Gamma_\chi^{(-\infty, -a]}$;
\end{enumerate}
\end{lem}
\begin{proof}
(i) It suffices to to establish the claim for $b = 0$.
Indeed, if $g \in G$ is a element with $\chi(g) = b$,
the automorphism induced by $g$ maps the subgraph $\Delta = \Gamma_\chi^{[a-b,\infty)}$
onto the subgraph $\Gamma_\chi^{[a, \infty)}$ 
and hence the components of the subgraph $\Delta$ onto those of $\Gamma_\chi^{[a,\infty)}$.

Consider a component $\CC$ of $\Delta$.
It contains a vertex, say $h$, with $\chi(h) \geq 0$. 
Choose an $\XX^\pm$-word  $w = y_1y_2 \cdots y_k$ that represents $h$
and reorder the letters of  $w$ so that the letters $y_j$ with positive $\chi$-values  
come before the other letters. 
If $w'$ denotes  the obtained word,
the path $p =(1,w')$ runs in the subgraph $\Delta$ from 1 to a vertex $h_1$  with $\chi(h_1) = \chi(h)$.
The later condition implies 
that there exists an element $h_0 \in \ker \chi$ with $h = h_0 h_1$.
The translated path $h_0.p =(h_0, w')$ then starts at $h_0$ and ends in $h$. 
The component $\CC$, containing the endpoint $h$,  contains therefore also $h_0 \in \ker \chi$.

(ii) Given components  $\CC$ and $\CC'$ of $\Gamma_\chi^{[a,\infty)}$,
there exist, by claim  (i),   vertices $g \in C$ and $ g' \in C'$  with $\chi(g) = \chi(g') = b$.
Then $h = g' \cdot g^{-1}$ lies in $\ker \chi$ 
and  the automorphism of $\Gamma(G, \XX)$ induced by $h$ sends $g$ onto $g'$ 
and hence $\CC$ onto $\CC'$. 
This reasoning shows 
that $\ker \chi$ acts transitively on the components of $\Gamma_{\chi} ^{[a, \infty)}$.

(iii) Statement (ii) holds for every non-zero character $\psi$,
in particular for  $-\chi$. 
Now use that $\Gamma_{\chi} ^{(-\infty, -a]} = \Gamma_{-\chi} ^{[a, \infty)}$.
\end{proof}
%
\subsubsection{The new definition} 
\label{sssec:New-criterion}
%
The new definition is provided by statement (i) in
\begin{thm}
\label{thm:Variant-Sigma-1-criterion}
Let $G$ be a group generated by the finite system $ \eta \colon \XX \to G$
and let $\chi \colon G \to \R$ be a non-zero character.
Given a real number $a$, 
let $\CC$ be one of the connected components of  $\Gamma_{\chi}^{[a, \infty)}$.
Then the following statements are equivalent: 
\begin{enumerate}[(i)]
\item  there is a real number $e \geq \max \{|\chi(x)| \mid x \in \XX \}$  such 
that the intersection $\CC\; \cap \;\Gamma_{\chi}^{[a, a+e]}$ is connected;
\item $-[\chi] \in \Sigma^1(G)$.
\end{enumerate}
\end{thm}
\begin{proof} 
Assume first statement (i) is valid and set $b = a+e$.
Since $e$ is at least as large as $\max \{|\chi(x)| \mid x \in \XX \}$, 
every vertex of $\Gamma = \Gamma(G,\XX)$ lies in at least one of the subgraphs 
$\Gamma'_b =\Gamma_\chi^{(-\infty,b]}$ and $\Gamma_a =\Gamma_\chi^{[a,\infty)}$. 
This fact allows one to prove that $\Gamma'_b$ is connected.
Indeed, given two vertices  $g_1$ and $g_2$ of $\Gamma'_b$,
there exists a path $p $ in $\Gamma$ from $g_1$ to $g_2$. 
The path $p$ has a  decomposition $p = p_1p_2 \cdots p_m$ 
where each subpath $p_i$ is either a path  of $\Gamma_a$  or a path of $\Gamma'_b$,
and each intermediate endpoint lies in $\Gamma_\chi^{[a,b]}$.
Let $J$ denote the set of indices of the subpaths running in  $\Gamma_a$. 
For every index $j \in J$,
 the subpath $p_j$ runs in a component $\CC_j$ of $\Gamma_a$.
This  component $\CC_j$ is  the image of the component $\CC$ under some graph automorphism  
induced by an element $h_j \in \ker \chi$ (see Lemma \ref{lem:Transitive-action-on components}). 
The intersection $\CC_j \; \cap\; \Gamma_\chi^{[a,b]}$  is therefore connected,
for it is isomorphic to  the intersection 
$\CC \; \cap\; \Gamma_\chi^{[a,b]}$ which is connected by statement (i).
So there exists  a path $p'_j$ in the slice $\Gamma_\chi^{[a,b]}$  with the same endpoints as $p_j$.
Upon replacing the subpaths $p_j$  by the paths $p'_j$ one obtains a path 
that runs in the subgraph $\Gamma_b$ and connects the given vertices $g_1$ and $g_2$ of $\Gamma_b$.
This proves that $\Gamma_b =\Gamma_\chi^{(-\infty, b]}$ is connected,
whence  $\Gamma_{-\chi} = \Gamma_\chi^{(-\infty, 0]}$  is so by Lemma \ref{lem:Connected-rays}
and thus $[-\chi]$ is a point of $\Sigma^1(G)$.
\smallskip

Conversely, assume that $[-\chi] \in \Sigma^1(G)$.
We aim at imitating the proof of implication (ii) $\Rightarrow$ (i) of Theorem \ref{thm:Sigma1-criterion}.
To bring out the similarity,
we set $\psi = -\chi$ and argue in terms of this new character.
By hypothesis,  $[\psi] \in \Sigma^1(G)$.
Choose an element $t \in \YY = \XX \cup \XX^{-1}$ with $\psi(t) > 0$ 
and use  implication (i) $\Rightarrow$ (ii) of Theorem \ref{thm:Sigma1-criterion}
to find, for every $y \in \YY$,  an edge path $p_y  = (h, w_y) \in P(\Gamma)$ from $t$ to $yt$ 
with $v_{\psi}(p_y)  > v_{\psi}(1,y)$. 
Let $e$ be the larger of the two numbers
\begin{align*}
&\max \{|\psi(x)| \mid x \in \XX \} = \max \{|\chi(x)| \mid x \in \XX \} ,\\
&\max\{\psi(g) \mid g \text{ a vertex of }  p_y \text{ for some } y \in \YY\}
\end{align*}
and set  $b = a+e$.
Suppose now that $g_1$ and $g_2$ are vertices of the slice $\SS = \Gamma_\psi^{[-b, -a]}$ 
which lie in a component $\CC$ of $\Gamma_\psi^{(-\infty, -a]} = \Gamma_\chi^{[a,\infty]}$.
Then there exists a path  $p$ in $\CC$ from $g_1$ to $g_2$.
We submit it to  an endpoint preserving transformation $T' \colon P(\Gamma) \to P(\Gamma)$
that is similar to the transformation $T$ used in the proof of implication
(ii) $\Rightarrow$ (i) of Theorem \ref{thm:Sigma1-criterion}. Two cases arise.
If $v_\psi(p) \geq -b$, 
the path $p$ itself runs in the slice  $\SS$ and the transformation $T'$ is defined to leave $p$ as it is.
Otherwise,
set 
\[
 d = \min\{v_\psi(p_y) - v_\psi(1,y) \mid Y \in \YY^\pm \}
 \]
and consider the edges  $(h,y)$ of $p$  which are not contained in the slice $\SS$.
Replace each such edge by the path $(h,tw_yt^{-1})$ and  delete afterwards subpaths of the form $(h', t^{-1}t)$.
The resulting path $T'(p)$ will then run inside the slice $\Gamma_\chi^{[c_1,-a]}$
with $c_1 = \min\{ -b, v_\psi(p) + d \}$;
indeed, 
if $h_0$ is a vertex of $p$ below the slice $\SS$
both the edge of $p$  that terminates  in $h_0$ and the following one are replaced.
By iterating the transformation $T'$,
one will then end up with a path in the slice $\SS = \Gamma_\chi^{[a,b]}$
which links the given vertices $g_1$ and $g_2$ of  $\SS$.
It follows that  the intersection $\CC\; \cap \; \Gamma_\chi^{[a,b]}$ is connected.
\end{proof}

Theorem \ref{thm:Variant-Sigma-1-criterion} 
has a pleasing corollary that deserves to be stated at this point,
in spite of the fact that its significance will only become clear in section 
\ref{sssec:Rank-1-fg-normal-subgroups}.
\begin{crl}
\label{crl:Connected-slice}
Let $G$ be a group generated by the finite system $\eta \colon \XX \to G$
and let $\chi \colon G \to \R$ denote a non-zero character.
For every  real number $a$ the following conditions are equivalent: 
\begin{enumerate}[(i)]
\item 
there is a real number $b \geq a  +\max \{|\chi(x)| \mid x \in \XX \}$ 
such that the slice $\Gamma_{\chi}^{[a, b]}$ is connected;  
\item $\Sigma^1(G,\XX)$ contains the pair of antipodal points  $\{[\chi], -[\chi] \}$.
\end{enumerate}
\end{crl}
\begin{proof}
Assume first the slice  $\Gamma_{\chi}^{[a, b]}$ is connected 
for some width $b-a \geq \max \{|\chi(x)| \}$. 
Thanks to the width of the slice 
every point in  $\Gamma_{\chi}^{[a, \infty)}$ 
can be connected inside $\Gamma_{\chi}^{[a, \infty)}$  to a point of the slice,
and so $\Gamma_{\chi}^{[a, \infty]}$ is connected.
Similarly one sees that $\Gamma_{\chi}^{(\infty, b]}$ is connected.
Lemma \ref{lem:Connected-rays} then implies
that the subgraphs $\Gamma_\chi$ and $\Gamma_{-\chi}$ are connected,
whence $[\chi]$ and $[-\chi]$ belong to $\Sigma^1(G)$.

Conversely,
if this conclusion holds 
the subgraphs $\Gamma_{-\chi}^{[-a, \infty)} = \Gamma_{\chi}^{(-\infty, a]}$
and $\Gamma_{\chi}^{[a, \infty)}$  are connected.
Since $\Gamma_{\chi}^{(-\infty, a]}$  is connected, 
implication (ii) $\Rightarrow$ (i) of Theorem \ref{thm:Variant-Sigma-1-criterion}
applies and provides us with a width $b-a  \geq \max \{|\chi(x)| \}$
such that the intersection of the slice $\Gamma_{\chi}^{[a, b]}$ 
with a fixed connected component $\CC$ of $\Gamma_{\chi}^{[a, \infty)}$ is connected.
As  $\Gamma_{\chi}^{[a, \infty)}$ itself is connected,
this says that the entire slice is connected. 
\end{proof}

\subsubsection{A  characterization of descending HNN-extensions} 
\label{sssec:Characterization-descending-HNN-extension}
%
Upon combining Proposition \ref{prp:Ascending-HNN-extension}
and Theorem \ref{thm:Variant-Sigma-1-criterion},
one arrives at the following characterization of \emph{descending} HNN-extensions:
\begin{prp}
\label{prp:Characterization-descending-HNN-extension}
\index{Invariant Sigma1 and@Invariant $\Sigma^1$ and!descending HNN-extensions} 
Let $G$ be a group generated by the finite system $ \eta \colon \XX \to G$,
let $\chi \colon G \epi \Z \incl  \R$ be a rank 1 character and let $u \in G$ be an element with $\chi(u) = 1$.
Finally,
let $\CC$ denote the connected component of  $\Gamma_{\chi}$ that contains the vertex $1_G$.
Then the following conditions are equivalent: 
\begin{enumerate}[(i)]
\item  the intersection $\CC\; \cap \;\Gamma_{\chi}^{[0, k]}$ is connected for some integer $k \geq 0$;
\item[(iii)] $N = \ker \chi$ contains a \emph{finitely generated} subgroup  $H$ with the properties
\begin{equation}
\label{eq:Descending-HNN-extension}
 uH u^{-1} \subseteq H \text{ and }  \bigcup\nolimits_{\ell \in \N} u^{-\ell }H u^{\ell} = N.
\end{equation}
\end{enumerate}
\end{prp}

The preceeding proposition can be illustrated neatly 
by the Baumslag-Solitar group $G$ considered in example 2 of section \ref{sssec:Sigma1-first-examples}.
The group  is generated by  elements  $a = (1,0)$ and $u = (0,1)$,
and it admits the presentation $\langle a, u \mid uau^{-1} = a^2 \rangle$,
describing it is a descending HNN-extension with stable letter $u$.
The following portion of the Cayley graph of $G$ 
(cf.\;Figure \ref{fig:Cayley-graph-metabelian-group-1} on page \pageref{fig:Cayley-graph-metabelian-group-1})
then shows
that condition (i) holds for every $k \geq0$.
\begin{figure}[htb]
\psfrag{Z}{\hspace*{-0.5mm}\footnotesize $m$}
\psfrag{Z2}{\hspace*{0mm}\footnotesize $x$}
\psfrag{1}{\hspace*{-2mm}\footnotesize $1$}
\psfrag{a}{\hspace*{-1mm}\footnotesize $a$}
\psfrag{am1}{\hspace*{-1.5mm}\footnotesize $a^{-1}$}
\psfrag{a2}{\hspace*{-1mm}\footnotesize $a^2$}
\psfrag{t}{\hspace*{-1mm}\footnotesize $u$}
\psfrag{t2}{\hspace*{-1mm}\footnotesize $u^2$}
\psfrag{ar1}{\hspace*{-2.2mm}\footnotesize $(1,a)$}
\psfrag{ar2}{\hspace*{-1mm}\footnotesize $(1,u)$}
\begin{center}
\includegraphics[width=11cm]{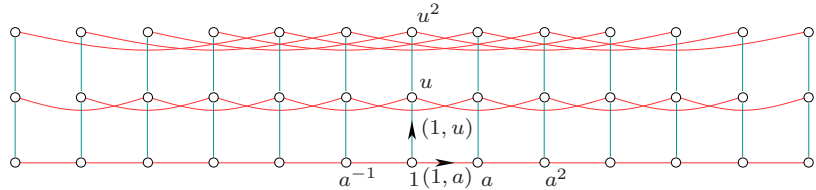}
\caption{A portion of the Cayley graph of $G$}
\label{fig:Cayley-graph-metabelian-group-3}
\end{center}
\end{figure}
%
%
 
%
%
\section{Finitely generated normal subgroups}
\label{sec:Finite-generation-normal-subgroups}
%
%
In this section,
we present one of the main results about the  invariant $\Sigma^1$.
It deals with normal subgroups $N$ of $G$ 
containing the commutator subgroup $G' =[G,G]$ of $G$ 
and gives an answer to the question 
whether such a subgroup is a finitely generated.
\begin{thm}
\label{thm:Characterizing-fg-N}
\index{Invariant Sigma1 and@Invariant $\Sigma^1$ and!fg normal subgroups}
\index{Characterization of!fg normal subgroups}
Let $G$  be a finitely generated group and
$N \vartriangleleft G$ a normal subgroup with abelian factor group. 
Then the  biimplication 
\begin{equation}
\label{eq:Characterization-fg-N}
N \text{  is finitely generated} \Longleftrightarrow S(G,N)  \subseteq \Sigma^1(G)
\end{equation}
holds. 
In particular,  the commutator subgroup $G'$ of $G$ is finitely generated
if $\Sigma^1(G) = S(G)$, and conversely.
\end{thm}

The proof of implication $\Rightarrow$ is fairly easy, 
as is the proof in the case where  $G/N$ is infinite cyclic 
(see Corollary \ref{crl:Characterizing-fg-kernel-rank-1-character});
these two proofs are  given in a preliminary discussion.
The justification of implication $\Leftarrow$ occupies the second section.
In the final section,
an application of Theorem \ref{thm:Characterizing-fg-N} will be presented.

\begin{note}
\label{note:Fg-normal-subgroups}
Theorem \ref{thm:Characterizing-fg-N} is the analogue of Theorem B1 in \cite{BNS}.
The proof given in section \ref{ssec:Proof-converse} is an amplified version of an argument due to Robert Bieri (cf.{} \cite[Thm.\;4.1]{BiSt92}).
\end{note}
\index{Bieri, R.}
\index{Neumann, W. D.}
\index{Strebel, R.}

%
\subsection{Preliminary investigation}
\label{ssec:Discussion-Main-result-normal-subgroups}
%
As before,
$G$ denotes a finitely generated group,
$N \triangleleft G$ a normal subgroup of $G$
and $\pi \colon G\epi Q$ the canonical projection of $G$ onto the factor group $Q = G/N$. 
%
\subsubsection{Normal subgroups with infinite cyclic factor group} 
\label{sssec:Rank-1-fg-normal-subgroups}
%
If the factor group $Q$ is infinite cyclic,
$N$ is the kernel of a rank 1 homomorphism $\chi \colon G \epi \Z \mono \R$ 
and Theorem \ref{thm:Characterizing-fg-N} takes on the following simple form:
\begin{crl}
\label{crl:Characterizing-fg-kernel-rank-1-character}
The kernel of a rank 1 character $\chi \colon G \epi \Z \mono \R$ is finitely generated if,
 and only if,  $\{[\chi] , [-\chi] \} \in  \Sigma^1(G)$.
\end{crl}

The corollary admits proofs that are far simpler than the general proof.
One of them can be based on Corollary \ref{crl:Connected-slice}
and Proposition \ref{prp:Criterion-finite-generation} below. 
A second one follows from the characterization of rank 1 points 
afforded by Proposition \ref{prp:Ascending-HNN-extension}.
Here are the details:
\begin{proof}
Let $t \in G$ be an element with $\chi(t) = 1$.
Assume first, $N = \ker \chi$ is finitely generated.
Then the inclusions $t^{-1} N t \subseteq N$ and $t N t^{-1} \subseteq N$ hold,
and so $G$  is an ascending  HNN-extension with respect to  $t$ and with respect  to $t^{-1}$;
by Proposition \ref{prp:Ascending-HNN-extension}  both $[\chi] $ and $[-\chi]$ 
belong therefore to $\Sigma^1(G)$.

Conversely,
if $[\chi]$ and $[-\chi]$ lie in $\Sigma^1(G)$ 
then $G$ is an ascending HNN-extension with finitely generated base group $B_1$ and stable letter $t$, say, 
and an ascending HNN-extension with finitely generated base group $B_2$, say, and stable letter $t^{-1}$.
So $t^{-1}B_1t \subseteq B_1$ and $tB_2t^{-1} \subseteq B_2$.
Since $B_1$ is finitely generated and $\bigcup_{j \in \N} t^{-j} B_2 t^{j} = \ker \chi$,
the base group $B_1$ is contained in a conjugate of $B_2$, say $B_1 \subseteq t^{-k}B_2 t^k$. 
The chain of inclusions
\[
t^\ell B_1 t^{-\ell}  \subseteq t^\ell ( t^{-k}  B_2 t^k )  t^{-\ell} 
\subseteq 
t^{-k} (t^\ell  B_2 t^{-\ell} )  t^k    \subseteq t^{-k}B_2 t^k
\]
holds therefore for every positive integer $\ell$.
As $\bigcup_{\ell \in \N} u^{\ell} B_1 u^{-\ell} = \ker \chi$ these inclusions show 
that  the subgroup  $u^{-k}B_2 u^k$  coincides with $\ker \chi$.
But if so,  $ \ker \chi = B_2$  and thus $\ker \chi$ is finitely generated.
\end{proof}

\begin{remark}
\label{remark:Characterizing-fg-kernel-rank-1-character}
The above proof makes use of Proposition \ref{prp:Ascending-HNN-extension}
and general properties of ascending HNN-extensions. 
It gives no bound on the number of generators of the normal subgroup $N = \ker \chi$.
An upper bound can be obtained by going back to the proof of implication (i) $\Rightarrow$ (ii) of that proposition.
Here are the details:
let $t \in G$ be an element generating a complement of $N = \ker \chi$. 
Choose a finite subset $\AA = \{a_1, \ldots a_m \}$ of $N$ 
such that $\XX = \AA \cup \{t\}$ generates $G$.
Then $N$ will be generated by the conjugates $a_{j , \ell} = t^\ell a_i  t^{-\ell}$ of the $a_j \in \AA$.

Assume now that $[\chi]$ and $[-\chi]$ lie in $\Sigma^1(G)$.
The hypothesis that $[\chi] \in \Sigma^1(G)$ and implication (i) $\Rightarrow$ (ii) in Theorem 
\ref{thm:Sigma1-criterion}
then provide us with words $w_j$ such that
\[
t^{-1}\cdot a_j \cdot t = w_j(\XX) \text{ and } v_\chi(w_j) = 0 \text{ for } j = 1, 2, \ldots m.
\]
The words $w_j$ can be rewritten as words, say $u_j$,  in the conjugates $a_{j,\ell}$; if this is done,
only non-negative indices $\ell$ will occur (because of the hypothesis that $v_\chi(w_j) = 0$.)
Let $\mu$ the largest value of $\ell$  that occurs in any of the words $u_j$
and set 
\[
N_+ = \gp(\{ a_{j,\ell} \mid j = 1, \ldots, m \text{ and } \ell = 0, \ldots, \mu\}).
\]
Then $t^{-1} \cdot N_+ \cdot t \subseteq N_+$.
By applying the previous argument to $-\chi$, one finds similarly a positive integer $\nu$ 
for which the subgroup
\[
N_- = \gp(\{ a_{j,\ell} \mid j = 1, \ldots, m \text{ and } \ell =  -\nu, \ldots, -1, 0 \}).
\]
satisfies the condition $t \cdot N_- \cdot t^{-1} \subseteq N_-$.
The subgroup $\gp(N_+, t^{\nu} \cdot N_- \cdot t^{-\nu})$ is then invariant under conjugation by $t$, 
coincides therefore with $N$,
and shows  that $N$ is generated by $m \cdot (1 + \max\{\mu, \nu\}) $ elements.
\end{remark}
%
\subsubsection{Comparison of the invariants of $G$ and of $G/N$} 
\label{sssec:Comparison-invariants-G-and-Q}
If $\eta \colon \XX \to G$  is a finite generating system of $G$, 
then $\pi \circ \eta \colon \XX \to G\epi Q$ is a  finite generating system of $Q$. 
So the  projection $\pi$ induces a  map $\pi_* \colon \Gamma(G,\XX) \to \Gamma(Q,\XX)$ 
between the Cayley graphs; it is a covering map. 
This fact is the key ingredient in the proof of the following result
which relates the invariants $\Sigma^1$ of  $G$ and $Q$:
\begin{prp}
\label{prp:Comparison-invariants-G-and-Q}
\index{Invariant Sigma1 and@Invariant $\Sigma^1$ and!quotient groups}
Let $\bar{\chi} \colon Q \to \R$ a character of $Q$ 
and let $\chi = \pi \circ \bar{\chi}  \colon G \to \R$ denote its pull-back to $G$.
Then the implication
\begin{equation}
\label{eq:Comparison-invariants-G-and-Q}
[\chi] \in \Sigma^1(G) \Longrightarrow [\overline{\chi}] \in \Sigma^1(Q) 
\end{equation}
holds. 
Its converse is valid if $N$ is a  finitely generated group.
\end{prp}

\begin{proof}
Since the $\bar{\chi}$ vanishes on $N$, 
the covering map 
$\pi_* \colon \Gamma(G,\XX) \to \Gamma(Q,\XX)$ restricts to a surjective graph map 
$\pi_* \colon \Gamma_{\chi} \to \Gamma_{\overline{\chi}}$.
If $\Gamma_{\chi} $ is connected, its quotient $\Gamma_{\bar{\chi}}$ is therefore also connected 
and so implication \eqref{eq:Comparison-invariants-G-and-Q} holds.

Conversely, assume $\Gamma_{\bar{\chi}}$ is connected.
Given a vertex $g \in \Gamma_{\chi}$, 
there exists a path $\bar{p}$ that runs in $\Gamma_{\bar{\chi}}$ 
and leads from $1_{Q}$ to $q = gN$.
It can be lifted to a path  $p_{1}$ that ends in $g$. 
This path runs in the subgraph $\Gamma_{\chi}$ and its origin $g_{1}$ is an element of $N$.
Assume now $N$ has a finite generating system $\AA$, say.
By the invariance property of $\Sigma^1$ 
(Theorem \ref{thm:Sigma1-well-defined}),
one may arrange that $\AA$ is part of the generating system $\XX$ 
 that is used in the construction of the Cayley graph $\Gamma$.
Then there clearly exists a path $p_{0}$ from $1_{G}$ to $g_{1}$ that runs inside $\Gamma_{\chi}$
and thus the concatenated path $p_{0}.p_{1}$ connects $1_{G}$ to $g$ in $\Gamma_{\chi}$.
\end{proof}

\begin{addendum}
\label{addendum:Comparison-invariants-G-and-Q}
The notation being as in Proposition \ref{prp:Comparison-invariants-G-and-Q},
assume $N$ is finitely generated and  $Q = G/N$ abelian. 
Then $S(G,N) \subseteq \Sigma^1(G)$.
\end{addendum}

\begin{proof}
Since $Q$ is abelian, its invariant is all of $S(Q)$ by example 
\ref{examples:Groups-with-non-trivial-centre} a);
if, in addition, $N$ is finitely generated Proposition \ref{prp:Comparison-invariants-G-and-Q}
therefore implies that every  non-zero character $\chi$ of $G$ with $\chi(N) = \{0\}$ 
represents a point of $\Sigma^1(G)$.
\end{proof}
%
\subsection[Proof of the main result]{Proof of the Theorem \protect \ref{thm:Characterizing-fg-N}}
\label{ssec:Proof-converse}
%
The proof of implication 
$
S(G,N)  \subseteq \Sigma^1(G) \Rightarrow N \text{\emph{  is finitely generated}} 
$
in Theorem \ref{thm:Characterizing-fg-N}
will be broken into several steps. 
We begin with two reductions 
and the introduction of coordinates on $S(G, N)$.
%
\subsubsection{Reduction to the case where $Q$ is free abelian and choice of $\XX$}
\label{sssec:Reduction-Q-free-abelian}
%
Let $G$, $N \vartriangleleft  G$ and $Q = G/N$  be as in in Theorem  \ref{thm:Characterizing-fg-N}.
Then $Q$ is a finitely generated group abelian group;
thus its torsion group $T(Q)$ is finite  and the factor group $Q/T(Q)$  is free abelian.
Let $\widehat{N}$ denote the preimage of $T(Q)$ under the canonical epimorphism $\pi \colon G \epi Q$.
Then $N$ has finite index in $\widehat{N}$ and so $N$ is finitely generated if, and only, 
$\widehat{N}$ has this property
(see, e.\;g., \cite[p.\;36, \textbf{1.6.18}]{Rob96} for the non-trivial implication).
Moreover, as the additive group of the field of reals $\R$ is torsion-free 
every character that vanishes on $N$ will vanish on  $\widehat{N}$;
so $S(G,N) = S(G,\widehat{N})$. 

It suffices therefore to prove the claim under the \emph{additional hypothesis 
that $Q$ be free abelian, say of rank} $k$.
Accordingly,
we choose the finite set of generators  $\XX = \TT \cup  \ZZ$ of $G$ in such a way
that  $\pi \colon G\epi Q$ maps $\TT$ bijectively onto a  basis of $Q$ and $\ZZ$ is contained in $N$. 
In addition,
we select  $\ZZ$ so that it contains all commutators  $[y_{1}, y_{2}]$ with $y_{1}$, $y_{2}$ in 
$\TT\cup \TT^{-1}$.
%
\subsubsection{Filtering the Cayley graph }
\label{sssec:Introduction-filtering-Cayley-graph}
%
We next introduce coordinates on $S(G,N)$.
\index{Character sphere!coordinates}
Let $\vartheta \colon G \epi Q= G/N  \iso \Z^k$ be an epimorphism of groups
that sends $\TT$ onto the set of standard basis vectors of $\Z^k$
and define $ \Z^k$ to be the standard lattice of the real vector space $\R^k$ 
equipped with  the usual inner product $\langle  -,-\rangle$ and norm $\| -\|$. 
Similarly as in section \ref{sssec:Coordinates-sphere}, 
we  introduce then coordinates on the sphere  $S(G, N)$
by assigning to  $u \in \s ^{k-1}$ the character $\chi_v \colon G \to \R$
given by $\chi_v(g) = \langle u, \vartheta(g)\rangle$.

With the help of these coordinates,
we construct a filtration on  $G$ by \emph{spherical} subsets $G(\rho)$.
These subsets are defined thus:
\begin{equation}
\label{eq:Filtration-G}
G(\rho) = \{g \in G \mid \|\vartheta(g)\| \leq \rho\}.
\end{equation}
Here $\rho $ ranges over the discrete subset $\{\sqrt{m} \mid m \in \N \}$ of $\R$. 
Since $N$ is the kernel of $\vartheta$,
each $G(\rho)$ is $N$-invariant;
but it is neither a subgroup nor a submonoid of $G$.

Let  $\Gamma(\rho) \subseteq \Gamma = \Gamma(G,\XX)$ 
to be the full subgraph  with vertex set $G(\rho)$. 
Then  $N$ acts on $\Gamma(\rho)$ and this action is free,
for it is induced by the action of $N$ on the Cayley graph $\Gamma$. 
The quotient graph $N\backslash G(\rho)$ is finite. 
If it is in connected for some sufficiently large radius $\rho$
then  $N$ will act freely on the connected graph $\Delta =  \Gamma(\rho)$ 
with finite quotient graph $N\backslash \Delta$
and will therefore be finitely generated by Proposition \ref{prp:Criterion-finite-generation} below.
%
\subsubsection{Search for a connected subgraph $\Gamma(\rho_0)$}
\label{sssec:Search-connected-Gamma-rho-0}
%
We next aim at proving that the graph $\Gamma(\rho)$ is connected 
whenever the radius $\rho$ is large enough.
The intuitive reason is this. 
Let $\rho_0$ be a large radius and let $g_{0}$ be a vertex in $\Gamma(\rho_0)$.
There exists then a path $p$ in the  Cayley graph $\Gamma = \Gamma(G, \XX)$  from $1$ to $g_{0}$.
If $p$ runs inside $\Gamma(\rho_0)$  all is well;
otherwise, let  $\rho_p  \in \R$ be the smallest real number $\rho$
for which $p$ is contained in the subgraph $\Gamma(\rho)$,
and let $M_p$ be the set of all vertices $g$ of $p$ with $\|\vartheta(g)\| =  \rho_p$.
Consider now a maximal consecutive sequence of vertices $(g, g', \ldots)$ in $M_p$.
Let  $(g,y_1) $ be the \emph{inverse} of the edge of $p$ that \emph{terminates} in $g$
and let  $(g,y_2)$ be the edge of $p$  that \emph{begins} in $g$.
\begin{figure}[htb]
\psfrag{hh}{\hspace*{-1mm} \small $h$}
\psfrag{gg}{  \hspace*{-1mm}   \small $g$}
\psfrag{ggp}{\hspace*{-0mm}\small $g'$}
\psfrag{ggh}{\hspace*{-0.5mm}\small $h'$}
\psfrag{qq3}{\hspace*{-0.5mm}\small $g''$}
\psfrag{yy1}{\hspace*{-0.5mm}\small $y_1$}
\psfrag{yy2}{\hspace*{-0.5mm}\small $y_2$}
\psfrag{tt}{\hspace*{-0.0mm}\small $t$}
\psfrag{tt1}{\hspace*{-0.5mm}\small $t'$}
\psfrag{L}{\hspace*{-7mm} \small $(g,y_{1})$}
\begin{center}
\includegraphics[width=9.0cm]{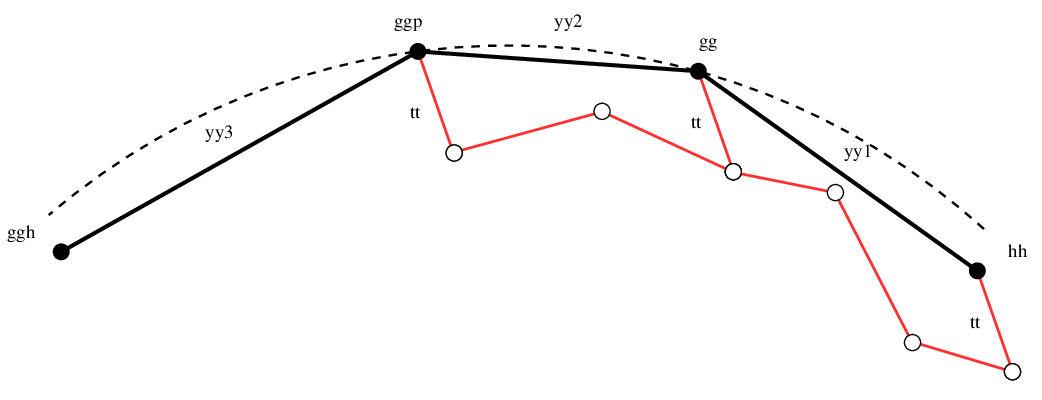}\par
\includegraphics[width=9.0cm]{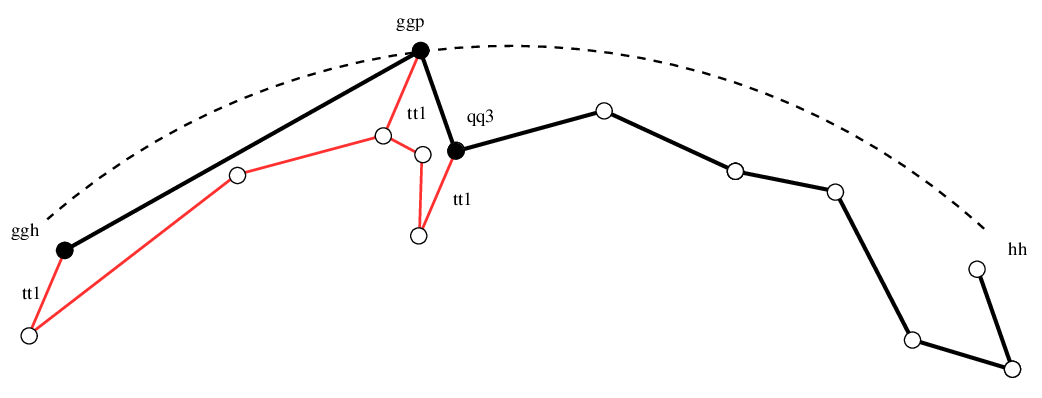}

\caption{One step in the transformation of the path $p$}
\label{fig:Transforming-path-fg-normal}
\end{center}
\end{figure}
Define  $u$ to be the  vector $u = -\vartheta(g)/\norm{\vartheta(g)}$.
Choose an element $t \in \YY = \XX \cup \XX^{-1}$ with $\chi_u(t) > 0$.  
By Theorem \ref{thm:Sigma1-criterion} there is then a path $p_{y_1}$ from $t$ to $y_1t$, 
and a path $p_{y_2}$ from $t$ to $y_2t$,
all in such a way that $v_u (p_{y_1} > v_u(1,y_1)$ and $v_u (p_{y_2} > v_u(1,y_2)$;
here  $v_u$ is short for $v_{\chi_u}$. 
The norm of $\vartheta(gy_1)$ does not exceed that of $\vartheta(gy_1)$ and so $\chi_u(y_1) \geq 0$.
Similarly, $\chi_u(y_2 \geq 0$. 
The paths $p_{y_1}$ and $p_{y_2}$ satisfy therefore the inequalities
\begin{equation}
\label{eq:Sharpened-inequalities}
v_{\chi_u} (p_y) > 0 \text{ for } y = y_1 \text{ and } y = y_2.
\end{equation}

Replace now the subpath $(h, y_1^{-1}y_2)$ of $p$ 
by the path 
\[
q = (h,t) \cdot (g.p_{y_1})^{-1} \cdot (g.p_{y_2})\cdot (g',t)^{-1},
\] 
ending up with a path $p'$ 
that contains no vertex $z$ between $h = gy_1$ and $g' = g y_2$ with $\|\vartheta(z)\| =  \rho_p$.
The last of the edges of the subpath $q$ is $(g',t)^{-1}$. 
Now apply to the inverse of this edge and the edge $(g', y_3)$ the analogous transformation.
Continuing in this way one will finally arrive at a path from 1 to $g_0$ 
all whose vertices $z$ satisfy the inequality $z$ with $\norm{\vartheta{z}} < \rho_p$.
Because the norm of a lattice point $\vartheta(z)$ lies in the discrete set $\{\sqrt{m} \mid m \in \N \}$,
iteration of the described transformation of paths will lead in finitely many steps 
to a path from 1 to the given $g_0$ that runs in  $\Gamma(\rho_0)$.
\smallskip

The crux of the proof thus lies in finding a radius $\rho_0$ 
that allows one to carry out the described transformation of paths.
To find it, we bring into play the hypothesis 
that the subsphere $S(G,N)$ of $S(G)$ is contained in $\Sigma^1(G)$.
We can construct an open cover of the sphere $\s^{k-1} \iso S(G,N)$ as follows: 
given a unit vector $u \in \s^{k-1}$, 
choose an element $t \in \YY = \XX \cup \XX^{-1}$ with $\chi_u(t) > 0$.
Then find, with the help of Theorem \ref{thm:Sigma1-criterion},
for each $y \in \YY$, a path $p_y(u)$ that leads from $t$  to $yt$ 
and satisfies the inequality $v_{u}(p_{y}(u)) > v_u((1,y))$. 
The couple $\psi_u = (t, \{p_y(u)\})$ gives rise to  a real valued function 
$f_{\psi_u}$;
it is defined on $\s^{k-1}$ and given by
\begin{align}
f_{\psi_u} (u') &= \min \{ v_{u'}(p_y(u)) - v_{u'}((1,y)) \mid y \in \YY \}. \label{eq:Auxiliary-functions-f}
\end{align}
This function enters into the definition of a neighbourhood of $u$,
namely
\begin{equation}
\label{eq:Neighbourhood-u}
\NN_u = \NN_u(\psi_u) = \{u' \in \s^{k-1} \mid 
f_{\psi_u}(u') > 0 \}.
\end{equation}
The set $\NN_u$ is an open subset of $\s^{k-1}$.
As $\s^{k-1}$ is compact, 
there exist a finite family $\FF = \{u_1, \ldots, u_\ell\}$ of unit vectors  
such that the collection of open subsets $\{ \NN_{u_j} \mid 1 \leq j \leq \ell \}$ covers $\s^{k-1}$.
This family is next used in the definition of an auxiliary function $\varepsilon \colon \s^{k-1} \to \R$;
it is given by
\begin{equation}
\label{eq:Defining-function-epsilon}
\varepsilon(u') 
= 
\max \{ f_{\psi_j} (u') \mid 1 \leq j \leq \ell 
\}.
\end{equation}
Here $\psi_j$ is short for $\psi_{u_j}$.
The definition of $\varepsilon$ and of the neighbourhoods $\NN_{u_j}(\psi_j)$ imply
that the function $\varepsilon$ is positive on the sphere $\s^{k-1}$.
As it is continuous and $\s^{k-1}$ is compact, 
it admits therefore a positive lower bound, say $\varepsilon_{0}$. 

To define the radius $\rho_0$ we need a last definition.
For each index $j \in \{1, \ldots, \ell \}$ and each generator $y \in \YY$,
let $\SS_{j, y}$ be the set of all vertices of the path $p_y(u_j)$.
 Define then an auxiliary radius $r$ by setting
\begin{equation}
\label{eq:Radius-ball-enclosing-paths}
r = \max \{\|\vartheta(h)\| \mid  h \in \SS_{j,y} \text{ and } (j,y) \in \{1, \ldots, \ell \} \times \YY\, \}.
\end{equation}
Set now $\rho_0  =  \tfrac{1}{2} r^2/\varepsilon_0$.

Consider a vertex $g$ on a path from $1$ to a given vertex $g_{0} \in \Gamma(\rho_0)$ 
for which $\norm{\theta(g)}$ is maximal, 
and put $z =\theta(g)$,  $\rho = \norm{z}$  and $ u = - \theta(g)/\rho$.
It $\rho \leq \rho_0$, all is well;
otherwise, choose an index $j$ with $f_{\psi_j} (u) \geq \varepsilon_0$. 
Let $y_0 \in \YY$ be a generator with 
$\chi_u(y_0) \geq 0$. The starting point of the path $p_{y_0}(u_j)$ is $t$ 
and $v_u(p_{y_0}(u_j) > v_u((1,y_0)) = 0$, and so $\chi_u(t) > 0$.
Define $(g,y_1) $ to be the inverse of the edge of $p$ that terminates in $g$
and $(g,y_2)$ to be the edge of $p$  that begins in $g$.  
Then formula  \eqref{eq:Sharpened-inequalities} applies to the generators $y_1$ and $y_2$.
Given a group element  $h \in \SS_{j,y_1} \cup \SS_{j,y_2}$,  
set $y = \vartheta(h)$.
Then
\begin{align*}
\norm{z+y}^2 
&=  
\norm{z}^2 +2 \langle z, y \rangle   + \norm{y}^2
=\rho^2  + \norm{y}^2 - 2 \langle u \cdot \rho, y\rangle \\
&\leq \rho^2 + r^2 - 2\varepsilon_0 \cdot \rho  < \rho^2.
\end{align*}

We conclude 
that the family of couples 
$
\{ \psi_{u_j} = (t_{u_j}, \{p_y(u_j)\}_{y \in \YY})\mid 1 \leq j \leq \ell \}
$
permits one to transform every the path $p$ from 1 to a vertex $g_0$ in the graph $\Gamma(\rho_0)$
into a path that runs inside this graph.
The subgraph $\Gamma(\rho_0)$ is thus connected.

The proof of Theorem  \ref{thm:Characterizing-fg-N}
is now complete save for the proof of  Proposition  \ref{prp:Criterion-finite-generation} below.
%
%
\subsubsection{A geometric criterion for finite generation}
\label{sssec:Graph-criterion-finite-generation}
%
The following criterion for the finite generation of a group acting on a connected graph is well-known.
\begin{prp}
\label{prp:Criterion-finite-generation} 
\index{Finite generation!graph theoretic criterion}
Let $G$  be a group acting on a 
connected graph $\Gamma$  so that the following two conditions  are satisfied:
\begin{enumerate}[(i)]
\item   the factor graph $G\backslash\Gamma$  is finite;
\item   the stabilizers  $G_v = \{g \in G \mid gv = v\}$ of all vertices $v \in\Gamma$ are finitely 
generated. 
\end{enumerate}
Then $G$  is finitely generated.
\end{prp}
\begin{proof} 
Let $\pi \colon \Gamma\epi G\backslash \Gamma$ be  the 
canonical projection. Choose a finite subgraph $\Gamma_0 \subseteq 
\Gamma$ with $\pi(\Gamma_0) = G\backslash\Gamma$. 
Next fix, for each couple $(v, v')$ of distinct vertices  of $\Gamma_0$ lying in the same  $G$-orbit,
an element $t \in G$ with $v = t.v'$ 
and let  $\TT$ be the  (finite) set of these elements $t \in G$.
Define  $K$ to be the subgroup of $G$  generated by $\TT$  and by the stabilizers $G_v$ 
of the vertices of  $\Gamma_0$;
then $K$ is a finitely generated. 

The group $K$ admits of an alternative description, namely
\begin{equation}
\label{eq:Alternative-description-H}
K = \gp(\{g \in G \mid g.\Gamma_0 \cap \Gamma_0 \neq \emptyset\}).
\end{equation}
Indeed, 
if $g.\Gamma_0 \cap \Gamma_0 \neq \emptyset$ then $\Gamma_0$ 
contains vertices $v$ and $w$ with $g.v = w$. 
If $v = w$  then $g\in G_v$;
otherwise there is some $t \in \TT$ with $t.v = w$ 
and so $g$ is a product $t \cdot g'$ with  $g' \in G_v$. 
So the right hand side of equation \eqref{eq:Alternative-description-H}  is  contained in $K$. 
The opposite inclusion is obvious.

We assert that $G$ equals $K$.
To see this,
consider an element of $g \in G$. Pick $v_0 \in \ver (\Gamma_0)$.
Since $\Gamma$ is connected, 
there exists a path from  $v_0$ to  $g.v_0$
say 
\[
p =(v_0 =w_0, w_1, \ldots, w_\ell = g.v_0).
\]
Since $\ver(\Gamma) = G.\ver(\Gamma_0)$ 
there exists for every index $j \in \{1, \ldots, \ell-1 \}$ 
a vertex $v_j \in \Gamma_0$ and an element $g_j \in G$ with $w_j = g_j.v_j$.
Moreover, $w_0 = 1.v_0$ and $w_\ell = g.v_\ell$.
Similarly, for $j \in \{1, \ldots, \ell \}$,
there exist for the edge $f_j$ from $w_{j-1}$ to $w_j$ an edge $e_j \in \edg(\Gamma_0)$ 
and a group element $ h_j \in G$ with $f_j = h_j.e_j$.
One has $w_0 = 1.v_0$ and $1 \in K$. 
Assume, inductively, that  $g_{\ell-1} \in K$. 
Since $h_\ell.\Gamma_0$ intersects $g_{\ell-1}.\Gamma_0$ 
the intersection $(g_{\ell-1}^{-1}h_\ell).\Gamma_0 \cap \Gamma_0$ is non-empty whence $g_{\ell-1}^{-1}h_\ell \in K$
and so $h_\ell =g_{\ell-1} \cdot k'$ for some $k' \in K$.
Similarly one sees that $g_{\ell} = h_\ell \cdot k''$ for some $k'' \in K$.
So $g_\ell =  h_\ell \cdot k'' = g_{\ell-1} \cdot k' \cdot k''$; 
as $g_{\ell-1} \in \K$ by the induction hypothesis this shows that $g \in K$.
\end{proof}
%
\subsection{Application: Normal subgroups of large co-rank}
\label{ssec:Applications-main-result-fg-N}
%
We conclude section \ref{sec:Finite-generation-normal-subgroups} 
with some applications of Theorem \ref{thm:Characterizing-fg-N}.
%
%
\subsubsection[Finitely generated normal subgroups of large co-rank]%
{Finding finitely generated normal subgroups of large co-rank} 
\label{sssec:Finding-fg-kernel-large-corank}
\index{Invariant Sigma1 and@Invariant $\Sigma^1$ and!fg normal subgroups}
%
Our applications of Theorem \ref{thm:Characterizing-fg-N}
will actually be corollaries of Theorem  \ref{thm:Finding-normal-sgroup-maximal-corank} given below.
In it,
the \emph{complement} of the invariant is assumed to be contained in a finite union of subspheres
and one aims at finding a finitely generated normal subgroup $N$ containing $G'$ 
with $G/N$ having large torsion-free rank.
\begin{thm}
\label{thm:Finding-normal-sgroup-maximal-corank}
Assume $G$ is a finitely generated group 
for which the complement of the invariant satisfies an inclusion of the form
\begin{equation}
\label{eq:Invariant-union-of-subspheres}
\Sigma^1(G)^c \subseteq \bigcup \{S(G,H_{j}) \mid j \in J\},
\end{equation}
each $H_{j}$ being a subgroup of $G$ and $J$ being finite. 
Then the following claims hold:
\begin{enumerate}[(i)]
\item there exists a finitely generated normal subgroup $N \supseteq G'$ with
\begin{equation}
\label{eq:Minimum-ranks}
 r_{0}(G/N)= \min \{ r_{0}((H_{j}G')/G') \mid j \in J \}.
\end{equation}
\item if equality holds in \eqref{eq:Invariant-union-of-subspheres} 
then no normal subgroup $N \supseteq G'$ with
\begin{equation}
\label{eq:Greater-minimum-ranks}
 r_{0}(G/N) > \min \{ r_{0}((H_{j}G')/G') \mid j \in J \}
\end{equation}
can be finitely generated.
\end{enumerate}
\end{thm}

\begin{proof}
Let $N$ be a normal subgroup containing the derived group $G'$.
According to Theorem \ref{thm:Characterizing-fg-N},
$N$ is finitely generated if, and only if, the subsphere $S(G,N)$ and the complement 
$\Sigma^1(G)^c$ of the invariant are disjoint.
If this complement satisfies inclusion \eqref{eq:Invariant-union-of-subspheres},
$N$ will therefore be finitely generated 
if the intersections $S(G,N) \cap S(G,H_{j})$ are empty for every index $j \in J$.
Now $S(G,N) \cap S(G,H_{j})$ is empty if, and only if, every character $\chi \colon G \to \R$ 
which vanishes on both $N$ and $H_{j}$ is zero or, put differently,
if the group $N \cdot H_{j}$ has finite index in $G$.

Let $\vartheta \colon G \epi G/G' \epi \Z^n$ be an epimorphism of $G$ 
onto the standard lattice of rank $n = r_{0}(G_{\ab})$ 
and let $L = \vartheta(N)$ and $L_{j} = \vartheta(H_{j})$ be the images of $N$ and  $H_{j}$ under $\vartheta$.
Then $L + L_{j}$ has finite index in $\Z^n$ 
and so the subspaces $U$, $U_{j}$ of $\Q^n$ spanned by $L$ and $L_{j}$, respectively, 
satisfy the relation $U + U_{j}= \Q^n$.

Conversely, 
if $U$ and $U_{j}$ are subspaces of $\Q^n$ satisfying $U + U_{j}= \Q^n$ 
then the subgroups $N = \vartheta^{-1} (U \cap \Z^n)$  and $H_{j} = \vartheta^{-1} (U_{j} \cap \Z^n)$ are normal subgroups 
and the subspheres $S(G,N)$ and $S(G,H_{j})$ are disjoint.
\smallskip

The preceding reasoning shows
that Theorem \ref{thm:Finding-normal-sgroup-maximal-corank}
can be reduced to the following claim about finite dimensional vector spaces:
\begin{lem}
\label{lem:Simultanous-supplement}
Given  a finite-dimensional vector space $V$ over an infinite field $\F$ 
and subspaces $U_{1}$, \ldots, $U_{\ell}$,
the exists a subspace $U$ of co-dimension 
\[
m = \min \{\dim U_{j} \mid 1 \leq j \leq \ell \}
\] 
which is a simultaneous supplement to all the subspaces $U_{j}$.
Moreover, there does not exist a simultaneous supplement of co-dimension strictly larger than $m$. 
\end{lem}

 The  proof of Lemma \ref{lem:Simultanous-supplement}
depends on another, more widely known, result about vector spaces,
namely
\begin{lem}
\label{lem:Covering-proper-subspaces}
A finite-dimensional vector space $V$ over an infinite field 
cannot be covered by finitely many \emph{proper} subspaces. 
\end{lem}

Lemma \ref{lem:Simultanous-supplement} will be proved by induction on  $d = \dim V-m$.
The claim is patent for $d=0$.
Assume now $d \geq 1$ and set $J_{1} = \{j \in J \mid U_{j} \neq V \}$.
By Lemma  \ref{lem:Covering-proper-subspaces}
there exists a vector $v_{1}$ which lies outside every $U_{j}$ with $ j \in J_{1}$.
Set $\overline{V} = V/(F \cdot v_{1})$ and let $\pi \colon V \epi \overline{V}$ denote  the canonical projection.
For every $j \in J_{1}$, 
the composition of $U_{j } \incl V$ with $\pi$ is injective,
and so the subspaces $\overline{U_{j}} = \pi(U_{j}) $ of $\overline{V}$ form a finite family of subspaces with 
\[
\dim( \overline{V}) - \min\{\overline{U_{j} } \mid j \in J_{1} \} = d-1.
\]
By the inductive hypothesis, 
there exists therefore a supplement $\overline{U} \subset \overline{V}$ of dimension $d-1$.
Its preimage $U = \pi^{-1} (\overline{U})$ is a supplement of each proper subspace $U_{j}$ with $j \in  J_{1}$, 
and hence of every subspace $U_{j}$ with $j \in  J$.

The addendum is clear and so we are left with establishing Lemma \ref{lem:Covering-proper-subspaces};
this will be by induction on $d = \dim V$.
If $d = 1$, the claim is evident, if $d=2$ it holds 
since V contains infinitely many lines passing through the origin.
Assume now that $d \geq 3$ and let  $U_{1}$, \ldots, $U_{\ell}$ be proper subspaces of $V$.
Some of them may have co-dimension 1, 
but as  $V$ contains infinitely many subspaces of co-dimension 1, 
there exist a subspace $U$ of co-dimension 1 with $U \neq U_{j}$ for all indices $j$, 
whence each of the subspaces $U'_{j} = U_{j}\cap U$ is a proper subspace of $U$.
As $\dim U = d-1$, the inductive hypothesis applies and provides one with a vector $v \in U$,
that lies in none of the intersections $U'_{j} = U_{j} \cap U$.
But, if so, $v$ does not lie in one of the original subspaces $U_{j}$, either.
\end{proof}

\begin{remark}
\label{remark:Author-Simultanous-supplement}
Lemma \ref{lem:Simultanous-supplement} may be well-known.
The given proof is a variant of an argument found in April  of 2008 by Simon Kurmann and Fred Rohrer,
two students of Markus Brodmann (University of Z{\"u}rich).
\end{remark}

\begin{example}
\label{example:Finding-normal-sgroup-maximal-corank}
Assume $G$ is the internal direct product of subgroups $G_{1}$, $G_{2}$
and let $\pi_{i} \colon G \epi G_{i}$ be the canonical projections onto the factors $G_{i}$.
Proposition \ref {prp:Sigma1-direct-product} expresses $\Sigma^1(G)^c$  
in terms of the subsets $\Sigma^1(G_{i})^c$ and implies the upper bound
\begin{equation}
\label{eq:Invariant-direct-product-bound}
\Sigma^1(G)^c =  \pi^*_1(\Sigma^1(G_1))^c \cup \pi^*_2(\Sigma^1(G_2))^c  
\subseteq S(G, G_{2}) \cup S(G,G_{1}).
\end{equation}
Theorem \ref{thm:Finding-normal-sgroup-maximal-corank} then allows us 
to find a normal subgroup $N$ containing $G' = G'_{1} \cdot G'_{2}$  and so that
\begin{align}
r_{0}(G/N) 
&= 
\min \{ r_{0}((G_{2}\cdot G')/G'), r_{0}((G_{1}\cdot G')/G') \}
\notag\\
&=
\min \{ r_{0}((G_{1})_{\ab}), r_{0}((G_{2})_{\ab}) \}.
\label{eq:Rank-quotient-example}
\end{align}

In this particular example,
claim (i)  of Theorem \ref{thm:Finding-normal-sgroup-maximal-corank} 
can be established by a simple, direct argument.
For $i \in \{1,2\}$, 
let $M_{i}$ denote the preimage of the torsion-subgroup of $(G_{i})_{\ab}$ 
under the canonical map $G_{i} \epi (G_{i})_{\ab}$. 
The quotient       groups $A_{1} = G_{1}/M_{1}$ and $A_{2} = G_{2}/M_{2}$
are free-abelian of rank $n$ and $n'$, say. 
Let $\{x_{1},x_{2}, \ldots, x_{n} \}$ be a subset of $G_{1}$ 
that projects onto a basis of $A_{1}$;
similarly,
let $\{x'_{1},x'_{2}, \ldots, x'_{n'} \}$ be a subset of $G_{2}$  projecting onto a basis of $A_{2}$.
Since $A_{1}$ is a finitely presentable group,
there exist a finite subset $\MM_{1} \subset M_{1}$ which generates  $M_{1}$ over the group 
$\gp(\{x_{1},x_{2}, \ldots, x_{n} \})$ (see, e.\;g., \cite[p. 53, {2.2.3}]{Rob96}).
Similarly, 
$M_{2}$ is generated over $\gp(\{x'_{1},\ldots, x'_{n'}\})$ by a finite subset $\MM_{2} \subset M_{2}$.

We are now ready to construct a finitely generated  normal subgroup $N$ of $G$ 
which contains $G'$ for which $G/N$ has torsion-free rank $m  = \min \{n, n'\}$.
Assuming, as we may, that $n \leq n'$ 
we define $N$ to be the subgroup generated by the finite set
\[
\MM_{1} \cup \MM_{2} \cup \{x_{i}\cdot x'_{i} \mid 1 \leq i \leq n \} \cup \{ x'_{i} \mid n <i \leq n'\} .
\]
Since  $x_{i}$ and $x_{i}\cdot x'_{i}$  induce by conjugation the same actions on $M_{1}$,
the normal subgroup $M_{1}$ is contained in $N$;
similarly, one sees that $N$ contains $M_{2}$.
Finally, $G/N$ is  free abelian of rank $m = \min \{ r_{0}((G_{1})_{\ab}), r_{0}((G_{2})_{\ab}) \}$.
\end{example}

\begin{remark}
\label{remark:Invariant-union-subspheres}
There exist several classes of groups $G$
for which $\Sigma^1(G)^c$ is a finite union of subspheres of the form $S(G,H)$.
Graph groups provide such a class; it will be treated below.
Another class consists of groups of automorphisms of free groups studied by J. McCool in \cite{McC86};
the invariants of these groups have been determined  by L. A. Orlandi-Korner in \cite{OrKo00}.
A third class is formed by the fundamental groups of compact \Ka{} manifolds (see \cite{Del10}).
\index{Delzant, T.}
\index{McCool, J.}
\index{Orlandi-Korner, L. A.}
\end{remark}
%
\subsubsection{Finitely generated normal subgroups in graph groups} 
\label{sssec:Fg-normal-subgroups-right-angled Artin-group}
%
Right angled Artin groups, or \emph{graph groups} as they used to be known,
are given by a presentation whose relations are commutator relations among the generators.
Such a presentation can be described by a finite combinatorial graph $\Delta = (\ver (\Delta), \edg (\Delta))$.
Its set of vertices $\ver (\Delta)$ is the set of distinguished generators  
$\XX = \{x_{1},x_{2}, \ldots, x_{n}\}$ of the group.
Each edge $\{x_{i},x_{j}\}$ of $\Delta$ corresponds to the relation $x_{i}x_{j} = x_{j}x_{i}$,
and there are no further defining relations.
The graph group $G_{\Delta}$ is thus given by
\begin{equation}
\label{eq:Right-angled-Artin-group}
\index{Graph groups!definition}
G_{\Delta}
= 
\langle x_{1}, x_{2}, \ldots, x_{n} 
\mid 
x_{i}x_{j} = x_{j}x_{i} \text{ for every edge } \{x_{i} , x_{j}\} \in \edg (\Delta) \rangle.
\end{equation}
Their simple presentations notwithstanding,
graph groups are surprisingly varied, a fact that has led to some remarkable applications
(cf. Ruth Charney's survey \cite{Cha07}).
\index{Right angled Artin groups|see{Graph groups}}

The invariant of a graph group can be described in terms of subsets $\SS \subseteq \XX$ 
whose removal results in a disconnected graph;   
such subsets are called \emph{separating} (cf. \cite[p.\;11]{Die05}).
\index{Graph!separating subset}
One has
\begin{prp}
\label{prp:Sigma-complement-right-angled-Artin-group}
\index{Computation of Sigma1@Computation of $\Sigma^1$ for!graph groups}
Let $G = G_{\Delta}$ be a graph group with graph $\Delta$.
If $\Delta$ is a complete graph, $G$ is free-abelian of rank $\card \ver (\Delta)$ 
and $\Sigma^1(G) = S(G)$.
Otherwise,
the complement of its invariant is given by 
\begin{equation}
\label{eq:Complement-Invariant-right-angled-Artin-group}
\Sigma^1(G)^c = \bigcup\nolimits_{\SS} S(G,\gp(\SS))
\end{equation}
where $\SS$ runs over the minimal separating  subsets of $\ver  (\Delta)$.
\end{prp}

Proposition \ref{prp:Sigma-complement-right-angled-Artin-group}
will be established in section \ref{ssec:Application-right-angled-Artin-groups}.
Upon allying it with Theorem \ref{thm:Finding-normal-sgroup-maximal-corank}
one arrives at a result going back to  J. Meier and L. VanWyck (see \cite[Theorem 6.3]{MeVa95}):
\index{Meier, J.}
\index{VanWyck, L.}
\begin{crl}
\label{crl:Finding-normal-sgroup-maximal-corank-right-angled-Artin-group}
Assume $\Delta$ is not the complete graph on $\ver (\Delta)$
and $G_{\Delta}$ is graph group with graph $\Delta$.
Then the number
\[\max \{r_{0}(G/N) \mid N \text{ is a finitely generated group containing } G_\Delta'\}
\]
coincides with the cardinality of a separating subset  of $\ver (\Delta)$ having least number of vertices. 
\end{crl}

\begin{remark}
\label{remark:Explicit-construction-fg-normal-group-N}
A finitely generated normal subgroup $N \triangleleft G_\Delta$ 
enjoying the property stated in Corollary 
\ref{crl:Finding-normal-sgroup-maximal-corank-right-angled-Artin-group}
can be found as follows.
Let  $\{x_1, x_2, \ldots, x_n\}$ be the vertex set of $\Delta$ 
and let $\vartheta \colon G_\Delta \epi \Z^n$  be the epimorphism 
sending $x_i$ to the $i$-th standard basis element of $\Z^n$.
Let $m$ be the cardinality of a minimal separating set of $\Delta$.
Formula \eqref{eq:Complement-Invariant-right-angled-Artin-group}
and the proof of Theorem \ref{thm:Finding-normal-sgroup-maximal-corank}
then show that it suffices to find a free abelian subgroup $A$ of rank $n-m$ in $\Z^n$ 
that is transversal to every subgroup $B \subset \Z^n$ generated by $m$ basis vectors,
and then to set $N = \vartheta^{-1}(A)$.
 An explicit example of $A$ is the subgroup generated by the $n-m$ rows $(1^k,2^k, \cdots, n^k)$
where $k$ ranges from 0 to $n-m-1$ (use the Vandermonde determinant). 
\end{remark}

\begin{examples}
\label{examples:Invariants-right-angled-Artin-groups}
a) Assume $\Delta$ has $n\geq 2$ vertices and is not a complete graph.
If $\Delta$ is not connected, 
the empty set is separating and Corollary \ref{prp:Sigma-complement-right-angled-Artin-group}
implies that $\Sigma^1(G_\Delta)$ is empty. 
Theorem \ref{thm:Characterizing-fg-N} then shows 
that $G_\Delta$ has no finitely generated normal subgroup $N$
whose factor group $G_\Delta/N$ is an infinite abelian group.
If, on the other hand, $\Delta$ is connected
every separating set has at least one vertex.
By Corollary \ref{crl:Finding-normal-sgroup-maximal-corank-right-angled-Artin-group}
the group $G_\Delta$ admits therefore finitely generated normal subgroups $N$ 
with infinite cyclic quotient $G_\Delta/N$. 
This conclusion is well-known;
indeed, by \cite[Thm.\:6.1]{MeVa95}, 
the kernel of every rank 1 character $\chi \colon G \to \Z$ is finitely generated
if $\chi$ sends every generator $x_i \in \ver(\Delta)$ to  a non-zero integer.
\index{Meier, J.}
\index{VanWyck, L.}

b) Let $\Delta$ be a tree with $n\geq 3$ vertices.
Every vertex  $x_{j}$ of  degree  greater than 1 separates the tree, 
while the removal of a  vertex of degree 1 results in a connected subtree.
So the minimal separating subsets $\SS$ are the singletons constituted by the vertices of degree greater than 1.
Each such singleton gives rise to a subsphere $S(G, \SS)$ of co-dimension 1.

c) Let $\Delta_{n}$ be an $n$-gon with $n \geq 4$ sides.
No singleton and no pair of adjacent vertices separates the polygon,
but every pair of non-adjacent vertices does.
By Proposition \ref{prp:Sigma-complement-right-angled-Artin-group}
 the complement of the invariant is therefore the union
\begin{equation}
\label{eq:Invariant-n-gon}
\bigcup \{ S(G, \gp(\{x_{i},x_{j}\}) \mid  j \notin \{ i-1,i,  i+1 \}  \}.
\end{equation}
of subspheres of co-dimension 2. 
According to Corollary \ref{crl:Finding-normal-sgroup-maximal-corank-right-angled-Artin-group},
the group $G$ contains  a finitely generated normal subgroup  $N$ with $G/N \approx \Z^2$,
but no finitely generated normal subgroup  $N$ with $G/N \approx \Z^3$.

d) The reader can find further examples in section 
\ref{sssec:Sigma1-right-angled-Artin-groups-list-minimal-separating-subsets}.
\end{examples}
\newpage
 
%
%
\section{Finitely related groups and $\Sigma^1$}
\label{sec:Sigma1-and-fp-groups}
%
%
In this section, 
a second main result about the invariant $\Sigma^1$ will be established.
It provides a necessary condition 
that every finitely related soluble group $G$ (with $G_{\ab} $ of positive rank) must satisfy;
it is this necessary condition that prompted R. Bieri and R. Strebel in the late 1970s
to introduce $\Sigma^0(G;A)$, a precursor of the protagonist of this chapter.

The result has two versions.
The local one,
Theorem \ref{thm:Structure-kernel-fp-group},
deals with the kernel $N$ of a non-zero character $\chi$ which is such 
that  both $[\chi]$ and $-[\chi]$  lie outside $\Sigma^1(G)$.
It assumes that $G$ admits a finite presentation
and asserts that $N$ has a decomposition $S_{-} \star_{S_0}  S_{+}$ as a free product with amalgamation,
in which the amalgam $S_0$ has infinite index in both  factors.
This implies that $N$ contains non-abelian free subgroups.
The global version arises from the first one by assuming
that $G$ has no non-abelian free subgroups.
Then $\Sigma^1(G)$ must contain at least one point out of each pair $\{[\chi], -[\chi] \}$ of antipodal points.
If one denotes the set of all antipodes of points of a subset $\Sigma \subseteq S(G)$  by  $-\Sigma$,
this result can be stated as
\begin{thm} 
\label{thm:Consequence-finite-presentation}
\index{Invariant Sigma1 and@Invariant $\Sigma^1$ and!finite presentability}
Let $G$ be a finitely presented group that contains no non-abelian free subgroup.
Then
\begin{equation}
\label{eq:Necessary-condition-fp-group}
\Sigma^1(G) \cup -\Sigma^1(G) = S(G).
\end{equation}
In particular, 
$\Sigma^1(G)$ 
is non-empty unless  $G/G'$  is finite.
\end{thm}

By combining this result with corollary 
\ref{crl:Characterizing-fg-kernel-rank-1-character}
one obtains the rather surprising
\begin{crl}
\label{crl:Existence-fg-kernels}
Let $G$ be a finitely presented group that does not contain a non-abelian free subgroup.
If $r_{0}(G_{\ab}) \geq 2$  then $G$ contains finitely generated  normal subgroups $N \lhd G$ 
with infinite cyclic factor group $G/N$. 
\end{crl}

\begin{proof}
Since the torsion-free rank of $G_{\ab}$ is greater than 1, the sphere $S(G)$ is connected;
so it can not be written as the disjoint union of two non-empty open subsets. 
Theorem \ref{thm:Consequence-finite-presentation},
on the other hand, 
shows that $S(G)$ is the union of two open subsets,
$\Sigma^1(G)$ and $-\Sigma^1(G)$.
The intersection $\Sigma^1(G) \cap - \Sigma^1(G)$  is therefore a non empty open set 
and so it contains a rank 1 point, say $[\chi\colon G \epi \Z \mono \R]$
(by Lemma \ref{lem:Density-rank-1-points} rank 1 points are dense in $S(G)$).
Corollary
\ref{crl:Characterizing-fg-kernel-rank-1-character} 
then implies 
that the kernel of $\chi$ is a finitely generated group.
\end{proof}

\begin{remarks} 
\label{remarks:Condition-fpr-group}
a) The conclusion of Theorem \ref{thm:Consequence-finite-presentation} need not hold 
if  $G$ does not admit a finite presentation
or if it contains a non-abelian free subgroup.
The first claim can be substantiated by the metabelian groups discussed in example 
\ref{example:Locally-infinite-cyclic-by-infinite-cyclic}:
if  the fraction $p/q \in \Q$ occurring in this example is in reduced form
and if neither $p$ nor $q$ is a unit in the ring $\Z$,
then $\Sigma^1(G)$ is empty.
The second claim is justified by the non-abelian free groups of finite rank:
these groups are finitely related,
their invariants, however, are empty
(see example 3 in section \ref{sssec:Sigma1-first-examples}).

b) The conclusion of Corollary \ref{crl:Existence-fg-kernels} need not be true 
if $G$ is a finitely related soluble group whose abelianization has rank 1.
Witnesses are the metabelian Baumslag-Solitar groups 
$\langle a, t \mid tat^{-1} = a^m \rangle$ for $m> 1$.

c) Theorem  \ref{thm:Consequence-finite-presentation}  goes back to Theorem C in  \cite{BNS}.
\index{Bieri, R.}
\index{Neumann, W. D.}
\index{Strebel, R.}
\end{remarks}
%
\subsection{Main result for finitely related groups} 
\label{ssec:Main-result-fpr-groups}
%
Throughout this section,
$G$ will denote a finitely generated group, with finite generating system $\eta \colon \XX \to G$,
and $\chi \colon  G \to \R$ a non-zero character. 
Our first aim is to construct a collection of connected graphs $\Delta_{b} = \Delta_b(\chi)$ 
that come equipped with a canonical action of $N = \ker \chi $.
We shall prove 
that the quotient graphs $N\backslash \Delta_{b}$ are isomorphic to segments, 
if the real number $b$ satisfies condition \eqref{eq:Condition-parameter-b},
and that they are trees whenever $G$ admits a finite presentation, 
and $b$ is sufficiently large (and satisfies the technical condition \eqref{eq:Condition-parameter-b}). 
%
\subsubsection{Definition and analysis of the graphs $\Delta_b$} 
\label{sssec:Construction-graph-Delta}
\index{Notation!Deltasubb@$\Delta_{b}$}
%
Given a real number $b \geq 0$,
let $D_+$ denote the set of the components of the subgraph $\Gamma_+ =\Gamma_\chi$,
let  $D_-$ be the set of  components of  $\Gamma_- =\Gamma_\chi^{(-\infty, b]}$
and $D_0$ the set of  components of the slice $ \Gamma_{0} = \Gamma_\chi^{[0,b]}$.
The kernel $N = \ker \chi$ acts on the three subgraphs 
$\Gamma_+$, $\Gamma_-$ and $\Gamma_{0}$ 
and hence on the three sets $D_+$,  $D_-$ and $D_0$.
These sets give rise to a bipartite, oriented graph $\Delta = \Delta_b(\chi)$ 
that is equipped with an $N$-action:
\begin{definition}
\label{definition:Graph-Delta-b}
The oriented graph $\Delta_b(\chi)$ has vertex set $\ver (\Delta_b) = D_- \cup D_+$ and
edge set $\edg (\Delta_b) = D_0$. 
The origin and terminus  of an edge  are given by
\begin{align*}
\iota(\EE) &= \text{path component  of $\Gamma_- = \Gamma_\chi^{(-\infty, b]}$ containing } \EE,\\
\tau(\EE) &= \text{path component  of $\Gamma_+ = \Gamma_\chi$ containing } \EE.
\end{align*}
\end{definition}

The next result summarizes some key properties of the graphs $\Delta_{b}(\chi)$.
\begin{prp}
\label{prp:Key-properties-graph-Delta}
Let $\EE_0$  be the component of $ \Gamma_0 = \Gamma_\chi ^{[0,b]}$ that contains the vertex 1 
and set $\CC_- = \iota(\EE_0)$ and $\CC_+ = \tau(\EE_0)$.
Let $S_0$, $S_{-} $ and $S_{+}$ be the stabilizers of the edge $\EE_0$  
and of its end points $\CC_-$, $\CC_+$,
and assume  the real number $b$ satisfies
\begin{equation}
\label{eq:Condition-parameter-b}
b \geq 2 \cdot \max\{ |\chi(x)| \mid x \in \XX\} \text{ and }  b \in \im \chi.
\end{equation}
Then the graph $\Delta_b = \Delta_{b}(\chi)$ and the action of  $N$ on it have the following properties:
\begin{enumerate}[(i)]
\item  $\Delta_b$ is connected;
\item  the  quotient graph $N \backslash \Delta_b$ is an edge with distinct end points;
\item  if $S_0$ has finite index in $S_{+}$   or if   $S_{+}$ is finitely generated
then $-[\chi] \in \Sigma^1(G)$;
\item if $S_0$ has finite index in $S_{-}$  or if  $S_{-}$ is finitely generated then $[\chi] \in \Sigma^1(G)$.
\item
Assume, in addition,
that the group $G$ is finitely related. 
Then $\Delta_{b}$ is a tree whenever $b$ is sufficiently large.
\end{enumerate}
\end{prp}

\begin{proof}
We first describe a procedure 
that associates to every path $p$ in the Cayley graph $\Gamma = \Gamma(G,\XX)$ 
a  unique path $\sigma(p)$ in the graph $\Delta_b(\chi)$.
The path $p$ is a sequence  $(g_1,g_2, \ldots, g_\ell)$ of adjacent vertices.
Set
\[
I = \{j \in \{1,2, \ldots, \ell \} \mid \chi(g_{j}) \notin [0, b] \}.
\]
This set may be empty. 
If so, 
the path runs in a single component $\EE$ of the slice $ \Gamma_{0} = \Gamma_\chi^{[0,b]}$
and we define $\sigma(p)$ to be the path of length 0 which starts at, and ends in, the terminus of the edge $\EE$.
Otherwise, $I$ is a union of subintervals $I_{1}$, $I_{2}$, \ldots, $I_{f}$ made up of consecutive integers.
For each subinterval $I_{k} = (j_{k}, j_{k}+1, \ldots, j_{k}+ m_{k})$, 
the subpath $(g_{j_{k}}, \ldots, g_{j_{k}+ m_{k}})$ runs in a single component $v \in D_{-}\cup D_{+}$.
Denote this component by $v'_{k}$.
It may happen that two consecutive vertices, say $v'_{j}$ and $v'_{j+1}$, coincide;
if so,
replace each maximal constant subsequence of $v'_{1}, v'_{2}, \ldots, v'_{f}$ by its first member; 
let $\sigma(p) = (v_{1}, v_{2}, \ldots, v_{h})$ be the sequence so obtained.
Then  $\sigma(p) $ is a sequence of adjacent vertices of $\Delta_{b}(\chi)$ and thus a path in this graph.
\smallskip

 (i) Given vertices  $v$ and $v'$ of $\ver(\Delta_b) = D_+ \cup D_-$,
we have to find an edge path in $\Delta_b$ from $v$ to $v'$.
If $v =v'$, the empty path will do;
so assume $v \neq v'$.
The vertex $ v$ is a component of either $\Gamma_+$ or $\Gamma_-$;
let  $g \in v$  be a group element that is in this component,
but outside the slice  $\Gamma_{0,b}  =\Gamma_\chi^{[0,b]}$;
similarly, let $g' \in v'$ be a vertex outside this slice.
Since the Cayley graph is connected, there exists a path $p$ from $g$ to $g'$.
This path defines a sequence $\sigma(p)$ in $\Delta_{b}(\chi)$
that leads from $v$ to $v'$.

(ii) By Lemma \ref{lem:Transitive-action-on components} the group $N= \ker \chi$ acts transitively,
both on the set of components  $D_+$ of $\Gamma_\chi$
and on the set of components $D_-$ of $\Gamma_\chi^{(-\infty, b]}$.
As $N$ maps every component of $D_+$ to a component of $D_+$
and every component of $D_-$ to a component of $D_-$,
the quotient graph $N\backslash \Delta_b$ has exactly two vertices.

Consider now an edge $\EE_1$ of the graph $\Delta_b(\chi)$ that terminates in the vertex $\CC_+$.
Then $\EE_{1}$ is a component of the intersection $\Gamma_\chi^{[0,b]} \; \cap\; \CC_+$;
by Lemma \ref{lem:Finding-vertex-in-N} below it contains therefore a vertex $h_{1} \in S_+$.
Its inverse $h_{1}$ then maps $\EE$ onto $\EE_0$.
All taken together,
we have shown that the quotient graph $N \backslash \Delta_b$ is a single edge with distinct end points.

(iii) According to lemma \ref{lem:Finding-vertex-in-N}
every component $\EE$ of $\Gamma_\chi^{[0,b]}\; \cap\; \CC_+$ contains a vertex in $S_{+}$;
so the action of $S_+$ on these components is transitive.
If $S_0$ has finite index in $S_{+}$ the intersection $\Gamma_\chi^{[0, b]}\; \cap\; \CC_+$  
has therefore finitely many components, 
say $\EE_0$, $\EE_1$, \ldots, $\EE_k$.
For each $j \in \{1, 2, \ldots, k\}$ one can find a path $p_j$ in $\Gamma_\chi$ 
that links  the vertex $1 \in  \EE_0$ to a vertex of $\EE_j$.
Let $b' > b$ be a real number that  satisfies condition \eqref{eq:Condition-parameter-b}
and is so large 
that all these paths run inside the slice $\Gamma_\chi^{[0,b']}$.
The components $\EE_0$,  \ldots, $\EE_k$ lie then  in a single component of $\Gamma_\chi^{0,b'}$.
The intersection $\Gamma_\chi^{[0,b']}\cap \CC_+$ is therefore connected;
indeed,
every of its components contains a vertex in $\Gamma_\chi^{[0,b]}$ 
and hence a vertex of one of the subgraphs $\EE_1$,  \ldots, $\EE_k$.
Theorem \ref{thm:Variant-Sigma-1-criterion}
then allows us to conclude  that $-[\chi]$ is in $\Sigma^1(G)$.
If $S_{+}$ is finitely generated, a similar argument shows
that the intersection  $\Gamma_\chi^{[0,b']} \cap \CC_+$ is connected for some $b'>b$, 
and so it follows as before that $-[\chi]$ is in $\Sigma^1(G)$.

(iv) By the choice of $b$,
there exists an element  $g \in G$ with  $\chi(g) = b$.
Its inverse $g^{-1}$ induces an isomorphism of graphs $\Gamma(G,\XX) \iso \Gamma(G, \XX)$ 
that maps the subgraph $\Gamma_- = \Gamma_\chi^{(-\infty, b]}$ onto the graph 
$\Gamma_\chi^{(-\infty,0]} = \Gamma_{-\chi}$.
As $\ker (-\chi) = \ker \chi$,
the claim about $\Gamma_-$ reduces therefore to that about $\Gamma_{-\chi}$. 

(v) We assume now that $G$ admits a finite presentation,
say $\pi \colon \langle \XX;\, \RR\rangle \iso G$.
Each relator $r \in \RR$ defines a closed path  $p_{1,r} = (1, r)$ in the Cayley graph $\Gamma(G, \XX)$.
Let $d$ be the largest of the diameters of these paths ``measured in the direction of''  $\chi$;
i.\;e., set
\begin{equation}
\label{eq:Definition-diameter-d}
d = \max\{ |\chi(g) - \chi(h)| \mid g,h  \text{ vertices of } p_{1,r}  \text{ and } r \in \RR\}.
\end{equation}
Let $b \geq d$ be a real number satisfying condition  \eqref{eq:Condition-parameter-b}. 
\emph{Then $\Delta_b(\chi)$ is a tree}.

Indeed, let $\bar{p} =(v_1,v_2, \ldots, v_\ell = v_1)$ be a closed path in $\Delta_b$;
we can assume it is not reduced to a point.
Pick a vertex $g_j$ in each of the component $v_j$ that lies outside the slice $\Gamma_\chi^{[0,b]}$
and choose,
for each $j < \ell$, 
a path $p_j$ that leads from $g_j$ to $g_{j+1}$. 
The concatenation of these paths is a closed path $p = (g_1, w)$ in the Cayley graph $\Gamma(G, \XX)$.
The  word $w$ is a relator of the group $G$;
as $\RR$ is a defining set of relators 
$w$  is therefore freely equivalent to a product $w_1 \cdot w_2 \cdots w_h$
each subword $w_j$ being the conjugate $u_jr_j^{\varepsilon_j} u_j^{-1}$ 
of an element in $\RR \cup  \RR^{-1}$ by an $\XX^\pm$-word  $u_j$.
The path $p = (g_1,w)$ is then freely equivalent to the concatenation  $p_1 \cdot p_2 \cdots p_h$
of paths $p_j = (g_1, w_j)$.
Each of them is a closed path based at $g_1$ 
which is made up of an initial path $p_{j,1} = (g_1, u_j$ 
leading from $g_1$ to the  point $h_j = g_1 \cdot \eta(u_j)$,
the  loop $p_{j,2} = (h_j, r_j^{\varepsilon_j})$ 
and the final path $p_{j,3}  = p_{j,1}^{-1}$.
Since $b \geq d$, where $d$ is the constant given by \eqref{eq:Definition-diameter-d},
the path $p_{j,2}$ runs either entirely in $\Gamma_{\chi}$,
or entirely in $\Gamma_\chi^{(-\infty, b]}$ and contains then a vertex $g$ with $\chi(g) < 0$.
The sequence $\sigma(p_{j,2})$ is therefore made up of a single vertex-
It follows that the given path $\bar{p} =(v_1,v_2, \ldots, v_\ell)$ in the graph $\Delta_b(\chi)$
is freely equivalent to a path $\bar{p'}$ 
that is a concatenation of subpaths of the form $(v_1, v_k, v_1)$
with $v_k = v_1$, or $v_k \in D_-$ if $v_1 \in D_+$, respectively  $v_k \in D_+$ if $v_1 \in D_-$.
This new path is clearly equivalent to the constant path based at $v_1$.
\end{proof}

We are left with establishing
\begin{lem}
\label{lem:Finding-vertex-in-N}
Assume the real number $b$ satisfies 
\begin{equation}
\label{eq:Condition-on-b-for-reordering}
b \geq 2 \cdot M \text{ where } M = \max\{ |\chi(x)| \mid x \in \XX\}.  
\end{equation}
Then every component $\EE$ of the intersection $\Gamma_\chi^{[0,b]} \; \cap\; \CC_+$ 
contains a vertex in $S_+$.
(As before,  $\CC_+$ denotes the path component of $\Gamma_\chi$ that contains the vertex 1.)
\end{lem}

\begin{proof}
Since the component $\EE$ is contained in $\CC_+$ 
there exists a path $p = (1, w) $ in $\CC_+$ 
that leads from 1 to a vertex $h \in \EE$;
let $y_1 y_2 \cdots y_\ell$ be the spelling of $w$.
We  aim at reordering  the letters of $w$ in such a way
that the new word $w'$ yields a path $p= (1, w')$
that stays inside $\Gamma_\chi^{(0,b]}$.
Then the  path $p' = (h,(w')^{-1})$ starts  with the vertex $h \in \EE$ and
terminates with a vertex having $\chi$-value 0;
its endpoint then yields the desired point in  $S_{+}$.

Consider the sequence $f \colon \{1,2, \ldots, \ell\} \to \R$ of real numbers $f_j = \chi(y_j)$.
By assumption, $s = \sum\nolimits_{1\leq j \leq \ell} f_j \in [0,b]$.
We have to find a permutation $\pi$ of $\{1,2, \ldots, \ell\}$ so that all the partial sums 
$s_{\pi, k} = \sum\nolimits_{1\leq j \leq k} f_{\pi(j)}$ lie in the intervall $[0,b]$.

To establish the existence of $\pi$
we prove a stronger assertion,
namely: 
\emph{for every number $c \in [0,b]$ and every sequence  $f \colon \{1,2, \ldots, \ell\} \to \R$
satisfying $|f_{j}|\leq M$ and $c + \sum\nolimits_{1\leq j \leq \ell} f_j \in [0,b]$,
there exists a permutation $\pi$ such that every partial sum 
$c +  \sum\nolimits_{1\leq j \leq k} f_{\pi(j)}$ lies in $[0,b]$}.
The proof will be by induction on $\ell$.

If $\ell = 0$ the claim holds.
So assume $\ell > 0$.
Let $J_+$ be the set of indices with $f_j \geq 0$
and let $J_-$ be the complement of $J_{+}$ in the interval $J = \{1,2, \ldots, \ell \}$. 
Set
\[
s_{+} = \sum \{s_{j}\mid j \in J_{+} \} \quad \text{and} \quad s_{-} = \sum \{s_{j}\mid j \in J_{-} \}.
\]
Two cases now arise.
If $c \leq b/2$, consider $c_{+} =c+ s_{+}$.
If $c_{+} \leq b$, let $\pi$ be a permutation which lists every $j \in J_{+}$ before the indices $j' \in J_{-}$.
Then the partial sums $c +  \sum\nolimits_{1\leq j \leq k} f_{\pi(j)}$  all belong to $[0,b]$.
Otherwise, 
there exists a subset $J'_{+}$ of $J_{+}$ so that $c' = c +  \sum \{f_{j}\mid j \in J'_{+} \}$
lies in the interval $] b/2, b]$. 
As the set $J'_{+}$ has at least one element,
$J' = \{1, 2, \ldots, \ell \} \smallsetminus J'_{+}$ has fewer than $\ell$ elements 
and so the induction hypothesis applies to the couple $(c',J')$.
If, on the other hand, $c_{+} \geq b$ a reasoning, 
similar to the preceding one but with the rôles of $J_{+}$ and of $J_{-}$ exchanged, 
allows one, either to complete the proof in a single step
or to  apply the induction hypothesis.
\end{proof}

%
\subsubsection{Statement and proof of the local version}
\label{sssec:Fpr-groups-first-main-result}
%
As mentioned in the introduction to section \ref{sec:Sigma1-and-fp-groups},
the main result has two versions.
We are now ready to establish the local version, 
namely
\begin{thm} 
\label{thm:Structure-kernel-fp-group}
Let $G$ be a finitely generated group and assume $\chi$ is a non-zero character of $G$
so that both $[\chi]$ and $-[\chi]$ lie outside $\Sigma^1(G)$.
If $G$ is finitely related,
the kernel $N$ of $\chi$ admits a decomposition 
$ N= S_{-} \star_{S_{0}} S_{+}$ 
as a free product with amalgamation 
where $S_0$ has infinite index in both factors $S_{+}$ and in $S_{-}$
and where both factors are infinitely generated.
Moreover, $N$ contains non-abelian free subgroups.
\end{thm}

\begin{proof}
Let $b \geq 0$ be a real number satisfying condition \eqref{eq:Condition-parameter-b}
and let $\Delta_b(\chi)$ be the $N$-graph given in Definition \eqref{definition:Graph-Delta-b}.
This graph is connected (claim  (i) of Proposition  \ref{prp:Key-properties-graph-Delta})
and its quotient graph $N\backslash \Delta_b$ is a segment (claim (ii)).
Let $\EE_0$ be the edge, i.\;e, the  component of the slice $\Gamma_\chi^{[0,b]}$ 
that contains the vertex $1 \in G$. 
Define $S_0$ to be the stabilizer of $\EE_0$ 
and $S_{-}$ and $S_{+}$ to be the stabilizers of the origin and the terminus of this edge.
Since $-[\chi] \notin \Sigma^1(G)$,
the index of $S_0$ in $S_{+}$ is infinite and $S_{+}$ is an infinitely generated group 
(by claim (iii) of the quoted proposition);
since $[\chi] \notin \Sigma^1(G)$ it follows similarly that the index $|S_- \colon S_0|$ is infinite 
and that $S_{-}$ is infinitely generated.
Finally,
the assumption that $G$ is finitely related implies that the graph $\Delta_b$ is a tree,
provided $b$ is large enough 
(this is the message of point (v) of Proposition \ref{prp:Key-properties-graph-Delta}).
The Bass-Serre Theory (see, e.\;g., \cite[§ 4, Theorem 6]{Ser80}) therefore allows us to conclude
that the canonical inclusions induce an isomorphism $\kappa \colon S_{-}\star _{S_{0}} S_{+} \iso N$.

We are left with proving that $N$ contains a free subgroup of rank 2;
this fact is widely known and can be seen as follows:
pick $a \in S_{+} \smallsetminus S_{0}$ 
and $b$, $b' \in S_{-} \smallsetminus S_{0}$   with $b^{-1}b' \notin S_0$
and put $x = [a,b] = a b a^{-1}b^{-1}$ and $y = [a,b']$. 
Then the elements 
\begin{align*}
x &= [a,b] = a b a^{-1}b^{-1},  &x^{-1} &= [b,a] = bab^{-1}a^{-1}, \\
y &= [a,b'] = ab'a^{-1} (b')^{-1},  &y^{-1} &= [b',q] = b'a(b')^{-1}a^{-1}
\end{align*}
are in normal form
and the products $xy$, $xy^{-1}$, $x^{-1}y$, \ldots, $y^{-1}x$ and  $y^{-1}x^{-1}$ admit so little of cancellation
that no  freely reduced word $w \neq 1$  in $x$, $x^{-1}$, $y$ and  $y^{-1} $ can represent $1 \in N$. 
\end{proof}

\begin{remarks}
\label{remarks:Structure-kernel-fpr-group-rank-1}
a) The idea of the above proof of Theorem \ref{thm:Structure-kernel-fp-group}
goes back to a lecture Robert Bieri gave in the late 1980's;
cf.\;the proof of Theorem 5.1 in Chapter I of  \cite{BiSt92}.
Bieri's proof was not well-known at the time,
as is indicated by a remark in section 1.6.4 of Peter Shalen's survey \cite{Sha91}.
\index{Bieri, R.} 
\index{Shalen, P.}

b) Theorem \ref{thm:Consequence-finite-presentation},
the global version of the result just proved, follows easily from the basic one.
Indeed, let $G$ be a finitely related group  that contains no free subgroup of rank 2
and let $\chi$ be a non-zero character. 
Then its kernel $N$ cannot be a generalized free product
$S_{-} \star_S  S_{+}$ where $S_0$ has infinite index in both $N_{-}$ and in $N_{+}$.
By Theorem \ref{thm:Structure-kernel-fp-group} 
at least one of the points $[\chi]$, $-[\chi]$ must therefore lie in $\Sigma^1(G)$.
\end{remarks}
%
%
\subsection{Main result for groups of type $\FP_2$}
\label{ssec:Variant-for-almost-fpr-groups}
%
In Theorem \ref{thm:Consequence-finite-presentation} 
the  group $G$ is assumed to admit a finite presentation;
the conclusion of the theorem need not be valid if this hypothesis is omitted
(cf.  remark \ref{remarks:Condition-fpr-group}a).
The hypothesis that $G$ be finitely related can, however, be slightly weakened:
it suffices that $G$ is of type $\FP_2$ over a commutative ring $K$.

We begin by recalling the relevant definitions and some facts.
Let $R \triangleleft F \overset{\pi}{\epi} G$  be a free presentation of $G$.
The abelianization $R_{\ab}$ of $R$ is a left $G$-module by conjugation in $F$, 
called the \emph{relation module} of the presentation. 
This module is the first term of the short exact sequence
 \begin{equation}
 \label{eq:Presentation-augmentation-ideal}
R_{\ab} \overset{\kappa}{\mono}\Z{G} \otimes_{\Z{F}} IF \overset{\nu}{\epi} IG 
 \end{equation}
 (see, \eg, \cite[Thm.\;IV,6.3]{HiSt97}).
 Here $IF$ and $IG $ denote the augmentation ideals of $\Z{F}$ and $\Z{G}$, respectively,
 and the group ring $\Z{G}$ in the term $\Z{G} \otimes_{\Z{F}} IF$ 
 is viewed as a right $\Z{F}$-module via the projection $\pi$.
 The map $\kappa$ sends the class $f \cdot R$ to $1_{\Z{G}} \otimes (f-1)$ 
 and $\nu$ takes  $g \otimes (f-1)$ to the product $g \cdot (\pi(f) -1)$.
 \index{Definition of!relation module}%
 
The augmentation ideal $IF$ of the group ring $\Z{F}$ is a free $\Z{F}$-module  
whose rank equals the rank of $F$ (see, \eg, \cite[Th,.\;IV.5.5]{HiSt97}), 
and so the short exact sequence \eqref{eq:Presentation-augmentation-ideal} 
is a \emph{free presentation of the the augmentation ideal} $IG$,
viewed as a left $\Z{G}$-module.

Consider now a commutative ring $K$.
\footnote{The ring $K$ is always assumed to possess a unit element $1 \neq 0$.}
Since the additive group of $IG$ is free abelian, the sequence \eqref{eq:Presentation-augmentation-ideal}
remains exact when tensored with $K \otimes_{\Z} -$.
Schanuel's Lemma (see, \eg, \cite[Lemma VIII.4.2]{Bro94}) therefore implies 
that whether or not the $KG$-module $K \otimes R_{\ab}$ is finitely generated 
does neither depend on the choice of the  \emph{finitely generated} free group $F$
nor on the epimorphism $\pi \colon F \epi G$.
\index{Schanuel's Lemma}%

The trivial $G$-module $K$ has the free presentation 
\[
0  \to K \otimes IF \longrightarrow KG \longrightarrow K \to 0.
\]
Upon splicing this sequence 
with the short exact sequence obtained by applying the functor $K \otimes -$ to 
sequence \eqref {eq:Presentation-augmentation-ideal}, 
one arrives at a beginning of a $KG$-free resolution of $K$, namely
\begin{equation}
 \label{eq:Beginning-free-resolution-K}
0 \to  K \otimes R_{\ab} 
\overset{1 \otimes \kappa}{\longrightarrow}
K{G} \otimes_{\Z{F}} IF 
\overset{\nu}{\longrightarrow} KG 
\longrightarrow K \to 0.
\end{equation}
The previous considerations and Schanuel's Lemma then lead to the following
 \index{Schanuel's Lemma}%
\begin{lem}
\label{lem:Independence-free-presentation}
For every finitely generated group $G$  and every commutative ring $K$ (with $1 \neq 0$),
the following statements are equivalent:
\begin{enumerate}[(i)]
\item there exists a free presentation $R \triangleleft  F \epi G$ with $F$ finitely generated 
so that $K \otimes_{\Z} R_{\ab}$ is a finitely generated left $KG$-module;
\item for every free presentation $R_1 \triangleleft  F_1 \epi G$ with $F_1$ finitely generated 
the $KG$-modules $K \otimes_{\Z} (R_1)_{\ab}$ is  finitely generated;
\item the group $G$ is of type $\FP_2$ over $K$.
\end{enumerate}
\end{lem}
%
\subsubsection{A characterization of groups of type $\FP_2$}
\label{sssec:Characterization-almost-fpr-groups}
%
The proof of Theorem \ref{thm:Structure-kernel-almost-fp-group} below 
makes use of a characterization of groups of type $\FP_2$ over a commutative ring $K$. 
This result is Lemma 2.1 in  \cite{BiSt78}, 
except for the fact that there groups of type $FP_2$ over $K$ are called 
\emph{almost finitely presented over $K$}.
\index{Almost finitely presented group}
\index{Bieri, R.}
\index{Strebel, R.}
\begin{lem}
\label{lem:Characterization-almost-fp}
Let $R \triangleleft F \overset{\pi}{\epi} G$ be a free presentation of the finitely generated group $G$
with  $F$ free of finite rank. Then the $KG$-module $K \otimes R_{\ab}$ is finitely generated if, and only if,
there exists a finite subset $\RR_1 \subset R$ so that the kernel $M$ of the projection 
$\pi_* \colon G_1 = F /\gp_F(\RR_1) \epi G$ has the property that $K \otimes M_{\ab} = 0$.
\end{lem}
\index{Finiteness condition FP2@Finiteness condition $\FP_2$!characterization}

\begin{proof} 
Let $R_1 \subseteq R$ be normal subgroup of $F$. 
The extension $R_1 \overset{\mu}{\mono}  R \epi R/R_1$ induces then the right exact sequence 
\begin{equation}
\label{eq:Auxiliary-short-exact-sequence}
K \otimes_\Z (R_1\cdot R')/R' \overset{\mu_*}{\longrightarrow} K \otimes_\Z R_{\ab} 
\longrightarrow K \otimes_\Z (R/R_1)_{\ab} \to 0
\end{equation}
of  $KG$-modules.
Assume now the $KG$-module $K \otimes R_{\ab}$ is finitely generated over $KG$.
Pick a finite set $\RR_1 \subset R_1$ 
so that the set $\{1 \otimes r_1R' \mid r_1 \in \RR_1 \}$ generates $K \otimes R_{\ab}$.
Then the embedding $\mu_*$  in the sequence \eqref{eq:Auxiliary-short-exact-sequence} is surjective
and so $K \otimes (R/R_1)_{\ab} = 0$.
But $R/R_1$ is the kernel $M$ of the projection $\pi_* \colon G_1 = F /\gp_F(\RR_1) \epi G$;
so we have proved that $K \otimes M_{\ab} = 0$.

Conversely, 
assume that $R_1  \subseteq R$ is the normal closure of a finite subset $\RR_1$ 
and $R/R_1 = \ker (G_1 = F/R_1 \epi F/R)$ is such that $K \otimes (R/R_1)_{\ab} = 0$.
Then the map $\mu_*$ in sequence \eqref{eq:Auxiliary-short-exact-sequence}
is an epimorphism, 
whence $K  \otimes R_{\ab}$ is generated, as a $KG$-module, 
by  the finite set $\{\mu(1 \otimes r \cdot R') \mid r \in \RR_1\}$. 
\end{proof}

\begin{remarks}
\label{remarks:Almost-finitely-presented}
a) Let $G$ be  a group that is of type $\FP_2$ over the commutative ring $K$ 
and  let $\rho \colon K \to L$  be a homomorphism of commutative  rings  (taking $1_K$ to $1_L$).
Upon tensoring the exact sequence \eqref{eq:Beginning-free-resolution-K} with $L \otimes_K -$, 
one obtains an exact sequence of $LG$-modules;
this sequence  shows that $G$ is of type $\FP_2$ over $L$.
The preceding remark applies, in particular, to the canonical ring homomorphism $\Z \to L$. 
Moreover, 
if $R \triangleleft F$ is finitely generated as a normal subgroup of  $F$, 
then $R_{\ab}$ is finitely generated as a $F/R$-module. 
So we have proved
\begin{lem}
\label{lem:Almost-finitely-presented}
For every finitely generated group $G$ the following claims hold:
\begin{enumerate}[(i)]
\item 
If $G$ is finitely presented then $G$  is of type $\FP_2$ over $\Z$.
\item 
If $G$ is  of type $\FP_2$ over $\Z$ 
then $G$ is  is of type $\FP_2$ over every commutative ring $K$.
\end{enumerate}
\end{lem}

b) The converse of implication (i) has been an open problem for many years.
In 1997 it was settled in the negative by  M. Bestvina and N. Brady in \cite{BeBr97}.
Their counter examples are  kernels of rank 1 characters of suitably chosen right angled Artin groups.
\index{Bestvina, M.}
\index{Noel, B.}
\index{Graph groups!examples}

c) The converse of (iii) is known to be false since 1980 (see \cite[p. 464]{BiSt80}).
A counter example is provided by the soluble matrix group 
put forward by H. Abels in  \cite{Abe79}.
\index{Abels, H.}
\begin{example}
\label{example:Abels-matrix-group}
\index{Soluble groups!Abels' example}
Let $G_1$ be  the multiplicative group of all matrices of the form
\[
\begin{pmatrix}
1 & \star & \star & \star\\
0 & \star & \star & \star\\
0 & 0 & \star & \star\\
0 & 0 & 0 & 1 
\end{pmatrix}
\in \text{GL}_4(\Z[\tfrac{1}{2}]).
\]
where $\star$ stands for entries in the ring $\Z[\frac{1}{2}]$ of all dyadic rationals;
in addition, the diagonal entries are required to be positive (and hence of the form  $2^m$ with $m\in \Z$). 
According to  \cite[p. 205]{Abe79}, the group $G_1$ is finitely presented. 

The centre $Z$ of $G_1$ consists of all matrices with diagonal entries equal to 1,
and all other entries equal to 0, 
except for the entry in the upper right corner 
which is an arbitrary element of $\Z[\tfrac{1}{2}]$.
So $Z$ is isomorphic to the additive group of $\Z[\tfrac{1}{2}]$. 
Since  $Z$ is not finitely generated as a normal subgroup of $G_1$, 
the factor group $G = G_1/Z$ cannot be finitely presented. 
On the other hand, 
$\F_2 \otimes Z = 0$ 
and so $G$ is of type $\FP_2$ over the  field $\F_2$
by Lemma \ref{lem:Characterization-almost-fp}.
\end{example}
\end{remarks}

%
\subsubsection{Strengthened version of Theorem \ref{thm:Structure-kernel-fp-group}} 
\label{sssec:Necessary-condition-fpr-group-improved}
The conclusion of Theorem  \ref{thm:Structure-kernel-fp-group} remains valid
if the assumption that $G$ be finitely related is replaced  by the weaker assumption 
that $G$ be of type $\FP_2$ over some commutative ring $K$:
\begin{thm} 
\label{thm:Structure-kernel-almost-fp-group}
\index{Invariant Sigma1 and@Invariant $\Sigma^1$ and!FP2@$\FP_2$}
Let $G$ be a finitely generated group and let $\chi$ be a non-zero character of $G$
so that both $[\chi]$ and $-[\chi]$ lie outside $\Sigma^1(G)$.
If $G$ is of type $FP_2$ over some commutative ring $K$,
the kernel $N$ of $\chi$ admits a decomposition  $ N= S_{-} \star_{S_{0}}  S_{+}$ 
as a free product with amalgamation 
where $S_0$ has infinite index in both factors $S_{-}$ and in $S_{+}$
and where both factors are infinitely generated.

Moreover, $N$ contains non-abelian free subgroups.
\end{thm}

\begin{proof}
The construction of the graphs $\Delta_b(\chi)$ and  properties (i) through (iv) of Proposition
\ref{prp:Key-properties-graph-Delta}
require merely that $G$ is finitely generated.
The finite presentation is only used in establishing that $\Delta_b$ is a tree 
if $b$ is sufficiently large. 
The fact that $\Delta_b$ is a tree under the weaker hypothesis can be seen as follows:

Lemma \ref{lem:Characterization-almost-fp} yields a short exact sequence of groups 
$M \mono  G_1 \overset{\pi}{\epi} G$ 
where $G_1$ is finitely related and $K \otimes  M_{\ab} = 0$. 
Let $\psi = \chi \circ \pi$ be the character of $G_1$ induced by the given character $\chi$
and put $N_1 = \ker \psi$. 
Next choose a finite presentation $F/R_1$ of $G_1$ 
and use the basis  $\XX$ of $F$ as a generating system of $G_1$  and, via $\pi$,  of $G$.
The canonical map $\pi \colon G_1 \epi G$  gives rise to a covering map 
$\pi_*\colon \Gamma_1 \epi \Gamma$  of the corresponding Cayley graphs.
This map  sends the subgraph $\Gamma_{1,+} =\Gamma(G_1, \XX)_\psi$ 
onto the subgraph $\Gamma_+ = \Gamma(G, \XX)_\chi$ 
and hence each of the connected components of $\Gamma_{1+}$ onto a component of $\Gamma_+$.
This implies, in particular,
that $\pi \colon G_1 \epi G$ maps the stabilizer  $S_{1+} = (\ker \psi) \; \cap \; \CC_{1+}$ 
of the component of $\Gamma_{1+}$ containing $1_{G_1}$  
onto the stabilizer $S_{+}$ of $\CC_+$.
One verifies similarly
that $\pi \colon G_1 \epi G$ maps the stabilizer $S_{1-}$  of $C_{1-}$,
the connected component of $\Gamma_\psi^{(-\infty, b]}$  containing $1_{G_1}$, 
onto the stabilizer $S_{-}$ of the connected component $\CC_-$,
and that $\pi$ maps the stabilizer $S_1$ of the connected component $\EE_{1,0}$ of  $\Gamma_\psi^{[0,b]}$ 
onto the stabilizer $S_0$ of the edge $\EE_0$. 

We are now ready to prove that the graph $\Delta_b(\chi)$ is a tree if $b$ is suitably chosen.
Recall that the character  $\psi$ is the composition $\chi \circ \pi$.
Since $M$ is the kernel of the epimorphism $\pi \colon G_1 \epi G$,
the columns of the following diagram \eqref{eq:First-diagram}
are exact.
Select now the real number $b \in \im \psi = \im \chi$ 
so large that $\Delta_b(G_1)$ is a tree (cf. Theorem \ref{thm:Structure-kernel-fp-group}).
Then the  inclusions of $S_{1-}$, $S_{1+}$ and of $S_1$ in  $G_1$ induce 
an isomorphism $S_{1-} \, \star_{S_1}  S_{1+}  \iso N_1 = \ker  \psi$.

\begin{equation} %
\label{eq:First-diagram}
\xymatrix{
M \ar[r]^{=}\ar[d]_{\iota_*}          &  M \ar[d]^{\iota}             & \\
N_1 = \ker \psi  \ar[r]\ar[d]_{\pi_*}& G_1\ar[r]^{\psi}\ar[d]^{\pi} & \R \ar[d]^{=}\\
N = \ker \chi \ar[r]                     &  G \ar[r]^{\chi}                 & \R
}
\end{equation}

If one replaces the middle term $N_1$ of the \emph{left} column of diagram  \eqref{eq:First-diagram}  
by $S_{1-} \, \star_{S_1}  S_{1+}$,
the \emph{first row} of diagram \eqref{eq:Second-diagram} results:
\begin{equation}
\label{eq:Second-diagram}
\xymatrix{
M \ar[r]\ar[d]^{\pi_*}  &   S_{1 -}\, \star_{S_1} \, S_{1+} \ar[r] \ar[d]^{\pi_*} & N \ar[d]^{=}\\
L = \ker \rho \ar[r] & S_{-} \, \star_{S_0} \, S_{+} \ar[r] &N
}
\end{equation}
Its maps are all induced by the obvious inclusions and so the diagram commutes.

In a previous part of the proof  we have seen 
that  $\pi$ maps the stabilizers $S_{1-}$, $S_{1+}$ and of $S_1$ 
occurring in the middle term of the first row 
onto the corresponding stabilizers occurring in the second row.
The middle vertical map  is therefore surjective, 
whence so is the left vertical map.
This fact implies that  $L$ is reduced to the neutral element.
Indeed,
since $ K \otimes M_{\ab} = \{0\}$ and as abelianization and  tensor product are right exact functors,
$ K \otimes L_{\ab} = \{0\}$. 
On the other hand, 
the epimorphism $ S_{-} \, \star_{S_{0}} \, S_{+}  \epi  N$ is induced by the inclusions of the factors;
so its kernel $L$ intersects them trivially
and thus acts freely on the canonical tree of   $ S_{-} \, \star_{S_0} \, S_{+}$, 
whence $L$ is a free group. 
As $ K \otimes L_{\ab} = \{0\}$ it must   be trivial. 

The canonical inclusions thus induce an isomorphism  $ S_{-} \, \star_{S_0}\, S_{-}  \iso  N$.
This can only hold if the graph $\Delta_b$ is a tree (see, e.g., \cite[§4, Theorem 6]{Ser80}).
\end{proof}
Theorem \ref{thm:Structure-kernel-almost-fp-group} is the local version of our final result,
just as Theorem \ref{thm:Structure-kernel-fp-group} is the local version of Theorem 
\ref{thm:Consequence-finite-presentation}:
\begin{thm} 
\label{thm:Consequence-almost-fp}
\index{Invariant Sigma1 and@Invariant $\Sigma^1$ and!FP2@$\FP_2$}
Let $G$ be a finitely generated group that does not contain a non-abelian free subgroup.
If $G$ is of type $FP_2$ over some commutative ring $K \neq \{0\}$ then
\begin{equation}
\label{eq:Necessary-condition-almost-fp-group}
\Sigma^1(G) \cup -\Sigma^1(G) = S(G).
\end{equation}
\end{thm}
%
 
%
 
%
\chapter{Complements to the Cayley graph approach to $\Sigma^1$}
\label{ch:Sigma-1-Cayley-graph-complements}
%
%
In Chapter \ref{ch:Sigma-1-Cayley-graph}, 
the invariant $\Sigma^1(G)$ has been defined via the Cayley graph $\Gamma(G, \XX)$  
associated to a  finite generating system $\eta \colon \XX \to G$ of the group $G$
and some of its properties have been established
There exist alternate definitions;
they are the subjects of  Chapters 
\ref{ch:Alternative-definitions-overview} 
and 
D.
Prior to moving on to them,
we discuss some further topics 
that can conveniently be treated in the frame-work of Cayley graphs.

In section \ref{sec:Computing-Sigma1-via-change-groups}
we explain how the invariants of a group and its subgroups are related.
An algebraic version of the $\Sigma^1$-criterion is then derived in section \ref{sec:Sigma1-criterion-revisited};
it can be used to construct algorithms which find a subset of $\Sigma^1$.
Section \ref{sec:Sigma1-via-HNN-extensions} concentrates on rank 1 points in $\Sigma^1$.
According to Proposition \ref{prp:Ascending-HNN-extension}, 
these points can be described in terms of ascending HNN-extensions with a finitely generated base group.
In some cases,
all of $\Sigma^1$ can be found  with this approach.

In the final section,
an algorithm for computing $\Sigma^1$ of a one-relator group is established;
it is due to Ken Brown \cite{Bro87b}.
The present justification combines results from sections \ref{sec:Sigma1-criterion-revisited} and 
\ref{sec:Sigma1-via-HNN-extensions}
with basic results about one-relator groups, first proved by W. Magnus in the 1930's.

%
%
%
\section{Computing $\Sigma^1$ via change of groups}
\label{sec:Computing-Sigma1-via-change-groups}
%
%
One of the themes of this monograph
is the computation of the invariant $\Sigma^1$ for various classes of groups.
In so doing, 
it is often helpful to consider, not only the group one is primarily interested in, 
but also related groups,  for instance subgroups or quotient groups. 
This  section collects results that are useful in this context.
It starts out with some remarks on the morphism of character spheres 
induced by a group homomorphism 
and then derives some relations between the invariants of two or more suitably related groups.
As an application, we determine in section \ref{ssec:Application-right-angled-Artin-groups}
the invariants of a graph groups.
%
\subsection{Morphisms between character spheres} 
\label{ssec:Morphisms}
\index{Character sphere!morphisms}
%
Let $\varphi \colon G \to G_{1}$ be a homomorphism of finitely generated groups. 
It induces an $\R$-linear  map  $\Hom(\varphi, \R) \colon \Hom(G_{1},\R) \to \Hom(G, \R)$ 
sending $\chi \colon G_{1} \to \R$ to $\chi \circ \varphi$. 
If a character $\chi \colon G \to \R$  vanishes on the subgroup $\im \varphi $ of $G_{1}$ 
the composition $\chi \circ \varphi$ represents the zero-class $[0]$, 
a class that lies outside of $S(G_{1})$.
Accordingly we define the morphism induced by $\varphi$ as follows:
\begin{equation}
\label{eq:Induced-map-spheres}
\varphi^* \colon  S(G_{1}, \im \varphi)^c \to S(G), \quad \varphi^*[\chi]  = [\chi \circ \varphi].
\end{equation}
%
\subsubsection{Determining the image of $\varphi^*$} 
\label{sssec:Determination-image-morphism}
%
It will be useful to have a good idea of the form of the fibers and of the image of the morphism $\varphi^*$.
We begin by determining its image. It is a subsphere,
for the image of the  $\R$-linear map $\Hom (G_{1}, \R) \to \Hom (G, \R)$ is a subspace 
and so its rays emanating from the origin form a subsphere.

One can describe this image more explicitly:
\begin{lem}
\label{lem:Image-morphism}
\begin{enumerate}[(i)]
\item The homomorphism $\varphi \colon G \to G_{1}$ induces the morphism
\begin{equation}
\label{eq:Morphism-induced-homomorphism}
\varphi^* \colon  S(G_{1}, \varphi(G))^c \epi S(G,  \varphi^{-1}(G'_{1})) \incl S(G)
\end{equation}
\item If $\varphi$ is surjective, 
the morphism takes on the simpler form
\begin{equation}
\label{eq:Morphism-induced-epimorphism}
\varphi^* \colon S(G_1) \iso S(G, \ker \varphi) \incl S(G).
\end{equation}
\end{enumerate}
\end{lem}
\begin{proof}
(ii) If $\varphi \colon G \epi G_1$ is surjective the induced morphism
\[
\varphi^* \colon \Hom(G_1, \R) \to \Hom(G,\R), \quad \psi \mapsto \psi \circ \varphi
\]
is injective and its image is the subspace of $\Hom(G,\R)$ 
consisting of the characters that vanish on $\ker \varphi$.
This proves assertion (ii).

(i) The homomorphism $\varphi \colon G \to G_1$ induces a homomorphism of abelian groups
$\varphi_* \colon G/\varphi^{-1}(G_1') \mono (G_1)_{\ab}$;
this homomorphism is injective,
As $\R$ is a divisible group, hence an injective $\Z$-module,
every character of  $G/\varphi^{-1}(G_1')$  extends to a character of $(G_{1})_{\ab}$
(see, \eg, \cite[4.1,2]{Rob96}).
The morphism 
\[
\varphi^* \colon S(G_1, \im \varphi)^c \to S(G/(\varphi^{-1}(G_1'))\,)
\] is therefore surjective.
Now use that $S(G/(\varphi^{-1}(G_1'))) \iso S(G, \varphi^{-1}(G_1'))$ is an isomorphism (by claim (ii) ).
\end{proof}
%
\subsubsection{Determining the fibers of $\varphi^*$} 
\label{sssec:Determination-fibers-morphism}
%
Let $[\psi]$ be a point in the image of $\varphi^*$ and let $[\chi]$ and $[\chi_{1}]$ 
two points in the fiber above $[\psi]$. 
Then $[\chi \circ \varphi] = [\psi] = [\chi \circ \varphi]$ and so there is a positive scalar $\lambda$ 
so that $\chi \circ \varphi= (\lambda. \chi') \circ \varphi$. The difference $ (\lambda. \chi') - \chi$ is then in the kernel $V$ of the $\R$-linear map 
$\Hom(\varphi, \R) \colon \Hom(G_{1},\R) \to \Hom(G, \R)$.
Conversely, 
if $v \in V$ then $[\chi]$ and $[\chi + v]$ have the same image under $\varphi^*$.
The points in the fiber above $[\psi]= \varphi^*([\chi])$ can therefore be parametrized by the characters in the \emph{affine} subspace $\chi + V$ of $\Hom(G_{1}, \R)$.
The fiber is therefore an open hemisphere in the sphere made up of the rays in the subspace spanned
by $V$ and $\chi$; 
in particular, 
the dimension of the fiber is equal to the dimension of $V$
and the boundary of the fiber is the subsphere  $S(G_{1}, \im \varphi)$.
Since the dimension of $V$ is the torsion-free rank of 
$(G_1)_{\ab}/(\im \varphi_{\ab}) \iso G_{1}/ (\varphi(G) \cdot G'_{1})$,
we have proved
\begin{lem}
\label{lem:Fiber-morphism}
The fiber of the point $\varphi^*([\chi])$ consists of the points represented 
by the characters in $\chi + V$
where $V$ is the kernel of 
\[
\Hom(\varphi, \R) \colon \Hom(G_{1},\R) \to \Hom(G, \R).
\]
It is a hemisphere in a subsphere of dimension $\dim V = r_{0}(G_{1}/(\varphi(G) \cdot G'_{1})$ 
and it is bounded by the subsphere $S(G_{1}, \varphi(G))$.
\end{lem}

The previous lemma allows one to spell out
when $\varphi$ induces an injective morphism:
\begin{crl} 
\label{crl:Injectivity-phi-star}
For every homomorphism $\varphi \colon  G \to G_{1}$ of finitely generated groups, 
the following three conditions are equivalent:
\begin{enumerate} [(i)]
\item $\varphi^* \colon  S(G, \im \varphi)^c \to S(G_1)$ is injective;
\item the subgroup $\im \varphi \cdot [G,G]$ has finite index in $G$;
\item $\varphi^*$ is defined on $S(G)$.
\end{enumerate}
\end{crl}                                             

Since the morphism $\varphi^* \colon  S(G, \im \varphi)^c \to S(G_1)$ is induced by a linear map,
it is compatible with the geometric structures of the spheres;
in particular,
it is continuous and sends a pair of antipodal points to a pair with the same property.
A further, useful fact is recorded in
\begin{lem} 
If $[\chi] \in S(G, \im \varphi)^c$ then 
$\rk [\chi] \geq \rk \varphi^*[\chi] \geq 1$.
In particular, the  image of a rational point is rational. 
\end{lem}
%
\subsection{Properties of $\Sigma^1$ under change of groups}
\label{ssec:Properties-Sigma1-change-groups}
\index{Invariant Sigma1 and@Invariant $\Sigma^1$ and!change of groups}
%
In this section, 
we present results that relate the invariant of the given group $G$
to invariants of related groups, such as quotient groups or subgroups of finite index.
%
\subsubsection{Isomorphisms and automorphisms of groups} 
\label{sssec:Isomorphisms-groups}
%
We begin with a remark that is long overdue. 
The definition of $\Sigma^1(G)$ uses a finite generating system $\eta \colon \XX \to G$.
By Theorem \ref{thm:Sigma1-well-defined},
the choice of this system plays no r{\^o}le, 
a first fact that justifies calling $\Sigma^1(G)$ an invariant.
But more is true: 
\emph{the construction $G \mapsto \Sigma^1(G)$ assigns to isomorphic groups isomorphic invariants}.
This fact is a direct consequence of the independence of $\Sigma^1$ on the choice of the generating system.
Indeed, let $\alpha \colon  G \iso G_{1} $ be an isomorphism of finitely generated groups,
let $\eta \colon \XX \to G$ be a generating system of $G$
and use $\alpha$ to push $\eta$ forward to a generating system $\eta_{1} = \alpha\circ \eta$ of $G_{1}$.
Then $\alpha$ induces an obvious graph isomorphism 
$\alpha_{*} \colon \Gamma(G, \XX) \iso \Gamma(G_{1}, \alpha(\XX))$ of Cayley graphs.
Moreover, if $\chi_{1}$ is a non-zero character of $G_{1}$ and $\chi = \chi_{1} \circ \alpha$,
then $\alpha_{*}$ maps the subgraph $\Gamma(G, \XX)_{\chi}$ 
onto the subgraph  $\Gamma(G_{1}, \alpha(\XX))_{\chi_{1}}$.
Consequently,  $\alpha$ induces an isomorphism  $\alpha^* \colon \Sigma^1(G_{1}) \iso \Sigma^1(G)$.

Note that the assignment  $\alpha \mapsto \alpha^*$ is contravariant.
If $\alpha$ is an automorphism of $G$,
we consider therefore the assignment $\alpha \mapsto (\alpha^{-1})^*$;
it yields a homomorphism of $\Aut(G)$,
the group of automorphisms of $G$, 
into the group of homeomorphisms of the sphere $S(G)$ 
under which the subset $\Sigma^1(G)$ is invariant.

For later reference we summarize the preceding discussion in
\begin{prp}
\label{prp:Automorphisms-and-invariant}
For every finitely generated group $G$, the assignment 
\[
\alpha \longmapsto (\alpha^{-1})^* \colon S(G) \iso S(G)
\]
is a group homomorphism of $\Aut(G)$ into the group of homeomorphisms of the sphere $S(G)$.
The invariant $\Sigma^1(G)$ is stable under this action of $\Aut(G)$ on $S(G)$.
\end{prp}
\index{Invariant Sigma1 and@Invariant $\Sigma^1$ and!automorphisms}
%
\paragraph{Application to free soluble groups.} 
\label{para:Application-relatively-free-groups}
The invariant $\Sigma^1$ provides a method for showing 
that finitely generated groups in certain classes of soluble groups are infinitely related.
Prior to turning to this application,
we record an easy consequence of Theorems 
\ref{thm:Consequence-finite-presentation} and
\ref{thm:Characterizing-fg-N}
and the preceding proposition:
\begin{crl}
\label{crl:Fp-groups-without-free-subgroups-and-automorphisms}
Let $G$ be a finitely presented group with no non-abelian free subgroups. 
If $G$ admits an automorphism inducing $- \id$ on $G/\sqrt{G'}$
then $\Sigma^1(G) = S(G)$ and so $G'$  is a finitely generated group.
\end{crl}
\begin{proof}
If $\alpha$ induces $-\id$ on $G/\sqrt{G'}$ 
then  $\Sigma^1(G) = (\alpha^{-1})^*(\Sigma^1(G)) = -\Sigma^1(G)$. 
On the other  hand, 
Theorem \ref{thm:Consequence-finite-presentation}
guarantees that  $S(G) = \Sigma^1(G) \cup -\Sigma^1(G)$. 
So $\Sigma^1(G) = S(G)$, 
whence $G'$ is finitely generated  by Theorem \ref{thm:Characterizing-fg-N}. 
\end{proof}

The corollary can be applied to relatively free groups in a variety ${\VV}$ 
that is distinct from the variety of all  groups
(see \cite[Section I.3]{Neu67b} for terminology and results relevant to the present context).
If $\VV$ is a soluble variety, 
Corollary \ref{crl:Fp-groups-without-free-subgroups-and-automorphisms},  
in conjunction with a result of J. R. J. Groves \cite{Gro71}, 
yields the next result (see \cite[Thm.\;8]{Str84}, cf. \cite[Cor.\;B1]{BiSt78}).
\index{Neumann, H.}
\index{Groves, J. R. J.}
\index{Bieri, R.}
\index{Strebel, R.}
\begin{crl}
A finitely generated relatively free group of a soluble variety $\VV$ admits a finite presentation 
if, and only  if, it is nilpotent-by-finite.
\end{crl}
\index{Free soluble groups}
\begin{proof}
By Theorem A(ii) in \cite{Gro71} every soluble variety 
either contains a met\-abelian subvariety of the form $\AA_p \AA$, 
or it is contained in a product variety of the form ${\NN}_c {\BB}_e$. 
Here $\AA, \AA_p$ are the varieties respectively, of 
abelian groups and elementary abelian $p$-groups, 
while ${\NN}_c$  denotes the variety of all nilpotent groups of class at most $c$ 
and ${\BB}_e$ is the Burnside variety of exponent $e$. 

Suppose $\VV$  contains $\AA_p \AA$ for some prime $p$ 
and $G$ is a non-cyclic, finitely generated, relatively free group of $\VV$ with  basis ${\BB}$. 
Then $G/(G'' \cdot (G')^p)$ is a non-cyclic $\AA_p  \AA$-free group 
and so the derived group of $G$ is infinitely  generated. 
Since the function which assigns to each basis element in ${\BB}$ its inverse, 
extends to an automorphism of $G$ inducing $-\id$ on  $G_{\ab}$, 
the previous corollary forces $G$ to be infinitely related. 
A  relatively free group $G$ in $\VV$ admitting a finite presentation is thus either cyclic, 
or $\VV$ is contained in a product variety of the form ${\NN}_c {\BB}_e$, 
whence $G$ is nilpotent-by-finite. 
The converse is clear.
\end{proof}

\subsubsection{Epimorphisms of groups} 
\label{sssec:Epimorphisms-groups}
%
The morphism of spheres induced by a group homomorphism $\varphi \colon G \to G_{1}$
is, in general, only defined on a proper subset of $S(G_{1})$.
The situation improves if the homomorphism is surjective or if is the inclusion of a subgroup with finite index
(see Corollary \ref{crl:Injectivity-phi-star}).
In both cases, 
a simple formula relates then the complements of the invariants to each other.

The case of an epimorphism has already been treated in Proposition 
\ref{prp:Comparison-invariants-G-and-Q}.
We restate this result as
 \begin{crl}
\label{crl:Sigma1-epimorphism}
\index{Invariant Sigma1 and@Invariant $\Sigma^1$ and!quotient groups}
If $\pi \colon G \epi Q$ be an epimorphism of finitely generated groups,
the induced morphism $\pi^*$  maps the sphere $S(Q)$ 
bijectively onto the subsphere  $S(G, \ker \varphi)$ of $S(G)$ 
and the inclusion
$\pi^*(\Sigma^1(Q) ^c) \subseteq\Sigma^1(G) ^c $ holds. 

If $\ker \pi$ is  finitely generated \emph{qua} group 
this inclusion can be sharpened to the equality 
$\pi^*(\Sigma^1(Q) ^c) = \Sigma^1(G) ^c \cap S(G, \ker \varphi)$.
\end{crl}
\begin{example}
\label{example:Direct-product-groups}
Let $G$ be the direct internal product of two normal subgroups $G_{1}$ and $G_{2}$
and let $\pi_i \colon G \epi G_{i}$ be the canonical projection with kernel $G_{j}$ where $\{i,j\} = \{1,2 \}$.
If the corollary is applied to the projections $\pi_{1}$ and $\pi_{2}$ 
the equalities
\[ 
\Sigma^1(G)^c \cap S(G,G_{2}) = \pi_{1}^*(\Sigma^1(G_{1})^c) \quad \text{and} \quad
\Sigma^1(G)^c \cap S(G,G_{1}) = \pi_{2}^*(\Sigma^1(G_{2}^*)^c)
\]
result. 
It follows  that $\Sigma^1(G)^c$ contains $\pi_{1}^*(\Sigma^1(G_{1})^c) \cup \pi_{2}^*(\Sigma^1(G_{2})^c)$. 
Proposition \ref{prp:Sigma1-direct-product} yields a better result:
it shows that this inclusion is an equality
\end{example}
%
\paragraph{Applications to infinitely related soluble groups.} 
\label{para:Application-infinitely-related-soluble-groups}
Assume $G$ is a finitely generated soluble group and $Q = G/N$ is a factor group of $G$.
It may happen that $G$ admits a finite presentation, but $G/N$ does no longer have this property
(cf.\;Abels' group mentioned in example \ref{example:Abels-matrix-group}).
\index{Abels, H.}
If, however, $Q$ can be shown to be infinitely related 
with the help of Theorem \ref{thm:Consequence-finite-presentation},
then $G$ will be infinitely related, too.
One has:
\begin{crl}
\label{crl:Groups-with-infinitely-related-quotients}
Let $G$ be a finitely generated group with no non-abelian free subgroups. 
If $Q$ is a quotient group of $G$ and if $\Sigma^1(Q)^c$ contains a pair of antipodal points,
then $G$ is infinitely related.
\end{crl}
\begin{proof}
Assume $\pi\colon G \epi Q$ is an epimorphism 
and $\psi \colon Q \to \R$ and $-\psi$ represent points of $\Sigma^1(Q)^c$.
By Corollary \ref{crl:Sigma1-epimorphism} the character $\chi = \psi \circ \pi$ and its negative 
represent points in $\Sigma^1(G)^c$ 
and so Theorem \ref{thm:Consequence-finite-presentation} permits us to conclude
that $G$ does not admit a finite presentation.
\end{proof}
%
\subsubsection{Subgroups of finite index} 
\label{sssec:Subgroup-finite-index}
%
In computing the invariant of a group it is sometimes useful 
to pass to a subgroup of finite index.
The following result provides the link between the invariants of the group and its subgroup:
\begin{prp}
\label{prp:Sigma1-finite-index}
\index{Invariant Sigma1 and@Invariant $\Sigma^1$ and!subgroups of finite index}
Suppose $G$ is finitely generated and
$\iota \colon  H \hookrightarrow G$  is the inclusion of a subgroup of  finite index.
Then the induced morphism $\iota^*$  maps the sphere $S(G)$ 
bijectively onto the subsphere  $S(H, H \cap G')$  
and the following formula holds: 
\begin{equation}
\label{eq:Invariant-Sigma1-passage-subgroup-finite-index}
\iota^*\left(\Sigma^1(G)^c\right) = \Sigma^1(H)^c \cap S(H, H \cap G').
\end{equation}
\end{prp}

\begin{proof}
The first claim follows from Lemma \ref{lem:Image-morphism}
and Corollary \ref{crl:Injectivity-phi-star}.
The second will be proved by going back to the Cayley-graph definition of the invariant 
and exploiting the independence of the invariant on the chosen generating system.

Let $\XX \subset G$ be a finite set generating $G$. 
Given $\chi \colon  G \to \R$,
find a subset  $\TT \subset G$  that contains the unit element 1,
represents the homogeneous space $H\backslash G$ and satisfies $\chi(\TT) \leq 0$. 
For $t \in  \TT$ and $x \in \XX$ 
let $\overline{tx}$ denote the element of  $\TT$ representing $H \cdot tx$. 
Then the family 
$\YY = \{t \cdot x \cdot \overline{tx}^{-1} \mid t \in \TT  \text{ and }  x \in \XX^\pm\}$
generates the subgroup $H$ (see, \eg, \cite[Lemma 7.2.2]{Hal76}).

Assume now that $\Gamma_\chi = \Gamma(G,{\XX})_\chi$ is connected.
For each $h$ in $H \cap  G_\chi$ there exists then a path $p = (1, w)$ in $\Gamma_\chi$ from $1$ to $h$. 
If $w$ has the spelling $x_1x_2 \ldots x_m$ with letters in $\XX^\pm$,
let $u_w$ denote the word
$
(1 \cdot x_1t_1^{-1}) \cdot (t_1x_2t_2^{-1}) \cdots  (t_{m-1}x_mt_m^{-1}) \cdot t_m,
$
where $t_i = \overline{t_{i-1} \cdot x_i}$ for each $i$. 
Since $h \in H$  and as  $1 \in \TT$,
one has $t_m =1$ and so $u$ is actually a  $\YY$-word. 
Moreover, if $y_1y_2 \ldots y_m$ is the ${\YY}$-spelling of $u$ then
\[
\chi(y_1 \ldots y_i) = \chi(x_1 \ldots x_i \cdot t_i^{-1}) \geq \chi(x_1  \ldots x_i) \geq 0
\]
for each $i$. This means that the path $(1, y_1 \ldots y_m)$ runs in the  subgraph $\Gamma(H, \YY)_{\chi|H}$. 
It follows that $\Gamma(G,\YY)_{\chi|H}$ is connected and so $\chi|H$ represents a point of $\Sigma^1(H)$.

Conversely
assume  that $[\chi|H] \in \Sigma^1(H)$.
Let   $\eta \colon \YY \to H$ of $H$ be a finite generating system of $H$
and choose, as before, a subset $\TT$ of $G$ that contains 1,
represents  $H \backslash G$ and satisfies $\chi(\TT) \leq 0$.
Then $\YY \cup \TT $ will generate $G$.

We claim  the graph $\Gamma_{\chi} = \Gamma(G, \YY\; \cup\; \TT)_{\chi}$ is connected.
This graph contains the subgraph $\Gamma(G,\YY)_{\chi|H}$ which is connected by hypothesis.
Moreover, given a vertex $g \in \Gamma_{\chi}$ there exists $h\in H$ and $t \in \TT$   with $g = h \cdot t$. 
Then  $\chi(h) = \chi (g\cdot t^{-1} ) \geq \chi(g) \geq 0$ 
and so there exists a path $p = (1, w(\YY^\pm))$ in the subgraph $\Gamma(G,\YY)_{\chi|H}$ from $1$ to $h$. The path $(1, w(\YY^\pm) \cdot t)$ then leads  inside $\Gamma_{\chi}$ from $1$ to $g$.
\end{proof}
\begin{remark}
\label{remark:Subgroup-finite-index}
An analogue of the finite index theorem is used in \cite{BiSt80} for the invariant $\Sigma^0(Q;A)$
with $Q$ a fg abelian group. 
The aim there is to simplify the construction of a finite presentation of a metabelian group $M$ with quotient $G$
by passing from the abelian quotient $G$ of $M$  to a free abelian quotient $H$ of a subgroup of $M$. 

I wonder whether there are situations 
where Proposition \ref{prp:Sigma1-finite-index} can be applied in the way indicated by
\begin{problem}
\label{problem:Use-finite-index}
Find situations where one is interested in the invariant $\Sigma^1(G)$  a group $G$ 
that admits a subgroup of finite index which is easier to deal with than $G$ itself
and  for which $\Sigma^1$ can be computed.
\end{problem}

The only classes I'm actually aware of are the class of abelian groups
and that of co-compact Fuchsian groups of positive genus.
(The groups of the second class have surface groups of positive genus as subgroups with finite index.)
\end{remark}

%
\subsubsection{Joins of subgroups} 
\label{sssec:Join-subgroups}
%
Let $G$ be a finitely generated group 
which is generated by two finitely generated subgroups  $G_{1}$ and $G_{2}$, say.
Familiar examples of this situation are a) the free product $G_{1} \star G_{2}$
and b) the free product with amalgam $G_{1} \star_{A} G_{2}$.
We aim at finding a useful condition that guarantees 
that a non-zero character $\chi \colon G \to \R$ represents a point of $\Sigma^1(G)$
if its restrictions to $G_{1}$ and to $G_{2}$ are non-zero and represent points in 
$\Sigma^1(G_{1})$ and in $\Sigma^1(G_{2})$, respectively.
Example a) reveals that some extra condition is needed;
indeed the invariant of a free product  $G_{1} \star G_{2}$ with factors distinct from the trivial group is empty
(see Example 3 in section \ref{sssec:Sigma1-first-examples}).

One such extra condition is spelled out in
\begin{lem}
\label{lem:Sigma-1-join-subgroups}
Assume $G$ is  a finitely generated group 
which is generated by two finitely generated subgroups  $G_{1}$ and $G_{2}$.
Then the implication
\begin{equation}
\label{eq:Sigma-1-join-subgroups}
\chi (G_{1} \cap G_{2}) \neq \{0\}, \ [\chi[G_{1}] \in \Sigma^1(G_{1}), \  [\chi[G_{2}] \in \Sigma^1(G_{2})
\Longrightarrow [\chi] \in \Sigma^1(G)
\end{equation}
is valid for  every non-zero character $\chi $ of $G$.
\end{lem}

\begin{proof}
For $i =1, 2$, let   $\eta_{i}\colon \XX_{i} \to G_{i}$ be  a finite generating system of $G_{i}$
and let $\Gamma$ denote the Cayley graph  $\Gamma(G, \XX_{1} \cup \XX_{2})$.
For every $g \in G_\chi$, there exists a sequence
\begin{equation}
\label{eq:Factorisation-g}
(g_{1,1}, g_{2,1},g_{1,2}, g_{2,2},\ldots, g_{1,m},g_{2,m}) \text{ with } g_{1,j} \in G_{1} \text{ and }  g_{2,j} \in G_{2}
\end{equation}
whose product $g_{1,1} \cdot g_{2,1}  \cdots g_{1,m} \cdot g_{2,m}$ is equal to $g$.
We aim at proving by induction on $m$ that there exist a path $p$
which leads inside $G_\chi$ from $1$ to $g$.

The existence of such a path is trivial for $m = 0$; 
as this case tells one little about the ideas involved in the inductive step, 
we consider next the case $m=1$. 
Then $g = g_{1} \cdot g_{2}$ with $g_{i} \in G_{i}$.
Since $\chi$ does not vanish on the intersection $H = G_{1} \cap G_{2}$,
there exists elements $h_{1}$, $h_{2}$ in $H$ so that $g'_{1} = g_{1}\cdot h_{1}$, as well as
$g'_{2} = h^{-1}_{1}\cdot g_{2} h_{2} $  and $h_2$ have non-negative $\chi$-values.
Since $g'_{1} \in G_{1}$ and $\Gamma(G_{1}, \XX_{1})_{\chi|G_{1}}$ is connected,
there exists a path $p_{1} = (1, w_{1})$ in $\Gamma(G_{1}, \XX_{1})_{\chi|G_{1}}$ that leads from $1$ to $g'_{1}$.
Similarly, 
there exists a path $p_{2} = (1, w_{2})$ that connects $1$ to $g'_{2}$ in 
$\Gamma(G_{2}, \XX_{2})_{\chi|G_{2}}$. 
The concatenated path $p_{12} = (1, w_{1}w_{2}) $ 
connects then the vertex $1$ with the vertex  $g'_{1} \cdot g'_{2}$ and it stays in the subgraph $\Gamma_{\chi}$.
As $h_{2}$ is a vertex of $\Gamma(G_{1}, \XX_{1})_{\chi|G_{1}}$, 
there exists, finally, a path $p_{3} = (1, w_{3})$ from $1$ to $h_{2}$ in $\Gamma(G_{1}, \XX_{1})_{\chi|G_{1}}$.
The path $p = (1, w_{1}w_{2}w_{3}^{-1})$ leads then from $1$ to 
$g'_{1} g'_{2}h^{-1}_{2} = g_{1} h_{1} \cdot h^{-1} g_{2}h_{2} \cdot h^{-1}_{2} = g_{1}g_{2} = g$
without leaving $\Gamma_{\chi}$.
This latter claim follows from formulae \eqref{eq:Properties-valuation-word}:
\begin{align*}
v_{\chi} (w_{1}w_{2}w_{3}^{-1}) 
&= 
\min\left\{
v_{\chi}(w_{1}), \chi(g'_{1})+ v_{\chi}(w_{2}),  \chi(g'_{1}g'_{2})+ v_{\chi}(w_{3}^{-1}) 
\right \}\\
&=
\min\left\{ 
0,  \chi(g'_{1})+ 0, \chi(g'_{1}g'_{2}) + \chi(h^{-1}_{2}) + v_{\chi}(w_{3})
\right \} \\
&= 0.
\end{align*}

The inductive step follows the same pattern as the case just treated.
Assume $g$ is the product of the sequence \eqref{eq:Factorisation-g} with $2m$ factors.
Let  $g'$ be the product of the first $2m-2$ factors.
There exists elements $h_{0}$, $h_{1}$, $h_{2}$ in $H_{\chi | H}$ so that 
\[
\chi(g'h_{0}) \geq 0,\quad g'_{1} = h^{-1}_{0}g_{1,m}h_{1} \in (G_{1})_{\chi|G_{1}}, \quad 
g'_{2} = h^{-1}_{1}g_{2,m}h_{2} \in (G_{2})_{\chi|G_{2}}.
\]
By the inductive hypothesis one can find a path $p' = (1,w')$ 
which leads from $1$ to $g'h_0$ and stays in $\Gamma_{\chi}$.
In addition, one sees as before that there exists paths  $p_{1}$, $p_{2}$ and $p_{3} $ 
enjoying the following properties:
\begin{itemize}
\item $p_{1} = (1, w_{1})$ runs from $1$ to $g'_{1}$ in $\Gamma(G_{1}, \XX_{1})_{\chi|G_{1}}$,
\item $p_{2} = (1, w_{2})$ runs  from $1$ to $g'_{2}$ in $\Gamma(G_{2}, \XX_{1})_{\chi|G_{2}}$,
\item $p_{3} = (1, w_{3})$ runs  from $1$ to $h_{2}$ in $\Gamma(G_{1}, \XX_{1})_{\chi|G_{1}}$.
\end{itemize}
The concatenated path $p = (1, w' w_{1}w_{2} w_{3}^{-1})$ 
then leads inside $\Gamma_{\chi}$ from $1$ to $g$.
\end{proof}

A useful generalization of the previous lemma is
\begin{prp}
\label{prp:Sigma-1-join-subgroups}
Assume $G$ is generated by a finite collection  $\{ G_v  \mid v \in V \} $ 
of finitely generated subgroups.
Given a non-zero character $\chi \colon G \to \R$,
let $\GG(\chi)$ denote the combinatorial graph with vertex set $V$ 
and edge set the set of those pairs $\{u,v\}$ for which $\chi$ is non-zero on the intersection $G_u \cap G_v$.
If the conditions
\begin{enumerate}[(i)]
\item for every $v \in V$,
the restriction of $\chi$ to $G_v$ represents a point of $\Sigma^1(G_v)$,
\item the graph $\GG(\chi)$ is connected
\end{enumerate}
are satisfied, the character $\chi$  represents a point of $\Sigma^1(G)$.
\end{prp}
\index{Computation of Sigma1@Computation of $\Sigma^1$ for!joins of subgroups}

\begin{proof}
We argue by induction on $n = \card  V$.
If $n=2$, the claim holds by Lemma \ref{lem:Sigma-1-join-subgroups};
so assume $n > 2$. 
Choose a spanning tree $T$ of $\GG(\chi)$
and then a vertex  $v_0$ of degree 1 in $T$.
Let $\GG_1$ denote the subgraph of $\GG(\chi)$ 
obtained by omitting $v_0$ and all the edges incident with it;
it is connected.
Define $G_1\subseteq G$ to be the subgroup generated by the subgroups $G_v$ with $v \neq v_0$.
Then $\GG(\chi |G_1)$ coincides with $\GG_1$,
is therefore connected,
and so $[\chi|G_1] \in \Sigma^1(G_1)$ by the induction hypothesis.
The claim now follows from Lemma  \ref{lem:Sigma-1-join-subgroups},
applied with $G_1$ as defined before and $G_2 = G_{v_0}$.
\end{proof}
%
%
\subsection[$\Sigma^1$ of graph groups]{Application to graph groups} 
\label{ssec:Application-right-angled-Artin-groups}
%
The notion of a graph group (or right angled Artin group) has been introduced in section 
\ref{sssec:Fg-normal-subgroups-right-angled Artin-group}.
Such a group is given by a finite combinatorial graph 
\[
\Delta = (\XX = \{x_{1}, \ldots,x_{n} \}, \edg (\Delta));
\]
the graph encodes a presentation of the group,
namely
\begin{equation}
\label{eq:Right-angled-Artin-group-bis}
G_{\Delta}
= 
\langle x_{1}, x_{2}, \ldots, x_{n} 
\mid 
x_{j}x_{\ell} = x_{\ell}x_{j} \text{ for every edge } \{x_{j} , x_{\ell}\} \in \edg (\Delta) \rangle.
\end{equation}
In this section we establish a formula,
announced in \ref{sssec:Fg-normal-subgroups-right-angled Artin-group},
for the complement of the invariant $\Sigma^1(G_{\Delta})$
and  illustrate it by some examples.
\index{Graph groups!definition}
%
%
\subsubsection{Determination of the invariant} 
\label{sssec:Right-angled-Artin-groups-computing-Sigma1}
%
The determination will rely on Proposition \ref{prp:Sigma-1-join-subgroups}
and on the following 
\begin{lem}
\label{lem:Structure-right-angled-Artin-groups}
If $\Delta'$ is a full subgraph of the graph $\Delta$,
the following assertions hold:
\begin{enumerate} [(i)]
\item the group $G_{\Delta'}$ is a retract of $G_{\Delta}$;
\item if $\Delta'$ is not connected,  $G_{\Delta'}$ is a free product of two non-trivial subgroups 
and so its invariant is empty.
\end{enumerate}
\end{lem}
\begin{proof}
(i)  For every full subgraph $\Delta'$ of $\Delta$ the assignments
\[
 \pi(x_{j}) =   x_{j}  \text{ if }  x_{j} \in \ver (\Delta'), \text{  and } \pi(x_{j}) =   1  \text{ otherwise},
 \]
extend to a homomorphism $\pi$ of $G_{\Delta}$ onto $G_{\Delta'}$. 
As $\pi$ is the identity on the generators of $G_{\Delta'}$ it is a retract of $G_{\Delta}$.

(ii) Suppose $\Delta'$ be the union of two full, disjoint and non-empty subgraphs $\Delta'_{1}$ and $\Delta'_{2}$.
By (i)  the canonical maps $\iota_{i} \colon G_{\Delta'_{i}} \to G_{\Delta'}$ are embeddings
and  the presentation of $G_{\Delta'}$ shows 
that $G_{\Delta'}$ is the free product of these subgroups.
Since their abelianizations are free-Abelian of rank  $\card (\ver (\Delta'_{i}))$,
none of them is reduced to the identity.
In view of  Example 3 in section \ref{sssec:Sigma1-first-examples})
the invariant of $G_{\Delta'}$ is therefore empty.
\end{proof}

One can ask for two types of depictions of $\Sigma^1(G_{\Delta})$, 
a local and a global one. 
The local description looks for a procedure that allows one to decide, 
given a specific non-zero character $\chi$, whether this character represents a point of the invariant.
The global description seeks a formula for the entire invariant.
The next theorem gives answers to both these questions.

The following terminology will be useful:
a non-zero character $\chi$ determines a non-empty subset 
$\{ x_{j} \in \ver(\Delta) \mid \chi(x_{j}) \neq 0 \}$;
define $\LL(\chi)$ to be the full subgraph on this set and call it the \emph{living subgraph of} $\chi$;
call $\LL(\chi)$ \emph{dominating}
if there exists, 
for every $x_{\ell} \in \Delta \smallsetminus \LL(\chi)$, 
a vertex $x_{j} \in \LL(\chi)$ so that $\{x_{j}, x_{\ell} \}$ is an edge of $\Delta$.
A subset $\SS$ of $\ver (\Delta)$ is called \emph{separating} 
if its removal results in a disconnected subgraph of $\Delta$.
\begin{thm}
\label{thm:Sigma1-right-angled-Artin-group}
\index{Computation of Sigma1@Computation of $\Sigma^1$ for!graph groups}
\index{Graph groups!invariant Sigma1@invariant $\Sigma^1$}
Let $G = G_{\Delta}$ be a right angled Artin group with graph $\Delta$
and let $\chi$ denote a non-zero character of $G$.
If $\Delta$ is a complete graph the group $G$ is free-Abelian of rank $\card\ver (\Delta)$ 
and $\Sigma^1(G) = S(G)$.
Otherwise,
the following assertions are valid:
\begin{enumerate}[(i)]
\item $[\chi] \in \Sigma^1(G)$ if, and only if, 
the living subgraph $\LL(\chi)$ is  connected and dominating;
\item the complement of $\Sigma^1(G)$ is  the union of subspheres  
\begin{equation}
\label{eq:Complement-Invariant-right-angled-Artin-group-bis}
\bigcup\nolimits_{\SS} S(G,\gp(\SS))
\end{equation}
where $\SS$ runs over the minimal separating subsets of $\ver  (\Delta)$.
\end{enumerate}
\end{thm}
\begin{remarks}
\label{remarks:Description-invariant-right-angled-Artin-groups}
a) Claim (i) in Theorem \ref{thm:Sigma1-right-angled-Artin-group}
 is due to  J. Meier and L. VanWyck   (\cite[Theorem 4.1]{MeVa95}) 
 and, independently, to H. Meinert  (\cite[Theorem 1]{Mei95a}),
while assertion (ii)  seems to have been observed first by
S. Papadima and A. Sugiu (see Theorem 5.5 and Proposition 5.8 in \cite{PaSu06}).
\index{Meier, J.}
\index{VanWyck, L.}
\index{Meinert, H.}
\index{Papadima, S.}
\index{Suciu, A.}

b)
There are at least two reasons for isolating, 
in the statement of Theorem
\ref{thm:Sigma1-right-angled-Artin-group}, 
the complete graphs from the other graphs. 
From the point of view of the theory of groups,
a right-angled Artin group contains free subgroups of rank 2,  and so is highly non-commutative,
unless its graph is complete.
Indeed, 
the subgroup $H =\gp(\{x_{j},x_{\ell}\})$ of a right-angled Artin group $G_{\Delta}$ generated by two non-adjacent vertices  is free of rank 2.
Note that $H$ is not only a subgroup, but a retract of $G$ 
(see  Lemma \ref{lem:Structure-right-angled-Artin-groups}).

A second reason lies in the connectivity properties of complete graphs.
In order to explain,
I recall a few definitions (see \cite[Section 1.4]{Die05}).
A (combinatorial) graph $\Delta$ is called \emph{connected} 
if it is \emph{non-empty} and if any two of its vertices can be linked by a path;
it is called $k$-\emph{connected} if it has at least $k+1$ vertices 
and if the removal of any subset $\SS$ with at most $k-1$ elements  does not separate  $\Delta$;
in particular,
a  graph $\Delta$ is $1$-connected if, and only if, it is connected and has at least 2 vertices.
The maximum of the integers $k \geq 1$ for which $\Delta$ is $k$-connected, 
is called the \emph{connectivity} of $\Delta$ and written $\kappa(\Delta)$.
\index{Graph!connectivity}

According to these definitions, 
the connectivity of the complete graph $K_{n}$ is $n-1$ for $n \geq 2$.
At first sight, this may look strange;
one of its benefits is that  Menger's  theorem  (see, e.\;g.,  \cite[Theorem 3.3.6]{Die05})
can be stated very simply in this terminology: 
\emph{a graph $\Delta$ with at least 2 vertices is $k$-connected if, and only, 
any two of its vertices can be linked by $k$ independents paths}.
\footnote{Two paths from $v$ to $v'$ are called independent  if they have no  interior vertices in common.}
\end{remarks}

\begin{proof}
(i) Assume first the living subgraph $\LL(\chi)$ is connected  and that it dominates $\Delta$.
Let $V$ denote the set of \emph{edges} of the graph $\Delta$
and, for every edge $e = \{x_i, x_j\} \in V$ set $G_e = \gp(x_i,x_j)$.
Then each $G_e$ is free abelian of rank 2 and so $\Sigma^1(G_e) = S(G_e)$ 
(cf. Example \ref{examples:Groups-with-non-trivial-centre}a).
But if so,
the restriction of $\chi$ to $G_e$  represents a point of $G_e$ 
if $e$ is an edge of the living graph $\LL(\chi)$
or if $e$ is an edge connecting a vertex outside $\LL(\chi)$ to a vertex in $\LL(\chi)$.
The assumption that $\LL(\chi)$ be connected and dominating  thus implies 
that the auxiliary graph $\GG(\chi)$ 
(defined in the statement of Proposition \ref{prp:Sigma-1-join-subgroups}) is connected,
whence $[\chi] \in \Sigma^1(G)$ by that proposition.

Assume next that  $\Delta' = \LL(\chi)$ is \emph{not connected}.
Then $Q = G_{\Delta'}$ is a non-trivial free product 
(by assertion (ii) of Lemma \ref{lem:Structure-right-angled-Artin-groups}) 
and so $[\chi'] = [\chi|Q]$ lies outside $\Sigma^1(Q)$,
whence Corollary \ref{crl:Sigma1-epimorphism} allows one to conclude 
that $[\chi] \in \Sigma^1(G_{\Delta})^c$.
If, finally, $\Delta' = \LL(\chi)$ is \emph{not dominating},
there exists a vertex $v \in \Delta \smallsetminus \LL(\chi)$ 
so that the full subgraph $\Delta_{v}$ on   $\{v\} \cup \ver (\LL(\chi))$  is not connected.
The group $Q_{v} = G_{\Delta_{v}}$ is then a non-trivial free product and a quotient of $G_{\Delta}$,
whence it follows as before that  $[\chi] \in \Sigma^1(G_{\Delta})^c$.
\smallskip

(ii) Let $\chi$ be a non-zero character.
Assume first that  $[\chi] \in S(G, \gp(\SS))$ for some separating subset  $\SS$ of $ \ver (\Delta)$. 
Then $[\chi]$ is the pullback of a character  $[\chi']$ along the obvious projection 
$\pi \colon G \epi  G_{\Delta \smallsetminus \SS}$ 
and the subgraph $\Delta \smallsetminus \SS$ is not connected.
Lemma \ref{lem:Structure-right-angled-Artin-groups}
and Corollary \ref{crl:Sigma1-epimorphism} then show that $[\chi]$ is in $\Sigma^1(G)^c$.

Suppose now that $[\chi]$ lies outsideof  every subsphere $S(G, \gp(\SS))$ 
defined by a separating subset $\SS$;
put differently, 
suppose that $\chi(\SS)  \neq \{0\}$ for every separating subset $\SS$ of $\Delta$.
Consider the subset $\SS' = \ver (\Delta) \smallsetminus \LL(\chi)$.
Then $\chi(\SS') = \{0 \}$ by the definition of $\LL(\chi)$,
so $\SS'$ does not separate $\Delta$, whence $\LL(\chi)$ is connected.
One sees similarly that every set containing $\ver (\Delta) \smallsetminus \LL(\chi)$ is connected.
The living graph $\LL(\chi)$ is thus connected and dominating,
and so  $[\chi] \in \Sigma^1(G)$ by part (i).

The previous reasoning proves 
that $\Sigma^1(G)^c$ coincides with the union of all subspheres $S(G, \gp(\SS))$
where $\SS$ ranges over the separating subsets.
But as the subsphere $S(G, \gp(\SS))$ contains the subsphere $S(G, \gp(\SS'))$ 
if $\SS$ is a subset of $\SS'$,
this union is not altered if $\SS$ ranges only over the \emph{minimal} separating subsets.
\end{proof}
%
\subsubsection{Finding the list of minimal separating subsets} 
\label{sssec:Sigma1-right-angled-Artin-groups-list-minimal-separating-subsets}
%
Let $\Delta$ denote a \emph{connected}, finite combinatorial graph that is not complete
and let $G = G_{\Delta}$ be the associated right-angled Artin group.
According to Theorem \ref{thm:Sigma1-right-angled-Artin-group}
the complement $\Sigma^1(G)^c$ of the invariant is a finite union of great subspheres $S(G, \gp(\SS))$
where $\SS$ ranges over the \emph{minimal, separating} subsets of the vertex set $\ver(\Delta)$. 
This description is very concise and handy;
the list $L_{\Delta}$ of minimal separating subset $\SS$ can, however,  be rather large
and,  worse, it may be difficult to determine it explicitly.
The following example illustrates the first difficulty.
 \begin{example}
 \label{example:Graph-with-many-minimal-separating-subsets}
 \index{Computation of Sigma1@Computation of $\Sigma^1$ for!graph groups|(}
For every positive integer $m \geq 2$,
 let $V_{m}$ denote the set 
 \[
 \{v_{0}, v_{1,1}, \ldots, v_{1,m}, v_{2,1}\ldots, x_{2,m}, v_{3} \}
 \]
 with $2m + 2$ vertices. 
 Define $E_{m}$ to be set consisting of the $3m$ edges
 \[
 \{v_{0}, v_{1,j} \}, \quad \{v_{1,j}, v_{2,j} \},\quad \{ v_{2,j}, v_{3} \}  \text{ with } j \in \{1,2, \ldots, m \}
 \]
 and set $\Delta_{m} = (V_{m}, E_{m})$.
 One can be visualize $\Delta_{m}$ as a bundle of $m$ rods of length 3
 which are  tied together at both ends.
  The graph $\Delta_{m}$ has $2^m + 2m +1$ minimal separating subsets.
  Indeed, 
 every subset $\SS$ 
 that contains exactly one inner point $v_{1,j}$ or $v_{2,j}$ of each rod, 
 but none of the end points, is separating and minimal; there are $2^m$ such subsets.
 In addition, 
 $\{v_{0}, v_{3}\}$ is a minimal separating subset,
 as are the subsets $\{v_{0}, v_{2,j}\}$ and $\{v_{1,j}, v_{3}\}$ for every index $j$.
 Consider now a non-empty subset $\SS'$ 
 that is distinct from the  $2^m + 2m +1$ minimal separating subsets found so far.
 If if contains none of the end points $v_{0}$, $v_{3}$ and omits no inner points of some rod, 
 it does not separate;
 if it contains none of the end points, but both inner points of some rod, it is not minimal.
 Similarly, one finds that $\SS'$ is non-separating or not minimal,
 if it contains one or two end points.
 
  These calculations show that
 \[
 \card(L_{\Delta_{m}}) = 2^m + 2m + 1 > \left(\sqrt{2}\right)^{\;2m} 
 = \tfrac{1}{2} \left(\sqrt{2}\right)^{\card(\ver(\Delta_{m}))}.
 \]
  \end{example}
  \index{Separating subsets of a graph!examples}

We now turn to the second problem,
that of finding the list of all minimal separating subsets.
The next lemma can help one in solving it.
\begin{lem}
\label{lem:Characterization-minimal-separating}
For every subset $\SS$ of the vertex set $\ver(\Delta)$ 
of a connected, finite combinatorial  graph $\Delta$
the following statements are equivalent :
\begin{enumerate}[(i)]
\item $\SS$ is a minimal separating subset,
\item let $\CC_{1}$, \ldots, $\CC_{\ell}$ be the connected components of the subgraph 
induced by the complement $\ver(\Delta) \smallsetminus \SS$ of $\SS$.
Then $\ell \geq 2$ and every vertex $v \in \SS$ is linked to each component $\CC_{j}$ by an edge.
\end{enumerate}
\end{lem}
\begin{proof}
Let $\SS \subset \ver(\Delta)$ be a non-empty, proper  subset and let $\CC_{1}$, \ldots, $\CC_{\ell}$ 
denote the connected components of the subgraph induced  by $\ver(\Delta) \smallsetminus \SS$.
Then $\SS$ is separating if, and only if, $\ell \geq 2$.
Assume now $\ell \geq 2$
and fix a vertex $v \in \SS$.
If $\SS$ is minimal,
the complement of $\SS \smallsetminus \{v\}$ must be connected
and so  there exists, for every component $\CC_{j}$, a vertex $v_{j}$  that is linked to $v$ by an edge.
Conversely, if every $v \in \SS$ has the stated property, $\SS$ is  minimal and so statement (i) holds.
  \end{proof}

The previous lemma allows one to replace the search for minimal separating subsets of a graph $\Delta$
by a search for full, connected subgraphs.
This simplifies the task in two ways:
a minimal separating subset $\SS$ is an unstructured object 
whose cardinality can \emph{a priori} be close to that of the vertex set $\ver(\Delta)$
whereas one of the connected components of $\Delta \smallsetminus \SS$ has a most 
$(\card(\ver(\Delta)) /2)$ vertices.
So one can proceed like this: 
one constructs, for each $m \geq1$, 
the full, connected subgraphs of $\Delta$ with $m$ vertices,
determines their boundaries and then checks
whether these boundaries fulfill  statement (ii) of Lemma \ref{lem:Characterization-minimal-separating}.
 \index{Separating subsets of a graph!determination}
%
\begin{example}
 \label{example:Minimal-separating-subsets-dodecahedral-graph}
\index{Separating subsets of a graph!examples|(}
 Let $\Delta$ be the 1-skeleton of a regular dodecahedron;
 it has 20 vertices and 30 edges.
 Our aim is to find the \emph{list of all connected subgraphs} $\CC$ with the following properties:
 \begin{enumerate}[(i)]
 \item the boundary $\delta(\CC)$ of $\CC$ is a minimal separating subset,
 \item none of the connected components $\CC'$ of $\Delta \smallsetminus \delta(\CC)$ 
 has fewer vertices than $\CC$.
 \end{enumerate}

It will turn out that $\Delta \smallsetminus \delta(\CC)$ has always two components.
In Figures  \ref{fig:Minimal-separating-sets-dodecahedral-graph-I}
through \ref{fig:Minimal-separating-sets-dodecahedral-graph-V},
the chosen connected component $\CC$ is displayed in red, 
the derived component in blue.
\smallskip

 (i) The connected subgraphs with $m\leq 3$ vertices are singletons, edges or lines of length 2.
They give rise to 20 minimal separating subsets with 3 vertices, 30 such subsets with 4 vertices 
 and 60 subsets with 5 vertices (see Figure 
 \ref{fig:Minimal-separating-sets-dodecahedral-graph-I}).
 \begin{figure}[htb!]
\begin{center}
\includegraphics[width = 3.2cm]{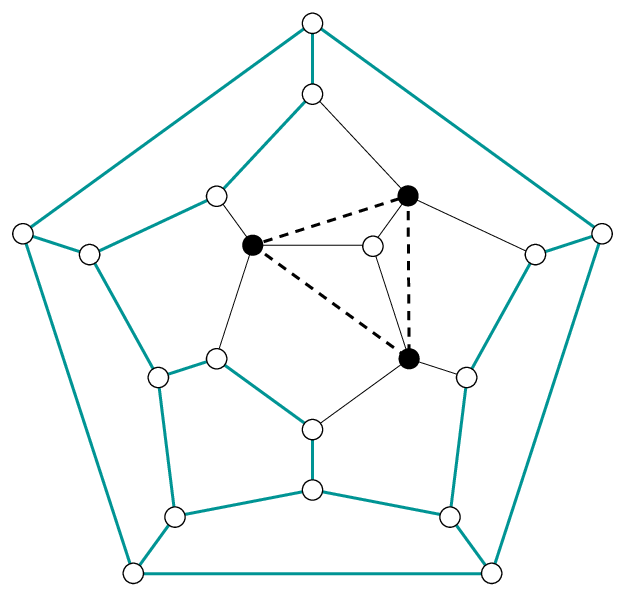}
\hspace{0.6cm}
\includegraphics[width = 3.2cm]{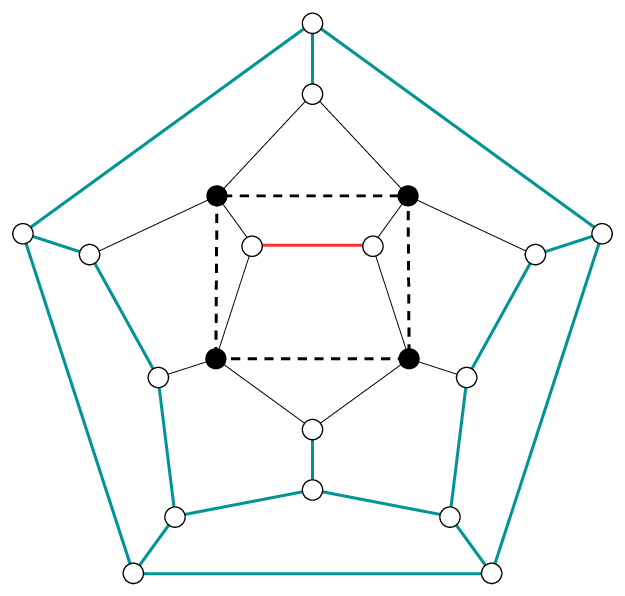}
\hspace{0.6cm}
\includegraphics[width = 3.2cm]{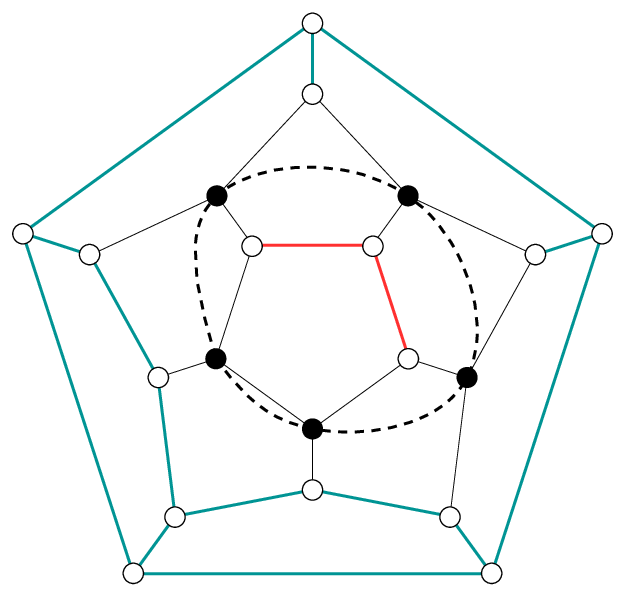}
\end{center}
\caption{Minimal separating sets of a dodecahedral graph for $m = 1$, 2 and 3}
\label{fig:Minimal-separating-sets-dodecahedral-graph-I}
\end{figure}

 (ii) If $m = 4$, the component $\CC$ is a  line or a tripod;
in the first case,
the boundary has either 5 or 6 vertices,  in the second it has 6 vertices.
The number of boundaries is $60 + 60 + 20 = 140$ separating subsets
(see  Figure  \ref{fig:Minimal-separating-sets-dodecahedral-graph-II}).
\begin{figure}[htb!]
\begin{center}
\includegraphics[width = 3.2cm]{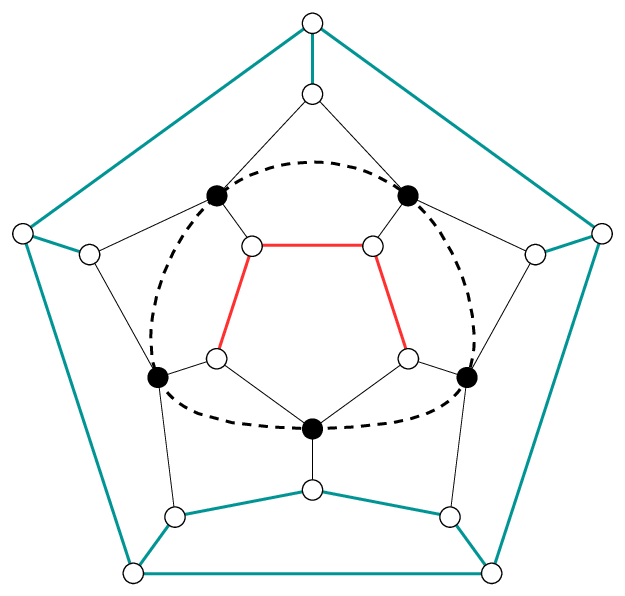}
\hspace{0.6cm}
\includegraphics[width = 3.2cm]{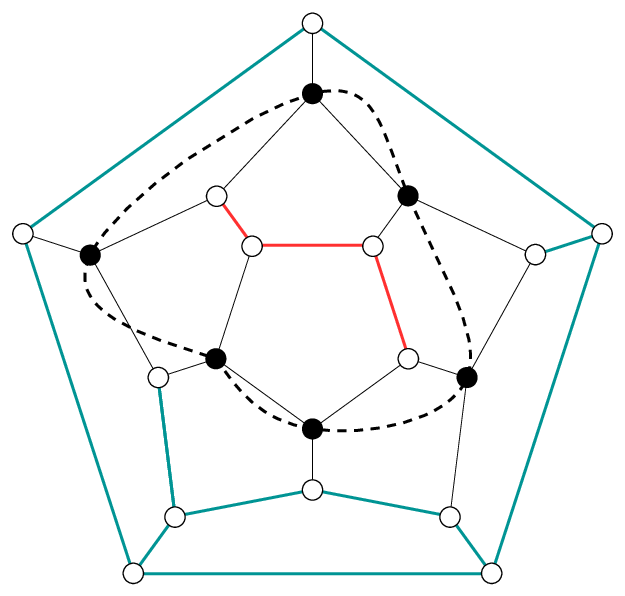}
\hspace{0.6cm}
\includegraphics[width = 3.2cm]{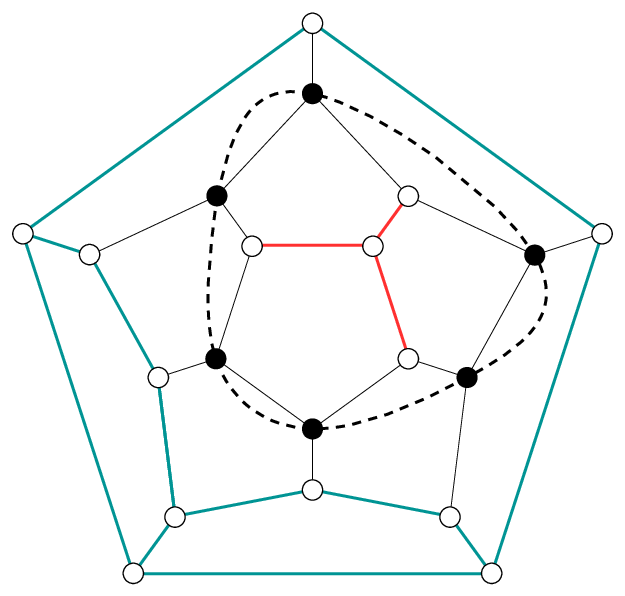}
\end{center}
\caption{Minimal separating sets of a dodecahedral graph for $m = 4$}
\label{fig:Minimal-separating-sets-dodecahedral-graph-II}
\end{figure}

(iii) If $m=5$, the component $\CC$ is a pentagon, a line not contained in a single face, or a tripod.
The first case leads to 12 minimal separating subsets, each with 5 vertices.
The second case gives rise to $ 12 \cdot 5 \cdot 3 = 180$ separating subsets with  6 vertices.
Finally, there are $20 \cdot 6 = 120$ tripods; their boundaries have 6 vertices.
See Figure  \ref{fig:Minimal-separating-sets-dodecahedral-graph-III}.
\smallskip

\begin{figure}[ht!]
\begin{center}
\includegraphics[width = 3.2cm]{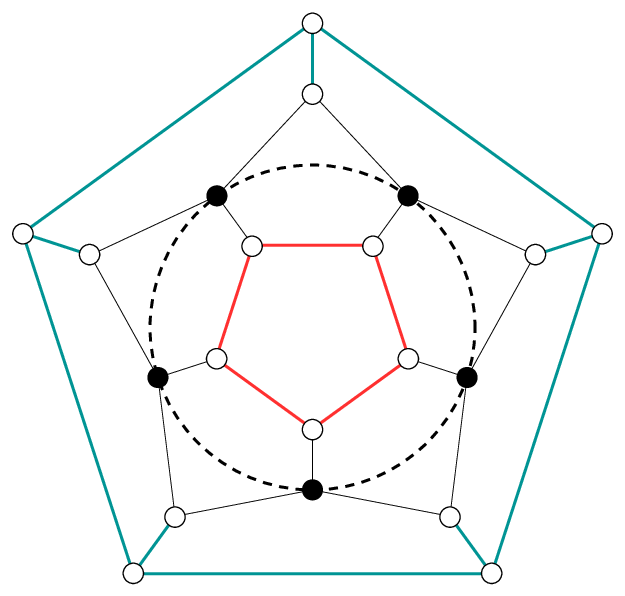}
\hspace{0.6cm}
\includegraphics[width = 3.2cm]{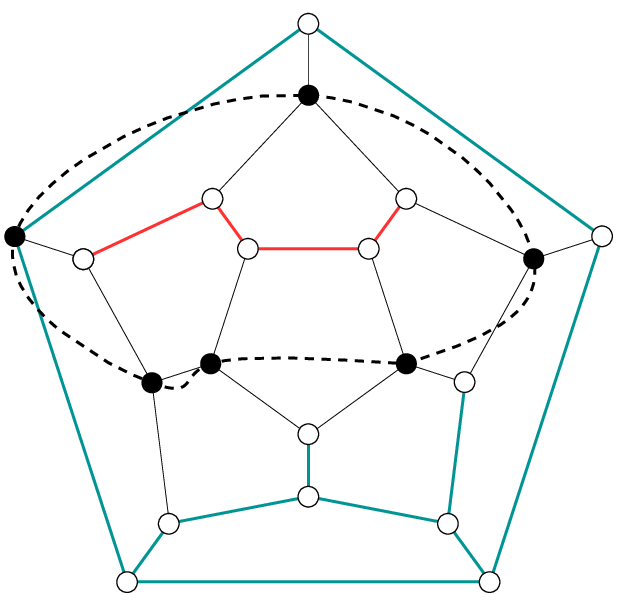}
\hspace{0.6cm}
\includegraphics[width = 3.2cm]{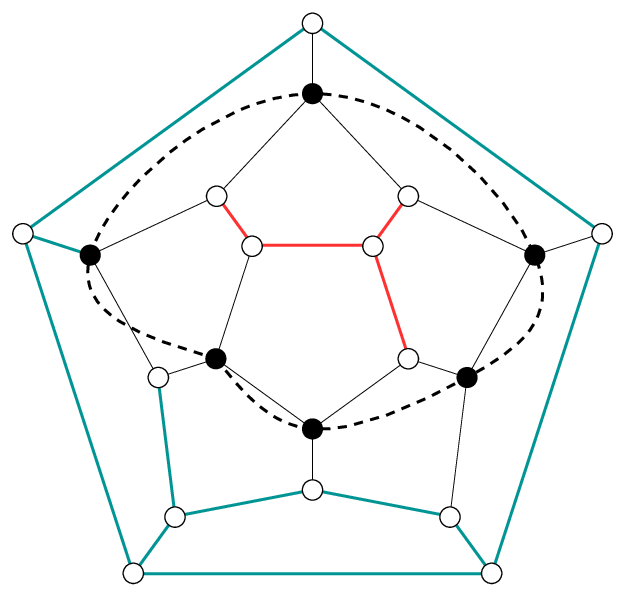}
\end{center}
\caption{Minimal separating sets  for $m = 5$}
\label{fig:Minimal-separating-sets-dodecahedral-graph-III}
\end{figure}

(iv) For $m=6$, 
the components can again take on three shapes: 
a pentagon with a spike, 
a capital letter H and a tripod.
The first leads to $12 \cdot 5 = 60$ separating subsets,
the second to 30 such subsets,
the third to $20 \cdot 3 \cdot  3 = 180$ separating subsets.
All separating subsets have 6 vertices; 
see Figure \ref{fig:Minimal-separating-sets-dodecahedral-graph-IV}.
\smallskip

\begin{figure}[htb!]
\begin{center}
\includegraphics[width = 3.1cm]{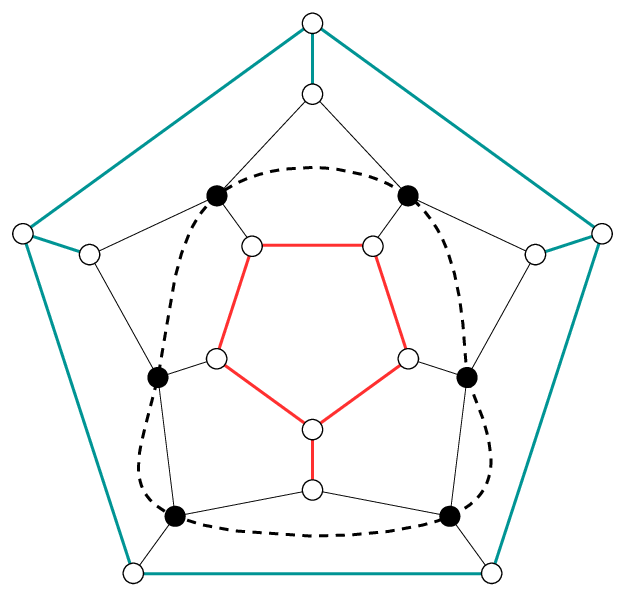}
\hspace{0.6cm}
\includegraphics[width = 3.1cm]{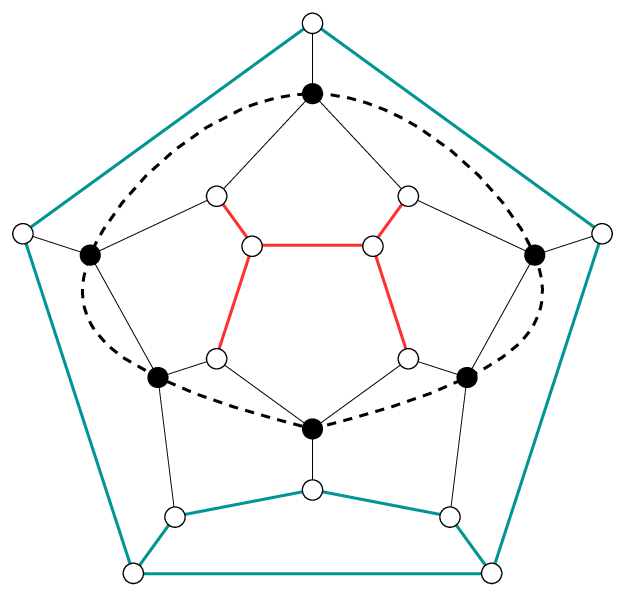}
\hspace{0.6cm}
\includegraphics[width = 3.1cm]{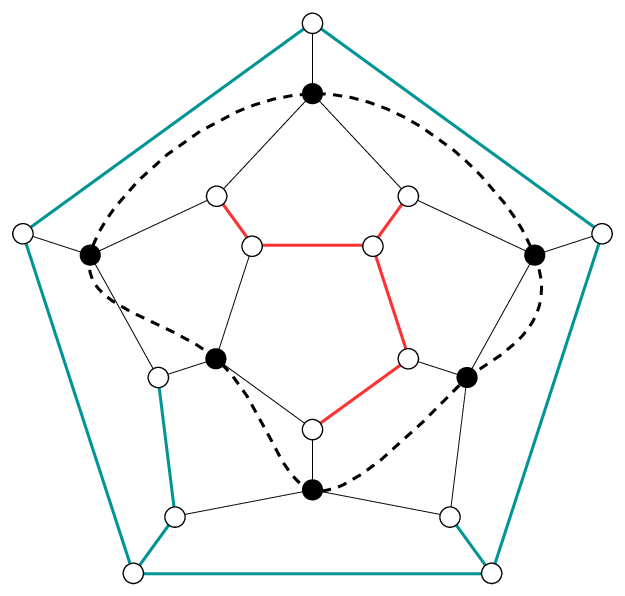}
\end{center}
\caption{Minimal separating sets  for $m = 6$}
\label{fig:Minimal-separating-sets-dodecahedral-graph-IV}
\end{figure}

(v) The case $m=7$ is more involved.
Detailed analysis shows that the components can be of six forms;
they give rise to four types of separating subsets. 
In Figure \ref{fig:Minimal-separating-sets-dodecahedral-graph-V},
the first components are
a pentagon with a long handle, 
a pentagon with two adjacent spikes,
a tripod with a long handle and a symmetric tripod.
The second component of the pentagon with a long handle is different from the first component,
as is the the second component of the tripod with a long handle,
while both components are isomorphic in the other two cases.
The separating subsets have always 6 vertices.
There numbers are: 
$12 \cdot 10 = 120$ for the pentagon with the long handle,
$\tfrac{1}{2} 12 \cdot 5 = 30$ for the pentagon with two spikes,
$20 \cdot 3 \cdot 2 \cdot 2 = 240$  for the tripod with a long handle
and
$\tfrac{1}{2} 20 \cdot 2  = 20$ for the symmetric tripod.
%
\begin{figure}[htb!]
\begin{center}
\includegraphics[width = 2.8cm]{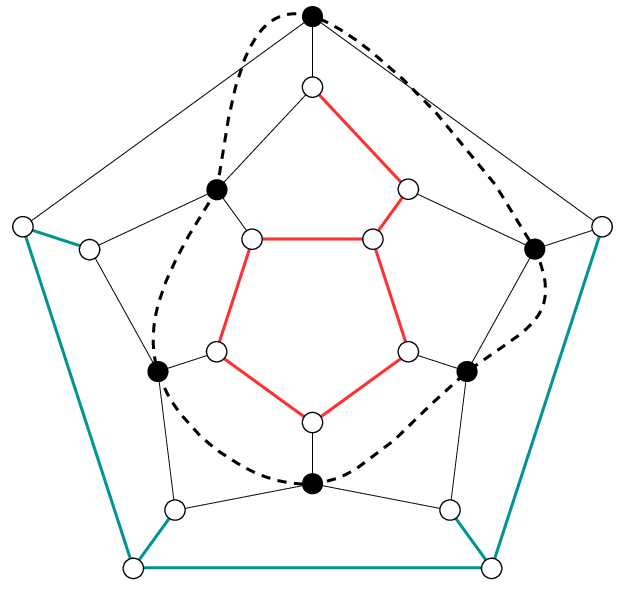}
\hspace{0.0cm}
\includegraphics[width = 2.8cm]{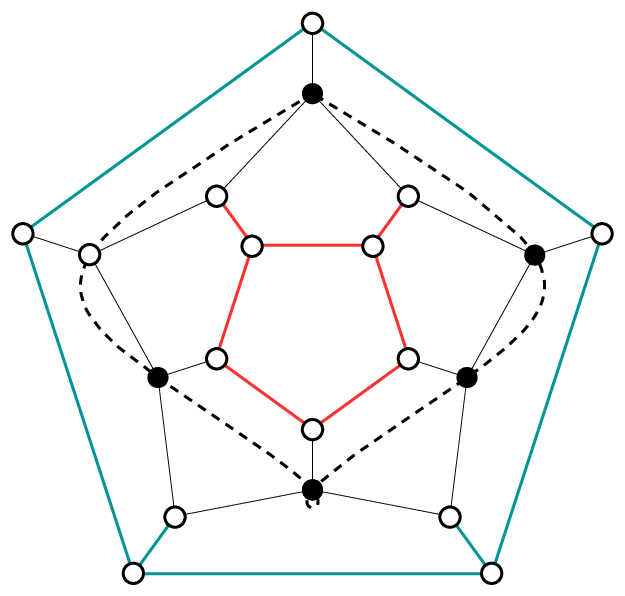}
\hspace{0.0cm}
\includegraphics[width = 2.8cm]{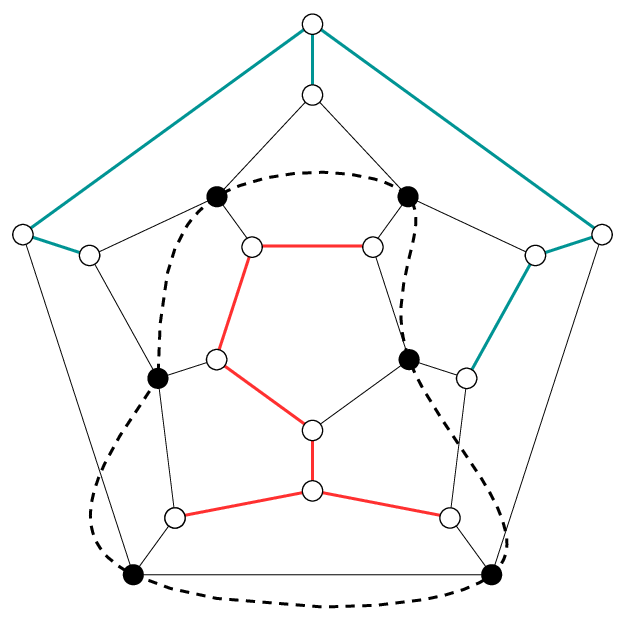}
\hspace{0.0cm}
\includegraphics[width = 2.8cm]{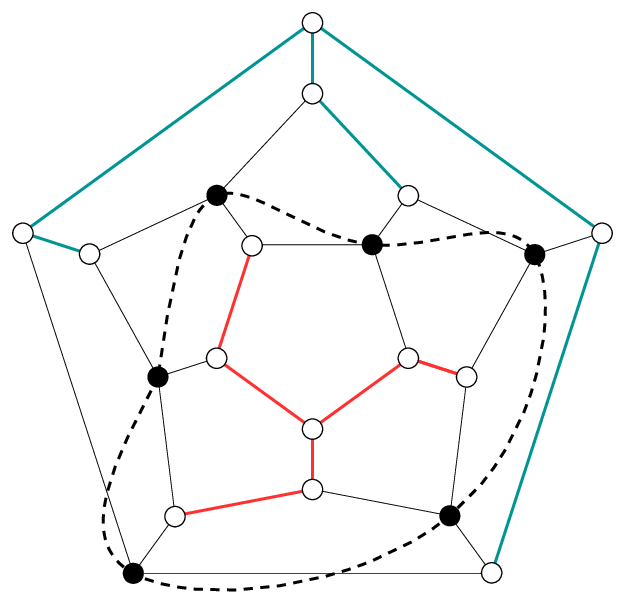}
\end{center}
\caption{Minimal separating sets for $m = 7$}
\label{fig:Minimal-separating-sets-dodecahedral-graph-V}
\end{figure}

All taken together, 
there are $20 + 30 + 60 + 140 + 312 + 270 + 410 = 1242$ separating subsets,
a quite impressive number.
 \end{example}
 \index{Computation of Sigma1@Computation of $\Sigma^1$ for!graph groups|)}
 \index{Separating subsets of a graph!examples|)}

 \vfill
 
\newpage
 
%
%
\section{The $\Sigma^1$-criterion revisited}
\label{sec:Sigma1-criterion-revisited}
%
%
%
The techniques for computing $\Sigma^1$  discussed so far can be summarized thus.
The invariant $\Sigma^1(G)$ is defined in terms of the Cayley graph $\Gamma(G,\XX)$ 
of a group $G$ with respect to a finite system of generators $\XX$.
The computation is therefore straightforward
if the Cayley graph has a simple geometric form 
and/or if it is made of well understood Cayley graphs with the help of suitable constructions.
Cayley graphs with simple geometric form are provided by free abelian groups,
and by non-abelian free groups or, more generally, free products.
In the first case, $\Sigma^1(G) = S(G)$, in the second case the invariant is empty.
Constructions that lead to Cayley graphs with properties 
that can be inferred from the Cayley graphs of the building blocks
include the direct product, the passage to a subgroup of finite index,
and, to a lesser degree, the join $G = \gp(G_{1} \cup G_{2})$ of two subgroups $G_{1}$ and $G_{2}$
(see sections  
\ref{sssec:Invariant-direct-product},  \ref{sssec:Subgroup-finite-index} and \ref{sssec:Join-subgroups}).
In calculating the invariant of groups obtained by these constructions,
one fact was used repeatedly, 
the \emph{independence} of the invariant on the generating system $\XX$
(see Theorem \ref{thm:Sigma1-well-defined}).

Most of the other results obtained in Chapter \ref{ch:Sigma-1-Cayley-graph}  
rely on Theorem \ref{thm:Sigma1-criterion}.
It states that
for each choice $t \in \YY = \XX \cup \XX^{-1}$ with $\chi(t) > 0$,
the subgraph $\Gamma(G, \XX)_{\chi}$ is connected
if, and only if,  the following condition holds:
\begin{quote}
$\Sigma^1$-\emph{criterion}:\quad
\emph{for each $y \in \YY  \smallsetminus \{t, t^{-1} \}$ 
there exists a path $p_y$  in $\Gamma(G,\XX)$
that leads from $t$ to $y\cdot t$ and satisfies the inequality}
\begin{equation}
\label{eq:Positivity-property-path-py}
v_\chi(p_y) > v_\chi((1,y)) = \min \{0, \chi(y)\}.
\end{equation}
 \end{quote}
This condition, being necessary and sufficient,
shows
 that if a point  $[\chi]$  lies in $\Sigma^1(G)$ 
 then an open neighbourhood of $[\chi]$ will be contained in $\Sigma^1(G)$.
 This fact is a key ingredient in the proofs of the \emph{openness} of $\Sigma^1$ 
 (Theorem \ref{thm:Openness-Sigma-1})
 and of the characterization of finitely generated normal subgroups above $G'$
 (Theorem \ref{thm:Characterizing-fg-N}).

The present section focuses on another aspect of the $\Sigma^1$-criterion:
it links a geometric condition to the existence of a set of relations with specific properties.
Now one may happen to know certain relations among the generators,
for instance because the group $G$ is given by a presentation.
If so,
one can try to select from these relations a subset 
which fulfills the $\Sigma^1$-criterion for some point in $S(G)$.
Each subset for which this strategy is successful yields then a non-empty open subset of $\Sigma^1(G)$, 
and the union of these sets is a lower bound of the invariant.
This simple minded approach yields rarely all of  $\Sigma^1(G)$,
but id does so in the case of groups defined by a single relator
(see Section \ref{sec:Invariant-one-relator-group}).
\smallskip

The contents of Section \ref{sec:Sigma1-criterion-revisited} are as follows:
in \ref{ssec:Algebraic-Sigma1-criterion}, 
the $\Sigma^1$-criterion is restated in algebraic terms.
As a first application,
a lower bound $\psi(\RR)$ of $\Sigma^1(G)$ is constructed in \ref{ssec:Subset-psi(R)};
here $\RR$ is a set of relators of $G$.
In section \ref{ssec:Groups-2-generators}
this lower bound is then worked out for some groups with two generators.
In the final part \ref{ssec:Whitehead-criterion-and-applications},
an improvement of the lower bound $\psi(\RR)$ is discussed.
%
\subsection{Algebraic version of the $\Sigma^1$-criterion}
\label{ssec:Algebraic-Sigma1-criterion}
%
Let $G$ be a finitely generated group,
$\eta \colon \XX \to  G$  a finite system of generators,
$\chi \colon  G \to \R$ a non-zero  character and $t \in \YY = \XX \cup \XX^{-1}$ an element with $\chi(t) > 0$. 
The  $\Sigma^1$-criterion requires 
that for  each $y \in \YY \smallsetminus \{t, t^{-1} \}$ there be a a path $p_y$ from $t$ to $y\cdot t$
which  satisfies the inequality $v_\chi(p_y) > v_\chi((1,y))$.
This geometric condition will now be restated in algebraic terms. 
%
\subsubsection{The criterion}
\label{sssec:Algebraic-version-Sigma1-criterion}
%
Let $y$ be an element of  $ \YY \smallsetminus \{t, t^{-1} \}$
and let  $w_{y}$ be the word in  $\YY$ that describes the path $p_{y}$;
so $p_{y} = (t,w_{y})$. 
Since the path $(t,w_{y})$ leads from $t$ to $y\cdot t$, 
the relations  $t\cdot w_{y} = y\cdot t$ and hence $t^{-1} y t = w_{y}$ hold in $G$.
Inequality  \eqref{eq:Positivity-property-path-py}  can then be restated 
by saying that $v_{\chi}(w_{y}) > \min\{0, \chi(y) \} - \chi(t) = v_\chi(t^{-1} y t)$.
 
Consider now the relator  $r_y = t^{-1} y t \cdot w^{-1}_y$. 
One can assume, without loss of generality,
that $w_y$ is a freely reduced word. 
This does not force $r_y$ to be  freely reduced, let alone to be cyclically reduced, 
but the assumption that $v_\chi(t^{-1} y t)$ be smaller than $v_\chi(w_y)$ 
allows one to find out what can happen. 
To cases will be of interest in the sequel.
If $\chi(y) = 0$ then $v_\chi(t^{-1} y t) =  \chi(t^{-1})$ 
and so $w_y$ does neither start with $t^{-1}$ nor does it end with $t$,  
whence $r_y$ is cyclically reduced. 
If $\chi(y) > 0$ then  $v_\chi(t^{-1} y t) =  \chi(t^{-1})$ and so $w_y$ does neither start with $t^{-1}$ 
nor does it end in $y t$; it may, however,  end  in $t$. 
If $w_{y}$ ends in $t$,
say $w_y = w'_y \cdot t$, 
then $r_y$ is freely equivalent to the  word $t^{-1} y(w'_y)^{-1}$ and this word is cyclically reduced. 

The upshot of the preceding analysis is this: 
if $\chi(y) \geq 0$ and if there exists a word $w_y$ with $v_\chi(w_y)  > v_\chi(t^{-1}yt)$ 
then there exists also a cyclically reduced relator $r'_y = s_1s_2 \cdots s_k$, 
starting  with $t^{-1}y t$ if $\chi(y) = 0$ and with $t^{-1}y$ if $\chi(y)>0$,
and so that the sequence
\[
\left(\chi(s_1),\, \chi(s_1s_2), \ldots,\, \chi(s_1s_2 \cdots  s_{k-1}), \chi(s_1s_2 \cdots  s_{k-1}s_k)\right)
\]
assumes its minimum only at $\chi(s_1)$ and $\chi(s_1s_2)$ in the first case,
and only at $\chi(s_1)$ in the second case. 
We finally pass to the cyclic permutation of the relator $r'_y$ which moves the first letter is moved to the end,
and obtain the relator $r_y$.

The new relators $r_y$ show
that the conditions described in equation 
\eqref{eq:Form-relator-algebraic-Sigma1-criterion} of the next theorem are necessary.
\begin{thm}
\label{thm:Algebraic-Sigma1-criterion}
\index{Invariant Sigma1@Invariant $\Sigma^1$!algebraic criterion} 
\index{Sigma-criterion@$\Sigma^1$-criterion!algebraic form}
Let $\eta \colon \XX\to G$ be a finite  system of generators,
$\chi \colon  G \to \R$ a non-zero character and 
$t \in  \YY = \XX \cup \XX^{-1}$ an element with positive $\chi$-value. 
Set  
\[\YY_{>0} = \{y \in \YY \mid \chi(y) > 0\} \text{ and  } \XX_{0} = \{x \in \XX \mid \chi(x) = 0 \}.
\]
Then the $\Sigma^1$-criterion holds for $\YY$,  $\chi$ and $t$  if, and only if, 
there exists, for each $y \in (\YY_{>0}  \cup \XX_{0}) \smallsetminus \{t\} $ 
a reduced word $r_y = s_1s_2 \cdots  s_k$ which is a relator of $G$ and has the form
\begin{equation}
r_y =
\begin{cases}
\label{eq:Form-relator-algebraic-Sigma1-criterion}
y \cdots t^{-1}\hspace*{1.4em} \text{ and } \chi(s_1 \cdots s_i) > 0  \text{ for }  1 < i < k 
& 
\text{ if } \; \chi(y) > 0,\\
yt\cdots t^{-1} \hspace*{1.1em}\text{ and } \chi(s_1 \cdots s_i) > 0 \text{  for }  1 < i < k 
& 
\text{  if } \; \chi(y) =  0.
\end{cases}
\end{equation}
\end{thm}

\begin{proof}
The transformations leading to the conditions 
stated in Theorem \ref{thm:Algebraic-Sigma1-criterion}
can be reversed. 
It follows that there exists, 
for every $x \in \XX$, 
a sign $\varepsilon$  with the following properties:
$\chi(x^\varepsilon) \geq 0$ and a path $p_{y, \varepsilon}$ from $t$ to $tx^\varepsilon$ can be found
so that $v_\chi (p_{y, \varepsilon}) > v_\chi (1,x^\varepsilon)$.
But if so, the geometric version of the $\Sigma1$-criterion holds,
for the missing paths can be constructed with the help of  Remark \ref{remark:Sigma-1-criterion}a.
\end{proof}

In the remainder of section \ref{sec:Sigma1-criterion-revisited},
we describe an algorithm that constructs a lower bound $\psi(\RR)$ of $\Sigma^1(G)$
and then apply it to some examples.
The algorithm presupposes that a non-empty set of relators $\RR$ be known at the outset;
it works with any number of generators.
Its outcome, however, is rarely satisfactory  if $G$ is generated by more than 2 or 3.
There exists an improvement that can cope with a large number of generators;
it will be discussed in section \ref{ssec:Whitehead-criterion-and-applications}.
Its justification will rely on properties of the weaker algorithm.
%
%
\subsection{The lower bound $\psi(\RR)$}
\label{ssec:Subset-psi(R)}
%
We begin by fixing the notation.
As usual,
$G$ denotes a group, $\eta \colon \XX \to G$ a finite generating system 
and $\YY$ the alphabet $\XX \cup \XX^{-1}$.
Set
\begin{equation}
\label{eq:Subst-Y-positive}
\YY_{>0} = \{y \in \YY \mid \chi(y) > 0\} \text{ and  } \XX_{0} = \{x \in \XX \mid \chi(x) = 0 \}.
\end{equation}
Given a set $\RR$ of relators of $G$ --- 
in other words,  a set  of reduced words  
that belong to the kernel of the epimorphism $\eta_{*} \colon F(\XX)\epi G$ ---
we want to define a subset $\psi(\RR)$ and then to prove 
that $\psi(\RR) \subseteq\Sigma^1(G)$.
In the definition of $\psi(\RR)$ enter finite sequences of functions $f_r$,  
one for each $r \in \RR$; 
they are defined as follows:
suppose $r$ has the spelling $s_{1}s_{2}\cdots s_{k}$ as a $\YY$-word 
and that $\chi$ is a non-zero character of $G$.
Then 
\begin{equation}
\label{eq:sequencef(r,chi)}
\index{Notation!fsub-r-chi@$f_{(r,\chi)}$}
f_r (\chi) = \left(\chi(s_1),\, \chi(s_1s_2), \ldots,\, \chi(s_1s_2 \cdots  s_{k-1}), \; \chi(s_1s_2 \cdots  s_{k-1}s_{k})\right).
\end{equation}
The domain of definition of  $f_r (\chi)$  will be thought of as being the circle $\Z/k\Z$;
in particular,
the symbol $f_{r}(k+1)$ will denote the term $f_{r}(1)$.
\begin{definition}
\label{definition:psi(R)}
\index{Notation!psi of R@$\psi(\RR)$}
\index{Definition of!lower bound psi of R@lower bound $\psi(\RR)$}
Let $t \in \YY$ be an element for which $\HH_{t} = \{[\chi] \in S(G) \mid \chi(t) > 0 \}$ is non-empty.
The subset  $\psi(\RR)_{t}$ is made up of the points $[\chi]$ which satisfy the following conditions:
\begin{enumerate}[(i)]
\item for every $y \in \YY_{>0} \smallsetminus \{t\}$
there exists a relator $r_{y} \in \RR$ such that the 
sequence $f_r (\chi)$  assumes its minimum only once, say in $j$, 
and the subword $s_{j}s_{j+1}$ is either $y^{-1}t$ or $t^{-1}y$;
\item for every $x \in \XX_0$ there exists a relator $r_{x} \in \RR$ such that the 
sequence $f_r (\chi)$ assumes its minimum in two consecutive indices $j$, $j+1$, 
and only there, 
and that the subword $s_{j}s_{j+1}s_{j+2}$ is either $t^{-1}xt$ or $t^{-1}x^{-1}t$.
\end{enumerate}
Define $\psi(\RR)$ to be the union $\bigcup_{t} \psi(\RR)_{t}$ 
where $t$ ranges over all letters $t \in \YY$ 
for which  $\HH_{t} = \{[\chi] \in S(G) \mid \chi(t) > 0 \}$ is non-empty. 
\end{definition}

\begin{remark}
\label{remark:Definition-Sigma(R)}
The letters $y$ and $t$ occurring in statement (i) will be called the 
\emph{letters involved in the minimum}. 
This minimum can be assumed at the last index $k$;
if so,  $\min f_{r_{y}} = \chi(r_{y}) = 0$ and either $r_{y} = t \cdots y^{-1}$ or $r_{y} = y \cdots t^{-1}$.
\index{Notation!fsub-r-chi@$f_{(r,\chi)}$}

Similarly, we say in the case of statement (ii)
that the letters $x^{\varepsilon}$ and $t$ are involved in the minimum 
if the subword $s_{j}s_{j+1}s_{j+2}$ equals $t^{-1}x^{\varepsilon} t$.
\end{remark}
\begin{example}
\label{example:Illustration-R(R)}
The following example illustrates the computation of the set $\psi(\RR)$ in the case 
where  $G_{\ab}$ has rank 1; 
then $S(G)$ has only 2 points and the determination of $\psi(\RR)$ is easy.

Let $G$ be given by the presentation $\PP = \langle a, b, c \mid r_1, r_{2} \rangle$ 
where
\begin{equation}
r_{1} = a ca cb ca^{-1}b^{-1}c^{-3} \quad \text { and } \quad 
r_{2} = ac^{-1}a^{-1}b^{-2}c^{-1}bcb^2c.
\end{equation}
The exponents sums of these relators with respect to the given generators  are
\[
(\sigma_{a}(r_{1}), \sigma_{b}(r_{1}),  \sigma_{c}(r_{1})) =(1,0,0)
\text{ and }
(\sigma_{a}(r_{2}), \sigma_{b}(r_{2}),  \sigma_{c}(r_{2})) =(0,1,0).
\]
So $a$ and $b$ represent elements in the commutator subgroup of $G$,
whence $G_{\ab}$ is an infinite cyclic group generated by the image of $c$.
The points in $S(G)$ are therefore represented by the character 
\[\chi \colon a \mapsto 0, \quad b \mapsto 0,\quad  c \mapsto 1
\]
and by its opposite $-\chi$.
The functions $f_{r_{1}}$ and $f_{r_{2}}$ are easily computed:
\begin{align*}
f_{r_{1}} (\chi) &= (0,1,1,2,2,3,3,3,0), \\
 f_{r_{2}} ( \chi) &=  (0,-1,-1,-1,-1,-2,-2,-1,-1,-1,0), \\
f_{r_1} ( -\chi) &=(0,-1,-1,-2,-2,-3,-3,-3,0),\\
 f_{r_2} (-\chi) &= (0,1,1,1,1,2,2,1,1,1,0).
 \end{align*}
 In the case of the character $\chi$,
 both sequences assume their minima exactly twice and at consecutive indices;
 in the case of $f_{r_1}$ the minima occur at the last letter $c^{-1}$ and the first letter $a$;
 in the case of $f_{r_2}$ the minima occur at the fifth letter $c^{-1}$ and the sixth letter $b$.
 The letter $c$ can play the r{\^o}le  of $t$ --- it is actually the only choice for $t$;
 it is involved in both minima, 
 while $a$ is involved in the minimum of $f_{r_1}$ and $b$ is involved in $f_{r_2}$.
 The point $[\chi]$ thus lies in $\psi(\RR)$.
 
Now to the character $-\chi$. 
The sequence $f_{r_1} ( -\chi)$ assumes its minimum thrice;
as there are only two relators in $\RR$ the point $[-\chi]$ is therefore outside $\psi(\RR)$.
So $\psi(\RR) = \{[\chi] \}$.
 \end{example}

The next result explains our interest in $\psi(\RR)$:
\begin{prp}
\label{prp:psi(R)-subset-Sigma1}
\index{Invariant Sigma1@Invariant $\Sigma^1$!lower bound psi@lower bound $\psi(\RR)$}
Let $\eta \colon \XX \to G$ be a finite generating system of the group $G$
and let  $\RR$ be a non-empty set of cyclically reduced words in 
$\YY = \XX \cup \XX^{-1}$. 
If  $\RR$ is made up of  relators of $G$ then $\psi(\RR)$ is an open subset of $\Sigma^1(G)$.
\end{prp}
\begin{proof}
The openness of $\psi(\RR)$ follows directly from
the continuity of the function $\chi \mapsto v_r(\chi)$ and the definition of $\psi(\RR)$.

Suppose $[\chi] \in \psi(\RR)$.
By the definition of $\psi(\RR)$ there exist then a letter $t \in \YY_{>0}$,
a relator $r_{y}$ for each $y \in \YY_{>0} \smallsetminus \{t\}$
and a relator $r_{x}$ for each $x \in \XX_0$,
all in such a way that the properties enunciated in statements (i) and (ii) hold.
We want to show that suitable cyclic permutations of one of the  words $r_{y}$ and $r_{y}^{-1}$, 
and of one of the words $r_{x}$ and $r_{x}^{-1}$ satisfy  condition
\eqref{eq:Form-relator-algebraic-Sigma1-criterion}.

Suppose first $y \in \YY_{>0}$.
Then there exists  a relator $r_{y}=s_{1} \cdots s_{k} \in \RR$ 
such that the sequence
\[
f_{r_y}(\chi)= \left( \chi(s_{1}),  \chi(s_{1}s_{2}), \ldots , \chi(s_{1}s_{2} \cdots s_{j}), \ldots, \chi(r_{y}) \right)
\]
has a unique minimum and that the letters $t$ and $y$ are involved in this minimum.
If the minimum occurs in $k$, the relator has either the form $t \cdots y^{-1}$ or $y \cdots t^{-1}$
and all its proper initial segments $u$ have positive $\chi$-value.
If the relator starts with $y$ it satisfies therefore the requirement stated in condition
\eqref{eq:Form-relator-algebraic-Sigma1-criterion};
if it starts with $t$, its inverse will be a relator satisfying the condition.

If the minimum occurs in $j < k$, 
we set $u = s_{1} \cdots s_{j}$ and $w= s_{j+1} \cdots s_{k}$
and consider the cyclic permutation $w\cdot u$ of $r_{y}$.
We claim that all the values
\[
\chi(s_{j+1}),\; \chi(s_{j+1}), \cdots, \chi(w),\; \chi(w \cdot s_{1}), \; \chi(w \cdot s_{1}s_{2}), \ldots, 
\chi(w \cdot s_{1} \cdots s_{j-1})
\]
are positive.
To see this, notice that $\chi(u)$ is the minimum of the sequence $f_{r_y}(\chi)$.
It follows firstly that $\chi(u \cdot s_{j+1} \cdots s_{\ell})> \chi(u)$, 
whence $\chi(s_{j+1} \cdots s_{\ell} ) > 0$ for $\ell = j+1, j+2, \ldots, k$;
and then that $\chi(s_{1}\cdots s_{h}) > \chi(u)$,
whence 
\[
\chi(w \cdot s_{1}\cdots s_{h}) = -\chi(u)  + \chi(s_{1}\cdots s_{h}) > 0.
\]
Recall now that the letters $t$ and $y$ are involved in the minimum of the sequence $f_{r_y}(\chi)$.
It follows that the cyclic permutation  $w \cdot u$ has either the form $t \cdots y^{-1}$  or $y \cdots t^{-1}$.
In the first case the word $u\cdot w$ itself has properties required by condition
\eqref{eq:Form-relator-algebraic-Sigma1-criterion};
in the second case, the relator $ (w \cdot u)^{-1}$ will have them.
\smallskip

Let's now investigate the case of a letter $x \in \XX_0$.
By the definition of $\psi(\RR)$ there exists a relator $r_{x} =s_{1} \cdots s_{k}$ 
such  that the minimum of  the sequence
\[
f_{r_x}(\chi) = \left( \chi(s_{1}),  \chi(s_{1}s_{2}), \ldots , \chi(s_{1}s_{2} \cdots s_{j}), \ldots, \chi(r_x) \right)
\]
occurs twice, at two consecutive indices $j$, $j+1$, 
and so  that $s_{j}s_{j+1}s_{j+2} = t^{-1}x^{\varepsilon} t$ for some sign $\varepsilon$.
Consider then the word $r$ obtained from $r_{x}$
by cyclic permutation which starts with the letter $s_{j+1}$.
This word is a relator of the form $x^{\varepsilon}t \cdots t^{-1}$;
one sees as in the first part of the verification 
that $\chi$ assumes positive values on all its proper initial segments.
If $\varepsilon=1$,
this relator satisfies therefore condition \eqref{eq:Form-relator-algebraic-Sigma1-criterion}
for a letter $y$ with $y \in \XX_0$.
If $\varepsilon=-1$ the cyclically permuted relator has the form $x^{-1}t \cdots t^{-1}$.
Its inverse is then of the form $t \cdots t ^{-1}x$
and has positive $\chi$-values on all proper initial segments.
By moving the last letter to the front, one arrives at a relator satisfying condition 
\eqref{eq:Form-relator-algebraic-Sigma1-criterion}.
\end{proof}
%
%
\subsubsection{Geometrical reformulation}
\label{sssec:Geometrical-reformulation}
%
If the rank of the abelianized group $G_{\ab}$ is larger than 1, 
the sphere $S(G)$ has infinitely many points and $\psi(\RR)$ cannot be determined 
by investigating the points $[\chi]$ one after another.
In some cases, a geometric reformulation will then help.
It splits the evaluation of the sequences into two part:
one uses first the relators to produce sequences with values in a Euclidean lattice $\Z^k$;
in the second part,
these sequences are composed with the linear forms 
$x \mapsto \langle u, x\rangle$;
here $u$ ranges over the unit sphere  of $\E^k$.
The details are as follows.

Let $\vartheta \colon G \epi G_{\ab}\epi \Z^k$ be an epimorphism of $G$ 
onto the standard lattice in $\R^k$ equipped with the usual scalar product.
Then $\vartheta$ induces an embedding 
\begin{equation}
\label{eq:Embedding-bis}
\sigma(\vartheta) \colon \s^{k-1} \iso S(G, \ker \vartheta) \incl S(G), 
\quad 
u \longmapsto [g \mapsto \langle u, \vartheta(g)\rangle]
\end{equation}
of spheres (see section \ref{sssec:Coordinates-sphere}).
Every relator $r = s_{1}s_{2}\cdots c_{k}$ of $G$ gives rise to a sequence
\begin{equation}
\label{eq:Vector-valued-sequence}
f_r(\vartheta) = \left( \vartheta(s_{1}), \vartheta(s_{1}) + \vartheta(s_{2}),  \ldots, \vartheta(s_{1}) + \cdots + \vartheta(s_{k}) \right)
\end{equation}
of lattice points.
The points of this sequence are the vertices of a path $\bar{p}$ in the Cayley graph $\Gamma(\Z^k, \XX)$
that ends in the origin;
we will think of $\overline{p}$ as being a loop, starting at and ending in the origin.
If $k=2$, this loop can easily be depicted.
As the examples in the next section will reveal,
such a visualization can be helpful in determining  $\psi(\RR)$.   
%

%
\subsection{Application to groups with two generators}
\label{ssec:Groups-2-generators}
%
The next aim is the computation of  $\psi(\RR)$ in the simplest interesting set-up,   
in that of a group with 2 generators.
The definition of $\psi(\RR)$ simplifies then considerably.

Suppose $G$ is generated $a$ and $b$,
and $\RR$ is a set of cyclically reduced words in $a$, $a^{-1}$ and $b$, $b^{-1}$ made up of relators of $G$. 
Given a non-zero character $\chi$ of $G$,
two cases arise: if $\chi(a)$ and $\chi(b)$ are both non-zero,
then $ \YY = \{a, a^{-1}, b, b^{-1}\}$ contains two elements,
say $y_{1}$, $y_{2}$, with positive values;
if now  $r \in \RR$ is a relator such that the associated sequence $f_{r}(\chi)$ has a unique minimum,
then $y_{1}$, $y_{2}$ must be the letters involved in this minimum.
Statement (i) in the definition \ref{definition:psi(R)}
of $\psi(R)$ is then satisfied for $t = y_{1}$ as well as $t = y_{2}$.

Suppose now that $\chi(a)= 0$. 
Then there exists a sign $\varepsilon$ so that $\chi(b^{\varepsilon} )$ is positive
and  $b^{\varepsilon} $ is the only element in $\YY$ that is eligible for $t$.
In this situation,
we seek a relator $r \in \RR$ whose associated sequence assumes its minimum  at consecutive indices,
and only there. The letters involved in this minimum are the automatically $t$ and one of $a$, $a^{-1}$.
The case where $\chi(b) = 0$ is similar.

Definition \ref{definition:psi(R)} can therefore be restated as follows:
\begin{definition}
\label{definition:psi(R)-2-generators}
Let $G$ be a group generated by two elements $a$, $b$
and let $\RR$ be a set of cyclically reduced words in $a$, $a^{-1}$, $b$ and $b^{-1}$
which are relators of $G$. 
Then a point $[\chi] \in S(G)$ belongs to $\psi(\RR)$ if, and only if, the following condition holds:
\begin{quote}
\emph{if  $\chi(a_{1})$ and $\chi(a_{2})$ are both non-zero,
the minimum of some sequence $f_{r}( \chi)$ with $r \in \RR$ occurs only once;
if one of $\chi(a)$, $\chi(b)$ is zero,
the minimum of some sequence $f_{r}( \chi)$ with $r \in \RR$ occurs at two consecutive indices and only there}.
\end{quote}
\end{definition}
\index{Computation of Sigma1@Computation of $\Sigma^1$ for!two generator groups}

In the remainder of this section \ref{ssec:Groups-2-generators},
the above proposition will be used in the geometrical interpretation described in 
\ref{sssec:Geometrical-reformulation}. 
%
\subsubsection{Groups with a single defining relator}
\label{sssec:One-relator-groups-example}
%
Let $\langle a, b \mid r \rangle $ be  a presentation with a single defining relator $r$
which has exponent sum 0 in each of the generators,
say
\begin{equation}
\label{eq:Browns-relator}
r =  a^{-1} b^{-1} ab^2 a^{-1} b^{-1} a^2 b^{-1}a^{-1} b a^{-1}ba b^{-1}.
\end{equation}
The assignments $a \mapsto (1,0)$ and $b \mapsto (0,1)$ 
extend to an epimorphism  $\vartheta \colon G \epi \Z^2$.
Under this epimorphism,
the relator gives rise to the loop $\bar{p} = ((0,0),r)$ in the Cayley graph  
$\overline{\Gamma} = \Gamma(\Z^2, \{a,b\} )$;
it is indicated on the left of Figure \ref{fig:Illustration-path-Brown}.
\index{Notation!path overline-p@$\bar{p}$}
\begin{figure}[htb]
\psfrag{1}{\hspace*{-1.2mm} \small 1}
\psfrag{2}{  \hspace*{-3mm}   \small 2}
\psfrag{3}{\hspace*{-1mm}\small 3}
\psfrag{4}{\hspace*{-0.8mm}\small 4}
\psfrag{5}{\hspace*{-2.0mm} \small 5}
\psfrag{6}{  \hspace*{-2.5mm}   \small 6}
\psfrag{7}{\hspace*{-0.8mm}\small 7}
\psfrag{8}{\hspace*{-0.5mm}\small 8}
\psfrag{9}{\hspace*{-2mm} \small 9}
\psfrag{10}{  \hspace*{-3mm}   \small 10}
\psfrag{11}{\hspace*{-1mm}\small 11}
\psfrag{12}{\hspace*{-2mm}\small 12}
\psfrag{13}{\hspace*{-2.5mm} \small 13}
\psfrag{14}{  \hspace*{-4.3mm}   \small 14}
\psfrag{15}{\hspace*{-1.4mm}\small 15}
\psfrag{16}{\hspace*{-1mm}\small 16}
\psfrag{ori}{  \hspace*{-7mm}   \small $$}
\psfrag{aa}{  \hspace*{-2.5mm}   \small $a$}
\psfrag{bb}{  \hspace*{-2.5mm}   \small $b$}
\psfrag{xx1}{\hspace*{-0.0mm}\small $(1,0)$}
\psfrag{yy1}{\hspace*{-0.5mm}\small $(0,1)$}
\begin{center}
\includegraphics[height=3.5cm]{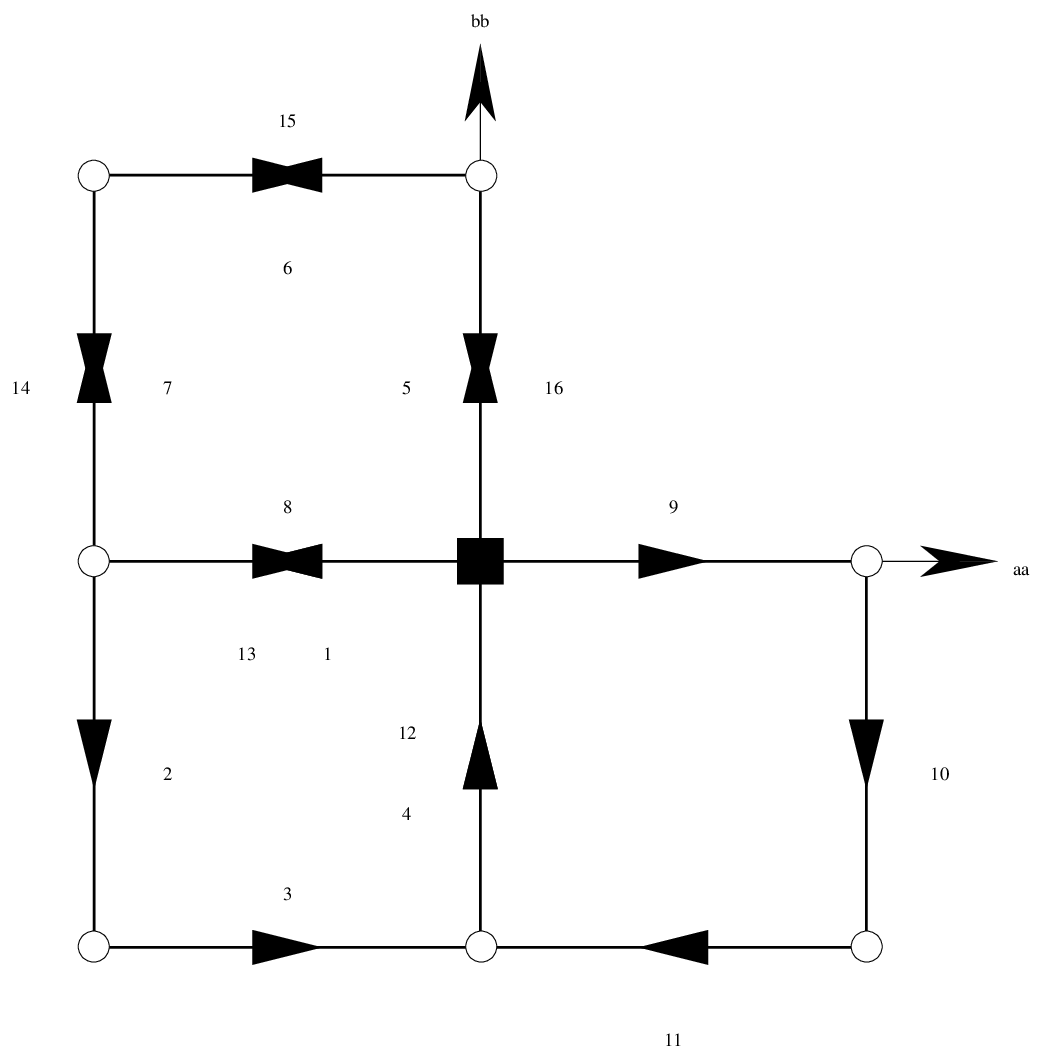}
\hspace*{1.2cm}
\includegraphics[height=3.5cm]{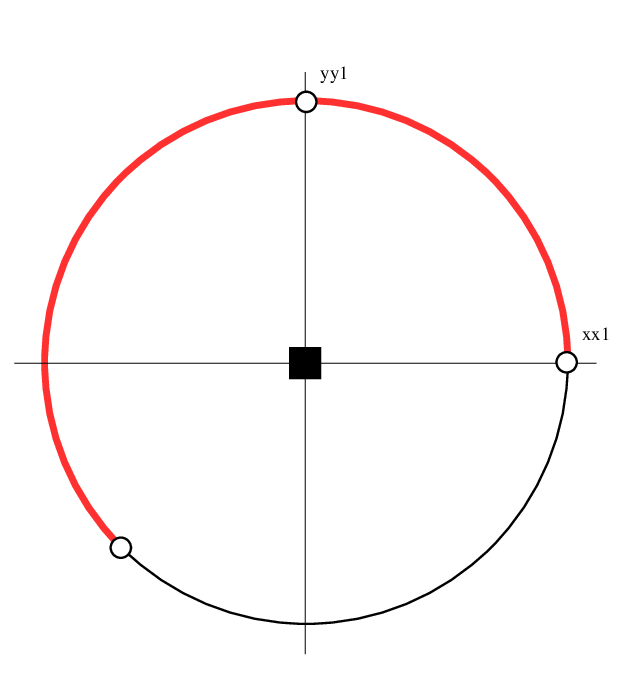}
\caption{The loop $\bar{p}$ in $\Gamma(\Z^2, \{a,b\} )$ 
and the subset $\sigma(\theta)^{-1} (\psi(\RR))$}
\label{fig:Illustration-path-Brown}
\end{center}
\end{figure}

A word on the interpretation of this figure is in order.
The path $\bar{p}$ is not simple;
worse, it contains not only multiple vertices, but edges traversed more than once.
One way of dealing with this problem is to set off  multiple occurrences of the same edge from one another
(see, e.\;g.,  \cite[p.\;493]{Bro87b}).
In Figure  \ref{fig:Illustration-path-Brown}
a different rendering is chosen: 
the edges are numbered consecutively, with the position number  written on the left of the oriented edge.
The path starts at the origin, indicated by the black square in the middle of the figure, 
then goes around the left, lower square in counterclockwise fashion,
then counterclockwise around the left, upper square,
then clockwise around the right square
and finally in clockwise fashion around the upper left square.

The  determination of the subset $\psi(\{r\})$ can now be carried out like this.
We seek unit vectors $u  = (u_{1}, u_{2}) \in \s^1$ 
for which the function $h_{u} \colon x \mapsto \langle u, x\rangle$
assumes its minimum at most twice along the vertices of the  path $\bar{p}$, 
the vertices being counted with multiplicity.
This condition is fulfilled if $u_{1}> 0$ and $u_{2}> 0$;
then the minimum is assumed  in $(-1,-1)$ and only there.
In view of definition  \ref{definition:psi(R)-2-generators}
the image of the open first quadrant of $\s^1$ under $\sigma(\vartheta) \colon \s^1 \iso S(G)$ 
belongs therefore to $\psi(\{r\})$.
If $u_{1}< 0 < u_{2}$,
the function $h_{u}$ assumes its minimum in $(1,-1)$ and only there,
whence the image of the open second quadrant of $\s^1$  lies in  $\psi(\{r\})$.
 
If $u = (-1,0)$, 
the function $f_{u}$ assumes its minimum at the endpoints of the tenth edge of the  path $\bar{p}$
and only there;
the second part of the condition stated in definition  \ref{definition:psi(R)-2-generators}  
therefore applies  
and shows that  $[\chi_{u}]$ lies in $\psi(\{r\})$.
If, finally,  $u_{1} < u_{2} < 0$, 
the minimum is assumed in $(1,0)$, and only there;
so $\psi(\{r\})$ contains a further arc. 
The union of the four subsets detected constitute all of $\psi(\{r\})$;
its preimage is depicted in red on the right of Figure
\ref{fig:Illustration-path-Brown}.
\begin{remark}
\label{remark:Browns-example}
The subset $\psi(\{r\}) $  coincides with $\Sigma^1(G)$;
see Theorem \ref{thm:Sigma1-one-relator-group}.
\end{remark}
%
\subsubsection{Groups of PL-homeomorphisms}
\label{sssec:Group-PL-homeomorphisms-example}
%
The second example presents an instance
where the invariant can be determined completely,
in spite of the fact 
that only some relators, but not a presentation,  are known.
The example fits into a larger class of groups (see \cite[Section 8]{BNS}).

Let $f$ and $g$ be piecewise linear homeomorphisms of the unit interval $[0, 1]$ 
that satisfy the following requirements  for some positive real numbers $\varepsilon < \delta$:
\begin{align}
f(t) &< t \text{ for  }  t \in\;  ]0, 1[,
\label{eq:Condition-PL-homeomorphism-group-irreducibility}\\
f (t) &= g(t) \text{ for }   t \leq \varepsilon,
\label{eq:Condition-PL-homeomorphism-group-character-lambda}\\
g (t) &= t  \text{ for }   t \geq \delta.
 \label{eq:Condition-PL-homeomorphism-group-character-rho}
\end{align}
Let $G$ denote the group generated by these homeomorphisms $f$ and $g$.

Requirements 
\eqref{eq:Condition-PL-homeomorphism-group-character-lambda}
and
\eqref{eq:Condition-PL-homeomorphism-group-character-rho} 
imply that $G_{\ab}$ is free-Abelian of rank 2.
Indeed,
the  derivative in 0 induces a homomorphism of $G$ into the multiplicative group of $\R_{>0}$;
upon composing it with the natural logarithm $\ln$, one obtains a character $\chi_{l} \colon G \to \R$;
similarly, the composition of the derivative in 1 with $\ln$ yields a character $\chi_{r}$ of $G$.
The claim now follows from the facts that 
\begin{equation}
\label{eq:Characterizing-chi-l-and chi-r}
\chi_{l}(f \circ g^{-1}) = 0 , \quad \chi_{l}(g^{-1} ) > 0, \quad \text{and}  \quad \chi_{r}(f \circ g^{-1}) > 0, 
\quad \chi_{r}(g^{-1} ) = 0.
\end{equation}
The characters $\chi_{l}$ and $\chi_{r}$ enter into the description of $\Sigma^1(G)$ as follows:
\begin{prp}
\label{prp:Invariant-PL-homeomorphism-group}
\index{Computation of Sigma1@Computation of $\Sigma^1$ for!groups of PL-homeomorphisms}
Let $G$ denote the group generated by two PL-homeomorphisms $f$, $g$ of the unit interval $[0, 1]$ 
satisfying requirements \eqref{eq:Condition-PL-homeomorphism-group-irreducibility}
through \eqref{eq:Condition-PL-homeomorphism-group-character-rho}
for some real numbers $\varepsilon < \delta < 1$.
Then $G_{\ab}$ is free Abelian of rank 2 and 
\begin{equation}
\label{eq:Invariant-PL-homeomorphism-group}
\Sigma^1(G) = S(G) \smallsetminus \{[\chi_{l}], [\chi_{r}]\}.
\end{equation}
\end{prp} 

The proof divides into two parts of different sorts.
To show that the points $[\chi_{l}]$ and $[\chi_{r}]$ lie outside $\Sigma^1(G)$,
one analyses the kernels of $\chi_{l}$, $\chi_{r}$ 
and then employs Proposition \ref{prp:Strictly-descending-HNN-extension},
a result on descending HNN-extensions.
To establish that all other points of $S(G)$ belong to $\Sigma^1(G)$
one determines some relators $r_{m}$ satisfied by the generators $f$, $g$ 
and calculates $\psi(\{r_{m}\})$.
An appeal to proposition \ref{prp:psi(R)-subset-Sigma1} 
and a short auxiliary computation then yield the desired conclusion.

\paragraph{Proof of proposition \ref{prp:Invariant-PL-homeomorphism-group}: part 1.}
We first show that $[\chi_{r}] \notin\Sigma^1(G)$.
To do so,
we introduce the \emph{support} of a permutation $h \in G$;
it is defined to be the subset $\{x \in X \mid h(x) \neq x \}$. 
If $h_1$ and $h_2$ are permutations,
the formula 
\begin{equation}
\label{eq:Support-conjugated-permutation}
\supp(h_{1}\circ h_{2}\circ h_{1}^{-1}) = h_{1}(\supp(h_{2})).
\end{equation}
holds, as one verifies easily.
Let  now $\delta_0$ be the supremum of $\supp(g)$; 
assumptions \eqref{eq:Condition-PL-homeomorphism-group-irreducibility} 
through \ref{eq:Condition-PL-homeomorphism-group-character-rho} 
imply that $\varepsilon < \delta_0 \leq \delta < 1$.
The  subset
\[
H = \{ h \in G \mid h(t) = t  \text{ for }  t \geq \delta_0\}.
\]
is actually a subgroup of $G$ and it satisfies the inclusion 
\[
 f \circ H\circ f^{-1}   = 
  \{ h \in G \mid h(t) = t  \text{ for }  t \geq f(\delta_0)\}\; \subsetneqq H 
 \]
for $g \in H \smallsetminus f \circ H \circ f^{-1}$.
It shows that $G$ is a \emph{properly descending} HNN-extension with base group $H$ and stable letter $f$.
Proposition \ref {prp:Strictly-descending-HNN-extension}
allows us therefore to conclude that $[\chi_{r}] \notin \Sigma^1(G)$. 
One can verify similarly that $[\chi_{l}] \notin \Sigma^1(G)$.
\smallskip

The following proposition provides a tool for showing 
that a rank 1 point lies \emph{outside} $\Sigma^1$.
 It can be generalized to points of arbitrary rank
(see Proposition \ref{prp:Characterization-complement-Sigma1}).

\begin{prp}
 \label{prp:Strictly-descending-HNN-extension}%
 \index{Invariant Sigma1 and@Invariant $\Sigma^1$ and!descending HNN-extension}%
 Let $\chi \colon G \to  \Z \incl \R$ be a rank 1 character of  $G$
 and let $t \in G$ be an element with $\chi(t) = 1$.
 If $G$ is a \emph{strictly} descending HNN-extension $\langle H, t \mid t H t^{-1}  \subsetneqq  H  \rangle $
 over a subgroup $H$ of $\ker \chi$  then $[\chi] \notin \Sigma^1(G)$.
 \end{prp}
 \begin{proof}
Set $N = \ker \chi$ 
and consider a finitely generated subgroup $B \subset N$ with $ B \subset tBt^{-1}$.
Since $B$ is finitely generated and $H$ is the basis of a descending HNN-extension with stable letter $t$,
there exists a positive integer $\ell_0$ such that  $B \subset t^{-\ell} H t^\ell$.
For all $j \geq 0$ the following chain of inclusions
\[
t^j Bt^{-j} \subseteq t^j (t^{-\ell} H t^\ell)t^{-j} = t^{-\ell} (t^{j} H t^{-j})t^{^\ell} \subseteq t^{-\ell} H t^\ell
\]
is then valid. 
Since $t^{-\ell} H t^\ell \subsetneqq N$ this chain precludes the  sequence $j \mapsto t^j B t^{-j}$  
from sweeping out $N$,
whence $[\chi] \notin \Sigma^1(G)$ by Proposition \ref{prp:Ascending-HNN-extension}.
\end{proof}
 
\paragraph{Proof of Proposition \ref{prp:Invariant-PL-homeomorphism-group}: part 2.}
Requirements  \eqref{eq:Condition-PL-homeomorphism-group-irreducibility}
through \eqref{eq:Condition-PL-homeomorphism-group-character-rho}
 imply some distinctive relations among the generators $f$ and $g$.
To detect them, we proceed as follows.
Requirement \eqref{eq:Condition-PL-homeomorphism-group-irreducibility} implies, first of all
that there exists a positive exponent $m_{0}$ with  $f^{m_{0}}(\delta)  \leq \varepsilon$.
Next requirement \eqref{eq:Condition-PL-homeomorphism-group-character-rho}
and formula \eqref{eq:Support-conjugated-permutation}
give rise to the chain of inclusions
\[
\supp(f^m \circ g \circ f^{-m}) = f^m(\supp(g)) \subseteq f^m([0,\delta]) \subseteq [0, \varepsilon].
\]
Set  $g_{m} = f^m \circ g \circ f^{-m}$. 
Then the commutativity relation
\begin{equation}
\label{eq:Relation-homeomorphisms-f-g}
(g^{-1} \circ f) \circ g_{m}= g_{m} \circ (g^{-1} \circ f)
\end{equation}
is valid for every $m \geq m_{0}$;
indeed, the functions $f$ and $g$ agree on the interval $[0,\varepsilon]$,
whence the support of $g^{-1} \circ f$ is contained in $[\varepsilon,1]$ and so disjoint from the support of $g_m$.
Relation \eqref{eq:Relation-homeomorphisms-f-g}  
is equivalent to the relator
\begin{equation}
\label{eq:Relator-homeomorphisms-f-g}
r_{m} =  f \circ (f^{m} \circ g \circ f^{-m}) \circ f^{-1} \circ g \circ (f^{m} \circ g^{-1} \circ f^{-m}) \circ g^{-1} .
\end{equation}
The following Lemma then allowsw us to conclude that
$\Sigma^1(G) \supseteq S(G) \smallsetminus \{[\chi_{l}], [\chi_{r}]\}$.

\begin{lem}
\label{lem:Sigma1-two-relator-group}
Given a positive integer $m$, set 
\begin{equation}
\label{eq:Definition-commutator-word}
r_{m} = a^{m+1}\cdot b \cdot a^{-(m+1)} \cdot b\cdot a^m\cdot b^{-1} \cdot a^{-m}\cdot b^{-1}
\end{equation}
and define  $G_{m}$ to be the group given by the presentation $\langle a, b \mid r_{m}, r_{m+1} \rangle$.
Then $G_{\ab}$ is free-Abelian of rank 2 
and $\Sigma^1(G) \supseteq S(G) \smallsetminus \{[\chi_{1}], [\chi_{2}]\}$
with $\chi_{1}$, $\chi_{2}$ the characters  given by
\begin{equation}
\label{eq:Definition-characters-1-and-2}
\chi_1(a) = 1, \quad \chi_{1}(b) = 0 \quad \text{and}\quad \chi_{2}(a) = \chi_2(b) = -1.
\end{equation}
\end{lem}

\begin{proof}
Let $\vartheta \colon G \epi \Z^2$ be the epimorphism sending $a$ to $(1,0)$ and $b$ to $(0,1)$.
Under this epimorphism the relator $r_{m}$ give rise to a closed path $\overline{p}_{m} = ((0,0), r_{m})$
in the Cayley graph $\Gamma(\Z^2, \{a, b\})$;
for $m>1$ it has the form indicated on the left of Figure \ref{fig:Path-relator-homeomorphism-group}.
The corresponding subset $\psi(\{r_{m}\})$ is depicted on the right of the figure. 
Here $\chi_{3}$ and $\chi_{4}$ denote the characters 
\[
\chi_{3}(a) = 0, \quad \chi_{3}(b) = 1 \quad \text{and} \quad \chi_{4} = - \chi_{3}.
\]
If $m=1$,  the subset $\psi(\{r_{1}\})$ contains also the point $[\chi_{4}]$. 
\begin{figure}[htb]
\psfrag{1}{\hspace*{-2mm} \tiny 1}
\psfrag{2}{  \hspace*{-3mm}  \tiny 2}
\psfrag{3}{\hspace*{-2.5mm} \tiny $\cdots$}
\psfrag{4}{\hspace*{-3.5mm}\tiny $m+1$}
\psfrag{5}{\hspace*{-7mm} \tiny $m+2$}
\psfrag{6}{  \hspace*{-5.5mm}   \tiny $m+3$}
\psfrag{7}{\hspace*{-0mm} \tiny}
\psfrag{8}{\hspace*{-0mm} \tiny }
\psfrag{9}{\hspace*{-0mm} \tiny}
\psfrag{10}{  \hspace*{-9.5 mm} \tiny $2m + 4$ }
\psfrag{11}{\hspace*{-4mm}\tiny $2m+5$}
\psfrag{12}{\hspace*{-1mm}\tiny }
\psfrag{13}{\hspace*{-2.5mm} \tiny }
\psfrag{14}{  \hspace*{-2.5mm}   \tiny $3m+5$}
\psfrag{15}{\hspace*{-1.5mm}\tiny }
\psfrag{16}{\hspace*{-1.5mm}\tiny }
\psfrag{17}{\hspace*{-1.5mm}\tiny }
\psfrag{18}{\hspace*{-0.5mm}\tiny  $4m+6$ }
\psfrag{aa}{\hspace*{-0.0mm}\footnotesize $a$}
\psfrag{bb}{\hspace*{-0.0mm}\footnotesize $b$}
\psfrag{c1}{\hspace*{-0.5mm}\footnotesize $[\chi_{1}]$}
\psfrag{c2}{\hspace*{-2mm}\footnotesize $[\chi_{3}]$}
\psfrag{c3}{\hspace*{-3.5mm}\footnotesize $[\chi_{2}]$}
\psfrag{c4}{\hspace*{-2mm}\footnotesize $[\chi_{4}]$}
\begin{center}
\includegraphics[height=3.5cm]{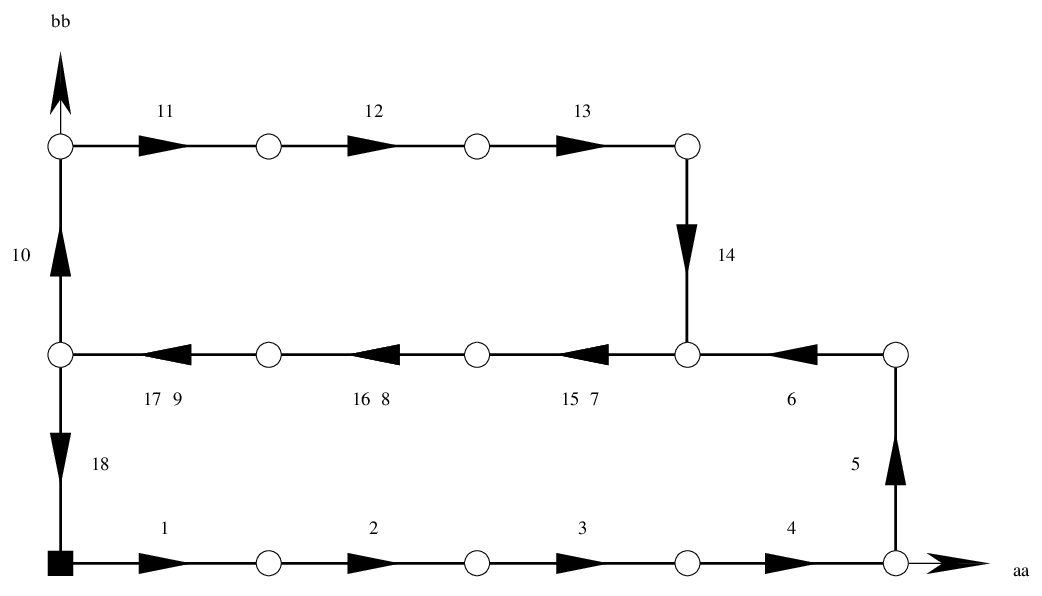}
\hspace*{0.7cm}
\includegraphics[height=3.5cm]{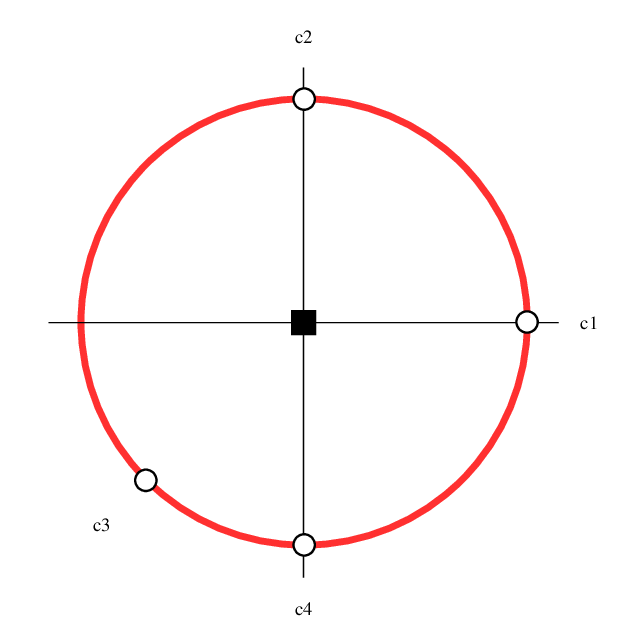}
\caption{Path $\overline{p}_m$ in $\Gamma(\Z^2, \{a,b\} )$ 
and subset $\psi(\{r_{m}\}) \subset S(Q)$}
\label{fig:Path-relator-homeomorphism-group}
\end{center}
\end{figure}

We are left with proving that $[\chi_{3}]$ and $[\chi_{4}]$ lie in $\Sigma^1(G_{m})$;
to this end we derive some auxiliary relations from the relators $r_{m}$ and $r_{m+1}$.
These relators are equivalent to the relations
\begin{equation}
\label{eq:Consecutive-relations}
\act{a^{m+1}}{1.0}{b} = \act{ba^m}{1.0}{b} \quad \text{and} \quad \act{a^{m+2}}{1}{b} = \act{ba^{m+1}}{1}{b} .
\end{equation}
If the first relation is conjugated by $a$,
the left members of both relations agree;
the right members then lead to the relator
\[
r_{3}  =  a b a^m  b a^{-m} b^{-1}a^{-1}\cdot  b a^{m+1} b^{-1} a^{-m-1} b^{-1}
\]
The two relations allow one also to derive the following chain of transformations
\[
\act{a^{m+1}}{1.0}{b} = \act{ba^m}{1.0}{b}  
=  \act{ba^{-1}b^{-1}} {0.0} {\left( \act{ba^{m+1}}{1.0}{b} \right) }
=\act{ba^{-1}b^{-1}a^{m+2}}{1.0}{b},
\]
which is equivalent to the relator
\[
r_{4}  =a^{m+1} ba^{-m-1} \cdot b  a^{-1} b^{-1}  a^{m+2} b^{-1} a^{-m-2} b a b^{-1}.
\]

Consider now the sequences $f_{r_3}(\chi_3)$ and $f_{r_4}(\chi_4)$. 
The first of them assumes its minimum twice, at the first letter $a$ and at $r_3$;
by definition \ref{definition:psi(R)-2-generators} and Proposition \ref{prp:psi(R)-subset-Sigma1},
the point $[\chi_3]$ lies therefore in $\Sigma^1(G_m)$. 
The second sequence assume its minimum again twice, at the initial segment $a^{m+1} ba^{-m-1}b$
and at the following one, whence $[\chi_4] \in \Sigma^1(G_m)$ by the same reasons.
\end{proof}

\begin{remarks}
\label{remarks:PL-homeomorphism-group-example}
 a)  The  calculations in the previous example can be summarized as follows:
 let $G$ be the group 
 generated by two piecewise linear homeomorphisms $f$ and $g$ of the unit interval $[0,1]$. 
Suppose the generators satisfy the requirements
 \eqref{eq:Condition-PL-homeomorphism-group-irreducibility}
 through
 \eqref{eq:Condition-PL-homeomorphism-group-character-rho}.
 Then the complement $\Sigma^1(G)^c$ of the invariant consists of only two points,
 represented by the characters $\ln \circ D_{0} $ and $\ln \circ D_{1}$.
 Here $D_{0}$ denotes the derivative at the fixed point 0
 and $D_{1}$ denotes the derivative at the left global fixed point 1 of the interval $[0,1]$.
 
This result shows 
that the invariant does not depend on the details of the functions $f$ and $g$.
By \cite[Thm.\;8.1]{BNS}
the same conclusion holds for a far larger class of groups made up of piecewise linear homeomorphisms.

b) The abelianization $G_{\ab}$ of $G$ is free abelian of rank 2.
The group $G$ contains therefore infinitely many normal subgroups $N$ with infinite cyclic quotient.
All of them are finitely generated,
except the kernels of the derivatives with respect to the end points 0 or 1
(use the previous calculation of $\Sigma^1(G)$ and Corollary \ref{crl:Characterizing-fg-kernel-rank-1-character}).
\end{remarks}
 %
 %
\subsection{The lower bound $\Psi(\RR)$}
\label{ssec:Whitehead-criterion-and-applications}
%
In this section,
an improved lower bound for $\Sigma^1(G)$ is derived.
To explain how it differs from $\psi(\RR)$, 
we recall the strategy underlying the construction of $\psi(\RR)$.

The algebraic $\Sigma^1$-criterion says 
that a non-zero-character $\chi$ represents a point of $\Sigma^1(G)$
if $\card(\XX) -1$ relators with specific properties can be found.
In the case of of $\psi(\RR)$,
these relators are sought in the symmetrization $\sym(\RR)$ of $\RR$,
that is the set consisting of the cyclic permutations of  the elements of $\RR \cup \RR^{-1}$.
In the improvement,  
one looks for relators in a larger set consisting of products 
whose factors are either elements of $\sym(\RR)$ 
or conjugates $ x  r x^{-1}$ with $x \in \XX$ and $r \in \sym(\RR)$.
%
%
\subsubsection{Defining $\Psi(\RR)$}
\label{sssec:Subset-Psi-definition}
%
We begin by fixing the notation.
As before,
 $\eta \colon \XX \to G$ denotes a finite system of generators of the group $G$ 
and  $\YY = \XX \cup \XX^{-1}$ the associated set of letters. 
Next, 
$\RR$ is a set of cyclically reduced relators in $\YY$.
Given $r \in \RR$ and a non-zero character $\chi \colon G \to \R$,
let $f_{r}(\chi)$ denote the sequence
\begin{equation}
\label{eq:sequencef(r,chi)-ter}
\left(\chi(s_1),\, \chi(s_1s_2), \ldots,\, \chi(s_1s_2 \cdots  s_{k-1}), \; \chi(s_1s_2 \cdots  s_{k-1}s_{k})\right)
\end{equation}
and view its domain as a circle with $k$ adjacent to 1.
Finally, set 
\[
\YY_{>0} = \{y \in \YY \mid \chi(y) > 0 \} \text{ and }\XX_0 = \{x \in \XX \mid \chi(x) = 0 \}.
\]
We next introduce a graph
that will be used in the definition of the subset $\Psi(\RR)$ of $S(G)$.
\begin{definition}
\label{definition:Auxiliary-grpahs}
Given a non-zero character $\chi$ of $G$, put
\begin{align*}
\RR_{{\chi,+}} &= \{ r \in \RR \mid  f_r(\chi) \text{ assumes its minimum once} \},\\
\RR_{{\chi,0}} &= \{ r \in \RR \mid  f_r(\chi)  \text{ assumes its minimum  twice,
at adjacent indices}\}.
\end{align*}
Using $\RR_{{\chi,+}}$,
define a graph  $\GG_{\RR}(\chi)$ 
with vertex set $\YY_{>0}$ and edge set
\begin{equation}
\label{eq:Edge-set-graph-sufficient-condition}
\left\{ 
\{ y_{1}, y_{2}  \} 
\mid \text{$y_{1}$ and $y_{2}$ are involved in $\min f_r(\chi) $ for some } r  \in \RR_{\chi,+} \right\}.
\end{equation}
\index{Notation!graph Gsub-R-chi@$\GG_{\RR}(\chi)$}
\end{definition}
\begin{definition}
\label{definition:Psi(R)}
A point $[\chi] \in S(G)$ belongs to $\Psi(\RR)$ if, and only if,
the following conditions are satisfied:
\begin{enumerate}[(i)]
\item the graph $\GG_{\RR}(\chi)$ is connected;
\item for every generator $x \in \XX_0$ 
there exist a relator $r = s_{1} \cdots s_{k} \in \RR_{{\chi,0}}$ so that $f_r(\chi)$ 
assumes its minimum at two consecutive indices $j$, $j+1$
and the subword $s_{j} s_{j+1}s_{j+2}$ has the form $y_{1}^{-1}x^{\varepsilon}y_{2}$ 
with $y_{1}$, $y_{2} \in \YY_{>0}$ and $\varepsilon \in \{1,-1\}$. 
\end{enumerate}
\end{definition} 
\index{Notation!Psi-R@$\Psi(\RR)$}
\index{Definition of!lower bound Psi-R@lower bound $\Psi(\RR)$}

We shall later show that $\Psi(\RR)$ is a subset of $\Sigma^1(G)$ and that it contains $\psi(\RR)$.
But first we give an example which illustrates the definition of $\Psi(\RR)$.
\begin{example}
\label{example:Illustration-Psi(R)}
Let $G$ be given by the presentation $\langle a, b, c \mid r_{1}, r_{2}\rangle$ with
\begin{equation}
\label{eq:SnapPeas-presentation-Dunfields-link}
r_{1} = abAB \quad \text{and} \quad r_{2} = ac^2bCAbCbc^2acBC^2AC^2Bc.
\end{equation}
Here $A$, $B$ are short for  $a^{-1}$ and $b^{-1}$.
The exponent sums of the second relator are
\[
(\sigma_{a}(r_{2}),\sigma_{b}(r_{2}),\sigma_{c}(r_{2})) = (0,1,0).
\]
It follows that $G_{\ab}$ is free-Abelian of rank 2 
and that $G$ admits an epimorphism $\vartheta \colon G \epi \Z^2$ taking $a$ to $(0,1)$ and $c$ to $(1,0)$ (and, of course,  $b$ to $(0,0)$).
\footnote{This choice reflects the disposition of the axes in Figure \ref{fig:Dunfields-example}.}

The paths  $\bar{p}_{r_{1}}$, $\bar{p}_{r_{2}}$  
in the Cayley graph $\Gamma = (\Z^2, \{a,b,c\})$ have  edges which are loops;
in Figure \ref{fig:Dunfields-example} they are indicated by small leaves.
The straight edges are numbered as usual;
the leaves are not labelled,
but they are taken into account in the numbering of the straight edges.
\begin{figure}[ht]
\psfrag{1}{\hspace*{-1mm} \small $1$}
\psfrag{2}{  \hspace*{-2.0mm}\small $2$}
\psfrag{3}{\hspace*{-1.5mm} \small $3$}
\psfrag{5}{\hspace*{0.0mm}\small $5$}
\psfrag{6}{\hspace*{-1.5mm} \small $6$}
\psfrag{8}{\hspace*{0mm}\small $8$}
\psfrag{10}{\hspace*{-1mm}\small $10$}
\psfrag{11}{\hspace*{-1.6mm}\small $11$}
\psfrag{12}{\hspace*{-1.4mm}\small $12$}
\psfrag{13}{\hspace*{-1.3mm}\small $13$}
\psfrag{14}{\hspace*{-1mm}\small $14$}
\psfrag{15}{\hspace*{-1.3mm}\small $15$}
\psfrag{16}{\hspace*{-1.7mm}\small $16$}
\psfrag{17}{\hspace*{-1.7mm}\small $17$}
\psfrag{18}{\hspace*{-2mm}\small $18$}

\psfrag{19}{\hspace*{-1mm}\small $19$}
\psfrag{21}{\hspace*{-0.7mm}\small $21$}
\psfrag{51}{\hspace*{-0.5mm}\small $1$}
\psfrag{53}{\hspace*{0.0mm}\small $3$}
\psfrag{a}{\hspace*{-0.3mm}\small $a$}
\psfrag{c}{\hspace*{0.0mm}\small $c$}
\psfrag{ori}{\hspace*{-0.3mm}\small $$}
\psfrag{thc}{\hspace*{1.0mm}\small $\theta(c)$}
\psfrag{tha}{\hspace*{0.0mm}\small $\theta(a)$}
\begin{center}
\includegraphics[height=2.6cm]{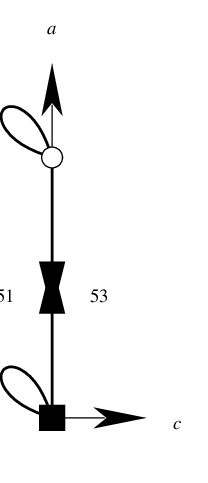}
\hspace*{0.1cm}
\includegraphics[height=2.6cm]{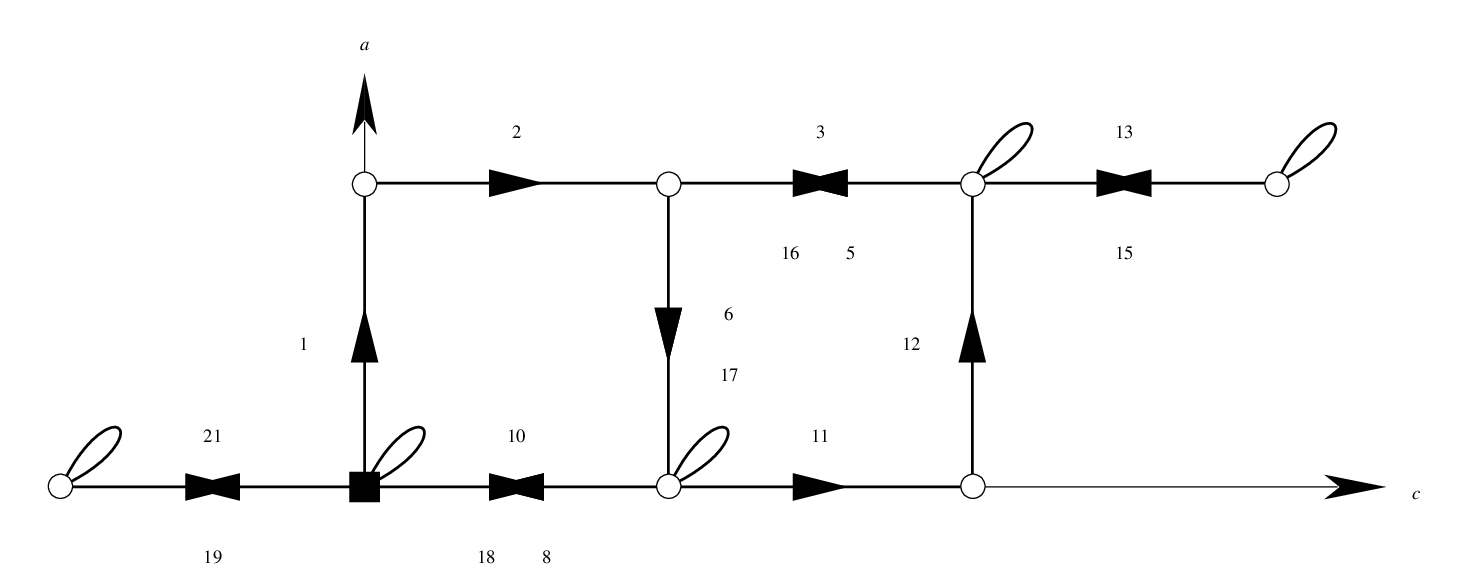}
\hspace*{0.6cm}
\includegraphics[height=2.6cm]{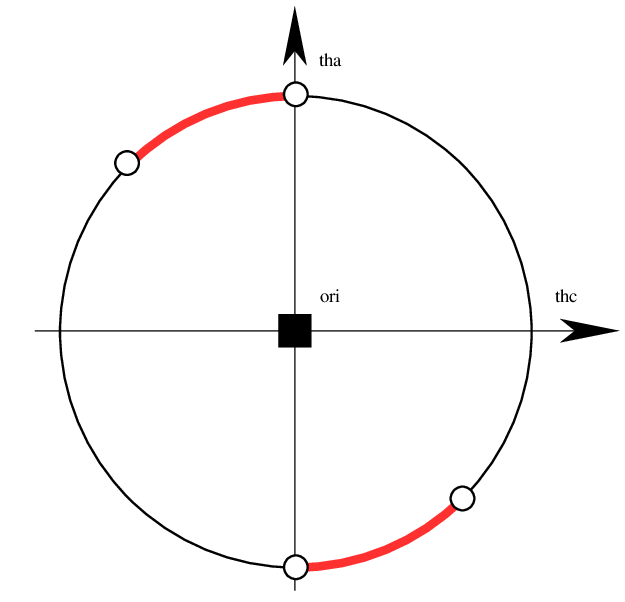}
\caption{Loops of the  relators $r_1$ and $r_{2}$ and the set $\Psi'$}
\label{fig:Dunfields-example}
\end{center}
\end{figure}

Let us determine $\Psi' = \sigma(\theta)^{-1}(\Psi(\RR))$ for $\RR = \{r_{1}, r_{2} \}$.
Two cases arise.
If $\chi_u(a)$ and $\chi_u(c)$ are both non-zero, 
the graph $\GG_{\RR}(\chi)$ has two vertices
and so it is only connected if one of the graphs in Figure \ref{fig:Dunfields-example}
contributes an edge to it. 
Only the path $p_{r_2}$ can do that, 
and only if the supporting half plane $\HH_u$ touches the convex hull in a single vertex 
with no leaf attached to it.
(Recall that the unit vector $u$ points into the half plane $\HH_u$.)
These conditions are fulfilled if, and only if,  $u$ lies in the first part of the second or the fourth open quadrant.
If the stated conditions are fulfilled, 
the relator $r_1$ shows 
that $b$ satisfies condition (ii) in definition \ref{definition:Psi(R)}.
It follows that $\Psi(\RR)$ contains the described arcs.

If one of  $\chi_u(a)$ or $\chi_u(c)$ is zero,
the graph $\GG_{\RR}(\chi)$ is connected
and and condition (ii) holds only
if each of the the sequences $f_{r_1}(\chi)$ and $f_{r_2}(\chi)$ assumes its minimum precisely twice, 
at consecutive indices.
The path $p_{r_2}$ shows that $f_{r_2}(\chi)$ assumes its minimum only twice if $\chi(a) = 0$.
As the first sequence $f_{r_1}(\chi)$ is then constant,
$\Psi(\RR)$ contains none of the 4 points under discussion,
and so $\Psi(\RR)$ is the union of the two arcs found before.
\end{example}
%
\subsubsection{Verification that $\Psi(\RR)$ is a lower bound}
\label{sssec:Whitehead-criterion-statement}
%
The following result is a variant of \cite[Theorem A]{GaMe98}, 
established by Susan Garner Garille and John Meier. 
It reveals that $\Psi(\RR)$ is a lower bound of the invariant
and that this new bound improves on $\psi(\RR)$.
\index{Meier, J.}
\index{Garille, S. G.}
\begin{prp}
\label{prp:Psi(RR)-subset-Sigma1}
Let $\eta \colon \XX \to G$ be a finite generating system of the group $G$
and $\RR$ a non-empty set of cyclically reduced words in  $\YY = \XX \cup \XX^{-1}$. 
If  $\RR$ is made up of  relators of $G$ then $\Psi(\RR)$ is an open subset of $S(G)$
and $\psi(\RR) \subseteq \Psi(\RR) \subseteq \Sigma^1(G)$.
\end{prp}
\index{Invariant Sigma1@Invariant $\Sigma^1$!lower bound Psi@lower bound $\Psi(\RR)$}

\begin{proof}
The openness of $\Psi(\RR)$ follows directly from
the continuity of the function $\chi \mapsto v_r(\chi)$ and the definition of $\Psi(\RR)$.

First we show that $\psi(\RR) \subseteq \Psi(\RR)$.
Suppose $\chi$ represents a point of $\psi(\RR)$.
By definition \ref{definition:psi(R)}
there exist then a letter $t \in \YY_{>0}$ and, for each $y \in \YY_{>0}\smallsetminus \{t\}$,
a relator $r_{y} \in \RR_{{\chi,+}}$ 
such that $t$ and $y$ are involved in the unique minimum of $f_r(\chi)$.
The graph $\GG_{\RR}(\chi)$ thus contains for every vertex $y$ distinct from  $t$ 
an edge connecting $y$ and $t$ and so it is connected.
Moreover, 
there exists by (ii) of definition \ref{definition:psi(R)},
for every $x \in \XX_{0}$, 
a relator $r_{x} \in \RR_{{\chi,0}}$ 
such that the minimum of $f_{(r_{x},\chi)}$ occurs at $j$ and $j+1$, say, 
and so that the subword $s_{j} s_{j+1}s_{j+2}$ equals $t^{-1}x^{\varepsilon}t$ with $|\varepsilon] = 1$.
These relators satisfy  then requirement (ii) of the definition of $\Psi(\RR)$,
and so $[\chi] \in \Psi(\RR)$.

We verify next that $\Psi(\RR) \subseteq \Sigma^1(G)$.
Suppose $[\chi]$ is in $\Psi(\RR)$.
By condition (i) in definition \ref{definition:Psi(R)}
the graph $\GG_{\RR}(\chi)$ is then connected.
Given an edge  $\{y_{1}, y_{2} \}$ of this graph
there exists a relator $r \in \RR$ such that the sequence  $f_r(\chi)$ has a unique minimum,
say at $j$, and that $y_{1}$ and $y_{2}$ are involved in this minimum.
So $r$ is a word of the form $w_{1}y_{1}^{-1}y_{2}w_{2}$ or the form $w_{1}y_{2}^{-1}y_{1}w_{2}$;
there is no loss in assuming that we are in first case.
A suitable cyclic permutation $r'$ of $r$ is then equal to $y_{2}w_{2}w_{1}y_{1}^{-1}$
and the $\chi$-values of all proper initial segments of $r'$ are positive.
Notice that the inverse of $r'$ has the form  $y_{1} \cdots y_{2}^{-1}$ 
and satisfies the condition that all its proper initial segment have positive $\chi$-values.
The vertex set of the graph $\GG_{\RR}(\chi)$ is not empty for $\chi$ is not the zero-homomorphism;
there exists therefore  a vertex in the graph; call it $t$.

Consider next a vertex $y \in \YY_{>0} \smallsetminus \{t \}$.
Since the graph $\GG_{\RR}(\chi)$ is connected;
there is an edge path from $y$ to  $t$, say $p = (y = y_{1}, y_{2}, \ldots, y_{k} = t)$.
By the previous paragraph
there exists, for each $i = 1$, \ldots, $k-1$, a relator $r'_{i}$ 
which has the form $r_{i} = y_{i} w_{i}y_{i+1}^{-1}$ 
and enjoys the property that all its proper initial segments have positive $\chi$-values.
The product
\[
 (yw_{1}y_{2}^{-1})\cdot (y_{2}w_{2}y_{3}^{-1}) \cdots  (y_{k-2}w_{k-1}y_{k-1}^{-1} )\cdot (y_{k-1}w_{k}t^{-1})
\]
is then a word that starts with $y$, ends with $t^{-1}$ and is a relator of  $G$.
It simplifies to $r'(y \to t) = yw_{1}w_{2}\cdots w_{k}t^{-1}$;
this word  $r(y \mapsto t)$ may not be reduced,
but is has the property that all its proper initial segments have positive $\chi$-values.
The unique freely reduced word $r_{y}$
that is freely equivalent  to $r'(y \mapsto t)$
starts then with $y$, ends with $t{-1}$ and has positive $\chi$-value on all of its proper initial segments. 
It satisfies thus the requirements of the relator denoted $r_{y}$ in equation 
\eqref{eq:Form-relator-algebraic-Sigma1-criterion}.

Consider, finally,  $x \in \XX_{0}$.
By condition (ii) in definition \ref{definition:Psi(R)},
there exists then a relator $r_{x} \in \RR$ with the following properties: 
the sequence $f_r(\chi)$ attains its minimum exactly twice,
at consecutive indices $j$, $j+1$, say,
and the subword $s_{j} s_{j+1}s_{j+2}$  has the form $y_{1}^{-1}x^{\varepsilon}y_{2}$ 
with $y_{1}$, $y_{2} \in \YY_{>0}$ and $\varepsilon \in \{1,-1\}$. 
If $\varepsilon = 1$,
a cyclic permutation of $r$ has the form $xy_{2} \cdots y_{1}^{-1}$
and $\chi$ is positive on each proper initial segment,
except, of course, on the first one $x$;
if $\varepsilon =-1$,
a cyclic permutation $r^{-1}$ will have the form $xy_{1} \cdots y_{2}^{-1}$ 
and enjoy the previously stated property.

So far we know
that there exists, for every $x \in \XX_{0}$, a reduced relator $r_{x}$ 
which has the form $xy_{1} w y_{2}^{-1}$  
and positive $\chi$-values on each proper initial segment except the first one.
Various cases now arise.
If $y_{1}= y_{2}= t$,
the relator  $r_{x}$ itself has the form stated in
\eqref{eq:Form-relator-algebraic-Sigma1-criterion}.
If $y_{2} \neq t$,
there exists, by the first part of the proof, 
a reduced relator $r(y_2 \to t)$ having the form $y_2w_{2} t^{-1}$
and positive $\chi$-values on all of its proper initial segments.
The product
\[
r_x \cdot r_{y_2 \to t} = xy_1w_{1} y_2^{-1} \cdot  y_{2} w_2 t^{-1}
\]
is then a relator that simplifies to the relator $x y_1 w_1 w_{2}t^{-1}$
and has positive $\chi$-values on its proper initial segments  distinct from the last one.
If $y_{1} = t$, this relator fulfills requirement  
\eqref{eq:Form-relator-algebraic-Sigma1-criterion}.
Otherwise, 
there  exists a relator $r(y_1 \to t) =y_1w t^{-1}$
with positive $\chi$-values on its proper initial segments.
The product
\[
xr(y_1 \to t)^{-1}x^{-1} \cdot (xy_1w_{1}w_2 y_{2}^{-1})
\]
simplifies then to a relator of the form  $xtw_{4}t^{-1}$
that satisfies requirement
\eqref{eq:Form-relator-algebraic-Sigma1-criterion}.

All taken together we have shown that there exists,
for every $y \in \YY_{>0}$, a reduced word of the form  $ys_{2} \cdots t{-1}$ 
and, for every $x \in \XX_{0}$,  a reduced word of the form  $xts_{3} \cdots t^{-1}$,
all in such a way that each of these words is a relator of $G$
for which condition \eqref{eq:Form-relator-algebraic-Sigma1-criterion} holds.
Theorem \ref{thm:Algebraic-Sigma1-criterion} thus implies
that $[\chi] \in \Sigma^1(G)$.
\end{proof}
%
\subsubsection{Application to graph groups}
\label{sssec:Whitehead-criterion-application-graph-groups}
\index{Computation of Sigma1@Computation of $\Sigma^1$ for!graph groups}
%
A graph group is given by a finite combinatorial graph $\Delta$;
its vertices $x_{j}$ are the standard generators of the group  $G = G_{\Delta}$
and  its edges $\{x_{j},x_{h} \}$ give rise to the defining relations  $x_{j}x_{h} = x_{h}x_{j}$.
Let $\RR$ be  the set consisting of the  commutator $[x_{j}, x_{h}] = x_{j}x_{h} x_{j}^{-1}x_{h}^{-1}$ 
and its inverse $[x_{h}, x_{j}]$ for every edge $\{j, h\}$ of $\Delta$.
We contend that $\Psi(\RR)$ \emph{coincides with} $\Sigma^1(G)$.
The inclusion $\Psi(\RR) \subseteq \Sigma^1(G)$ holds by Proposition \ref{prp:Psi(RR)-subset-Sigma1};
it suffices thus to establish the opposite inclusion.
This we do with the help of Theorem \ref{thm:Sigma1-right-angled-Artin-group}.
It shows 
that if a non-zero character $\chi$ represents a point of $\Sigma(G)$ 
then its living subgraph $\LL(\chi)$ is \emph{connected} 
and every vertex outside $\LL(\chi)$ must be \emph{adjacent to a vertex in} $\LL(\chi)$.

These conditions are, in essence, 
requirements (i) and (ii) enunciated in the definition of $\Psi(\RR)$.
To see this,
we compare first the living subgraph $\LL(\chi)$,
defined in section \ref{sssec:Right-angled-Artin-groups-computing-Sigma1},
with the graph $\GG_{\RR}(\chi))$ occurring in definition \ref{definition:Psi(R)}.
The vertex set of $\LL(\chi)$ consists of the generators $x_{i}$ with non-zero $\chi$-values;
its edge set is induced by that of the defining graph $\Delta$,
for  $\LL(\chi)$ is a full subgraph of $\Delta$.
The vertex set of the graph  $\GG_{\RR}(\chi))$, on the other hand,
is the set of letters
\[
\YY_{>0} 
= \{ x_{i}^{\varepsilon_{i}} \mid \varepsilon_{i} \cdot \chi(x_{i}) > 0 \text{ and } |\varepsilon_{i}| = 1 \}.
\]
Two letters $y_{1}$, $y_{2}$ are connected in $\GG_{\RR}(\chi))$ 
if there is a relator $r \in  RR$ such that the sequence $f_{(r, \chi)}$ assumes its minimum only once
and if the letters involved in the minimum are $y_{1}$, $y_{2}$.
Now each of the relators in $\RR$ contains only two generators
and the only relator in which the letters 
$y_{1} = x_{i}^{\varepsilon_{i}}$  and $y_{2} = x_{j}^{\varepsilon_{j}} $  occur 
are the commutator $[x_{i}, x_{j}]$ and its  inverse $[x_{j}, x_{i}]$.
Figure \ref{fig:Paths-commutators} then allows us to conclude  
that the graph $\GG_{\RR}(\chi))$ contains the edge $\{x_{i}^{\varepsilon_{i}}, x_{j}^{\varepsilon_{j}}\}$ 
if, and only if the living graph contains the edge $\{x_{i}, x_{j}\}$. 
The two graphs are therefore isomorphic.
\begin{figure}[htb]
\psfrag{1}{\hspace*{-2.5mm} \small $x_{i}$}
\psfrag{2}{  \hspace*{-4mm}  \small $x_{j}$}
\psfrag{3}{\hspace*{-2.5mm}  \small $x_{i}^{-1}$}
\psfrag{4}{\hspace*{-1.5mm} \small $x_{j}^{-1}$}

\psfrag{5}{\hspace*{-3mm}  \small $x_{j}$}
\psfrag{6}{  \hspace*{-3.5mm}   \small $x_{i}$}
\psfrag{7}{\hspace*{-2.5mm}  \small $x_{j}^{-1}$}
\psfrag{8}{\hspace*{-1.5mm}  \small $x_{i}^{-1} $}
%
\psfrag{a}{\hspace*{-1.5mm} \small$x_{i}$}
\psfrag{b}{\hspace*{-1.7mm} \small $x_{j}$}
\psfrag{B}{\hspace*{-1.7mm} \small$x_{i}$}
\psfrag{A}{\hspace*{-1.5mm} \small $x_{j}$}

\begin{center}
\includegraphics[height=3.2cm]{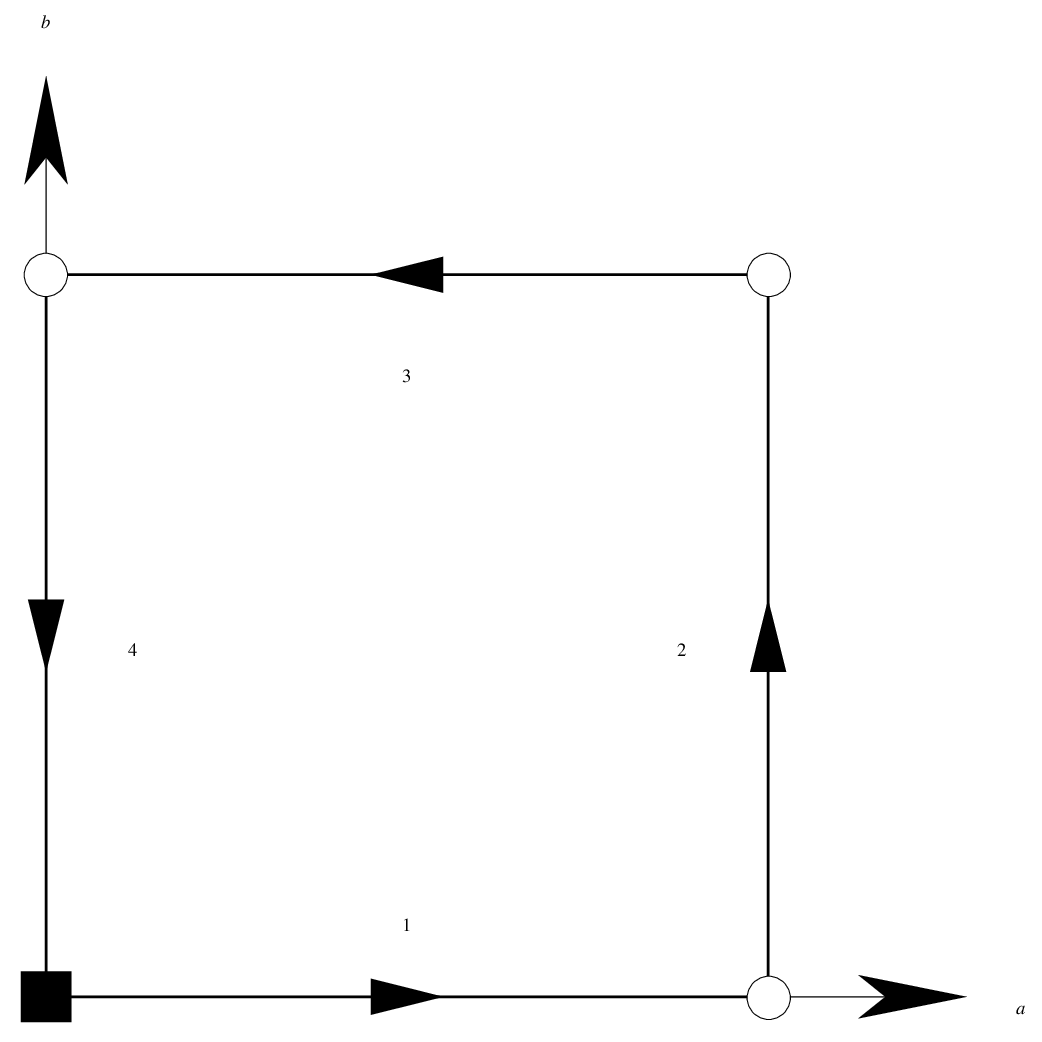}
\hspace*{2.0cm}
\includegraphics[height=3.2cm]{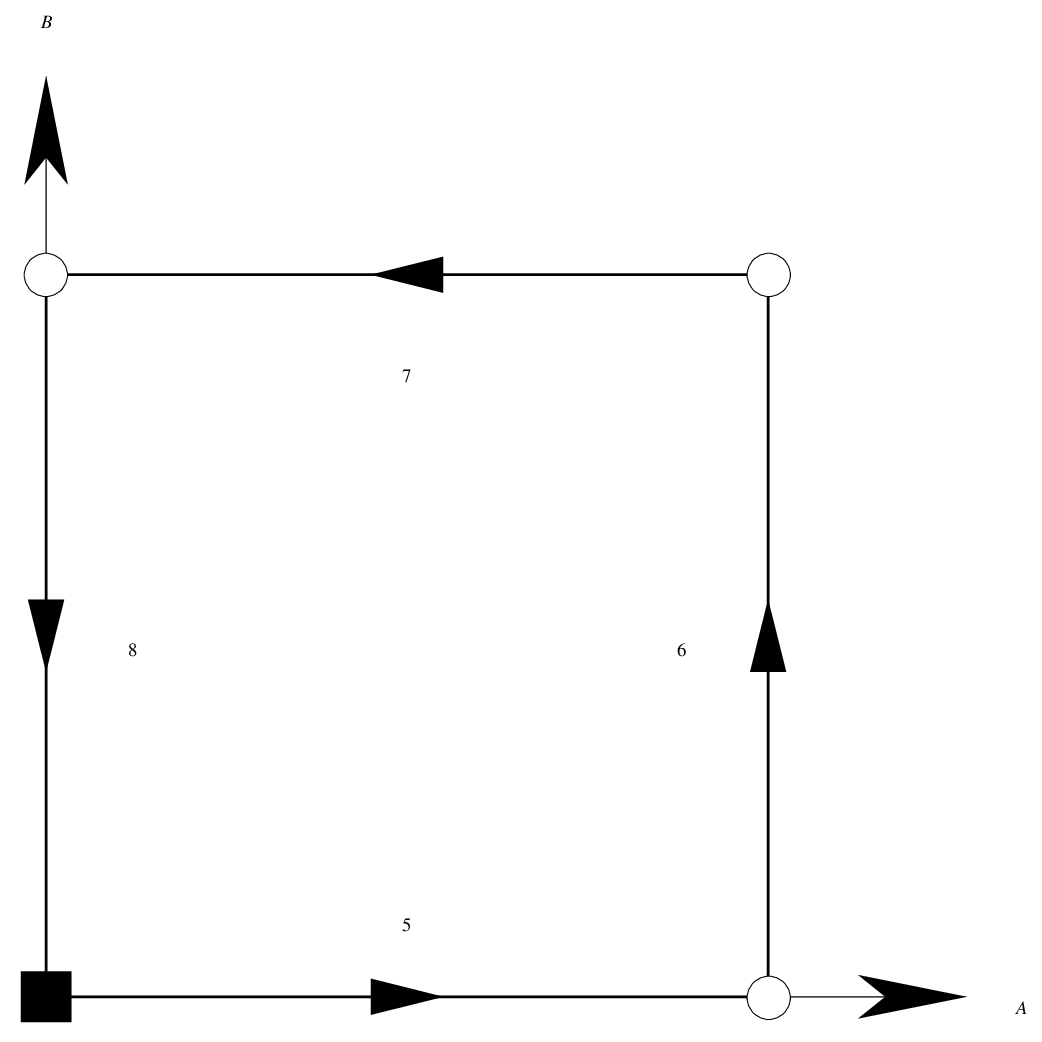}
\caption{Paths of the  relators $[x_{i}, x_{j}]$ and $[x_{j}, x_{i}]$}
\label{fig:Paths-commutators}
\end{center}
\end{figure}

Consider now a generator $x_j$  with $\chi(x_{j})= 0$.
Then $x_{j}$ is not a vertex of the living graph
and so it is, by assumption, an end point of an edge  
whose other end point is a vertex $x_{i}$, say, of $\LL(\chi)$.
Since $x_{j}$ is adjacent to $x_{i}$,
the commutators $[x_{i},x_{j}]$ and $[x_{j},x_{i}]$ are elements of $\RR$.
A glance at Figure \ref{fig:Paths-commutators} then shows
that each of the sequences $f_{[x_{i}, x_{j}]} (\chi)$ and $f_{[x_{j}, x_{j}]}( \chi)$ assumes its minimum twice
and that the letters involved in this minimum are $x_{i}^{\sigma}$ and $x_{j}^{\varepsilon}$
with $\sigma$ the sign of $\chi(x_{i})$.
So the generator $x_{j} \in \XX_{0}$ satisfies condition (ii)  of definition  \ref{definition:Psi(R)}.
%
\subsubsection{Application to knot groups}
\label{sssec:Whitehead-criterion-application-knot-groups}
\index{Computation of Sigma1@Computation of $\Sigma^1$ for!knot groups}
Let $k$ be an oriented knot in the 3-sphere, 
represented by an oriented, regular diagram $D$ in the euclidian plane.
The crossings of $D$ give rise to arcs $\alpha_1$, \ldots, $\alpha_m$, 
going from one undercrossing to the next one.
The Wirtinger presentation associates to each arc $\alpha_j$  a generators $x_j$ of $G $
and to each crossing a defining relator of the form 
\begin{equation}
\label{eq:Relators-Wirtinger-presentation}
r_j = x_j \cdot  x_{\beta(j)}^{\sigma(j)} \cdot x_{j+1}^{-1}\cdot  x_{\beta(j)}^{-\sigma(j)} \text{ for } 
j  \in \{1, \ldots, m\}.
\end{equation}
Here $x_{m+1}$ is to be read as $x_1$ and $\sigma(j) = \pm 1$.
(See See \cite[pp. 121--122]{Fox62}, \cite[pp. 56--57]{Rol76} or   \cite[p. 8 and pp. 32--35]{BZ03} 
for terminology and proofs.)

The relators imply that $G_{\ab}$ is infinite cyclic, and that every character $\chi \colon G \to \R$ 
assume the same value on the generators $x_j$. The sphere has therefore only two points,
represented by $\chi_+$ with $\chi_+(x_j) = 1$ and its negative.
Set $\RR = \{r_1, \ldots, r_m\}$.
Each of the graphs $\GG_\RR (\chi_+)$ and  $\GG_\RR (\chi_-)$ has $m$-vertices and $m$-edges.

The edge set of the first graph is the set of generators and its edges are 
\[
e_j = \begin{cases} 
\{x_j, x_{\beta(j)} \} &\text{ if } \sigma(j) = 1, \\
\{x_{j+1}, x_{\beta(j)} \} &\text{ if } \sigma(j) = -1
\end{cases}
\quad \text{for }  j  \in \{1, \ldots, m\}.
\]
The definition of $\GG_\RR (\chi_-)$ is similar.
Notice that $[\chi_+] \in \Psi(\RR)$ precisely if $\GG_\RR (\chi_+)$ is connected,
and similarly for $[\chi_-]$.

The graphs $\GG_\RR (\chi_+)$ and  $\GG_\RR (\chi_-)$ are easy to work out, 
but they depend strongly on the chosen diagram.
We shall come back to them in a later chapter.
\newpage
 
%
%
\section{Rank 1 points in $\Sigma^1$}
\label{sec:Sigma1-via-HNN-extensions}
%
%
Theorem \ref{thm:Consequence-finite-presentation}
in Chapter \ref{ch:Sigma-1-Cayley-graph}
describes an obstruction to the finite presentability of a soluble group $G$.
The theorem goes back to Theorem  C in \cite{BNS}, 
but it had a precursor, namely Theorem A in \cite{BiSt78}.
This old result is a \emph{structure theorem} for finitely presented groups admitting a rank 1 character $\chi$
and reads like this:

\emph{Let $G$ be a finitely presented group and $\chi \colon G \to\R$ a rank 1 character.
Then $G$ is a HNN-extension with a single stable letter, 
a \emph{finitely generated} base group and \emph{finitely generated} associated subgroups, 
all three contained in the kernel of $\chi$.}

This structure theorem leads to several types of consequences;
some of them can not be inferred from Theorem  \ref{thm:Consequence-finite-presentation}.
Moreover, recent investigations have led to a generalization of the structure theorem;
this generalization   has resulted in new applications of $\Sigma^1$.
These consequences and the generalization call for an adequate treatment of the structure theorem
in this monograph on Sigma invarants.

The present section provides such an account. 
It begins with the statement and proof of the structure theorem,
followed by a summary of important consequences of the result.
Then three major consequences will be discussed in some detail.  
The section will close with a description of the generalization alluded to.
%
\subsection{Structure theorem for indicable groups of type $\FP_2$}
\label{ssec:Structure-theorem-fp-indicable-groups}
%
Let $G$ be a finitely generated soluble group 
or, more generally, a \emph{finitely generated group which contains no free subgroup of rank 2}.
If $G$ admits a finite presentation and $\chi \colon G \to \R$ is a non-zero characters,
then Theorem \ref{thm:Consequence-finite-presentation} asserts
that  $\Sigma^1(G)$ contains at least one of te points  $[\chi]$ or $[-\chi]$.
In this statement, $\chi$ is allowed to be a character of rank greater than 1.

If the rank of $\chi$ is  1, 
the conclusion of Theorem \ref{thm:Consequence-finite-presentation} 
is a consequence of the following more general result:
\begin{thm}
\label{thm:FP2-groups-and-HNN-extensions}
Assume $G$ is a finitely generated group 
that is of type $\FP_2$ over some commutative ring $K$ (with $1 \neq 0$)
and $\chi \colon G \epi \Z \incl \R$ is a rank 1 character. 
Choose an element  $t \in G$ with $\chi(t) = 1$ and set $N = \ker \chi$.
Then there exist \emph{finitely generated subgroups} $S$,  $T$ and $B$ of  $N$,
with $S$ and $T$ contained in $B$,
so that conjugation by $t$ induces an isomorphism $\mu \colon S \iso T$
and such that the inclusion of $B$ in $G$ and the assignment $y \mapsto t$ induce an isomorphism
\begin{equation}
\label{eq:FP2-groups-and-HNN-extensions}
\tau \colon \langle B, y \mid y\cdot s\cdot y^{-1} = \mu(s) \text{ for all } s \in S \rangle \iso G.
\end{equation}
\end{thm}

Stated in less technical terms, Theorem \ref{thm:FP2-groups-and-HNN-extensions} asserts 
that \emph{every indicable group of type $\FP_2$} is an HNN-extension with finitely generated base group $B$ 
and finitely generated associated groups; 
moreover, $B$ can be chosen to  be a subgroup of the kernel of a preassigned rank 1 character $\chi$.
(Recall that a group is called \emph{indicable} if it projects onto $\Z$.)
\index{HNN-extensions!and fp groups}
\index{Finitely presented groups!and HNN-extensions}
\index{Structure theorem!for fp indicable groups}

\subsubsection{Proof of the structure theorem} 
\label{sssec:Proof-structure-theorem-for-fp-indicable-groups}
%
In many applications, the group $G$ is finitely presentable,
not merely of type $\FP_2$ over somme commutative ring.
As the proof of Theorem \ref{thm:FP2-groups-and-HNN-extensions}
simplifies sensibly under this stronger hypothesis,
we give first the proof with this stronger hypothesis, and then indicate
how the proof can be modified so as to become sound under the weaker hypothesis,

\paragraph{Proof for finitely presented groups.}
Assume $G$ is a finitely presented group and let $\chi \colon G \epi \Z \incl \R$ be a rank 1 character. 
Fix $t \in G$ with $\chi(t) = 1$
and choose a finite set  of generators $b_1$, \ldots, $b_f$ in $N =  \ker \chi$ 
so that $G = \gp(b_1, \ldots, b_f, t)$.
Next, let $F$ be the free group on the set $\{x_1,\ldots, x_f, y\}$ 
and $\rho$ the epimorphism of $F$ onto $G$ 
determined by $x_1 \mapsto a_1$, \ldots, $x_f \mapsto a_f$ and $y \mapsto t$.
By a basic result of B. H. Neumann's (see \cite[Lemma 8]{Neu37},
the kernel $R$ of the projection $\rho\colon F \epi G$ 
is then the normal closure of a finite set of relators, 
say $r_1$, \ldots, $r_k$.
Each relator $r_j$ is a product of powers of conjugates $y^\ell x_i y^{-\ell}$;
by replacing, if need be, some of these relators by conjugates,
we can arrange that only non-negative powers $\ell$ occur in the new defining relators. 
Since their number is finite
there is a natural number $m$ such that every $r_j$ is contained in the group
\[
V = \gp(\{ y^\ell x_i y^{-\ell} \mid 0 \leq \ell \leq m +1 \text{ and }   1 \leq i \leq f\}).
\]
Define  $B \subset G$ to be the image of $V$ under the epimorphism $\rho$;
so
\[
B = \gp(\{ t^\ell b_i t^{-\ell}  \mid 0 \leq \ell \leq m +1 \text{ and }   1 \leq i \leq f\}),
\]
and introduce the groups
$S = \gp(\{ t^\ell b_i t^{-\ell} \mid 0 \leq \ell \leq m\text{ and }   1 \leq i \leq f\})$ 
and $T = tSt^{-1}$. 
Clearly,  $S$ and $T$ are subgroups of $B$.

The inclusion of $B$ in $G$ and the assignment $u \mapsto t$ 
induce a homomorphism
\[
\tau \colon \tilde{G} = \langle H, u\mid u\cdot s\cdot u^{-1} = \mu(s) \text{ for all } s \in S \rangle \to G
\]
where $\mu(s) = t s t^{-1}$.
This homomorphism is \emph{surjective};
indeed,
the elements $b_1$, $b_2$, \ldots, $b_f$, $t$ generate $G$ and every generator $b_i$ lies in $B$.
Next the assignments $x_i \mapsto b_i$, where $1 \leq i \leq f$, and $x \mapsto u$
induce a homomorphism $\nu \colon F \epi \tilde{G}$ of the free group $F$ onto the HNN-extension $\tilde{G}$.
The defining relations of $\tilde{G}$ state 
that 
\[
u\cdot (t^\ell b_i t^{-\ell}) \cdot u^{-1} 
=
 t\cdot (t^\ell b_i t^{-\ell})\cdot t^{-1} \text{  for }0 \leq \ell \leq m \text{ and }1 \leq i \leq f.
\]
This implies 
that $u^\ell b_i u^{-\ell} = t^{\ell} b_i t^{-\ell}$ in $\tilde{G}$ 
for all for $0\leq \ell \leq m + 1$ and $1 \leq i \leq f$.
Consequently, $\nu$ maps $r_j = w_j(y^{\ell} x_i y^{-\ell})$ to
$w_j(u^\ell b_i u^{-\ell}) = w_j(t^\ell b_i t^{-\ell}) = 1$,
and so $\nu$ induces an epimorphism $\nu_* \colon  F/R \epi \tilde{G}$.
The composition 
\[
\tau \circ \nu_* \colon F/R \epi \tilde{G} \epi G
\]
is nothing but the isomorphism induced by $\rho$ and so $\tau \colon \tilde{G} \to G$ is an isomorphism.   

\paragraph{Adaptation of the proof for groups of type $\FP_2$.} 
Assume $G$ is of type $\FP_2$ over the commutative ring $K$.
Then Lemma \ref{lem:Characterization-almost-fp}
allows one to find a short exact sequence 
$M \mono G_1 \stackrel{\pi}{\epi} G$  
of groups in which $G_1$ is finitely related and $K \otimes_{\Z} M_{\ab} = 0$.
Choose a finite set of generators $b_1$, \ldots, $b_f$, $t_1$ in $G_1$ in such a way 
that $\pi \colon G_1 \epi G$ maps $t_1$ onto $t$ 
and the remaining generators $b_i$  onto elements $a_i$ in the kernel of $\chi \colon G \to \R$.

Next define $F$ to be the free group on the set $\{x_1,\ldots, x_f, y\}$ 
and $\rho$ the epimorphism of $F$ onto $G_1$ 
determined by $x_1 \mapsto b_1$, \ldots, $x_f \mapsto b_f$ and $y \mapsto t_1$.
The kernel $R$ of the projection $\rho\colon F \epi G_1$ is the normal closure of a finite set of relators, 
say $r_1$, \ldots, $r_k$.
Each relator $r_j$ is a product of powers of conjugates $y^\ell x_i y^{-\ell}$;
by replacing, if need be, some of these relators by conjugates,
we can arrange that only non-negative powers $\ell$ occur in the new defining relators. 
Since their number is finite
there is a natural number $m$ such that every $r_j$ is contained in the group
\[
V = \gp(\{ y^\ell x_i y^{-\ell} \mid 0 \leq \ell \leq m +1 \text{ and }   1 \leq i \leq f\}).
\]
Define  $B \subset G$ to be the image of $V$ under the composition $\pi \circ \rho$;
so
\[
B = \gp(\{ t^\ell a_i t^{-\ell}  \mid 0 \leq \ell \leq m +1 \text{ and }   1 \leq i \leq f\}),
\]
and introduce the groups
$S = \gp(\{ t^\ell a_i t^{-\ell} \mid 0 \leq \ell \leq m\text{ and }   1 \leq i \leq f\})$ 
and $T = tSt^{-1}$. 
Clearly,  $S$ and $T$ are subgroups of $B$.

The inclusion of $B$ in $G$ and the assignment $u \mapsto t$ 
induce a homomorphism
\[
\tau \colon \tilde{G} = \langle B, u\mid u\cdot s\cdot u^{-1} = \mu(s) \text{ for all } s \in S \rangle \to G;
\]
as before, $\mu(s) = t s t^{-1}$.
This homomorphism is \emph{surjective}.
Next the assignments $x_i \mapsto a_i$, where $1 \leq i \leq f$, and $x \mapsto u$
induce a homomorphism $\nu \colon F \epi \tilde{G}$ of the free group $F$ onto the HNN-extension $\tilde{G}$.
It then follows, exactly as in the previous proof,
that  $\nu$ induces an epimorphism $\nu_* \colon G_1 = F/R \epi \tilde{G}$.
The composition 
\[
\tau \circ \nu_* \colon G_1 \epi \tilde{G} \epi G
\]
is nothing but $\pi$ and so its kernel is $M$.
This kernel $M$ is mapped by $\nu_*$ onto the kernel of $\tau$. 
As $K \otimes_{\Z} M_{\ab} = 0$, it follows that $ K \otimes_{\Z} (\ker \tau)_{\ab} = 0$.
On the other hand, 
$\tau$ maps the base group $B$ identically onto $B \subset G$ 
and thus its kernel is a free group (\cite[Corollary 1]{KaSo71}; cf. \cite[p.\;212, Corollary 2]{Coh89}.
But if so, $\ker \tau$ must be trivial,
whence $\tau \colon \tilde{G} \epi G$ \emph{is an isomorphism}.

\begin{note}
\label{note:Structure-theorem}
a) Theorem \ref{thm:FP2-groups-and-HNN-extensions} and its proof go back to \cite[Thm.\;A]{BiSt78};
the given, more detailed, version of the proof is taken from the review \cite{Str84}.
 
b) In Theorem \ref{thm:FP2-groups-and-HNN-extensions} 
one considers a finitely presented group $G$ 
that maps onto an infinite cyclic group  with kernel $N$
and deduces that $G$ can be written as an HNN-extension with finitely generated base group 
and finitely generated  associated subgroups contained in $N$.
In \cite{BaMi07},
G. Baumslag and C. Miller III generalize
this result to a theorem involving HNN-extensions with several stable letters;
their theorem reads as follows (\cite[Thm.\;1]{BaMi07}):
\emph{Let $G$ be a finitely presented group which projects onto a free group of rank $n$; 
let $N$  denote the kernel of the projection.
Then $G$ is an HNN extension with $n$ stable letters, of a finitely generated group B and finitely
generated associated subgroups, all contained in $N$.}
\end{note}
\index{Bieri, R.}
\index{Strebel, R.}
\index{Baumslag, G.}
\index{Miller III, C. F.}
%
%
\subsubsection{Reflections on the consequences of the structure theorem} 
\label{sssec:Reflections-consequences-structure-theorem}
%
In the remainder of Section \ref{sec:Sigma1-via-HNN-extensions},
four consequences of Theorem \ref{thm:FP2-groups-and-HNN-extensions} will be discussed.
The first one is straightforward: given an indicable  group $G$ of type $\FP_2$ 
one infers from the structure theorem that $G$ splits as an HNN-extensions over a finitely generated subgroup,
and uses this information to deduce further properties of $G$.
Whether this approach leads to useful insights depends, of course, 
on additional properties of the group $G$.
Some situations where this approach has lead to satisfactory results 
will be described in section \ref{sssec:Direct-applications-structure-theorem}.
\smallskip

The next two applications make use of a \emph{dichotomy in the class of HNN-extensions};
there are \emph{ascending} HNN-extensions,
here the base group coincides with one of the associated subgroups,
and \emph{non-degenerate} HNN-extensions where the base group is distinct from both associated subgroups.
This dichotomy can be used as follows.

As before, let $G$ be an indicable group of type $\FP_2$ (over some commutative non-zero ring)
and let $\chi$ be a fixed rank 1 homomorphism.
In one application, 
one assumes that $G$ cannot be written as a non-degenerate HNN-extension
\footnote{This assumption holds, for instance, if $G$ does not contain a non-abelian free subgroup.}
or, less restrictive, as a non-degenerate HNN-extension with finitely generated base group and associated subgroups, all contained in the kernel $N$ of  $\chi$.
The structure theorem then implies 
that $G$ is an ascending HNN-extension with finitely generated base group contained in $N$;
or, alternatively, that $\Sigma^1(G)$ contains (at least) one of the points $[\chi]$ or $[-\chi]$.
This conclusion is an analogue of the conclusion of Theorem \ref{thm:Consequence-finite-presentation}
applied to the character $\chi$. 

In the other application,
one assumes that $G$ is not an ascending HNN-extension 
with finitely generated base group contained  the kernel of $\chi$ or, in other words,
that $[\chi]$ and $[-\chi]$ lie both outside of $\Sigma^1(G)$
and one deduces that $G$ is a non-degenerate HNN-extension 
with finitely generated base group and associated subgroups,
all three contained in the kernel of $\chi$.
\smallskip

There is a forth application.
It is a contraposition of the structure theorem.
This time, $G$ is a finitely generated group and $\chi \colon G \to \R$ is a cleverly chosen rank 1 character.
One assumes that $G$ cannot be written as an HNN-extension 
with finitely generated base group and and finitely generated associated subgroups, 
all three contained in the kernel of $\chi$.
The structure theorem then implies that $G$ is \emph{infinitely related} 
and that it is not of type $\FP_2$, no matter how the commutative ring is chosen.
\begin{remarks}
\label{remarks:Role-of-the-rank-1-character}
a) If the sphere $S(G)$ has positive dimension,
there are infinitely many rank 1 points and hence infinitely many essentially different choices for $\chi$.
In some of the applications, one uses this infinite set of choices,
in others they are of little significance.
In the first application,
one is interested in a splitting of $G$ as an HNN-extension with finitely generated subgroups $B$, $S$ and $T$,
and the choice of $\chi$ will typically play no rôle in the subsequent analysis.
If the aim is to show that $\Sigma^1(G)$ contains at least one of the points $[\chi]$ or $-\chi]$,
one often varies $[\chi]$ so as to find many points in $\Sigma^1(G)$.
In the other two applications, a single well chosen character will lead to the conclusion 
that $G$ admits a non-degenerate HNN-decomposition and hence contains a non-abelian free subgroup,
respectively that $G$ is infinitely related.
\smallskip

b) As stated before,
the forth application is a contraposition of the structure theorem.
It has come to light recently that one can do better:
by going through the proof of the structure theorem 
one can derive additional properties of the group $G$
(see  \cite[Thm.\;6.1]{BCGS12}). 
In section  \ref{ssec:Improved-structure-theorem-for-fp-indicable-groups}
I shall briefly explain the additional insights afforded by the new approach.
\end{remarks}
%
\subsubsection{Direct applications of the structure theorem} 
\label{sssec:Direct-applications-structure-theorem}
%
Recall the succinct version of Theorem \ref{thm:FP2-groups-and-HNN-extensions}:
\emph{every indicable group of type $\FP_2$ is an HNN-extension with finitely generated base group $B$ 
and finitely generated associated groups; 
moreover, 
$B$ can be chosen to be a subgroup of the kernel of a preassigned rank 1 character $\chi$}.

In the literature,
one can find several instances where this structure theorem is applied directly.
Often, however,  it is difficult to describe the use of the theorem 
without recalling a fair number of details from the context of the application.
In what follows, I sketch two applications that need little preparations.

\paragraph{Poincaré duality groups of dimension 2.}
The first and earliest application can be found in the paper \cite{EcMu80} by B. Eckmann and H. M{\"u}ller.
The aim of this paper is to show 
that \emph{an indicable Poincaré group $G$ of dimension 2 is a surface group}.
A first problem on the path towards this goal is the fact
that the hypothesis on the homology and cohomology of $G$ implies only
that the group is of type $\FP_2$ over $\Z$ (by \cite{BiEc74a}, cf.\;\cite{Bro75} or \cite{Str76a}),
while a surface group has a one-relator presentation.
The authors proceed as follows.

Assume $G$ is an indicable  Poincaré duality group of dimension 2. 
Then $G$ is of type $\FP_2$ over $\Z$ and so it is an HNN-extension 
with finitely generated base group and associated subgroups 
(by Theorem \ref{thm:FP2-groups-and-HNN-extensions}).
As the base group $B$ has infinite index in $G$, 
it is of cohomological dimension 1  (by  \cite{Str77}) 
and hence free by Stalling's Theorem. 
If the rank of $B$ is  greater than 1
the splitting can be changed so as to become a splitting of $G$,
as a free product with amalgamation or as an HNN-extension, over a subgroup of smaller rank
(by  the \emph{decomposition theorems for group pairs} established in \cite{Mue81}).
One is thus reduced to the case 
where the amalgamated subgroup $A$  or the the associated subgroup $S$ is infinite cyclic. 
Then the group pairs $(G_1; A)$ and $(G_2; A)$ in the case of an generalized free product,
or the pair $(B; \{S,T\}))$ in the case of an HNN-extension, 
are $\PD^2$-pairs (see \cite{BiEc78}).
By \cite[Thm.\,2]{EcMu80} these  $\PD^2$-pairs are geometric. 
It then follows easily that $G= G_1\star_A G_2$ 
or $G= \langle B, t \mid tSt^{-1} = T \rangle$ is a surface group.
\index{Bieri, R.}
\index{Brown, K. S.}
\index{Eckmann, B.}
\index{Mueller, H.@Müller, H.}
\index{Stallings, J. R.}
\index{Strebel, R.}

\paragraph{Structure of finitely presented LERF groups.}
A group $G$ is said to be \emph{LERF} (short for \emph{locally extended residually finite})
or \emph{subgroup separable}
if every \emph{finitely generated} subgroup $H$ of $G$ is an intersection of finite index subgroups of $G$.
Examples of such groups are free groups (\cite[Thm.\,5.1]{Hal49}),
polycyclic groups (see \cite[5.4.16]{Rob96}) and 
fundamental groups of closed surfaces (\cite[Thm.\,3.3]{Sco79}.
\index{Hall, M.}
\index{Scott, P.}
\index{Definition of!LERF group}
\index{Definition of!subgroup separable group}
\index{LERF groups!definition}
\index{LERF groups!example}

A LERF group is, of course, residually finite, but the converse does not hold.
This fact is made clear by the main result of \cite{BlNe74}, 
asserting that
\emph{a subgroup $H$ that is an intersection of subgroups of finite index 
cannot be conjugated to a proper subgroup of itself}.
(This implies, in particular, that the Baumslag-Solitar $\BS(1,m)$ with $m>1$ are not LERF.)
\index{Blass, A.}
\index{Neumann, P. M.}

In \cite{But08a}, 
J. O. Button establishes the following property of finitely presented LERF groups:
\begin{prp}[\protect{\cite[Theorem 7.3]{But08a}}]
\label{prp:Structure-fp-indicable.LERF-group}
A finitely presented LERF group is large
\footnote{Following Gromov in \protect{\cite[p.\;82, Thm.\:(B)]{Gro82a}},
a group will be called \emph{large} if it has a subgroup of finite index
that maps onto a free subgroup of rank 2.} 
or all the kernels of its rank 1  characters are finitely generated.
\end{prp}
\index{Button, J. O.}
\index{LERF groups!structure theorem}
\index{Large groups!examples}

\begin{proof}
As the conclusion is clearly true if $G_{\ab}$ is finite,
we assume henceforth that $G$ is indicable 
and consider a rank 1 character $\chi \colon G \to \R$.
The structure theorem then shows that $G$ is an HNN-extension of the form
\[
\langle B, y \mid y s y^{-1} = \mu(s)  \text{ for all } s \in S\rangle
\]
with finitely generated base group $B$ contained in $N = \ker \chi$ 
and finitely generated associated subgroups $S$, $T = \mu(S)$.
Since $G$ is LERF 
this HNN-extension is neither properly ascending nor properly descending
(by \cite[Theorem]{BlNe74}),
and so only two possibilities remain:
either $S = T = B$, or $S \subsetneqq B$ and  $T \subsetneqq B$.
In the first case, $B$ is normal in $G$, hence coincides with $N = \ker \chi$ 
and so $N$ is finitely generated.

In the second case, 
we choose $b_0 \in B \smallsetminus S$ 
and then use the hypothesis that $G$ be LERF
to find a finite index subgroup $H$ such that $H \supset S$ and $b_0 \notin H$.
Define $L = \bigcap_{g \in G} g H g^{-1}$ to be the core of $H$ in $G$.
Then $L$ is a normal subgroup of $G$ with finite index;
moreover,   $S \cdot L \subseteq H$ and so $S \cdot L \subsetneqq B \cdot L$.
It follows that the canonical projection $\pi \colon G \epi \bar{G} = G/L$ 
maps $S$ onto a proper subgroup $\bar{S}$ of $\bar{B} = \pi(B)$.
The associated subgroup $T$ maps onto a subgroup $\bar{T}$ of $\bar{B}$;
as it is conjugated to $S$, its image $\bar{T}$ is conjugated to $\bar{S}$ and so distinct from $\bar{B}$,
for $\card(\bar{S}) < \card(\bar{B}) < \infty$. 
All taken together, this shows that  $\bar{G} = G/L$ 
is a non-degenerate HNN-extension 
\[
\langle \bar{B}, \bar{y} \mid \bar{y}\bar{S} \bar{y}^{-1} = \bar{T} \rangle
\]
with finite base group.
Such an HNN-extension is large (this follows, \eg, from \cite[II, §2, Prop.\,11]{Ser77b} 
or from \cite[Lemma 7.4]{ScWa79},
and the fact that $\bar{G}$ contains infinitely generated subgroups), 
whence $G$ is large.
\end{proof}
%
\subsection{Applications to finitely presented soluble groups}
\label{ssec:Structural-implications-for-fp-soluble-groups}
%
The next consequence of Theorem \ref{thm:FP2-groups-and-HNN-extensions} 
is the \emph{HNN-criterion for finitely presented groups $G$ containing no free subgroups of rank 2}.
 The criterion reads as follows:
\begin{prp}[HNN-criterion]
\label{prp:Soluble-FP2-groups-and-HNN-extensions}
Assume $G$ is a group of type $\FP_2$ which does not contain a non-abelian free subgroup
and whose abelianization has positive rank.
Let $\chi \colon G \to \R$ be a rank 1 character.
Then $G$ is an ascending HNN-extension with finitely generated base group $B \subseteq \ker \chi$;
put differently,  
$\Sigma^1(G)$ contains (at least) one of the points $[\chi]$ or $[-\chi]$.
\end{prp}
\index{HNN-criterion}

\begin{proof}
The hypotheses of Proposition \ref{prp:Soluble-FP2-groups-and-HNN-extensions}
include the assumptions of Theorem \ref{thm:FP2-groups-and-HNN-extensions}
and so the conclusion of the latter result shows that $G$ is an HNN-extension 
with finitely generated subgroup $B$ contained in $N = \ker \chi$ 
and finitely generated associated subgroups $S$ and $T$.
This HNN-extension cannot be non-degenerate 
for otherwise $N$ would contain non-abelian free subgroups 
by Lemma \ref{lem:Existence-free-subgroups-in-HNN-extension} below, 
in contradiction to the hypothesis of Proposition \ref{prp:Soluble-FP2-groups-and-HNN-extensions}.
So at least one of the equalities $S = B$ or $T = B$ holds,
\ie, the HNN-ascension is descending or ascending.
By Proposition \ref{prp:Ascending-HNN-extension} this finding implies  
that at least one of the points $[-\chi$ or $[\chi]$ lies in $\Sigma^1(G)$.
\end{proof}
\begin{lem}
\label{lem:Existence-free-subgroups-in-HNN-extension}
\index{Non-abelian free subgroups!existence}
Assume $G$ is an HNN-extension with stable letter $t$ 
and associated subgroups $S$, $T$ both distinct from the base group $B$. 
Then $G$ contains a non-abelian free subgroup.
\end{lem}
\begin{proof}
Since $S \subsetneqq H$ and $T \subsetneqq H$,
there exist elements $a \in H \smallsetminus S$ and $b \in H \smallsetminus T$. 
Then $x = at^{-1}$ and $y= bt$ generate a subgroup of $G$;
it  is free on $\{x,y\}$ by Britton's Lemma  (see, e.\;g., \cite[p.\;181]{LS77}).
\end{proof}
%
\subsubsection{Remark on $\Sigma^1$ and non-degenerate HNN-extensions} 
\label{sssec:Remark-Sigma1-non-degenerate-HNN-extension}
%
If one applies the proof of Theorem \ref{thm:FP2-groups-and-HNN-extensions}
to a group given by an explicit finite presentation,
one obtains finite generating systems for $B$, $S$ and $T$.
One may then be able to find out by direct  inspection
whether $S = B$ or $T = B$.
In the first case, $[-\chi] \in \Sigma^1(G)$, in the second case, one has  $[\chi] \in \Sigma^1(G)$.

Suppose now that  $S \neq B$ or $T \neq B$.
Then $N$ must be infinitely generated,
whence Corollary \ref{crl:Characterizing-fg-kernel-rank-1-character}
shows that at least one point in the antipodal pair $\{[\chi], -\chi] \}$ lies outside of $\Sigma^1(G)$.
One can do better:
\begin{prp}
\label{prp:HNN-extension-and-Sigma1}
Let $G$ be a finitely generated group that can be written as an HNN-extension
as given in equation \eqref{eq:FP2-groups-and-HNN-extensions},
the map $\mu\colon S \iso T$ being the restriction of the inner automorphism $g \mapsto tgt^{-1}$ 
induced by $t = \tau(y) \in $G.

Let $\chi \colon G \epi \Z \incl \R$ denote the character
that sends $t$ to 1 and maps the base group  $B$ to $\{0\}$.
If $B$ is finitely generated, the following implication holds:
\begin{equation}
\label{eq:chi-in-Sigma1-implies-ascending-HNN-extension}
T \neq B  \Longrightarrow [\chi] \notin \Sigma^1(G)
\end{equation}
\end{prp}
\index{Computation of Sigma1@Computation of $\Sigma^1$ for!HNN-extensions}
\index{Invariant Sigma1 and@Invariant $\Sigma^1$ and!non-degenerate HNN-extension}

\begin{proof}
Let $\eta \colon \ZZ \to B$ be a finite generating system of the base group $B$ 
and set $\XX = \ZZ \cup \{t\}$.
Then $\XX$ generates $G$ and so the Cayley graph $\Gamma = \Gamma(G,\XX)$ is connected.
Consider now a path $p$ with origin 1 that runs in the subgraph $\Gamma_{\chi} = \Gamma(G, \XX)_\chi$.
This path can be written as $p = (1, w)$ 
where $w$ is a word of the form
\[
t^{e_{0}} w_{1}t^{e_{1}} w_{2}t^{e_{1}} \cdots t^{e_{k - 1}} w_{k} t^{e_k}
\]
with each subword $w_{j}$ a word in $\ZZ^\pm$,
each exponent $e_j$ in $\{1,-1\}$,
and each partial sum $e_{0} + \cdots +e_{j}$ non-negative.
It $w$ contains a subword of the form $t w_{j}t^{-1}$ with $b = \eta_{*}(w_{j}) \in S$,
we replace the subword by a word $w'_{j}$ in $\ZZ^\pm$ that represents $\mu(h)$.
The resulting path $w'$ will then again run in the subgraph $\Gamma_{\chi}$.
If $w$ contains a subword of the form $t^{-1} w_{j}t$ with $\eta_{*}(w_{j}) \in T$,
one proceeds similarly.
There exists therefore a word $w'$ with $\eta_{*}(w') = \eta_{*}(w)$ 
such that $p' = (1,w')$ runs in $\Gamma_{\chi}$ 
and that the  associated sequence 
\begin{equation}
\label{eq:Reduced-sequence-HNN-extension}
(t^{e_{0}} ,b_{1}, t^{e_{1}}, b_{2}, t^{e_{2}},  \ldots t^{e_{\ell -1}},  b_{\ell}, t^{e_{\ell}} )
\quad \text{with}\quad  b_{j} = \eta_{*}(w_{j})
\end{equation}
is \emph{reduced},
i.\,e.,  it contains neither a subsequence $(t,b_{j},t^{-1})$ with $b_{j} \in S$ 
nor  a subsequence $(t^{-1},b_{j},t)$ with $b_{j} \in T$.

We are now ready to show that $\Gamma_\chi$ is not connected,
whence $[\chi] \notin \Sigma^1(G)$ by Theorem \ref{thm:Sigma1-well-defined}.
By assumption, there exists $b_{0} \in B \smallsetminus T$.
The product $g_{0} = t^{-1}b_{0}t$ is a vertex  of $\Gamma_{\chi}$;
if it could be connected to 1 inside $\Gamma_\chi$,
the above reasoning would provide us with a reduced sequence of the form 
\eqref{eq:Reduced-sequence-HNN-extension} with product equal to $g_{0}$.
As this reduced sequence would correspond to a path in $\Gamma_{\chi}$,
one would have $\ell > 0$ and $e_{0} \geq 0$ 
and so the sequence
\[
(t^{-1},b_{0}^{-1},t^{1+e_{0}},b_{1}, t^{e_{1}},  \ldots t^{e_{\ell -1}},  b_{\ell}, t^{e_{\ell}} )
\]
would be reduced. 
As the product of this sequence would be the unit element,
it would contradict Britton's lemma (see, e.\,g., \cite[p. 181]{LS77}).
\end{proof}
\begin{note}
\label{note:prp-HNN-extension-and-Sigma1}
Proposition \ref{prp:HNN-extension-and-Sigma1} goes back to
\cite[Prop.\;4.4]{BNS}.
It will be generalized by part (ii) of Proposition \ref{prp:Point-outside-Sigma1-HNN-extensions};
the generalization will reveal that the hypothesis that $B$ be finitely generated is redundant.
\end{note}
%
\subsubsection{Application to finitely presented nilpotent-by-cyclic groups} 
\label{sssec:Consequences-fp-nilpotent-cyclic}
%
We continue with applications of Proposition \ref{prp:Soluble-FP2-groups-and-HNN-extensions}
to some selected classes of finitely related soluble groups.
We start out with nilpotent-by-cyclic groups.

The class of \emph{finitely generated} nilpotent-by-cyclic groups contains the centre-by-metabelian 3-generator subgroup
\begin{equation}
\label{eq;Fg-centre-by-metabelian}
G = \left\langle
\begin{pmatrix} 1&1&0\\1&0&0\\0&0&1\end{pmatrix},
\begin{pmatrix} 1&0&0\\0&X&0\\0&0&1\end{pmatrix},
\begin{pmatrix} 1&0&0\\0&1&1\\0&0&1\end{pmatrix}
\right\rangle
\end{equation}
of $\GL(3, \Q(X))$ with $X$ an indeterminate,
and the class contains all homomorphic images of $G$.
It follows that the groups in the class are in general neither residually finite, 
nor do they satisfy max-n, the maximal condition on normal subgroups.
These facts are already pointed out in \cite{Hal54}.

The situation is entirely different for \emph{finitely related} nilpotent-by-cyclic groups:
\begin{crl}
\label{crl:Fp-nilpotent-by-cyclic-group}
Assume $G$ is a finitely generated  extension of a nilpotent group $N$ by a cyclic group $G/N$.
Then $G$ is finitely related if, and only if, it is polycyclic 
or else an ascending HNN-extension whose base group is finitely generated and nilpotent.
\end{crl}
\index{Finitely presented groups!nilpotent-by-cyclic}
\index{Groups!fp nilpotent by cyclic}
\index{Characterization of!fp nilpotent-by-cyclic groups}
\begin{proof}
Assume $G$ is finitely related 
and let $N \triangleleft G$ be a nilpotent normal subgroup with cyclic quotient $G/N$.
If $G/N$ is finite, then $G$ is  polycyclic. 
Otherwise, $G$ is by the HNN-criterion  an ascending HNN-extension 
$\langle B, t \mid t^{-1}B t \subseteq B \rangle$
whose base group $B$ is finitely generated and nilpotent. 
The converse claim is clearly true.
\end{proof}
\begin{remarks}
\label{remarks:Fp-nilpotent-by-cyclic}
a) Corollary \ref{crl:Fp-nilpotent-by-cyclic-group} is result 11.4.5 in \cite{LeRo04}
and goes back to Theorem C in \cite{BiSt78}.
It shows, first of all, that the class of finitely presented nilpotent by infinite cyclic groups 
admits of an explicit description, 
in contrast to the class of finitely generated nilpotent by infinite cyclic groups.

Moreover, the corollary implies 
that a finitely presented nilpotent-by-infinite cyclic group satisfies max-n
(\cite[Cor.\,5]{BaBi76}, cf.\,\cite[11.4.4 and 11.2.6]{LeRo04}), 
is residually finite (\cite[11.4.4 and 11.2.4]{LeRo04}), cf.\,\cite[Cor.\,7]{BaBi76}), 
that the torsion subgroup of its nilpotent normal subgroup $N$ is finite 
and the group has finite Hirsch length and that it is coherent
\footnote{A group is called coherent if each of its finitely generated subgroups is finitely related.}
(\cite[Thm.\,B]{BiSt79b} or \cite{Gro78c}, cf.\,\cite[11.4.4]{LeRo04}). 
\index{Bieri, R.}
\index{Strebel, R.}
\index{Groves, J. R. J.}
\index{Lennox, J. C.}
\index{Robinson, D. J. S.}
\smallskip

b) Suppose $G$ is a finitely related nilpotent by infinite cyclic group.
Then Corollary \ref{crl:Fp-nilpotent-by-cyclic-group} describes the structure of $G$ 
and it implies that $G$ satisfies a number of useful properties,
but the corollary does not provide a classification of the groups in question.
One reason is this:
given a finitely generated nilpotent group,
one would like to know all ascending HNN-extensions having $B$ as base group.
This presupposes, in particular, that one can describe the injective endomorphisms of $B$.
Some information on this problem can be found on page 261 of \cite{LeRo04}.
\end{remarks}

\begin{example}
\label{example:Locally-infinite-cyclic-by-infinite-cyclic-2}
Let $p$ and $q$ be positive and relatively prime integers 
and let $G_{p,q}$ denote the semi-direct product $\Z[1/(p \cdot q)] \rtimes \gp(t)$
with $t$ is an element of infinite order 
that acts on $A = \Z[1/(m \cdot n)]$ by multiplication by $m/n$.
(See Example  \ref{example:Locally-infinite-cyclic-by-infinite-cyclic} for more details.)
A finitely generated subgroup of $A$ is cyclic 
and so $G$ is an ascending HNN-extension with finitely generated base group if, and only if, $m = 1$ or $n= 1$.
Corollary \ref{crl:Fp-nilpotent-by-cyclic-group} thus implies 
that $G$ is finitely related precisely if it is one of the Baumslag-Solitar groups $\BS(m,1)$ or $\BS(1,n)$. 
\end{example}
\index{Groups!of Baumslag-Solitar}

\begin{note}
\label{note:Locally-infinite-cyclic-by-infinite-cyclic-2}
At the beginning of the 1970s both R. Bieri and J. Conway stumbled over the question
whether the group $G_{2,3}$ is infinitely related. G. Higman answered John's question in the affirmative,
as did G. Baumslag in 1973 when asked by Robert (cf.\,\cite[(1.11)]{BaSt76}).
The determination of the finitely related specimens among the groups $G_{p.q}$ 
predates \cite{BiSt78} and is contained in \cite{BaSt76} (use Lemma C on p.\,57).
The fact that $G_{p,q}$ is finitely related if, and only if, it is of type $\FP_2$ follows from \cite{BiSt78}
and is used in the classification of the soluble groups of cohomological dimension 2
published by D. Gildenhuys in  \cite{Gil79}.
\end{note}
\index{Bieri, R.}
\index{Baumslag, G.}
\index{Gildenhuys, D.}
\index{Strebel, R.}

\begin{example}
\label{example:Abelian-by-infinite-cyclic}
Let $G$ be a \emph{finitely generated, abelian by infinite cyclic group}.
Then $G$ is, of course, metabelian and a semi-direct product $A \rtimes C$ of an abelian group $A$ 
by an infinite  cyclic group $C = \gp(p)$. 
The wreath products $(\Z/k\Z) \wr C$ and $\Z \wr C$ make it clear
that the torsion subgroup $\tor(A)$ of $A$ can be infinite
and that $\bar{A} = A /\tor(A)$ can have infinite torsion-free rank.
Conjugation by the elements of $C$ turns $A$ into a $\Z{C}$-module
and as such $A$ is a finitely presentable $\Z{C}$-module;
in other words, $A$ has a presentation of the form
\[
\Z{C}^m \xrightarrow{\partial} \Z{C}^n \xrightarrow{\pi} A \to 0
\]
where $\partial$ is given by multiplication on the right 
by an $m \times n$ matrix with entries in $\Z{C}$.
(Special presentations of this kind have been studied in classical \emph{Knot Theory};
see, \eg, \cite{Cro63} and \cite{Tro74}).
\index{Crowell, R. H.}
\index{Trotter, H. F.}

Suppose now that the group $G = A \rtimes C$ is \emph{finitely related}.
Then the torsion subgroup of $A$ is finite 
and $ \bar{A} =A/\tor(A)$ is a torsion-free group of finite torsion-free rank $d$, say.
Moreover, for a suitable generator $t \in C$,
the $\Z{C}$-module $\bar{A}$ contains a free-abelian subgroup $\bar{B}$ of rank $d$ 
with $t.\bar{B} \subseteq \bar{B}$
and $\bar{A} = \bigcup_{j < 0} t^j.\bar{B}$.

It follows that 
each torsion-free, finitely presented abelian by infinite cyclic group
is given by an integer square matrix $M$ with non-zero determinant.
\end{example}
\index{Groups!fp abelian by infinite cyclic}
\index{Characterization of!fp abelian by infinite cyclic groups}

\begin{note}
\label{note:Abelian-by-infinite-cyclic}
The fact that a finitely presented abelian by infinite cyclic group $G$
can be described by an integer square matrix $M$ 
is the starting point of several papers dealing with the large-scale geometry of finitely presented abelian-by-cyclic groups that are \emph{not} polycyclic 
(see, \eg, \cite{FaMo98}, \cite{FaMo99a}, \cite{FaMo00}).
Similarly, the characterization of finitely presented nilpotent-by-cyclic groups,
afforded by Corollary \ref{crl:Fp-nilpotent-by-cyclic-group}, is used in \cite{Ahl05}.
\end{note}
\index{Farb, B.}
\index{Mosher, L.}
\index{Ahlin, A. Reiter}
%
\subsubsection{Application to coherent soluble groups} 
\label{sssec:Application-coherent-soluble-groups}
%
A group is termed \emph{coherent} if every finitely generated subgroup is finitely presented. 
Obvious examples of coherent groups are free groups, 
locally finite groups and polycyclic-by-finite groups. 
On the other hand,
finitely presented metabelian groups are typically not coherent.
This statement can be illustrated by the following simple example.
\index{Coherent group!definition}
\index{Definition of!coherent group} 
\begin{example}
\label{example:Non-coherent-metabelian-group}
Let $Q$ a free-abelian group of rank 2, generated by $s$, $t$, 
and let $p$ and $q$ be relatively prime positive integers. 
Set  $A =  \Z[1/(p \cdot q)]$
and turn $A$ into a $\Z{Q}$-module by declaring 
that $s$ act by multiplication by $p$ and $t$ by multiplication by $q$.
Finally, put $G = G_{p,q} = A \rtimes Q$.
Then $G$ is a finitely related, metabelian group with presentation
$
\langle a , s, t \mid sas^{-1} = a^p, tat^{-1} = a^q \text{ and  } st = ts \rangle.
$

Every couple $(m, n)  \in \Z^2$ gives then rise to a 2-generator subgroup
\begin{equation}
\label{eq:Non-coherent-metabelian-group}
H_{m,n}  = \gp(a = (1_A, 1_Q), u_{m,n} = s^mt^n).
\end{equation}
Here $1_A$ denotes the unit element of the ring $A =  \Z[1/(p \cdot q)]$
and $1_Q$ the neutral element of the group $Q$.
The subgroup $H_{m,n}$ is finitely presented whenever $m \cdot n \geq 0$,
but it is infinitely related if $m \cdot n$ is negative 
(see Example \ref{example:Locally-infinite-cyclic-by-infinite-cyclic-2}).
\end{example}

There exists a very satisfying characterization of finitely generated coherent soluble groups.
It is a consequence of Proposition \ref{prp:Soluble-FP2-groups-and-HNN-extensions}
and some powerful results about the structure of finitely generated soluble groups.
\begin{thm}[\protect{\cite{Gro78a}, \cite{BiSt79b}}] 
\label{thm:Coherent-soluble-groups}
A finitely presented soluble group is coherent if, and only, 
if it is polycyclic or an ascending HNN-extension with polycyclic base group.
\end{thm}
\index{Groves, J. R. J.}
\index{Bieri, R.}
\index{Strebel, R.}
\index{Characterization of!coherent soluble groups}

\begin{sketch-proof}
By \cite{Hal54}, polycyclic groups are finitely related.
As subgroups of polycyclic groups are again so,
polycyclic groups are coherent.
Assume now $G$ is an ascending HNN-extension 
$\langle B, t \mid t^{-1}B t \subseteq B \rangle$
with polycyclic base group $B$ and stable letter $t$.
The normal closure $N$ of $B$ in $G$ coincides then with the ascending union 
$\bigcup \{ t^j B t^{-j}\mid j \in \N \}$
\smallskip

Let $G_1 \subseteq G$ be a finitely generated subgroup.
Two cases arise:
if $G_1 \subseteq N$ then $G_1$ is contained in a conjugate of $B$, 
hence polycyclic and therefore finitely related.
Otherwise, set $N_1 = G_1 \cap N$.
Then $G_1/N_1 \iso (G_1 \cdot) N / N \subseteq G/N$ is infinite cyclic, 
say generated by $t_1 = x \cdot t^k$ with $x \in N$.
As $N$ is an ascending union of copies of $B$ 
there exists $\ell \in \N$ with $x  \in t^\ell Ht^{-\ell}$.
Since $G$ is also an ascending HNN-extension with base group $t^\ell Ht^{-\ell}$;
there is no harm in assuming that $x \in B$.
But if so, we may as well set $x = 1$.
Set $B_1 = B \cap H$ and $N_1 = N \cap H$.
Then $B_1$ is polycyclic and
\[
t_1^{-1}B_1 t_1  = t^{-k}(B \cap G_1)t^k =  t^{-k}B t^k  \cap t^{-k} G_1t^k  \subseteq B \cap G_1 = B_1.
\]
Moreover, 
\begin{align*}
N_1 &= \bigcup\nolimits_{j \in \N} (t^k)^j B (t^k)^{-j } \cap G_1 
= 
\bigcup\nolimits_{j \in \N} \left(t_1^j B t_1^{-j } \cap  t_1^j G_1 t_1^{-j }\right)=  
 \bigcup\nolimits_{j \in \N} t_1^j B_1 t_1^{-j }.
\end{align*}
These calculations show 
that $G_1$ is an ascending HNN-extension with polycyclic base group $B_1$ 
and hence finitely related.

So far we know that polycyclic groups 
and ascending HNN-extensions with polycyclic base groups are coherent.
The proof of the converse statement is more involved.
One has to show, in particular,
that a finitely generated coherent soluble group is nilpotent-by-abelian-by-finite
(as is every polycyclic group).
For details, see pages 236--238  in  \cite{BiSt79b} 
or pages 259--261 in  \cite{LeRo04}.
\end{sketch-proof}
\index{Robinson, D. J. S.}
\index{Lennox, J. C.}

\begin{remark}
\label{remark:Coherent-soluble-groups}
The statement of Theorem \ref{thm:Coherent-soluble-groups}
remains valid if the requirement that $G$ \emph{be coherent} 
is replaced by the condition that \emph{every finitely generated subgroup of $G$ be of type $\FP_2$};
see \cite[Proposition]{BiSt79b}.
\end{remark}
%
\subsubsection{Application to finitely presented centre-by-metabelian groups} 
\label{sssec:Application-centre-by-metabelian}
%
The group described in equation \eqref{eq;Fg-centre-by-metabelian}
is a centre-by-metabelian group whose centre is free abelian of rank $\aleph_0$. 
This example permits one to see 
that the centre of a \emph{finitely generated} centre-by-metabelian
group can be any non-trivial countable abelian group. 
The situation is, however, entirely different for \emph{finitely presented} centre-by-metabelian groups.
Indeed,
the following important result, published in 1978 by J. R. J. Groves (see \cite[Thm.\;1]{Gro78b}), holds:
\begin{thm}
\label{thm:Fp-centre-by-metabelian-groups}
A finitely presented centre-by-metabelian group is abelian-by-poly\-cyclic.
In particular, 
it has the maximal condition on normal subgroups and is residually finite.
\end{thm}
\index{Groves, J. R. J.}
\index{Finitely presented groups!centre-by-metabelian}
\begin{sketch-proof}
Let $G$ be a finitely presented centre-by-metabelian group.
Then the derived group $G'$ is nilpotent of class at most 2;
let $Z =\zeta(G')$ be the centre of $G'$ and set $M = G'/Z$.
Suppose we have established that $M$ is a finitely generated (abelian) group.
Then $G$ is an extension of an abelian normal subgroup $Z$ 
by the polycyclic quotient group $G/Z$; 
it satisfies therefore the maximal condition on normal subgroups 
(by \cite[Thm.\.3]{Hal54}, cf. \cite[4.2.2]{LeRo04}) 
and is residually finite 
(by \cite[Thm.\,3]{Jat74} or \cite{Ros76a},  cf.\,\cite[7.2.1]{LeRo04}).

Now to the proof that $M = G'/Z$ is finitely generated.
The group $A$ has the structure of a $\Z{G_{\ab}}$-module;
as such it is finitely generated (for $G_{\ab}$ is finitely related).
To pin down its structure, 
one has to exploit the hypothesis that $G$ be finitely related.
The idea is to translate the conditions 
obtained by applying the conclusion of Proposition \ref{prp:Soluble-FP2-groups-and-HNN-extensions}
to all rank 1 characters  $\chi \colon G \to  \R$ 
into a statement about the $\Z{G_{\ab}}$-module $M$.
In the account given in section 5.4 of \cite{Str84}, 
the translated statement is expressed in terms of the invariant $\Sigma_M(G_{\ab}) = \Sigma^0(G_{\ab};M)$.
As this invariant will only be discussed in Section D1,
we postpone the proof 
that $M$ is finitely generated \emph{as an abelian group} to this section.
\end{sketch-proof}

\begin{note}
\label{note:Fp-centre-metabelian-group}
Historically speaking,
Theorem \ref{thm:Fp-centre-by-metabelian-groups}
 is the first result where a feature of Proposition \ref{prp:Soluble-FP2-groups-and-HNN-extensions} is exploited 
that is not spelled out explicitly in its statement:
if the abelianisation of a finitely presented soluble group $G$ is greater than 1,
there are infinitely many pairs $\{\chi\colon G \epi \Z \incl \R, -\chi \}$
of rank 1 characters to which Proposition \ref{prp:Soluble-FP2-groups-and-HNN-extensions} applies.
\end{note}
%
\subsection{Finding non-abelian free subgroups in fp groups}
\label{ssec:Fp-groups-groups-with-free-subgroups}
%
We come now to a third consequence of Theorem \ref{thm:FP2-groups-and-HNN-extensions}.
Recall that this theorem asserts, loosely speaking, that 
\emph{every indicable group of type $\FP_2$} is an HNN-extension with finitely generated base group $B$ 
and finitely generated associated groups.
In section \ref{ssec:Structural-implications-for-fp-soluble-groups}, 
the HNN-extension has been forced to be ascending by imposing the condition 
that $G$ contain no non-abelian free subgroups.
In this section, we require the HNN-extension to be non-degenerate 
and conclude that $G$ must contain non-abelian free subgroups.

Here is the consequence of Theorem \ref{thm:FP2-groups-and-HNN-extensions}
that will be the basis of our discussion.
\begin{prp}
\label{prp:Fp-groups-having-non-abelian-free-subgroups}
Suppose $G$ is a finitely presented group
which admits a rank 1  character $\chi \colon G \to \R$ 
so that $G$ is not an ascending HNN-extension with finitely generated base group $B$ contained in $\ker \chi$.
Then $G$ contains non-abelian free subgroups.
\end{prp}

We shall give several applications of the preceding proposition.
Each of them presupposes 
that one knows a rank 1 character $\chi \colon G \to \R$ such 
that $G$ is not an ascending HNN-extension with finitely generated base group contained in $\ker \chi$
or, in terms of the invariant $\Sigma^1$, 
such that
\begin{equation}
\label{eq:Excluding-ascending-HNN-extension}
\{[\chi], [-\chi] \} \subseteq \Sigma^1(G)^c.
\end{equation}
We shall find characters satisfying condition \eqref{eq:Excluding-ascending-HNN-extension}
with the help of the following simple, but surprisingly widely applicable 
\begin{lem}
\label{lem:Infinite-locally-finite-group}
Let $G$ be a finitely generated group 
and $\chi \colon G \to \R$ a rank 1 character with kernel $N$.
Assume $G$ contains a normal subgroup  $M \subset N$ 
such that $N/M$ is an infinite, locally finite group.
Then $\{[\chi], [-\chi] \} \subseteq \Sigma^1(G)^c$.
\end{lem}

\begin{proof}
Let $t \in G_\chi \smallsetminus N$ be an element that generates a complement of $N$ in $G$
and consider a finitely generated subgroup $B \subseteq N$ with $B \subseteq tBt^{-1}$.
The  image $\bar{B}$ of $B$  under the canonical projection $\pi \colon G \epi G/M$ is then finite 
and hence $\bar{B} \subseteq \pi(t) \bar{B} \pi(t)^{-1}$ implies that $\bar{B} = \pi(t) \bar{B} \pi(t)^{-1}$.
It follows that $\bigcup\{t^\ell B t^{-\ell} \mid \ell \in \N\}$ maps onto $\bar{B}$.
But, by hypothesis,  $\bar{B} \neq \bar{N}$ and so $\bigcup\{t^\ell B t^{-\ell} \mid \ell \in \N\} < N$.

The preceding argument shows that $G$ is not an ascending HNN-extension
with finitely generated base group contained in $N$ and stable letter $t$.
Similarly, one sees that $G$ is not an ascending HNN-extension extension 
with finitely generated base group contained in  $N$ and stable letter $t^{-1}$.
The stated conclusion now follows from Proposition \ref{prp:Ascending-HNN-extension}.
\end{proof}
%
\subsubsection{Application to groups with non-negative deficiency} 
\label{sssec:Application-groups-non-negative-deficiency}
%
The preceding lemma and Proposition \ref{prp:Fp-groups-having-non-abelian-free-subgroups}
allow one to show that a finitely presented group $G$ with $m \geq 2$ generators and $n \leq m-2$ relators 
contains a non-abelian free subgroup. 
If $n = 1$, \ie, if $G$ is a one-relator group,
this fact is an obvious consequence of the \emph{Freiheitssatz} (see, e.\,g., \cite[p. 198]{LS77}).
The lemma yields actually a further result: it implies that the invariant $\Sigma^1(G)$ is empty
whenever $G$ is a finitely presented group with deficiency greater than 1.
In what follows, we first establish this second result 
and deduce then the first claim from Proposition \ref{prp:Fp-groups-having-non-abelian-free-subgroups}.

We begin by recalling the notion of deficiency of a finite presentation.
\begin{definition}
\label{definition:Deficiency}
Let $\PP = \langle \XX \mid \RR \rangle$ be a finite presentation of a finitely presentable group $G$.
The difference $\card(\XX) - \card (\RR)$ is called  the \emph{deficiency of} $\PP$;
it will be noted $\defi \PP$.
The supremum
\begin{equation}
\label{eq:Definition-deficiency-group}
\defi G = \sup\{\defi \PP \mid \PP \emph{ is a finite presentation of } G \}
\end{equation}
is called the \emph{deficiency} of $G$ and denoted by $\defi G $.
\end{definition}
\index{Definition of!deficiency of a group}%
\index{Deficiency of a group!definition}
\index{Notation!def group@$\defi(G)$}%

The supremum in equation \ref{eq:Definition-deficiency-group} is always a maximum.
Indeed, a finite presentation  $\PP$ of the group $G$ 
gives rise to a presentation of the abelianized group $G_{\ab}$;
therefore $G_{\ab}$ is a quotient of a free abelian group of rank $\card(\XX)$
modulo a subgroup generated by $\card (\RR)$ elements.
The torsion-free rank $r_{0} (G_{\ab}) = \dim_{\Q}(G_{\ab} \otimes \Q)$  of $G_{\ab}$
is thus an upper bound of $\defi \PP$.

The deficiency $\defi G$ is bounded by $r_{0} (G_{\ab})$, 
but this bound need not be optimal.
An upper bound that is often smaller is 
\begin{equation}
\label{eq:Improved-upper-bound-def}
\defi G \leq r_{0}(G_{\ab}) - d(\Cohom_{2}(G,\Z)).
\end{equation}
In this formula $d(A)$ denotes the minimal number of generators of the group $A$.
(See, e.\;g., \cite[p. 419]{Rob96} for a proof of inequality \eqref{eq:Improved-upper-bound-def}.)
\smallskip
\index{Deficiency of a group!upper bound using homology}

We come now to the announced consequence of Lemma \ref{lem:Infinite-locally-finite-group},
actually to a generalization. It asserts, roughly speaking, 
that the invariant $\Sigma^1$ of the metabelian top of a finitely related group with few relators is empty.
The precise version of the generalization  is spelled out by the next proposition.
\begin{prp}
\label{prp:Invariant-group-large-deficiency}
Let $G$ be a finitely related group 
that admits a presentation with $m \geq 2$  generators and $n$ relators.
Fix a prime number $p$.
\begin{enumerate}[(i)]
\item If $n \leq m-2$, 
or
\item if $n = m-1$ and one relator is a proper power, say $r = w^k$, and $p$ divides $k$, 
or
\item if $n = m$ and two relators are proper powers, 
say $r_1 = w_1^{k_1}$ and $r_2 = w_2^{k_2}$,
and if $p$ divides both $k_1$ and $k_2$,
\end{enumerate}
then $\Sigma^1(Q) $ is empty for every quotient group $Q = G/M$ with $M \subseteq G'' \cdot (G')^p$.

Moreover, if $S(G)$ is non-empty $G$ maps onto the wreath product $(\Z/p\Z) \wr C_\infty$ of a cyclic group of order $p$ by an infinite cyclic group.
\end{prp}
\index{Computation of Sigma1@Computation of $\Sigma^1$ for!groups with large deficiency}

\begin{proof}
The proof divides into three parts.
We first verify that $\Sigma^1(G)$ contains no point of rank 1
and then deduce that the same is true for $\Sigma^1(Q)$.
In the final part we invoke the openness of $\Sigma^1(Q) \subseteq S(Q)$ and deduce
that $\Sigma^1(Q)$ is empty.
 
Let  $\chi \colon G \epi \Z \incl \R$ be a rank 1 character of $G$,
set $N = \ker \chi$ and pick $t \in G$ with $\chi(t) = 1$. 
Given a prime number $p$, 
consider the elementary abelian $p$-group $\bar{N} = N/(N' \cdot  N^p)$, 
viewed as a left $\F_p(\gp(t))$-module via conjugation. 
Next, let $\eta \colon F  \epi  G$ be the epimorphism 
involved in the presentation with $m$ generators and $n$ relators.
Set $R =  \ker \eta$ and $U = \eta^{-1}(N)$. 
Then $\eta$ gives rise to an extension
\[
\bar{R} = (R \cdot U' \cdot U^p)/(U' \cdot U^p) \quad \mono \quad
\bar{U} = U/(U' \cdot U^p) \quad \stackrel{\eta_*}{\epi} \quad  \bar{N}
\]
 of $\F_p(\gp(t))$-modules.
 Its middle term $\bar{U}$ is a free $\F_p(\gp(t))$-module of rank $m-1$.
 
Indeed,
the matrix group $\GL(m, \Z)$ is generated by elementary matrices
and each automorphism of $F_{\ab} \iso \Z^m$ 
that is induced by an elementary matrix 
lifts to an automorphism of $F$ (cf. the proof of Theorem 3.5 in \cite[Chapt.\;3]{MKS}).
The free group $F$ admits therefore a basis $\{x_1, \ldots , x_{m-1}, x_m\}$  
such that $U$ is the normal closure of the the basis elements $x_1$, \ldots, $x_{m-1}$  
and that $\eta(x_m)$ equals $t$. 
So $\bar{U}$ is a free $\F_p(\gp(t))$-module of rank $m-1$, as claimed.

We assert that the  $\F_p(\gp(t))$-module $\bar{R}$ can be generated by $m-2$ elements.
This is clear in case (i).
In case (ii) it holds
because a relator of the form $r = w^k$, with $p$ dividing $k$, maps to  $\bar{1} \in \bar{R}$.
Indeed, 
since $r = w^k \in R \subset U$,
the root $w$ represents an element  $\bar{w} \in  F/U$ of finite order;
as $F/U \iso \Z$ is torsion-free, 
$w$ belongs therefore  to $U$ and so $w^p$ maps to $ 1 \in \bar{U} = U/(U' \cdot U^p)$.

So the $\F_p(\gp(t))$-module $\bar{R}$ is again generated by $m-2$ elements. 
In case (iii), one sees similarly that $\bar{R}$ is generated by $m-2$ elements. 
The module $\bar{N}$ is therefore isomorphic to the quotient of a free $\F_{p}(\gp(t))$-module of rank $m-1$ 
by a submodule generated by $m-2$ elements,
thus $\bar{N}$ is an \emph{infinite} elementary-abelian $p$-group
and so the character $\chi$ represents a point outside of $\Sigma^1(G)$
by Lemma \ref{lem:Infinite-locally-finite-group}.
Moreover, since $\F_{p}(\gp(t))$ is a principal ideal domain, 
the module $\bar{N}$ maps actually onto a copy of the free cyclic module $\F_{p}(\gp(t))$,
whence the wreath product $(\Z/p\Z) \wr C_\infty$ is a quotient of the group $G$.
Note, though, that in case (iii) the abelianization of $G$ can be finite; 
so one needs the extra hypothesis to justify the addendum.
\smallskip

So far we know that $\Sigma^1(G)$  contains no rank 1 point.
Consider now a normal subgroup $M \subseteq G'' \cdot (G')^p$
and let $\pi \colon G \epi Q = G/M$ be the canonical projection.
This projection induces an isomorphism of spheres $\pi^* \colon S(Q) \iso S(G)$.
Under this isomorphism,
a  character $\bar{\chi} \colon Q \to \R$ corresponds to the character $\chi = \bar{\chi} \circ \pi$ of $G$.
Now $\ker \bar{\chi} = (\ker \chi) / M = N/M$ and $M \subseteq G'' \cdot (G')^p \subseteq N' \cdot N^p$.
So $\ker \bar{\chi}$ maps onto the infinite elementary $p$-group $N/(N'\cdot N^p)$
whence $[\bar{\chi}] \notin \Sigma^1(Q)$,
again by Lemma \ref{lem:Infinite-locally-finite-group}.

The preceding paragraph shows that $\Sigma^1(Q)$ contains no rank 1 points. 
By Lemma \ref{lem:Density-rank-1-points} below these points are dense in $S(Q)$.
As $\Sigma^1(G)$  is an open subset of $S(G)$ (by Theorem  \ref{thm:Openness-Sigma-1}),
the subset $\Sigma^1(Q)$ is therefore empty. 
\end{proof}
In the above proof the lemma given next has been quoted.
\begin{lem}
\label{lem:Density-rank-1-points}
\index{Character sphere!density of rank 1 points}%
The rank 1 points constitute a dense subset in the sphere $S(G)$ of a finitely generated group $G$.
\end{lem}
\index{Density of rank 1 points in $S(G)$}
\begin{proof}
In view of the coordinate isomorphism  \eqref{eq:Definition-sigma-theta}
it suffices to prove the claim for the unit sphere  $\s^{n-1}$.
Given $u \in \s^{n-1}$ and $\varepsilon > 0$,
there is a rational point $x \in \Q^n$ with $\norm{x} \geq 1$ and $\norm{u - x} \leq \varepsilon$. 
Then the distance from $u$ to  $x/\norm{x}$ is at most $\varepsilon$;
indeed, the triangle with vertices $u$, $x/\norm{x}$ and the origin $o$ is isosceles
and so the foot $\hat{u}$ of the perpendicular from $u$ to the line through $o$ and $x$ 
lies on the segment $[0,x/\norm{x}\,]$,
whence $ d_2(u ,x/\norm{x}) \leq d_2(u, x) \leq \varepsilon$.
\end{proof}
\begin{remarks}
\label{remarks:Invariant-groups-deficiency-bigger-1}
a) Proposition \ref{prp:Invariant-group-large-deficiency}
is a refinement of  \cite[Theorem 7.2]{BNS} which, in turn, generalizes the main proposition in \cite{BiSt81b}.
In both results, one views $\bar{R}$ as a module over the principal ideal domain $\F_p (\gp(t))$,
an idea that can already be found in Section 2 of Baumslag's paper \cite{Bau76}.
\index{Baumslag, G.}
\index{Bieri, R.}
\index{Neumann, W. D.}
\index{Strebel, R.}

b) The proposition implies 
that \emph{every normal subgroup $N$ with $G/N$ infinite cyclic is infinitely generated
whenever $G$ is finitely presented with deficiency greater than 1}.
This conclusion holds in far greater generality 
but the proof of this generalization relies on two results from the theory of $L_2$-Betti numbers:
firstly,  
the $L^2$-Betti number $\beta_1(G)$ is 0
whenever $G$ is a finitely presented group 
that admits a finitely generated normal subgroup $N$ with an infinite quotient group $G/N$
(see \cite[Thm.\;6.8]{Gab02}).
Secondly,
the $L^2$-Betti number $\beta_1(G)$ and the deficiency of a finitely presented group $G$ 
are related by the inequality $\defi(G) \leq 1 + \beta_1(G)$ (see, \eg, \cite[Thm.\;2]{Hil97}).
It follows that \emph{a finitely presented group with $\defi(G) > 1$ 
does not contain a finitely generated normal subgroup $N \neq \{1\}$ with infinite quotient $G/N$}.
\end{remarks}
\index{Gaboriau, D.}
\index{Hillman, J. A.}
%

We continue with some examples illustrating the use of Proposition \ref{prp:Invariant-group-large-deficiency}.
\begin{examples}
\label{examples-Groups-deficiency-greater-1}

a) \emph{One relator groups $G$ with at least 3 generators.}
Then the deficiency is at least 2 and so the invariant $\Sigma^1(G/M)$ is empty 
for every quotient group $Q = G/M$ of $G$ with $M \subseteq  G'' \cdot (G')^p$;
here $p$ is a suitable prime number.

Well-known specimens of one-relator groups with deficiency at least 2
are non-abelian free groups and non-abelian, orientable surface groups.
These later groups admit presentations of the form
\begin{equation}
\label{eq:Presentation-orientable-surface-group}
\langle a_1, b_1, \ldots, a_g, b_g \mid [a_1,b_1] \cdots [a_g,b_g] \rangle \text{ with } g > 1.
\end{equation}
They are special cases of discrete and co-compact groups 
of orien\-tation-preserving automorphisms of the hyperbolic plane; 
these more general groups admit presentations of the form
\begin{equation}
\label{eq:Presentation-Fuchsian-group}
\langle a_1, b_1, \ldots, a_g, b_g, c_1, \ldots, c_\ell 
\mid [a_1,b_1] \cdots [a_g,b_g]c_1 \cdots c_\ell, c_1^{k_1}, \ldots,  c_\ell^{k_\ell} \rangle. 
\end{equation}
Here  $k_i > 1$ for $i = 1,\ldots,  \ell$, and either $g > 1$, or $g = 1$ and $\ell > 0$.
(See, \eg, the theorem on page 98  in \cite{HKS71}).
The deficiency of these groups is at least 1; as these presentations involve relators that are proper powers,
Proposition \ref{prp:Invariant-group-large-deficiency}
applies also to these more general groups.
Note, however, that for $g = 1$,
the prime $p$ has to chosen as a divisor of one of the exponents $k_1$, \ldots, $k_\ell$
\smallskip
\index{Surface groups}
\index{Fuchsian groups}

b) \emph{Fundamental groups of connected, orientable and bounded 3-manifolds.}
The deficiency of these groups is positive if the manifold $M$ has a boundary component of positive genus
and it is greater than 2 if one of its boundary components has genus greater than 1
(see, \eg, Lemma V.3 in \cite{Jac80}).
\smallskip
\index{3-manifolds groups}
\index{Jaco, W.}

c) \emph{One-relator groups with torsion.}
Let $G$ be a one relator group, say $G = \langle x_1, \ldots, x_m \mid r \rangle$,
and assume that $m \geq 2$ and that the defining relator $r$ is a proper power of a non-empty word $w$.
Then case (ii) of Proposition \ref{prp:Invariant-group-large-deficiency} is satisfied.
\smallskip

d) \emph{Two generator two relator groups.}
We finally consider groups with a presentation of the form
\[
 G = \langle x,y \mid r_1 = u^k, r_2 = w^\ell \rangle 
\text{ with $u \neq 1$, $w \neq 1$ and } k > 1, \ell > 1.
\]
The abelianisation of such a group can be finite; 
indeed, $G$ can be an infinite dihedral group (and thus a finitely related metabelian group).
We concentrate therefore on some special presentations 
where $G_{\ab}$ is infinite,
namely
\begin{equation}
\label{eq:Two-generator-two-relator-infinite-abelianization}
G = \langle x,y \mid r_1 =[x,y]^2,\;  r_2 = y^\ell \rangle  \text{ with  } \ell > 1.
\end{equation}
Let $\chi \colon G \to \R$ be the character given by $\chi(x) = 1$ and let $N$ be the kernel of $\chi$.
Then $N$ is the normal closure of $y$ 
and the Reidemeister-Schreier procedure, 
applied with the Schreier transversal $\TT = \{x^j \mid j \in \Z \}$,
furnishes the presentation
\begin{equation}
\label{eq:Presentation-N}
N = \langle \{y_j\}_{j \in \Z} \mid (y_{j+1}y_j^{-1})^2  \text{ and } y_j^\ell \text{ for  } j \in \Z \rangle.
\end{equation}
of $N$. 
In this presentation the generator $y_j$, with $j \Z$,  maps onto $x^j y x^{-j}$.

Assume first that $\ell$ is \emph{even}. 
Then case (iii) of Proposition \ref{prp:Invariant-group-large-deficiency} holds
and $G_{\ab}$ is infinite. 
The normal subgroup $N$ maps obviously onto the abelian group
\[
\bar{N} =  
\langle y_j \mid [y_j, y_k] = y_j^2 = 1 = \text{ for  } j \in \Z  \text{ and } (j,k) \in \Z^2\rangle
\]
which is isomorphic to the additive group of the group algebra $(\Z/2\Z)\gp(x)$,
in accordance with the addendum to the proposition.

Assume now that $\ell$ is \emph{odd}.
Then the hypothesis of case (iii) is not satisfied;
things improve, however, if one passes to a normal subgroup $G_1$ of index $\ell$,
namely the kernel of the homomorphism $\rho \colon G \epi \Z/\ell\Z$ given by
$\rho(x) = 0 + \ell \Z$ and $\rho(y) = 1 +  \ell \Z$.
The set $\TT = \{1, y, \ldots, y^{\ell-1} \}$ is a Schreier transversal for $G_1$
and the Reidemeister-Schreier procedure based on this transversal leads to the following presentation
\begin{equation}
\label{eq:Presentation-subgroup-G1}
 \langle x_0, x_1, \ldots, x_{\ell-1} \mid 
(x_0\cdot x_1^{-1})^2, (x_1\cdot x_2^{-1})^2, \ldots, (x_{\ell-1}\cdot x_0^{-1})^2 \rangle.
\end{equation}
for $G_1$.
This presentation satisfies the hypothesis for case (iii) of Proposition \ref{prp:Invariant-group-large-deficiency} 
and so $\Sigma^1(G_1)$ is empty;
moreover,
as the abelianization of $G_1$ has torsion-free rank 1,
the groups $G_1$ maps onto the wreath product $(\Z/2\Z) \wr C_\infty$ by the addendum.
The group $G$ itself does not map onto a wreath product of the form $(\Z/p\Z) \wr C_\infty$
with $p$ a divisor of $\ell$; 
in fact,  under a homomorphism $\varphi \colon (\Z/p\Z)C_\infty \rtimes C_\infty$
the commutator $[x, y]$ must map into $ (\Z/p\Z)C_\infty$, 
hence onto the 0 element and whence the image of $\varphi$ is abelian.
The invariant $\Sigma^1(G)$, however, is empty, just as is that of $G_1$ 
(use Proposition \ref{prp:Sigma1-finite-index}).
\end{examples}
%

%
\subsubsection{Groups having no fp extension without non-abelian free subgroups} 
\label{sssec:Groups-having-no-fp-extension-without-free-subgroups}
%
In this section we have a new look at Proposition \ref{prp:Fp-groups-having-non-abelian-free-subgroups}.
To get the proper perspective,
I quote a passage from the introduction to the recent article \cite{BGH12} 
by M. G. Benli, R. Grigorchuk and P. de la Harpe.
\index{Benli, M. G.}
\index{Grigorchuk, R.}
\index{Harpe, P. de la}

The introduction to \cite{BGH12} begins thus:
\begin{quotation}
{\itshape
\noindent
In the study of finiteness conditions on groups,
the following kind of question is natural: 
\begin{questionBGH}
Given a Property $(\PP)$ of groups, 
is any finitely generated group with $(\PP)$ 
a quotient of some finitely presented group with $(\PP)$?
\end{questionBGH}

The answer can be positive for trivial reasons,
for example when Property $(\PP)$ holds for free groups
(such as exponential growth)
or when Property $(\PP)$ implies finite presentation
(such as nilpotency, or polynomial growth); [\ldots]
\par

Here, we concentrate on a case with a negative answer;
the goal of this article is to study examples and results concerning 
\emph{finitely generated amenable groups
that do not have finitely presented amenable extensions.}
More precisely, amenable groups will often be groups of subexponential growth, 
and sometimes soluble groups;
and non-ame\-nable groups will usually have non-abelian free subgroups.
Recall that an \textbf{extension} of a group $G$ 
is a group $E$ given together with an epimorphism $E \epi G$.}
\end{quotation}

In the previous quotation amenable groups are mentioned; 
they are defined like this:
a group $G$ is \emph{amenable} 
if there exists a finitely additive,  left-invariant measure $\mu$, 
defined on the set of all subsets of $G$, and such that $\mu(G) = 1$.
Abelian and finite groups are amenable, 
and the class of amenable groups is closed under extensions and directed unions,
and the formation of subgroups and quotient-groups.
(For some further informations we refer the reader to section 2 of the survey \cite{Cec01} 
and the many references cited therein.)
\index{Amenable groups!definition}
\index{Amenable groups!examples}
\index{Definition of!amenable group}

Now to some examples illustrating the new look at Proposition
\ref{prp:Fp-groups-having-non-abelian-free-subgroups}.
\begin{examples}
\label{examples:Elementary-abelian}
a) Let $G$ be finitely generated group
that contains an infinite, locally finite normal subgroup $N$ with infinite cyclic quotient $G/N$. 
Then $S(G)$ is a 0-dimensional sphere 
and $\Sigma^1(G)$ is empty by Lemma \ref{lem:Infinite-locally-finite-group}.
The group $G$ admits, of course, finitely presented extensions $E\epi G$ of $G$,
but according to Proposition \ref{prp:Fp-groups-having-non-abelian-free-subgroups} 
every such extension contains non-abelian free subgroups.
(Concrete examples of such groups $G$ are given in Examples \ref{examples:Torsion-by-infinite-cyclic}.)

b) Elementary amenable groups are a sweeping generalization of the groups just discussed. 
By definition, the class $\EG$ of \emph{elementary amenable groups}
is the smallest class of groups that contains all finite and all abelian groups 
and is closed under extensions and directed unions. 
All soluble groups, but also all locally finite or locally nilpotent groups, are in $\EG$.
The class is also closed under the passages to subgroups and to quotient groups
(see \cite{Cho80} for further informations on elementary amenable groups).
\index{Elementary amenable groups!definition}
\index{Elementary amenable groups!examples}
\index{Definition of!elementary amenable group}

A finitely generated, elementary amenable group $G$ has no non-abelian free subgroup.
If it admits a rank 1 character $\chi \colon G \to \R$ 
such that $[\chi]$ and $[-\chi]$ are both outside of $\Sigma^1(G)$,
then Proposition \ref{prp:Fp-groups-having-non-abelian-free-subgroups} 
applies
and guarantees that every finitely presented extension has a free subgroup of rank 2.
\end{examples}

\subsubsection{Large finitely presented groups} 
\label{sssec:Large-fp-groups}
%
Proposition \ref{prp:Fp-groups-having-non-abelian-free-subgroups},
in conjunction with Proposition \ref{prp:Invariant-group-large-deficiency},
allows one to prove that finitely presented groups with few relators 
contain non-abelian free subgroups,
but  by combining the addendum to  Proposition  \ref{prp:Invariant-group-large-deficiency}
with a result of G. Baumslag,
one can obtain a better conclusion and show that the groups are large.
Here, following M.~Gromov in \cite[p.\;82, Thm.\:(B)]{Gro82a},
a group will be called \emph{large} if it has a subgroup of finite index
that maps onto a non-abelian free subgroup. 
\index{Large groups!definition}
\index{Definition of!large group}

In the sequel, I shall reprove the mentioned result of G. Baumslag.  
Its proof is fairly tricky and based on ideas
that go back to the paper \cite{BaPr78} by B.~Baumslag and S.~J.Pride.
I begin by explaining the gist of the proof with the help of a simple example.

\begin{example}
\label{example:Baumslags-theorem}
Let $G$ be a one-relator group with 3 generators, 
say $G = \langle a, b, t \mid r \rangle$.
Assume $r$ has exponent sum 0 with respect to $t$ and let $\chi \colon G \to \R$
be the rank 1 character sending $t$ to 1 and the other generators to 0.
Let $F$ be the free group on $a$, $b$, and $t$ 
and let $\eta \colon  F \epi E$ denote the epimorphism involved in the presentation of $G$.
Then $r$ lies in the normal subgroup $R = \gp_F(a, b)$ of $F$.
We may and shall assume 
that $r$ is a word in the conjugates of $a$ and $b$ by non-negative powers of $t$,
say a word  in the conjugates
\[
a_0 = a, \;a_1 = tat^{-1}, \ldots, a_m = t^{m} b t^{-m} 
\text{ and }
b_0 = b,  \ldots, b_m = t^{m} b t^{-m}
\]
and their inverses.
Not all of the listed generators need actually occur in $r$.
Put now $\ell = m+1$ and $G_\ell = \gp_G(a,b,t^\ell)$.
The set $\TT = \{1, t,\ldots,  t^{m} \}$ is then a Schreier transversal of $G_\ell$ in $G$,
and the Reidemeister-Schreier procedure based on it 
leads to the generators
\begin{align*}
a_0 &\mapsto a,  a_1 \mapsto tat^{-1}, \ldots, a_{m} \mapsto t^{\ell-1} a t^{ -m},\\
b_0 &\mapsto b,  b_1 \mapsto tbt^{-1}, \ldots, b_{\ell-1} \mapsto t^{\ell-1} b t^{1 -\ell}, 
\quad \text{and}\quad
 T \mapsto t^\ell.
\end{align*}

The relators of $G_\ell$ are obtained by conjugating $r$ by the elements of the transversal 
and rewriting the conjugates in terms of the generators listed above.
If the last conjugate $t^{\ell-1}r t^{1- \ell}$  is rewritten,
the generators occurring in it are among the words
\begin{gather*}
a_{m}, \quad a_{m+1} \mapsto T a_0 T^{-1}, \quad a_{m + 2} \mapsto T a_1 T^{-1}, \ldots, 
a_{2m} \mapsto T a_{m - 1} T^{-1}, \\
b_{m},\quad  b_{m+1} \mapsto T b_0 T^{-1}, \quad b_{m+2} \mapsto T b_1 T^{-1}, \ldots, 
b_{2m} \mapsto T b_{m - 1} T^{-1}.
\end{gather*}

A \emph{first essential observation}, going back to the Baumslag-Pride paper \cite{BaPr78}, is now this: 
in the conjugate $t^{m}r t^{-m}$
the generator $T$ occurs only in subwords of the form $T a_j T^{-1}$
with $0 \leq j \leq m-1$.
This fact continues to be true for the other conjugates $r$, $trt^{-1}$, \ldots, $t^{m-1}r t^{1-m}$ of $r$.

Now to a \emph{second, crucial idea} of B. Baumslag and S. J. Pride:
let $M \triangleleft E_\ell$ be the normal closure of the elements
\[ 
a_0, a_1, \ldots, a_{m-1}  \text{ and } b_0, b_1, \ldots, b_{m-1}
\]
and set $Q_\ell = G_\ell /M$.
The  observation implies then
that the generator $T$ does not occur in any relator of the presentation of $Q_\ell$ 
that is obtained from that of $G_\ell$ by adding the listed generators $a_j$ and $b_j$ as relators.
Put differently, $Q_\ell$ is the free product of a subgroup $K$  
and the infinite cyclic group generated by $T$.
Moreover, the character $\chi \colon G \to \R$ induces a character $\bar{\chi} \colon Q_\ell \to \R$
with $K$ in its kernel and $\bar{\chi}(T) = \ell$.

The idea is now to show that the group $G_\ell = K \star \gp(T)$ 
has a subgroup of  finite index which maps onto a non-abelian free group.
If $K$ is the trivial group this is impossible.
To exclude this and other unpleasant cases,  B. Baumslag and S. J. Pride assume in \cite{BaPr78}
that the deficiency of the presentation defining $E$ is at least 2 
and deduce that $K$ is a finitely presented group of deficiency at least 1
and so maps onto $\Z$.
\end{example}
\index{Baumslag, B.}
\index{Pride, S. J.}

\paragraph{Approach taken by G.Baumslag.}
In his lecture notes \cite{Bau93}, 
G. Baumslag shows that a finitely presented group is large 
whenever it maps onto the wreath product $W =\gp(b) \wr C_\infty$ 
of a cyclic group of order some prime $p$ and an infinite cyclic group. 
The details of G. Baumslag's proof are as follows:
\index{Baumslag, G.}

assume $G$ is a finitely presented group that admits an epimorphism $\rho \colon G \epi W$
onto the wreath product $W = \gp(b) \wr \gp(s)$
of a finite cyclic group $\gp(b)$ of order a prime number $p$ by an infinite cyclic group generated by $s$.
There exists then a finite set of generators $\{ a^{(1)}, \ldots, a^{(f)},  t \}$ of $G$ 
with $\rho(a^{(1)}) = b$,
with $a^{(2)}$, \ldots, $a^{(f)}$ in the kernel of $\rho$ and with $\rho(t) = s$.
Let $\psi \colon  W \to \R$ be the (unique) character that sends $s$ to 1
and set $\chi = \psi \circ \rho$.
Define $F$ be the free group with basis $\{a^{(1)}, \ldots, a^{(f)}, t\}$
and let $\eta\colon F \epi G$ denote the obvious epimorphism.
The kernel of $\eta$ is then the normal closure of a finite set $\RR$ of relators 
(here one uses that $G$ is finitely presentable and Lemma 8 in \cite{Neu37}).

Set $R=  \gp_F(a^{(1)}, \ldots, a^{(f)})$.
Then  $\RR$ is contained in $R$;
upon conjugating the elements of $\RR$ by suitable non-negative powers of $t$,
there exists a positive integer $m$ such that
each relator $r \in \RR$ can be written as a word in the generators $t^j a^{(i)} t^{-j}$ 
with  $1 \leq i \leq f$ and $j \in \{0, 1, \ldots, m\}$.
Next, set $\ell = m+1$ and consider the subgroup $G_\ell = \gp_G(t^\ell,  a^{(1)}, \ldots, a^{(f)})$  
having index $\ell$ in $G$. 
Use the Schreier transversal $\TT = \{1,t, \ldots, t^{m} \}$ of $G_\ell$ in $G$ 
to obtain a finite presentation of $G_\ell$.
Let now $M \triangleleft E_\ell$ be the normal closure of all the conjugates  $t^j a^{(i)}t^{-j}$ 
with $j \in \{0,1, \ldots, m-1 \}$ and $1 \leq i \leq f$.
It then follows, as in the Example \ref{example:Baumslags-theorem},
that $Q_\ell = G_\ell/M$ is a free product $K \star \gp(t^\ell)$
with 
\[
K = (\gp(t^m a^{(1)}t^{-m}, \ldots, t^m a^{(f)}t^{-m})\cdot M) /M.
\]
We claim that $K$ maps onto a cyclic group of order $p$.

To justify this, 
we recall the definition of the projection $\rho \colon G \epi W$;
here $W$ is the wreath product $\gp(b) \wr \gp(s)$.
By definition, 
$\rho$ sends the generators $a^{(2)}$, \ldots, $a^{(f)}$ to $1 \in W$, 
then $a^{(1)}$ to the generator $b$ and finally $t$ to $s$.
Let $\rho_\ell$ be the restriction of $\rho$ to $G_\ell = \gp_G(t^\ell,  a^{(1)}, \ldots, a^{(f)})$.
The image of $\rho_\ell$ in $W$ is the subgroup 
\[
W_\ell =\gp\big(b, sbs^{-1}, \ldots, s^{m} b s^{-m} \big) \wr \gp(s^\ell) 
= (\Z/p\Z)\gp(s) \rtimes \gp(s^\ell).
\]
The normal subgroup $M \triangleleft G_\ell$ is the normal closure of all the conjugates  $t^j a^{(i)}t^{-j}$ 
with $j \in \{0,1, \ldots, m-1 \}$ and $1 \leq i \leq f$.
So $\rho_\ell(M)$ is the normal subgroup in $W_\ell$ generated by the group ring elements
1, $t$, \ldots, $t^{m-1}$.
We conclude that $\rho_\ell$ induces a projection of $Q_\ell = G_\ell/M$ 
onto a wreath product of the form $\Z/p\Z \wr \gp(s^\ell)$.
Under this projection the factor $K$ of $Q_\ell = K \star \gp(t^\ell)$ 
maps onto the cyclic group $\Z/p\Z$,
as asserted. 

So far we know that $G$ 
maps onto the free product $L =(\Z/p\Z) \star \gp(s^\ell)$ of a cyclic group of order $p$ and an infinite cyclic group.
Consider now the projection $\pi_1 \colon L \epi \Z/p\Z$ onto the first free factor.
Its kernel is then a free group of rank  $p \geq 2$ (this follows, \eg, by the Reidemeister-Schreier procedure).
We have thus established the following result of G. Baumslag's:
\begin{thm}[\protect{\cite[Thm.\;IV.3.7]{Bau93}}]
\label{thm:Baumslag-large-groups}
Let $G$ be a finitely generated group that maps onto a wreath product $W =(\Z/p\Z) \wr C_\infty$ of a cyclic group of order a prime $p$ and an infinite cyclic group. 
If $G$ admits a finite presentation, it is large.
\end{thm}
\index{Baumslag, G.}

\subsubsection{Examples of large finitely presented groups} 
\label{sssec:Examples-large-fp-groups}
%
We conclude section \ref{ssec:Fp-groups-groups-with-free-subgroups} with some applications of 
Theorem \ref{thm:Baumslag-large-groups}.
This theorem presupposes 
that the finitely presented group $E$ admits an epimorphism $\rho \colon G \epi W$ 
onto the wreath product  $W = (\Z/p\Z) \wr C_\infty$ 
of a cyclic group of order a prime $p$ by an infinite cyclic group.
This hypothesis implies that the sphere $S(G)$ is non-empty 
and also that $\Sigma^1(G)^c$ contains an antipodal pair of rank 1 points
(use Lemma \ref{lem:Infinite-locally-finite-group}).

The statement of Proposition \ref{prp:Invariant-group-large-deficiency}
lists three easily formulated assumptions 
that imply the hypothesis of Baumslag's theorem.
The first two of these assumptions are illustrated by the examples given below.
\index{Large groups!examples|(}

\paragraph{1) Groups of deficiency at least 2. }
If $G$ is a finitely presented group of deficiency at least 2
the sphere $S(G)$ has positive dimension 
and, for every rank 1 character $\chi$ and for every prime $p$,
the kernel $N$ of $\chi$ maps onto an infinite elementary $p$-group  
and hence onto a copy of the additive group of the group algebra $(\Z/p\Z)C_\infty$
(by the addendum in the statement of Proposition \ref{prp:Invariant-group-large-deficiency}).
So Theorem \ref{thm:Baumslag-large-groups} applies and shows that $G$ is large,
a conclusion first established by B. Baumslag and S. J. Pride in \cite{BaPr78}.
\index{Baumslag, B.}
\index{Pride, S. J.}

\paragraph{2) Groups of deficiency 1. }
A group $G$ of deficiency 1 need not be large,
witnesses being the free abelian groups of rank 1 or 2 
or, more generally, the soluble Baumslag-Solitar groups 
$\BS(1,\ell) = \langle a, t \mid tat^{-1} = a^\ell \rangle$ with $\ell \in \Z$.
\index{Groups!of Baumslag-Solitar}

The situation is more complex for the non-soluble Baumslag-Solitar groups  
\[
G_{k,\ell} = \BS(k,\ell) = \langle a, t \mid ta^kt^{-1} = a^\ell \rangle,
\]
say with  $k > 1$ and $\ell > 1$.
If $k$ and $\ell$ have a common factor $d > 0$ then $\BS(k, \ell)$ 
maps obviously onto $(\Z/d\Z) \star \gp(t)$ and so it is large.
If, on the other hand, $k$ and $\ell$ are relatively prime 
then $\BS(k, \ell)$ is not large (see \cite[Example 3.2]{EdPr84}),
in spite of the fact that $\Sigma^1(G_{k,\ell} )$ is empty
(use that $\Sigma^1$ of the metabelian top of $G_{k,\ell}$ is empty
by Example \ref{example:Locally-infinite-cyclic-by-infinite-cyclic})
and $G_{k,\ell}$ contains non-abelian free subgroups.
 
The strategy of M. Edjvet and S. J. Pride is this:
 If $G$ is a finitely generated large group it contains subgroups $M_1 \triangleleft H_1  \leq G$
so that $H_1/M_1$ is free of finite rank $r \geq 2$ and $H_1$ has finite index in $G$.
Since the ranks of the finite index subgroups of $H_1/M_1$ are not bounded,
$G$ will have subgroups of finite index $H$ 
whose abelianisation $H_{\ab}$ has arbitrary large torsion-free rank.
A finitely generated group $L$ is therefore not large if there exists an integer $b > 0$ 
so that the inequality $r_0(H_{\ab}) < b$ holds for every subgroup $H$ of finite index in $L$.
Edjvet and Pride now show by a tricky argument 
that if $k >1$ and  $\ell > 1$ are relatively prime integers and $H < \BS(k, \ell)$ is a subgroup of finite index,
then $r_0(H_{\ab}) < 2$; see \cite[Ex.\,3.3]{EdPr84}.
\index{Groups!of Baumslag-Solitar}
\index{Edjvet, M.}
\index{Pride, S. J.}

Another group of deficiency 1 that is not large is the one relator group
\begin{equation}
\label{eq:Group-with-only-cyclic-finite-images}
G = \langle a, t \mid tat^{-1} \cdot a \cdot ta^{-1} t^{-1} = a^2 \rangle
\end{equation}
G. Baumslag detected in the late 1960s 
that all finite quotients of this group are cyclic (see \cite{Bau69b}).
It follows easily that $G$ is not large.
The group $G$ has, however, non-abelian free subgroups and $\Sigma^1(E)$ is empty.
These facts can be established 
by adapting the proof of Theorem \ref{thm:FP2-groups-and-HNN-extensions},
to the given presentation of $G$.
If one proceeds as in the first proof of Theorem \ref{thm:FP2-groups-and-HNN-extensions}
one discovers, with the help of the Freiheitssatz, 
that $G$ is the HNN-extension $\langle B, t \mid t a_0t^{-1} = a_1 \rangle$
with base group  $B = \langle a_0, a_1 \mid a_1 a_0 a_1^{-1} = a_0^2 \rangle$.
The associated subgroups are $S = \gp(a_0)$ and $T = \gp(a_1)$;
as they are clearly distinct from $B$,
the invariant $\Sigma^1(G)$  is empty by Proposition \ref{prp:HNN-extension-and-Sigma1}
whence $G$ has a non-abelian free subgroups
by Proposition \ref{prp:Fp-groups-having-non-abelian-free-subgroups}.
\index{Baumslag, G.}
\smallskip

The above examples should not leave the impression
that groups of deficiency 1 are rarely large.
First of all, 
the addendum in the statement of Proposition \ref{prp:Invariant-group-large-deficiency})
shows that $E$ maps onto a wreath product
whenever it admits a presentations of deficiency 1 in which one relator is a proper power,
whence such a group $G$ will be large by Theorem \ref{thm:Baumslag-large-groups}.
(This consequence is due to
M.\ Gromov in \cite[p.\;291. Thm. (B$'$)]{Gro82a} and to R.\ St\"ohr in \cite{Sto83}, independently.
\index{Gromov, M.}
\index{Stoehr, R.@St{\"o}hr, R.}

Examples of groups fulfilling the stated assumptions include two generator one-relator groups with torsion.
Moreover, 
the investigations of J. O. Button, carried out in \cite{But08a}, 
have brought to light that many groups with a deficiency 1 presentation in which no relator is a proper power 
are large, too.
\index{Large groups!examples|)}
\index{Button, J. O.} 
%
\subsection{Applications to infinitely related groups}
\label{ssec:Infinitely-related-soluble-groups}
%
The consequence of Theorem \ref{thm:FP2-groups-and-HNN-extensions}
that will be discussed in this section 
is nothing but a contraposition of the theorem.
For ease of reference we state it as
\begin{prp}
\label{prp:Infinitely-related-soluble-groups}
Assume $G$ is a finitely generated group which contains no non-abelian free subgroup
and admits a rank 1 character $\chi \colon G \to \R$ 
so that both $[\chi]$ and $[-\chi]$ are points outside of $\Sigma^1(G)$.
Then $G$ is not of type $\FP_2$ over any commutative ring
and it does not admit a finite presentation.
\end{prp}
\index{Infinitely related groups!sufficient condition}

\begin{remarks}
 \label{remarks:Fg-groups-not-ascending-HNN-extension}
 If one wants to apply Proposition \ref{prp:Infinitely-related-soluble-groups}
 to a finitely generated group containing no free subgroups of rank 2,
 one has to find a rank 1 character  $\chi \colon G \to \R$ 
 with $\{[\chi],  [-\chi] \} \subseteq \Sigma^1(G)^c$.
 
 a)  For some groups,
 finding such a character is easy: in Application 1 below the group $G$ is locally finite by infinite cyclic
 and so the sphere $S(G)$ has only two points.
 In the situation of Proposition \ref{prp:Invariant-group-large-deficiency} 
 the sphere $S(G)$ is often infinite and so there are infinitely many candidates.
 But, as the invariant $\Sigma^1(G)$ is empty, all choices lead to the desired result.
 
 b) For other groups, the search for an appropriate character may be difficult.
 Here is a case in point.
 Let $\tilde{G}$ be a one-relator group with two generators, 
 say $\tilde{G} = \langle x, y \mid r \rangle$, 
 and assume that the abelianisation of $\tilde{G}$ is free abelian of rank 2.
 Consider the metabelian quotient $G = \tilde{G}/\tilde{G}''$.
 Then $S(G)$ is a circle containing infinitely many rank 1 points.
 It is thus not clear
 whether a character exists 
 that allows one to infer that $G$ is infinitely related 
 and, it so, how it can be found.
 Fortunately, the answer to both questions can be found algorithmically;
 see Theorem B in \cite{Str81a}.
\end{remarks}
\index{Strebel, R.}
%
\subsubsection{How to prove that a finitely generated group is infinitely related?} 
\label{sssec:How-prove-that-fg-group-is-not-fp}
Prior to listing some applications of Proposition \ref{prp:Infinitely-related-soluble-groups},
I would like to describe the historical context 
in which the proposition came into being.

In Note \ref{note:Neumanns-groups} I have mentioned 
that, in \cite{Neu37},  B. H. Neumann studies the question 
as to whether there are finitely generated, infinitely related groups
and that he answers it in the affirmative, giving two kinds of justifications.
To do so, Neumann constructs groups suited to his aim;
they do not shed light on the related question 
which of the finitely generated groups discussed in a branch of group theory,
say in the \emph{Theory of Infinite Soluble Groups},
admit of a finite presentation.
\index{Neumann, B. H.}
\index{Infinitely related groups!existence}

This second question is addressed by P. Hall in his influential paper \cite{Hal54}.
He proves that every extension of two finitely presented groups is again so (Lemma 1 on p.\;426)
and deduces that every polycyclic group admits a finite presentation (Corollary on p.\;426).
Extensions provide one way of producing new groups from given groups, 
the wreath product furnishes another one.
It turns out, however,
that the wreath product  $K \wr L$ of two finitely presented non-trivial groups $K$ and $L$ 
is rarely finitely presentable;
indeed, as shown by G. Baumslag in \cite{Bau61}, this happens if, and only if, $L$ is finite
(Baumslag's proof is lengthy; shorter proofs can be found in \cite{Str84} (Theorem 4) 
and in \cite{LeRo04} (Result 11.1.2).)
\index{Hall, P.}
\index{Baumslag, G.}
\index{Strebel, R.}
\index{Robinson, D. J. S.}
\index{Lennox, J. C.}
\index{Wreath products!finite presentation}
\index{Finitely presented groups!and wreath products}

In proving the stated characterization,
Baumslag makes use of free products; 
in a later paper \cite{Bau71c}, 
he invokes the theory of free products with amalgamation to prove that a specific group is infinitely related.
At the beginning of his survey \cite{Bau74}, 
he lists some classes of finitely generated soluble groups 
the members of which have been shown to be infinitely related
and turns then to methods that permit one to justify the listed results.
On page 66, he writes:
\begin{quotation}
{\itshape
In general there seem to be two ways of proving 
that a given finitely generated group $G$ is not finitely presented.

The first of these is a two-stage procedure. 
The first stage is to produce an explicit presentation of $G$ 
in terms of a finite set of generators and an infinite set of defining relations. 
The second stage is to prove that no finite subset of these relations suffice to define $G$. 
(The fact that $G$ is not finitely presented
no matter which finite system of generators one may choose for G is the content of
a theorem of B. H. Neumann \cite{Neu37}.) 
Both parts of this procedure can be difficult to effect; 
the second tends to be the more awkward and often invokes the use of generalized free products.

The second way to prove that a finitely generated group $G$ is not finitely presented
is to show that its multiplicator $m(G)$ is not finitely generated, 
for the multiplicator of a finitely presented group is always finitely generated. 
Actually it was an open question for a while whether conversely a finitely generated group
with a finitely generated multiplicator is finitely presented. 
This is false; there are even metabelian counter-examples [\ldots].}
\end{quotation}
\index{Baumslag, G.}

The proof of Proposition \ref{prp:Infinitely-related-soluble-groups}
uses the first way described in the above quotation,
but differs from it in two important respects.
First of all, it relies on the theory of HNN-extensions and so presupposes 
that the abelianisation of the group $G$ be infinite.
Many groups that are known to be infinitely related do not satisfy this assumption
and so cannot be shown to be infinitely related with the help of the proposition;
but every finitely generated infinite soluble group admits an indicable subgroup of finite index.
Secondly, no details of the presentation enter into the hypotheses of the proposition:
one merely assumes the existence of a finite presentation and derives a structural property.
\smallskip

We now give some typical applications of Proposition \ref{prp:Infinitely-related-soluble-groups}.
In these applications,
the verification that the points $[\chi]$ and $[-\chi]$ lie outside of $\Sigma^1$ 
will often be based on Lemma \ref{lem:Infinite-locally-finite-group}.
%
\subsubsection{Application 1: locally finite by infinite cyclic groups} 
\label{sssec:Locally-finite-by-infinite-cyclic-groups}
%
We begin with the special case of of Lemma \ref{lem:Infinite-locally-finite-group} 
where the normal subgroup $M$ is reduced to the unit element.
Proposition  \ref{prp:Infinitely-related-soluble-groups} and the lemma then yield
\begin{crl}
\label{crl:Locally-finite-by-infinite-cyclic}
Assume $G$ is a finitely generated extension of an infinite, locally finite group by an infinite cyclic group.
Then $\Sigma^1(G)$ is empty and $G$ does not admit a finite presentation. 
\end{crl}
\index{Infinitely related groups!locally finite by cyclic}
\index{Infinitely related groups!examples}
\begin{remarks}
\label{remarks:Torsion-by-infinite-cyclic-groups}
a) Corollary \ref{crl:Locally-finite-by-infinite-cyclic} is an easy consequence of Theorem
\ref{thm:FP2-groups-and-HNN-extensions},
\ie, of Theorem A in \cite{BiSt78} published in 1978.
In view of this fact is is surprising that the corollary has only been detected  in the last few years;
in fact, the earliest reference I am aware of is Theorem D in the paper \cite{BaMi09} 
by G. Baumslag and C. H. Miller III. The corollary is also part of Corollary C.4 in \cite{BGH12}.
\index{Baumslag, G.} 
\index{Miller III, C. F.}

b) The conclusion of Corollary \ref{crl:Locally-finite-by-infinite-cyclic} need not be true
if $G$ is an extension of a locally finite group by an abelian group of torsion-free rank greater than 1. 
Examples testifying to this are provided by a sequence of groups introduced 
by C. H. Houghton in \cite{Hou78} on p.\;257;
these groups are finitely related (see \cite{Bro87a}).
Their invariants will be determined in section \ref{sssec:Application-invariant-Houghtons-group}.
\index{Houghton, C. H.} 
\index{Brown, K. S.}

c) The situation can also be different
if the kernel $N$ is an \emph{infinite torsion group that is not locally finite}.
Then the group $G$ can be finitely related, but if so 
it must be an ascending HNN-extension over a finitely generated base group
(by Theorem \ref{thm:FP2-groups-and-HNN-extensions}).
Groups meeting these requirements are hard to come by,
but they exist.

The first such group has been constructed by R. Grigorchuk in \cite{Gri98}; 
it is an example of a \emph{finitely presented amenable group that is not elementary amenable}.
\index{Grigorchuk, R.}
\index{Amenable groups!examples}
Later A. Yu. Ol'shanskii and M. V. Sapir manufactured another group
that is an extension of an infinite torsion group by a cyclic group (see \cite{OlSa02}).
It was the first example of a finitely presentable non-amenable group  
that contains no non-abelian free subgroups.
\end{remarks}
\index{Ol'shanskii, A. Yu.}
\index{Sapir, M.}

\begin{examples}
\label{examples:Torsion-by-infinite-cyclic}
If one wants to apply Corollary \ref{crl:Locally-finite-by-infinite-cyclic}
one has to find an infinite locally finite group $N$ 
that admits of an automorphism $\alpha \colon N \iso N$ of infinite order,
all in such a way that $N$ is finitely generated as a $\gp(\alpha)$-group.
Here are some well-known groups of this kind.
\index{Locally finite by cyclic groups!examples|(}
\smallskip

a) $N$ is the base group of the wreath product $B \wr C_\infty$ of a finite group $B$ 
by an infinite cyclic group $C_\infty$.
\index{Wreath products!examples}

b) $N$ is the group of all permutation of $\Z$ that fix all but finitely many elements
and $\alpha$ is the automorphism induced by conjugation by the (infinitary) permutation 
$ j \mapsto j+1$.
The split extension $G = N \rtimes \gp(\alpha)$ is then generated by two elements,
namely by $\alpha$ and the transposition $\tau$ of two adjacent elements of $\Z$, say 0 and 1.

c) Far more sophisticated examples have been constructed by B. H. Neumann 
in his remarkable note \cite{Neu37}.
Here is a brief description of these groups.

Let $S$ be an infinite subset of the odd integers $n \geq 5$ 
and let $P_S$ denote the unrestricted direct product
\[
P_S = \prod\nolimits_{n \in S} A_n.
\]
Let $B= B_S$ denote the subgroup of $P_S$ 
that is generated by the following two sequences $\alpha = \alpha_S$ and $\tau = \tau_S$  of even cycles:
\begin{align}
\alpha_S \colon n &\mapsto  a_n = (1 \mapsto 2 \mapsto 3 \mapsto 1), 
\text{ the remaining elements being fixed,}
\label{eq:Definition-alpha}\\
\tau_S \colon n &\mapsto  t_n = (1,2,3, \cdots ,n-1,n).
\label{eq:Definition-tau}
\end{align}

Neumann proves that two such groups $B_S$ and $B_{S'}$ are isomorphic if, and only if, $S = S'$,
whence there are $2^{\aleph_0}$ pairwise non-isomorphic groups $B_S$.
At the end of his note he states
that $B_S$ contains the restricted direct product $D = \Dr\nolimits_{n \in S} A_n$,
that the quotient group $B_S/D$ does not depend on $S$ 
and is isomorphic to the semi-direct product of the alternating group 
on a countably infinite set by an infinite cyclic group. 
(A proof of the second claim may be found in \cite[pp.\,65+66]{Har00}).
It follows that the normal closure $N = \gp_{B_S}(\alpha)$ of the element $\alpha_S$
 is an infinite locally finite group,
and thus $B_S$ is \emph{locally finite by infinite cyclic}, as asserted. 

The fact that $N$ is locally finite is pointed out in \cite{BaMi09} (see Corollary 11).
To verify this fact it suffices to show that, for every $\ell \geq 0$,
the subgroup 
\[
M_\ell = \gp(\alpha, \tau \cdot \alpha \cdot \tau^{-1}, \ldots,
\tau^\ell \alpha \tau^{-\ell})
\]
is finite.
To do so,
fix $\ell$ and consider the images $\bar{M}_{\ell, n}$ of $M_\ell$ 
under the various projections $ \pi_n \colon B \to A_n$.
If $n \geq \ell + 3$, this image does not depend on $n$.
It follows that for each index $n_0 \geq \ell + 3$ 
the kernel of $\pi_{n_0}$ will consist of sequences $(g_n)_{n \in S}$ 
that are equal to 1 for every $n \geq n_0$.
But if so, the kernel of $\pi_{n_0}$, and hence the group $M_\ell$, are finite.
\index{Neumann, B. H.}
\index{Locally finite by cyclic groups!examples|)}

\begin{note}
\label{note:Neumanns-groups}
On pages 125--127 of his paper, B. H. Neumann discusses the question
\emph{whether there are any groups which
have a finite number of generators but which cannot be defined by a finite set of relations}.
He asserts that such groups exist and offers two proofs.
The first is direct and consists in giving appropriate examples 
(actually Neumann constructs $2^{\aleph_0}$ examples no two of which are isomorphic).
The second relies on the fact 
that there are only countably many finitely presentable groups.
So most of the previously described groups $B_S$ must be infinitely related.
But more is true: by Corollary \ref{crl:Locally-finite-by-infinite-cyclic} 
we know that all of them are infinitely related.
This fact seems to have been pointed out first by G. Baumslag and C. H. Miller III in \cite[Thm.\;C]{BaMi09}.
(The authors present two proofs, one of them is based on \cite[Thm.\;A]{BiSt78}.)
\end{note}
\end{examples}
\index{Baumslag, G.}
\index{Miller III, C. F.}
%
\subsubsection{Application 2: Infinitely related generalized soluble groups} 
\label{sssec:Infinitely-related-generalized-soluble-groups}
%
It turns out that the conclusion of Corollary \ref{crl:Locally-finite-by-infinite-cyclic}
holds for a lager class of groups.
Let $\PP$ be a property of groups and $G$ a finitely generated group
which is an extension of a normal subgroup $N$, that is locally $\PP$, by an infinite cyclic group.
\emph{We look for an easily stated condition that forces $G$ to be infinitely related.}
The condition should imply that $G$ cannot be an ascending HNN-extension 
with finitely generated base group $B$ contained in $N$.

If $\PP$ is the property of being finite a suitable requirement is that $N$ do not have $\PP$,
that is to say, be infinite (see  Corollary \ref{crl:Locally-finite-by-infinite-cyclic}).
In general, we shall require that $N$ does not satisfy $\PP$ 
and that $\PP$ has the following special feature
\begin{equation}
\label{eq:Special-feature-P}
\left.
\begin{minipage}[c]{9cm}
If  $B$  has $\PP$ and $B_\infty$ is the limit of an infinite sequence\\
  $B \mono B_1 \mono B_2 \mono \cdots$
of copies of $B$ then $B_\infty$ has $\PP$
\end{minipage}
\right\}.
\end{equation}
The following corollary then holds:
\begin{crl}
\label{crl:Locally-PP-by-infinite-cyclic}
Let $G$ be a finitely group that is an extension of $N$ by an infinite cyclic group.
Assume $\PP$ has the special feature stated in \eqref{eq:Special-feature-P}
and that $N$ is locally $\PP$, but not $\PP$.
Then $G$ is infinitely related.
\end{crl}
\index{Infinitely related groups!examples}
\index{Infinitely related groups!generalized soluble}

\begin{proof}
Thanks to Proposition \ref{prp:Infinitely-related-soluble-groups}
it suffices to verify 
that $G$ is not an ascending HNN-extension with finitely generated base group $B$ contained in $N$.
If it were, then $B$, being finitely generated,  would have  property $\PP$.
As $N$ would then be a directed union of the form $\bigcup_{\ell \geq 0} t^\ell B t^{-\ell}$ with $t \in G$,
it would satisfy $\PP$ by property \eqref{eq:Special-feature-P}, 
contrary to the hypothesis on $N$.
\end{proof}

To construct explicit examples covered by the corollary,
one has to fix a property $\PP$ with the special feature \eqref{eq:Special-feature-P}.
As pointed out before,
the property of being \emph{finite} qualifies 
and examples based on it have been discussed in Examples \ref{examples:Torsion-by-infinite-cyclic}.
The property of being \emph{nilpotent} has also the special feature.
Corollary \ref{crl:Locally-PP-by-infinite-cyclic} thus guarantees 
that a finitely generated group $G$ which is locally nilpotent by cyclic, but not nilpotent by cyclic,
does not admit a finite presentation.

Here is a well-known specimen of such a group.
Let $A$ be the free abelian group with basis $\{e_j \mid j \in \Z\}$ 
and $G$ the group generated by the automorphisms $\tau$ and $\alpha$ of $A$ given by
\begin{align*}
\tau(e_i) &= e_{i+1} \text{ for } i \in \Z,  \\
\alpha(e_0) &= e_0 + e_1 \text{ and } \alpha(e_j) = e_j \text{ for } j \in \Z \smallsetminus \{0\}.
\end{align*}
For every $n \geq 0$ the subgroup  
$B_n = \gp(\alpha, \tau\circ \alpha \circ \tau^{-1}, \ldots, \tau^\ell\circ \alpha \circ \tau^{-\ell})$
is then nilpotent of class $n+1$ and so $N = \gp_G(\alpha)$ is locally nilpotent, but not nilpotent.

The group $G$ is a variation of the characteristically simple groups due to D. H. McLain
(see \cite[pp.\,361--362]{Rob96}) and is an example of an elementary amenable group; 
it is described by C. Chou in \cite{Cho80} as Example 1 on page 402. 
\index{Elementary amenable groups!examples}
\index{Infinitely related groups!locally nilpotent by cyclic}
\index{Generalized soluble groups!examples}
\index{McLain, D. H.}
\index{Chou, C.}

%
\subsubsection{Application 3: free soluble groups} 
\label{sssec:Relatively-soluble-groups-with-few-relators}
%
In the early 1970s, G. Baumslag and V. N. Remeslennikov discovered
that there are many more finitely generated met\-abelian groups with a finite presentation
than one might have suspected. 
The borderline between the finitely related
and the finitely generated but infinitely related, metabelian groups had therefore become blurred.
For further progress,
a good supply of simply described finitely generated met\-abelian groups
that could be shown to be infinitely related was thus called for.

With this goal in mind, 
G. Baumslag considers in \cite{Bau76} the met\-abel\-ian top $G/G''$ of a finitely presented group $G$ 
whose deficiency is greater than 1. 
He proves that the multiplicator $H_2(G/G'',\Z)$ is infinitely generated 
whence $G/G''$ cannot be defined by finitely many relators.
(This implication follows from Hopf's formula for $H_2(-,\Z)$; 
see Formula (9.2) in \cite[Chapt.\;VI]{HiSt97}). 
\index{Baumslag, G.}

Baumslag reaches his conclusion in two steps.
First, he verifies (with the technique used in the proof of Proposition \ref{prp:Invariant-group-large-deficiency})
that $G$ maps onto the wreath product $W = \Z/p\Z \wr C_\infty$ 
of a finite cyclic group with order a prime $p$
by an infinite cyclic group $C_\infty$; the prime can be prescribed arbitrarily. 
The wreath product $W$ has an infinitely generated multiplicator (see, \eg, \cite[p.\;241]{LeRo04}).
Consider now the 5-term sequence
\[
H_2(G/G'', \Z) \to H_2(W,\Z) \to \Z \otimes_{\Z{Q}} N_{\ab} \to G_{\ab}  \to W_{\ab} \to 0
\]
associated to the extension $N \triangleleft G/G'' \epi W$ (see, \eg, \cite[Cor.\;8.2]{HiSt97}).
Since the finitely generated metabelian group $G/G''$ 
satisfies the ascending chain condition on normal subgroups, 
$N_{ab}$ is finitely generated as a $\Z{Q}$-module.
So the exactness of the above sequence at the term $H_2(W, \Z)$
and the fact that $H_2(W, \Z)$ is infinitely generated force $H_2(G/G'', \Z)$ to be likewise infinitely generated.
\smallskip
\index{Baumslag, G.}

The given argument is no longer sound if one replaces the metabelian top $G/G''$ of $G$ 
by a soluble quotient that is not metabelian, say by $G/G'''$, 
for the quotient may then not satisfy the maximal condition on normal subgroups.
Proposition \ref{prp:Infinitely-related-soluble-groups}, however,  can step in.
Here is a summary of the way towards this goal.

If one combines  Proposition \ref{prp:Infinitely-related-soluble-groups}
with Proposition \ref{prp:Invariant-group-large-deficiency},
one arrives at the following generalization of Baumslag's result:
\begin{crl}
\label{crl:Soluble-group-with-few-relators}
Let $p$ be a prime number and $\VV$ be a variety of groups 
that contains the metabelian variety $\AA_p\cdot \AA$,
but is not the variety of all groups. 

Suppose $G$ is a finitely presented group satisfying one of the assumptions (i), (ii) or (iii) listed in the statement of Proposition \ref{prp:Invariant-group-large-deficiency}
and that $G_{\ab}$ is infinite.
Then the canonical image $\bar{G} = G/\VV(G)$ of $G$ in $\VV$ is infinitely related.
\end{crl}
\index{Infinitely related groups!free soluble}
\index{Free soluble groups}

\begin{note}
\label{note:Soluble-group-with-few-relators}
Corollary \ref{crl:Soluble-group-with-few-relators} is an amplification of Theorem B in \cite{BiSt78}. 
It generalizes several earlier results, 
notably {\v{S}}mel$'$kin's theorem that a free soluble group $F/F^{d}$ 
of derived length $d\geq 2$ admits a finite presentation if, and only if it is cyclic (\cite{Sme65})
and Baumslag's Theorem F in \cite{Bau74} dealing with one-relator metabelian groups.
\end{note}
\index{Baumslag, G.}
\index{Bieri, R.}
\index{Smelkin, A. L.@{\v{S}}mel$'$kin, A. L.}
\index{Strebel, R.}
%
\subsection[Improved structure theorem for finitely generated indicable groups]%
{Improved structure theorem for fg indicable groups}
\label{ssec:Improved-structure-theorem-for-fp-indicable-groups}
%
Theorem \ref{thm:FP2-groups-and-HNN-extensions} asserts 
that every indicable group of type $\FP_2$ can be written 
as an HNN-extension with a finitely generated base group $B$
and finitely generated associated subgroups $S$ and $T$;
moreover, the base group $B$ can be required to be contained 
in the kernel $N$ of a given rank 1 character $\chi \colon G \to \R$.
If one looks at the proof of the theorem
one notices 
that the hypothesis that $G$ be of type $\FP_2$ is only used towards the end
and that the construction employed in the proof works for every finitely generated group.
The conclusion in the case of a finitely group will be weaker 
than that of Theorem \ref{thm:FP2-groups-and-HNN-extensions},
but, as detected by Y. de Cornulier and L. Guyot, 
it has useful applications,
notably in the context of \emph{condensation groups}.

In this last part of Section \ref{sec:Sigma1-via-HNN-extensions},
I explain how the idea of the proof of Theorem \ref{thm:FP2-groups-and-HNN-extensions} 
can be twisted  so as to yield the justification of a new structure theorem
and compare then the new result with its predecessor, 
Theorem \ref{thm:FP2-groups-and-HNN-extensions}.
\index{Condensation group}
\index{Cornulier, Y. de}
\index{Guyot, L.}
%
\subsubsection{Statement and proof of the theorem}
\label{sssec:Twisting-the-proof-of-Theorem-A-inBiSt78}
%
Let $G$ be a finitely generated indicable group and let $\chi \colon G \epi \Z \incl \R$ a rank 1 character.
Choose an element $t \in G$  with $\chi(t) = 1$.
Since $G$ is finitely generated, there exists a finite subset $\AA$ in  $N = \ker \chi$
such that $\AA \cup \{ t \}$ generates $G$.
The kernel $N$ will then be the normal closure of $\AA$ in $G$.

For $m> 0$, we next define a subgroup $B_m$ of $N$ 
and subgroups $S_m$, $T_m= tS_mt^{-1}$ of $B_m$,
similarly as we did  in the proof of Theorem \ref{thm:FP2-groups-and-HNN-extensions}.
We set
\[
B_m = \gp(\{ t^\ell a t^{-\ell}  \mid 0 \leq \ell \leq m \text{ and }   a \in \AA\}),
\]
and then $S_m = \gp(\{ t^\ell a t^{-\ell} \mid 0 \leq \ell < m \text{ and }   a \in \AA \})$ 
and $T_m = tS_mt^{-1}$. 
These subgroups enter in the construction of the HNN-extension
\[
G_m = \langle B_m, y_m \mid  y_m \cdot s \cdot y_m^{-1} = \tau_m(s)  \text{ for } s \in S_m \rangle,
\]
the isomorphism $\tau_m \colon S_m \iso T_m$ being given by conjugation by $t$. 
There exist unique epimorphisms $\rho_m \colon G_m \epi G_{m+1}$ and $\lambda_m \colon G_m \epi G$ 
that extend the inclusions $B_m \incl B_{m+1}$, respectively $B_m \incl N$, 
and map $y_m$ to $y_{m+1}$, respectively to $t$.
These epimorphisms satisfy the commutativity relation $\lambda_{m} = \lambda_{m+1} \circ \rho_m$
for every $m \geq 1$; 
they induce therefore an epimorphism
\[
\lambda_\infty \colon \colim\nolimits_{m \to \infty} G_m \epi G.
\]

Let $K_m$ denote the kernel of the limiting map $\lambda_m \colon G_m \epi G$ 
and consider the action of $G_m$ on the Bass-Serre tree $X_m$ associated to the HNN-extension $G_m$.
Since $\lambda_m$ is injective on the base group $B_m$,
the action of $K_m$ on $X_m$ is free and so $K_m$ is a free group.
We claim that $\lambda_\infty$ is an isomorphism. 
This amounts to show 
that every element $x \in K_0$ lies in the kernel of the composition 
$\pi_m = \rho_{m-1} \circ \cdots \circ \rho_0 \colon  G_0 \epi G_m$ for $m$ large enough. 
Since  $K_0 \subseteq N = \ker \chi$,
every $x \in K_0$ is a product of conjugates of elements $a \in \AA$ by powers of $t$;
hence a conjugate of $x$,  say $x' = t^n x t^{-n}$ with $n \geq 0$, 
will be mapped  into $B_m$ for all all sufficiently large $m$.
Since $\lambda_m$ is injective when restricted to $B_m$
and as $x' \in K_0 = \ker (\lambda_m \circ \pi_m)$ 
it follows that $\pi_m(x') = 1$ and hence $\pi_m(x) = 1$. 

Suppose now that, for some $j \geq 1$, the kernel $K_{j}$ is the normal closure of a finite set, 
say of $\FF$. 
Since $G$ is the colimit of the $G_m$ there exists then an index $h \geq j$ 
such that the canonical image of $\FF$ under the canonical epimorphism $G_{j} \epi G_{h}$ 
is trivial in $G_{h}$. 
The canonical isomorphism $(\lambda_{j})_* \colon G_j/K_j \iso G$ factors therefore through $G_h$
and so the epimorphism $\lambda_{h} \colon G_{h} \epi G$ is actually an isomorphism.

The previous reasoning constitutes a proof of 
\begin{thm}
\label{thm:Structure-theorem-improved}
Let $G$ be a finitely generated group and $\chi \colon G \epi \Z \incl \R$ a rank 1 character.
Choose $t \in G$ with $\chi(t) = 1$  and set $N = \ker \chi$.
Then $G$ is a colimit of a sequence of epimorphisms 
$\rho_m \colon G_m \epi G_{m+1}$ of finitely generated groups  $G_m = \gp(B_m, y_m)$
with the following properties:
\begin{enumerate} [a)]
\item $N$ is the ascending union $\bigcup\nolimits_{m \geq 0} B_m$ of finitely generated subgroups $B_m$,
\item each $B_m$ is generated by finitely generated subgroups $S_m$ and $T_m = tS_mt^{-1}$,
\item each $G_m$ is an HNN-extension of the form 
\[
\langle B_m, y_m \mid y_m \cdot  s \cdot y_m^{-1} = \tau_m(s) \text{ for } s \in S_m \rangle,
\]
\item each epimorphism $\rho_m \colon G_m \epi G_{m+1}$ is induced by the inclusion $B_m \incl B_{m+1} $
and the assignment  $y_m \mapsto y_{m+1}$, and each epimorphism $\lambda_m \colon G_m \epi G$
is induced by the inclusion $B_m \subset G$ and the assignment $y_m \mapsto t$.
\end{enumerate}
Then one of the following two statements holds:
\begin{enumerate}[(i)]
\item there is an index $m_0$ such that $\lambda_{m_0} \colon G_{m_0} \epi G$ is an isomorphism,
\item the kernel of every epimorphism  $\lambda_m \colon G_m \epi G$ is free of infinite rank
and infinitely generated as a normal subgroup.
\end{enumerate}
\end{thm}
\index{Structure theorem!for fg indicable groups}
\begin{note}
\label{note:Improved-structure-theorem}
Theorem \ref{thm:Structure-theorem-improved} and its proof are taken from Section 6  in \cite{BCGS12}.
\end{note}
\index{Cornulier, Y. de}
\index{Guyot, L.}
\index{Bieri, R.}
\index{Strebel, R.}
%
\subsubsection{Comparison of the new structure theorem 
with Theorem \protect{\ref{thm:FP2-groups-and-HNN-extensions}}}
\label{sssec:Comparison-new-and-old-structure-theorem}
%
Let $G$ and $\chi \colon G \epi \Z \incl \R$ be as in the statement of Theorem 
\ref{thm:Structure-theorem-improved}.
The group $G$ is then a colimit of a sequence of epimorphisms
\[
G_1 \overset{\rho_1}{\epi} G_2 \overset{\rho_2}{\epi} G_3 \overset{\rho_3}{\epi}\cdots
\]
of finitely generated groups. 
Choose a free group $F$ of finite rank that maps onto $G_0$, say $\pi \colon F \epi G_0$,
and set $R = \ker (\lambda_0 \circ \pi)$. 
Then $R$ is the ascending union of the normal subgroups 
$R_m = \ker(\rho_{m-1} \circ \cdots \rho_1 \circ \pi )$.
\enlargethispage{2mm}

Suppose now
that \emph{statement (i) in Theorem \ref{thm:Structure-theorem-improved} does not hold}.
Then there are infinitely many indices $m$ with $R_m \subsetneqq R_{m+1}$,
so $R$ cannot be the normal closure of a finite set 
and thus the group $F/R \iso G$ is infinitely related.

Next let $\RR$ be a finite subset of $R$ and let $S$ be the normal closure of $\RR$ in $F$.
Then  there exists an index $m_*$ 
so that $S \subseteq R_{m_*}$. 
Set $H = F/S$ and let $\pi_* \colon H = F/S \epi F/R \iso G$ be the obvious epimorphism;
its kernel $M$ is equal to $R/S$.
The extensions $M \triangleleft H \epi G$ and $K_{m_*} \triangleleft G_{m_*} \epi G$
then fit into the commutative diagram
\begin{equation} %
\label{eq:Third-diagram}
\xymatrix{
 0 \ar[r] & M  = R/S \ar[r] \ar@{>>}[d]^{\can} & H = F/S \ar[r] \ar@{>>}[d]^{\can} &   G \ar[r] \ar[d]^{=} & 0 \\
 0 \ar[r] & K_m = R/R_{m_*} \ar[r]                        & G_m = F/R_{m_*} \ar[r]                       &   G \ar[r] & 0.
}
\end{equation}
Its rows are exact. Let $K$ be a commutative ring (with $1 \neq 0$).
The left vertical epimorphism  induces  then an epimorphism
$K \otimes M_{ab} \epi K \otimes (K_m)_{\ab}$.

By statement (ii) in Theorem \ref{thm:Structure-theorem-improved}
the kernel $K_m$ is a infinitely generated free group;
its abelianisation is therefore a non-trivial free abelian group
and so the $K$-module $K \otimes M_{ab}$ is non-zero.
In view of Lemmata \ref{lem:Independence-free-presentation}  
and \ref{lem:Characterization-almost-fp} this conclusion implies
that $G$ is not of type $\FP_2$ over the ring $K$.
\smallskip

So far we have shown
that the group $G$ is infinitely related and that it is not of type $\FP_2$.
These findings have previously been obtained from a contraposition of Theorem 
\ref{thm:FP2-groups-and-HNN-extensions}; see Proposition \ref{prp:Infinitely-related-soluble-groups}.
But this assumption that statement (i) does not hold
has further consequences:
it shows that $G$ is the colimit of a sequence of groups $G_m$ having the property 
that each group $G_m$ contains a normal subgroup $K_m$ which is free of infinite rank.
This later property implies that $G_m$ has continuously many normal subgroups 
(by a basic result in the variety of groups, 
proved independently by Adian, Olshanskii and Vaughan-Lee in \cite{Adj70}, \cite{Ols70} and \cite{VaL70})
and that $G_m$ a condensation point in the space of marked groups (by \cite[Cor.\,5.4]{BCGS12}).
It then follows that the colimit $G$ is a condensation point in the space of marked groups.

Here is a summary of the insights obtained in the above:
\begin{crl}
\label{crl:Consequence-if-i-is-false}
Let $G$ be a finitely generated group that admits a rank 1 character $\chi \colon G \epi \Z \incl \R$ 
such that $G$ is \emph{not} an HNN-extension with finitely generated base group $B $ contained in $\ker \chi$ 
and finitely generated associated groups.

Then $G$ is not of type $\FP_2$ over a non-zero commutative ring 
and it is a condensation point in the space of marked groups.
\end{crl}

We close with a word on the verification 
that a rank 1 character does not satisfy statement $(i)$ in Theorem \ref{thm:Structure-theorem-improved}.
Let $\chi \colon G \epi \Z \incl \R$ be a rank 1 character with kernel $N$ and pick $t \in \chi^{-1}(\{1\})$.
Assume $\chi$ and $-\chi$ represent points outside of $\Sigma^1(G)$ 
and consider one of the approximating
HNN-extensions $G_m$ constructed in the proof of \ref{thm:Structure-theorem-improved}.
Then the characters $\pm \chi \circ \lambda_m$ represent points outside of $\Sigma^1(G_m)$, 
so $G_m$ cannot be an ascending HNN-extension 
(see Corollary \ref{crl:Sigma1-epimorphism}) and  thus $G_m$ contains a non-abelian free subgroup.
If $G$ contains no such subgroup then $\lambda_m$ cannot be an isomorphism; 
as this argument is valid for every index $m$, we have proved that statement (i) does not hold.

The hypothesis that $G$ contain no non-abelian free subgroup is a familiar one 
in the applications of the $\Sigma$-theory; 
see, \eg, Theorem \ref{thm:Consequence-finite-presentation}
or Proposition \ref{prp:Infinitely-related-soluble-groups}.
Recently weaker substitutes have been found.
They are listed in Proposition 6.6 in \cite{BCGS12}
and lead to interesting applications; see Examples 6.7 and 6.11 in the cited paper.
\index{Cornulier, Y. de}
\index{Guyot, L.}
\index{Bieri, R.}
\index{Strebel, R.}
%
 
%
%
\section{Computation of $\Sigma^1$ for one relator groups}
\label{sec:Invariant-one-relator-group}
%
%
\enlargethispage{2mm}
In this section, 
we calculate the invariant of a group  with $m$ generators and a single defining relator $r$.
The crucial case is that where $m=2$ 
and $r$ is a cyclically reduced word which involves both generators.
In fact,
if $m = 1$, the group $G$ is cyclic and $\Sigma^1(G) = S(G)$;
if $m>2$ the invariant is empty by part (i) of  Proposition \ref{prp:Invariant-group-large-deficiency};
if, finally,  $m=2$ and $r$ is  a power of one of the generators, say if $r = a^\ell$,
the group $G$ is cyclic if $|\ell| = 1$ and a non-trivial free product otherwise.
In the first case, $\Sigma^1(G) = S(G) $ consists of two points, in the other case it is empty
(see Example \ref{examples:Groups-with-non-trivial-centre} a) for the first case
and Example 3 in section \ref{sssec:Sigma1-first-examples} 
or part (ii) of  Proposition \ref{prp:Invariant-group-large-deficiency} for the second one).

The determination of the invariant in the remaining case is more challenging.
It is due to Ken Brown (\cite{Bro87b}) and described in
\begin{thm}
\label{thm:Sigma1-one-relator-group}
Let $G$ be a group given by a presentation $\eta_{*} \colon \langle a, b, \mid r \rangle$
where $r= s_{1}\cdots s_{k}$ is a cyclically reduced, non-empty word involving both generators.
Then a non-zero character $\chi \colon G \to \R$ represents a point of $\Sigma^1(G)$ 
if, and only if, the sequence 
\begin{equation}
\label{eq:Sequence-f-one-relator}
f_r(\chi) = \left(\chi(s_{1}), \chi(s_{1}s_{2}),  \ldots \chi(s_{1}\cdots s_{k}) \right)
\end{equation}
satisfies the following condition:
\begin{equation}
\label{eq:Condition-Sigma1-one-relator-group}
\left.
\begin{gathered}
\text{if one of } \chi(a)  \text{ and } \chi(b) \emph{ is zero, } 
f_r(\chi) \text{ assumes its minimum  twice,}\\
\text{otherwise } f_r(\chi)  \emph{ assumes its minimum  once.}
\end{gathered}
\right\}
\end{equation}
\end{thm}
\index{Computation of Sigma1@Computation of $\Sigma^1$ for!one relator groups}
\index{Brown's algorithm}
\index{Brown, K. S.}

\begin{remarks}
\label{remarks:Sigma1-one-relator-group}
a) Condition \eqref{eq:Condition-Sigma1-one-relator-group} is the defining property of the set $\psi(\{r\})$
(see Definitions  \ref{definition:psi(R)} and \ref{definition:psi(R)-2-generators}).
So the conclusion of Theorem \ref{thm:Sigma1-one-relator-group} can be restated by saying 
that the subset $\psi(\{r\})$ coincides with $\Sigma^1(G)$.
\index{Invariant Sigma1@Invariant $\Sigma^1$@lower bound psi@lower bound $psi(\RR)$}

b)  Theorem \ref{thm:Sigma1-one-relator-group} will be obtained by combining four ingredients:
the algebraic version of the $\Sigma^1$-criterion 
embodied in the subset $\psi(\{r\}) \subseteq \Sigma^1(G)$,
the standard method of analyzing one-relator groups due to W. Magnus \cite{Mag30},
Proposition \ref{prp:HNN-extension-and-Sigma1} and the openness of $\Sigma^1(G)$.
\end{remarks}

\subsection{Proof of Theorem \ref{thm:Sigma1-one-relator-group}}
\label{sec:Proof-Theorem-Sigma1-one-relator-group}
%
If the sequence $f_r(\chi)$ satisfies condition \eqref{eq:Condition-Sigma1-one-relator-group}
then $[\chi]$ belongs to the subset $\psi(\{r\})$ (see Definition \ref{definition:psi(R)-2-generators})
and thus lies in $\Sigma^1(G)$ by Proposition \ref{prp:psi(R)-subset-Sigma1}. 
For the proof of the  converse,  three cases will be distinguished.
We shall first treat the case where $\chi$ vanishes on one of the generators,
then the case of a character that is non-zero on both generators but has rank 1.
Finally, characters of rank 2 will be considered.
%
\subsubsection{Case 1}
\label{sssec:Proof-Sigma1-one-relator-group-Case1}
%
If  $\chi$ vanishes on one of the generators,
it has rank 1 and four, very similar subcases arise;
it will suffice to analyze the situation of the character $\chi\colon G \epi \Z \incl \R$ 
with $\chi(a) = 1$  and $\chi(b) = 0$.
Let $F$ denote the free group on the generators $t = a$ and $b$, 
let $\eta_{*} \colon F \epi G$ be the epimorphism induced  by $\eta \colon \{t,b\} \to G$. 
Set $N =  \ker \chi$ and $ U = \ker ( \chi \circ \eta_{*}) =  \eta_{*}^{-1} (N)$.

The normal subgroup $U$ is freely generated by the conjugates $b_{j} = t^j bt^{-j}$ of $b$ with $j \in \Z$.
The relator is a product of conjugates of $b$ and of $b^{-1}$
and can thus be written as a cyclically reduced word $w$ in the letters $b_{j}$ and $b_{j}^{-1}$;
let $\mu$ be the smallest and $M$ the largest of the indices that occur 
when $r$ is written as a word $w$ in these letters; 
since $r$ involves both generators, $\mu < M$.
Define subgroups $B$, $S$ and $T$ by setting
\[
B  = \gp(\{\eta_{*}(b_{j}) \mid \mu \leq j \leq M \}), \quad
S   = \gp(\{\eta_{*}(b_{j}) \mid \mu \leq j < M \}) 
\]
and $ T   = \gp(\{\eta_{*}(b_{j}) \mid \mu < j \leq M \}) = \mu( S )$.
By the Freiheitssatz (see, e.\,g., \cite[p. 198]{LS77}),
the groups $S$ and $T$ are free on the displayed generators,
whence $G$ is an HNN-extension with base group $B$, associated subgroups $S$, $T$ and stable letter $t$.
The hypothesis that  $[\chi]$ lies in $\Sigma^1(G)$
and Proposition \ref{prp:HNN-extension-and-Sigma1} allow one to deduce
that the base group $B$ coincides with the associated subgroup $T$.
\index{Magnus, W.}

By going back to the definition of $B$ and $T$ 
one sees that the generator $b_{\mu}$ can be omitted from the given set of generators of $B$.
Thus there exists a word $v$ in the generators $b_{\mu + 1}$, \ldots, $b_{M}$  
such that $b_{\mu}$ and $v$ define the same element in $B$.
But $B$, by the proof of the Freiheitssatz, is a group 
having  $\BB = \{b_{\mu}, \ldots,  b_{M} \}$ as its set of generators and a single defining relator $w$;
here $w$ denotes a word 
obtained by rewriting the original relator $r$ as a (cyclically reduced) word in the generators $b_{j}$.
So $B$ can be written as quotient $\Phi/R$ of the free group on $\BB$
modulo the normal subgroup $R$ generated by $w$.
By the previous paragraph $B$ can also be written as $\Phi/R_{1}$
where $R_{1} \triangleleft \Phi$ denotes the normal subgroup generated by $w_{1} = b_{\mu}\cdot v^{-1}$.
Since $R_{1}\subseteq R$
the identity on $\Phi$ induces an epimorphism $\rho \colon \Phi/R_{1} \epi \Phi/R$;
as the groups $\Phi/R_{1}$ and $\Phi/R$ are both isomorphic to the finitely generated free group $T$
and free groups of finite rank are hopfian (see, \eg, \cite[Sec.\;2.4, Thm.\;2.13]{MKS}),
the epimorphism $\rho$ is injective, whence $R_{1}$ equals $R$.

At this point, we invoke another theorem of W. Magnus',
the \emph{Conjugacy Theorem for Groups with One Defining Relator} (see Theorem 4.11 in \cite[p. 261]{MKS}).
It guarantees 
that the relator $w_{1} = b_{\mu} \cdot v^{-1}$ is a cyclic permutation of $w$ or of $w^{-1}$.
But if so,
the generator $b_{\mu}$ occurs only once in $w$, either as $b_{\mu}$ or $b_{\mu}^{-1}$.
This fact can be restated in terms of the original relator $r$ by saying 
that the minimum of the sequence $f_r(\chi)$ occurs twice,
once before an occurrence of $b^\pm$ and once afterwards.
(If the relators starts out with $b^\pm$ and the minimum is 0, the preceding conclusion is to be interpreted as saying that the minimum occurs after the first letter and at the end of the relator.)
Thus $f_r(\chi)$ satisfies condition \eqref{eq:Condition-Sigma1-one-relator-group}.
%
\subsubsection{Case 2}
\label{sssec:Proof-Sigma1-one-relator-group-Case2}
In Case 2,
the  character $\chi \colon G \to \R$ has rank 1 and takes non-zero values on both generators.
Its image can be assumed to be $\Z$; set $p = \chi(a)$ and $q = \chi(b)$.
Our aim is to reduce this set-up to that considered in Case 1.
To do so,
we adjoin to $G$ a $p$-th root of $a$ and a $q$-th root of $b$ ,
ending up with a group $\widetilde{G}$ with generating set $\{x,y\}$,
a character $\tilde{\chi} \colon \tilde{G} \epi \Z \incl \R$ and a defining relator $\tilde{r}$.
We then verify that $[\tilde{\chi}] \in \Sigma^1(\tilde{G})$,
introduce a new generating system $\{x, z = yx\}$,
express the relator $\tilde{r}$ in terms of  the new generators,
reduce it and arrive thus at a relator of the type studied in Case 1.
A comparison with Case 1 will then disclose
that the sequence $f_{\tilde{r}} (\tilde{\chi})$ assumes its minimum only once
and to conclude 
that the original sequence $f_r(\chi)$ has this property, too.
\smallskip

The details of the outlined verification are as follows.
By replacing, if need be, the defining relator $r$ with a cyclic permutation of itself,
one can assume that $r$ has the form
\begin{equation}
\label{eq:Original-relator}
r = a^{e_{1} } b^{f_{1}} \cdot  a^{e_{2} } b^{f_{2}}   \cdots  a^{e_{\ell} } b^{f_{\ell}}
\end{equation} 
where each of the exponents $e_{j}$, $f_{j}$ is non-zero.
Adjoin a $p$-th root $x$ of $a$ to $G$, obtaining a group with the presentation
\[
G_{1} = \langle a, b, x, \mid r(a,b), a \cdot x^{-p} \rangle \isoinv \langle x, b \mid r(x^{a}, b) \rangle.
\]
Then adjoin a $q$-th root $y$ of $b$ to $G_{1}$, ending up with the group
\begin{equation}
\label{eq:Group-with-adjoined-roots}
\tilde{G} = \langle x, y \mid \tilde{r}(x,y) \rangle
\quad  \text{with} \quad 
\tilde{r}(x,y)  = r(x^{a},y^b) =  x^{pe_{1} } \cdot y^{qf_{1}}  \cdots  x^{pe_{\ell} } \cdot y^{qf_{\ell}}.
\end{equation}
The character $\chi$ of $G$ extends uniquely to a character $\tilde{\chi}$ of $\tilde{G}$;
its image is $\Z$.
By hypothesis, $[\chi]$ belongs to $\Sigma^1(G)$  
and $\chi$ is non-zero on $a$ as well as on $b$.
Proposition \ref{prp:Sigma-1-join-subgroups} and the facts
that the invariants of the infinite cyclic groups $\gp(x)$ and $\gp(y)$ 
coincide with the spheres $S(\gp(x))$ and $S(\gp(y))$, respectively,  
thus allow one to deduce that $\tilde{\chi}$ represents a point of $\Sigma^1(\tilde{G})$.
\index{Computation of Sigma1@Computation of $\Sigma^1$ for!join of subgroups} 

The function $f_{\tilde{r}} (\tilde{\chi})$ arises from the function $f_r(\chi)$ by ``affine interpolation''.
If $f_{\tilde{r}} (\tilde{\chi})$ assumes its minimum exactly once so will therefore the function $f_r(\chi)$.
To justify the claim in Case 2, it suffices therefore to establish 
\begin{lem}
\label{lem:Special-case-for-case-2}
Assume $L$ is a one-relator group with generators $x$, $y$ and defining relator of the form
\begin{equation}
u = x^{m_1} y^{n_1} \cdots x^{m_\ell} y^{n_\ell}
\end{equation}
where $\ell > 0$ and all exponents $m_i$ and $n_i$ are non-zero.
Let $\psi \colon L \to \R$ be the character that sends $x$ and $y$ to 1.
If $[\psi] \in \Sigma^1(L)$ then the function $f_u(\psi)$ assumes its minimum only once.
\end{lem}

\begin{proof}
By replacing $u$ with a suitable cyclic conjugate and then, if need be, with the inverse of this conjugate,
we can arrange that the minimum of the function $f_u(\psi)$ is 0 
and that the new relator begins with $x$ and ends with $y^{-1}$.
We thus assume that $m_1 > 0$ and $n_\ell < 0$ and have to show 
that the $\psi$-value of no proper initial segment $s_1 \cdots s_h$ of $u$ is 0.

To arrive at this goal, we introduce a new basis $\{x, z\}$ in the free group on $\{x,y\}$.
Put $z = y \cdot x^{-1}$.
Then $\psi$ vanishes on $z$ and $y = z \cdot x$.
Upon substituting $z\cdot x$ for $y$ in the given relator $u$
one arrives at the new relator
\begin{equation}
\label{eq:Transformed-relator}
\hat{u}   = x^{m_{1} } (z\cdot x)^{n_{1}} \cdots  x^{m_{\ell} } (z\cdot x)^{n_{\ell}}.
\end{equation}
This relator is not freely reduced and so we have to determine
the relationship between $\hat{u}$ and the word $\hat{u}_{red}$ 
that is obtained from $\hat{u}$ by free reduction.
To see where the crux lies,
we first look at
\begin{example}
\label{example:Study-free-reduction}
Consider the relator
$u = xYxxYxYxyXYXyxyXYY$
where $X$, $Y$ denote the inverses of $x$ and $y$.
The graph of the sequence $f_u(\psi)$ looks then as depicted in Figure 
\ref{fig:Graph-function-word-u}.
\begin{figure}[htb]
\psfrag{1}{\hspace*{-1.8mm} \scriptsize $1$}
\psfrag{2}{\hspace*{-1.5mm}  \scriptsize $2$}
\psfrag{3}{\hspace*{-1.3mm} \scriptsize $3$}
\psfrag{4}{\hspace*{-0.1mm}\scriptsize $4$}
\psfrag{5}{\hspace*{-1.4mm} \scriptsize $5$}
\psfrag{6}{  \hspace*{-3mm}   \scriptsize $6$}
\psfrag{7}{\hspace*{-1.2mm} \scriptsize $7$}
\psfrag{8}{\hspace*{-1.2mm} \scriptsize $8$}
\psfrag{9}{\hspace*{-1.3mm} \scriptsize $9$}
\psfrag{10}{\hspace*{-2mm} \scriptsize $10$}
\psfrag{11}{\hspace*{-2.5mm} \scriptsize $11$}
\psfrag{12}{  \hspace*{-2.0mm}\scriptsize $12$}
\psfrag{13}{  \hspace*{-2.5mm}\scriptsize $13$}
\psfrag{14}{\hspace*{-1.3mm}\scriptsize $14$}
\psfrag{15}{\hspace*{-1.2mm}\scriptsize $15$}
\psfrag{16}{  \hspace*{-3mm}   \scriptsize $16$}
\psfrag{17}{\hspace*{-2.0mm} \scriptsize $17$}
\psfrag{18}{\hspace*{-2.0mm} \scriptsize $18$}
\psfrag{x}{\hspace*{-2.2mm} \small $x$}
\psfrag{X}{  \hspace*{-1.8mm}  \small $X$}
\psfrag{y}{\hspace*{-2.5mm}  \small $y$}
\psfrag{Y}{\hspace*{-1.0mm} \small $Y$}
\psfrag{ps}{\hspace*{-4.5mm} \small $f_u(\psi)$}
\psfrag{j}{\hspace*{-2mm} \small $j$}
\begin{center}
\includegraphics[width = 12cm]{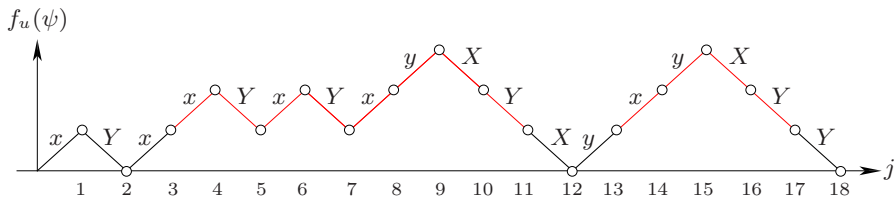}
\caption{Graph of the function $f_u(\psi)$}
\label{fig:Graph-function-word-u}
\end{center}
\end{figure}

Now replace $y$ and $Y$ by $zx$, $XZ$, respectively; 
there results the word $\hat{u}$. 
This word is not reduced; if one reduces it, one obtains the word $\hat{u}_{red}$.
Its graph is shown in Figure  \ref{fig:Graph-function-word-u-transformed}.
\begin{figure}[htb]
\psfrag{1}{\hspace*{-1.2mm} \scriptsize $1$}
\psfrag{2}{\hspace*{-1.5mm}  \scriptsize $2$}
\psfrag{3}{\hspace*{-1.3mm} \scriptsize $3$}
\psfrag{4}{\hspace*{-0.1mm}\scriptsize $4$}
\psfrag{5}{\hspace*{-1.4mm} \scriptsize $5$}
\psfrag{6}{  \hspace*{-3mm}   \scriptsize $6$}
\psfrag{7}{\hspace*{-1.2mm} \scriptsize $7$}
\psfrag{8}{\hspace*{-1.2mm} \scriptsize $8$}
\psfrag{9}{\hspace*{-1.3mm} \scriptsize $9$}
\psfrag{10}{\hspace*{-2mm} \scriptsize $10$}
\psfrag{11}{\hspace*{-2.5mm} \scriptsize $11$}
\psfrag{12}{  \hspace*{-2.0mm}\scriptsize $12$}
\psfrag{13}{  \hspace*{-2.5mm}\scriptsize $13$}
\psfrag{14}{\hspace*{-1.3mm}\scriptsize $14$}
\psfrag{15}{\hspace*{-1.2mm}\scriptsize $15$}
\psfrag{16}{\hspace*{-2mm}   \scriptsize $16$}
\psfrag{17}{\hspace*{-1.8mm} \scriptsize $17$}
\psfrag{x}{\hspace*{-2.2mm} \small $x$}
\psfrag{X}{  \hspace*{-1.8mm}  \small $X$}
\psfrag{z}{\hspace*{-1.5mm}  \small $z$}
\psfrag{Z}{\hspace*{-1.5mm} \small $Z$}
\psfrag{ps}{\hspace*{-4.5mm} \small $f_{\hat{u}_{red}}(\psi)$}
\psfrag{j}{\hspace*{-2mm} \small $j$}
\begin{center}
\includegraphics[width = 12cm]{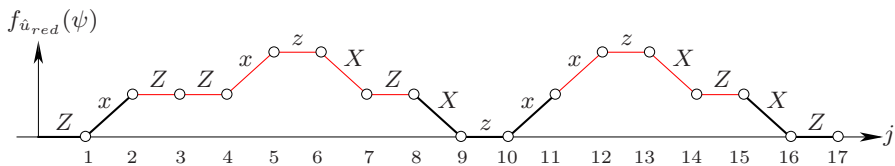}
\caption{Graph of the function $f_{\hat{u}_{red}}(\psi)$}
\label{fig:Graph-function-word-u-transformed}
\end{center}
\end{figure}

A comparison of the two graphs leads to the following observations.
The domain $\{1,2, \ldots, 18\}$ of the function $f_u(\psi)$ 
has five consecutive subintervals of maximal lengths;
these subintervals are of two types:
those of the first type are characterized by the fact 
that the function $f_u(\psi)$ assumes on them only the values 0 and 1
whereas $f_u(\psi)$ assumes at least once the value 2 on the intervals of the other type. 
The graphs of the function on these subintervals is drawn alternatively in black and in red.
Let $v_1$, $w_1$, $v_2$, $w_2$ and $v_3$ be the subwords that give rise to these intervals.
More precisely, 
let $v_1$ be the longest initial segment of $u$ 
on which the function takes only the values  0 and 1.
Next, let $v_1w_1$ be the longest initial segment such that $f_u(\psi)$ takes on the second part 
only values in $\{1,2, \ldots \}$ 
and define the subwords $v_2$, $w_2$ and $v_3$ similarly.

If the occurrences of the generator $y$ in these subwords are replaced by $zx$,
one obtains the words $\hat{v}_1$, $\hat{w}_1$,  \ldots, $\hat{v}_3$; they need not be reduced.
Let  $(\hat{v}_1)_{red}$, $(\hat{w}_1)_{red}$, \ldots, $(\hat{v}_3)_{red}$ be their reductions.
The crucial observations are now these:
the word $v_1$ has the form $v_i = x (Yx)^k$ with $k \geq 0$ 
and  $\hat{v}_i = x (XZx)^k$, whence $(\hat{v}_i)_{red}  = Z^kx$.
The word $v_2$ has the form $v = (Xy)^k$ with $k > 0$ and $\hat{v}_{red} = X z^k x$. 
For future use, we note that the word $v' =(Yx)^k$ leads to $\hat{v}'_{red} = X Z^kx$.
The word $v_3$, finally, has the form $v_\ell = (Yx)^kY$ with $k \geq 0$
and  $(\hat{v}_\ell)_{red} = XZ^{k+1}$.
Consider now the word $w = w_j$ with $j \in \{1,2\}$. 
Then $f_w(\psi)$ is nowhere negative and has the value 0 on $w$.
It follows that the functions $f_{\hat{w}}(\psi)$ and $f_{\hat{w}_{red}}(\psi)$ are nowhere negative, 
whence \emph{the word $\hat{w}_{red}$ cannot start with $x^{-1}$ and cannot end in} $x$.

Now to the upshot of these observations. 
The word $(\hat{v}_1)_{red}\cdot (\hat{w}_1)_{red}\cdots (\hat{v}_3)_{red}$
has the form
\[
Z^{k_i}x \cdot (\hat{w}_1)_{red} \cdot X Z^{\pm k} x \cdot (\hat{w}_2)_{red} \cdot  XZ^{k_\ell+1}
\]
with $k_i \geq 0$, $k > 0$ and $k_f \geq0$.
By the statement in italics, this product is therefore reduced as written; 
actually, it is even cyclically reduced.
The function $f_u(\psi)$ assumes the minimum 0 in the first subinterval, 
in the third subinterval and in the last one
and the transformed function $f_{\hat{u}_{red}}(\psi)$ has the analogous property.
\end{example}

We are now ready to deal with an arbitrary relator $u$ of length at least 2.
We get first rid of the exceptional case where $u = (xY)^k$ for some $k > 0$. 
Then $f_u(\psi)$ assumes $k$ times the value 0.
Since $\hat{u}_{red} = Z^k$ 
the function $f_{\hat{u}_{red}}(\psi)$ assumes its minimum also $k$ times.
In view of Case 1, the assumption that $[\psi] \in \Sigma^1(L)$ implies therefore that $k = 1$,
or alternatively, that $f_u(\psi)$ assumes its minimum only once. 

Suppose now that $u$ is not of the form $(xY)^k$
and subdivide the domain $\{1,2, \ldots, m\}$ of $f_u(\psi)$ into subintervals $v_1$, $w_1$, \ldots, $v_\ell$ 
as explained in the example. Three cases arise.
If $v_1 = x(Yx)^{k_1}$ and $k_1 > 0$ 
the reduced word $\hat{u}_{red}$ starts with $Z^{k_1}x$  and ends in $X Z^{k_\ell + 1}$,
whence the function $f_{\hat{u}_{red}}(\psi)$ assumes its minimum at least $k_1 + (k_\ell + 2) \geq 3$ times.
Case 1 then allows one to conclude that $[\chi] \notin \Sigma^1(L)^c$.
One sees similarly that $[\chi] \notin \Sigma^1(L)^c$ whenever $k_\ell > 0$.

In the third and last case, one has $v_1 = x$ and $v_\ell = Y$.
If $\ell = 2$ then $u = v_1 w_1 v_\ell$ and so $f_u(\psi)$ assumes its minimum only once;
if however, $\ell > 2$ 
the subword  $v_1 \cdot w_1 \cdot v_2 = x \cdot w_1 \cdot (Xy)^{\pm k_2}$ of $u$
will force the function $f_u(\psi)$ to have at least 2 absolute minima.
The transformed word $\hat{u}_{red}$ will then start with $x \cdot (\hat{w}_1)_{red} \cdot X z^{\pm k_2}x$;
since $k_2 > 0$ 
the function $f_{\hat{u}_{red}}(\psi)$ assumes therefore the values 0 at least $(1 + k_2) + 2 \geq 4$ times 
and so $[\chi] \notin \Sigma^1(L)^c$.
This completes the verification of Lemma \ref{lem:Special-case-for-case-2} 
and establishes Case 2 in the proof of Theorem \ref{thm:Sigma1-one-relator-group}.
\end{proof}
%
\subsubsection{Case 3}
\label{sssec:Proof-Sigma1-one-relator-group-Case3}
%
We want to show that every rank 2 point of $\Sigma^1(G)$ is in the subset $\psi(\{r\})$ 
defined by condition  \eqref{eq:Condition-Sigma1-one-relator-group}.
The proof of this assertion is fairly easy and is best understood 
if interpreted in the geometric set-up described in section \ref{sssec:Geometrical-reformulation}.

Rank 2 points can only occur if $G_{\ab}$ is free abelian of rank 2.
Let $\vartheta \colon G \epi  \Z^2$ be the epimorphism of $G$ that sends the couple $(a,b)$ 
to the  standard basis $((1,0),(0,1))$ of the Euclidean plane $\R^2$ equipped with the usual scalar product.
Then $\vartheta$ induces an isomorphism
\begin{equation*}
\sigma(\vartheta) \colon \s^{1} \iso  S(G), 
\quad 
u \longmapsto [g \mapsto \langle u, \vartheta(g)\rangle]
\end{equation*}
of circles (see section \ref{sssec:Coordinates-sphere}).
The relator $r = s_{1}s_{2}\cdots s_{k}$ of $G$ gives rise to  a sequence
\begin{equation}
\label{eq:Sequence-of-vertices}
\left( \vartheta(s_{1}), \vartheta(s_{1}) + \vartheta(s_{2}),  \ldots, \vartheta(s_{1}) + \cdots + \vartheta(s_{k}) \right)
\end{equation}
of lattice points.
The points of this sequence are the vertices of a path $\bar{p}$ in the Cayley graph 
$\Gamma(\Z^2, \{a,b\})$ that ends in the origin;
we shall think of $\bar{p}$ as being a loop, starting at the origin and also ending there.

The  determination of the subset $\psi(\{r\})$ amounts now to this.
One seeks unit vectors $u  = (u_{1}, u_{2}) \in \s^1$ 
for which the function $h_{u} \colon x \mapsto \langle u, x\rangle$
assumes its minimum at most twice along the vertices of the  path $\overline{p}$, 
the vertices being counted with multiplicity.
As an aid in finding these unit vectors,
construct the convex hull of the set of vertices  $v_{j} = \vartheta(s_{1} \cdots s_{j})$ of $\bar{p}$.
The boundary  $\CC$ of this hull is a closed, convex polygon 
whose vertices form a subset of the vertices of the path $\bar{p}$.
Call a vertex $v$ of $\CC$ \emph{simple}
if it equals $v_{j}$ for exactly one $j$.
Since the relator $r$ has at least length 4,
the polygon $\CC$ always contains horizontal edges at both the top and the bottom 
and a vertical edge on either side.
Let $e$ be one of these four edges.
If $e$ contains exactly two vertices,
call it \emph{special}.
(Note that a special edge is necessarily of length 1 and that its two vertices are simple.)
\index{Simple vertex}
\index{Special edge}

The proof of Case 3 can now be completed like this.
Each edge $e$ of the polygon $\CC$ determines a unit vector $u_{e} \in \s^1$ 
that is orthogonal to the line $\ell_{e}$ carrying the edge $e$ and has the property that 
the function $v_{j} \mapsto \langle u, v_{j}\rangle$ assumes its minimum in the vertices lying on $e$.
This unit vector $u_{e}$ has  rank 1.
A unit vector $u$ of rank 2 lies therefore in the complement of the set $\{u_{e}\mid e \text{ edge of } \CC \}$,
and thus belongs to an open arc $\alpha$ bounded by endpoints $u_{e}$ and $u_{e'}$, say.
Let $v_{j}$ be the common end point of the edges $e$ and $e'$.
If $v_{j}$ is simple, the arc $\alpha$ is contained in $\psi(\{r\})$ and hence in $\Sigma^1(G)$;
otherwise, $\alpha \cap \Sigma^1(G)$ contains no rank 1 point by Case 2 and hence no rank 2 points either,
rank 1 points being  dense in $\alpha$ (see Lemma \ref{lem:Density-rank-1-points}).

\begin{remarks}
\label{remarks:Browns-theorem}
a) The above proof of Brown's Theorem \ref{thm:Sigma1-one-relator-group} 
is close to that given in \cite[Section 4]{Bro87b}, 
the main differences being more detailed verifications 
and the fact 
that references to \emph{HNN-valuations} have been replaced 
by references to the characterization of rank 1 points 
in terms of ascending HNN-extension with finitely generated base group 
(Proposition \ref{prp:Ascending-HNN-extension})
and references to Proposition \ref{prp:psi(R)-subset-Sigma1}.
In \cite[Section 7]{BiRe88}, 
Bieri and Renz give a different proof of Theorem \ref{thm:Sigma1-one-relator-group}. 
\index{Brown, K. S.}
\index{Bieri, R.}
\index{Renz, B.}

b) The core of Brown's result can be stated without using the invariant $\Sigma^1$:
\emph{given an epimorphism $\chi \colon G \epi \Z$ and $t \in G$ with $\chi(t) = 1$,
the sequence of values of $\chi$ on the initial segments $s_{1} \cdots s_{j}$ 
of the relator $r = s_{1}\cdots s_{\ell}$ allows one to decide 
whether $G$ is an ascending HNN-extension $\langle B, t \mid B \subseteq tBt^{-1} \rangle$
with finitely generated base group $B$}.

c) Brown's Theorem yields an algorithm for deciding 
whether the kernel of a rank 1 character $\chi \colon G \to \R$ is finitely generated.
This algorithm has found striking applications in 3-manifold theory;
see \cite{Dun01}, \cite{But05b} and \cite{DuTh06a}. 
\index{Button, J. O.}
\index{Dunfield, N. M.}
\index{Thurston, D. P.}

d) Let $G = \langle a, b \mid r \rangle $ be a one relator group 
and $\chi \colon G \to \R$ a rank 1 character with kernel $N$.
Pull $\chi$ back to a character $\tilde{\chi} \colon F \to \R$ of  the free group $F$ on $a$, $b$.
Then $F_{\ab}$ admits a basis $(\bar{s}, \bar{t})$ 
such that $\ker \tilde{\chi} \cdot F'$ is generated by $s  \cdot F'$.
Now every automorphism of $F_{ab}$ is induced by an automorphism of $F$
(see, \eg, \cite[Chapt.\;I, Prop.\;4.4]{LS77});
so the group $F$ admits an ordered basis $(s,t)$
which projects to the previously chosen basis $(\bar{s}, \bar{t})$  of $F_{\ab}$.
It follows that $G$ has a one relator presentation with generators $s$, $t$  
such that the kernel $N$ of $\chi$ is the normal closure of the generator $s$.

Assume now that $[\chi] \in \Sigma^1(G)$. 
The analysis in Case 1 of the proof of Theorem \ref{thm:Sigma1-one-relator-group}, 
but with generators  $s$, $t$ instead of $a$, $b$, then shows 
that $G$ is an ascending HNN-extension whose base group is free (of finite rank). 
Moreover,  if both $[\chi]$ and $-[\chi]$ lie in $\Sigma^1(G)$ the kernel $N$ of $\chi$ is free.
The preceding reasoning implies 
that every finitely generated kernel of a rank 1 character of a one relator group is free.
As shown by R. Bieri in \cite[Cor.\;B]{Bi07}, 
this conclusion continues to be valid if $G$ is a finitely presented group of deficiency 1.
\index{Bieri, R.}
\index{Groups!of deficiency 1}
\end{remarks}
 %
 \subsection{Some examples}
\label{ssec:Examples-Sigma1-one-relator-group}
%
The following two examples have different objectives:
the first illustrates the actual computation of $\Sigma^1$ by means of Brown's algorithm;
the second one reinterprets a result of G. Baumslag in the frame work of the theory of the invariants Sigma.
\begin{example}
\label{example:Sigma1-one-relator-group-knots-K8n20-K8n21}
\index{Computation of Sigma1@Computation of $\Sigma^1$ for!group of knot K8n20@group of knot \texttt{K8n20}}
Figure \ref{fig:Diagrams-K8n20-K8n21}  shows diagrams  of knots of types $8_{20}$  and $8_{21}$,
as listed in \texttt{KnotInfo}
\footnote{This data base, created and maintained by C. Livingston,
is available at the URL  \url{http://www.indiana.edu/~knotinfo}, January 31, 2013.};
let  $G_{20}$ and $G_{21}$ denote the corresponding groups.
J. Weeks' program \texttt{SnapPea} 
\footnote{Available at the URL \url{http://www.geometrygames.org/weeks/}}
is able to find one-relator presentations for them.
For $G_{20}$
one gets generators $a$, $b$ and the defining relator
\begin{equation}
\label{eq:Relator-K8n20}
r_{20}  = a^{2} \cdot b^{2} \cdot a \cdot b^{-2} \cdot a^{-1} \cdot b\cdot a^{-1} \cdot b^{-2}a \cdot b^{2}
\end{equation}
The exponent sums $(\sigma_{a}, \sigma_{b})$ of $r_{20}$ are  $(2,1)$;
so $G_{20}$ admits a rank 1 character $\chi$ with $\chi(a) = 1$ and $\chi(b) = -2$.
The sequence $f_{r_{20}} ( \chi)$ reads thus
\[
(1, 2,0, -2, -1, 1, 3, 2, 0, -1, 1, 3, 4, 2, 0).
\]
Its minimum is $-2$ and it is achieved only once.

The sequence $f_{r_{20}} ( -\chi)$ is the negative of $f_{r_{20}} (\chi)$;
so $f_{(r_{20}, -\chi)}$ assumes its minimum once if the maximum of $f_{r_{20}} (\chi)$ is achieved once;
this latter condition is fulfilled.
All taken together, 
we see that $\Sigma^1(G) = S(G)$;
so the commutator subgroup of $G_{20}$ is finitely generated 
by  Corollary  \ref{crl:Characterizing-fg-kernel-rank-1-character}.
Remark \ref{remarks:Browns-theorem}d allows one to say more: $G'_{20}$ is a free group of finite rank.
\begin{figure}[htb]
\psfrag{k20}{\hspace*{-2mm} \small $8_{20}$}
\psfrag{k21}{\hspace*{-2mm} \small $8_{21}$}
\begin{center}
\includegraphics[height=3.1cm]{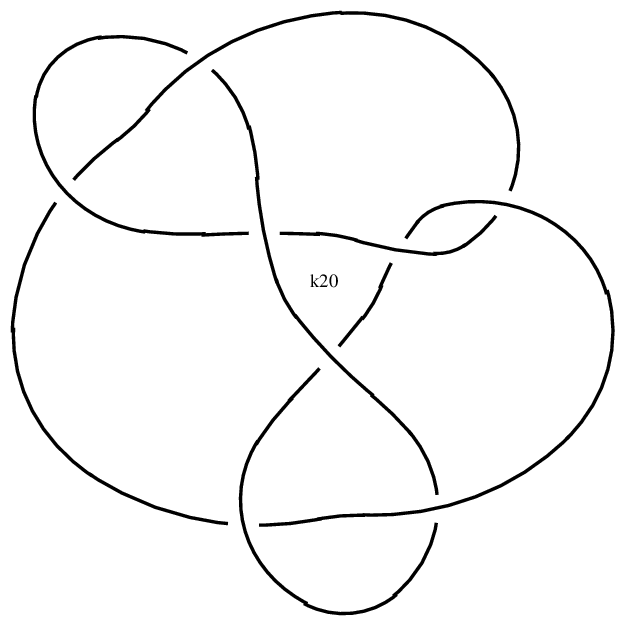}
\hspace*{2.3cm}
\includegraphics[height=3.1cm]{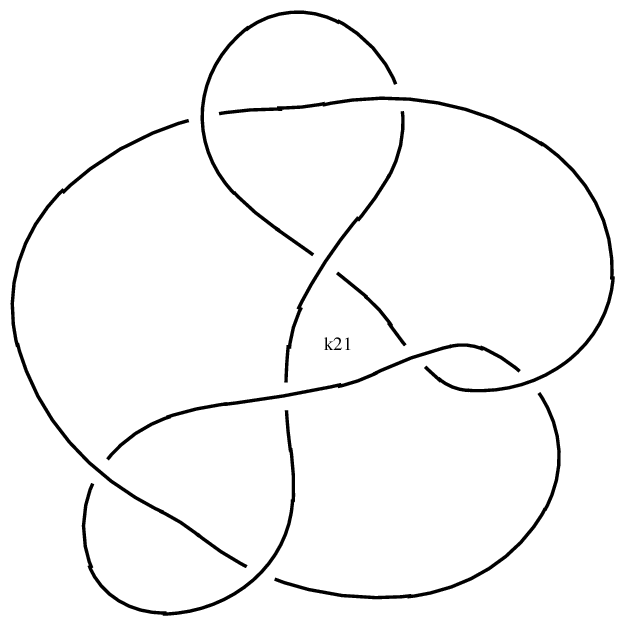}
\caption{Diagrams for knots of types $8_{20}$ and $8_{21}$}
\label{fig:Diagrams-K8n20-K8n21}
\end{center}
\end{figure}

Now to the group $G_{21}$.
\index{Computation of Sigma1@Computation of $\Sigma^1$ for!group of knot K8n21@group of knot \texttt{K8n21}}%
A run of \texttt{SnapPea} has produced a presentation with generators $a$, $b$ and defining relator
\begin{equation}
\label{eq:Relator-K8n21}
r_{21}  = 
a^2 b\cdot a^{-1} b^{-2}\cdot a^{-1} b\cdot a^{-1} b^{-1}\cdot a b^2\cdot a b^{-1}\cdot a b^2
\cdot 
a b^{-1}\cdot a^{-1} b\cdot a^{-1} b^{-2}\cdot a^{-1} b.
\end{equation}
There is no problem in finding a rank 1 character $\chi \colon G_{21} \to \R$
and in computing the sequence  $f_{r_{21}} (\chi)$;
there is, however, an alternative approach which is often more informative 
and which I want to use in the sequel.
\begin{figure}[htb]
\psfrag{1}{\hspace*{-1.8mm} \scriptsize $1$}
\psfrag{2}{\hspace*{-1.5mm}  \scriptsize $2$}
\psfrag{3}{\hspace*{-1.3mm} \scriptsize $3$}
\psfrag{4}{\hspace*{-0.1mm}\scriptsize $4$}
\psfrag{5}{\hspace*{-1.4mm} \scriptsize $5$}
\psfrag{6}{  \hspace*{-3mm}   \scriptsize $6$}
\psfrag{7}{\hspace*{-1.2mm} \scriptsize $7$}
\psfrag{8}{\hspace*{-1.2mm} \scriptsize $8$}
\psfrag{9}{\hspace*{-1.3mm} \scriptsize $9$}
\psfrag{10}{\hspace*{-2mm} \scriptsize $10$}
\psfrag{11}{\hspace*{-2.5mm} \scriptsize $11$}
\psfrag{12}{  \hspace*{-2.3mm}\scriptsize $12$}
\psfrag{13}{  \hspace*{-2.5mm}\scriptsize $13$}
\psfrag{14}{\hspace*{-1.3mm}\scriptsize $14$}
\psfrag{15}{\hspace*{-1.2mm}\scriptsize $15$}
\psfrag{16}{  \hspace*{-3mm}   \scriptsize $16$}
\psfrag{17}{\hspace*{-2.3mm} \scriptsize $17$}
\psfrag{18}{\hspace*{-2.5mm} \scriptsize $18$}
\psfrag{19}{\hspace*{-2.5mm} \scriptsize $19$}
\psfrag{20}{\hspace*{-2mm} \scriptsize $20$}
\psfrag{21}{\hspace*{-1.5mm} \scriptsize $21$}
\psfrag{22}{  \hspace*{-2.3mm}\scriptsize $22$}
\psfrag{23}{  \hspace*{-1.6mm}\scriptsize $23$}
\psfrag{24}{\hspace*{-0.8mm}\scriptsize $24$}
\psfrag{25}{\hspace*{-1.2mm}\scriptsize $25$}
\psfrag{26}{  \hspace*{-3mm}   \scriptsize $26$}
\psfrag{27}{\hspace*{-2.2mm} \scriptsize $27$}
\psfrag{aa}{\hspace*{-1mm} \small $a$}
\psfrag{bb}{\hspace*{-1.5mm} \small $b$}
\begin{center}
\includegraphics[height=4.6cm]{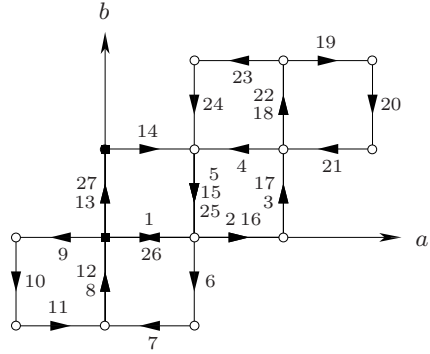}
\caption{Path of the relator $r_{21}$ defining the knot group $8_{21}$}
\label{fig:Path-relator-r21}
\end{center}
\end{figure}

In this approach, 
one determines the path $\bar{p}$ in $\Gamma(\Z^2, \{a,b\})$ 
which corresponds to $r_{21}$ 
and starts at the origin;
this path is displayed in Figure \ref{fig:Path-relator-r21}.
It has length 27 and ends in $(0,1)$;
so the exponent sum of $a$ in $r_{21}$ is 0 and that of $b$ is 1.
The unit vectors giving rise to characters of $G_{21}$ are therefore 
$u =(1,0)$ and its opposite.
The first vector leads to the linear function $x \mapsto x_{1}$;
along the path,  this function is minimal on the 10th edge and maximal on the 20th edge.
It follows that $\Sigma^1(G_{21})$ equals $S(G_{21})$.
\end{example}

%
\subsubsection{One-relator groups defined by positive words}
\label{sssec:Baumslag-positive-one-relator}
\index{Computation of Sigma1@Computation of $\Sigma^1$ for!one relator groups}
%
In \cite{Bau83} G. Baumslag considers the class of one-relator groups 
whose defining relator $r$ has the form $u \cdot v^{-1}$; 
here $u$ and $v$ are positive words in the generators 
and the exponent sums of $u$ and $v$ agree for each generator.
He proves 
that  such a group  is an extension of a free group $N$ by an infinite cyclic group
(\cite[Thm.\;1]{Bau83}). The normal subgroup $N$ need not be finitely generated;
indeed it can only be so if $G$  is generated by two elements
(see, e.\;g., Proposition \ref{prp:Invariant-group-large-deficiency}).
\index{Baumslag, G.}

We consider the two-generator case in more detail
\footnote{The other cases can reduced to this special case by embedding the group into  a two generator group
in such a way that the defining relator has still the special form $u' \cdot (v')^{-1}$ 
where $u'$ and $v'$ are positive words having the same exponents sums for both generators
(\cite[Lemma 2]{Bau83}).}
but content ourselves with an example,
say the group $G = \langle a,b \mid u \cdot v^{-1}\rangle$ with 
\[
u = a^2 \cdot b \cdot a \cdot b^3 \cdot a^3  
\quad \text{and} \quad
v = b \cdot a \cdot b \cdot a^4 \cdot b \cdot a \cdot b.
\] 
The assumption on the exponent sums of $u$ and $v$ imply 
that the relator $r = u \cdot v^{-1} $ lies in the commutator subgroup of the free group on $\{a, b \}$;
so there is an epimorphism $\vartheta \colon G \epi \Z^2$
which maps $(a, b)$ to the standard basis of $\R^2$.
The closed path $\bar{p}$ associated to $r$ in $\Gamma(\Z^2, \{a,b\})$ is depicted 
on the left of Figure \ref{fig:Path-invariant-group-positive-relation}. 
\begin{figure}[htb]
\psfrag{1}{\hspace*{-2.0mm} \footnotesize $1$}
\psfrag{2}{  \hspace*{-2.0mm}\footnotesize $2$}
\psfrag{3}{\hspace*{-1.5mm} \footnotesize $3$}
\psfrag{4}{\hspace*{-1.5mm}\footnotesize $4$}

\psfrag{5}{\hspace*{-1.5mm} \footnotesize $5$}
\psfrag{6}{  \hspace*{-2.5mm} \footnotesize $6$}
\psfrag{7}{\hspace*{-1.5mm} \footnotesize $7$}
\psfrag{8}{\hspace*{-1.7mm} \footnotesize $8$}
\psfrag{9}{\hspace*{-2.0mm} \footnotesize $9$}
\psfrag{10}{\hspace*{-2.5mm} \footnotesize $10$}
\psfrag{11}{\hspace*{-2mm} \footnotesize $11$}
\psfrag{12}{  \hspace*{-2.2mm}\footnotesize $12$}
\psfrag{13}{\hspace*{-2.0mm} \footnotesize $13$}
\psfrag{14}{\hspace*{-0.5mm}\footnotesize $14$}
\psfrag{15}{\hspace*{-2mm} \footnotesize $15$}
\psfrag{16}{  \hspace*{-3.2mm} \footnotesize $16$}
\psfrag{17}{\hspace*{-2.1mm} \footnotesize $17$}
\psfrag{18}{\hspace*{-2mm} \footnotesize $18$}
\psfrag{19}{\hspace*{-2.2mm} \footnotesize $19$}
\psfrag{20}{\hspace*{-1.5mm} \footnotesize $20$}

\psfrag{xx1}{\hspace*{-0.0mm}  \footnotesize $(1,0)$}
\psfrag{yy1}{\hspace*{-3.5mm}  \footnotesize $(0,1)$}
\psfrag{a}{\hspace*{-2.0mm} \small $a$}
\psfrag{b}{\hspace*{-1.5mm} \small $b$}

\begin{center}
\includegraphics[height=4.8cm]{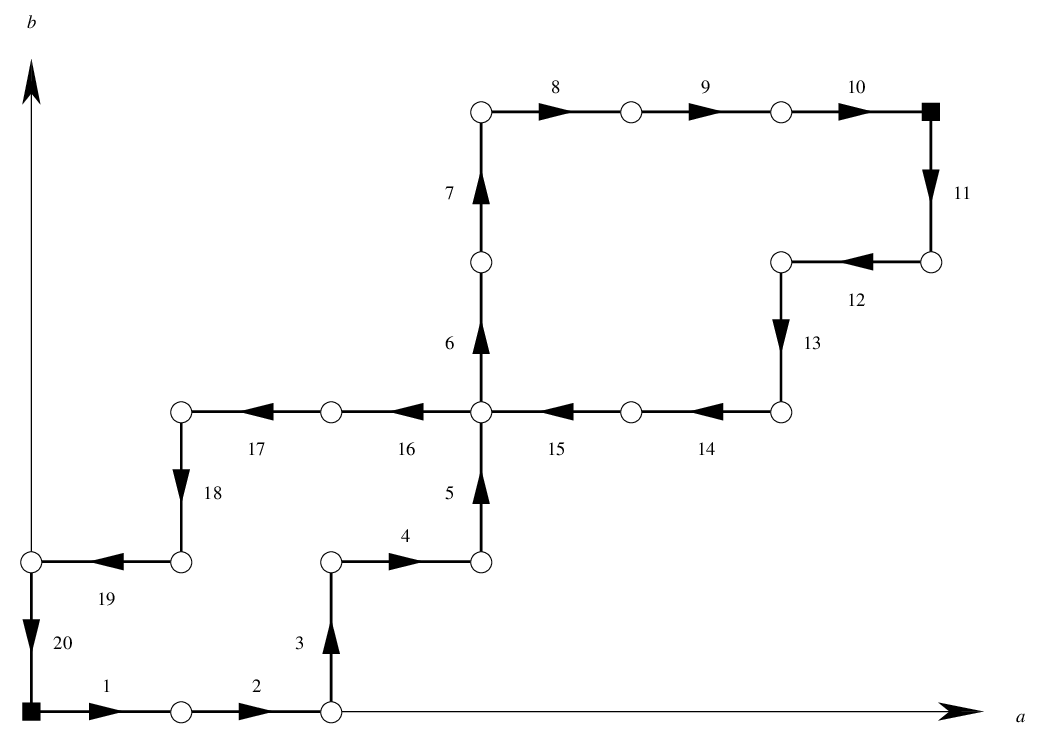}
\hspace*{0.5cm}
\includegraphics[height=3.6cm]{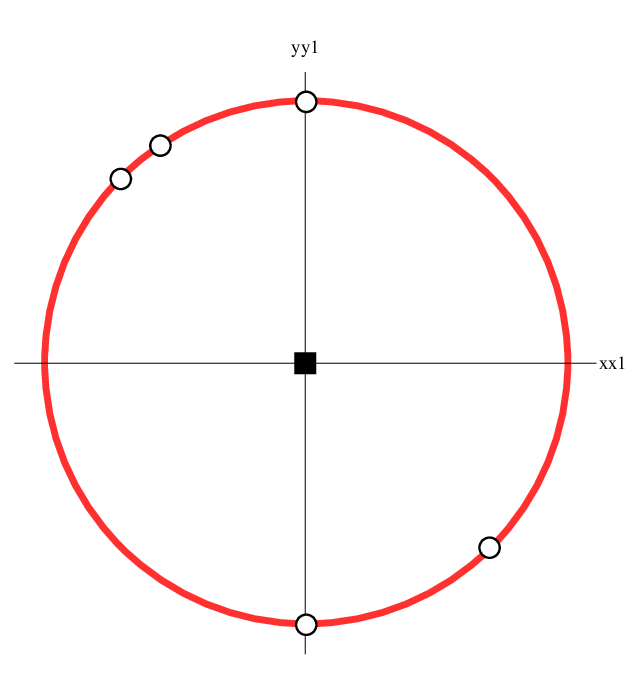}
\caption{Path and invariant of a group with positive relation}
\label{fig:Path-invariant-group-positive-relation}
\end{center}
\end{figure}

The words $u$ and $v$ do neither start with nor do they end in the same generator;
the relator $r = u \cdot v^{-1}$ is therefore cyclically reduced and 
the vertices $(0,0)$ and $\vartheta(u) = (6,4)$ are simple.
Moreover, 
the fact that $u$ and $v$ are positive words with the same exponent sums implies 
that the first and third open quadrants lie in $\Sigma^1(G)$,
and so the kernel of every rank 1 character representing a point of the open first quadrant 
is finitely generated by Corollary \ref{crl:Characterizing-fg-kernel-rank-1-character}
and free by Remark \ref{remarks:Browns-theorem}d.
There is no problem in determining $\Sigma^1$ completely; 
the outcome is depicted on the right of Figure \ref{fig:Path-invariant-group-positive-relation}.
\index{Computation of Sigma1@Computation of $\Sigma^1$ for!one relator groups}
%
 
%
 
%
\chapter[Alternate definitions of $\Sigma^1$: Overview]%
{Alternate definitions of Sigma~1: Overview}
\label{ch:Alternative-definitions-overview}
%
%
%
In section \ref{ssec:Definition-Sigma1}, 
the invariant $\Sigma^1$ has been introduced thus:
given a finitely generated group $G$,
choose a finite  generating system $\eta \colon \XX \to G$
and let $\Gamma(G, \XX)$ denote the associated Cayley graph of $G$. 
Consider then a non-zero character $\chi \colon G \to \R$ and 
define $G_{\chi}$ to the submonoid $\{ g \in G \mid \chi(g) \geq 0 \}$ of $G$.
The Cayley graph $\Gamma(G, \XX)$ is connected, 
but its subgraph $\Gamma_{\chi}$ induced by the monoid $G_{\chi}$ need not be so.
Whether or not $\Gamma_{\chi}$ is connected does not depend on the generating system,
it depends only on $\chi$, more precisely on the ray $[\chi]$, and on $G$.
One is thus led to consider the subset $\Sigma^1(G)$ made up of the rays $[\chi]$ 
with a connected subgraph.

The definition of $\Sigma^1$ in terms of Cayley graphs 
is not the original definition of the invariant $\Sigma_{G'}(G)$ proposed in 1987 by Bieri, Neumann and Strebel.
Robert Bieri and Burkhardt Renz,
Gaël Meigniez, Ken S.~Brown, Gilbert Levitt and Jean-Claude Sikorav found,
at about the same time, alternate definitions.
Some of these definitions encompass larger, others more restricted, classes of groups.
These various definitions show 
that the concept of Sigma invariant is not only of interest in the area of finitely presented soluble groups,
the area which led in 1980 to the introduction of the first such invariant.
\index{Bieri, R.}
\index{Brown, K. S.}
\index{Levitt, G.}
\index{Meigniez, G.}
\index{Neumann, W. D.}
\index{Renz, B.}
\index{Strebel, R.}

In this chapter, and in Chapter D,
some of the previously mentioned variations on the definition of $\Sigma^1$ will be discussed.
In each variation, the Cayley graph is replaced by a structure $X$
on which $G$ acts and satisfies a certain property $P$. 
Then the action is restricted to that of the various submonoids $G_{\chi}$ 
and one collects the points $[\chi] \in S(G)$ 
for which the action of $G_{\chi}$ continues to enjoy property $P$.

The variations of this chapter will be grouped into three classes.
Those of the first class rephrase in algebraic terms the idea of using $G_\chi$ 
to define a subgraph $\Gamma_\chi$ of a Cayley graph $\Gamma(G, \XX)$
and asking whether or not $\Gamma_\chi$ is connected.
The definitions of the second class take as their starting point 
the set of path components of the subgraph $\Gamma_{\chi}$; 
they admit straightforward generalizations to infinitely generated groups.
In the variations of the third class 
the groups under study are fundamental groups $G$ of connected, compact manifolds $M$
and the Cayley graph of $G$ gets replaced by the universal abelian cover $\hat{M}$ of $M$.

The overview afforded by this chapter is centered on the Cayley graph definition of the invariant:
starting from this definition, one moves on to variations of it.
This organization must not be construed as implying 
that the variations arose out of the desire to widen the applicability of the invariant $\Sigma^1$:
many of them were invented  and investigated independently of the invariant $\Sigma^1$.
Only later, did the creators became aware of the relation of their invariants with 
the invariant $\Sigma_{G'}(G)$ introduced in \cite{BNS}.
%
%
%
%
\section{Variations for finitely generated groups}
\label{ssec:Variations-fg-groups}
%
%
A distinctive feature of the variations surveyed in this section is the fact 
that their natural domain of definition is the class of all finitely generated groups.
The first of the variations is, at first sight, 
hardly more than a simple restatement of the Cayley graph definition;
it turns out, however, to be a key opening up new horizons.
The second variation is a sufficient condition for a point to lie in $\Sigma^1$;
it allows one, in certain examples,  to detect large subsets of $\Sigma^1$ with ease.
The surprise is that this condition is also necessary.
\smallskip

\emph{Notation.
Throughout this section,
$G$ will denote a finitely generated group, $\XX \subset G$ a finite \emph{set}  generating $G$
and $\chi \colon G \to \R$ a non-zero character.}
%
%
\subsection{The invariant $\Sigma^1(G; \Z)$} 
\label{ssec:Invariant-Sigma1(G,Z)}
%
The reformulation of the defining property of $\Sigma^1$,
given in this section,
will reveal 
that the invariant $\Sigma^1(G)$ is the first member of an infinite sequence of invariants 
\[
\Sigma^1(G;\Z) \supseteq \Sigma^2(G;\Z) \supseteq \Sigma^3(G;\Z) \supseteq  \cdots .
\]
%
\subsubsection{The route from $\Sigma^1(G)$ to $\Sigma^1(G; \Z)$} 
\label{sssec:Path-to-Sigma1(G,Z)}
%
We begin with an easy consequence of the defining property of $\Sigma^1(G)$.
Suppose the subgraph $\Gamma_{\chi} = \Gamma(G, \XX)_{\chi}$ is \emph{connected}.
Replacing, if need be, some generators $x \in \XX$ by their inverses, 
we may assume that $\chi(\XX) \geq 0$.
For each element $g \in G_{\chi}$ there exists then a word $w = y_{1} y_{2}\cdots y_{k}$ in 
$\YY = \XX \cup \XX^{-1}$ which represents $g$ and has the property that the sequence
\begin{equation}  
\label{eq:Non-negative-chi-track}
\left( \chi(y_{1}), \chi(y_{1}y_{2}), \ldots, \chi(y_{1}y_{2} \cdots y_{k})\right)
\end{equation}
consists of non-negative real numbers.
\footnote{This sequence will be referred to as the \emph{$\chi$-track}  of the word $y_{1}\cdots y_{k}$.}
\index{Track of a word}
It follows that the element $g-1$ of the group ring $\Z{G}$ admits the expansion
\begin{align}
g - 1 
&= 
y_{1}\cdot y_{2}\cdots y_{k} -1 = (y_{1}- 1) + y_{1} \cdot (y_{2} \cdots y_{k} -1) 
\notag\\
&=
(y_{1}- 1) + y_{1}(y_{2} - 1) + y_{1}y_{2} \cdot (y_{3}\cdots y_{k} - 1) = \cdots
\notag\\
&= (y_{1} - 1) + \sum \nolimits_{1 \leq j \leq k} y_{1}\cdots y_{j-1}(y_{j}- 1).
\label{eq:Expressing-difference}
\end{align}
In the above,
each factor $y_{1} \cdots y_{j}$ lies in the monoid $G_{\chi}$ and hence in the monoid ring $\Z{G_\chi}$.
\footnote{The multiplication of  $\Z{G_{\chi}}$ is inherited from the group ring $\Z{G}$.}

Let $\varepsilon_{\chi} \colon \Z{G_{\chi}} \to \Z$ denote the augmentation map; 
it sends each group element $g \in G_\chi$ to the number $1 \in \Z$.
The kernel $I_{\chi} = IG_{\chi}$ of $\varepsilon_{\chi}$ 
consists of the linear combinations of differences $g - 1$ with $g \in G_\chi$
and it is a two-sided ideal of $\Z{G_{\chi}}$; 
in the sequel,  we shall view it as a \emph{left} ideal.

The previous calculation shows that the left ideal $I_{\chi}$ 
is generated by the differences $y-1$ with $y \in \XX$ and establishes implication (i) $\Rightarrow$ (ii) in 
\begin{prp}
\label{prp:Equality-Sigma1-Sigma1-G-Z}
\index{Invariant Sigma1Z@Invariant $\Sigma^1(-;\Z)$!coincidence with Sigma1@coincidence with $\Sigma^1$}
Let $G$ be a finitely generated group and $\XX \subset G$ a finite generating \emph{set}. 
For each non-zero character $\chi \colon G \to \R$, set 
$\YY_{+} = \{y \in \XX \cup \XX^{-1} \mid \chi(y) \geq 0 \}$.
Then the following statements imply each other:
\begin{enumerate}[(i)]
\item $[\chi] \in \Sigma^1(G)$;
\item the left ideal $IG_{\chi}$ is generated by the differences $y - 1$ 
with $y \in \YY_{+}$;
\item $ IG_{\chi}$ is a finitely generated \emph{left} ideal.
\end{enumerate}
\end{prp}

\begin{proof}
Implication (i) $\Rightarrow$ (ii) has been justified in the above 
and implication (ii) $\Rightarrow$ (iii) clearly holds. 
So all is well
if we can establish implication (iii) $\Rightarrow$ (i).

Suppose $I_{\chi} = IG_{\chi}$ is generated by a finite set of elements $\lambda = \sum m_{j} g_{j}$.
Then $ \varepsilon_\chi (\lambda) = \sum m_{j}$  is 0
and so $\lambda$ can be rewritten in the  $\sum m_{j}(g_{j}- 1)$.
It follows that there exists a finite set $\ZZ \subset G_{\chi}$ such that,
for each $g \in G_{\chi}$, 
the difference $g-1$ can be written in the form
\begin{equation}
\label{eq:Represention-g-minus-1}
g-1 = \sigma_{1}g_{1}(z_{1}-1) + \sigma_{2}g_{2}(z_{2}-1) + \cdots  + \sigma_{m}g_{m}(z_{m}-1) 
\end{equation}
where each $\sigma_{j}$ is 1 or $-1$, each $g_{j}$ lies in $G_{\chi}$ and each $z_{j}$ is an element of $\ZZ$.

We shall now prove by induction on the number $m$ of summands 
that every $g \in G_\chi$ can be represented by a $\ZZ^\pm$-word $w = s_{1}\cdots s_{k}$
whose $\chi$-track
\begin{equation}
\label{eq:Definition-chi-track}
\left( \chi(s_{1}), \chi(s_{1}s_{2}), \ldots, \chi(s_{1}s_{2} \cdots s_{k})\right).
\end{equation}
is non-negative.
Assume first $m=1$. 
Then $g-1 = \sigma_1 g_{1}(z_{1}-1)$ and two cases arise:
if $\sigma_{1} = 1$, one has $g_{1} = 1$ and  $g = g_{1} z_{1}= z_1$,
whence the letter $z_{1}$ can be taken as word $w$;
if $\sigma_{1} = -1$
then $g = g_1$ and $g_{1}z_{1} = 1$,  and so $ g =z_{1}^{-1}$.
Since $g$ and $z_{1}$ both belong to $G_{\chi}$, 
the equality  $\chi(z_{1}^{-1}) = 0$ must hold. 
The word $w$ can therefore be chosen to be $z_{1}^{-1}$.
Suppose now that $m>1$.
Then $g$ is either a product of the form $g_{j} z_{j}$ with $j \in \{1,2, \ldots, m\}$ or one of the $g_{j}$.
In both cases, there is no harm in assuming that $j = m$.
In the first case we obtain the equation
\[
g-1 =  \left(\sum\nolimits _{1 \leq i < m}\sigma_{i}g_{i}(z_{i}-1)\right) + g_{m}(z_{m}- 1) \text { with } g =  g_{m}z_{m},
\]
whence $g_{m} - 1 = \sum\nolimits _{1 \leq i < m}\sigma_{i}g_{i}(z_{i}-1)$.
Since $g_{m} \in G_{\chi}$, the inductive hypothesis applies 
and thus $g_{m} $ can be represented by a $\ZZ^\pm$-word $w'$ with non-negative $\chi$-track.
But if so,
$g = g_{m}z_{m}$ can be represented by the word $w = w' z_{m}$ which has non-negative $\chi$-track.
Consider, finally, the case where $g = g_{m}$ and $\varepsilon_m = -1$.
Then
\[
g -1 = -g_{m}(z_{m}- 1) + \sum\nolimits _{1 \leq i < m}\sigma_{i}g_{i}(z_{i}-1)
\]
and so $g_{m} z_{m}-1 = \sum\nolimits _{1 \leq i < m}\sigma_{i}g_{i}(z_{i}-1)$.
It follows, as in the previous case, that there exists a $\ZZ^\pm$-word $w'$ with non-negative $\chi$-track
that represents $g_{m} z_{m}$. 
The word $w = w'z_{m}^{-1}$ then represents $g = g_{m}$ and its $\chi$-track is non-negative,
for $\chi(w) = \chi(g) \geq 0$.

The above reasoning implies that $\ZZ$ generates the group $G$ 
and shows that the subgraph $\Gamma(G; \ZZ)_\chi$ is connected; 
so statement (i) is true.
\end{proof}

\begin{remarks}
\label{remarks:Left-versus-right}
a) The $\chi$-track of the word $w = s_1 \cdots s_k$  (see formula \eqref{eq:Non-negative-chi-track}) 
allows one to describe algebraically whether the associated path $(1,w)$ runs inside the subgraph 
$\Gamma_\chi$ of the Cayley graph $\Gamma(G; \XX)$.
The definition of the $\chi$-track is identical with that of the sequence $f_r(\chi)$ used in Section 
\ref{sec:Sigma1-criterion-revisited}, 
but the purpose of that sequence is different: 
there the word $r$ is a relator of $G$
and one wants to know whether the minimum of the sequence is achieved at most twice,
and if assumed twice, whether it is taken at points that are neighbours.
In this context, first and last points have to be considered as neighbours.

b) The ideal $IG_{\chi}$ is a two-sided ideal of the monoid ring $\Z{G_{\chi}}$.
In the above,
it is considered it as a left module;
if one views $IG_{\chi}$  as a right module over  $\Z{G_{\chi}}$
its properties as a right module do not correspond to those of the left module $\Z{G_{\chi}}$,
but to those of the left module $\Z{G_{-\chi}}$.
Indeed, the bijection $ g \mapsto g^{-1}$ 
extends to a linear map $\iota \colon \Z{G} \iso \Z{G}$ which is  a ring anti-automorphism.
It sends $(g-1) \cdot IG_{\chi}$ 
to $\iota(IG_{\chi}) \cdot (g^{-1} - 1) = IG_{-\chi} \cdot (g^{-1} -1)$.
A consequence of the stated fact is
\begin{lem}
\label{lem:Changing-sides}
$IG_{\chi}$ is finitely generated \emph{qua right} $\Z{G_{\chi}}$-module if, and only if,
$IG_{-\chi}$ is finitely generated \emph{qua left} $\Z{G_{\chi}}$-module.
\end{lem}

c) Proposition \ref{prp:Equality-Sigma1-Sigma1-G-Z} (more precisely, a variant of it)
is due to R. Bieri and B. Renz (see \cite[Cor.\,6.3]{BiRe88}). 
The proof given in the above seems to be new.
\end{remarks}
\index{Bieri, R.}
\index{Renz, B.}

\subsubsection{The homological invariants $\Sigma^k(G; A)$} 
\label{sssec:Introducing-Sigmak(G;A)}
%
Proposition \ref{prp:Equality-Sigma1-Sigma1-G-Z} shows
that the defining property of $\Sigma^1(G)$ can be stated  in terms of module theory.
This fact becomes more significant if one goes one step further 
and rephrases the proposition in terms of homological algebra.

The ideal $IG_\chi$ is, by definition,  the kernel of a ring epimorphism 
$\varepsilon_{\chi} \colon \Z{G_{\chi}} \epi \Z$.
Now  $IG_\chi$ can also be viewed as the kernel of the epimorphism of the free cyclic left module 
$\Z{G_{\chi}}$  onto the left module $\Z$,
each element $g \in G_{\chi}$ acting on $\Z$ by the identity.
Schanuel's Lemma
\footnote{see, e.\,g., \cite[p. 431, Lemma 3.1]{Pas77}}
then shows 
that the augmentation ideal $IG_{\chi}$ is a finitely generated left ideal if, and only if,
the left $\Z{G_{\chi}}$-module $\Z$ admits a finite presentation.

The condition we arrived at is a special case of a well-known homological finiteness condition:
\begin{definition}
\label{definition:Type-FPk}
Let $R$ be an associative ring with $1 \neq 0$ and $A$ a left  $R$-module.
Given $k \geq 0$,
one says that $A$ is \emph{of type} $\FP_{k}$ if there exists an exact sequence
\begin{equation}
\label{eq:Resolution}
\cdots \to P_{k+1} \to P_{k} \to F_{k-1} \to \cdots \to F_{1} \to F_{0} \to A \to 0
\end{equation}
in which each of the modules $P_{k}$, $P_{k-1}$, \ldots, $P_{1}$ and $P_{0}$ 
is a finitely generated projective  $R$-module.
\end{definition}
\index{Type FPK@Type $\FP_k$}%
\index{Definition of!type FPk@Type $\FP_k$}%

A module of type $\FP_{0}$ is nothing but a finitely generated module; 
a module is of type $\FP_{1}$ if, and only if,  it admit a finite presentation.

In the terminology just introduced, 
Proposition \ref{prp:Equality-Sigma1-Sigma1-G-Z} can be restated as
\begin{crl}
\label{crl:Equality-Sigma1-Sigma1-G-Z}
For every finitely generated group $G$ the following equation is valid:
\begin{equation}
\label{eq:Sigma1(G)-equals-Sigma1(G,Z)}
\Sigma^1(G) = \{ [\chi] \in S(G) \mid \Z \text{ is of type $\FP_{1}$ over the ring } \Z{G_{\chi}}  \}.
\end{equation}
\end{crl}

The condition appearing on the right of  equation \eqref{eq:Sigma1(G)-equals-Sigma1(G,Z)}
admits of an obvious generalization to  (finitely generated) $\Z{G}$-modules $A$ and arbitrary dimensions $k$.
This generalization leads to the invariants $\Sigma^k(G;A)$  introduced by R. Bieri and B. Renz in \cite{BiRe88}.
Here we content ourselves with stating their definition and mentioning one obvious property:

\begin{definition}
\label{definition:Invariants-Sigmak(G;A)}
Given a finitely generated group $G$, a finitely generated left  $\Z{G}$-module and an integer $k\geq 0$, set
\begin{equation}
\label{eq:Definition-Sigmak(G,A)}
\Sigma^m(G; A) = \{ [\chi] \in S(G) \mid  A \text{   is of type $\FP_{m}$ over the ring } \Z{G_{\chi}}   \}.
\end{equation}
Then $\Sigma^m(G;A)$ is a subset of the sphere $S(G)$; 
it is called \emph{homological geometric invariant of $G$ and $A$ in dimension $m$}.
\end{definition}
\index{Definition of!invariant Sigmam(-;A)@invariant $\Sigma^m(-;A)$}
\index{Invariant Sigmam(-;A)@Invariant $\Sigma^m(-;A)$!definition}
\index{Bieri, R.}
\index{Renz, B.}

The new invariants are open subsets of the sphere $S(G)$, as is $\Sigma^1(G)$.
This property, however, is far from evident.
The only obvious fact about the new invariants is the descending chain of inclusions
\[
S(G) \supseteq \Sigma^0(G;A) \supseteq \Sigma^1(G;A) \supseteq \Sigma^2(G;A) \supseteq \cdots \supseteq \Sigma^m(G;A)\supseteq \cdots .
\]
\begin{remarks}
\label{remarks:Equality-Sigma1(G)-Sigma1(G,Z)}
a)
Corollary \ref{crl:Equality-Sigma1-Sigma1-G-Z}
asserts that the invariant $\Sigma^1(G)$ coincides with the homological invariant $\Sigma^1(G; \Z)$.
This fact is already pointed out in \cite[Proposition 6.4]{BiRe88}). 
The proof given in the above seems to be new.

b) In Section D1,
the invariant $\Sigma^0$ will be investigated in greater detail.
\end{remarks}
\index{Invariant Sigma1Z@Invariant $\Sigma^1(-;\Z)$!definition}%
\index{Definition of!invariant Sigma1Z@invariant $\Sigma^1(-;\Z)$}%
%

%
\subsection{The invariant $\Sigma_{G'}(G)$} 
\label{ssec:Invariant-Sigma-subGprime}
%
In this section,
we show that $\Sigma^1(G)$ coincides with the invariant $\Sigma_{G'}(G)$ introduced and studied in \cite{BNS}.
We begin with a restatement of the defining property of $\Sigma^1$ that is long overdue. 
%
\subsubsection{Reexpressing the connectedness of  $\Gamma_{\chi}$} 
\label{sssec:Reformulation-connectedness-subgraph-G-sub-chi}
%
Suppose the subgraph $\Gamma(G, \XX)_{\chi}$ is connected.
Then every element of $G_{\chi}$ can be written as an $\XX^\pm$-word with non-negative $\chi$-track.
This holds, in particular,  
for every element $g$ in the commutator subgroup $G'$.
Conversely, 
assume every element in $G'$ admits a representation as an $\XX^\pm$-word with non-negative $\chi$-track
and consider an element $g \in G_{\chi}$.
Since $\XX$ generates $G$, 
there exists a word $w = y_{1} y_{2} \cdots y_{k}$ in $\XX^\pm$  that represents it. 
The $\chi$ track of this word may be negative,
but by rearranging the order of the letters in $w$ 
we can construct a word $w' = y_{\sigma(1)} \cdots y_{\sigma(k)}$ with non-negative $\chi$-track;
it suffices to bring all letters with positive $\chi$-value to the front.
The words $w$ and $w'$ represent the same element of $G/G'$.
So there exists a word $w''$ representing a word of $G'$ with $w \equiv w'w''$.
Since $w''$ can be chosen to have non-negative $\chi$-track
this shows that every element $g \in G_{\chi}$ 
can be represented by an $\XX$-word with $v_{\chi}(g) \geq 0$,
and so $\Gamma(G, \XX)_{\chi}$ is connected.

The preceding argument holds for arbitrary groups,
not only for finitely generated ones,
a fact that will turn out to be useful in section \ref{sssec:Generalizing-Sigma1}.
Since the kernel of a character contains the commutator subgroup of $G$,
we have established
\begin{lem}
\label{lem:Restatement-connectivity-property}
Let $G$ be an arbitrary group,
$\eta \colon \XX \to G$ a generating system
and $\chi \colon G \to \R$ a non-zero character. 
Then the following statements imply each other:
\begin{enumerate}[(i)]
\item the subgraph $\Gamma(G, \XX)_{\chi}$ is connected;
\item for every $g \in \ker \chi$ there exists an $\XX^\pm$-word with non-negative $\chi$-track;  
\item for every $g \in G'$ there exists an $\XX^\pm$-word with non-negative $\chi$-track.
\end{enumerate}
\end{lem} 
\subsubsection{Relation to $\Sigma^1(G)$} 
\label{sssec:Inclusion-Sigma-sub-Gprime-Sigma1(G)}
%
We begin by stating a condition, 
depending on a finitely generated group $G$ and a non-zero character $\chi$, 
and then proceed to show that it implies that $[\chi] \in \Sigma^1(G)$.

Given $G$ and $\chi$, consider the condition
\begin{equation}
\label{eq:Defining-property-Sigma-sub-Gprime}
G' \text{  is generated by a finite subset } \AA \text{ over a fg submonoid } M \subset G_{\chi}. 
\end{equation}
In this statement  $G'$ is viewed as a \emph{left} $G$-operator group,
the operation being given by conjugation, thus $\act{g}{-0.5}{a} = gag^{-1}$.

Assume $G$ and $\chi$ satisfy condition \eqref{eq:Defining-property-Sigma-sub-Gprime}.
Let $\XX$ be a finite set generating the finitely generated monoid $M$
and consider an element $g \in G'$.
There exist then a sequence of elements $g_{1}$, \ldots, $g_{\ell}$ in $M$ 
and a sequence of letters $a_{1}$, \ldots, $a_{\ell}$ in $\AA \cup \AA^{-1}$ so that
\[
g = \act{g_{1}}{0.5}{a_{1}}\cdot \act{g_{2}}{0.5}{a_{2}} \cdots \act{g_{\ell}}{0.5}{a_{\ell}}
\]
By assumption, each element $g_{j}$ can be expressed by a \emph{positive} word $w_{j}$ in the alphabet $\XX$.
The word
\[
w = w_{1}a_{1}w_{1}^{-1} \cdot w_{2}a_{2}w_{2}^{-1} \cdots w_{\ell}a_{\ell}w_{\ell}^{-1}
\]
is then a word in the alphabet $\XX \cup \XX^{-1} \cup \AA \cup \AA^{-1}$;
it represents $g$  and has non-negative $\chi$-track.
In view of  Lemma  \ref{lem:Restatement-connectivity-property} 
the subgraph  $\Gamma(G, \XX \cup \AA)_{\chi}$ is therefore connected
and so $[\chi] \in \Sigma^1(G)$.
\begin{remark}
\label{remark_Definition-SigmaBNS}
Condition \eqref{eq:Defining-property-Sigma-sub-Gprime} is the defining property of the invariant 
$\Sigma_{G'}(G)$, introduced and investigated by R.~Bieri, W.~D.~Neumann and R.~Strebel in \cite{BNS},
except for the fact 
that in \cite{BNS} the group $G'$ is considered as a \emph{right} $G$-operator group.
This implies that the invariant studied there is the image under the antipodal map of the set 
\index{Invariant SigmaBNS@Invariant $\Sigma_{G'}(G)$!definition}%
\index{Definition of!invariant SigmaBNS@invariant $\Sigma_{G'}(G)$}%
\begin{equation}
\label{eq:Definition-Sigma-sub-Gprime}
\Sigma_{G'}(G) = \{ [\chi] \in S(G) \mid G' \text{ is fg over a fg submonoid of  } G_{\chi} \}.
\end{equation}
There is thus a clash of notation between the invariant of \cite{BNS} 
and the invariant that will be investigated in this monograph.
I prefer to live with this conflict rather than to introduce new notation
and to make sure 
that the reader in never in doubt as to whether left or right action is used in the context at hand.
\end{remark}

The insight arrived at so far is summarized by the following
\begin{prp}
\label{prp:Inclusion-Sigma-sub-Gprime-Sigma1(G)}
Given a finitely generated group $G$,
view its commutator subgroup $G'$ as a \emph{left} $G$-operator group.
Then  $\Sigma_{G'}(G) \subseteq \Sigma^1(G)$.
\end{prp}
%
\subsubsection[$\Sigma^1$ of Houghton's groups]%
{Application: the invariant of Houghton's groups} 
\label{sssec:Application-invariant-Houghtons-group}
%
Proposition \ref{prp:Inclusion-Sigma-sub-Gprime-Sigma1(G)} can be helpful in situations
where the action of  $G$ on its commutator subgroup $G'$ is sufficiently well known.
Instances where this condition is met with are provided by a sequence of permutation groups 
studied by C. Houghton in \cite[pp.\;257--258]{Hou78}.
These groups generalize example \ref{examples:Torsion-by-infinite-cyclic} b) 
and are defined like this.
\index{Computation of Sigma1@Computation of $\Sigma^1$ for!Houghton's groups|(}

Given a natural number $m \geq 2$, 
set $\SS =\SS_m = \N \times \{1,2,\ldots, m\}$;
one can think of $\SS_m$ as being the disjoint union of $m$ rays that emanate from a point in the plane.  
Define $G_{m}$ to be the group of all permutations of $S_m$ 
which are eventually a translation.
More precisely, 
a permutation $g \colon \SS_m \iso \SS_m$ belongs to $G_{m}$
if, and only if,  
there is a vector $x_{g} = (x_{1}, x_{2}, \ldots, x_{m}) \in \Z^m$ 
such that the equation $g((n,j)) = (n+x_{j},j)$ holds for each $j \in \{1,\ldots, m\}$ 
and all sufficiently large $n \in \N$.

The vector $x_{g}$ is uniquely determined by $g$ 
and the assignment $g \mapsto x_{g}$ defines a homomorphism $\widetilde{\vartheta} \colon G_m \to \Z^m$
whose image coincides with the the subgroup 
\[
Q_{m} = \{ x \in \Z^m \mid x_{1}+ x_{2} + \cdots + x_{m} = 0 \} 
\]
of $\Z^m$.
Indeed, it is clear that $\im \widetilde{\vartheta} \subseteq Q_m$; 
to prove the opposite inclusion, 
we construct a collection of ``translations''  in $G_{m}$.
For each $j \in \{1,2, \ldots, m\}$,
let $t_{j}$  denote the permutation 
that fixes the rays $\N \times \{i\}$ with $i \notin\{ j,j+1\}$ pointwise
and acts on the line formed by the remaining two rays by the rule
\footnote{Here and in the sequel, the index $j$ is taken modulo $m$.}
\begin{equation}
\label{eq:Translations-generating-Houghtons-group}
t_{j} (n,j) = 
\begin{cases} (n-1,j) & \text{ if } n \geq 1\\ (0,j+1) & \text{ if } n = 0, \end{cases} 
\quad \text{ and } \quad 
t_{j}(n,j+1) = (n+1, j+1).
\end{equation}
The claim now follows from the facts 
that $\widetilde{\vartheta}(t_{j})$ is  the vector $e_{j+1}- e_{j}$
and that these differences generate the subgroup $Q_m$.

Let $\vartheta \colon G_{m} \epi Q_{m}$ denote the homomorphism 
obtained from $ \tilde{\vartheta}$ by restricting the domain of values to $Q_{m} = \im \tilde{\vartheta}$.
The kernel of $\vartheta$ is the group of all finitary permutation of the set $\SS_m = \N \times \{1,2,\ldots, m\}$.
\begin{prp}
\label{prp:Invariant-Houghton-groups}
For every $m > 2$, the group $G_{m}$ is finitely generated. 
The complement  of its invariant $\Sigma^1(G_{m})$ consists of $m$ rank 1 points;
more precisely 
\begin{equation}
\label{eq:Invariant-Houghton-groups}
\Sigma^1(G_{m})^c  = \{ [-\chi_{1}], [-\chi_{2}], \ldots , [-\chi_{m}] \}
\end{equation}
where each $\chi_{j}$ is the character sending $g$  to the $j$-th component $(x_g)_{j}\in \Z$ 
of $\vartheta(g)$.
\end{prp}

\begin{proof} 
We first show that $ G =G_{m}$ is generated by the translations $t_{j}$ defined by equation
\eqref{eq:Translations-generating-Houghtons-group}.
Let $T$ denote the subgroup generated by them. 
Then $\vartheta(T) = Q = Q_m$ 
and so it suffices to verify that $T$ contains the kernel of $\vartheta$.
This kernel is the group of all finitary permutations of $S_m$,
and so generated by transpositions;
we claim each transposition is a product of commutators.
Indeed, 
the commutator $[t_{j-1},t_j]= t_{j-1} \circ t_{j} \circ t_{j-1}^{-1} \circ t_{j}^{-1}$
fixes each set  $(\N \smallsetminus \{0\}) \times \{i\}$ with $i  \in \{1,2, \ldots, m \}$ pointwise,
it fixes the the endpoints $(0,i)$ with $i \notin \{j,j+1\}$ 
and, as $m > 2$, it acts on the two remaining endpoints like this
\begin{align*}
&(0,j) \mapsto (1,j) \mapsto (0,j) \mapsto  (0,j+1)\mapsto (0,j+1),\\
&(0,j+1) \mapsto (0,j) \mapsto (0,j-1) \mapsto (0,j-1) \mapsto (0,j).
\end{align*}
So $[t_{j-1},t_j]$ is the transposition exchanging the endpoints $(0,j)$ and $(0,j+1)$.
In view of the action of $t_{j}$ on the lines $\N \times \{j\} \cup \N \times \{j+1\}$
it follows, firstly,  
that every transposition of adjacent points $(j,n)$ and $(j,n+1)$ on the $j$-th ray $\N \times \{j\}$ 
belongs to $T$ and is a commutator in $T$,
and then that every transposition is a product of commutators in $T$.
We conclude that $\ker \vartheta$ is contained in $T'$ and then that $T = G$.

We move on to the proof of formula \eqref{eq:Invariant-Houghton-groups}.
The derived group of $G_m$ is generated 
by the transpositions of adjacent elements of the rays $\N \times \{j\}$
and by the transpositions that exchange the end points of two adjacent rays.

Consider now a non-zero character $\chi \colon G \to \R$.
Suppose there exists, for each ray $\N \times \{j\}$, an element $g_j \in G_\chi$ such that $\chi(g_j) > 0$;
let $p_j$ be a positive integer such that $g_j$ acts on $\{p_j, p_j+1, \ldots \} \times \{j\}$ 
by a translation with amplitude $a_j = \chi(g_j)$.
Then $G'_m$ is generated by the transpositions of the finite set 
\[ 
\FF = \bigcup\nolimits_{1 \leq j \leq m} \{0, 1, \ldots,  p_j + a_j \}\times \{j\}
\]
and their conjugates under the positive powers of the elements $g_1$, \ldots, $g_m$.
The derived group $G'_m$ is thus finitely generated over the monoid $M$ 
generated by the elements $g_1$, \ldots, $g_m$,
whence $[\chi] \in \Sigma_{G'_m}(G_m)$ by definition \eqref{eq:Definition-Sigma-sub-Gprime}
and thus $[\chi] \in \Sigma^1(G_m)$ by Proposition \ref{prp:Inclusion-Sigma-sub-Gprime-Sigma1(G)}.

There remains the problem of finding out 
when there exists a sequence of elements $g_1$, \ldots, $g_m$ with the stated properties.
This is a problem in euclidean geometry.
Indeed, by introducing coordinates, as detailed in section \ref{sssec:Coordinates-sphere},
the previous problem becomes the following one: 
let $u_1$, \ldots,  $u_m$ be rationally defined, pairwise distinct unit vectors 
in the standard euclidean space $\R^{m-1}$.
Each $u_j$ gives rise to an open half lattice 
$\HH_j = \{ x \in \Z^{m-1} \mid \langle u, x\rangle > 0 \}$ of $\Z^{m-1}$.
Then $\HH_j$  intersects a closed half space $\HH_u = \{y \in \R^{m-1} \mid \langle u, y \rangle \geq 0 \}$
non-trivially if, and only if, $u \neq -u_j$.
The previous argument therefore implies 
that 
\[
\Sigma^1(G_m)^c \subseteq  \{ [-\chi_{1}], [-\chi_{2}], \ldots , [-\chi_{m}] \}.
\]

We are left with proving that the points $-[\chi_{1}]$, \ldots, $-[\chi_{m}]$ lie in $\Sigma^1(G_m)^c$.
Given an index $j \in \{1,2, \ldots, m\}$, 
consider the subgroup $H_{j}$ of $G_{m}$ made up of all permutations 
that fix the ray $\N \times \{j + 1\}$ pointwise.
The translation  $t_{j}$ maps the set 
$S_j = \SS_m \smallsetminus \N \times \{j+1\}$ onto $S_j \cup \{(0,j+1\}$
and the images of $S_j$ under the positive powers of $t_{j}$ sweep out all of $\SS_m$.
It follows that $G_m$ is a strictly ascending HNN-extension with base group $H_{j}$ and stable letter $t_{j}$,
whence $G_m$ is a strictly \emph{descending} HNN-extension 
with base group $H_{j}$  and stable letter $t_{j}^{-1}$.
Since $(-\chi_j)(t_j^{-1}) = 1$,
Proposition \ref{prp:Strictly-descending-HNN-extension}
allows us to see that $-[\chi] \notin \Sigma^1(G_m)$.
\end{proof}
\index{Computation of Sigma1@Computation of $\Sigma^1$ for!Houghton's groups|)}

\begin{note}
\label{note:Computation-invariant-Houghtons-groups}
The invariant of Houghton group $G_{m}$ has been worked out by Ken Brown around 1985 
by means of his characterization of $\Sigma^1$ 
in terms of actions on $\R$-trees (see \cite[Section 5]{Bro87b}).
The proof given here was found by R. Bieri and R. Strebel at about the same time.
\end{note}
\index{Houghton, C. H.}
\index{Brown, K. S.}
\index{Bieri, R.}
\index{Strebel, R.}
%

\subsubsection{Equality of $\Sigma^1(G)$ and  $\Sigma_{G'}(G)$} 
\label{sssec:Equality-Sigma1(G)-Sigma-sub-Gprime}
%
Proposition \ref{prp:Inclusion-Sigma-sub-Gprime-Sigma1(G)}
states that the invariant $\Sigma_{G'}(G)$, propounded in \cite{BNS}, is a subset of $\Sigma^1(G)$.
Actually the reverse inclusion also holds:
\begin{thm}
\label{thm:Equality-Sigma1(G)-Sigma-sub-Gprime}
For every finitely generated group $G$ the invariant $\Sigma^1(G)$ coincides with 
\begin{equation}
\label{eq:Definition-Sigma-sub-Gprime-bis}
\Sigma_{G'}(G) = \{ [\chi] \in S(G) \mid G' \text{ is fg over a fg submonoid of  } G_{\chi} \}.
\end{equation}
\end{thm}
\index{Invariant SigmaBNS@Invariant $\Sigma_{G'}(G)$!equality with Sigma1@equality with $\Sigma^1$}

\begin{proof}
The inclusion $\Sigma_{G'}(G) \subseteq \Sigma^1(G)$ is covered by Proposition 
\ref{prp:Inclusion-Sigma-sub-Gprime-Sigma1(G)};
the reverse inclusion will be established in two steps, 
summarized by the formula
\begin{equation}
\label{eq:Description-two-steps}
[\chi] \in \Sigma^1(G) 
\Longrightarrow 
\gp(\act{\WW(\XX,\chi)}{1}{\AA}) = G'
\Longrightarrow 
[\chi] \in \Sigma_{G'}(G).
\end{equation}
In the above,
the following notation is used:
$\XX \subset G $ is a finite set of generators of $G$,
next $\YY$ is the alphabet $\XX \cup \XX^{-1}$
and $\WW(\XX, \chi)$ denotes the set of all $\YY$-words with non-negative $\chi$-track.
Finally, 
$\AA$ denotes the set of all commutators 
\[
[y_{1}, y_{2}] = y_{1} y_{2} y_{1}^{-1}y_{2}^{-1}  \text{ with } y_1 \text{ and } y_2 \text{ in } \YY.
\]

Let $\XX$, $\YY$ and $\AA$ are as explained before
and let $\chi \colon G \to \R$ a non-zero character 
that represents a point of $\Sigma^1(G)$.
Every element $g \in G'$ can then be represented by a $\YY$-word $w'$ with non-negative $\chi$-track
(cf. implication (i) $\Rightarrow$ (iii) of Lemma \ref{lem:Restatement-connectivity-property}).
We intend to rewrite $w'$ as a product  $w$ of conjugates of the commutators in $\AA$.
If such a product exists, 
it will have zero exponent sum with respect to every generator $x \in \XX$;
the word $w'$, however, may not have this property.
One can remedy such a default as follows.

Let $\eta$ denote the obvious projection of $F$ onto $G$.
The preimage $U = \eta^{-1}(G')$ contains the derived group $F'$ of $F$;
let $\UU'$ be a finite of words in $\YY$ that generates $U$ \emph{modulo} $F'$.
The image $\eta(u')$ of each  $u'\in \UU'$ lies in $G'$;
since $\eta$ maps $F$ onto all of $G$, 
there exists a word $w_{u'} \in F'$ with $\eta(u') = \eta(w_{u'})$.
By construction, 
each word in the finite set $\UU' \cup \{w_{u'} \mid u' \in \UU' \}$ represents an element of $G'$,
but its $\chi$-track may be negative.
There exists, however, 
a letter $y \in \YY$ with $\chi(y) > 0$ and an exponent $k > 0$ 
so that the $\chi$-tracks of all the elements in
\[
\UU = \{u = y^ku'y^{-k} \mid u' \in \UU'\} \quad \text{ and } \quad 
 \{w_u = y^k w_{u'}y^{-k} \mid u' \in \UU'\}
\]
are 0.
Let now $g$ be an element of $G'$.
As $[\chi] \in \Sigma^1(G)$ 
there exists a word $w' \in F$ with $\eta (w') = g$ and  $\chi$-track equal to 0.
Next, 
there exist a product $u_1 \cdots u_k$ of words in $\UU$ 
so that $w'' = w' \cdot u_1 \cdots u_k$ has exponent sum 0 for every $x \in \XX$.
The $\chi$-track of the word $w''$ is then 0,
as is the $\chi$-track of the word $w = w'' \cdot (w_{u_1} \cdots w_{u_k})^{-1}$.
Notice that the word $w$ represents the given element $g \in G'$ and that it lies in $F'$,
as we set out to show. 

We are now ready to justify the first implication in \eqref{eq:Description-two-steps}.
For every $g \in G'$ the previous argument provides one with a $\YY$-word $w \in F'$ 
and $\chi$-track equal to 0.
We intend to show by induction on the number $m$ of letters in $w$ 
that $w$ is freely equivalent to a word of the form
\[
\act{v_{1}}{0}{a_{1}} \cdot \act{v_{2}}{0}{a_{2}} \cdots \act{v_{f}}{0}{a_{f}}
\]
where each $v_{j}$ is a $\YY$-word with non-negative $\chi$-track and each $a_{j}$ is in $\AA$.
The claim is obvious if $m=0$.
If $m>0$, 
consider the leftmost letter $z$ in $w$ with $\chi(z) \leq 0$.
This letter is either preceded or succeeded by an occurrence of $z^{-1}$, 
for the exponent sums of $w$ are equal to 0.
In the first case, set $y = z^{-1}$; in the second case, set $y = z$.
Then $w$ has the form $w_{1}yw_{2}y^{-1}w_{3}$ and so it is freely equivalent to
\[
w_{1} \cdot y  w_{2}y^{-1} w_{2}^{-1} \cdot w_{2}w_{3} \equiv 
\act{w_{1}}{0}{[y,w_{2}]} \cdot w_{1}w_{2}w_{3}. 
\]
The word $w_{1}w_{2}w_{3}$ is shorter than $w$ and its exponent sums are 0.
Its $\chi$-track is likewise 0: 
if $\chi(y) \leq 0$, this assertion is obvious;
if not, each letter in $w_{1} w_{2}$ has a positive $\chi$-value by the choice of $z$,
and so the assertion holds likewise.
By the inductive assumption, the word $w_{1}w_{2}w_{3}$ represents therefore 
an element in the subgroup $\gp(\act{\WW(\XX,\chi)}{1}{\AA})$. 
Let $y_{1} \cdots y_{k}$ be the spelling of $w_{2}$.
The commutator identity $[a, bc] = [a,b] \cdot \act{b}{0}{[a,c]}$  implies then
that $\act{w_{1}}{0}{[y,w_{2}]}$ is freely equivalent to 
\[
\act{w_{1}}{0}{\left([y,y_{1}]\cdot \act{y_{1}}{0}{[y,y_{2}] }\cdots \act{y_{1}\cdots y_{k-1}}{0}{[y,y_{k}]}\right)} 
\equiv 
\act{w_{1}}{0}{[y,y_{1}]}  \cdot \act{w_{1}y_{1}}{0} {[y,y_{2}] } \cdots \act{w_{1}y_{1}\cdots y_{k-1}}{0} {[y,y_{k}]}.
\]
Let $w_0$ denote the word 
that stands on the right hand side of this equivalence.
As each of the conjugating words $wy_{1}\cdots y_{j}$ occurring $w_{0}$ 
is an initial word of $w_{1}w_{2}$  it has non-negative $\chi$-track.
Hence $w_{0}$ represents a word in the group 
$\gp(\act{\WW(\XX,\chi)}{1}{\AA})$.
The inductive step of the first implication in \eqref{eq:Description-two-steps} is now complete.
\smallskip

In the second part,
the representation obtained in the first step,
is applied to the commutators 
$\act{y}{0}{a}$ with  $y \in \YY = \XX^\pm$ and $a \in \AA = \{[y_1,y_2] \mid (y_1,y_2) \in \YY^2 \}$.
Let
\begin{equation}
\label{eq:Expression-for-second-step}
\act{y}{0}{a} = \act{w_{1}}{0}{a_{1}} \cdot \act{w_{2}}{0}{a_{2}} \cdots \act{w_{f}}{0}{a_{f}}
\end{equation}
be the expression resulting for $\act{y}{0}{a}$  
and let $\II(y,a)$ be the set of all initial words of the conjugating words $w_{j}$  occurring in this expression.
Define $\II$ to be the union of these finite sets $\II(y,a)$ 
and let $M$ be the monoid generated by the image of $\II$ in the group $G$.
We claim that $\AA$  generates $G'$ over $M$.

Let $H \subseteq G'$ denote the subgroup generated by the set $\act{M}{0}{\AA}$.
Since $G'$ is generated by $\AA$ over $G$,
the equality of $G'$ and $H$ will follow
if $H$ is stable under conjugation by the elements of $\YY$.
Let $w = y_{1} \cdots y_{k}$ be a $\YY$-word that is a product of words in $\II$.
By the construction of $\II$,
each initial segment $w'$ of $w$ is a product of words in $\II$.
 Given $a \in \AA$ and $y \in \YY$,
 we now rewrite $\act{yw}{0}{a}$ as follows:
 \begin{align*}
 \act{yy_{1}\cdots y_{k}}{0}{a}
 &=
 \act{yy_{1}y^{-1}y_{1}^{-1}\cdot y_{1} y y_{2}y_{k}}{0}{a}\\
 &= 
 [y,y_{1}] \cdot \act{y_{1} y y_{2} \cdots y_{k}}{0}{a} \cdot  [y,y_{1}]^{-1}
=
 a_{1} \cdot \act{y_{1} y y_{2} \cdots y_{k}}{0}{a} \cdot  a_{1}^{-1}.
  \end{align*}
Continuing in this manner we end up with the equation
\[
\act{yy_{1}\cdots y_{k}}{0}{a} 
= 
(a_{1} \cdot \act{y_{1}}{0}{a_{2}}   \cdots \act{y_{1} \cdots y_{k-1}}{0}{a_{k}})  
\cdot (\act{y_{1}y_{2}  \cdots y_{k} y}{1}{a}) \cdot 
(a_{1} \cdot \act{y_{1}}{0}{a_{2}}   \cdots \act{y_{1} \cdots y_{k-1}}{0}{a_{k}})^{-1}.
\]
 The first and the third factor of the right hand side are contained in $H$
 since the occurring conjugating words are initial segments of $w$;
 the second factor lies in $H$ by relation \eqref{eq:Expression-for-second-step} 
 and the definition of $M$.
 So  $\act{y}{0}{H} \subseteq H$, as desired. 
\end{proof}

\begin{note}
\label{note:Proof-equality-Sigma1(G)-Sigma-sub-Gprime}
The justification for the second implication in formula
\eqref{eq:Description-two-steps} 
goes back to the proof of implication (iv) $\Rightarrow$ (i) in \cite[Proposition 2.1]{BNS}.
\end{note}

We conclude section \ref{ssec:Invariant-Sigma-subGprime} 
with two applications of the theorem just proved.

\subsubsection{A characterization of the complement of $\Sigma^1$} 
\label{sssec:Characterization-complement-Sigma1}
%
So far, 
various  results have been established
that allow one show to that a given character $\chi \colon G \to \R$ represents a point of the invariant $\Sigma^1$.
By contrast, 
only three general results are at one's disposal
if one sets out to prove that a character $\chi \colon G \to \R$ does \emph{not} represent a point of $\Sigma^1$.
The first of them deals with  the invariant of a non-abelian free group 
or, more generally, of  a non-trivial free product $ G = G_1 \star G_2$
and states that the invariant is empty (see example 3) in section \ref{sssec:Sigma1-first-examples}).
This result will be generalized in section \ref{sssec:Illustration-HNN-extensions};
there it will be shown that if $G$ is a free product with amalgamation $ G_1 \star_A G_2$ 
with $A$ distinct from both $G_1$ and $G_2$ then $S(G, A) \subseteq \Sigma^1(G)^c$
(see part (i) of Proposition \ref{prp:Point-outside-Sigma1-HNN-extensions}).

Secondly, one knows that in case $G$ maps onto a quotient group $Q$
every point $[\psi] \in \Sigma^1(Q)^c$ pulls back to a point in $\Sigma^1(G)^c$
(see Corollary \ref{crl:Sigma1-epimorphism}).
This observation is used, for instance, in the proofs of Theorem \ref{thm:Sigma1-right-angled-Artin-group}
or of Proposition \ref{prp:Invariant-group-large-deficiency}. 

Thirdly, 
Proposition \ref{prp:Strictly-descending-HNN-extension} states 
that the canonical character of a \emph{strictly descending HNN-extension} 
represents a point in the complement of $\Sigma^1$;
this result deals only with rank 1 characters.
Theorem \ref{thm:Equality-Sigma1(G)-Sigma-sub-Gprime} 
allows one to generalize it as follows:
\begin{prp}
\label{prp:Characterization-complement-Sigma1}
\index{Characterization of!complement of Sigma1@complement of $\Sigma^1$}
Let $G$ be a finitely generated group and let $\chi$ be a non-zero character of $G$.
Consider a subgroup  $N$ of $G$ that contains $G'$ and is contained in $\ker \chi$.
Then $[\chi] \in \Sigma^1(G)^c$ if, and only if, $N$ contains an increasing chain
$B_0 \subseteq B_1 \subseteq \cdots$ of \emph{proper} subgroups with the following two properties:
\begin{enumerate}[(i)]
\item $\bigcup\nolimits_{j \in \N} B_j = N$;
\item for every $g \in G_\chi$ 
there exists an index $ j(g)$ such 
that  the inclusion $g \cdot B_j \cdot g^{-1} \subseteq B_j$ holds for each $j \geq j(g)$.
\end{enumerate}
\end{prp}

\begin{proof}
Assume first a chain of subgroups $B_j$ satisfying requirements (i) and (ii) exists.
Then every subgroup $B$ of $G'$ 
that is finitely generated over a finitely generated submonoid $M \subseteq G_\chi$ is contained in $B_j$ 
for all sufficiently large indices $j$.
Since no member  $B_j$  of the chain $B_0 \subseteq B_1 \subseteq \cdots$  can contain $G'$ ---
for each member is a proper subgroup and $N/G'$ is finitely generated ---
Theorem \ref{thm:Equality-Sigma1(G)-Sigma-sub-Gprime} allows one to conclude
that $[\chi] \notin \Sigma^1(G)$.

Conversely, 
assume that $[\chi] \notin \Sigma^1(G)$.
Let $\eta \colon \XX \to G$ be a finite generating system of $G$ 
and $\Gamma = \Gamma(G, \XX)$ the associated Cayley graph of $G$.
Define $B_j$ to be the set of elements $g \in N$ 
that can be represented by a word with $\chi$-track bounded from below by $-j$.
Lemmata \ref{lem:Connected-rays} and \ref{lem:Restatement-connectivity-property},
imply that each of these subsets is distinct from $N$.
The sets $B_j$ are subgroups of $N$ (cf. Remark \ref{remark:Valuation-for-words})
and they form an ascending chain whose union is $N$.
Consider now an element $g \in G_\chi$.
Represent it by some word $w(g)$ with $\chi$-track bounded from below by some integer, say $-j(g)$.
Then $g B_j g^{-1} \subseteq B_j$ for every index $j \geq j(g)$.
The subgroups $B_j$ are therefore proper subgroups of $N$ and they form an increasing chain 
that satisfies requirements  (i) and (ii)
stated in Proposition \ref{prp:Characterization-complement-Sigma1}.
\end{proof}
\begin{note}
\label{note:Characterization-complement-Sigma1}
Proposition \ref{prp:Characterization-complement-Sigma1}  is a variant of Proposition 9.1 of  \cite{BNS}.
\end{note}
%
\subsubsection[$\Sigma^1$ of a wreath product]{Application: $\Sigma^1$ of a wreath product} 
\label{sssec:Application-invariant-wreath-product}
%
In section \ref{sssec:Invariant-direct-product},
the invariant of a direct product $H \times Q$ of two finitely generated groups has been determined.
The outcome is summarized in Proposition \ref{prp:Sigma1-direct-product};
it shows 
that $\Sigma^1(H \times Q)^c$ is,  in essence, the union of $\Sigma^1(H)^c$ and $\Sigma^1(Q)^c$.
Below,
we shall establish a formula for the invariant of the wreath product $G = H \wr Q$ 
of two finitely generated groups $H$ and $Q$.
This formula differs from that for a direct product in an important aspect:
$\Sigma^1(G)^c$ \emph{depends only on the abelianisations of $H$ and of $Q$}.
\begin{definition}
\label{definition:wreath-product}
The (restricted standard) wreath product of the group $H$ by the group $Q$ 
is the quotient of the free product $H \star Q$ 
\emph{modulo} the normal subgroup generated by the commutators
\begin{equation}
\label{eq:Definition-wreath-product}
[h_1 , qh_2q^{-1}] = h_1\cdot qh_2q^{-1} \cdot h_1^{-1}\cdot qh_2^{-1} q^{-1} ;
\end{equation}
here  $(h_1, h_2)$ ranges over $H^2$ and  $q $ over $Q \smallsetminus \{1\}$.
The wreath product of $H$ by $Q$ will be denoted by $H \wr Q$.
\end{definition}
\index{Definition of!wreath product}
\index{Wreath products!definition}

Definition \ref{definition:wreath-product}  describes $H \wr Q$ by a presentation.
Alternatively, one can define the wreath product by an explicit construction:
it is the semi-direct product $L \rtimes Q$ of the so called base group $L = \gp_G(H)$ by $Q$
where $L$  denotes the restricted direct product $\Dr \{qHq^{-1} \mid q \in Q \}$ of $\card(Q)$ copies of $H$.
\smallskip

In the sequel, $H$ and $Q$ will be assumed to be finitely generated.
If one of them is reduced to the unit element,
the wreath product is nothing but the other group;
this case presents no interest and will be excluded in the sequel.

The definition of $H \wr Q$ implies 
that there is an epimorphism $\pi \colon H \wr Q \epi Q$
which sends $h \in H$ to $1 \in Q$ and $Q$ onto itself by the identity.
This epimorphism gives rise to the embedding $\pi^* \colon S(Q) \incl S(H \wr Q)$.
In addition, there is an epimorphism $\bar{\rho}$ of $H \wr Q$ onto $H_{\ab}$;
it maps $h \in H$ to $hH'$ and $q \in Q$ to $1H'$.
The epimorphisms $\bar{\rho}$ and  $\pi$ give rise to an isomorphism groups $(H \wr Q)_{\ab} \iso H_{\ab} \times Q_{\ab}$
and to an isomorphism of spheres
\[
S(H_{\ab} \times Q) \iso S(H \wr Q).
\]
It follows, in particular, that $S(H \wr Q)$ is empty if the abelianisations of $H$ and $Q$ are both finite;
so there exist wreath products that are of  little interest in the context of $\Sigma^1$,
but which are admitted in Proposition \ref{prp:Sigma1-wreath-product} below.

This proposition describes the \emph{complement} of the invariant 
of  the wreath product of two arbitrary finitely generated  groups,
the sole restriction being that neither of the factors $H$ and $Q$ be the trivial group.
%
\begin{prp}
\label{prp:Sigma1-wreath-product}
Let $G$ be the wreath product $H \wr Q$ of two finitely generated, non-trivial groups $H$ and $Q$.
Then 
$\Sigma^1(G)^c   = S (G,H)$.
\end{prp}
\index{Computation of Sigma1@Computation of $\Sigma^1$ for!wreath products}

 \begin{proof}
Let $\chi \colon G \to \R$ be a non-zero character.
The proof splits into three cases, 
depending on as to whether $\chi$ vanishes on $H$, on $Q$, or on neither of them.
 
 Assume first that $\chi$ \emph{vanishes on} $H$.
 The aim is to construct a sequence of subgroups $B_0 \subset B_1 \subset \cdots$ in  the kernel $N$ of  $ \chi$ 
 and to deduce,  with the help of Proposition \ref{prp:Characterization-complement-Sigma1}, 
 that $[\chi] \notin \Sigma^1(G)$. 
Put $Q_0 = Q \cap  \ker \chi $ and 
 \[
 B_j = \Dr \{ qHq^{-1} \mid \chi(q) \geq -j \} \rtimes Q_0.
 \]
Then each of these sets $B_j$ is a proper subgroup of $N = \ker \chi$ and their union is $N$.
Moreover, each $B_j$ is invariant under conjugation by the submonoid $Q _{\chi|_Q}$.
As each element $g \in G_\chi$ is a product of  the form $n \cdot q$ with $n \in N$ and $q \in  Q _{\chi|_Q}$,
Proposition  \ref{prp:Characterization-complement-Sigma1} therefore applies and shows
that $[\chi] \notin \Sigma^1(G)$.
\smallskip

Suppose next that $\chi$ \emph{vanishes on} $Q$.
The plan is to show that the normal subgroup $N = G' \cdot Q$ is finitely generated.
The Addendum to Corollary \ref{prp:Comparison-invariants-G-and-Q}
will then imply that $S(G,Q) \subseteq \Sigma^1(G)$.
Let $B$ denote the base group of $G$, \ie, the direct sum $\Dr \{q H q^{-1} \mid q \in Q\}$.
Choose a finite set of generators $\HH$ of $H$, 
and a finite set of generators $\QQ$ of $Q$;
we can and shall assume that $\QQ \subset Q \smallsetminus \{1 \}$.
Set
\begin{align}
\CC &= \{ [h_1, h_2] \mid h_i \in \HH \},
\label{eq:Definition-C}\\
\KK &=  \{ [h,q] \mid (h, q) \in \HH^\pm \times \QQ \},
\label{eq:Definition-K}
\text{ and } \\
L &= \gp(\CC \cup\KK \cup \QQ).
\label{eq:Definition-L}
\end{align}
Then $L $ is a finitely generated subgroup of $N$;
we are going to prove that $L$ equals $N$ whence $N = G' \cdot Q$ is finitely generated.

We first show 
that $L$ contains the derived group $B' = \Dr \{q H' q^{-1} \mid q \in Q\}$ of the base group $B$.
The group $H'$ is the normal closure in $H$ of the subset $\CC \subset L$.
Since conjugation by $[h,q] = h \cdot qh^{-1}q^{-1}$ has the same effect on $H$ 
as conjugation by $h$ whenever $q \in Q \smallsetminus \{1\}$ and as $\KK \cup \QQ \subset L$,
it follows, first, that $H' \subset  L$ and then that $B' \subset L$.

Now $B'$ is a normal subgroup of $G$ and $G/B'$ is isomorphic to $B_{\ab} \rtimes Q$;
it suffices therefore to prove  that $\bar{L} = L/B'$ coincides with  $\bar{N} = N/B'$.
The derived group of $\bar{G} = G/B'$ has the form $M \rtimes Q'$ 
where $M \subset B_{\ab}$ is the submodule 
generated the elements $(1-q) \cdot \bar{b}$ with $q \in Q$ and $\bar{b} \in B_{\ab}$.
Since $N = G' \cdot Q$, this shows that $\bar{N} = M \rtimes Q$.
But $B_{\ab} = \Z{Q} \otimes H_{ab}$ 
and the augmentation ideal $IQ = \ker (\Z{Q} \epi \Z)$ is generated by the finite set $\{1- q \mid q \in \QQ \}$.
The claim thus follows from the fact that $((1-q )\otimes h) \cdot B'$ is the additive notation of the element 
$[h,q] \cdot B' \in \KK \cdot B'$ 
and so it lies in $ L/B'$ by the definition of $L$.
\smallskip

Consider, finally, the case where \emph{$\chi|_Q$ and $\chi|_H$ are both non-zero}.
Choose a finite set of generators  $\HH $ of $H$ and a finite set of generators $\QQ$ of $Q$
so that $\chi$ is positive on each of the elements of  $\XX = \HH \cup \QQ$.
The set $\XX$ generates $G$;
we use it and Proposition \ref{prp:Psi(RR)-subset-Sigma1} to show that $[\chi] \in \Sigma^1(G)$.

 Pick elements $h \in \HH$ and  $q \in \QQ$. 
Then $q \neq 1$ and so the words 
  \[
  r_{h,q} = [h, qhq^{-1}]= h \cdot qhq^{-1} \cdot  h^{-1} \cdot qh^{-1}q^{-1}
  \]
are relators of $G$.
The minimum of the $\chi$-track of each of these words is 0 and it occurs once, at the end of the word.
In view of definition \ref{definition:Psi(R)}  and Proposition \ref{prp:Psi(RR)-subset-Sigma1}, 
these facts imply that $[\chi] \in \Sigma^1(G)$. 
In more detail:
   
Put  $\RR = \{ [h,qhq^{-1}] \mid h \in \HH, q \in \QQ \}$.
Then the set $ \RR_{\chi, + }$, occurring in definition \ref{definition:Psi(R)}, coincides with $\RR$ 
and the graph $\GG_{\RR,\chi}$ is the bipartite graph with $\HH$ and $\QQ$ as sets of vertices.  
This graph  is connected, 
for $\chi|_H \neq 0$ and $\chi|_Q \neq 0$,
and so the sets $\HH$ and  $\QQ$ are non-empty.
By definition \ref{definition:Psi(R)},
the point $[\chi]$ thus belongs to $\Psi(\RR)$ 
and hence to $\Sigma^1(G)$ by Proposition \ref{prp:Psi(RR)-subset-Sigma1}.
\end{proof}

\begin{remarks}
\label{remarks:Invariant-wreath-products}
a) The wreath product allows one to construct groups 
that are fairly easy to  analyze, but have interesting properties.
One such construction goes back to  Philip Hall's paper \cite{Hal54}.
In section 2.4,
Hall proves that the wreath product  $G = H \wr Q$ of two finitely generated groups 
satisfies the \emph{maximal condition max-n on normal subgroups}
if the first factor $H$ has this property and if the group ring $\Z{Q}$ of the second factor $Q$  is right (or left) noetherian.
As an infinite cyclic group $C $ is noetherian and its group ring is likewise noetherian,
this result permits one to construct a chain of finitely generated subgroups
\[
G_1 = C, \quad G_2 = C \wr C, \quad G_3 = G_2 \wr C = (C \wr C) \wr C, \ldots, \quad G_{m} = G_{m-1} \wr C, \ldots
\]
which all satisfy max-n.
Let $x_1$, $x_2$, \ldots, $x_m$ be elements in $G_m$ 
that generate the infinite cyclic groups involved in the successive wreath products making up $G_m$.
The abelianisation of $G_m$ is free abelian of rank $m$, generated by the canonical images of the $x_j$,
and, for $m > 1$, the complement of $\Sigma^1(G_m)$ is the 0-dimensional sphere $S(G_m, \gp(x_1, \ldots, x_{m-1})) $.
The derived length of $G_m$ is $m$,
as one can see by the following inductive argument. 
The group  $G_1$ is a non-trivial abelian group and so of derived length 1.
Assume, inductively,  that $G_{m-1}$ has derived length $m-1$. 
Since $G_{m}$ is an extension of the direct product  $B = \Dr_j \{x_m^{j}G_{m-1} x_m^{-j}\}$ 
by the infinite cyclic group $\gp(x_m)$, 
the derived length of $G_m$ is at most $m$.
On the other hand, the derived group of $G$ contains the commutators $[x_1,x_m]$, \ldots, $[x_{m-1}, x_m]$.
They generate a subgroup $H$ of $G_{m-1} \oplus x_m G_{m-1} x_m^{-1}$ 
which projects onto the first factor  $G_{m-1} $.
Hence the derived length of $G_m$ is at least $m$.
\index{Soluble groups!satisfying max-n}%
\index{Soluble groups!with preassigned derived length}%
\index{Wreath products!examples}%
\index{Hall, P.}
\smallskip

b) Suppose  $H$ is a non-trivial abelian group.
The base group $A= \Dr_j \{qHq^{-1}\}$ of  $G = H \wr Q$ is then a non-trivial abelian normal subgroup.
Proposition \ref{prp:Sigma1-wreath-product} shows that
every character $\chi \colon G \to \R$   with $\chi(A) \neq \{0\}$ represents a point of $\Sigma^1(G)$.
This conclusion  holds in greater generality:
\begin{lem}
\label{lem:Abelian-normal-subgroup}
Let $G$ be a finitely generated group with an abelian normal subgroup $A \triangleleft G$.
Then $S(G,A)^c$ is contained in $\Sigma^1(G)$.
\end{lem}
\index{Computation of Sigma1@Computation of $\Sigma^1$ for!groups with abelian normal subgroups}

\begin{proof}
We use a variation of the argument employed in the third part of the proof of Proposition \ref{prp:Sigma1-wreath-product}.
Suppose $\chi \colon G \to \R$ does not vanish on $A$
and $t \in A $ is an element with $\chi(t) > 0$.
There exists then a finite generating system $\XX$ 
that includes $t$ and satisfies $\chi(x) > 0$  for every $x \in \XX$.
For each $x \in \XX$ the commutator  $r_x = t \cdot x tx^{-1} \cdot t^{-1} \cdot tx^{-1} t^{-1}$ 
is then a relator of $G$ 
whose $\chi$-track has positive values on every proper initial segment,
whence $[\chi] \in \Sigma^1(G)$ by Theorem \ref{thm:Algebraic-Sigma1-criterion}.
\end{proof}
Lemma \ref{lem:Abelian-normal-subgroup} 
will be generalized in section \ref{sssec:Subnormal-subgroups};
see Corollary \ref{crl:Sigma-1-subnormal-subgroups-with-Sigma1=S}.
\smallskip

c) Random walks on wreath products exhibit unusual behaviour;
see, \eg, \cite{PiSC02} and the references cited therein. 
In the literature on random walks some of these wreath products, 
in particular the metabelian group  $(\Z/2\Z) \wr C_\infty$,
are called \emph{lamplighter groups}.
\index{Computation of Sigma1@Computation of $\Sigma^1$ for!lamplighter groups}
\end{remarks}
\index{Lamplighter groups!definition}
%

\newpage

%
 
%
%
\section{Extensions to infinitely generated groups}
\label{sec:Extensions-infinitely-generated-groups}
%
%
The variations treated in this section have two features in common:
\begin{enumerate}[a)]
\item their definition makes sense for arbitrary groups $G$, 
and
\item they allow one to establish properties of $\Sigma^1$ with ease  
that are awkward to prove in the set-up of Chapter \ref{ch:Sigma-1-Cayley-graph}.
\end{enumerate}
It turns out that the Cayley graph definition of $\Sigma^1$ 
admits also of a straightforward extension to infinitely generated groups.
We start out with this extension and then discuss some of its properties and applications.
In section \ref{ssec:Ends-direction-character} 
we move on to the variation put forward and investigated  by Gaël Meigniez 
in \cite{Mei87}, \cite{Mei88} and \cite{Mei90}.
Then we discuss a characterization due to Ken Brown (see \cite{Bro87b})
(sections \ref{ssec:Tree-like-structures-associated-Cayley-graph} 
through \ref{ssec:Construction-morphisms-into-trees}).
In section \ref{ssec:Some-applications-of-Browns-characterization}, finally,  
we list some consequences of this characterization.
%
\subsection{Extending $\Sigma^1$ to arbitrary groups} 
\label{ssec:Extending-Sigma1-arbitrary-groups}
%
We begin by explaining 
how the definition of $\Sigma^1$ in terms of Cayley graph can be adapted to infinitely generated groups.
This extension will then be shown to coincide  with the alternative definition of Gaël Meigniez 
and also with that of Ken Brown.
The Cayley graph definition will allow one to comprehend 
why the definitions of Meigniez and Brown are often easier to work with
than the original definition,
a fact that is true even for finitely generated groups.
%
\subsubsection{The generalization}
\label{sssec:Generalizing-Sigma1}
%
Let $G$ denote an arbitrary group,
$\eta \colon \XX \to G$ a generating system and $\chi \colon G \to \R$ a non-zero character of $G$.
As in sections \ref{sssec:Terminology-Cayley-graph} and \ref{ssec:Definition-Sigma1},
one can then define the Cayley graph $\Gamma(G, \XX)$ 
and the full subgraph $\Gamma(G,\XX)_{\chi}$ induced by the submonoid 
$G_{\chi} = \{ g \in G \mid \chi(g) \geq 0 \}$.
The Cayley graph is connected, 
while its subgraph $\Gamma(G,\XX)_{\chi}$ need not be so.
Whether it is connected depends on $\chi$ --- this is intended --- but also on $\eta$. 
To get rid of this undesired dependence, 
we shall work, not with a single generating system, 
but with all of them.

Whether this manner of getting around a difficulty allows one to derive useful results 
remains to be seen. 
Here we are first concerned with the question 
whether the proposed definition is compatible with the original one.
The affirmative answer is given by

\begin{lem}
\label{lem:Compatibility-definition}
For every finitely generated group $G$ 
and every non-zero character $\chi \colon G \to \R$ the following statements are equivalent:
\begin{enumerate}[(i)]
\item there exists a finite generating system $\eta \colon \XX_{f} \to G$ 
for which the subgraph  $\Gamma(G, \XX_{f})_\chi$ is connected,
\item the subgraph  $\Gamma(G, \XX)_\chi$ is connected for every generating system $\eta \colon \XX \to G$.
\end{enumerate}
\end{lem}

\begin{proof}
Assume $\eta \colon \XX_{f} \to G$ is a finite generating system 
for which $\Gamma(G, \XX_{f})_{\chi}$ is connected, 
and let $\XX$ is a generating system of $G$.
Since $G$ is finitely generated $\XX$ contains a finite subset $\XX_{0}$ 
which generates $G$. 
Theorem   \ref{thm:Sigma1-well-defined} then shows
that the graph $\Gamma(G, \XX_{0})_{\chi}$ is connected.
Since $\Gamma(G, \XX_{0})_{\chi}$ is a subgraph of $\Gamma(G, \XX)_{\chi}$
and as both graphs have the same set of vertices,
the graph $\Gamma(G, \XX)_{\chi}$  is connected, too.

The previous argument shows that (i) implies (ii).
The converse is obvious. 
\end{proof} 

In view of  Lemma \ref{lem:Compatibility-definition},
the following definition of the invariant $\Sigma^1$ for arbitrary groups extends
the definition given in section \ref{ssec:Definition-Sigma1}:
\begin{definition}
\label{definition:Sigma1-arbitrary-group} 
\index{Invariant Sigma1 Z@Invariant $\Sigma^1$!for arbitrary groups}%
Given a group $G$, let $S(G)$ denote the set of rays $[\chi]$ in the vector space $\Hom(G,\R)$
that emanate from the origin and put
\begin{equation}
\label{eq:Sigma1-arbitrary-group}
\Sigma^1(G) = \{ [\chi] \in S(G) \mid  \Gamma(G, \XX)_{\chi} \text{ is connected for every gen. system } \XX\}.
\end{equation}
\end{definition}
\begin{remarks}
\label{remarks:Extending-Cayley-graph-definition}
a) 
In section \ref{ssec:Definition-Sigma1},
 the invariant $\Sigma^1$ has only been defined for \emph{finitely generated groups} $G$.
This restriction is responsible for two hallmarks of the theory treated up to now:
\begin{itemize}
\item the vector space $\Hom(G, \R)$ is finite dimensional and thus carries a canonical topology;
\item there exists a finitary condition $\FF\CC$ that characterizes the points $[\chi] \in \Sigma^1(G)$.
\end{itemize}
This finitary condition $\FF\CC$ --- 
introduced in Section \ref{sec:Sigma1-criterion} and called \emph{$\Sigma^1$-criterion} --- implies 
that the invariant is an open subset of the sphere $S(G)$;
in addition,
it is the main tool in the proof of many basic results,
but also of various explicit calculations, of $\Sigma^1$
treated in Chapters \ref{ch:Sigma-1-Cayley-graph} and 
\ref{ch:Sigma-1-Cayley-graph-complements}:
see, \eg., Theorems 
\ref{thm:Characterizing-fg-N},
\ref{thm:Algebraic-Sigma1-criterion},
\ref{thm:Sigma1-one-relator-group}, 
Propositions \ref{prp:Ascending-HNN-extension},
\ref{prp:psi(R)-subset-Sigma1} 
and \ref{prp:Psi(RR)-subset-Sigma1}
for results,
and 
section \ref{sssec:Group-PL-homeomorphisms-example}
for examples.

b)
The preceding chapters contain, however, also results
which are obtained by direct geometric arguments.
Instances of such results are sprinkled over the first two chapters,
starting with the very first examples 
(abelian groups and non-trivial free products in section \ref{sssec:Sigma1-first-examples}\,),
the computation of the invariant of a direct product
in example \ref{example:Direct-product-groups},
that of  a subgroup of finite index (Proposition \ref{prp:Sigma1-finite-index})
or that of a join of subgroups (Proposition \ref{prp:Sigma-1-join-subgroups}).
In these instances the finite generation of the groups plays  a minor r{\^{o}}le;
so the question arises whether this assumption can dispensed with.

c)
There are other reasons 
which  prompt one to generalize the invariant to infinitely generated groups $G$,
in spite of the fact 
 that for such a group  the vector space $\Hom(G, \R)$ can be infinite dimensional
 and, if so, does not carry a canonical topology.
 One reason is that a generalization to infinitely generated groups puts one in a position 
 to obtain results for a finitely generated group 
 by studying certain of its infinitely generated subgroups.
 In addition,
 the constructions treated in sections \ref{ssec:Ends-direction-character} 
 through \ref{ssec:Characterization-Sigma1-action}
 produce invariants that coincide with $\Sigma^1$ for finitely generated groups
 \footnote{See \cite[Th\'{e}or\`{e}me 3.19]{Mei90} and \cite[Theorem 5.2]{Bro87b}} 
 but have as their natural habitat the class of all groups.
\end{remarks}

To help the reader with getting used to determine the invariant  of an arbitrary group,
we continue with the determination of $\Sigma^1$ for four classes of examples.
%
\subsubsection[$\Sigma^1$ of groups with no free submonoids of rank 2]%
{Illustration 1: groups with no free submonoids of rank 2}
\label{sssec:Illustration-invariant-groups-with-no-free-submonoid}
Let $G$ be a group and $\chi \colon G \to \R$ a non-zero character.
In order to prove that $[\chi]$ lies in $\Sigma^1(G)$, 
one has to verify that the subgraphs $\Gamma_\chi$ of an infinite number of Cayley graphs 
$\Gamma(G, \XX)$ are connected.
In this section and the following one,
we show that this verification is very easy
if $G$ is a group that does not contain a free submonoid of rank 2
or if $G$ is an Engel group.

The basis of our verification will be the following
\begin{lem}
\label{lem:No-free-submonoid}
Let $\chi$ be a non-zero character of the group $G$.
Assume that for every couple $(x_1,x_2)$ of elements in G with $\chi(x_1) = \chi(x_2) > 0$ 
there exists words $u_1$ and $u_2$ in $\{x_1, x_1^{-1}, x_2, x_2^{-1} \}$ such that 
\begin{equation}
\label{eq:No-free-submonoid}
\begin{gathered}
x_1 \cdot u_1 \text{ and }  x_2 \cdot u_2  \text{ represent the same element in } G \text{ and}\\
 \text{the $\chi$-tracks of the words $u_1$ and $u_2$ are non-negative}.
\end{gathered}
\end{equation}
Then $[\chi] \in \Sigma^1(G)$.
\end{lem}
\begin{proof}
Let $\eta \colon \XX \to \R$ be a generating system of $G$ 
and $\Gamma$ the corresponding Cayley graph of $G$.
For every element $g \in \ker \chi \smallsetminus \{1_G \}$ 
there exists a path $p= (1_G, w)$ in $\Gamma$ 
that leads from $1_G$  to $g$; we may assume that $p$ is simple.
Let $h$ be a vertex of $p$ with minimal $\chi$-value 
and let $w_1$ the initial segment of $w$ that describes the part of $p$ from $1_G$ to $h$;
similarly, let $w_2$ the terminal segment of $w$ that describes the part of $p$ from $h$ to $g$.
If $\chi(h) = 0$, then $p$ runs inside the subgraph $ \Gamma_\chi$ 
and shows that  $g$ belongs to the component $\CC_1$ of 1 in $\Gamma_\chi$.

Otherwise, 
set $x_1 = h^{-1}$ and $x_2 = h^{-1} \cdot g$.
Then $\chi(x_1) = \chi(x_2) > 0$;
by hypothesis there exists therefore words $u_1$, $u_2$ 
which satisfy requirement \eqref{eq:No-free-submonoid}.
Let $\tilde{u}_1$, $\tilde{u}_2$ be the words in $\XX \cup \XX^{-1}$ 
that are obtained from $u_1$ and $u_2$, respectively,
by replacing each occurrence of the letter $x_1^\varepsilon$  by the word $w_1^{-\varepsilon}$ 
and each occurrence of the letter $x_2^\varepsilon$ by $w_2^{\varepsilon}$. 
Since $h$ has minimal $\chi$-value among the vertices of $p$,
the paths $(1_G,w_1^{-1})$ and  $(1_G,w_2)$ run inside $\Gamma_\chi$;
as the $\chi$-tracks of $u_1$ and $u_2$ are non-negative, 
the path $p_1 = (1_G, \tilde{u}_1)$ runs likewise inside the subgraph $\Gamma_\chi$.
Similarly, one sees that the path $(1_G, \tilde{u}_2)$ stays inside the subgraph $\Gamma_\chi$.
The facts 
that $x_1 \cdot u_1$ and $x_2 \cdot u_2$ represent the same element in $G$ and that $\chi(x_1) = \chi(x_2)$
imply next 
that $\chi(\tilde{u}_1) = \chi(\tilde{u}_2)$
and so the path $(1_G, \tilde{u}_1 \cdot \tilde{u}_2^{-1})$ runs inside $\Gamma_\chi$.
Its endpoint is $g$; indeed,
\[
\eta_*(\tilde{u}_1 \cdot \tilde{u}_2^{-1}) = u_1 \cdot u_2^{-1} 
= 
x_1^{-1} \cdot x_2 = (h^{-1})^{-1} \cdot h^{-1} g= g.
\]
All together together,
we have shown that every element of $\ker \chi$ belongs to $\CC_1$.
So Lemma \ref{lem:Restatement-connectivity-property} allows us to conclude 
that the subgraph $\Gamma(G,\XX)_\chi$ is connected. 
As $\eta \colon \XX \to G$ is an arbitrary generating system, the proof is complete
\end{proof}

There are two noteworthy situations 
in which the previous lemma can be applied to every non-zero character of the group.
The first one is treated in Corollary \ref{crl:Sigma1-Group-with-no-free-submonoid},
the second one in Corollary \ref{crl:Sigma1-Engel-group}.
\begin{crl}
\label{crl:Sigma1-Group-with-no-free-submonoid}
Assume $G$ is a group that contains no free submonoid of rank 2.
Then $\Sigma^1(G) = S(G)$.

If, in addition,  $G$ is finitely generated the members of the derived series
\[
G' = [G,G], \quad G'', \quad G''', \ldots
\] 
of $G$ are finitely generated and every soluble quotient of $G$ is polycyclic.
\end{crl}
\index{Computation of Sigma1@Computation of $\Sigma^1$ for!groups with no free submonoid}
\begin{proof}
Given a non-zero character $\chi$ of $G$ and two elements $x_1$, $x_2$ with $\chi(x_1) = \chi(x_2) > 0$,
there exist, by hypothesis, two distinct words $v'_1$, $v'_2$ in the alphabet $\{x_1, x_2\}$ 
which represent the same element of $G$;
they will have the same length. 
If they begin with the same letter, 
let $v$ be the largest common initial subword
and let $v_1$ and $v_2$ be the remainders. 
Then $v_1$ and $v_2$ are distinct words which also represent the same element of $G$
and start, one with $x_1$, the other with $x_2$.
By exchanging $v_1$ and $v_2$, if need be, 
we may arrange that $v_1$ begins with $x_1$ and $v_2$ with $x_2$.
Then $v_1$ has the form $x_1 \cdot u_1$ and $v_2$ the form $x_2 \cdot u_2$.
Since $u_1$, $u_2$ have non-negative $\chi$-tracks (for they are positive words)
their  $\chi$-tracks satisfy requirement \eqref{eq:No-free-submonoid}
and so $[\chi] \in \Sigma^1(G)$ by Lemma \ref{lem:No-free-submonoid}.

Assume now that $H$ is a finitely generated subgroup of $G$.
Then $H$ contains no a free submonoid of rank 2,
whence $\Sigma^1(H) = S(H)$ by the part already proved. 
Theorem \ref{thm:Characterizing-fg-N}  then implies
that the derived group $H'$ of $H$ is finitely generated.
If the group $G$ itself is finitely generated, 
this argument can be applied successively to the terms $G^{(m)}$ of the derived series of $G$ 
and shows that each of them is finitely generated. 
It follows, in particular, 
that every quotient group $G^{(m)}/G^{(m+1)}$ is finitely generated abelian
and that every quotient $G/G^{(m+1)}$ is polycyclic.
\end{proof}
\begin{note}
\label{note:Sigma1-Group-with-no-free-submonoid}
Corollary \ref{crl:Sigma1-Group-with-no-free-submonoid} is due to G. Levitt
(see \cite[p.\;659, Remarques]{Lev87}).
In \cite{LMR95}, the addendum to the corollary is proved by a direct argument (see Corollary 3 on page 1421.)
\end{note}
\index{Levitt, G.}
\index{Longobardi, P.}
\index{Maj, M.}
\index{Rhemtulla, A. H.}

\begin{examples}
\label{examples:No-free-submonoid}
a) There are two classes of groups 
which, for obvious reasons,  do not contain non-abelian free submonoids:
the class of \emph{torsion groups} and the class of \emph{abelian groups}.
More surprisingly, 
nilpotent and thus locally nilpotent groups contain no non-abelian free submonoids, either.
A  simple proof of this fact is given in A. Shalev's paper  \cite{Sha91} (see part (iv) of Proposition 7.1).
To describe this proof,
we introduce two sequences $n \mapsto u_n$ and $n \mapsto v_n$ of positive words on the alphabet $\{x,y\}$,
defined as follows:
\begin{align*}
u_1&= x y, &\quad v_1 &=yx,\\
u_2&= u_1 v_1 = x y^2 x, &\quad v_1&=v_1u_1 = yx^2y,
\end{align*}
and, in general, $u_{n+1}= u_n v_n$ and $v_{n+1} =v_n u_n$. 
Notice that $u_n(x,y)$ and $v_n(x,y)$ are distinct words (of length $2^n$).

Suppose now $x$ and $y$ are elements in a group $H$ 
such that the relation $u_n(x,y) = v_n(x,y)$ holds  modulo the centre $\zeta(H)$ of $H$.
There exists then a central element $z \in H$ with $v_n(x,y) = u_n(x,y) \cdot z$
and so 
\[
u_{n+1}(x,y) = u_n v_n = u_n \cdot u_n z =  u_n z  \cdot u_n  = v_n u_n = v_{n+1}(x,y).
\]
It follows inductively 
that the relation $u_n = v_n$ is a law in every nilpotent group of class at most $n$.

b) Consider next an \emph{extension $G$ of a locally nilpotent normal subgroup $N$ by a torsion group $G/N$}.
Given $g$ and $h$ in $G$ there exists then a positive integer $m$ 
such that $x = g^m$, $y = h^m$ lie in $N$;
by hypothesis, the elements $x$, $y$ generate a nilpotent subgroup, say of class $n$, 
whence part a) allows us to conclude 
that the equation $u_n(g^m,h^m) = v_n(g^m,h^m)$ holds in $G$. 
This equation testifies that $x$ and $y$ do not generate a free submonoid of $G$.

A concrete example of such a group $G$ is described on page 1420 of \cite{LMR95}: 
let $F$ be a non-abelian free group and $R \triangleleft F$ a normal subgroup 
such that $F/R$ is a torsion group.
Then $G = F/R'$ is abelian-by-torsion. 
Moreover, the canonical projection $F \epi F/R'$ induces an isomorphisms $F_{\ab} \iso G_{\ab}$.
It shows that $G_{\ab}$ is free abelian of positive rank and implies that $G$ is torsion-free.

c) Consider finally a finitely generated \emph{soluble} group $G$  
that does not contain a non-abelian free submonoid.
By a result of J. M. Rosenblatt $G$ is then nilpotent-by-finite 
(\cite{Ros74}, see also \cite[Theorem 1]{LMR95}). 
Conversely, 
by the previous example b) every nilpotent-by-finite group contains no non-abelian free submonoid. 
\end{examples}
\index{Rosenblatt, J. M.}
%

\subsubsection[$\Sigma^1$ of Engel groups]{Illustration 2: Engel groups}
\label{sssec:Illustration-invariant-Engel-groups}
%
Engel groups are generalizations of nilpotent groups.
They are defined like this (cf. \cite[pp.\;371--376]{Rob96}):
\begin{definition}
\label{definition:Engel-group}
Let $x$, $y$ be elements of a group and $n \geq 1$ a natural number.
The iterated commutators $ n \mapsto c_n(x,y)$ are defined inductively like this:
\begin{align}
c_1(x,y) &= [x,y] = x y x^{-1} y^{-1},
\label{eq:Iterated-commutator-1}\\
c_{n+1}(x,y) &= [c_{n}(x,y),y] = c_{n}(x,y) \cdot y \cdot c_{n}(x,y)^{-1} \cdot y^{-1} \text{ for } n \geq 1.
\label{eq:Iterated-commutator-n}
\end{align}

A group is called an \emph{Engel group} 
if there exists for every couple $(x,y)$ of elements in $G$ 
a natural number $k \geq1$ such that $c_k(x,y) = 1$. 
\end{definition}
\index{Definition of!Engel group}
The following result is the analogue of Corollary \ref{crl:Sigma1-Group-with-no-free-submonoid}:
\begin{crl}
\label{crl:Sigma1-Engel-group}
If $G$ an Engel group then $\Sigma^1(G) = S(G)$.
If $G$ is finitely generated so are the members of its derived series.
\end{crl}
\index{Computation of Sigma1@Computation of $\Sigma^1$ for!Engel groups}

\begin{proof}
Given a non-zero character $\chi$ of $G$,
let $x_1$, $x_2$ two elements of $G$ 
satisfying the condition $\chi(x_1) = \chi(x_2)$. 
By hypothesis, 
there exists then an index $m > 1$ so that the iterated commutator $c_m(x_1,x_2)$ vanishes.
If $m = 1$, 
the relation $x_1 \cdot x_2 = x_2 \cdot x_1$ is then valid 
and shows that the hypothesis of Lemma \ref{lem:No-free-submonoid} is satisfied for the triple $(\chi, x_1, x_2)$.

Assume next that $m>1$ and set $n = m-1$ as well as  $x= x_1$, $y=x_2$.
The relation 
\begin{equation}
\label{eq:Relation-derived-from-c(n+1)}
c_n(x,y) \cdot y = y \cdot c_n(x,y)
\end{equation} 
then holds.
In general, the inductive definition of  $c_n(x,y)$ does not produce a freely reduced word;
but is is easy to describe the free reduction of $c_n(x,y)$ by an inductive scheme.

\begin{lem}
\label{lem:Analysis-of-iterated-commutators}
Let $n \mapsto w_n$ be the sequence of words in $\{x,x^{-1}, y, y^{-1}\}$ defined by
\begin{equation}
\label{eq:Analysis-of-iterated-commutator}
w_1 = 1 \quad\text{and}\quad w_{n+1} = w_n \cdot x^{-1}yx \cdot w_n^{-1} \cdot y^{-1}.
\end{equation}
Then each word $w_n$ is freely reduced; if $n>1$, it starts with $x^{-1}$ and ends with $y^{-1}$.
Moreover,
for each $n \geq 1$, the word $c'_n = xy \cdot w_n \cdot x^{-1}y^{-1} $ is freely reduced and is the free reduction of $c_n(x,y)$.
\end{lem}
The lemma can be verified by a straightforward induction.
\smallskip 

Let's now go back to the relation  \eqref{eq:Relation-derived-from-c(n+1)}.
If one replaces in it the commutator $c_n(x,y)$ 
by its free reduction $c'_n = xy \cdot w_n \cdot x^{-1}y^{-1}$ 
and discards the subword $y^{-1} y$ at the end of the left term, 
one arrives at the relation
\[
x \cdot u = y \cdot v \quad \text{with  } u = yw_n x^{-1} \text{ and } v = x(y w_n x^{-1})y^{-1}.
\]
Since $\chi(x) = \chi(y) > 0$, 
the inductive definition of $w_n$ allows one to deduce
that the $\chi$-track of $u$ is non-negative,
whence so it that of $v$.
The hypothesis of  Lemma \ref{lem:No-free-submonoid} holds therefore also for indices $m > 1$.

The additional claim then follows as in the proof of Corollary \ref{crl:Sigma1-Group-with-no-free-submonoid}.
\end{proof}

\begin{note}
\label{note:Addendum-crl-Sigma1-Engel-group}
The addendum to Corollary \ref{crl:Sigma1-Engel-group} is due to K. W. Gruenberg (\cite[Theorem 2]{Gru53}).
\end{note}
\index{Gruenberg, K. W.}

\begin{examples}
\label{examples:Engel-group}
Every nilpotent group, hence every locally nilpotent group, is an Engel group.
Conversely, every \emph{finite} Engel group is nilpotent (\cite{Zor36}), 
as is every \emph{finitely generated soluble} Engel group (\cite[Theorem 1]{Gru53}).

On comparing these results with those stated in examples \ref{examples:No-free-submonoid} a) and c),
one sees that, for finite groups and for finitely generated soluble groups,
the hypothesis of being an Engel group is more stringent 
than that of being a group with no free submonoids of rank 2.
\end{examples}
%
\subsubsection[Free products with amalgamations and HNN-extensions]
{Illustration 3: amalgamated free products and HNN-extensions}
\label{sssec:Illustration-HNN-extensions}
%
In the preceding sections \ref{sssec:Illustration-invariant-groups-with-no-free-submonoid}
and \ref{sssec:Illustration-invariant-Engel-groups},
we considered situations which allow one to conclude
that every point $[\chi]$ lies in $\Sigma^1(G)$.
In this section, we turn to constructions 
that allow one to infer that certain points are outside of $\Sigma^1$.
The constructions will amalgamated free products and HNN-extensions.
One has:
\begin{prp}
\label{prp:Point-outside-Sigma1-HNN-extensions}
\begin{enumerate}[(i)]
\item Let $G$ be an amalgamated free product $G_1 \star_A G_2$
and $\chi \colon G \to \R$ a non-zero character.
If  $A$ is distinct from both $G_1$ and $G_2$
and $\chi$ vanishes on $A$, then $[\chi] \notin \Sigma^1(G)$.

\item Let $G$ be an HNN-extension 
$\langle B, t \mid t \cdot  s \cdot t^{-1} = \sigma(s) \text { for } s \in S \rangle$
with base group $B$, associated subgroups $S$, $T$ and isomorphism $\sigma\colon S \iso T$.
If the character $\chi$ vanishes on $S$ then $[\chi] \notin \Sigma^1(G)$,
except, possibly, if  $\chi(t) < 0$ and $S = B$ or $\chi(t) > 0$ and $T = B$.
\end{enumerate}
\end{prp}
\index{Computation of Sigma1@Computation of $\Sigma^1$ for!amalgamated free products}
\index{Computation of Sigma1@Computation of $\Sigma^1$ for!HNN-extensions}

\begin{proof}
(i) Assume first that $\chi$ is non-zero on $G_1$.
Choose elements $g_1 \in G_1$ with $\chi(g_1) > 0$
and $g_2 \in G_2 \smallsetminus A$ with $\chi(g_2) \geq 0$,
and set $g = g_1^{-1} g_2 g_1$.
Then $\chi(g_1^{-1}) < 0$, but  $\chi(g)\geq 0$.
Consider now the generating set $\XX = G_1 \cup G_2$; 
let $\Gamma = \Gamma(G, \XX) $ be the associated Cayley graph 
and $\Gamma_\chi$ its subgraph corresponding to $\chi$. 

We claim that $\Gamma_\chi$ is not connected.
Indeed,
the word $w_0 = (g_1^{-1}, g_2, g_1)$ is in reduced form and it describes a path $p = (1, w_0)$ in $\Gamma$
that start and ends in $\Gamma_\chi$, but leaves $\Gamma_\chi$.
If $w$ is another word in the alphabet $\XX^\pm = G_1 \cup G_2$ 
that defines the element $g \in G$,
then $w$ can be reduced to a word $w'$ in normal form.
In this reduction the value of the function $v_\chi \colon W(\XX^\pm) \to \R$ will not decrease.
By the \emph{Reduced-Form Theorem} (see, \eg, \cite[Chapt.\,1, Thm.\,26]{Coh89}),
$w'$ has the form $(h_1, h_2, h_3 )$ with $h_1 \in G_1 \smallsetminus A$,
$h_2 \in G_2 \smallsetminus A$, $h_3 \in G_1 \smallsetminus A$ and $h_1 \in g_1^{-1} A$.
But if so, the assumption that $\chi$ vanishes on $A$ implies
that $\chi(h_1) = \chi (g_1^{-1}) < 0$ and thus the subgraph $\Gamma_\chi$ is not connected.

If $\chi$ vanishes on $G_1$,
exchange the rôles of $G_1$ and $G_2$ in the preceding argument
and infer, as before, that  $[\chi] \notin \Sigma^1(G)$.
\smallskip

(ii) The proof will be similar to the preceding one; 
it is based on the \emph{Reduced Form Theorem} for HNN-extensions 
(see, for instance, \cite[p.\,36, Thm.\,37]{Coh89}).

Three cases will be distinguished;
in each of them we choose $\XX = B \cup \{t\}$ as generating set. 
Assume first that $\chi$ \emph{does not vanish on} $B$.
Choose an element $b \in B$ with $\chi(b) > 0$; notice that $b \notin S \cup T$.
Consider the word $w_0 = (b^{-1}, t^\varepsilon, b)$
where the sign $\varepsilon$ is to be chosen so that $\chi(y^\varepsilon) \geq 0$.
Then $w_0$ is in reduced form;
as $\chi$ is negative on its first letter it follows, as in part (i),
that $\Gamma_\chi$ is not connected.

Assume now that $\chi(B) = \{0\}$.
If $\chi(t) < 0$ and $S \neq B$, pick $b \in B \smallsetminus S$ 
and consider the word $w_0 = (t, b, t^{-1})$. 
It is in reduced form and allows one to deduce that $\Gamma_\chi$ is not connected.
If, on the other hand, $\chi(t) > 0$ and $T \neq B$, find $b \in B \smallsetminus T$, 
and set $w_0 = (t^{-1}, b, t)$.
This word is reduced and reveals that $\Gamma_\chi$ is not connected.
\end{proof}

\begin{remarks}
\label{remarks:Sigma1-HNN-extension}
a) Part (i) of Proposition \ref{prp:Point-outside-Sigma1-HNN-extensions}
generalizes the conclusion stated in Example  3 of section \ref{sssec:Sigma1-first-examples}.

b) The exceptional cases listed in statement (ii) of the preceding proposition deserve a comment.
Assume that $\chi(t) > 0$.
If $[\chi] \in \Sigma^1(G)$ the proposition implies that $T$ coincides with $B$.
This conclusion brings to light that Proposition \ref{prp:HNN-extension-and-Sigma1} continues to be valid 
if the base group $B$ is allowed to be infinitely generated.

Suppose now that $T = B$. 
If $B$, and hence $G$, are \emph{finitely generated}, 
then $[\chi] \in \Sigma^1(G)$ by Proposition \ref{prp:Ascending-HNN-extension}.
If, however, $B$ is infinitely generated  the point $[\chi] $ may lie outside of $\Sigma^1(G)$, 
but it can also be contained in it.
One reason is this: 
the assumption that $T = B$ is infinitely generated does not exclude the case where $B = N = \ker \chi$.
Then each of the three situations
\[
\text{a) } [\chi] \in \Sigma^1(G) \text{ and } [-\chi] \notin \Sigma^1(G),
\quad
\text{b) } [\chi] \notin \Sigma^1(G) \text{ and } [-\chi] \in \Sigma^1(G),
\]
and $[\chi] \notin \Sigma^1(G) $ and  $[-\chi] \notin \Sigma^1(G)$ are possible.

In the previous example the group $G$ may be chosen to be finitely generated.
Here is a different type of example where $G$ will be infinitely generated.
Let $B$ be a strictly ascending union $\bigcup\nolimits_{j \in \N} B_j$ of finitely generated subgroups $B_j$
and assume that $B_j \subseteq tB_j t^{-1} $ for every index $j \in \N$. 
Then $G = \gp(B \cup \{t\})$ is an ascending union of  finitely generated subgroups $G_j = \gp(B_j \cup \{t\})$,
and $[\chi_{|G_j}] \in \Sigma^1(G_j)$  for every index $j$.
Proposition \ref{prp:Sigma-1-join-subgroups-2} thus allows us to conclude that $[\chi] \in \Sigma^1(G)$.

Particular instances of the previous type of example are certain extensions of locally finite groups 
by infinite cyclic groups.
They reveal that the conclusion of Lemma \ref{lem:Infinite-locally-finite-group}
need not be true if the group $G$ is not finitely generated.
\end{remarks}

We conclude with an application of Proposition \ref{prp:Point-outside-Sigma1-HNN-extensions}.
\begin{crl}
\label{crl:Invariant-group-with-infinitely-many-ends}
If $G$ is a finitely generated group with infinitely many ends then $\Sigma^1(G)$ is empty.
\end{crl}
\begin{proof}
By Stallings' structure theorem (see \cite[5.A.9]{Sta71}), 
a finitely generated group $G$ with infinitely many ends
is either a non-trivial free product $G_1\star_A G_2$ with finite amalgamated group $A$,
properly contained in both factors and of index greater than 2 in at least one factor,
or an HNN-extension  
\[
\langle B,t; t \cdot  s \cdot t^{-1} = \sigma(s) \text{ for } s \in S\rangle
\] 
with finite associated subgroups $S$, $T$, both properly contained in $B$. 
As every character vanishes on a finite subgroup,
Proposition \ref{prp:Point-outside-Sigma1-HNN-extensions}
allows us to see that no character represents a point of $\Sigma^1(G)$.
\end{proof}
\index{Computation of Sigma1@Computation of $\Sigma^1$ for!groups with infinitely many ends}
%
\subsubsection{A glimpse of the invariants $\Sigma^M$ and $\Sigma^B$}
\label{sssec:Glimpse-invariants-Meigniez-and-Brown}
%
In this final part of section \ref{ssec:Extending-Sigma1-arbitrary-groups}
I shall describe two constructions
that underlie the proofs of the equivalence of $\Sigma^1$ with the invariants 
propounded by Gaël Meigniez in \cite{Mei87}, \cite{Mei90} and by Ken Brown in \cite{Bro87b}.

Let $G$ be a group and  $\chi \colon G \to \R$ a non-zero character.
Given a generating system $\eta \colon \XX \to G$ of $G$ 
we are interested in finding out 
whether the subgraph $\Gamma_\chi = \Gamma(G, \XX)_\chi$ 
of the Cayley graph $\Gamma(G, \XX)$  is connected.
If the subgraph $\Gamma_\chi$ is \emph{not} connected
it has infinitely many components.
\footnote{See the proof of Corollary \ref{crl:Infinitely-many-components}.}
In a case where one suspects that $\Gamma_\chi$ is not connected 
it might therefore be more profitable to consider subsets of connected components 
rather than individual components,
and try to show that there exists such a subset whose union is neither empty nor all of $\Gamma_\chi$.
One is thus led to consider subsets of $G$ 
with properties resembling those of the vertex sets of unions of connected components of $G_\chi$.
Gaël Meigniez has found a framework in which such subsets can conveniently be discussed.
One benefit of his approach is the fact 
that the generating system $\XX$ does not enter into the calculations.

Ken Brown has detected another avenue to the problem of finding out 
whether the subgraph $\Gamma_\chi $ is connected;
it involves an action of $G$.
Let  $C = C(\Gamma_\chi)$ denote the set of all connected components of $\Gamma_\chi$.
The group $G$ acts on the Cayley graph $\Gamma(G, \XX)$ by left translations,
but this action does not induce one on $C$, 
as $C$ in only invariant under $N = \ker \chi$.
Now in the applications of $\Sigma^1$ one typically has some information on $G$ 
and aims at using it to deduce properties of $N$.
So one should try to enlarge the $N$-set $C(\Gamma_\chi)$ to a $G$-set,  say $T$.
This goal can be reached in different ways, one solution being the following one.

Let $T = T(G, \XX, \chi)$ be the set of all connected components of the subgraphs
$\Gamma(G, \XX)_\chi^{[b, \infty)}$ 
with $b$ ranging over $\im \chi \subseteq \R$.
Then $T$ carries an obvious $G$-action and it is equipped with additional structure:
the homomorphism $\chi$, for instance,  induces  a  map $\lambda \colon T \to  \R$ 
that sends a component of $\Gamma(G, \XX)_\chi^{[b, \infty)}$  to $b$.
This map $\lambda$ is $G$-equivariant;
it is injective if, and only if, $\Gamma_\chi$ is connected 
or, expressed in a more suggestive, geometric language, if $T$ is a ``line''.
One crucial observation is now this:
every element $g \in G$ with $\chi(g) \neq 0$ gives rise to a subset $A_g$ 
that is invariant under $\gp(g)$ and enjoys properties similar 
to those of the axis of a hyperbolic element of a group acting on an $\R$-tree
(see Lemma \ref{lem:Properties-A-sub-g} for details).
Using the axes $A_g$ of suitably chosen elements $g$ 
one can try to deduce that $T$ is a line;  
if one succeeds, one has shown that $\Gamma_\chi$ is connected.

The preceding paragraph indicates 
how the construction of $T$ can help one in proving 
that a given point $[\chi]$ belongs to $\Sigma^1(G)$.
But more is true.
The additional structure of the $G$-set $T$ turns it into a kind of tree, 
to be called the \emph{measured tree} associated to $G$, $\XX$ and $\chi$.
One has also a notion of an abstract measured tree
which  comprises, in particular, certain quotients of $T$.
An astonishing fact is now this:
it there exists, 
given a group $G$ and a  non-zero character $\chi$ of $G$,
an abstract measured tree $T'$ which is \emph{not} a line 
there exists a generating system $\XX$ of $G$ 
so that the measured tree $T(G, \XX, \chi) $ associated to $G$, $\XX$ and $\chi$ is not a line
(see Proposition \ref{prp:Constructing-morphism-from-T(G,XX,chi)}),
whence $[\chi] \notin \Sigma^1(G)$.
So the approach via actions on tree-like structures also provides a tool for showing 
that a point $[\chi]$ does \emph{not} belong to $\Sigma^1(G)$.

The approach has a third merit.
The notion of an abstract measured tree $T'$ does not involve the choice of a generating system;
in particular,
an abstract measured tree $T'$ associated to $G$ and $\chi$ 
gives rise to an abstract measured tree for every subgroup $G_1$ of $G$ with $\chi(G_1) \neq \{0\}$.
This fact allows one to use the hypothesis that $[\chi|G_1] \in \Sigma^1(G_1)$ 
in proving that $[\chi] \in\Sigma^1(G)$. 
In this way one can show, for instance,
that the statement of Proposition \ref{prp:Sigma-1-join-subgroups},
dealing with joins of subgroups, carries over to infinitely generated groups
(see Proposition \ref{prp:Sigma-1-join-subgroups-2}).
%
%

%
\subsection{Ends in the direction of a character (G. Meigniez)} 
\label{ssec:Ends-direction-character}
%
Gaël Meigniez introduces in  \cite{Mei87} (cf.\,\cite{Mei88} and \cite{Mei90}) an invariant $\Sigma^M$
whose definition is phrased in terms reminiscent of the number of ends of a space.
This section begins therefore with a reminder of this notion.
Then follow the definition of  the invariant $\Sigma^M$
and the proof of the equality of $\Sigma^M$ and $\Sigma^1$
(see Theorem \ref{thm:Equality-Sigma1-Meigniez-invariant}).
%
\subsubsection[Number of ends of a topological space]%
{Number of ends of a topological space --- a reminder}
\label{sssec:Number-ends-topological-space}
%
Let  $X$ be a locally finite, connected simplicial complex.
For each finite subcomplex $K$, 
the number of connected components of $X \smallsetminus K$ is finite;
let $n(K)$ denote the number of the infinite ones.
The number of ends $e(X)$ of the space $X$ is then defined to be the supremum
\begin{equation}
\label{eq:Number-ends-topological-space}
\index{Number of ends!of a topological space}%
\sup \{n(K) \mid K \subset  X \text{ finite subcomplex} \}.
\end{equation}
This number is either a positive integer or $\infty$.
(See, \eg, \cite[Section 5]{ScWa79} for further information.)
\index{Number of ends!of a topological space}

The preceding definition applies, in particular, to the Cayley graph $\Gamma = \Gamma(G, \XX)$ 
of an infinite group with a finite generating system $\eta \colon \XX \to G$. 
The infinite components of a subset $\Gamma  \smallsetminus K$ 
are then infinite connected subgraphs  $S \subset \Gamma  \smallsetminus K$ 
whose boundaries $\delta(S) = \delta_{\XX}(S)$ are finite.
Here a vertex $v$ is said to lie in the boundary of $S$ 
if it is the endpoint of an edge having one extremity in $S$ and the other in $G \smallsetminus S$;
more precisely,
\begin{equation}
\label{eq:Definition-boundary}
\delta_{\XX}(S) 
= 
\bigcup\nolimits_{y \in \XX \cup \XX^{-1}} (S\cdot y \smallsetminus S)\cup (S\smallsetminus S \cdot y).
\end{equation}

The concept of the \emph{number of ends} of a topological space goes back to \cite{Hop44}. 
In this paper,  Heinz Hopf establishes, inter alia, the following results:
\begin{enumerate}[(i)]
\item  let $X = \widetilde{Y}$ be the universal covering of a finite, connected  simplicial complex $Y$
with infinite fundamental group $G = \pi_{1}(Y)$.
Then the number of ends of $X$  is 1, 2 or $\infty$ (\cite[p.\;91, Satz I]{Hop44});
\item the number of ends of a Cayley graph $\Gamma = \Gamma(G, \XX)$ of an infinite, 
finitely generated group does not depend on the choice of the finite generating system $\eta \colon \XX \to G$
and is therefore an invariant of the group $G$ (\cite[p.\;96, Art.\;16]{Hop44}).
\end{enumerate}
\index{Hopf, H.}

Hopf establishes claims (i) and (ii) by topological arguments 
and then raises the question of finding a purely algebraic theory of the number of ends of a group.
Hans Freudenthal proposes such a theory in \cite{Fre45}.
\index{Freudenthal, H.}

Here are some elements of his theory;
they will be important in the sequel.
Assume $G$ is a finitely generated, infinite  group,
$\XX \incl G$ is a finite generating set of $G$,
and $\Gamma = \Gamma(G, \XX)$  the corresponding Cayley graph.
Consider a subset $A \subseteq G$. 
It gives rise to  a (full) subgraph $\Gamma_{|A}$ of $\Gamma$;
let $\delta_{\XX}(A)$ denote its boundary  in $\Gamma$.
Then  $\delta_{\XX}(A)$ is finite if, and only if,  the set
 \begin{equation}
 \label{eq:Definition-delta-g}
 \delta_{h}(A) = (Ah \smallsetminus A) \cup (A \smallsetminus Ah)
\end{equation}
is finite for every $h \in G$. 
It follows 
that the answer to the question whether a set $A$ has a finite boundary $\delta_{\XX}(A)$ 
with respect to a finite generating set $\XX$ of $G$ 
does not depend on the choice of $\XX$.

Let us apply this fact to an \emph{infinite connected component}  $\CC$ 
of the subgraph of $\Gamma$ spanned by the complement  $G  \smallsetminus \FF$ of a finite subset $\FF$.
The boundary of $\CC$ is then finite.
Assume now $\XX'$ is another finite generating set of $G$.
Then $A =V(\CC)$ induces a subgraph $\Gamma'_{|A}$ of $\Gamma(G,\XX')$.
This subgraph may not be connected, but it has a finite boundary 
and so it contains at least one infinite connected component.
As the rôles of $\Gamma(G, \XX)$ and $\Gamma(G, \XX')$ can be interchanged,
the previous reasoning constitutes an algebraic proof of Hopf's assertion (ii); 
this proof yields, in addition, a characterization of groups  with one end:
\emph{a finitely generated group $G$ has one end if, 
and only if, 
the following implication holds for every subset $A \subset G$}:
 \begin{equation}
 \label{eq:Group-one-end}
 \delta_{h}(A) \text{ is finite for every }  h \in G \Longrightarrow A \text{ or } G \smallsetminus A \text{ is finite}.
 \end{equation}
This algebraic characterization of a finitely generated group with one end
will be the starting point of the next section.
%
\subsubsection{Number of ends in the direction of a character}
\label{sssec:Number-ends-direction-character}
%
Let $G$ be a group and $\chi \colon G \to \R$ a non-zero character.
The definition of a group with one end in the direction of $\chi$ 
is obtained from characterization  \eqref{eq:Group-one-end} of a finitely generated group with one end
by replacing in the latter characterization the word \emph{finite} 
by the phrase \emph{with $\chi$-values that are bounded from above}.
\begin{definition}
\label{definition:Group-one-end-direction-character}
\index{Number of ends!in the direction of chi@in the direction of $\chi$}%
$G$ has \emph{one end in the direction of  $\chi$} 
if the following implication holds for every subset $A \subset G$:
\begin{gather*}
\chi(\delta_{h}(A)) \text{ is bounded from above for every }  h \in G \\
\Downarrow\\
\chi(A) \text{ or } \chi(G \smallsetminus A) \text{ is bounded from above}.
 \end{gather*}
 \end{definition}
 
 A glance at  definition \ref{definition:Group-one-end-direction-character} shows 
that if $G$ has one end in the direction of $\chi$ 
it has one end in the direction of every positive multiple of $\chi$.
The definition allows one therefore to single out a subset $\Sigma^M(G)$ of the sphere $S(G)$,
namely  
 \begin{equation}
\label{eq:Definition-Meigniez-invariant}
\index{Invariant SigmaM@Invariant $\Sigma^M$!definition}%
\index{Definition of!invariant SigmaM@invariant $\Sigma^M$}%
\Sigma^{M} (G) = \{ [\chi] \in S(G) \mid G \text{ has one end in the direction of } \chi \} .
\end{equation}

We shall prove that $\Sigma^{M}$ coincides with $\Sigma^1$.
Prior to establishing this result, we pause for some historical comments.
 \begin{notes}
 \label{notes:Group-infinitely-many-ends-direction-character}
a) The concept of \emph{ends of a group $G$ in the direction of a character} 
 is due to Gaël Meigniez  (see \cite[\S 1]{Mei87}).
 In \cite{Mei87}, a number of results involving this concept are established,
 in particular the following two:
 firstly, if $G$ is finitely generated, 
 the subsets $\Sigma^{M}(G)$ and $\Sigma^1(G)$ coincide (Corollaire 5); 
 secondly, the analogue of Proposition \ref{prp:Point-outside-Sigma1-HNN-extensions} for the invariant $\Sigma^M$
 (see Theorems 4.1 and 4.2).

  b) The concept of ends of a group in a given direction is also of interest for other types of ``characters''.
  This novel type of generalization is discussed in Gaël Meigniez' thesis \cite{Mei88} 
  and in his article \cite{Mei90};
  it has been taken up by Robert Bieri in \cite{Bi93}.
\end{notes}
\index{Bieri, R.}
\index{Meigniez, G.}
%
\subsubsection{Equivalence of $\Sigma^M$ and $\Sigma^1$}
\label{sssec:Equivalence-Sigma-super-M-Sigma1}
%
In this section, 
it is shown that a group $G$ has one end in the direction of a character if, and only if, 
that character represents a point of $\Sigma^1(G)$.
We begin with
\begin{lem}
\label{lem:Boundary-delta-h}
Let $G$ be a group with generating system $\eta \colon \XX \to G$.
Consider a subset $A$ of $G$.
If  $\chi(\delta_{x}(A))$ is bounded from above for each $x \in \XX$
then  $\chi(\delta_{h}(A))$ is bounded from above for every $h \in G$.
\end{lem}

\begin{proof}
Given an element $h \in G$, 
write it as a product   $y_{1}\cdot y_{2}\cdots y_{\ell}$
with factors $ y_{j} = x_{j}^{\varepsilon_{j }} \in \XX \cup \XX^{-1}$.
By hypothesis, the set $\chi( \delta_{x}(A))$ is bounded from above 
for each $x \in \XX$. 
The calculation 
\begin{align*}
\delta_{x^{-1}}(A) 
&= 
(Ax^{-1} \smallsetminus A) \cup (A\smallsetminus Ax^{-1} )  \\
&=
(A \smallsetminus Ax) \cdot x^{-1} \cup  (Ax \smallsetminus A)\cdot x^{-1}
=
\delta_{x}(A) \cdot x^{-1}
\end{align*}
then shows that each $\chi(\delta_{x^{-1}}(A))$ is likewise bounded from above. 
It follows, in particular,  that each set $\chi(\delta_{y_{i}}(A)$ is bounded from above;
in particular, $\chi(\delta_{y_{1}}(A))$.
Write $h$ as a product $h' \cdot y$ with $h' = y_{1}\cdot y_{2}\cdots y_{\ell-1}$ and $y = y_{\ell}$
and assume, inductively, that $\chi(\delta_{h'}(A))$ is bounded from above.
The computation
\begin{align*}
\delta_{h'y}(A) 
&= 
(Ah'y \smallsetminus A) \cup (A\smallsetminus Ah'y )  \\
&\subseteq
(Ah'y \smallsetminus Ay) \cup  (Ay \smallsetminus A) \cup  (A \smallsetminus Ay)\cup (Ay \smallsetminus Ah'y)\\
&=(Ah' \smallsetminus A) \cdot y \cup \delta_{y}(A) \cup (A \smallsetminus Ah') \cdot y\\
&= \delta_{y}(A) \cup \delta_{h'}(A)\cdot y
\end{align*}
then shows  that $\chi(\delta_{h}(A))$ is bounded from above.
\end{proof}
\begin{thm}
\label{thm:Equality-Sigma1-Meigniez-invariant}%
\index{Invariant SigmaM@Invariant $\Sigma^M$!equality with Sigma1@equality with $\Sigma^1$}%
The invariants $\Sigma^{M}(G)$ and $\Sigma^1(G)$ coincide for every group $G$.
\end{thm}
\begin{proof}
Let $\chi \colon G \to \R$ be a non-zero character.
If $[\chi] \notin \Sigma^1(G)$,
there exists a generating system $\eta \colon \XX \to G$ 
for which  the subgraph $\Gamma_{\chi}$ is not connected. 
We claim that every path component $\CC$ of $\Gamma_{\chi}$ has arbitrary large $\chi$-values,
but boundaries which are bounded from above.
Set $A = V(\CC)$. 
Clearly $\chi(A)$ is not bounded from above.
Consider now a boundary $\delta_{h}(A)$ with $h \in G$.
Assume first that $h = x \in \XX$.
If $g$ is an element of
\[
 \delta_{x}(A) = (Ax \smallsetminus A) \cup (A \smallsetminus Ax)
 \]
then either $g \notin A$ but $g x^{-1} \in A$, or $g \in A$ but $gx^{-1} \notin A$. 
In the first case, $ \chi(g) < 0$;
in the other case, $\chi(x) > 0$ and $\chi(g) < \chi(x)$.
In both cases, 
$\chi(g)$ is bounded from above by $|\chi(x)|$.
It follows that $\chi(\delta_{x}(A) )$ is bounded from above for each $x \in \XX$
and so Lemma \ref{lem:Boundary-delta-h} implies
that $\chi(\delta_{h}(A) )$ is bounded from above for every $h \in G$.
The fact that $\Gamma(G, \XX)_\chi$ has several (infinite) path components allows us, finally,  
to deduce that $G$ \emph{has more than one end in the direction of} $\chi$.
\smallskip

Conversely,  assume that $[\chi] \in \Sigma^1(G)$.
Let $A$ be a subset of $G$ which has unbounded $\chi$-values,
but boundaries $\delta_{h}(A)$ whose $\chi$-values are bounded from above.
We aim at showing
that $\chi(G \smallsetminus A)$ is bounded from above.
To reach this goal,
we shall construct a generating set $\XX$ of $G$ 
and then consider the subgraph $\Gamma(G, \XX)_\chi$;
since $[\chi] \in \Sigma^1(G)$ this subgraph is connected.
The construction of $\XX$ will imply 
that the vertex set of a translate of $\Gamma(G, \XX)_\chi$  is contained in $A$; 
this, in turn, will allow us to conclude 
that the $\chi$-values of $G \smallsetminus A$ are bounded from above.

Choose an element $t \in G$ with $\chi(t) > 0$ 
and let $\beta_{t}$ denote an upper bound of  $\chi(\delta_{t}(A))$.
Since $\chi$ is not bounded on $A$, there exists a vertex $g_{*} \in A$ with $\chi(g_{*}) \geq \beta_{t}$.
The ray $\ell \mapsto g_* \cdot t^\ell$ belongs then entirely to $A$;
here $\ell \geq 0$.
Consider now an element $h \in G_\chi \smallsetminus \{t\}$. 
Since $\chi( \delta_h(A))$ is bounded from above 
there exists a positive integer $\ell(h)$ so that the path 
\[
p_{h} = (g_{*}, t^{\ell(h)} h t^{-\ell(h)})
\] 
has all its vertices in  $A$.
(Recall that $\chi(\delta_t(A)$ is bounded from above by $\beta_{t}$ and that $\chi(g_*) \geq \beta_t$.)
The set
\[
\XX  
= 
\left\{t \right\} \cup \left\{x_{h} = t^{\ell(h)} h t^{-\ell(h)} \mid h \in G_{\chi} \smallsetminus \{t\} \right \}.
\]
obviously generates $G$;
since $[\chi] \in \Sigma^1(G)$,
the subgraph $\Gamma(G, \XX)_{\chi}$ is therefore connected.
For each $g \in G_{\chi}$ there thus exists an $\XX^\pm$-word $w$
that represents $g$ and has non-negative  $\chi$ values on all its initial segments;
let 
$x_{1}^{\varepsilon_{1}} \cdots x_{k}^{\varepsilon_{k}}  $
be the spelling of $w$. 
Each generator $x_{j}$ is a conjugate, say 
\[
x_j =  t^{\ell_j} h_{j} t^{-\ell_j},
 \]
 of an element $h_j \in G_{\chi}$ (if $x_j = t$ set $\ell_j = 0$.
 The word $w$ gives rise to the sequence
 \[
 s 
 =  
 (t^{\ell_1} ,h_{1}^{\varepsilon_{1} }, t^{-\ell_1} , t^{\ell_2} ,h_{2}^{\varepsilon_{2} }, t^{-\ell_2},
 \ldots , 
 t^{\ell_k}, h_{k}^{\varepsilon_{k} },  t^{-\ell_k)} ).
 \]
 This sequence, in turn, leads to a path $p_{g} = (g_{*}, s)$ that starts at $g_{*}$ and ends at $g_{*}g$;
 we assert that all of its vertices lie in $A$.
 
 For the origin  and the vertices of the subsequence
 $s_1 =(t^{\ell_1} ,h_{1}^{\varepsilon_{1} }, t^{-\ell_1} )$
 this holds by the choice of $g_{*}$  and the integer $\ell_1$;
 it follows that  the end point $g_{1} =g_{*}x_{1}^{\varepsilon_{1}}$ of the subpath 
 corresponding to $s_1 = x_1^{\varepsilon_1}$ lies in $A$.
 Since the $\chi$-value of the end point $g_{1}$ is at least as large as that of $g_{*}$,
 the choice of $h_2$ then implies 
 that the vertices of the second subsequence 
 $s_2 =(t^{\ell_2} ,h_{2}^{\varepsilon_{2} }, t^{-\ell_2} )$
 lie in $A$, 
 and that the $\chi$-value of the end point $g_{2}$ of the subpath corresponding to 
 $x_{1}^{\varepsilon_1} \cdot x_{2}^{\varepsilon_2}$  
 is at least as large as  $\chi(g_{*}) \geq \beta_t$.
 And so on and so forth.
It follows that $g_{k} = g_{*}g$, the end point of the entire path, lies in $A$.

The subgraph $\Gamma(G, \XX)_\chi$ is connected;
hence so the translated subgraph
 \[
 G_{\chi}^{[\chi(g_*), \infty)} = \{g \in G \mid \chi(g) \geq \chi(g_*) \}.
 \]
The preceding argument thus implies that the vertex set of this entire subgraph is contained in $A$
and so $\chi(G \smallsetminus A)$ is bounded from above, 
namely by by $\chi(g_*)$.
\end{proof}

\begin{note}
\label{note:Equality-Sigma-Meigniez-Sigma1}
The above theorem is due to G. Meigniez:
the inclusion $\Sigma^1(G)^c \subseteq \Sigma^{M}(G)^c$ is established in \cite{Mei87}
for finitely generated groups
and the argument given there generalizes easily to arbitrary groups.
The proof of inclusion $\Sigma^1(G) \subseteq \Sigma^{M}(G)$ 
is taken from a letter which G. Meigniez posted to R. Strebel in 1991.
\end{note} 
\index{Meigniez, G.} 
%

%
\subsection{Tree-like structures associated to a Cayley graph} 
\label{ssec:Tree-like-structures-associated-Cayley-graph}
%
The discussion of Ken Brown's approach will be in two parts.
In this section, 
we first explain how a Cayley graph $\Gamma(G,\XX)$ and the choice of a point $[\chi] \in S(G)$
give rise to a tree-like structure $T = T(G, \XX, \chi)$ and a canonical action of $G$ on $T$.
We next show by way of examples
that the action of $G$ on $T$ can help one in proving that $\Gamma(G,\XX)_\chi$  is connected.
In sections \ref{ssec:Characterization-Sigma1-action}
and \ref{ssec:Construction-morphisms-into-trees},
we go one step further
and prove that membership in $\Sigma^1$ can be characterized
in terms of certain actions on $\R$-trees or on similar tree like structures.
%
\subsubsection{The construction of $T(G, \XX, \chi)$}
\label{sssec:Tree-like-structures-associated-Cayley-graph-construction}
%
Let $G$ denote a group, $\eta \colon \XX \to G$ a generating system 
and $\chi \colon G \to \R$ a non-zero character.
Set $\Gamma_{\chi} = \Gamma(G, \XX)_{\chi}$.
The goal is to construct a tree-like structure $T$, equipped with a $G$-action, 
\emph{that is a line if, and only if,  $\Gamma_{\chi}$ is connected}.

We begin with 
\begin{definition}
\label{definition:(T-leq-lambda)}
Define the set $T$, the relation $\leq$ and the function $\lambda \colon T \to \R$ as follows:
\begin{align}
T &= \bigcup\nolimits_{r \in \im \chi} 
\{ \CC  \mid \CC \text{ connected component of } \Gamma(G, \XX)_\chi^{[r, \infty)}\},
\label{eq:Definition-T}\\
\leq &\hspace*{1.2mm}: \text{relation defined by \;} \CC \leq \CC' \Longleftrightarrow \CC \supseteq \CC',
\label{eq:Definition-relation} \\
\lambda&\hspace*{1.8mm}\colon \CC  \mapsto \inf \{\chi(h) \mid h \in \ver(\CC) \}.
\label{eq:Definition-lambda}
\end{align}
The triple $(T, \leq, \lambda)$ enjoys some characteristic properties 
that will be enunciated in Lemma \ref{lem:Properties-(T-leq-lambda)} below;
in view of these properties, 
the structure $(T, \leq, \lambda)$ will be called the \emph{measured tree} 
\footnote{See Remark \protect{\ref{remarks:Root-descending-from-c}}a 
for an explanation of this designation.}
associated to  $(G, \XX, \chi)$. 
It will be denoted by $T(G, \XX, \chi)$.
\end{definition}
\index{Notation!T-leq-lambda@$T(G,\leq,\lambda)$}%
\index{Definition of!T-leq-lambda@$T(G,\leq,\lambda)$}%
\index{Measured tree!associated to (G,XX,chi)@associated to $(G, \XX, \chi)$}
\index{Notation!T(G-XX-chi)@$T(G,\XX,\chi)$}

\begin{remark}
\label{remark:Definition-T(G-XX-chi)}
Every point of $T$ is a connected component of a subgraph 
\begin{equation}
\label{eq:Definition-subgraph-Gamma-sub-r}
\Gamma_r =\Gamma(G, \XX)_\chi^{[r, \infty)}.
\end{equation} 
Since $r $ is required to be the $\chi$-value of an element $h \in G$,
the component $\CC_1$, say,  containing $h$ has the property  
that  $\chi(h) = \inf\{ \chi(g_1) \mid g_1 \in \ver(\CC_1) \}$.
Actually, every component $\CC$ of $\Gamma_r$ contains an element $h_{\CC}$ with $\chi(h_{\CC}) = \chi(h)$,
the reason being that $\ker \chi$ acts transitively on the components of  $\Gamma_r$
(see Lemma \ref{lem:Transitive-action-on components}). 
In the sequel, 
it will  be convenient to call a component $\CC_h$ 
if $h $ is a vertex of the component on which the infimum is attained.
\end{remark}

The next result lists some basic properties of $T(G, \XX, \chi)$;
later on,
these properties will be taken as the axioms of an (abstract) measured $G$-tree.
\begin{lem}
\label{lem:Properties-(T-leq-lambda)}
For every group $G$, generating system $\eta \colon \XX \to G $ and non-zero character $\chi$,
 the associated measured tree $T(G, \XX, \chi)$ satisfies the properties
 \begin{enumerate}[(i)]
\item the relation $\leq$ is a partial order on $T$;
\item  the relation $\leq$ is \emph{directed};
more precisely, for every pair $a_{1}$, $a_{2}$ in $T$ 
there exists a point $a \in T$ with $a \leq a_{1}$ and $a \leq a_{2}$;
\item  for every $b \in T$ the set $(-\infty, b] = \{a \in T \mid a \leq b \} $ is linearly ordered 
and $\lambda$ maps it bijectively and in an order-preserving way onto the subset 
$(-\infty, \lambda(b)] \cap \im \lambda$ of the real line $\R$.
\end{enumerate}
\end{lem}
\index{Measured tree T-leq-lambda@Measured tree $T(-,\leq,\lambda)$!properties}%

\begin{proof}
Assertion (i) is clear.
To prove (ii), 
assume $a_1$, $a_2$ are  the components $\CC_{g_{1}}$, $\CC_{g_{2}}$, 
and choose a path $p$ that connects $g_{1}$ with $g_{2}$. 
If $h$ is a vertex with minimal $\chi$-value along $p$,
the component $a = \CC_{h}$ contains both vertices $g_{1}$ and $g_{2}$, 
hence the components $a_1$ and $a_2$,
and so the relations $a \leq a_1$ and $a \leq a_1$ hold.

We are left with claim (iii). 
Assume $b = \CC_{g}$ and
consider components  $a_1  = \CC_{h_{1}}$ and $a_2  = \CC_{h_{2}}$
which belong to the subset $(-\infty, b]$ of $T$.
Then $g$ is a vertex of both $\CC_{h_{1}}$ and $\CC_{h_{2}}$. 
If $\lambda(\CC_{h_{1}}) \leq \lambda(\CC_{h_{2}})$ this implies
that $\CC_{h_{1}}$  contains a path from $h_1$ to $h_2$,
hence the component $\CC_{h_{2}}$ and so  $a_1 \leq a_2$. 
If $\lambda(\CC_{h_{2}}) \leq \lambda(\CC_{h_{1}})$ one sees similarly that $a_2 \leq a_1$.
So the subset $(-\infty, b]$ of $T$ is linearly ordered.

It remains to show that the restriction $\lambda_* \colon (-\infty, b] \to [-\infty,\lambda(b) ] \cap \im \chi$
of the function $\lambda$ has the stated properties.
The definitions of the order relation on $T$ and of the function $\lambda$,
given in \eqref{eq:Definition-relation}  and \eqref{eq:Definition-lambda},
guarantee that $\lambda_*$ preserves the order relations 
defined on the set  $(-\infty, b] \subset T$ and on the interval $(-\infty, \lambda(b)]$.

The injectivity of $\lambda_*$ is a direct consequence of the above proof of claim (ii).
The definition of $\lambda_*$ and Remark \ref{remark:Definition-T(G-XX-chi)} imply next 
that the value $\lambda(a)$ of a point $a$ in the set $(\infty, b]$ lies in $\im \chi$.
So we are only left with proving that $\lambda$* maps the subset $(-\infty, b]$ surjectively onto the interval
$(-\infty, \lambda(b)\,] \cap \im \chi$.

Suppose $b = \CC_g$ and $r'$ is an element of $\im \chi$ with $r' \leq r = \lambda(\CC_g)$.
Then $r - r'$ is the $\chi$-value of some element $h_0 \in G_\chi$
Write $h_0$ as a word $w_0$ in $\XX \cup \XX^{-1}$ 
and reorder the letters of $w_0$ so as to obtain a word $w_1$ with non-negative $\chi$-track.
The path $(g, w_1^{-1})$ will then lead from $g$ to a element $h'$ with $\chi(h') = \chi(g) - (r-r') = r'$;
moreover, as the $\chi$-track of $w_1$ is non-negative,
this path runs in $\Gamma_{r'}$.
The connected component $a' = \CC_{h'}$ belongs therefore to the subset $(-\infty, b\,]$ of $T$ 
and $\lambda(a') = r'$, as desired.
\end{proof}

\begin{remarks}
\label{remarks:Root-descending-from-c}
a) The chosen designation of the structure $T(G, \XX, \chi)$ deserves some comment.
The set $T$, equipped with the order relation $\leq$, is a partially ordered set or a poset.
The fact that the subsets $(-\infty, b]$ sitting below a point $b$ are linearly ordered
turns the poset into a \emph{forest in the sense of order theory}.
As every pair of components $\CC_1$, $\CC_2$ in $T$ has a lower bound,
the forest is ``connected'' and hence a tree.
The function $\lambda$, finally allows one to measure the length of intervals 
$[a, b] = \{c \in T \mid a \leq c \leq b\}$, 
a fact that explains the choice of the adjective ``measured''.

b) Claim (iii) of the preceding lemma shows that every point $b \in T$ 
is the end point of a linearly ordered subset $(-\infty, b]$.
This subset is uniquely determined by $c$ 
and will be referred to as the \emph{ray descending from} $c$.
In view of (ii) two rays intersect in an infinite set;
this set, however, need not be a ray as it may not possess a largest element.

c) The concept of a measured tree and the analysis of its structure
--- to be given in sections 
\ref{sssec:Tree-like-structures-associated-Cayley-graph-construction}
and
\ref{sssec:Canonical-action-on-(G,XX,chi)} ---
go back to sections II.1 and II.2 in the monograph \cite{BiSt92},
except for its name: in \cite{BiSt92} it is called \emph{measured rooted pre-tree} or \emph{romp-tree}, for short.
The new name arouse out of a discussion of R. Strebel with Reinhard Diestel.
\end{remarks}
\index{Ray descending from a point!definition}%
\index{Diestel, R.}

So far we know that the measured tree $T(G, \XX, \chi)$ contains half-lines,
namely the rays descending from its elements.
As we shall see in the next section,
it also contains \emph{lines} in the sense of the following:
\begin{definition}
\label{definition:Line-in-a-pre-tree}
A (non-empty) subset $L \subseteq T$ is called a \emph{line} 
if it is linearly ordered, contains the ray of each of its elements 
and if  $ \sup \{\lambda(c) \mid c \in L\}= \infty$.
\end{definition}
\index{Definition of!line in a measured tree}%
\index{Line in a measured tree!definition}%

A ray of  $T(G, \XX, \chi)$  is determined by its largest element;
a line by the choice of a linearly ordered sequence $c_{0} < c_1 < c_2 < \cdots$
with $\lim_{n \to \infty} \lambda(c_n) = + \infty$.
The line defined by such a sequence sequence 
is the union  $L = \bigcup \{(-\infty, c_{j}] \mid j \in \N \}$.

\begin{remark}
\label{remark:Visualizing-measured-tree}
In general, it is difficult to get a good intuitive picture of a measured tree $T(G, \XX, \chi)$,
unless, of course, it is a line. If $\chi$ has rank 1, the situation improves:
then $T(G, \XX, \chi)$ is, in essence, a combinatorial tree (or $\Z$-tree) 
and can be described algebraically rather easily (see Claim (i) on p.\,157).
\end{remark}

\subsubsection{Canonical action on $T(G, \XX, \chi)$}
\label{sssec:Canonical-action-on-(G,XX,chi)}
%
The group $G$ acts on  $T =T(G, \XX, \chi)$;
indeed, each component $\CC_{h}  \in T$ is a subgraph of the Cayley-graph  $\Gamma(G, \XX)$
and the action of $G$ on $\Gamma(G, \XX)$ induces an action of $G$ on its set of subgraphs,
in particular, on the subset of the components that are the elements of $T$.
The following result amounts to say
that $G$ acts on $T(G, \XX, \chi)$ by automorphisms.
\begin{lem}
\label{lem:Properties-action}
Let $T = T(G,\XX, \chi)$  be the measured tree associated to $G$, $\XX$, $\chi$,
 and let $\mu \colon G \times T \to T$ denote the canonical action of $G$ on $T$.
Then $G$ acts transitively on $T$ and the following formulae hold for all $g$, $h \in G$ and all $a$, $b \in T$: 
\begin{align}
g.\CC_{h}& = \CC_{gh}, \label{eq:Formula-action-T}\\
a \leq b &\Rightarrow g.a \leq g.b, 
\label{eq:Compatibility-action-relation}\\
\lambda (g. a) &= \chi(g) +  \lambda(a).
\label{eq:Equivariance-lambda}
\end{align}
\end{lem}
\index{Measured tree T-leq-lambda@Measured tree $T(-,\leq,\lambda)$!canonical action}%

\begin{proof}
The transitivity of the action and
formulae \eqref{eq:Formula-action-T}
and \eqref{eq:Compatibility-action-relation}  are immediate consequences 
of the definitions (see formula \eqref{eq:Definition-relation}),
while formula \eqref{eq:Equivariance-lambda}
 follows from the definition of $\lambda$ and the computation
\[
\lambda (g. \CC_{h}) = \inf \{ \chi(gh_1) \mid h_1 \in V(\CC_{h} ) \} = \chi(g) + \lambda(\CC_{h}).
\]
In the above calculation, 
one uses that $g. \CC_{h} $ is the image of $\CC_{h}$ 
under the left translation $h \mapsto gh$.
\end{proof}

The action of $G$ on $T$ allows one to assign to each $g \in G$  a subset $A_g$:
\begin{definition}
\label{definition:A-sub-g}
For every $g \in G$, the subset
\begin{equation}
\label{eq:Definition-A-sub-g}
A_{g} = \{a \in T \mid a \leq g.a \text{ or } g.a \leq a \}.
\end{equation}
is called the \emph{characteristic subset} of $g$.
\end{definition}
\index{Definition of!characteristic subset}%
\index{Notation!A-sub-g@$A_g$}%

The basic properties of these  characteristic subsets are described in
\begin{lem}
\label{lem:Properties-A-sub-g}
For every $g \in G$ and measured tree $T = T(G, \XX, \chi)$ 
the subset $A_{g}$ has the following properties:
\begin{enumerate} [(i)]
\item $A_{g}$ is not empty and contains with each $c \in A_{g}$
the ray descending from $c$.

\item If $\chi(g) = 0$ then  $A_g $ coincides with the set  $T^g = \{a \in T \mid g.a = a \}$ of points fixed by $g$.

\item If   $\chi(g) \neq 0$  then $A_g$ is a line.
It is the unique $\gp(g)$-invariant line of $T$ 
and, for every point $a \in A_g$,
the line  $A_g$ coincides with the union of the rays descending from the points of the orbit $\gp(g).a$.
\end{enumerate}
\end{lem}

\begin{proof}
(i) Choose a point $b \in T$ and consider the intersection  $I = (-\infty, b] \cap (-\infty, g.b]$.
It is non-empty (the rays intersect because $T$ is directed; 
see claim (ii) of Lemma \ref{lem:Properties-(T-leq-lambda)}).
Let $x$ be a point in  $I$. Then $x \leq b$ and $x \leq g.b$ and so $g.x \leq g.b$; 
so  $x$ and $g.x$ lie both on the ray $(-\infty, g.b]$.
As this ray which linearly ordered (by part (iii) of Lemma \ref{lem:Properties-(T-leq-lambda)})
either $x \leq g.x$ or $g.x \leq x $ must hold 
and so $x \in A_{g}$ by Definition \ref{definition:A-sub-g}. 
We next show that $A_g$ contains, with every point $c \in T$, the ray descending from $c$.
In this verification, we may and shall assume that $\chi(g) \leq 0$, for $A_g = A_{g^{-1}}$.
Fix a point $c \in A_g$. Then $g.c \leq c$.
If $a \leq c$ then $g.a \leq g.c$ by property \eqref{eq:Compatibility-action-relation}
and so $g.a \leq c$ by the transitivity of the relation $\leq$.
So $a$ and $g.a$ lie both on the  linearly ordered ray descending from $c$.
Since $\chi(g) \leq 0$, this shows that $g.a \leq a$ whence $a \in T$.
\smallskip

(ii) Clearly  $ T^g \subseteq A_g$. 
Conversely, if $a  \in A_g$ then $a \le ga$ or $ga \le a$.
The assumption that $\chi(g) = 0$ 
and statement (iii) in Lemma \ref{lem:Properties-(T-leq-lambda)}
then imply that $\lambda(a) = \lambda(g.a)$,  whence $a \in T^g$.
\smallskip

(iii) By claim (i) the subset $A_g$ is not empty; let $a$ be one of its points.
Then $a < g.a$ or $g.a < a$; in either case,
$\{(g^m).a \mid m \in \Z \}$ is a totally ordered set that is not bounded from above 
and so the union $L_a =\bigcup  \{ ( -\infty, g^m a]  \mid m \in \Z  \}$ is a line;
by claim (i), this line is contained in $A_{g}$.
It follows that $A_{g}$ contains all lines generated by the $\gp(g)$-orbits of points in $A_{g}$.

Let now $a_1$ and $a_2$ be two points  of $A_g$ 
and let $L_1$ and $L_2$ be the lines generated by the orbits $\gp(g).a_1$ and $\gp(g).a_2$, respectively.
By claim (ii) in Lemma \ref{lem:Properties-(T-leq-lambda)}
there exists $a \in T$ with $a \leq a_1$ and $a \leq a_2$.
It then follows that $L_1 \cap L_2$ contains the line generated by $\gp(g).a$.
Hence $L_1 \cap L_2$ is a line and so $L_1 = L_2$.
This implies, in particular, that $A_g$ is a line.

Consider, finally, a $\gp(g)$-invariant line $L$.
If $a \in L$ then $g.a \in L$ by the $gp(g)$-invariance of $L$;
as $L$ is totally ordered, $a$ and $g.a$ are comparable and so $a \in A_g$ by the very definition of $A_g$.
Thus $L \subseteq A_g$.
\end{proof}

\begin{definition}
\label{definition:Elliptic-hyperbolic-elements}
Let  $G$ be a group, $\eta \colon \XX \to G$ a generating system, 
$\chi \colon G \to \R$ a non-zero character
and $T = T(G;\XX, \chi)$ the associated measured tree.
An element $g \in G$ will be called \emph{elliptic} if it lies in the kernel of $\chi$,
and otherwise \emph{hyperbolic}.
By the previous lemma,
the characteristic subset $A_{g}$ of a hyperbolic element is a line;
it will be referred as  the \emph{axis} of $g$.
\end{definition}
\index{Elliptic elements!definition}%
\index{Hyperbolic elements!definition}%
\index{Definition of!elliptic element}%
\index{Definition of!hyperbolic element}%
\index{Axis of a hyperbolic element!definition}%
\index{Definition of!axis of a hyperbolic element}%

%
\subsubsection{A first application}
\label{sssec:Application-canonical-action-on-(G,XX,chi)}
%
In sections
\ref{sssec:Tree-like-structures-associated-Cayley-graph-construction}
and
\ref{sssec:Canonical-action-on-(G,XX,chi)},
we saw that every triple $(G, \eta \colon \XX \to G, \chi)$ leads to a measured tree $T(G, \XX, \chi)$
that carries a canonical $G$-action.
In this section,
we illustrate 
how the properties of $T(G, \XX, \chi)$ 
--- 
stated in lemmata
\ref{lem:Properties-(T-leq-lambda)},
\ref{lem:Properties-action} and
\ref{lem:Properties-A-sub-g} ---
lead to a simple proof of an attractive result established earlier by different techniques.
The result deals with a group generated by a set $\XX$ having the property 
that many pairs of elements $x_1$, $x_2$ in $\XX$ commute.

We begin with a simple remark on the characteristic sets of commuting elements.
 \begin{lem}
 \label{lem:Axes-commuting-elements}
 Let $\eta \colon \XX \to G$ be a generating system 
and $\chi \colon G \to \R$ a non-zero character of $G$.
Consider the measured $G$-tree $T = T(G, \XX, \chi)$.

Assume $g$ and $h$ are elements of $G$ which commute and $h$ is hyperbolic.
Then the axis $A_h$  of $h$ is invariant under $g$.
If, in addition, $g$ is hyperbolic then $A_g = A_h$.
\end{lem}

\begin{proof}
The definition \eqref{eq:Definition-A-sub-g} of characteristic subsets implies 
that $A_{h} = A_{h^{-1}}$.
We can, and shall, therefore assume that $\chi(h) > 0$.
The calculation
\[
g.A_{h} = g.\{a \in T \mid a \leq h.a \} = \{g.a \in T \mid g.a \leq (ghg^{-1}).(g.a) \} = A_{ghg^{-1}}
\]
then shows that  $g. A_{h}$  is the axis of  $ghg^{-1}$.
If $g$ and $h$ commute, the axis $A_{h}$ is  therefore invariant under $g$ and so it is a $\gp(g)$-invariant line. 
Moreover, if $g$ is hyperbolic,
part (iii) Lemma \ref{lem:Properties-A-sub-g} shows that $A_h$ is the axis of $g$.
\end{proof}

Now to the announced application.
 \begin{example}
 \label{example:Right-angled-Artin-groups}
 \index{Computation of Sigma1@Computation of $\Sigma^1$ for!right angled Artin groups}
 Let $G$ be a graph group, alias right angled Artin group, given by the (finite) combinatorial graph $\Delta$,
 and let $\chi \colon G \to \R$ be a non-zero character.
 Let $\LL(\chi)$ denote the living subgraph of $\chi$;
 by definition, this is the subgraph spanned by the set $\{ x \in \ver(\Delta) \mid \chi(x) \neq 0 \}$.
 According to Theorem \ref{thm:Sigma1-right-angled-Artin-group},
 the character $\chi$ represents a point of $\Sigma^1(G)$ if, and only if,
 $\LL(\chi)$ \emph{is connected and dominates} $\Delta$.
 In the proof, 
 this condition is shown to be necessary by the construction of a suitable quotient group of $G$;
 the sufficiency is deduced from a result on the invariant of a join of groups
 (namely Proposition \ref{prp:Sigma-1-join-subgroups}).
 
The properties of the measured tree $T(G, \XX, \chi)$ enable one to give a neater proof of the sufficiency.
 Pick  a vertex  $x_{1}$ be in $\LL(\chi)$ and let $A = A_{x_{1}}$ be the axis of the corresponding generator. 
 Consider another vertex $x$ of $\LL(\chi)$;
 since $\LL(\chi)$ is assumed to be connected, 
 there exists a path $(x_{1}, x_{2}, \ldots, x_{k} = x)$ 
 in $\LL(\chi)$ from $x_{1}$ to $x$. 
 The generators $x_{1}$ and $x_{2}$ commute;
 by Lemma \ref{lem:Axes-commuting-elements}
 the axis $A$ is invariant under $x_{2}$;
 in addition, it is the axis of the generator $x_2$.
 Continuing in the indicated manner, 
 one sees that $A$ is a line which is invariant under every generator corresponding to a vertex of $\LL(\chi)$.
Now  $\LL(\chi)$ is also assumed to dominate $\Delta$. 
So each generator outside $\ver(\LL_{\chi})$ commutes  with a generator in $\ver(\LL_{\chi})$, 
whence the line $A$ is invariant under all generators of $G$ and so under $G$.
As $G$ acts transitively on  $T(G, \XX, \chi)$ this implies, in particular, 
that  $\Gamma(G, \ver(\Delta))_{\chi}$ is connected.
\end{example}

\begin{remark}
\label{remark:Quotients-of-graph-groups}
The above proof applies also to quotient groups of graph groups 
and it allows of some modifications.
A striking modification is used in the proof of the main result of L.\,A.\,Orlandi-Korner's  article \cite{OrKo00}.
\end{remark}
%

%
\subsection{Characterization of $\Sigma^1$ in terms of actions on trees} 
\label{ssec:Characterization-Sigma1-action}
%
In this section, $\Sigma^1$ will be described in terms of actions on trees.
We first introduce a generalization of the measured tree $T(G,\XX, \chi)$,
discuss the properties of this new structure
and then state a characterization of $\Sigma^1$, valid for arbitrary groups.
The easy parts of the proof will be given in this section \ref{ssec:Characterization-Sigma1-action}, 
the more involved parts in section \ref{ssec:Construction-morphisms-into-trees}.
%
\subsubsection{Notion of a measured tree}
\label{sssec:Notion-rooted-pre-tree}
%
In order to discuss quotients of the measured tree $T(G,\XX, \chi)$,
we introduce the concept of a measured tree,
taking the properties of $T(G,\XX,\chi)$ enunciated in Lemma
\ref{lem:Properties-(T-leq-lambda)} as axioms of the new structure $(T,\leq,\lambda)$.
There is a natural notion of morphism for these structures;
moreover,
the action of a group $G$ by automorphisms  on an measured tree $(T,\leq,\lambda)$
gives rise to a canonical character  $\chi_{T} \colon G \to \R$.
\begin{definition}
\label{definition:Rooted-pre-tree}
A \emph{measured tree}  is a triple $(T, \leq, \lambda)$,
consisting of  a non-empty set $T$, an order  relation  $\leq$  on $T$ and a function  $\lambda \colon T \to \R$,
which satisfies the following axioms:
\begin{enumerate}[(i)]
\item $(T, \leq)$ is a tree in the sense of \emph{order theory};
so $(T, \leq)$ is a poset with the additional properties 
that every finite subset of $T$ has a lower bound 
and that each subset $(-\infty, a] = \{b \in T \mid b \leq a \} $ is linearly ordered;
\item $\lambda$ maps each subset of the form $(-\infty, a]$  injectively and in an order-preserving fashion onto a subset of $\R$
that is not bounded from below.
\end{enumerate}
\end{definition}
\index{Measured tree!definition}%
\index{Definition of!measured tree}%

\begin{remarks}
\label{remarks:Root-descending-from-a}
a) As in the case of the measured tree $T(G, \XX, \chi)$, 
the subset $(-\infty, a]$ occurring in axiom (i) will be called \emph{ray descending from} $a$.

b) The addendum to axiom (ii),
namely that $\lambda((-\infty, a])$ be not bounded from below,
is automatically satisfied  in all the examples that are of interest in the sequel.
This property has therefore been required from the outset, 
 although some results remain valid without it.
\end{remarks}
\index{Ray descending from a point!definition}%
\index{Definition of!ray descending from a point}%

The measured tree $T(G, \XX, \chi)$, associated to a triple $(G, \XX, \chi)$ is, of course,
an example of a measured tree;
other examples are provided by $\R$-trees with a distinguished end 
(see example \ref{examples:rooted-pre-trees}c).
Prior to turning to them, 
we introduce the notion of morphism of measured trees.

\begin{definition}
\label{definition:Morphism-rooted-pre-trees}
Let $(T, \leq,\lambda)$ and $(T,' \leq',\lambda')$ be measured trees.
A \emph{morphism} from the first structure to the second one is a function $\varphi \colon T \to T'$ 
that is order-preserving  
and for which there exists a real number $r$ such that the equation 
$
\lambda'(\varphi(a)) = \lambda(a) + r
$
holds for every point $a \in T$.
\end{definition}
\index{Measured trees!morphisms}%
\index{Definition of!morphism of measured trees}%
\index{Morphisms of measured trees!definition}%

In the sequel, three types of morphisms will be important:
\emph{automorphisms} induced by groups acting on measured trees,
\emph{inclusions} of subtrees and  \emph{epimorphisms} of measured trees.
We begin with group actions.
\begin{prp}
\label{prp:Character-associated-to-G-romp-tree}
Let $G$ be a group which acts by morphisms on a measured tree $(T, \leq, \lambda)$.
Then the function $g \mapsto \lambda (g.a) - \lambda (a)$ does not depend on  the point $a \in T$ 
and yields a homomorphism $\chi_{T} \colon G \to \R$.
\end{prp}

\begin{proof}
For every $g \in G$ the function $x \mapsto g.x$ is a morphism of $(T, \leq, \lambda)$;
by definition \ref{definition:Morphism-rooted-pre-trees}
there exists therefore a number $r$ 
such that the equation $\lambda (g.x) = \lambda(x) + r$ is valid for every $x \in T$.
It follows that the function $\chi_{T}(g) = \lambda(g.a) - \lambda(a)$  does not depend on the choice of $a$.
Assume now that $g$ and $h$ are elements of $G$, 
pick $a \in T$ and set $b = h.a$.
The computation
\begin{align*}
\chi_{T}(gh) &= \lambda(gh.a) - \lambda(a) \\
&= (\lambda (g.b)) -\lambda(b) )+( \lambda(h.a) - \lambda(a) )
= \chi_{T}(g) + \chi_{T}(h)
\end{align*}
then shows that $\chi_{T}$ is a homomorphism of $G$ into the additive group of $\R$.
\end{proof}

The homomorphism $\chi_{T}$ provided by the above proposition 
will be referred to as the \emph{character associated to the action of} $G$ 
on the measured tree $T$.
\index{Character associated to an action!on a measured tree}

\begin{examples}
\label{examples:rooted-pre-trees}
\index{Measured trees!examples|(}
\index{Morphisms of measured trees!examples|(}
a) Consider the triple $(T = \R, \leq, \lambda =\id)$ 
where $\leq$ denotes the usual order relation on the field of real numbers.
This triple is a measured tree.
 A  function of  $\R$ into itself  is an endomorphism $\varphi$ of the triple
 if it is increasing and if there there exists a real number $r$  with   $\varphi(x)  = x + r$ for $x \in \R$;
 this additional property says that $\varphi$ is a translation   $\tau$ with amplitude $r$.
Suppose now $G$ is a group which acts on $(\R, \leq, \id)$ by automorphism.
Then $G$ acts by translations and the character $\chi_{\R} $ associated to this action is the function 
which associates to each $g \in G$ the amplitude $r_{g}$ of the  translation $\tau_g \colon \R \iso \R$.
\smallskip

b)
Let $T(G,\XX, \chi)$ be the measured tree associated to a group $G$, 
a generating system $\eta \colon \XX \to G$ and a non-zero character $\chi$.
By Lemma \ref{lem:Properties-(T-leq-lambda)},
it is a measured tree in the sense of definition
\ref{definition:Rooted-pre-tree}.
The group $G$ acts on $T(G,\XX,\chi)$ in a canonical way;
in view of lemma \ref{lem:Properties-action}
 it acts  by automorphisms.
 Equation \eqref{eq:Equivariance-lambda} then shows that the character $\chi_{T}$,
 associated to the action by $G$, is nothing but $\chi$.
 \smallskip

 c) Let $(T,d)$ be an $\R$-tree.
 So $(T,d)$ is a metric space satisfying the following two requirements:
 \footnote{See, \eg, \cite[p. 276, sections  2.3 and  2.5]{Sha87} 
 or \cite{Chi01}, Definition on page 29 and remark at the top of  page 30.}
 \begin{itemize}
 \item For every couple  $(a, b)$ of points in $T$, 
 there exists a unique geodesic arc $\gamma_{a,b} \colon  [0, d(a,b)] \to T$ 
 that starts in $a$ and ends in $b$;
 let $[a,b]$ denote the image of $\gamma$; 
 \item if the geodesic segments $[a,b]$  and $[b,c]$ have only the end point $b$ in common, 
 then $[a,b] \cup [b,c]$ is the geodesic segment from $a$ to $c$.
 \end{itemize} 
 Assume now that  $(T,d)$ admits a geodesic ray $(-\infty,b_{*}]$ of infinite length
 and let $e$ be the end of $T$ represented by this ray.
 For any point $a \in T$ there exists then a unique geodesic ray $(-\infty, a]$ with endpoint $a$
 that represents $e$, i.\;e. which intersects $(-\infty,b_{*}]$ in a subray.
 The geodesic rays representing $e$ enable one to define a relation on $T$,
 defined by
 \begin{equation}
 \label{eq:Definition-order-relation-rooted-R-tree} 
 a_1 \leq a_2 \Longleftrightarrow (-\infty,a_1] \subseteq (-\infty, a_2].
 \end{equation}
 This relation is an order relation; 
 in addition, it is directed.
 
The distance function $d$, on the other hand,  
allows one to define a function $\lambda \colon T \to \R$ as follows:
 one chooses a base point $b_*\in T$ and sets
 \begin{equation}
 \label{eq:Definition-lambda-R-tree}
 \lambda(x)   = d(c, x) - d(c, b_*) \quad \text{ with  } c \leq x \text{ and } c \leq b_*.
 \end{equation}
 (The properties of an $\R$-tree imply that $\lambda(x)$ does not depend on the choice of $c$.)
 The function  $\lambda$ maps each ray $(-\infty,a]$ 
 bijectively and in a fashion that respects the order relations
 onto the interval $(-\infty, \lambda(a)]$ of the real line $\R$.
 \smallskip
 
So far we have verified that the triple $(T, \leq, \lambda)$,
associated to the $\R$-tree $(T,d)$ and its end $e$, 
 is a measured tree in the sense of definition \ref{definition:Rooted-pre-tree}. 
 Suppose now that a group $G$  acts on $(T,d)$ by isometries which fix the end $e$;
 this assumption amounts to say 
 that $G$ permutes the rays descending from the points of $T$.
 We claim that $G$ acts on $(T, \leq, \lambda)$ by morphisms.
Two properties need to be checked:
that $G$ acts by automorphisms of the poset $(T, \leq)$
and that the function  $\varphi_{g} \colon x \mapsto \lambda(g.x) - \lambda(x)$ 
is constant for each $g\in G$.
The first property follows immediately from the definition 
\eqref{eq:Definition-order-relation-rooted-R-tree} 
 of the order relation $\leq$ and from the fact 
 that $G$ permutes the rays descending from the points of $T$.
 
 The verification of the second property needs a short computation. 
 Let $b_*$ denote the selected base point (occurring in definition \eqref{eq:Definition-lambda-R-tree})
 and let $x$ be an arbitrary point.
 Pick a point $c$ in the intersection of the rays descending from the points $x$, $g.x$ and $b_*$.
 Then $c \leq x$, whence $g.c \leq g.x $;
 thus $c$ and  $g.c$ lie both on the ray descending from $g.x$.
 Since rays are linearly ordered, either $c \leq g.c$, or $g.c < c$.
 If $c \leq g.c$,
 then $d(c, g.x)= d(c, g.c) + d(g.c, g.x) = d(c, g.c) + d(c, x)$
 and  so
 \begin{align*}
 \lambda(g.x) - \lambda(x) &= (d(c,g.x) - d(c,b_*)) - (d(c,x) - d(c,b_*)) \\ 
 &= d(c, g.x) - d(c,x) = d(c,g.c).
 \end{align*}
If $g.c < c$, 
one finds similarly that $ \lambda(g.x) - \lambda(x) = -d(c,g.c)$.
The claim now follows from the fact that the number $\lambda(x)$,
defined by equation \eqref{eq:Definition-lambda-R-tree}  with the help of $b_*$ and $c$,
 does not depend on $c$.
 \footnote{The distance $d(c, g.c)$ is the translation length of $g \in G$.}
 \end{examples}
 \index{Measured trees!examples|)}%
 \index{Morphisms of measured trees!examples|)}
 %
\subsubsection{Lines and axes}
\label{sssec:Lines-axes-romp-trees}
The characterization of $\Sigma^1$ in terms of actions on measured trees 
uses the concept of a line;
in the applications, lines often arise as axes of hyperbolic elements.
It is therefore essential that the notion of a line can be defined for all measured trees
and that it enjoys the familiar properties.

Definition \ref{definition:Line-in-a-pre-tree} 
makes sense for all measured trees;
we use it to define lines in the more general set-up: 
 \begin{definition}
\label{definition:Line-in-rooted-pre-tree}
Let $(T, \leq, \lambda)$ be a measured tree.
A subset $L \subseteq T$ is called a \emph{line} 
if it is linearly ordered, contains the ray of each of its elements 
and if the set $\{\lambda(c) \mid c \in L\}$ is not bounded from above.
\end{definition}
\index{Definition of!line in measured tree}%
\index{Line in a measured tree!definition}%

The lines that play a crucial rôle in the sequel are the axes of hyperbolic elements.
Their definition relies on that of a characteristic subset.
This definition, in turn, can be stated as in definition \ref{definition:A-sub-g}:
 \begin{definition}
\label{definition:A-sub-g-general-case}
Suppose $G$ acts  on a measured tree $(T, \leq, \lambda)$ by automorphisms.
For every $g \in G$, the subset
\begin{equation}
\label{eq:Definition-A-sub-g-general-case}
A_{g} = \{a \in T \mid a \leq g.a \text{ or } g.a \leq a \}.
\end{equation}
is then called the \emph{characteristic subset} of $g$.
\end{definition}
\index{Definition of!characteristic subset}
\index{Characteristic subset!properties}

The salient features of a characteristic subset are collected in the next lemma.
It is the analogue of Lemma \ref{lem:Properties-A-sub-g} and it can be proved as its model,
except for the fact that references to the properties of $T(G,\XX,\chi)$
have to be replaced  by references to the axioms in Definition \ref{definition:Rooted-pre-tree}.
\begin{lem}
\label{lem:Properties-A-sub-g-general-case}
Suppose $G$ acts by automorphism on a measured tree $(T, \leq, \lambda)$.
Then, for every $g \in G$, the characteristic subset $A_{g}$ enjoys the following properties:
\begin{enumerate} [(i)]
\item $A_{g}$ is not empty and contains with each $c \in A_{g}$
the ray descending from $c$.

\item If $\chi(g) = 0$ then  $A_g $ coincides with the set  $T^g = \{a \in T \mid g.a = a \}$ of points fixed by $g$.

\item If $\chi(g) \neq 0$  then $A_g$ is a line.
It is the unique $\gp(g)$-invariant line of $T$ 
and, for every point $a \in A_g$,
the line $A_g$ coincides with the union of the rays descending  from the points of the orbit $\gp(g).a$.
\end{enumerate}
\end{lem}
\index{Characteristic subset!properties}
\begin{remark}
\label{remark:Elliptic-hyperbolic-elements-2}
An element $g \in G$ will be called \emph{elliptic} if it belongs to the kernel of $\chi$,
and otherwise \emph{hyperbolic}.
By the previous lemma,
the characteristic subset $A_{g}$ of a hyperbolic element is a line;
it will be referred as  the \emph{axis} of $g$.
\end{remark}
\index{Elliptic elements!definition}%
\index{Hyperbolic elements!definition}%
\index{Definition of!elliptic element}%
\index{Definition of!hyperbolic element}%
\index{Axis of a hyperbolic element!definition}%
\index{Definition of!axis of a hyperbolic element}%
%
\subsubsection{Characterizations of $\Sigma^1$}
\label{sssec:Characterization-action-rooted-pre-trees}
%
We now come to the main result of section \ref{ssec:Characterization-Sigma1-action},
a characterization of $\Sigma^1$ 
in terms of actions on measured trees
and a second one in terms of actions on $\R$-trees.
\begin{thm}
\label{thm:Characterization-Sigma1-via-actions-on-trees}
For every group $G$ and every non-zero character $\chi \colon G \to \R$
the following statements are equivalent:
\begin{enumerate} [(i)]
\item $[\chi] \in \Sigma^1(G)$;
\item for every generating system $\eta \colon \XX \to G$
 the measured $G$-tree $T(G, \XX, \chi)$ is a line;
\item every measured $G$-tree 
with  associated character $\chi_{T} = \chi$ contains a $G$-invariant line;
\item every $(G, \R)$-tree $(T,d)$ which satisfies the restrictions
\begin{enumerate} [a)]
\item $G$ fixes an end $e$ of $(T,d)$,
\item the character $\chi_{T}$ associated to this end $e$ coincides with $\chi$,
\end{enumerate}
 contains a $G$-invariant line.
\end{enumerate}
\end{thm}
\index{Brown's characterization of Sigma1@Brown's characterization of $\Sigma^1$}%

\begin{proof}
By its very construction, the measured tree $T(G, \XX, \chi)$ is a line if, and only if, 
the subgraph $\Gamma(G, \XX)_{\chi}$ is connected. 
Indeed, $T(G, \XX, \chi)$ is a line if, and only if,  all the subgraphs $\Gamma(G,\XX)_{\chi}^{[\chi(g),\infty)}$ 
are connected. These various subgraphs are permuted by $G$ and hence isomorphic.
Biimplication $(i) \Leftrightarrow (ii)$ is thus valid.

Consider now biimplication $(ii) \Leftrightarrow (iii)$.
For every generating system $\eta \colon \XX \to G$,
 the measured tree $T(G, \XX, \chi)$ is a measured tree in the sense of Definition 
 \ref{definition:Rooted-pre-tree}
 (see example \ref{examples:rooted-pre-trees}b). 
 If statement (iii) holds, $T(G, \XX, \chi)$ will therefore contain a $G$-invariant line. 
 As $G$ acts transitively on $T(G, \XX, \chi)$, 
 this means that $T(G, \XX, \chi)$ is a line. So  (iii) implies  (ii).
 The proof of the converse is harder. 
 It relies on the fact that every measured tree $(T',\leq',\lambda')$
 gives rise to a generating system $\eta \colon \XX \to G$  
 that permits one to construct a $G$-equivariant morphism 
 $\varphi \colon T(G, \XX,\chi) \to (T', \leq', \lambda')$.
 If $T(G, \XX,\chi)$ is a line,
its image under this morphism $\varphi$ will be a $G$-invariant line of $(T',\leq',\lambda')$.
Details are given in the proof of Proposition 
\ref{prp:Constructing-morphism-from-T(G,XX,chi)}   
below.

We are left with justifying the implications  $(iii) \Rightarrow (iv)$ and $(iv) \Rightarrow (ii)$.
Let $(T,d)$ be a $(G,\R)$-tree that satisfies conditions a), b) enunciated in statement (iv).
Then the metric space $(T,d)$ gives rise to a measured tree $(T, \leq, \lambda)$
on which $G$ acts by automorphisms (see example \ref{examples:rooted-pre-trees}c for details).
As every line in the measured tree $(T, \leq, \lambda)$ is a geodesic line in the $\R$-tree $(T,d) $
(see Definition \ref{eq:Definition-order-relation-rooted-R-tree} of the order relation $\leq$) 
and as the action of $G$ on the set $T$ is the same in both structures,
statement (iv) is therefore a consequence of statement (iii).
    
The proof of implication $(iv) \Rightarrow (ii)$ employs another construction.
Assume statement (ii) does \emph{not} hold. 
Then there exists a generating system $\eta \colon \XX \to G$  
so that the associated measured tree $T(G,\XX,\chi)$ is \emph{not} a line.
The aim is now to fabricate, with the help of $T(G,\XX,\chi)$, a $(G, \R)$-tree $(X,d)$,
and then to verify that this tree does not contain a $G$-invariant line.
The details of this construction will be given in section
\ref{sssec:Construction-morphism-real-tree}.
\end{proof}
\begin{remarks}
\label{remarks:Browns-characterization-Sigma1}
a) 
For \emph{finitely} generated groups $G$,
biimplication $(i) \Leftrightarrow (iv)$ is due to Ken Brown,
except for the fact that Brown uses the invariant $\Sigma_{G'}(G)$ discussed in \cite{BNS}
(hence the restriction to finitely generated groups),
and right actions (see \cite[p.\,499, Corollary 7.4]{Bro87b}).
Brown's proof involves a characterization of $\Sigma^1(G)$ 
in terms of so-called \emph{HNN valuations} and divides into two parts:
in the first one,
he establishes that
a point $[\chi]$ lies outside $\Sigma^1(G)$ if, and only if, 
there is a non-trivial HNN valuation on $G$ with associated character $\chi$ (see Theorem 5.2).
In the proof of Theorem 7.1 Brown then constructs,
given a non-trivial HNN valuation of a group $G$ with associated character $\chi$,  a $(G,\R)$-tree 
which has a fixed end and associated character $\chi$, 
but contains no $G$-invariant line.
This construction is based on the Alperin-Moss generalization of Chiswell's Theorem 1 in \cite{Chi76}
(see, \eg, \cite[p.\,303, Theorem 3.9]{AlBa87}).

Both Chiswell and Alperin-Moss start out with a Lyndon length function defined on a group $G$.
This function is intended to measure the distance $d(x_0, g.x_0)$ in the $(G,\R)$-tree $(T,d)$
between the base point $x_0$ and its image $g.x_0$ under the action of $g \in G$,
once the tree $(T,d)$ and the action of $G$ on it have been found.
In Brown's situation the $G$-action on the tree $(T,d)$ is of a special sort: 
it fixes an end and $G$ acts by translations along the end.

The question thus arises whether a suitable $(G,\R)$-tree 
can be obtained in a more direct manner.
This question has been answered in the affirmative, first by Robert Bieri in the late 1980s 
(cf. \cite[Chapt.\,II, proof of Thm.\,4.1]{BiSt92}) and later by Gilbert Levitt in \cite{Lev94} (see Proposition 3.3).
In section \ref{sssec:Construction-morphism-real-tree}, 
I shall give a construction using bits of these solutions.
\smallskip

b) Theorem \ref{thm:Characterization-Sigma1-via-actions-on-trees} has many useful applications,
in particular to joins of subgroups (see section \ref{sssec:Join-subgroups-2}),
groups with a locally nilpotent normal subgroup (see section \ref{sssec:Subnormal-subgroups})
and to direct products of groups (see section \ref{sssec:Illustration-invariant-direct-product}).
\end{remarks}
%

\subsection{Construction of morphisms} 
\label{ssec:Construction-morphisms-into-trees}
%
In this section,
two auxiliary results will be established
that have been used in the proof of Theorem \ref{thm:Characterization-Sigma1-via-actions-on-trees}.
In the first of them, 
one is given a measured  $G$-tree $(T', \leq', \lambda)$ 
whose associated character $\chi \colon G \to \R$ is non-zero.
One then shows 
that there exists a set of generators $\XX$ 
and a $G$-equivariant morphism $\varphi \colon T(G,\XX,\chi) \to (T',  \leq', \lambda')$.
The second result starts out with a measured $G$-tree $T(G, \XX, \chi)$
whose character $\chi \colon G \to \R$ is non-zero,
and constructs a $(G,\R)$-tree $(X, d_X)$ with associated character $\chi$,
all in such a way that $(X, d_X)$ has no $G$-invariant line
if the measured $G$-tree $T(G, \XX, \chi)$ is not a line.
%
\subsubsection{Morphism into a measured $G$-tree}
\label{sssec:Construction-morphism-rooted-pre-tree}
In the proof of Proposition \ref{prp:Constructing-morphism-from-T(G,XX,chi)}
the following lemma will be used:
\begin{lem}
\label{lem:Constructing-morphism-from-transitive-rooted-pre-tree} 
Suppose $(T,\leq, \lambda)$ and $(T',\leq', \lambda')$ are measured $G$-trees
and $G$ acts  transitively on $T$.
 Then a  $G$-equivariant morphism $\varphi \colon (T,\leq, \lambda) \to (T',\leq', \lambda')$ exists
if, and only if,
one can find points $b \in T$ and $b' \in T'$ so that the implication
\begin{equation}
\label{eq:Existence-morphism}
b \leq h. b\quad  \Longrightarrow \quad b' \leq' h. b'
\end{equation}
holds for every element $h \in G$.
\end{lem}
\index{Morphisms of measured trees!existence}

\begin{proof}
Assume first a $G$-equivariant morphism $\varphi \colon (T,\leq, \lambda) \to (T',\leq', \lambda')$ exists.
Pick $b \in T$ and set $b' = \varphi(b)$. 
Then implication \eqref{eq:Existence-morphism} holds,
for it is a consequence of the compatibility of $\varphi$ 
with respect to the partial orderings $\leq$ and $\leq'$ required in Definition
\ref{definition:Morphism-rooted-pre-trees}.

Conversely, 
assume $b\in T$ and $b'\in T'$ are points enjoying property \eqref{eq:Existence-morphism}.
Let $S$ be the stabilizer of the point $b$.
Then the inequalities $b \leq h. b$ and $b \leq h^{-1}.b$ hold for each element $h \in S$
whence  $S$ stabilizes $b'$ by implication \eqref{eq:Existence-morphism}.
The assignment $g.b \mapsto g.b'$ ist therefore licit;
as $G$ acts transitively on $T$,
it defines a map $\varphi \colon T \to T'$. 
Implication \eqref{eq:Existence-morphism} also entails 
that $\varphi$ is compatible with the order relations $\leq$ and $\leq'$;
indeed,
if $g_1. b \leq g_2.b$ then $b \leq (g_1^{-1}g_2).b$ 
(for $G$ acts by order preserving automorphisms on $(T,\leq, \lambda)$),
so $b' \leq (g_1^{-1}g_2).b$ by \eqref{eq:Existence-morphism}
and thus $g_1. b' \leq g_2.b'$.
The computation
\begin{align*}
\lambda'(\varphi(g.b)) - \lambda (g.b)  &=\lambda'(g.b') - \chi(g.b) 
= (\chi(g) + \lambda'(b')) - (\chi(g) + \lambda(b)\\
&=
 \lambda'(b') - \lambda(b)
\end{align*}
finally implies 
that the difference $\lambda'(\varphi(a)) - \lambda(a)$ does not depend on $a = g.b$
and so $\varphi$ satisfies also the second property required in Definition
\ref{definition:Morphism-rooted-pre-trees}. 
\end{proof}

Now to the construction of a morphism into a given measured $G$-tree:
\begin{prp}
\label{prp:Constructing-morphism-from-T(G,XX,chi)} 
Let $G$ be a group, $\chi \colon G \to \R$ a non-zero character
and $(T',\leq', \lambda')$ a measured $G$-tree with associated character $\chi$.
Choose an element $t \in G$ with positive $\chi$-value and a point $b' \in T'$ on the axis of $t$.
Then there exists a generating set $\eta \colon \XX \incl G$ 
so that the function $g \mapsto g.b'$ induces a $G$-equivariant morphism
$\varphi \colon T(G, \XX, \chi) \to (T', \leq', \lambda')$.
\end{prp}
 \index{Morphisms of measured trees!construction}

\begin{proof}
We begin by constructing a subset $\XX \subset G_{\chi}$ with the property
that the relation $b' \leq x.b'$ holds  for every $x \in \XX$. 
Since $b'$ is on the axis of $t$ and as  $\chi(t) > 0$, 
the desired relation holds for $x = t$.
Consider now an arbitrary element $z \in G$. 
The rays descending from $z.b'$  and $b'$ intersect;
since $b' \in A_{t} $ 
there exists therefore a positive integer $m_{z} $ with $t^{-m_{z}}.b' \leq z.b'$,
and so $b' \leq (t^{m_{z}} z).b'$. 
Set 
\[
\XX  = \{t\} \cup \{ t^{m_{z}} \cdot z \mid z \in G \}.
\]
Then $\XX$  is contained in $G_\chi$ and it generates the group $G$.

Let $\Gamma$ be the Cayley graph $\Gamma( G, \XX)$ 
and let $\CC_{1}$ denote the connected component 
of the subgraph $\Gamma_{\chi}$ containing the unit element $1 $ of $G$.
We assert that $b' \leq h.b'$ for every group element $h$ in the vertex set of $\CC_1$.
The proof will be by induction on the length $\ell$ of the $(\XX \cup \XX^{-1})$-word 
$w =y_{1}\cdots y_{\ell}$
that describes a path $p = (1, w)$ from  1 and $h$.
If $\ell =0$, then $h = 1$ and the claim holds.
Assume therefore that $\ell > 0$ and write $h = h' \cdot y_{\ell}$.
Then $b' \leq h'.b'$ by the inductive hypothesis.  
If $y = y_{\ell} \in \XX$ 
then $b' \leq y.b'$ by the choice of $\XX$,
whence $h'.b' \leq h'.(y.b') = h.b'$ and so $b' \leq h.b'$.
If, on the other hand,
$y^{-1}  \in \XX$ then $y. b' \leq b'$ and $h.b' = h'.(y.b') \leq h'.b'$.
The points $b'$ and $h.b'$ lie therefore both to the ray descending from $h'.b'$.
As this ray is totally ordered and as $\chi(h) \geq 0$, 
the relation $b' \leq h.b'$ must hold.

The component $\CC_1$ is a point of the measured tree $T(G, \XX, \chi)$; call it $b$.
The inequality $b \leq h.b$ holds precisely for the vertices of $\CC_1$;
in view of the preceding paragraph implication \eqref{eq:Existence-morphism} is therefore satisfied.
But if so,
Lemma \ref{lem:Constructing-morphism-from-transitive-rooted-pre-tree}  allows one to conclude 
that the assignment $g. \CC_1 \mapsto g.b'$ defines a $G$-equivariant morphism 
$\varphi \colon T(G, \XX, \chi) \to (T', \leq', \lambda')$.
\end{proof}
\begin{note}
\label{note:prp-Constructing-morphism-from-T(G,XX,chi)} 
The proof of Proposition \ref{prp:Constructing-morphism-from-T(G,XX,chi)} 
goes back to the proof of Theorem 4.1 in Chapter II of \cite{BiSt92}.
\end{note}
%
\subsubsection{Morphism into a $(G,\R)$-tree with fixed end}
\label{sssec:Construction-morphism-real-tree}
In this section, 
we show that a measured tree of the form $T(G, \XX, \chi)$ can be ``completed'' to a 
$(G,\R)$-tree $(X,d_X)$ with a fixed end and associated character $\chi$.
The details are spelled out in
\begin{prp}
\label{prp:Constructing-morphism-into-G-R-tree}  
Given a group $G$, a generating system $\eta \colon \XX \to G$ of $G$ and a non-zero character $\chi$ of $G$,
let $T = T(G, \XX, \chi)$ be the measured tree associated to the triple $(G, \XX, \chi)$.
Define a function $d \colon T \times T \to \R$ be the formula
\begin{equation}
\label{eq:Defining-pseudo-metric-on-T}
d(a_1,a_2) 
= 
\inf\{\lambda(a_1) + \lambda(a_2) - 2 \lambda(a) \mid  a\leq a_1 \text{ and } a \leq a_2 \}.
\end{equation}
Then $d$ is a pseudo-metric on $T$;
let $(\bar{T},\bar{d})$ denote the associated metric space.
This metric space can be embedded into a $(G,\R)$-tree $(X,d_X)$ 
with the following properties:
\begin{enumerate}[(i)]
\item if the rank $\rk(\chi)$ of $\chi$ is 1, then $d$ is a metric, $(T,d)$ is a $(G, \Z)$-tree 
and $(X,d_X)$ can be chosen to be the geometric realization of $(T,d)$.
\item if $\rk(\chi) > 1$,
there exists a $G$-equivariant isometry $\varphi \colon (T,d) \to (X, d_X)$ with dense image;
\item if the $(G, \R)$-tree $(X,d_X)$ contains a $G$-invariant line 
the measured tree $ T = T(G, \XX, \chi)$ is a line
 and the subgraph $\Gamma(G, \XX)_\chi$ is connected.
\end{enumerate}
\end{prp}
\index{Construction of!morphism into R-tree@morphism into an $\R$-tree}

\begin{proof}
Let $d \colon T \times T \to \R$ be the function defined by formula
\eqref{eq:Defining-pseudo-metric-on-T}.
Since every pair of points $\{a_1, a_2\}$ in $T$ is preceded by some element $a \in T$,
the definition is licit.
The function $d$ is symmetric and vanishes if $a_1=a_2$;
moreover, a short calculation will confirm that it satisfies the triangle inequality,
and so it is a pseudo-metric.
The computation
\begin{align*}
d(g.a_1,g.a_2) 
&= 
\inf\{\lambda(g.a_1) + \lambda(g.a_2) - 2 \lambda(b) \mid  b\leq g.a_1 \text{ and } b \leq g.a_2 \}\\
&=
\inf\{\lambda(g.a_1) + \lambda(g.a_2) - 2 \lambda(g.a) \mid  a \leq a_1 \text{ and } a \leq a_2 \}\\
&=
\inf\{\lambda(a_1) + \lambda(a_2) - 2 \lambda(a) \mid  a\leq a_1 \text{ and } a \leq a_2 \}\\
&=
d(a_1,a_2)
\end{align*}
finally proves that this pseudo-metric is $G$-invariant.

Let $(\bar{T}, \bar{d})$ be the resulting metric space;
as the pseudo-metric $d$ is $G$-invariant, so is the metric $\bar{d}$.
Let $\pi \colon T \epi \bar{T}$ denote the canonical projection;
it sends the point $a$ to the ball $\bar{a}$ of radius 0 centered at $a$,
and it is $G$-equivariant.

\paragraph{Claim (i)}If $\im \chi$ is infinite cyclic, we may and shall assume that $\im \chi = \Z$.
The function $\lambda$ maps then $T$ onto $\Z$.
For every $j \in \Z$, 
let $V_j = \lambda^{-1}(\{ j \})$ denote the set of components of the subgraph $\Gamma_\chi^{[j, \infty)}$. 
Then $T = \bigcup_{j \in \Z} V_j$.
Moreover, every vertex $v \in V_j$ is preceded by a unique  $w_v \in V_{j-1}$ 
and so gives rise to a positively oriented edge $e_v =(v, w_v)$;
let $E^+$ be the set of all these positive edges and set $E = E^+ \cup E^-$ 
(cf.\,section \ref{ssec:Terminology-graphs}).
Then $(T, E)$ is a $\Z$-tree
and the action of $G$ on $T$ turns it into a $(G,\Z)$-tree.
It can be embedded into an $(G,\R)$-tree 
by replacing the oriented edges $(v, w_v)$ by segments $[v,w_v]$
that are isometric to the unit interval $[0,1] \subset \R$ 
and extending the metric and the action suitably
(see, \eg,  \cite[Cor.\,4.5]{AlBa87}, \cite[Section 1.9]{BH99}, or  \cite[pp.\,44--49]{Chi01}).

\paragraph{Claim (ii)} If the rank of $\chi$ is bigger than 1,
the image of $\chi$ is a dense subgroup of $\R_{\add}$;
the space $(\bar{T}, \bar{d})$, however, need not be path-connected 
and so it may not be an $\R$-tree.
We therefore add to $\bar{T}$ certain points of its metric completion.
One way of controlling these additions consists in resorting to Chiswell's construction of a rooted tree
(see, \eg, \cite[pp.\,299--308]{AlBa87}). 
We shall reach our goal by a different route:
we shall equip $\bar{T}$ with the structure of a measured $G$-tree,
use this structure to define a subspace $(X, d_X)$ of the metric completion of $(\bar{T}, \bar{d})$
and then verify that $(X, d_X)$ carries a canonical $G$-action and is actually a $(G,\R)$-tree.
\smallskip

\emph{Step 1: construction of the measured $G$-tree  $(\bar{T}, \leq', \lambda')$.}
We introduce a relation $\leq'$ on $\bar{T}$ and a function $\lambda' \colon \bar{T} \to \R$ by setting:
\begin{align}
\bar{a}_1 \leq' \bar{a}_2
&\Longleftrightarrow 
\text{there exists } b \in T \text{ with } d(a_1,  b) = 0 \text{ and } b \leq a_2,
\label{eq:Definition-ordering-relation-metric-space-X}\\
\lambda'(\bar{a}) &\hspace*{2mm}= \hspace*{1.5mm}  \lambda(a).
\label{eq:Definition-function-lambda-metric-space-X}
\end{align}
The definition of $\lambda'$ is licit 
because the function $\lambda$ is constant on each ball of radius 0.
Indeed,
suppose $c_1$ and $c_2$ are at distance 0.
For every positive real number $\varepsilon$ 
there exists then a point $c \in T$
which precedes both $c_1$ and $c_2$ and satisfies the inequality 
\[
(\lambda(c_1) - \lambda(c) )+(\lambda(c_2) - \lambda(c) )  < \varepsilon.
\]
This inequality implies that $|\lambda(c_1) - \lambda(c_2)|$ is bounded by $\varepsilon$.

The definition of  $\leq'$ is also licit.
Indeed, 
suppose $\tilde{a}_1$ and $\tilde{a}_2$ are points 
with $d(\tilde{a}_1, a_1) = d(\tilde{a}_2, a_2) = 0$.
Then  $d(\tilde{a}_1, b) = 0$ and two cases arise:
if $\lambda(b) < \lambda(a_2)$,
the definition of the pseudo metric $d$ implies 
that the rays descending from $a_2$ and from $\tilde{a}_2$ both contain $b$ whence $b \leq \tilde{a}_2$.
So suppose that $\lambda(b) = \lambda(a_2)$.
Then the points $b$ and  $a_2$ coincide and so  $d(b, \tilde{a}_2) = 0$,
whence $d(\tilde{a}_1,  \tilde{a}_2) = 0$.
The rôle of the auxiliary element $b$ occurring in equation \eqref{eq:Definition-ordering-relation-metric-space-X}
can therefore be played by $\tilde{a}_2$.
\smallskip

We verify next that $\leq'$ is an order relation.
It is clear that the relation is reflexive.
Suppose that $\bar{a}_1 \leq' \bar{a}_2$ 
and let $b \in T$ be an element with $d(a_1, b) = 0$ and $b \leq a_2$.
Then $\lambda(a_1) \leq \lambda(a_2)$. 
If one has also $\bar{a}_2 \leq' \bar{a}_1$ 
then $\lambda(a_1) = \lambda (a_2)$, 
hence $b = a_2$ and  $\bar{a}_1 = \bar{a}_2$.
To see that the relation is transitive,
assume that $\bar{a}_1 \leq' \bar{a}_2$ and  $\bar{a}_2 \leq' \bar{a}_3$.
There exist then points $b$ and $c$ such that 
$d(a_1,b) = 0$, $ b \leq a_2$ and $d(a_2,c) = 0$, $ c \leq a_3$.
Two cases arise:
if $b = a_2$ then $d(a_1,c) = 0$ and so $\bar{a}_1 \leq' \bar{a}_3$;
otherwise, $\lambda(b) < \lambda(a_2)$ 
whence $b$ lies on the ray descending from $a_3$ and so $b \leq a_3$.
In both cases, $\bar{a}_1\leq' \bar{a}_3$. 
All taken together this proves that $\leq'$ an an order relation.

The fact that $(T,\leq)$ is directed 
and definition \eqref{eq:Definition-ordering-relation-metric-space-X}
immediately imply that  $(\bar{T},\leq')$ is directed. 
In addition,
the fact that $\lambda$ maps each ray $(-\infty, a]$ isometrically  into the interval $(-\infty, \lambda(a)] \subset \R$ and the definition of $\leq'$ guarantee
that $\lambda'$ has the analogous property.
The triple $(\bar{T}, \leq', \lambda')$ is therefore a measured tree;
the group $G$ acts on it by automorphisms, 
for it acts by automorphisms on $(T, \leq, \lambda)$.
It follows 
that the canonical projection $\pi \colon T \epi \bar{T}$ 
is a $G$-equivariant morphism of measured $G$-trees. 
\smallskip

\emph{Step 2: constructions of the metric space $(X,d_X)$ 
and of the measured $G$-tree $(X,\leq_X, \lambda_X)$.}
The metric space $(\bar{T}, \bar{d})$ can embedded into a complete metric space, say $(Y, d_Y)$.
This space is too large for our purpose, 
and so we define a subspace $(X, d_X)$
in which only the rays $(-\infty, \bar{a}]$ of $(\bar{T}, \bar{d})$ are ``completed''.

Let $X$ be the subset made up of the points $y \in Y$ 
that can be represented by a \emph{descending Cauchy-sequences},
\ie, by Cauchy-sequences  
\begin{equation}
\label{eq:Definition-special-sequence}
n \mapsto \bar{a}_n \quad \text{with}\quad  \bar{a}_{n+1}  \leq' \bar{a}_{n} \text{ for all } n \in \N.
\end{equation}
The metric $\bar{d}$ on $\bar{T}$ extends uniquely to a metric $d_X$ on $X$.

The next aim is to equip the metric space $(X, d_X)$ with the structure of a measured tree.
The definition of the measuring function $\lambda_X \colon X \to \R$ is straightforward:
the function $\lambda' \colon \bar{T} \to \R$ is Lipschitz-continuous (with constant 1) 
and so it extends uniquely to a function $\lambda_Y$ on the complete metric space $Y$;
let $\lambda_X \colon X \to \R$ be its restriction to $X$.
Consider now the ray $\bar{R}_{\bar{a}} = (-\infty, \bar{a}]$ descending from a point $\bar{a} \in \bar{T}$.
The function $\lambda'$ maps it bijectively 
onto the subset $(\infty, \lambda'(\bar{a})] \cap \im \chi$ of the real line $\R$.
Let $R_{\bar{a}} \subset X$ be the subset 
containing all the limit points of sequences in $\bar{R}_{\bar{a}}$. 
The extended function $\lambda_X \colon X \to \R$ maps $R_{\bar{a}}$ bijectively and isometrically
onto the interval $(-\infty, \lambda_X(\bar{a})]$ of $\R$. 
We use this bijection to pull back the order relation on the interval $(-\infty, \lambda_X(\bar{a})]$
onto the ``ray'' $R_{\bar{a}}$ 
and to obtain thus an order relation $\leq_X$ on $R_{\bar{a}}$.
This order relation extends the original order relation on the ray $\bar{R}_{\bar{a}}$
and it does not change on $R_{\bar{a}}$ if the upper limit $\bar{a}$ is replaced by a point $\bar{b} \in \bar{T}$
with $\bar{a} \leq' \bar{b}$. 
The union of the order relations on all the ``rays'' $R_{\bar{c}}$ with $\bar{c} \in X$ 
is therefore an order relation on $X$; we denote it by  $\leq_X$. 

For future reference we summarize the above reasonings in
\begin{lem}
\label{lem:Rays-in-X}
Every ray $(-\infty, x]$ descending from a point $x \in X$ is linearly ordered and 
$\lambda_X$ maps it isometrically 
and in an order-preserving way onto the interval $(-\infty,  \lambda_X(x)]$ of $\R$.
\end{lem}

The triple $(X, \leq_X, \lambda_X)$ is thus a measured tree.
The facts that $G$ acts by isometries on the metric space $(\bar{T}, \bar{d})$ 
and permutes the rays of the measured tree $(\bar{T}, \leq', \lambda')$, finally,  implies 
that $G$ acts by isometries on the metric space $(X, d_X)$ 
and by morphisms on the measured tree $(X, \leq_X, \lambda_X)$.
Moreover, 
the inclusion $\mu \colon \bar{T} \incl X$ is a $G$-equivariant morphism of measured $G$-trees. 
\smallskip

\emph{Step 3: verification that $(X,d_X)$ is a $(G,\R)$-tree.}
By Steps 1 and 2 we know 
that the composite $\varphi = \mu \circ \pi$  is a $G$-equivariant morphism 
of the measured $G$-tree  $T(G,\XX,\chi)$ into the measured $G$-tree $(X,\leq_X,\lambda_X)$
and that $X$ is equipped with a metric $d_X$. This metric satisfies the equation
\[
d_X(x_1,x_2) = \inf\{\lambda_X(x_1) + \lambda_X(x_2) - 2 \lambda_X(x) \mid  x\leq_X x_1 \text{ and } x \leq_X x_2 \}.
\]

To verify that $(X,d_X)$ is an $\R$-tree,
we have to \emph{check two properties}:
given points $x_1$, $x_2$ of $X$ 
there  must exist a \emph{unique geodesic segment} $[x_1,x_2]$ from $x_1$ to $x_2$,
and given three points $x_1$, $x_2$ and $x_3$ 
such that the geodesic segments $[x_1, x_2]$ and $[x_2, x_3]$ have only the point $x_2$ in common, 
\emph{the union $[x_1,x_2] \cup [x_2, x_3]$ must be the geodesic segment from $x_1$ to $x_3$}
(see, e.\;g., \cite[p. 276, items  2.3 and  2.5]{Sha87} or \cite[Definition 2.9]{AlBa87}).

Let $x_1$ and $x_2$ be points of $X$. 
For the proof of the \emph{existence} of a geodesic arc from $x_1$ to $x_2$,
two cases will be distinguished:
if $x_1 \leq_X x_2$ or $x_2 \leq_X x_1$, the claim follows directly from Lemma
\ref{lem:Rays-in-X};
otherwise, the local compactness of the rays implies 
that there exists a point $x$ satisfying the properties
\[
x \leq_X x_1, \quad x \leq_X x_2 \text{ and } d_X(x_1, x_2) 
= (\lambda_X(x_1) - \lambda_X(x)) + (\lambda_X(x_2) - \lambda_X(x)).
\]
Set $r_1 = d_X(x, x_1)$.
Lemma \ref{lem:Rays-in-X} then permits one to define a path
\[
\gamma \colon [0, d_X(x_1,x_2)] \to (\-\infty, x_1] \cup (\-\infty, x_2],
\]
given by the formula
\begin{equation}
\gamma(r) = 
\begin{cases} 
\text{point $z$ on $(\-\infty, x_1]$ with } \lambda(x_1) - \lambda(z) = r
&\text{if } r \leq r_1, \\
\text{point $z$ on $(\-\infty, x_2]$ with } \lambda(z) - \lambda(x) = r -r_1 
&\text{ if } r > r_1.
\end{cases}
\label{eq:Definition-geodesic-with-comparable-end-points}
\end{equation}
This path is a geodesic arc (cf.\,\cite[Lemma 2.1.12]{Chi01}).

We next show
that the geodesic arc between two points is \emph{unique}.
The strategy will be similar to that used in the proof of Lemma 2.1.12 in Chiswell's monograph \cite{Chi01}.
Let $[x_1,x_2]$ denote the image of the previously constructed geodesic arc $\gamma$ from $x_1$ to $x_2$.
This set contains for every real $r \in[0, d_X(x_1, x_2)]$ a unique point $x$ with  $d_X(x_1,x) = r$.
Consider now an arbitrary geodesic arc $\gamma'$ form $x_1$ to $x_2$.
If its image lies in $[x_1,x_2]$,  the geodesic arc $\gamma'$  must coincide with $\gamma$.
Otherwise, it contains a point $z$ outside $[x_1, x_2]$. 
Two cases then arise.

If $x_1$ and $x_2$ are \emph{comparable}, 
we may assume that $x_2 \leq_X x_1$.
The rays descending from $x_1$ and $z$, respectively,
intersect then in a ray $(-\infty, u]$.
If $x_1=u$, then $x_1 <_X z$ and so
\[
 d_X(x_1,x_2) = d_X(x_1, z) + d_X(z,x_2) = 2 d_X(x_1, z) + d_X(x_1, x_2),
\]
a contradiction. 
If $u < x_2$ one arrives at a similar contradictory situation.
If, finally, $x_2 \leq_X u \leq_X x_1$, 
one deduces the equality
$d_X(x_1, x_2) = 2 d_X(u, z) + d_X(x_1, x_2)$, 
a third contradiction.

Assume now that $x_1$ and $x_2$ are \emph{not comparable}
and let $x$ be the point where the rays descending form $x_1$, respectively $x_2$, meet.
Various cases arise;
they are summarized in Figure \ref{fig:Uniqueness-geodesics}.
Each of these cases leads to a contradiction in the manner exemplified before.
All taken together, we have thus shown that
there is a unique geodesic arc from a given point $x_1$ to a given point $x_2$.
\begin{figure}[h]
\psfrag{xx}{\hspace*{-1mm}\small $x$}
\psfrag{xx1}{\hspace*{1mm}\small $x_1$}
\psfrag{xx2}{\hspace*{1mm} \small $x_2$}
\psfrag{zz1}{\hspace*{-1mm} \small $z_1$}

\psfrag{zz3}{\hspace*{-0.5mm}\small $z_2$}
\psfrag{zz4}{\hspace*{-0.3mm}\small $z_3$}
\psfrag{zz5}{\hspace*{-1.3mm}\small $z_4$}
\psfrag{zz6}{\hspace*{-0.5mm}\small $z_5$}
\psfrag{uu3}{\hspace*{-0.5mm}\small $u_2$}
\psfrag{uu4}{\hspace*{-0.5mm}\small $u_3$}
\psfrag{uu5}{\hspace*{-1.2mm}\small $u_4$}
\begin{center}
\includegraphics[width=9cm]{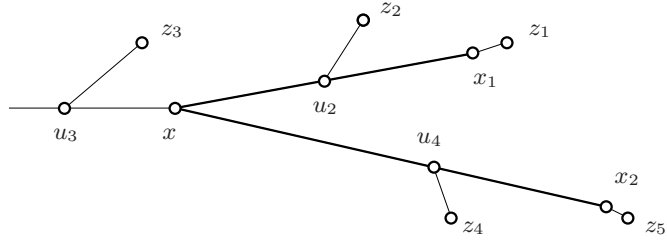}
\caption{Proof of the uniqueness of the geodesics in $(X, d_X)$}
\label{fig:Uniqueness-geodesics}
\end{center}
\end{figure}

We are left with the verification of the second characteristic property.
Let $x_1$, $x_2$, $x_3$ be three points of $X$
and let $[x_1, x_2]$ denote the geodesic segment with endpoints $x_1$, $x_2$,
and  $[x_2,x_3]$ the segment with endpoints $x_2$, $x_3$.
Suppose these segments have only the point $x_2$ is common;
we claim that $[x_1, x_2] \cup [x_2,x_3]$ is a geodesic segment.
This is clearly the case if $x_2 \leq_X x_1$ or if $x_2 \leq_X x_3$.
If, on the other hand, neither $x_2 \nleq_X x_1$ nor $x_2 \nleq_X x_3$ were valid,
there would exist a point $x <_X x_2$ 
so that the segment $[x,x_2]$ would be contained  both in $[x_1,x_2]$ and in $[x_2,x_3]$,
a flagrant contradiction to the assumption 
that the segments $[x_1, x_2]$ and $[x_2,x_3]$ meet only in the point $x_2$.

\paragraph{Claim (iii)}  Assume first that $\rk \chi > 1$; then $\im \chi$ is a dense subset of $\R$.
Let $L$ be a $G$-invariant line in  the $(G,\R)$-tree $(X, d_X)$.
Our first aim is to prove 
that $L$ is a linearly ordered subset of the measured tree $(X, \leq_X, \lambda_X)$.

Fix an element $h \in H$ with $\chi(h) >0$. 
Since $h$ is an automorphism of the measured tree $(X, \leq_X, \lambda_X)$,
the equation
\begin{equation}
\label{eq:Equivariance-lambda-new}
\lambda_X(h.y) = \chi(h) + \lambda_X(y) > \lambda_X(y)
\end{equation}
holds for every point $y \in X$. 
This equation shows that $h$ has no fixed points
and so the isometry $\alpha_h$ induced by $h$ the on $(G, \R)$-tree $(X, d_X)$ is hyperbolic
(cf.\,\cite[p.\,82]{Chi01}).
The axis $L_h$ of $h$ is a geodesic line; 
it coincides with $L$ and can be written as s union of segments
\[
\bigcup\nolimits_{j \in \Z} [h^j.x,  h^{j+1}.x] \,;
\]
here $x$ is an arbitrary point of $L_h$
(cf.\;\cite[Thm.\,3.1.4]{Chi01}).
Consider now the segment $[x, h.x]$ of $X$.
According to Step 3 in the construction of $(X, d_X)$,
this segment contains a point $z$ 
so that both subsegments $[x,z]$ and $[z, h.x]$ are linearly ordered.
If $z$ is the endpoint $h.x$, then $x$ lies on the ray $(-\infty, h.x]$ descending from $h.x$.
Otherwise, 
there would exist an element $z \in X$ with $z \leq_X x$ and $<_X h.x$, hence $h.z \leq_X h.x$. 
So $z$ and $h.z$ would both belong to the ray descending from $h.x$.
Equation \eqref{eq:Equivariance-lambda} with $y = z$ would then show that $z < h.z$,
and so the computation
\[
d_X(z, h.z) < d_X (x,z) +d_X(z, h.z) +d_X(h.z, h,x) = d_X(x, h.x)
\]
would be valid; it would contradict the fact that $h$ acts on its axis by a translation. 

The above argument shows that $x \leq_X h.x$ for every $x \in L$. 
So $L$ is contained in the axis $A_h$ of $h$ in the measured tree $(X, \leq_X, \lambda_X)$;
since $L$ contains the orbit $\gp(h).x$ for every $x \in L$, 
the line $L$ actually coincides with the axis $A_h$ of $h$.
Lemma \ref{lem:Rays-in-X} next shows 
that the  function $\lambda_X$ maps the axis $A_h$ bijectively onto $\R$ 
and so $L$ contains points of the measured subtree  $(\bar{T}, \leq', \lambda')$.
Since $\bar{T}$ is a $G$-invariant subset of $X$,
the intersection $\bar{A}_h = \bar{T} \cap A_h$ is $G$-invariant 
and it is the axis of the action of $h$ on $\bar{T}$.
As $G$ acts transitively on the tree $T$, hence on the quotient $\bar{T}$ of $T$,
the tree $\bar{T}$ coincides therefore with the axis $\bar{A}_h$ of $h$,
and so it is a line.

But if so, \emph{the tree $T$ must be a line}.
Indeed:  the axis $A_h$ of $h$ in the measured tree $T = T(G, \XX, \chi)$ is a line
and every point $a_1 \in T$ is at distance 0 from some point $a_2 \in A_h$ 
(this follows from the definition of $\bar{T}$ and the equality $\bar{T} = \bar{A}_h$).
The definition \eqref{eq:Defining-pseudo-metric-on-T} of the pseudo metric $d$ then implies that
$h^{-1}.a_1$ and $h^{-1}.a_2$ lie on the ray descending from $a_2$
and so they coincide. So $a_1 = h.(h^{-1}(a_1)) = a_2 \in A_h$.
\smallskip

Assume now that  $\rk \chi = 1$. 
We may and shall assume that $\im \chi = \Z$.
The measured tree $T = T(G, \XX, \chi)$ can then be viewed as a $\Z$-tree,
with each edge $(v,w)$ of length 1 and $\lambda(v) = \lambda(w) + 1$
(cf.\,proof of claim (i)). 
The tree $T$ is a subtree of the $(G, \R)$-tree $(X, d_X)$
and the definition of the metric $d_X$ implies 
that every isometric embedding $\varphi \colon R \mono X$ passes through a point of this subtree $T$.

Let $L$ be a $G$-invariant line of $(X, d_X)$.
Then $L$ is the image of an isometric embedding $\varphi \colon R \mono X$; 
by the previous remark it thus contains a point $a_0 \in T$.
Choose $t \in G$ with $\chi(t) = 1$. 
It then follows as in the previous part 
that $t$ acts on $(X, d_X)$ by a hyperbolic isometry
and that it acts on the line $L$ by a translation, say with amplitude $\alpha$.
Since $t.a_0 \in T$ this amplitude is a positive integer.
It is actually 1, 
for otherwise the segment $[a_0, t.a_0] \subset X$ would contain 
more than two elements of $T$ 
and one could find, as before, an element $z \in [a_0, t.a_0]$ with $d_X(z, t.z) < c$. 
The segment $[a_0, h.a_0]$ is therefore 
the geometric realization of an edge of the $\Z$-graph $T = T(G, \XX, \chi)$.
It follows, first,  that the line $L$ is linearly ordered 
and then that the intersection $T \cap L$ is a non-empty, linearly ordered $G$-invariant subset of $T$.
Since $G$ acts transitively on $T$ this means that $T$ is linearly ordered and thus a line.
\end{proof}

\subsection{Some applications} 
\label{ssec:Some-applications-of-Browns-characterization}
%
Theorem \ref{thm:Characterization-Sigma1-via-actions-on-trees}
is a very powerful tool 
that allows one to establish,
by pleasant geometric arguments,
that a point $[\chi]$ belongs to the invariant $\Sigma^1(G)$ of the group under consideration.
In the section, this assertion will be illustrated by some examples.
%
\subsubsection{Application 1: joins of subgroups}
\label{sssec:Join-subgroups-2}
We begin with a generalization of Proposition \ref{prp:Sigma-1-join-subgroups}.
\begin{prp}
\label{prp:Sigma-1-join-subgroups-2}
Assume $G$ is generated by a (non-empty) collection  of subgroups $\{ G_v  \mid v \in V \} $.
Given a non-zero character $\chi \colon G \to \R$,
let $\GG(\chi)$ denote the combinatorial graph with vertex set $V$ 
and edge set the set of those pairs $\{u,v\}$ for which $\chi$ is non-zero on the intersection $G_u \cap G_v$.
If the conditions
\begin{enumerate}[(i)]
\item for every $v \in V$,
the restriction of $\chi$ to $G_v$ is non-zero and represents a point of $\Sigma^1(G_v)$,
\item the graph $\GG(\chi)$ is connected
\end{enumerate}
are satisfied, then $\chi$  represents a point of $\Sigma^1(G)$.
\end{prp}
\index{Computation of Sigma1@Computation of $\Sigma^1$ for!joins of subgroups}
\begin{proof}
Assume conditions (i) and (ii) are satisfied.
If $V$ has only one element, then $G = G_v$ and so the claim holds by assumption (i);
so suppose $\card(V) > 1$ and consider a measured $G$-tree $(T, \leq, \lambda)$ 
with associated character $\chi$. 
Our aim is to find a $G$-invariant line in this tree.
By restricting the action of $G$ on $(T, \leq, \lambda)$ to that of a subgroup $H$ 
one obtains a measured $H$-tree.
If this observation is applied to one of the subgroups $G_v$ with $v \in V$,
hypothesis (i) implies that 
the restriction $\chi_v$ of $\chi$ to $G_v$ is non-zero and $[\chi_v] \in G_v$;
by Theorem \ref{thm:Characterization-Sigma1-via-actions-on-trees} 
the measured $G_v$-tree $T$ contains therefore an invariant line $L_v$.
This line coincides with the axis of every hyperbolic element $h_v \in G_v$
(cf.\,claim (iii) in Lemma \ref{lem:Properties-A-sub-g-general-case}).

Consider next an edge $\{u, v\}$ in the graph $\GG(\chi)$.
Then $G_u \cap G_v$ will contain a hyperbolic element $h_{uv}$; 
let $A_{uv}$ be its axis.
As $L_u = A_{uv}$, the axis $A_{uv}$ is $G_u$ invariant;
since $L_v = A_{uv}$, the axis $A_{uv}$ is $G_v$ invariant, too, 
and so it is invariant under the join $\gp(G_u \cup G_v)$ of these two subgroups.

Fix now a vertex $v_0$ and let $L_0$ an invariant line of $G_{v_0}$.
The hypotheses that $\GG(\chi)$ be connected and an obvious induction then allow one to see
that  $L_0$ is invariant under every subgroup $G_v$ with $v \in V$ 
and so under $G$.

As the previous argument is valid for every measured $G$-tree,
implication $(iii) \Rightarrow (i)$ of Theorem \ref{thm:Characterization-Sigma1-via-actions-on-trees} 
applies and shows that $[\chi] \in \Sigma^1(G)$.
\end{proof}

\begin{remarks}
\label{remarks:thm-Sigma-1-join-subgroups}
a) In the above proof, we considered an arbitrary measured $G$-tree; 
one could equally well have used the measured trees $T(G, \XX, \chi)$ associated to $(G, \XX, \chi )$ 
with $\XX$ ranging over the generating sets of $G$.

b) The previous proof is phrased in terms of measured $G$-trees.
One could also use $(G,\R)$-trees with a fixed end and associated character $\chi$.
Whether one type of tree or the other is employed seems to be merely a matter of taste;
indeed, I know of no instance 
where a proof that a character $\chi$ represents a point of $\Sigma^1$ 
which is based on actions on one type of tree cannot be transformed in a straightforward manner 
into a proof using actions on the other type of tree.

Things can be different if one uses Proposition \ref{prp:Sigma-1-join-subgroups-2}
to establish that a given character does \emph{not} represent a point in $\Sigma^1$;
then a tree without an invariant line of only one type may be close at hand.
\end{remarks}

\subsubsection{Application 2: subnormal subgroups}
\label{sssec:Subnormal-subgroups}
Our second result generalizes Lemma \ref{lem:Abelian-normal-subgroup}:
\begin{prp}
\label{prp:Sigma-1-subnormal-subgroups}
Let $N$ be a subnormal subgroup of the group $G$ 
and let $\chi \colon G \to \R$ be a character
that does not vanish on $N$. 
Assume the restriction $\chi_{|N}$ of $\chi$ represents a point of $\Sigma^1(N)$.
Then $[\chi] \in \Sigma^1(G)$.
\end{prp}
\index{Computation of Sigma1@Computation of $\Sigma^1$ for!groups with subnormal subgroups}

\begin{proof}
Let $(T, \leq, \lambda)$ be a measured $G$-tree with character $\chi$.
View it as a measured $N$-tree by restricting the action.
In view of Theorem \ref{thm:Characterization-Sigma1-via-actions-on-trees}
and the hypothesis that $[\chi_{|N}] \in \Sigma^1(N)$,
the measured $N$-tree $T$ contains an $N$-invariant line $L$.
Moreover,
if $g \in G$ is an element which normalizes $N$ 
then $g.L$ is invariant under $g Ng^{-1}$.
As $g Ng^{-1}  = N$,
this shows
that $L$ is invariant under the normalizer of $N$.

If the previous argument is applied to successive members of a chain 
$N = N_0 < N_1 < \cdots <  N_\ell = G$ 
of subgroups in $G$ which testifies that $N$ is a subnormal subgroup,
one sees first that $L$ is invariant under $N_1$, then under  $N_2$, and finally under $G$.
Implication $(iii) \Rightarrow (i)$ of Theorem \ref{thm:Characterization-Sigma1-via-actions-on-trees}, 
then permits one to conclude that $[\chi] \in \Sigma^1(G)$.
\end{proof}

\begin{crl}
\label{crl:Sigma-1-subnormal-subgroups-with-Sigma1=S}
Let $N$ be a subnormal subgroup of the group $G$.
If  $\Sigma^1(N) = S(N)$
then $\Sigma^1(G)$ contains the complement of the subsphere $S(G,N)$.
\end{crl}

\begin{remarks}
\label{remarks:prp-Sigma-1-subnormal-subgroups}
a) Proposition \ref{prp:Sigma-1-subnormal-subgroups} 
and its corollary are taken from the monograph \cite{BiSt92} (see section II.4.3).

b) The hypotheses that $\Sigma^1(N)$ be the entire sphere $ S(N)$ holds, in particular, 
if $N$ is nilpotent or,  more generally, \emph{locally nilpotent}.
This follows, \eg, from Corollary \ref{crl:Sigma1-Group-with-no-free-submonoid}
and Remark \ref{examples:No-free-submonoid}b; it can also be derived from the fact 
that the derived group of a finitely generated nilpotent subgroup is finitely generated 
and Proposition \ref{prp:Sigma-1-join-subgroups-2}.
We thus see that Corollary  \ref{crl:Sigma-1-subnormal-subgroups-with-Sigma1=S} 
generalizes lemma \ref{lem:Abelian-normal-subgroup} 
(in that Lemma, $N$ is assumed to be an abelian normal subgroup of $G$).
\end{remarks}
%
\subsubsection{Application 3:  $\Sigma^1$ of a direct product}
\label{sssec:Illustration-invariant-direct-product}
%
As a third application, 
we establish a formula for the complement of the invariant of a direct product of arbitrary groups;
in section \ref{sssec:Invariant-direct-product} 
the same formula has already been shown to hold for direct products of finitely generated groups.
\begin{prp}
\label{prp:Sigma1-direct-product-generalization}
Let $G = G_{1}\times G_{2}$ be the direct product of two arbitrary groups
and let $\pi_{1}$, $\pi_2$ denote the canonical projections onto the factors.
Then
\begin{equation}
\label{eq:Sigma1-direct-product-bis}
\Sigma^1(G_1 \times G_2)^c = \pi^*_1(\Sigma^1(G_1))^c \cup \pi^*_2(\Sigma^1(G_2))^c .
\end{equation}
\end{prp}
\index{Computation of Sigma1@Computation of $\Sigma^1$ for!direct products}
\begin{proof}
Let $\chi = (\chi_1, \chi_2) \colon G_1 \times G_2 \to \R$ be a non-zero character.
Assume first that the restrictions $\chi_i \colon G_i \to \R$ are both non-zero
and let $g_1 \in G_1$ and $g_2 \in G_2$ be elements with non-zero $\chi$-value.
Consider now a measured $G$-tree $(T, \leq, \lambda)$
and let $A_1$ be the axis of the hyperbolic element $g_1$.
Since $g_1$ is centralized by $G_2$,
Lemma \ref{lem:Axes-commuting-elements} shows that the axis is $G_2$-invariant;
moreover, since $\chi(g_2) \neq 0$ this lemma tells us also 
that $A_1$ coincides with the axis $A_2$ of $g_2$.
But if so,
one sees, 
by exchanging the rôles of $G_1$ and $G_2$ in the preceding reasoning,
that $A_2$ is a $G_1$-invariant line,
whence $A_1 = A_2$ is $G$-invariant.
Use now implication $(iii) \Rightarrow (i)$ of Theorem \ref{thm:Characterization-Sigma1-via-actions-on-trees}.
\smallskip

Suppose next that $\chi$ vanishes on one of the factors of $G = G_1 \times G_2$,
say on $G_{1}$. 
We assert that $[\chi]$ lies in $\Sigma^1(G)$ if, and only, if $[\chi_2] \in \Sigma^1(G_2)$.
Assume first that $[\chi_2] \in \Sigma^1(G_2)$ and consider a measured $G$-tree $(T, \leq, \lambda)$.
Since $G_2$ is a subgroup of $G$ the tree can be viewed as a measured $G_2$-tree.
By implication  $(i) \Rightarrow (iii)$ of Theorem \ref{thm:Characterization-Sigma1-via-actions-on-trees}
this tree has a $G_2$-invariant line; it is the axis of $A_2$ of a hyperbolic element $g_2 \in G_2$.
As $g_2$ is centralized by $G_1$,
Lemma \ref{lem:Axes-commuting-elements} next shows that $A_2$ is $G_2$-invariant.
But if so, this axis is $G = \gp(G_1 \cup G_2)$-invariant;
as $(T, \leq, \lambda)$ is an arbitrary measured $G$-tree, 
implication $(iii) \Rightarrow (i)$ of Theorem \ref{thm:Characterization-Sigma1-via-actions-on-trees}
thus shows that $[\chi] \in \Sigma^1(G)$.
Conversely, assume that $[\chi_2] \in \Sigma^1(G_2)^c$.
By  Theorem \ref{thm:Characterization-Sigma1-via-actions-on-trees}
there exists then a measured $G_2$-tree $(T, \leq, \lambda)$ that contains no $G_2$-invariant line.
View this tree as a $G$-tree by pulling back the action along $\pi_2 \colon G \epi G_2$.
The resulting $G$-tree has no $G$-invariant line,
whence implication $(i) \Rightarrow (iii)$ of Theorem \ref{thm:Characterization-Sigma1-via-actions-on-trees}
allows one to conclude that $[\chi] \in \Sigma^1(G)^c$.
\end{proof}
%

%
%
\section{Variation adapted to fundamental groups}
\label{sec:Variation-adapted-fundamental-groups}
%
%
In this section,
we present an alternate description of the invariant for groups
that are fundamental groups of finite connected CW-complexes $Y$ or of compact,  connected manifolds.
Given a non-zero character $\chi \colon G \to \R$,
the requirement  that the subgraph $\Gamma_\chi$ of a Cayley graph $\Gamma(G, \XX)$ be connected 
will be replaced by a condition  
that involves a function $f \colon X  \to \R$  and the subspace  $f^{-1}([0, \infty))$.
Here $X = \hat{Y}$ is the \emph{universal abelian cover} of the space $Y$ 
and $f$ denotes a continuous function satisfying the equation
\begin{equation}
\label{eq:Definition-G-equivariant}
f(g.x) = f(x) + \chi(g).
\end{equation}
The formula says that $f$ is $G$-\emph{equivariant:}
it is compatible with the action on $X$ 
given by the deck transformations of the covering projection $p \colon X \epi Y$ 
and the action by the translations $r \mapsto r + \chi(g)$ on the real line.
(TO BE COMPLETED)
%
 
%

%
\backmatter
\pagestyle{back}
\chapterstyle{back1}
%
%
%
\markboth{Bibliography}{Bibliography}
\addcontentsline{toc}{chapter}{Bibliography}
\bibliography{Z.main_2012}
\bibliographystyle{amsalpha}%
\markboth{Bibliography}{Bibliography}
\addcontentsline{toc}{chapter}{Index}
\markboth{Index}{Index}
\begin{theindex}

  \item 3-manifolds groups, 91

  \indexspace

  \item \textbf{A}bels, H., 41, 49
  \item Ahlin, A. Reiter, 84
  \item Almost finitely presented group, 40
  \item Amenable groups
    \subitem definition, 92
    \subitem examples, 92, 99
  \item Ascending HNN-extension
    \subitem definition, 18
  \item Axis of a hyperbolic element
    \subitem definition, 147, 152

  \indexspace

  \item \textbf{B}aumslag, B., 94, 96
  \item Baumslag, G., vii, 77, 84, 90, 94--96, 98--100, 102, 114
  \item Benli, M. G., 92
  \item Bestvina, M., 41
  \item Bieri, R., vii, ix--xii, 15, 23, 35, 39, 40, 48, 77, 79, 83--85, 
		90, 102, 104, 106, 111, 112, 115, 118, 119, 123, 141
  \item Blass, A., 80
  \item Brown's algorithm, 106
  \item Brown's characterization of $\Sigma^1$, 153
  \item Brown, K. S., xii, 79, 99, 106, 111, 115, 123
  \item Button, J. O., 80, 97, 112

  \indexspace

  \item \textbf{C}ayley graph
    \subitem definition, 5
    \subitem examples, 5--7
  \item Character associated to an action
    \subitem on a measured tree, 150
  \item Character of a group
    \subitem definition, 1
    \subitem equivalence relation, 2
    \subitem rank, 2
    \subitem rational, 2
  \item Character sphere
    \subitem coordinates, 3, 26
    \subitem definition, 2
    \subitem density of rank 1 points, 89
    \subitem morphisms, 45
    \subitem rank, 2
    \subitem sub-hemispheres, 3
    \subitem subspheres, 3
    \subitem topology, 2
  \item Characteristic subset
    \subitem properties, 152
  \item Characterization of
    \subitem coherent soluble groups, 85
    \subitem complement of $\Sigma^1$, 126
    \subitem fg normal subgroups, 23
    \subitem fp abelian by infinite cyclic groups, 84
    \subitem fp nilpotent-by-cyclic groups, 83
  \item Chou, C., 101
  \item Coherent group
    \subitem definition, 84
  \item Computation of $\Sigma^1$ for
    \subitem abelian groups, 8, 11
    \subitem amalgamated free products, 136
    \subitem direct products, 12--13, 164
    \subitem Engel groups, 135
    \subitem free groups, 10
    \subitem free products, 10
    \subitem graph groups, 32, 54--58, 73
    \subitem group $\BS(1,2)$, 9
    \subitem group of knot \texttt{K8n20}, 112
    \subitem group of knot \texttt{K8n21}, 113
    \subitem group of torus knot, 12
    \subitem groups of PL-homeomorphisms, 66
    \subitem groups with abelian normal subgroups, 129
    \subitem groups with centre, 11
    \subitem groups with fg derived group, 12
    \subitem groups with infinitely many ends, 137
    \subitem groups with large deficiency, 89
    \subitem groups with no free submonoid, 133
    \subitem groups with subnormal subgroups, 163
    \subitem HNN-extensions, 82, 136
    \subitem Houghton's groups, 121--123
    \subitem join of subgroups, 108
    \subitem joins of subgroups, 52, 162
    \subitem knot groups, 74
    \subitem lamplighter groups, 129
    \subitem metabelian groups, 19
    \subitem one relator groups, 106, 114
    \subitem right angled Artin groups, 148
    \subitem two generator groups, 65
    \subitem wreath products, 127
  \item Condensation group, 103
  \item Construction of
    \subitem morphism into an $\R$-tree, 156
  \item Cornulier, Y. de, 103, 104, 106
  \item Crowell, R. H., 84

  \indexspace

  \item \textbf{D}eficiency of a group
    \subitem definition, 88
    \subitem upper bound using homology, 88
  \item Definition of
    \subitem amenable group, 92
    \subitem axis of a hyperbolic element, 147, 152
    \subitem characteristic subset, 146, 152
    \subitem coherent group, 84
    \subitem deficiency of a group, 88
    \subitem elementary amenable group, 93
    \subitem elliptic element, 147, 152
    \subitem Engel group, 134
    \subitem hyperbolic element, 147, 152
    \subitem invariant $\Sigma^1$, 8
    \subitem invariant $\Sigma^1(-;\Z)$, 119
    \subitem invariant $\Sigma_{G'}(G)$, 121
    \subitem invariant $\Sigma^M$, 140
    \subitem invariant $\Sigma^m(-;A)$, 119
    \subitem large group, 93
    \subitem LERF group, 80
    \subitem line in a measured tree, 145
    \subitem line in measured tree, 152
    \subitem lower bound $\psi(\RR)$, 61
    \subitem lower bound $\Psi(\RR)$, 70
    \subitem measured tree, 149
    \subitem morphism of measured trees, 149
    \subitem oriented graph, 4
    \subitem ray descending from a point, 149
    \subitem relation module, 39
    \subitem subgroup separable group, 80
    \subitem $T(G,\leq,\lambda)$, 144
    \subitem Type $\FP_k$, 119
    \subitem wreath product, 127
  \item Delzant, T., 32
  \item Density of rank 1 points in $S(G)$, 90
  \item Diestel, R., 145
  \item Dunfield, N. M., 112

  \indexspace

  \item \textbf{E}ckmann, B., 79
  \item Edjvet, M., 96
  \item Elementary amenable groups
    \subitem definition, 93
    \subitem examples, 93, 101
  \item Elliptic elements
    \subitem definition, 147, 152

  \indexspace

  \item \textbf{F}arb, B., 84
  \item Finite generation
    \subitem graph theoretic criterion, 28
  \item Finitely presented groups
    \subitem and HNN-extensions, 75
    \subitem and wreath products, 98
    \subitem centre-by-metabelian, 86
    \subitem nilpotent-by-cyclic, 83
  \item Finiteness condition $\FP_2$
    \subitem characterization, 41
  \item Free soluble groups, 48, 102
  \item Freudenthal, H., 140
  \item Fuchsian groups, 91

  \indexspace

  \item \textbf{G}aboriau, D., 90
  \item Garille, S. G., 71
  \item Generalized soluble groups
    \subitem examples, 101
  \item Gildenhuys, D., 84
  \newpage
  \item Graph
    \subitem Cayley graph, 5
    \subitem connectivity, 54
    \subitem separating subset, 32
  \item Graph groups
    \subitem definition, 32, 53
    \subitem examples, 41
    \subitem invariant $\Sigma^1$, 54
  \item Grigorchuk, R., 92, 99
  \item Gromov, M., 97
  \item Groups
    \subitem fp abelian by infinite cyclic, 84
    \subitem fp nilpotent by cyclic, 83
    \subitem of Baumslag-Solitar, 19, 83, 96
    \subitem of deficiency 1, 112
  \item Groves, J. R. J., 48, 83, 85, 86
  \item Gruenberg, K. W., 135
  \item Guyot, L., 103, 104, 106

  \indexspace

  \item \textbf{H}all, M., 80
  \item Hall, P., 98, 129
  \item Harpe, P. de la, 92
  \item Hillman, J. A., 90
  \item HNN-criterion, 81
  \item HNN-extensions
    \subitem and fp groups, 75
  \item Hopf, H., 140
  \item Houghton, C. H., xiii, 99, 123
  \item Hyperbolic elements
    \subitem definition, 147, 152

  \indexspace

  \item \textbf{I}nfinitely related groups
    \subitem examples, 99, 101
    \subitem existence, 98
    \subitem free soluble, 102
    \subitem generalized soluble, 101
    \subitem locally finite by cyclic, 99
    \subitem locally nilpotent by cyclic, 101
    \subitem sufficient condition, 97
  \item Invariant $\Sigma^1$
    \subitem algebraic criterion, 60
    \subitem and rank 1 points, 17
    \subitem definition, 8
    \subitem geometric criterion, 15
    \subitem independence, 10
    \subitem lower bound $\Psi(\RR)$, 71
    \subitem lower bound $\psi(\RR)$, 62
    \subitem openness, 17
  \item Invariant $\Sigma^1$ and
    \subitem ascending HNN-extensions, 17
    \subitem automorphisms, 48
    \subitem change of groups, 47
    \subitem descending HNN-extension, 67
    \subitem descending HNN-extensions, 22
    \subitem fg normal subgroups, 23, 29
    \subitem finite presentability, 34
    \subitem $\FP_2$, 42, 43
    \subitem non-degenerate HNN-extension, 82
    \subitem quotient groups, 25, 49
    \subitem subgroups of finite index, 50
  \item Invariant $\Sigma^1$
    \subitem for arbitrary groups, 131
  \item Invariant $\Sigma^1(-;\Z)$
    \subitem coincidence with $\Sigma^1$, 117
    \subitem definition, 119
  \item Invariant $\Sigma_{G'}(G)$
    \subitem definition, 121
    \subitem equality with $\Sigma^1$, 123
  \item Invariant $\Sigma^M$
    \subitem definition, 140
    \subitem equality with $\Sigma^1$, 141
  \item Invariant $\Sigma^m(-;A)$
    \subitem definition, 119

  \indexspace

  \item \textbf{J}aco, W., 91

  \indexspace

  \item \textbf{L}amplighter groups
    \subitem definition, 129
  \item Large groups
    \subitem definition, 93
    \subitem examples, 80, 95--97
  \item Lennox, J. C., 83, 86, 98
  \item LERF groups
    \subitem definition, 80
    \subitem example, 80
    \subitem structure theorem, 80
  \item Levitt, G., xii, 115, 133
  \item Line in a measured tree
    \subitem definition, 145, 152
  \item Locally finite by cyclic groups
    \subitem examples, 99--100
  \item Longobardi, P., 133

  \indexspace

  \item \textbf{M}agnus, W., 107
  \item Maj, M., 133
  \item McCool, J., 32
  \item McLain, D. H., 101
  \item Measured tree
    \subitem associated to $(G, \XX, \chi)$, 144
    \subitem definition, 149
  \item Measured tree $T(-,\leq,\lambda)$
    \subitem canonical action, 146
    \subitem properties, 144
  \item Measured trees
    \subitem examples, 150--151
    \subitem morphisms, 149
  \item Meier, J., 32, 33, 54, 71
  \item Meigniez, G., x, 115, 141, 143
  \item Meinert, H., 54
  \item Miller III, C. F., 77, 99, 100
  \item Morphisms of measured trees
    \subitem construction, 155
    \subitem definition, 149
    \subitem examples, 150--151
    \subitem existence, 154
  \item Mosher, L., 84
  \item M\"uller, H., 79

  \indexspace

  \item \textbf{N}eumann, B. H., 98, 100
  \item Neumann, H., 48
  \item Neumann, P. M., 80
  \item Neumann, W. D., vii, ix, 15, 23, 35, 90, 115
  \item Noel, B., 41
  \item Non-abelian free subgroups
    \subitem existence, 81
  \item Notation
    \subitem $[\chi]$, 2
    \subitem $A_g$, 146
    \subitem $\sigma(\vartheta)$, 3
    \subitem $\chi$, 1
    \subitem $\defi(G)$, 88
    \subitem $\Delta_{b}$, 35
    \subitem $\edg(\Gamma)$, 4
    \subitem $\eta \colon \XX \to G$, 5
    \subitem $f_{(r,\chi)}$, 61, 62
    \subitem $G_\chi$, 7
    \subitem $G_\chi^I$, 14
    \subitem $\Gamma(G, \XX)$, 5
    \subitem $\Gamma_\chi$, 8
    \subitem $\GG_{\RR}(\chi)$, 70
    \subitem $\Hom(G,\R)$, 1
    \subitem $P(\Gamma)$, 4
    \subitem $\bar{p}$, 65
    \subitem $\psi(\RR)$, 61
    \subitem $\Psi(\RR)$, 70
    \subitem $r_{0}$, 2
    \subitem $S(G,H)$, 3
    \subitem $T(G,\XX,\chi)$, 144
    \subitem $T(G,\leq,\lambda)$, 144
    \subitem $v_\chi$, 10
    \subitem $\ver(\Gamma)$, 4
  \item Number of ends
    \subitem in the direction of $\chi$, 140
    \subitem of a topological space, 139

  \indexspace

  \item \textbf{O}l'shanskii, A. Yu., 99
  \item Orlandi-Korner, L. A., 32

  \indexspace

  \item \textbf{P}apadima, S., 54
  \item Pride, S. J., 94, 96

  \indexspace

  \item \textbf{R}ay descending from a point
    \subitem definition, 145, 149
  \item Remeslennikov, V. N., vii
  \item Renz, B., xi, xii, 111, 115, 118, 119
  \item Rhemtulla, A. H., 133
  \item Right angled Artin groups, \see{Graph groups}{32}
  \item Robinson, D. J. S., 83, 86, 98
  \item Rosenblatt, J. M., 134

  \indexspace

  \item \textbf{S}apir, M., 99
  \item Schanuel's Lemma, 40
  \item Scott, P., 80
  \item Separating subsets of a graph
    \subitem determination, 56
    \subitem examples, 56--58
  \item Shalen, P., 39
  \item $\Sigma^1$-criterion
    \subitem algebraic form, 60
    \subitem geometric form, 15
  \item Simple vertex, 111
  \item {\v{S}}mel$'$kin, A. L., 102
  \item Soluble groups
    \subitem Abels' example, 41
    \subitem examples of Baumslag-Solitar, 8
    \subitem satisfying max-n, 129
    \subitem with preassigned derived length, 129
  \item Special edge, 111
  \item Stallings, J. R., 79
  \item St{\"o}hr, R., 97
  \item Strebel, R., vii, ix, 15, 23, 35, 40, 48, 77, 79, 83--85, 90, 
		97, 98, 102, 104, 106, 115, 123
  \item Structure theorem
    \subitem for fg indicable groups, 104
    \subitem for fp indicable groups, 75
  \item Suciu, A., 54
  \item Surface groups, 91

  \indexspace

  \item \textbf{T}hurston, D. P., 112
  \item Track of a word, 116
  \item Trotter, H. F., 84
  \item Type $\FP_k$, 119

  \indexspace

  \item \textbf{V}anWyck, L., 32, 33, 54

  \indexspace

  \item \textbf{W}reath products
    \subitem definition, 127
    \subitem examples, 99, 129
    \subitem finite presentation, 98

\end{theindex}

%
%
\end{document}